\documentclass[fontsize=10]{scrartcl}
\usepackage[utf8]{inputenc}
\usepackage{float}
\usepackage{mathtools}
\usepackage{url}
\usepackage{amsmath}
\usepackage{amsthm}
\usepackage{amssymb}
\usepackage{graphicx}
\usepackage[a4paper]{geometry}
\usepackage[pdfusetitle,
 bookmarks=true,bookmarksnumbered=false,bookmarksopen=false,
 breaklinks=false,pdfborder={0 0 1},backref=false,colorlinks=false]
 {hyperref}

\makeatletter

\providecommand{\tabularnewline}{\\}

\theoremstyle{plain}
\newtheorem{thm}{\protect\theoremname}
\theoremstyle{remark}
\newtheorem{rem}{\protect\remarkname}
\theoremstyle{plain}
\newtheorem{prop}{\protect\propositionname}
\theoremstyle{plain}
\newtheorem{choice}{\protect\choicename}
\theoremstyle{plain}
\newtheorem{lem}{\protect\lemmaname}
\theoremstyle{remark}
\newtheorem{summary}{\protect\summaryname}
\theoremstyle{plain}
\newtheorem{cor}{\protect\corollaryname}

\usepackage{tikz}

\makeatother

\providecommand{\choicename}{Choice}
\providecommand{\corollaryname}{Corollary}
\providecommand{\lemmaname}{Lemma}
\providecommand{\propositionname}{Proposition}
\providecommand{\remarkname}{Remark}
\providecommand{\summaryname}{Summary}
\providecommand{\theoremname}{Theorem}

\begin{document}
\title{Finite-time singularity via multi-layer degenerate pendula for the
2D Boussinesq equation with uniform $C^{1,\sqrt{\frac{4}{3}}-1-\varepsilon}\cap L^{2}$
force}
\author{Diego Córdoba\thanks{Instituto de Ciencias Matemáticas CSIC-UAM-UCM-UC3M, Spain. E-mail:
dcg@icmat.es}, Andrés Laín-Sanclemente\thanks{Instituto de Ciencias Matemáticas CSIC-UAM-UCM-UC3M, Spain. E-mail:
andres.lain@icmat.es} and Luis Martínez-Zoroa\thanks{University of Basel, Switzerland. E-mail: luis.martinezzoroa@unibas.ch}}
\maketitle
\begin{abstract}
We establish the existence of compactly supported solutions of the
inviscid incompressible 2D Boussinesq equation with $C^{1,\sqrt{\frac{4}{3}}-1-\varepsilon}\cap L^{2}$
force that develop a singularity in finite time. Importantly, the
force preserves this regularity at the blow-up time. Moreover, the
forces in the vorticity and density equations have compact support.
The mechanism behind the blow-up is an accumulated hysteresis effect
on the vorticity caused by an infinite chain of ``degenerate'' pendula
and flickering density.
\end{abstract}
\tableofcontents{}

\section{Introduction}

\subsection{Physical model}

The main object of study in the present paper will be the forced inviscid
incompressible 2D Boussinesq system, which is given by
\begin{equation}
\begin{aligned}\frac{\partial u}{\partial t}+u\cdot\nabla u & =-\nabla P-\rho\hat{e}_{1}+f_{u} & \left(\text{momentum equation}\right),\\
\frac{\partial\rho}{\partial t}+u\cdot\nabla\rho & =f_{\rho} & \left(\text{energy equation}\right),\\
\nabla\cdot u & =0 & \left(\text{incompressibility constraint}\right).
\end{aligned}
\label{eq:Boussinesq system velocity formulation}
\end{equation}
These equations describe the evolution of the velocity field $u:\left[0,T\right]\times\mathbb{R}^{2}\to\mathbb{R}$,
the hydrodynamic pressure $P:\left[0,T\right]\times\mathbb{R}^{2}\to\mathbb{R}$
and the density fluctuations $\rho:\left[0,T\right]\times\mathbb{R}^{2}\to\mathbb{R}$
of a compressible fluid subject to gravity, weak external heating
$f_{\rho}:\left[0,T\right]\times\mathbb{R}^{2}\to\mathbb{R}$ and
external inertial forcing $f_{u}:\left[0,T\right]\times\mathbb{R}^{2}\to\mathbb{R}^{2}$
in the limit of small Mach number and ``powerful'' gravity. In other
words, these equations can be deduced from the compressible Euler
equations through appropriate approximations (see \cite{Boussinesq}
and \cite{Oberbeck} for more details). Physically speaking, $\hat{e}_{1}$
is vertical and points upwards, as it represents the opposite direction
to that of gravity, while $\hat{e}_{2}$ constitutes some other direction
perpendicular to $\hat{e}_{1}$ given by the right-hand-rule. On the
one hand, the Boussinesq system has found great practical applicability
in the study of stratified flows in geophysical fluid dynamics (see,
for instance, \cite{Majda} and \cite{Rieutord}). On the other hand,
from a purely mathematical perspective, the study of \eqref{eq:Boussinesq system velocity formulation}
is of great interest, because of the deep correspondence that exists
between the 2D Boussinesq system and the incompressible axisymmetric
3D Euler equations (see \cite{Constantin} and \cite{Majda Bertozzi}).

The incompressibility constraint $\nabla\cdot u=0$ is a consequence
of the vanishing Mach number limit, in the same way that it appears
in the compressible-to-incompressible limit of the Euler equations.
It is important to emphasize that $\rho$ only denotes a fluctuation
with respect to a fixed mean fluid density and, consequently, can
have both positive and negative values. This means that the term $\rho\hat{e}_{1}$
in the momentum equation represents the buoyancy force due to these
density variations and is, consequently, a first order term. Actually,
the main effect of gravity, the zeroth order term, is ``included''
in the pressure term. The force $f_{u}$ plays the same role as an
external force in the momentum equation of the compressible Euler
system. Moreover, another important feature of the Boussinesq system
is that the density is closely related to the temperature, so much
so that one can think of $-\rho$ as the temperature fluctuations
with respect to a fixed background temperature. It is this intricate
connection that allows us to think of the source term $f_{\rho}$
as an external energy contribution achieved through heating (for example).
Nonetheless, it is important to point out that this $f_{\rho}$ must
also be a first order term, in the sense that it must be a correction
to a situation of no external heating, in the same way that $\rho$
is a correction to the background density. This is why we call $f_{\rho}$
a ``weak'' external heating.

If one wishes to conceive a lab model of system \eqref{eq:Boussinesq system velocity formulation},
one could consider a very thin and vertically oriented sheet of gas
contained between two walls. These choices account for the fact that
the flow is two-dimensional and for the effect of gravity. On the
one hand, the source term $f_{\rho}$ could be brought into reality
by installing an array of infinitesimal piezoelectric crystals on
one of the walls. Through these, one could add and subtract heat from
the system. On the other hand, the term $f_{u}$ could be realized
through an array of infinitesimal fans built on the other wall.

Similarly to what happens in the incompressible Euler equations, taking
the curl in the momentum equation, one can eliminate not only the
pressure but also the incompressibility constraint. Indeed, taking
the curl in \eqref{eq:Boussinesq system velocity formulation} provides
\begin{equation}
\begin{aligned}\frac{\partial\omega}{\partial t}+u\cdot\nabla\omega & =\frac{\partial\rho}{\partial x_{2}}+f_{\omega} & \left(\text{vorticity equation}\right),\\
\frac{\partial\rho}{\partial t}+u\cdot\nabla\rho & =f_{\rho} & \left(\text{energy equation}\right),
\end{aligned}
\label{eq:Boussinesq system vorticity formulation}
\end{equation}
where $\omega=\frac{\partial u_{2}}{\partial x_{1}}-\frac{\partial u_{1}}{\partial x_{2}}$
is the fluid vorticity, $f_{\omega}=\frac{\partial f_{u_{2}}}{\partial x_{1}}-\frac{\partial f_{u_{1}}}{\partial x_{2}}$
is a source term and the velocity field $u$ can be recovered from
the vorticity as $u=\nabla^{\perp}\Delta^{-1}\omega$. Throughout
the paper, we will focus on system \eqref{eq:Boussinesq system vorticity formulation},
which, although equivalent to \eqref{eq:Boussinesq system velocity formulation},
only contains two scalar unknowns and is, therefore, easier to deal
with. Despite that $f_{\omega}$ and $f_{\rho}$ are not forces in
the light of physics, for the sake of simplicity, from this point
onward, we will refer to them as so.

\subsection{Context}

Local existence of smooth solutions of \eqref{eq:Boussinesq system velocity formulation}
was first established by Chae and Nam in \cite{Chae Nam} (1997).
Afterwards, Chae, Kim and Nam were able to extend this existence result
to $C^{1,\alpha}$ solutions ($0<\alpha<1$) in \cite{Chae Kim Nam}
(1999). On the one hand, criteria for the formation of singularities
are provided in the works of Chae and Nam \cite{Chae Nam} (1997);
Chae, Kim and Nam \cite{Chae Kim Nam} (1999); Danchin \cite{Danchin}
(2013) and Chae and Constantin \cite{chae constantin} (2021). On
the other hand, sufficient conditions for continuation of solutions
are given by Danchin in \cite{Danchin} (2013) and Chae and Wolf in
\cite{Chae Wolf} (2019).

As has been pointed out before, one very important peculiarity of
the 2D incompressible Boussinesq system is that it shares a deep connection
with the incompressible axisymmetric 3D Euler equations \cite{Constantin,Majda Bertozzi}.
Actually, this connection is so profound that, in multiple occasions,
singularities found for the 2D Boussinesq system have been ``translated''
to the 3D Euler equations, as we shall see shortly. As a consequence
of this, we have favored a joint exposition of the singularities of
both incompressible 2D Boussinesq and incompressible 3D Euler instead
of presenting only the ones associated to 2D Boussinesq. One more
comment before we begin: unless otherwise stated, all singularities
mentioned have finite energy.

The search for singularities in Euler and Boussinesq started in the
1990's through numerical simulations accomplished by Pumir and Siggia
\cite{Pumir Siggia} and E and Shu \cite{E Shu}. These groups achieved
results that contradicted one another. The next important numerical
milestone is due to Luo and Hou \cite{Luo Hou} (2014), whose work
was in agreement with the existence of finite-time singularities in
the axisymmetric 3D Euler equations in a cylindrical domain with boundary.
Furthermore, a one-dimensional model constructed by Choi, Hou, Kiselev,
Luo, Sverak and Yao in \cite{Choi Hou Kiselev Luo Sverak Yao} (2014)
was shown by the authors to develop a finite-time blow-up from a smooth
initial datum. The ideas behind the simulation \cite{Luo Hou} could
be made rigorous in a work of Chen and Hou \cite{Chen Hou C1alpha}
(2019), where they found finite-time blow-ups of solutions of both
2D Boussinesq and axisymmetric 3D Euler without swirl with $C^{1,\alpha}$
initial velocity (also $C^{1,\alpha}$ density in the case of Boussinesq)
with $0<\alpha\ll1$ in certain bounded domains with $C^{1,\alpha}$
boundary. Actually, by means of a computer-assisted proof, Chen and
Hou were able to extend the aforementioned result to $C^{\infty}$
initial data in some bounded domains with $C^{\infty}$ boundary in
\cite{Chen Hou Cinfity I,Chen Hou Cinfty II} (2022-2023).

Among rigorous results, the first finite-time blow-up was proved by
Sarria and Wu \cite{Sarria Wu} (2015) for solutions of the 2D Boussinesq
equation with infinite energy in the infinite strip $\left[0,1\right]\times\mathbb{R}^{+}$
with smooth initial data. Elgindi and Jeong \cite{Elgindi Jeong Boussinesq}
(2017) were also able to find a finite-time singularity with Lipschitz
initial data (and smooth away from the origin) and finite energy for
2D Boussinesq in another type of unbounded domain with boundary. Inspired
by this result, Elgindi and Jeong \cite{Elgindi Jeong Euler} (2018)
developed the first finite-time singularity for the incompressible
3D Euler equations with $C^{1,\alpha}$ initial data ($0<\alpha<1$)
except at a point in a certain unbounded domain with boundary. The
first blow-up result in the whole space was achieved by Elgindi in
\cite{Elgindi C1alpha I} (2019) using a self-similar $C^{1,\alpha}\left(\mathbb{R}^{3}\right)$
solution of axisymmetric 3D Euler without swirl with $0<\alpha\ll1$,
which contained infinite energy. However, by adding a $C^{1,\alpha}\left(\mathbb{R}^{3}\right)\cap L^{2}\left(\mathbb{R}^{3}\right)$
force, the solution could be made to have finite energy. Moreover,
a stability analysis performed by Elgindi, Ghoul and Masmoudi in \cite{Elgindi C1alpha II}
(2019) proved the existence of asymptotically self-similar solutions
of axisymmetric 3D Euler without swirl with $C^{1,\alpha}\left(\mathbb{R}^{3}\right)$
initial datum ($0<\alpha\ll1$) with finite energy without the need
of an external force. This technique of asymptotically self-similar
solutions was refined by Elgindi and Pasqualotto in \cite{Elgindi Boussinesq I,Elgindi Boussinesq II}
(2023), where, by means of a computer-assisted proof, they constructed
finite-time singularities for both 2D Boussinesq and axisymmetric
3D Euler with swirl for $C^{1,\alpha}$ initial data ($0<\alpha\ll1$)
in the whole space. In the case of Boussinesq, the initial data are
$C^{\infty}$ except at the origin, and, in the case of Euler, the
initial velocity is $C^{\infty}$ except at a circle. Meanwhile, Huang,
Qin, Wang and Wei devised in \cite{Huang-Qin-Wang-Wei} (2023) an
exact self-similar blow-up with smooth profiles for the Hou-Luo model,
which is a one-dimensional model of the 2D Boussinesq equation. Lastly,
Wang, Lai, Gómez-Serrano and Buckmaster \cite{Wang Lai G=0000F3mez-Serrano Buckmaster}
(2022) have found a numerical candidate for an asymptotically self-similar
blow-up profile both for the 2D Boussinesq and axisymmetric 3D Euler
equations in the whole space using physics-informed neural networks,
which could be the basis of a future computer-assisted proof of singularities
with smooth initial data.

A completely different blow-up mechanism for the axisymmetric 3D Euler
equations without swirl with $C^{1,\alpha}$ initial datum ($0<\alpha\ll1$)
was presented in the work of Córdoba, Martínez-Zoroa and Zheng \cite{Euler no autosimilar}
(2023). Instead of making use of self-similar coordinates, the singularity
follows from an infinite chain of ODEs that governs the evolution
of a family of vorticity bumps. This was the first blow-up without
boundary where the initial datum was $C^{\infty}$ except at a single
point, where it was $C^{1,\alpha}$. Inspired by this result, Chen
showed in \cite{Chen} (2023) that the initial conditions of \cite{Chen Hou C1alpha}
and \cite{Elgindi C1alpha I,Elgindi C1alpha II} can be modified so
as to be smooth except at one point.

Yet another technique that has provided finite-time singularities
is the one uncovered by Córdoba and Martínez-Zoroa in \cite{Fuerza Euler}
(2023), where the authors accomplish a finite-time blow-up with smooth
initial datum for the 3D Euler equations with uniform $C^{1,\frac{1}{2}-\varepsilon}\cap L^{2}$
force. Using this method, Córdoba, Martínez-Zoroa and Zheng \cite{Fuerza Navier Stokes}
(2024) were able to prove a finite-time blow-up with smooth initial
datum for the hypodissipative Navier-Stokes equations with a force
in $L_{t}^{1}C_{x}^{1,\varepsilon}\cap L_{t}^{\infty}L_{x}^{2}$.
More recently, refining this procedure, Córdoba and Martínez-Zoroa
\cite{Fuerza IPM} (2024) have managed to construct finite-time singularities
for the 2D incompressible porous media equation (IPM) with a smooth
source. The current paper also employs an external force to carry
out the blow-up and, in this sense, it should be understood as a ``sibling''
of \cite{Fuerza Euler}, \cite{Fuerza Navier Stokes} and \cite{Fuerza IPM}.
Nevertheless, as we shall comment in greater detail in subsection
\ref{subsec:differences and similarities}, our construction displays
some features that make it clearly distinct from its ``siblings''.

Apart from singularities, another attribute of the incompressible
2D Boussinesq equations that has been extensively studied is the question
of stability of stratified solutions. Elgindi and Widmayer \cite{Elgindi Widmayer}
(2014) proved the stability in Sobolev spaces of the stationary stratified
solution $\left(\rho,u\right)=\left(-y,\left(0,0\right)\right)$ for
small times in the whole plane. Later, Tao, Wu, Zhao and Zheng \cite{Tao Zhao Zheng}
(2020) were able to establish the long-time stability of $\left(\rho,u\right)=\left(\beta y,\left(0,0\right)\right)$
in Sobolev spaces in the viscous case in the periodic box. Moreover,
Masmoudi, Said-Houari and Zhao \cite{Masmoudi} (2020) showed the
stability of the Couette flow in Gevrey-$\frac{1}{s}$ spaces with
$\frac{1}{3}<s\le1$ in the viscous case in $\mathbb{T}\times\mathbb{R}$.
However, in the absence of viscosity, as presented by Bedrossian,
Bianchini, Zelati and Dolce in \cite{bedrossian bianchini zelati}
(2021), small perturbations in Gevrey-$\frac{1}{s}$ spaces ($\frac{1}{2}<s\le1$)
of the stratified Couette flow in $\mathbb{T}\times\mathbb{R}$ develop
a shear-buoyancy instability. Furthermore, recently, Jurja and Widmayer
\cite{Jurja Widmayer} (2024) have obtained the longest known timescale
of existence of perturbations of the stratified solution $\left(\rho,u\right)=\left(-y,\left(0,0\right)\right)$
in Sobolev spaces in the whole plane, improving on \cite{Elgindi Widmayer}.

Lastly, the strong ill-posedness in $L^{\infty}$ of the 2D Boussinesq
equations was established by Bianchini, Hientzsch and Iandoli in \cite{Bianchini Hientzch Iandoli}
(2023) and the growth of solutions was studied by Kiselev, Park and
Yao in \cite{Kiselev Park Yao} (2022), where they present a broad
class of initial data whose associated solutions display infinite
growth in infinite time, provided they exist globally in time.

\subsection{Main result}

Without further ado, we present the main result of the paper.
\begin{thm}
\label{thm:MAIN THM}Let $\alpha\in\left(0,\alpha_{*}\right)$, where
$\alpha_{*}=\sqrt{\frac{4}{3}}-1$. There are classical solutions
$\left(u,\rho\right)$ in $\left[0,1\right)\times\mathbb{R}^{2}$
of the forced Boussinesq system \eqref{eq:Boussinesq system velocity formulation}
that satisfy:
\begin{enumerate}
\item $\forall\varepsilon>0$
\[
u\in C_{t}^{\infty}C_{x,c}^{\infty}\left(\left[0,1-\varepsilon\right]\times\mathbb{R}^{2}\right),\quad\rho\in C_{t}^{\infty}C_{x,c}^{\infty}\left(\left[0,1-\varepsilon\right]\times\mathbb{R}^{2}\right).
\]
\item $\forall\varepsilon>0$
\[
f_{\omega}\in C_{t}^{\infty}C_{x,c}^{\infty}\left(\left[0,1-\varepsilon\right]\times\mathbb{R}^{2}\right),\quad f_{\rho}\in C_{t}^{\infty}C_{x,c}^{\infty}\left(\left[0,1-\varepsilon\right]\times\mathbb{R}^{2}\right).
\]
\item ~
\[
f_{\omega}\in C_{t}^{0}C_{x,c}^{\alpha}\left(\left[0,1\right]\times\mathbb{R}^{2}\right),\quad f_{\rho}\in C_{t}^{0}C_{x,c}^{1,\alpha}\left(\left[0,1\right]\times\mathbb{R}^{2}\right).
\]
In particular, $\left|\left|f_{\omega}\left(t,\cdot\right)\right|\right|_{C^{\alpha}\left(\mathbb{R}^{2}\right)}$
and $\left|\left|f_{\rho}\left(t,\cdot\right)\right|\right|_{C^{1,\alpha}\left(\mathbb{R}^{2}\right)}$
are uniformly bounded in $t\in\left[0,1\right]$.
\item $f_{\omega}\left(t,x\right)$ is odd in $x_{2}$ $\forall t\in\left[0,1\right]$.
\item There is a finite-time singularity at $t=1$, i.e.,
\[
\lim_{T\to1^{-}}\int_{0}^{T}\left|\left|\nabla\rho\left(t,\cdot\right)\right|\right|_{L^{\infty}\left(\mathbb{R}^{2};\mathbb{R}^{2}\right)}\mathrm{d}t=\infty.
\]
(See \cite{Chae Kim Nam} for the blow-up criterion).
\end{enumerate}
\end{thm}
\begin{rem}
\label{rem:local existence}The finite-time singularity presented
in Theorem \ref{thm:MAIN THM} takes place in the well-posedness regime
of the Boussinesq system \eqref{eq:Boussinesq system velocity formulation}.
Indeed, the argument used in \cite{Chae Kim Nam} to prove local existence
for initial data satisfying $\rho_{0}\in C^{1,\alpha}\left(\mathbb{R}^{2}\right)$,
$u_{0}\in C^{1,\alpha}\left(\mathbb{R}^{2}\right)$, $\omega_{0}\in L^{q}\left(\mathbb{R}^{2}\right)$
and $\nabla\rho_{0}\in L^{q}\left(\mathbb{R}^{2}\right)$ (for some
$1<q<2$) can be extended to the case $f_{\rho}\in C_{t}^{0}C_{x}^{1,\alpha}\left(\left[0,\mathcal{T}\right]\times\mathbb{R}^{2}\right)$,
$\nabla f_{\rho}\in C_{t}^{0}L_{x}^{q}\left(\left[0,\mathcal{T}\right]\times\mathbb{R}^{2}\right)$,
$f_{u}\in C_{t}^{0}C_{x}^{1,\alpha}\left(\left[0,\mathcal{T}\right]\times\mathbb{R}^{2}\right)$
and $f_{\omega}\in C_{t}^{0}L_{x}^{q}\left(\left[0,\mathcal{T}\right]\times\mathbb{R}^{2}\right)$
without changing its structure, obtaining that $\rho\in C_{t}^{0}C_{x}^{1,\alpha}\left(\left[0,T\right]\times\mathbb{R}^{2}\right)\cap C_{t}^{0}L_{x}^{q}\left(\left[0,T\right]\times\mathbb{R}^{2}\right)$
and $u\in C_{t}^{0}C_{x}^{1,\alpha}\left(\left[0,T\right]\times\mathbb{R}^{2}\right)\cap C_{t}^{0}L_{x}^{2}\left(\left[0,T\right]\times\mathbb{R}^{2}\right)$
for some $T>0$ sufficiently small, depending on $\left|\left|\rho_{0}\right|\right|_{C^{1,\alpha}}$,
$\left|\left|\nabla\rho_{0}\right|\right|_{L^{q}}$, $\left|\left|u_{0}\right|\right|_{C^{1,\alpha}}$,
$\left|\left|\omega_{0}\right|\right|_{L^{q}}$, $\left|\left|f_{\rho}\right|\right|_{C_{t}^{0}C_{x}^{1,\alpha}}$,
$\left|\left|\nabla f_{\rho}\right|\right|_{C_{t}^{0}L_{x}^{q}}$,
$\left|\left|f_{u}\right|\right|_{C_{t}^{0}C_{x}^{1,\alpha}}$, and
$\left|\left|f_{\omega}\right|\right|_{C_{t}^{0}L_{x}^{q}}$.
\end{rem}
\begin{rem}
\label{rem:force velocity equation}It is important to point out that,
if we consider the Boussinesq system in the velocity formulation (see
equation \eqref{eq:Boussinesq system velocity formulation}), the
force $f_{u}$ associated to our solution does not have compact support.
Nevertheless, $f_{u}\in C_{t}^{0}C_{x}^{1,\alpha}\left(\left[0,1\right]\times\mathbb{R}^{2}\right)\cap C_{t}^{0}L_{x}^{2}\left(\left[0,1\right]\times\mathbb{R}^{2}\right)$
because $f_{\omega}$ is compactly supported and has zero mean. Indeed,
since $f_{\omega}\left(t,\cdot\right)$ is odd in $x_{2}$, for $W>0$
large enough, we have
\[
\int_{\mathbb{R}^{2}}f_{\omega}\left(t,x\right)\mathrm{d}x=\int_{\left[-W,W\right]^{2}}f_{\omega}\left(t,x\right)\mathrm{d}x=\int_{-W}^{W}\underbrace{\int_{-W}^{W}\underbrace{f_{\omega}\left(t,\left(x_{1},x_{2}\right)\right)}_{\text{odd in }x_{2}}\mathrm{d}x_{2}}_{=0}\mathrm{d}x_{1}=0.
\]
\end{rem}
\begin{rem}[Regularity of the solution at the blow-up time]
\label{rem:regularity at blow-up}At the instant of the blow-up,
there is a regularity loss $r_{\text{loss}}\in\left(0,1\right)$ such
that
\[
u\left(1,\cdot\right)\notin C^{\beta}\left(\mathbb{R}^{2}\right),\quad\rho\left(1,\cdot\right)\notin C^{\beta}\left(\mathbb{R}^{2}\right)\quad\forall\beta\in\left(1-r_{\text{loss}},1\right).
\]
Choosing the parameters $\delta$ and $\mu$ of the solution small
enough, this regularity loss $r_{\text{loss}}$ can be made as small
as we wish, independently of $\alpha$. Nevertheless, it has an upper
limit $r_{\text{loss},\max}\left(\alpha\right)$, which does depend
on $\alpha$. We have
\[
\lim_{\alpha\to0}r_{\text{loss},\max}\left(\alpha\right)\geq\frac{\alpha_{*}}{2}\approx0.0774,\quad\lim_{\alpha\to\alpha_{*}}r_{\text{loss,\ensuremath{\max}}}\left(\alpha\right)=0.
\]
\end{rem}
\begin{rem}[Blow-up rate]
\label{rem:blow-up rate}Neither the vorticity $\omega$ nor the
gradient of the density $\nabla\rho$ have a well defined blow-up
rate. Nonetheless, one can find increasing time sequences $\left(t_{1,n}\right)_{n\in\mathbb{N}},\left(t_{2,n}\right)_{n\in\mathbb{N}}\subseteq\left[0,1\right]$
such that $\lim_{n\to\infty}t_{1,n}=\lim_{n\to\infty}t_{2,n}=1$ and
so that, $\forall\varepsilon>0$,
\[
\begin{aligned}\frac{1}{\left(1-t_{1,n}\right)^{\frac{1}{1-\gamma}}}\lesssim_{Y,\varepsilon,\delta,\gamma,\varphi} & \left|\left|\omega\left(t_{1,n},\cdot\right)\right|\right|_{L^{\infty}\left(\mathbb{R}^{2}\right)}\lesssim_{Y,\varepsilon,\delta,\gamma,\varphi}\frac{1}{\left(1-t_{1,n}\right)^{\frac{1+\varepsilon}{1-\gamma}}},\\
\frac{1}{\left(1-t_{1,n}\right)^{\frac{1}{1-\gamma}}}\lesssim_{Y,\varepsilon,\delta,\gamma,\varphi} & \left|\left|\frac{\partial\rho}{\partial x_{2}}\left(t_{1,n},\cdot\right)\right|\right|_{L^{\infty}\left(\mathbb{R}^{2}\right)}\lesssim_{Y,\varepsilon,\delta,\gamma,\varphi}\frac{1}{\left(1-t_{1,n}\right)^{\frac{1+\varepsilon}{1-\gamma}}},\\
\frac{1}{1-t_{2,n}}\lesssim_{Y,\varepsilon,\delta,\gamma,\varphi} & \left|\left|\omega\left(t_{2,n},\cdot\right)\right|\right|_{L^{\infty}\left(\mathbb{R}^{2}\right)}\lesssim_{Y,\varepsilon,\delta,\gamma,\varphi}\frac{1}{\left(1-t_{2,n}\right)^{1+\varepsilon}},\\
 & \left|\left|\frac{\partial\rho}{\partial x_{2}}\left(t_{2,n},\cdot\right)\right|\right|_{L^{\infty}\left(\mathbb{R}^{2}\right)}=0.
\end{aligned}
\]
$Y,\delta,\gamma,\varphi$ are different parameters of the construction
which will be introduced throughout the paper. For the moment, it
is enough to bear in mind that the value of $\gamma$ will lie very
close to one and that $\gamma\to1$ as $\alpha\to\alpha_{*}$. These
huge differences in the blow-up rate of different sequences are due
to the highly oscillatory behavior in time of both $\omega$ and $\frac{\partial\rho}{\partial x_{2}}$.
\end{rem}
\begin{rem}[Regularity of the time derivatives of $f_{\rho}$ and $f_{\omega}$]
\label{rem:regularity of time derivatives}Let $k_{0}\in\mathbb{N}$
be any whole number. Choosing the parameter $\delta$ of our construction
small enough (the smallness is determined by the value of $k_{0}$),
we can actually guarantee that
\[
f_{\omega}\in C_{t}^{k}C_{x,c}^{\alpha}\left(\left[0,1\right]\times\mathbb{R}^{2}\right),\quad f_{\rho}\in C_{t}^{k}C_{x,c}^{1,\alpha}\left(\left[0,1\right]\times\mathbb{R}^{2}\right),\quad\forall k\in\mathbb{N}\text{ s.t. }k\le k_{0}.
\]
Importantly, bear in mind that we can make $\delta$ as small as we
wish independently of the value of the regularity parameter $\alpha$.
In a nutshell, the above is true because differentiating once in time
roughly corresponds to dividing by the time scale $1-t_{n+1}\sim C^{-\delta\left(\frac{1}{1-\gamma}\right)^{n}}$.
Thus, differentiating $k$ times with respect to $t$ amounts roughly
to multiplying by the time scale $k$ times and we can compensate
this by taking $\delta$ small enough.
\end{rem}
\begin{rem}[Extension to axisymmetric 3D Euler]
\label{rem:extension to Euler}Thanks to the close resemblance that
exists between the axisymmetric incompressible 3D Euler equations
and the 2D Boussinesq system, it is possible to use the same construction
to prove a blow-up for the axisymmetric 3D Euler equations, as will
appear in a forthcoming paper \cite{Boussinesq a Euler}.
\end{rem}

\subsection{\label{subsec:ideas behind blow-up}Ideas behind the blow-up}

We decompose our vorticity as a countable sum of functions, i.e.,
$\omega=\sum_{n=1}^{\infty}\omega^{\left(n\right)}$. Furthermore,
we choose these $\left(\omega^{\left(n\right)}\right)_{n\in\mathbb{N}}$
so that the support of $\omega^{\left(n\right)}$ is much smaller
than the support of $\omega^{\left(n-1\right)}$. Each one of these
$\left(\omega^{\left(n\right)}\right)_{n\in\mathbb{N}}$ receives
the name of ``vorticity layer'' and we shall refer to $\omega^{\left(n\right)}$
as the $n$-th vorticity layer. We apply the same decomposition to
the density, i.e., $\rho=\sum_{n=1}^{\infty}\rho^{\left(n\right)}$,
with $\rho^{\left(n\right)}$ representing the $n$-th density layer.
$\omega^{\left(n\right)}$ and $\rho^{\left(n\right)}$ will both
have the same support. This means that, spatially speaking, the vorticity
and density layers look exactly the same. Consequently, we will just
write ``the $n$-th layer'' to refer to the support of either $\omega^{\left(n\right)}$
or $\rho^{\left(n\right)}$. These layers are introduced sequentially
in time, i.e., at the beginning ($t=0$), we will only have one layer
and, as time advances, more and more layers will appear in the construction.
The sequence of time instants when the layers are introduced, which
we denote by $\left(t_{n}\right)_{n\in\mathbb{N}}$, will have an
accumulation point: the blow-up time, which we will fix at $t=1$.

In a nutshell, and somewhat imprecisely, the mechanism behind the
finite-time blow-up is the following: Suppose we have already introduced
the first $n$ layers. Then, we use the force in the density equation
to create a new density layer of small amplitude in a certain location.
The interaction between the vorticity of the $n$-th layer $\omega^{\left(n\right)}$
and the new density layer $\rho^{\left(n+1\right)}$ leads to displacement
and deformation of layer $n+1$ and to an increase of the amplitude
of $\omega^{\left(n+1\right)}$. Actually, this growth is so significant,
that, by the time we introduce layer $n+2$ (at time $t=t_{n+2}$),
the amplitude of $\omega^{\left(n+1\right)}$ will greatly exceed
that of $\omega^{\left(n\right)}$ (see figure \ref{fig:two layer mechanism}).
In this manner, iterating this procedure an infinite number of times,
$\left|\left|\omega\left(t,\cdot\right)\right|\right|_{L^{\infty}\left(\mathbb{R}^{2}\right)}$
grows so much that it blows up (see figure \ref{fig:vorticity blow-up}).

\begin{figure}
\begin{centering}
\includegraphics[width=0.32\columnwidth]{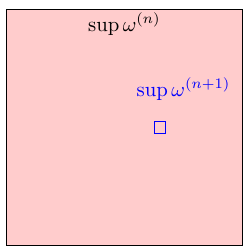}\includegraphics[width=0.32\columnwidth]{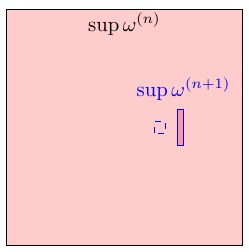}\includegraphics[width=0.32\columnwidth]{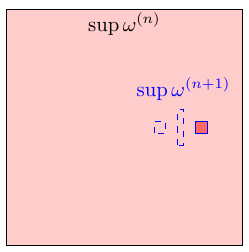}
\par\end{centering}
\caption{\label{fig:two layer mechanism}Three time snapshots of the evolution
of the supports and amplitudes of $\omega^{\left(n\right)}$ and $\omega^{\left(n+1\right)}$.
The support of $\omega^{\left(n\right)}$ is denoted by a square of
black edges, whereas the support of $\omega^{\left(n+1\right)}$ is
represented by quadrilaterals with blue edges. The amplitude of $\omega^{\left(n\right)}$
and $\omega^{\left(n+1\right)}$ is hinted by the intensity of the
red color that is used to fill the supports; the more intense, the
greater the amplitude. The left image corresponds to the instant $t_{n+1}$,
when a new density layer of small amplitude is introduced. After some
time, the $\left(n+1\right)$-th layer translates and deforms, leading
to the central image. During this time, the amplitude of $\omega^{\left(n+1\right)}$
also increases. Lastly, by the time instant $t_{n+2}$, when we will
introduce layer $n+2$, the $\left(n+1\right)$-th layer has evolved
to the situation of the right picture, with an even greater growth
of the amplitude of $\omega^{\left(n+1\right)}$. The fact that the
support of $\omega^{\left(n+1\right)}$ at time $t_{n+2}$ has the
same geometric form as at time $t_{n+1}$ is no arbitrary choice,
it reflects the actual behavior of our construction (see subsection
\ref{subsec:completion of the toy model} for more details).}

\end{figure}

\begin{figure}
\begin{minipage}[t]{0.49\columnwidth}%
\begin{center}
\includegraphics[width=0.65\columnwidth]{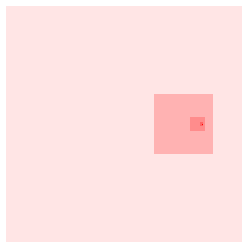}
\par\end{center}
\caption{\label{fig:vorticity blow-up}Schematic representation of the amplitude
of the vorticity. The more intense the red color is, the greater the
amplitude. This image represents a cascade of vorticity layers that
grow in amplitude but decrease in support.}
\end{minipage}\hfill{}%
\begin{minipage}[t]{0.49\columnwidth}%
\begin{flushright}
\includegraphics[width=0.8\columnwidth]{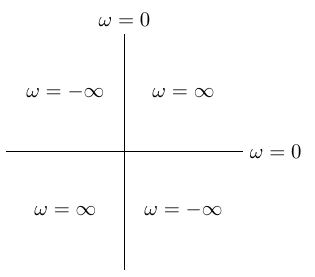}
\par\end{flushright}
\caption{\label{fig:vorticity singularity}Schematic representation of the
singularity developed by the vorticity at the blow-up time.}
\end{minipage}
\end{figure}

One curious quirk of our construction is that the growth of the vorticity
does not occur around a fixed point in space; i.e., if, for each time
instant $t\in\left[0,1\right)$, we consider the smallest possible
ball whose center lies in the $x_{1}$ axis and where $\left|\left|\omega\right|\right|_{L^{\infty}\left(\mathbb{R}^{2}\right)}$
is attained, we will see that this ball does not only get smaller
as time evolves, but also displaces to the right. However, as $t\to1$,
the center of this ball will converge to a specific point in the $x_{1}$
axis, around which an odd-odd discontinuity singularity will develop
for the vorticity (see figure \ref{fig:vorticity singularity} for
a graphical representation).

We have claimed that the interaction between $\omega^{\left(n\right)}$
and $\rho^{\left(n+1\right)}$ leads to a displacement and deformation
of $\omega^{\left(n+1\right)}$ and to an increase of its amplitude,
but we have given no justification for this fact. In subsection \ref{subsec:vorticity growth mechanism}
we will try to provide the main reasons why we should expect such
behavior. Furthermore, we will see that the displacement of a layer
with respect to the one that precedes it is given by an equation closely
related to that of a pendulum (we shall explore this connection in
subsection \ref{subsec:pendula}).

\subsubsection{\label{subsec:vorticity growth mechanism}Mechanism of growth for
the vorticity}

For now, consider only two layers: the background layer and the new
layer. We will denote the physical quantities of the background layer
with the superscript $\left(0\right)$, whereas the physical quantities
of the new layer will carry the superscript $\left(1\right)$. The
physical quantities of the background layer will have a spatial scale
of order $1$, while the physical quantities of the new layer will
have spatial scale $\frac{1}{a^{\left(1\right)}\left(t\right)}$ in
the $x_{1}$ axis and $\frac{1}{b^{\left(1\right)}\left(t\right)}$
in the $x_{2}$ axis. Take the following function as density of the
new layer
\begin{equation}
\rho^{\left(1\right)}\left(t,x\right)=-A^{\left(1\right)}\left(t\right)\sin\left(a^{\left(1\right)}\left(t\right)\left(x_{1}-c^{\left(1\right)}\left(t\right)\right)\right)\cos\left(b^{\left(1\right)}\left(t\right)x_{2}\right),\label{eq:intro form rho}
\end{equation}
where $A^{\left(1\right)}\left(t\right)$ is a function whose purpose
is to ``create'' this new density from nothing and eliminate it
if needed. On the one hand, $a^{\left(1\right)}\left(t\right)$ and
$b^{\left(1\right)}\left(t\right)$ are the frequencies of the sinusoidal
functions of $\rho^{\left(1\right)}$ and, consequently, their multiplicative
inverses are related to the scales of spatial variation of $\rho^{\left(1\right)}$.
Hence, studying the time evolution of $a^{\left(1\right)}\left(t\right)$
and $b^{\left(1\right)}\left(t\right)$ will tell us how the new layer
deforms. On the other hand, the point $\left(c^{\left(1\right)}\left(t\right),0\right)$
corresponds to the position of the center of the new layer. In this
way, the time dynamics of $c^{\left(1\right)}\left(t\right)$ provide
the displacement of the new layer.

Notice that the density given in \eqref{eq:intro form rho} does not
have compact support and, clearly, we would need it for the energy
of our solution to remain finite. Nonetheless, for the purpose of
explaining the mechanism of growth for the vorticity, the compact
support is not important and, actually, working without it makes all
computations much simpler. Hence, we will ignore this requirement
in this subsection. 

The main idea of our construction is that, with this density, the
vorticity displays hysteresis, i.e., it is possible to produce a growth
in vorticity that remains even after the density has been put out.
In simpler terms, when we turn on the density, the vorticity will
grow and, as we turn the density off, the vorticity will \textbf{not}
decrease. Let us informally see that this can be done. 

Let us suppose that our vorticity can be expressed as the sum of two
contributions $\omega=\omega^{\left(0\right)}+\omega^{\left(1\right)}$.
$\omega^{\left(0\right)}$ acts as the background vorticity, while
$\omega^{\left(1\right)}$ represents the vorticity generated by our
density \eqref{eq:intro form rho}. Likewise, the velocity field will
have two contributions, $u=u^{\left(0\right)}+u^{\left(1\right)}$.
On the other hand, we will assume that there is no background density,
i.e., $\rho=\rho^{\left(1\right)}$. Recall that the vorticity equation
is (see equation \eqref{eq:Boussinesq system vorticity formulation})
\[
\frac{\partial\omega}{\partial t}+u\cdot\nabla\omega=\frac{\partial\rho}{\partial x_{2}}+f_{\omega}.
\]
Expanding the quadratic terms, we obtain
\begin{equation}
\frac{\partial\omega^{\left(0\right)}}{\partial t}+\frac{\partial\omega^{\left(1\right)}}{\partial t}+u^{\left(0\right)}\cdot\nabla\omega^{\left(0\right)}+u^{\left(0\right)}\cdot\nabla\omega^{\left(1\right)}+u^{\left(1\right)}\cdot\nabla\omega^{\left(0\right)}+u^{\left(1\right)}\cdot\nabla\omega^{\left(1\right)}=\frac{\partial\rho^{\left(1\right)}}{\partial x_{2}}+f_{\omega}.\label{eq:intro vorticity equation 0}
\end{equation}
Suppose for now that the background $\omega^{\left(0\right)}$ is
a solution of the unforced 2D Euler equation, i.e.,
\begin{equation}
\frac{\partial\omega^{\left(0\right)}}{\partial t}+u^{\left(0\right)}\cdot\nabla\omega^{\left(0\right)}=0.\label{eq:intro assumption omega0}
\end{equation}
Then, equation \eqref{eq:intro vorticity equation 0} simplifies to
\[
\frac{\partial\omega^{\left(1\right)}}{\partial t}+u^{\left(0\right)}\cdot\nabla\omega^{\left(1\right)}+u^{\left(1\right)}\cdot\nabla\omega^{\left(0\right)}+u^{\left(1\right)}\cdot\nabla\omega^{\left(1\right)}=\frac{\partial\rho^{\left(1\right)}}{\partial x_{2}}+f_{\omega}.
\]
Moreover, assume, for now, that the newly generated velocity is small,
so small that the term $u^{\left(1\right)}\cdot\nabla\omega^{\left(0\right)}$
can be easily ``neglected''. ``Neglected'' means that is not a
leading order term and, consequently, we can use the force to absorb
it\footnote{Using the force to compensate non leading-order terms, the hope is
that the force will remain more regular than the solution at the blow-up
time.}. Furthermore, assume that this new velocity $u^{\left(1\right)}$
will be such that $u^{\left(1\right)}\cdot\nabla\omega^{\left(1\right)}=0$
(we will check this later). If all these things are true, we obtain
the simple equation
\begin{equation}
\frac{\partial\omega^{\left(1\right)}}{\partial t}+u^{\left(0\right)}\cdot\nabla\omega^{\left(1\right)}=\frac{\partial\rho^{\left(1\right)}}{\partial x_{2}}.\label{eq:intro vorticity equation 1}
\end{equation}
Notice that the force disappears because we have used it to cancel
``negligible'' terms. This equation is telling us that the new vorticity
is transported by the old velocity and grows in $\left|\left|\cdot\right|\right|_{L^{\infty}\left(\mathbb{R}^{2}\right)}$
by the effect of $\frac{\partial\rho^{\left(1\right)}}{\partial x_{2}}$.
To continue studying this problem, we need to know the functional
form of $u^{\left(0\right)}$. Since we expect some sort of functional
similarity between the same physical quantities of different layers,
it makes sense for $\omega^{\left(0\right)}$ to be similar in form
to $\omega^{\left(1\right)}$. Now, because of equation \eqref{eq:intro vorticity equation 1},
we expect $\omega^{\left(1\right)}\sim\frac{\partial\rho^{\left(1\right)}}{\partial x_{2}}$
(here, $\sim$ means ``similar in functional dependence'', we do
not imply a similarity in a certain norm). Differentiating with respect
to $x_{2}$ in \eqref{eq:intro form rho}, we get
\begin{equation}
\frac{\partial\rho^{\left(1\right)}}{\partial x_{2}}\left(t,x\right)=A^{\left(1\right)}\left(t\right)b^{\left(1\right)}\left(t\right)\sin\left(a^{\left(1\right)}\left(t\right)\left(x_{1}-c^{\left(1\right)}\left(t\right)\right)\right)\sin\left(b^{\left(1\right)}\left(t\right)x_{2}\right).\label{eq:intro derivative density}
\end{equation}
As we want $\omega^{\left(0\right)}$ to have a spatial scale of order
$1$, it makes sense to take
\begin{equation}
\omega^{\left(0\right)}\left(t,x\right)=2\sin\left(x_{1}\right)\sin\left(x_{2}\right).\label{eq:intro omega0}
\end{equation}
(The factor $2$ is for future convenience). We are not adding any
temporal dependence to $\omega^{\left(0\right)}$, because, thanks
to the layer splitting, the time scale of the background layer should
be much larger than the time scale of the new layer. This means that,
if we wish to calculate anything related to $\omega^{\left(1\right)}$
or $\rho^{\left(1\right)}$, we can essentially treat all physical
quantities of the background layer as constant in time (not in space,
though). The vorticity given in \eqref{eq:intro omega0} is so simple
that it is easy to check that it comes from the velocity field
\[
u^{\left(0\right)}\left(t,x\right)=\left(\begin{matrix}\sin\left(x_{1}\right)\cos\left(x_{2}\right)\\
-\cos\left(x_{1}\right)\sin\left(x_{2}\right)
\end{matrix}\right).
\]
And, with these expressions for $u^{\left(0\right)}$ and $\omega^{\left(0\right)}$,
it is immediate to check that $\frac{\partial\omega^{\left(0\right)}}{\partial t}=0$
and $u^{\left(0\right)}\cdot\nabla\omega^{\left(0\right)}=0$. Thus,
equation \eqref{eq:intro assumption omega0} is indeed satisfied.
A key property of this velocity field is that it generates a hyperbolic
point for $x_{1}=\frac{\pi}{2}$ and $x_{2}=0$.

Recall that we wanted to solve equation \eqref{eq:intro vorticity equation 1}.
Now that we have an expression for $u^{\left(0\right)}$, we can actually
compute the change of variables that eliminates transport. For this,
we want to solve the ODE
\begin{equation}
\frac{\partial\Phi^{\left(1\right)}}{\partial t}\left(t,x\right)=u^{\left(0\right)}\left(t,\Phi^{\left(1\right)}\left(t,x\right)\right)=\left(\begin{matrix}\sin\left(\Phi_{1}^{\left(1\right)}\left(t,x\right)\right)\cos\left(\Phi_{2}^{\left(1\right)}\left(t,x\right)\right)\\
-\cos\left(\Phi_{1}^{\left(1\right)}\left(t,x\right)\right)\sin\left(\Phi_{2}^{\left(1\right)}\left(t,x\right)\right)
\end{matrix}\right),\quad\Phi^{\left(1\right)}\left(0,x\right)=x.\label{eq:intro ODE transport}
\end{equation}
Let us focus first on calculating the particle trajectories for points
in the $x_{1}$ axis, i.e., assume $x_{2}=0$. Then, the ODE above
actually guarantees that $\Phi_{2}^{\left(1\right)}\left(t,x\right)=0$
for all time. Consequently, we just need to solve
\begin{equation}
\frac{\partial\Phi_{1}^{\left(1\right)}}{\partial t}\left(t,\left(x_{1},0\right)\right)=\sin\left(\Phi_{1}^{\left(1\right)}\left(t,\left(x_{1},0\right)\right)\right).\label{eq:intro ODE transport easy}
\end{equation}
This is the ODE of an inverted degenerate half-pendulum (see subsection
\ref{subsec:pendula} for more details) and can be integrated explicitly
(see Lemma \ref{lem:the good ODE}) to obtain
\[
\sin\left(\Phi_{1}^{\left(1\right)}\left(t,\left(x_{1},0\right)\right)\right)=\frac{1}{\cosh\left(t_{\max}^{\left(1\right)}\left(x_{1}\right)-t\right)},\quad t_{\max}^{\left(1\right)}\left(x_{1}\right)=\ln\left(\frac{1+\cos\left(x_{1}\right)}{\sin\left(x_{1}\right)}\right).
\]
With this done, our next task is to find the particle trajectories
outside of the $x_{1}$ axis. Nonetheless, since we expect the spatial
scale of our new layer to be very small in comparison to our background
layer, it makes sense to settle for a first order spatial Taylor expansion
of ODE \eqref{eq:intro ODE transport}. We will do this expansion
using the point $\left(c\left(0\right),0\right)$ as center. In this
manner, we obtain
\begin{equation}
\Phi^{\left(1\right)}\left(t,x\right)=\left(\begin{matrix}\Phi_{1}^{\left(1\right)}\left(t,\left(c^{\left(1\right)}\left(0\right),0\right)\right)+\frac{\cosh\left(t_{\max}^{\left(1\right)}\left(c\left(0\right)\right)\right)}{\cosh\left(t_{\max}^{\left(1\right)}\left(c\left(0\right)\right)-t\right)}x_{1}\\
\left(\frac{\cosh\left(t_{\max}^{\left(1\right)}\left(c\left(0\right)\right)\right)}{\cosh\left(t_{\max}^{\left(1\right)}\left(c\left(0\right)\right)-t\right)}\right)^{-1}x_{2}
\end{matrix}\right)+O\left(\left|\left|\left(x_{1}-c^{\left(1\right)}\left(0\right),x_{2}\right)\right|\right|_{2}^{2}\right),\label{eq:intro expansion transport}
\end{equation}
where $\left|\left|\cdot\right|\right|_{2}$ should be understood
as the Euclidean norm in $\mathbb{R}^{2}$. In what follows, we will
denote $t_{\max}^{\left(1\right)}\left(c\left(0\right)\right)\equiv t_{\max}^{\left(1\right)}$
for ease of notation.

Next, recall that the equation for the density is (see equation \eqref{eq:Boussinesq system vorticity formulation})
\begin{equation}
\frac{\partial\rho}{\partial t}+u\cdot\nabla\rho=f_{\rho}.\label{eq:intro equation density}
\end{equation}
Making our two layers explicit, as $\rho^{\left(0\right)}=0$, we
have
\[
\frac{\partial\rho^{\left(1\right)}}{\partial t}+u^{\left(0\right)}\cdot\nabla\rho^{\left(1\right)}+u^{\left(1\right)}\cdot\nabla\rho^{\left(1\right)}=f_{\rho}.
\]
Again, suppose that $u^{\left(1\right)}$ is so small that the term
$u^{\left(1\right)}\cdot\nabla\rho^{\left(1\right)}$ can be neglected,
i.e., compensated for with the force. Then, we basically arrive to
the transport equation
\begin{equation}
\frac{\partial\rho^{\left(1\right)}}{\partial t}+u^{\left(0\right)}\cdot\nabla\rho^{\left(1\right)}=0,\label{eq:intro equation density simp}
\end{equation}
were the force has disappeared because we have used it to compensate
$u^{\left(1\right)}\cdot\nabla\rho^{\left(1\right)}$. Equation \eqref{eq:intro equation density simp}
means that our density is just transported by the old velocity. With
this in mind, we will take $a^{\left(1\right)}\left(t\right)$, $b^{\left(1\right)}\left(t\right)$
and $c^{\left(1\right)}\left(t\right)$ in equation \eqref{eq:intro form rho}
so that \eqref{eq:intro equation density simp} is satisfied at first
order in the spatial expansion of $\Phi^{\left(1\right)}\left(t,x\right)$
without the need of the force. Moreover, we will choose $a^{\left(1\right)}\left(0\right)=b^{\left(1\right)}\left(0\right)$.
Hence, we take
\[
c^{\left(1\right)}\left(t\right)\coloneqq\Phi_{1}^{\left(1\right)}\left(t,\left(c^{\left(1\right)}\left(0\right),0\right)\right),\quad a^{\left(1\right)}\left(t\right)\coloneqq b^{\left(1\right)}\left(0\right)\left(\frac{\cosh\left(t_{\max}^{\left(1\right)}\right)}{\cosh\left(t_{\max}^{\left(1\right)}-t\right)}\right)^{-1},\quad b^{\left(1\right)}\left(t\right)\coloneqq b^{\left(1\right)}\left(0\right)\frac{\cosh\left(t_{\max}^{\left(1\right)}\right)}{\cosh\left(t_{\max}^{\left(1\right)}-t\right)}.
\]
(Bear in mind that we actually have to take $a^{\left(1\right)}\left(t\right)$,
$b^{\left(1\right)}\left(t\right)$ and $c^{\left(1\right)}\left(t\right)$
so that they match the \textbf{inverse} of the expansion obtained
in \eqref{eq:intro expansion transport}). Recall that $a^{\left(1\right)}\left(t\right)$
and $b^{\left(1\right)}\left(t\right)$ are related to the deformation
of the new layer, while $\left(c^{\left(1\right)}\left(t\right),0\right)$
represents the position of its center. By the equation above and equation
\eqref{eq:intro ODE transport easy}, we see that $c^{\left(1\right)}\left(t\right)$
is a solution of the ODE
\begin{equation}
\frac{\mathrm{d}c^{\left(1\right)}}{\mathrm{d}t}\left(t\right)=\sin\left(c^{\left(1\right)}\left(t\right)\right),\label{eq:intro ODE center}
\end{equation}
which we shall study further in subsection \ref{subsec:pendula} and
in Lemma \ref{lem:the good ODE}. Moreover, notice that $a^{\left(1\right)}\left(t\right)b^{\left(1\right)}\left(t\right)=\left[b^{\left(1\right)}\left(0\right)\right]^{2}$
is constant in time, as expected by the incompressibility of the fluid.
We will use the force $f_{\rho}$ to compensate the higher order terms
of the Taylor expansion. With these choices of $a^{\left(1\right)}\left(t\right)$,
$b^{\left(1\right)}\left(t\right)$, $c^{\left(1\right)}\left(t\right)$
and of the external force $f_{\rho}$, \eqref{eq:intro form rho}
actually becomes an exact solution of \eqref{eq:intro equation density}.
Consequently, equation \eqref{eq:intro derivative density} is true
and becomes
\[
\begin{aligned}\frac{\partial\rho^{\left(1\right)}}{\partial x_{2}}\left(t,x\right) & =b^{\left(1\right)}\left(0\right)\frac{\cosh\left(t_{\max}^{\left(1\right)}\right)}{\cosh\left(t_{\max}^{\left(1\right)}-t\right)}A^{\left(1\right)}\left(t\right)\sin\left(b^{\left(1\right)}\left(0\right)\frac{\cosh\left(t_{\max}^{\left(1\right)}\right)}{\cosh\left(t_{\max}^{\left(1\right)}-t\right)}x_{2}\right)\cdot\\
 & \quad\cdot\sin\left(b^{\left(1\right)}\left(0\right)\left(\frac{\cosh\left(t_{\max}^{\left(1\right)}\right)}{\cosh\left(t_{\max}^{\left(1\right)}-t\right)}\right)^{-1}\left(x_{1}-c^{\left(1\right)}\left(t\right)\right)\right).
\end{aligned}
\]
Using the first order expansion of the particle trajectories \eqref{eq:intro expansion transport},
we obtain that
\begin{equation}
\frac{\partial\rho^{\left(1\right)}}{\partial x_{2}}\left(t,\Phi^{\left(1\right)}\left(t,x\right)\right)=b^{\left(1\right)}\left(0\right)\frac{\cosh\left(t_{\max}^{\left(1\right)}\right)}{\cosh\left(t_{\max}^{\left(1\right)}-t\right)}A^{\left(1\right)}\left(t\right)\sin\left(x_{1}\right)\sin\left(x_{2}\right)+O\left(\left|\left|\left(x_{1}-c^{\left(1\right)}\left(0\right),x_{2}\right)\right|\right|_{2}^{2}\right).\label{eq:intro derivative rho in transport}
\end{equation}
On the other hand, we can express equation \eqref{eq:intro vorticity equation 1}
under this same change of variables, which eliminates the transport
term and leads to
\[
\frac{\partial}{\partial t}\left(\omega^{\left(1\right)}\left(t,\Phi^{\left(1\right)}\left(t,x\right)\right)\right)=\frac{\partial\rho^{\left(1\right)}}{\partial x_{2}}\left(t,\Phi^{\left(1\right)}\left(t,x\right)\right)+O\left(\left|\left|\left(x_{1}-c^{\left(1\right)}\left(0\right),x_{2}\right)\right|\right|_{2}^{2}\right).
\]
Again, we will use the force to compensate the higher order terms,
which means we can basically focus on solving
\[
\frac{\partial}{\partial t}\left(\omega^{\left(1\right)}\left(t,\Phi^{\left(1\right)}\left(t,x\right)\right)\right)=\frac{\partial\rho^{\left(1\right)}}{\partial x_{2}}\left(t,\Phi^{\left(1\right)}\left(t,x\right)\right),
\]
which is easy. Using the expression given in \eqref{eq:intro derivative rho in transport},
we obtain
\begin{equation}
\omega^{\left(1\right)}\left(t,\Phi^{\left(1\right)}\left(t,x\right)\right)-\omega^{\left(1\right)}\left(0,x\right)=b^{\left(1\right)}\left(0\right)\left(\int_{0}^{t}\frac{\cosh\left(t_{\max}^{\left(1\right)}\right)}{\cosh\left(t_{\max}^{\left(1\right)}-s\right)}A^{\left(1\right)}\left(s\right)\mathrm{d}s\right)\sin\left(x_{1}\right)\sin\left(x_{2}\right).\label{eq:intro expression omega(1)}
\end{equation}
The first key property to notice here is that, if we take $A^{\left(1\right)}\left(t\right)\ge0$,
the growth that $\omega^{\left(1\right)}$ experiences does not depend
directly on $A^{\left(1\right)}\left(t\right)$, only on its integral.
This means that we can start from $A^{\left(1\right)}\left(0\right)=0$,
make $A^{\left(1\right)}\left(t\right)$ grow, which will make the
amplitude of the vorticity grow as well, and then make $A^{\left(1\right)}\left(t\right)=0$
again and, although this will stop the growth of the amplitude of
$\omega^{\left(1\right)}$, it will never make it shrink. In order
words, we can start from zero density and zero vorticity, then introduce
some density, which will make the vorticity grow, and then, when we
try to go back to the initial state in density (zero density), we
will see that the vorticity does not return to its initial state (zero
vorticity); it sort of ``remembers'' what has happened. This is
the hysteresis effect we were talking about before.

If we undo the change of variables in \eqref{eq:intro expression omega(1)}
we arrive to
\begin{equation}
\begin{aligned}\omega^{\left(1\right)}\left(t,x\right)-\omega^{\left(1\right)}\left(0,x\right) & =\overbrace{b^{\left(1\right)}\left(0\right)\left(\int_{0}^{t}\frac{\cosh\left(t_{\max}^{\left(1\right)}\right)}{\cosh\left(t_{\max}^{\left(1\right)}-s\right)}A^{\left(1\right)}\left(s\right)\mathrm{d}s\right)}^{\text{growth in amplitude}}\sin\left(b^{\left(1\right)}\left(0\right)\overbrace{\frac{\cosh\left(t_{\max}^{\left(1\right)}\right)}{\cosh\left(t_{\max}^{\left(1\right)}-t\right)}}^{\text{deformation}}x_{2}\right)\cdot\\
 & \quad\cdot\sin\left(b^{\left(1\right)}\left(0\right)\overbrace{\left(\frac{\cosh\left(t_{\max}^{\left(1\right)}\right)}{\cosh\left(t_{\max}^{\left(1\right)}-t\right)}\right)^{-1}}^{\text{deformation}}\left(x_{1}-\overbrace{c^{\left(1\right)}\left(t\right)}^{\text{displacement}}\right)\right).
\end{aligned}
\label{eq:intro omega 1 true scale}
\end{equation}
Hence, in expression \eqref{eq:intro omega 1 true scale}, we can
see the three effects that time evolution has on the vorticity of
the new layer, which we mentioned in the informal description at the
beginning of subsection \ref{subsec:ideas behind blow-up}: the growth
of the amplitude, the deformation and the displacement in the $x_{1}$
direction (given by $c^{\left(1\right)}\left(t\right)$). Moreover,
we see that, making $b^{\left(1\right)}\left(0\right)$ very big,
we can make the spatial scale of $\omega^{\left(1\right)}\left(t,x\right)$
very small compared to the spatial scale of the background $\omega^{\left(0\right)}$,
as we wanted. Besides, it is easy to see that this vorticity comes
from the velocity field
\begin{equation}
\begin{aligned}u^{\left(1\right)}\left(t,x\right)-u^{\left(1\right)}\left(0,x\right) & =-b^{\left(1\right)}\left(0\right)\left(\int_{0}^{t}\frac{\cosh\left(t_{\max}^{\left(1\right)}\right)}{\cosh\left(t_{\max}^{\left(1\right)}-s\right)}A^{\left(1\right)}\left(s\right)\mathrm{d}s\right)\cdot\\
 & \quad\cdot\left(\begin{matrix}\frac{1}{b^{\left(1\right)}\left(0\right)\left(\frac{\cosh\left(t_{\max}^{\left(1\right)}\right)}{\cosh\left(t_{\max}^{\left(1\right)}-t\right)}\right)}\cos\left(b^{\left(1\right)}\left(0\right)\frac{\cosh\left(t_{\max}^{\left(1\right)}\right)}{\cosh\left(t_{\max}^{\left(1\right)}-t\right)}x_{2}\right)\cdot\\
\cdot\sin\left(b^{\left(1\right)}\left(0\right)\left(\frac{\cosh\left(t_{\max}^{\left(1\right)}\right)}{\cosh\left(t_{\max}^{\left(1\right)}-t\right)}\right)^{-1}\left(x_{1}-c^{\left(1\right)}\left(t\right)\right)\right)\\
\\-\frac{1}{b^{\left(1\right)}\left(0\right)\left(\frac{\cosh\left(t_{\max}^{\left(1\right)}\right)}{\cosh\left(t_{\max}^{\left(1\right)}-t\right)}\right)^{-1}}\sin\left(b^{\left(1\right)}\left(0\right)\frac{\cosh\left(t_{\max}^{\left(1\right)}\right)}{\cosh\left(t_{\max}^{\left(1\right)}-t\right)}x_{2}\right)\cdot\\
\cdot\cos\left(b^{\left(1\right)}\left(0\right)\left(\frac{\cosh\left(t_{\max}^{\left(1\right)}\right)}{\cosh\left(t_{\max}^{\left(1\right)}-t\right)}\right)^{-1}\left(x_{1}-c^{\left(1\right)}\left(t\right)\right)\right)
\end{matrix}\right).
\end{aligned}
\label{eq:intro u1 true scale}
\end{equation}
From expressions \eqref{eq:intro omega 1 true scale} and \eqref{eq:intro u1 true scale},
one can check that $u^{\left(1\right)}\cdot\nabla\omega^{\left(1\right)}=0$,
as we had assumed. Lastly, if we suppose $u^{\left(1\right)}\left(0,x\right)=0$
and we evaluate the velocity in the first order approximation of the
particle trajectories, we get
\begin{equation}
\begin{aligned}u^{\left(1\right)}\left(t,\Phi^{\left(1\right)}\left(t,x\right)\right) & =-b^{\left(1\right)}\left(0\right)\left(\int_{0}^{t}\frac{\cosh\left(t_{\max}^{\left(1\right)}\right)}{\cosh\left(t_{\max}^{\left(1\right)}-s\right)}A^{\left(1\right)}\left(s\right)\mathrm{d}s\right)\left(\begin{matrix}\frac{1}{b^{\left(1\right)}\left(0\right)\left(\frac{\cosh\left(t_{\max}^{\left(1\right)}\right)}{\cosh\left(t_{\max}^{\left(1\right)}-t\right)}\right)}\sin\left(x_{1}\right)\cos\left(x_{2}\right)\\
-\frac{1}{b^{\left(1\right)}\left(0\right)\left(\frac{\cosh\left(t_{\max}^{\left(1\right)}\right)}{\cosh\left(t_{\max}^{\left(1\right)}-t\right)}\right)^{-1}}\cos\left(x_{1}\right)\sin\left(x_{2}\right)
\end{matrix}\right).\end{aligned}
\label{eq:intro u1 with transport}
\end{equation}

Expression \eqref{eq:intro u1 with transport} will come in handy
as reference when we start with the rigorous construction. Contrary
to what we have done here, where we have started from the density
and deduced everything from there, we will depart from the velocity
field or, to be more precise, from the stream function. This will
allow us to avoid working with the singular integral operator through
which one obtains the velocity from the vorticity, making the construction
completely explicit in the spatial variables. Furthermore, it will
make it straightforward to ensure that all physical quantities have
compact support. In spite of this, the mechanism of growth for the
vorticity, although somewhat hidden, will be exactly the same.

\subsubsection{\label{subsec:pendula}Where are the pendula?}

We depart from equation \eqref{eq:intro ODE center}:
\begin{equation}
\dot{z}=\sin\left(z\right),\quad z\left(t\right)=c^{\left(1\right)}\left(t\right),\label{eq:degenerate half-pendulum}
\end{equation}
where $c^{\left(1\right)}\left(t\right)$ was the displacement in
the $x_{1}$ axis of the center of the new layer with respect to the
background layer. If we differentiate again, employing the double
angle formula, we obtain
\[
\ddot{z}=\cos\left(z\right)\dot{z}=\sin\left(z\right)\cos\left(z\right)=\frac{1}{2}\sin\left(2z\right).
\]
Doing the change of variables $y=2z$, we get
\[
\ddot{y}=2\ddot{z}=\sin\left(2z\right)=\sin\left(y\right)\iff\ddot{y}-\sin\left(y\right)=0.
\]
In other words, $y\left(t\right)=2z\left(t\right)=2c^{\left(1\right)}\left(t\right)$
acts like a pendulum of mass $1$, length $1$ and gravity $-1$,
i.e., an inverted pendulum.

Actually, if we consider $x=y+\pi$, we can write
\[
\ddot{x}=\ddot{y}=\sin\left(y\right)=-\sin\left(x\right)\iff\ddot{x}+\sin\left(x\right)=0.
\]
Thereby, $x\left(t\right)=2z\left(t\right)+\pi=2c^{\left(1\right)}\left(t\right)+\pi$
behaves like a real pendulum of mass $1$, length $1$ and gravity
$1$. This is why we say that $z\left(t\right)=c^{\left(1\right)}\left(t\right)$,
the displacement in the $x_{1}$ axis of the center of a layer with
respect to the preceding layer, is an inverted half-pendulum: the
$+\pi$ factor is responsible for the ``inverted'' part and the
$2$ explains the ``half''. 

But, where does the ``degenerate'' part come from? To answer this
question, we need to consider the mechanical energy of the pendulum
$x\left(t\right)$. This total energy is a preserved quantity and
is given by the sum of the kinetic and potential energies, i.e., 
\begin{equation}
E=T\left(x\right)+V\left(x\right)=\frac{1}{2}\dot{x}^{2}-\cos\left(x\right).\label{eq:total energy pendulum}
\end{equation}
Writing $x\left(t\right)=2z\left(t\right)+\pi$, we arrive to
\[
E=2\dot{z}^{2}+\cos\left(2z\right).
\]
Employing equation \eqref{eq:degenerate half-pendulum} and using
the double angle formula, we infer that
\[
E=2\sin^{2}\left(z\right)+\left[\cos^{2}\left(z\right)-\sin^{2}\left(z\right)\right]=1.
\]
In other words, the pendulum $x\left(t\right)$ has mechanical energy
$1$. This means that it reaches the unstable position $x=\pi$ with
zero kinetic energy (since $-\cos\left(\pi\right)=1$ and \eqref{eq:total energy pendulum}
must be satisfied). Moreover, the pendulum needs infinite time to
reach this unstable position. This is why we call this pendulum ``degenerate''.

Consequently, we say that the displacement of the center of a layer
with respect to the preceding layer,
\[
c^{\left(1\right)}\left(t\right)=z\left(t\right)=\overbrace{\frac{1}{2}}^{\text{halves}}\left(\overbrace{x\left(t\right)}^{\text{real pendulum}}\overbrace{-\pi}^{\text{inverts}}\right),
\]
is a degenerate (reaches the unstable position with zero kinetic energy
and in infinite time) inverted (gravity goes in the opposite direction,
as indicated by the $-\pi$ factor) half (the $\frac{1}{2}$ factor)
pendulum.

\subsubsection{\label{subsec:basic idea of the construction}How does the construction
actually look like?}

In order to formalize the ideas presented at the beginning of subsection
\ref{subsec:ideas behind blow-up}, we will actually follow a different
approach than the one presented above. Instead of starting from the
density and computing the associated growth in the vorticity, we shall
prescribe a particular stream function and choose the density afterwards
so that the vorticity displays the desired hysteresis effect. This
will greatly simplify the computations and the burdens due to adding
the cutoff functions. In this manner, the proof can be summed up as
the following set of instructions:
\begin{enumerate}
\item \label{enu:stream function}Take as stream function
\[
\begin{aligned}\psi\left(t,x\right) & =\sum_{n\in\mathbb{N}}\psi^{\left(n\right)}\left(t,x\right),\\
\psi^{\left(n\right)}\left(t,x\right) & =B_{n}\left(t\right)\varphi\left(\lambda_{n}a_{n}\left(t\right)\left(x_{1}-c_{1}^{\left(n\right)}\left(t\right)\right)\right)\varphi\left(\lambda_{n}b_{n}\left(t\right)\left(x_{2}-c_{2}^{\left(n\right)}\left(t\right)\right)\right)\cdot\\
 & \quad\cdot\sin\left(a_{n}\left(t\right)\left(x_{1}-c_{1}^{\left(n\right)}\left(t\right)\right)\right)\sin\left(b_{n}\left(t\right)\left(x_{2}-c_{2}^{\left(n\right)}\left(t\right)\right)\right),
\end{aligned}
\]
where $B_{n}\left(t\right)\equiv0$ $\forall t\in\left[0,t_{n}\right]$,
$t_{n}\in\left[0,1\right]$ being the moment in time when we introduce
the $n$-th layer, and $\varphi\in C_{c}^{\infty}\left(\mathbb{R}\right)$
with $\left.\varphi\right|_{\left[-8\pi,8\pi\right]}\equiv1$ and
$\text{supp}\varphi\subseteq\left[-16\pi,16\pi\right]$. The function
$\varphi$ is there to assure that our solution has compact support.
It is easy to check that this stream function generates a velocity
which is compatible with the one given in equation \eqref{eq:intro u1 true scale}
(for this, ignore all terms that have non-zero order in the parameter
$\lambda_{n}$).
\item \label{enu:force leading order}This stream function will \textbf{not}
be a solution of the unforced 2D Euler equation. Consequently, take
advantage of the force in the vorticity equation $f_{\omega}$ and
the term $\frac{\partial\rho}{\partial x_{2}}$ appropriately to compensate
the errors. In other words, choose the density $\rho$ so that $\frac{\partial\rho}{\partial x_{2}}$
balances the leading order term of the error and use the force $f_{\omega}$
for the other terms. This step is key, because, although $\limsup_{t\to1}\left|\left|\frac{\partial\rho}{\partial x_{2}}\left(t,\cdot\right)\right|\right|_{L^{\infty}\left(\mathbb{R}^{2}\right)}$
will blow-up, $\limsup_{t\to1}\left|\left|f_{\omega}\left(t,\cdot\right)\right|\right|_{L^{\infty}\left(\mathbb{R}^{2}\right)}$
will not.
\item \label{enu:force non-leading order}Use the force in the density equation
$f_{\rho}$ to cancel out the errors due to the choice of the density
in point 2.
\item \label{enu:optimize parameters}Optimize the parameters $B_{n}\left(t\right)$,
$a_{n}\left(t\right)$, $b_{n}\left(t\right)$, $c_{1}^{\left(n\right)}\left(t\right)$,
$c_{2}^{\left(n\right)}\left(t\right)$ and $\lambda_{n}$ to make
the forces $f_{\rho}$ and $f_{\omega}$ as regular as possible.
\end{enumerate}

\subsection{\label{subsec:differences and similarities}Differences and similarities
with other blow-ups}

We will start with a comparison with Elgindi's and Pasqualotto's construction
(see \cite{Elgindi Boussinesq I,Elgindi Boussinesq II}), since it
is the known blow-up of 2D Boussinesq that lies closest to the one
we present in this paper. In both cases, the domain considered is
the whole plane $\mathbb{R}^{2}$ and the solutions that arise have
compact support. Here is where the similarities end. While Elgindi's
and Pasqualotto's result addresses the unforced 2D Boussinesq system,
the external force is a required component of our solution. Besides,
Elgindi's and Pasqualotto's solution is asymptotically self-similar
while ours is not. Moreover, in Elgindi's and Pasqualotto's result,
the limiting regularity appears in the initial condition, i.e., in
the solution itself, and is $C^{1,\alpha}\left(\mathbb{R}^{2}\right)\cap C^{\infty}\left(\mathbb{R}^{2}\setminus\left\{ 0\right\} \right)$
with very small $\alpha$ ($0<\alpha\ll1$), whereas, in our case,
the solution itself is $C^{\infty}\left(\mathbb{R}^{2}\right)$ except
at the blow-up time and the limiting regularity comes forth in the
forces $f_{\rho}$ and $f_{u}$ (only at the blow-up time) and is
$C^{1,\alpha}\left(\mathbb{R}^{2}\right)\cap C^{\infty}\left(\mathbb{R}^{2}\setminus\left\{ \text{one point}\right\} \right)$
with ``non-small'' $\alpha$ constrained by $0<\alpha<\sqrt{\frac{4}{3}}-1$.
Another key difference is that Elgindi's and Pasqualotto's initial
condition satisfies $\rho_{0},u_{0}\in H^{2+\delta}\left(\mathbb{R}^{2}\right)$
for some small $\delta$ and, consequently, enters in the well-posedness
regime in Sobolev spaces of the 2D Boussinesq equations, while in
contrast we think that our forces $f_{\rho}$ and $f_{u}$ fail to
belong to any subcritical Sobolev spaces and, thus, our solution does
not lie in the well-posedness regime in $L^{2}\left(\mathbb{R}^{2}\right)$-based
Sobolev spaces. Lastly, another dissimilarity is that Elgindi and
Pasqualotto needed to execute a computer-assisted proof for a certain
part of their argument, whereas our proof is conducted with pen and
paper.

Now, we will put our result alongside its ``siblings'' \cite{Fuerza Euler}
and \cite{Fuerza IPM}, the finite-time singularities with external
forcing for incompressible 3D Euler and IPM. In all three results,
a series of layers with decreasing support is used to generate the
growth. These layers are introduced sequentially in time and the increase
of the amplitude of a layer is mainly caused by its interaction with
the layer that immediately precedes it. Similarly, the support of
each layer is always contained in the support of the preceding layer.
Moreover, in all three cases, the functional form of the fluid variables
of each layer is chosen so that the pure quadratic terms (i.e., the
self-interactions) vanish at zeroth order. Furthermore, the regularity
loss suffered by the solutions at the blow-up time is comparable in
all three situations. The methods employed in Euler and Boussinesq
both rely on the shaping of a hyperbolic point to foster the growth
in each layer, whereas no hyperbolic points are used in IPM. Because
of this fact, the deformation experienced by the layers is substantially
more pronounced in Euler and Boussinesq than in IPM. Nevertheless,
a key difference between Boussinesq and Euler is that, in Euler, the
deformation only happens one-way; while, in Boussinesq, it takes place
both-ways. In other words, in Boussinesq, our layers start with a
certain shape that gets deformed, but, later, they recover the shape
they had at the beginning. Meanwhile, in Euler, the deformation only
increases with time. 

Let us continue mentioning features of the construction of Boussinesq
for which no analogies in Euler and IPM can be found. Boussinesq is
the only case where the centers of the layers experience translation.
It should be noted that this is no ``cosmetic'' effect, i.e., the
displacement of the layers is required for the blow-up to occur. Moreover,
the fact that the density and the velocity are independent variables
in Boussinesq leads to both unknowns having dissimilar behaviors.
This does not happen in IPM because there is only one scalar unknown
and, although the behavior of a velocity layer in Euler does depend
on the direction considered, the ``cyclic rotation'' of the construction
means that all three axes display the same behavior globally. This
clear distinction between velocity and density in Boussinesq can be
clearly seen in the fact that the amplitude of a density layer increases
and later decreases, while the amplitude of the associated vorticity
layer only increases. Another example is the fact that only one density
layer is active at any given time, whereas all already introduced
vorticity layers are active simultaneously. An interesting consequence
of this first fact is that $\left|\left|\frac{\partial\rho}{\partial x_{2}}\left(t,\cdot\right)\right|\right|_{L^{\infty}\left(\mathbb{R}^{2}\right)}$,
whose time integral blows-up at $t=1$, vanishes at a certain sequence
of times that accumulates at $t=1$ (as explained in remark \ref{rem:blow-up rate}).
Of course, these three features have no analogues in Euler or IPM.
In addition, the blow-up rate of Boussinesq, even though it is not
well-defined, can be said to be power-like in time at most, while
the blow-ups of Euler and IPM have a faster-than-power-like blow-up
rate. Lastly, in Boussinesq, there is an explicit ideal model (where
the inverted degenerate half-pendula play a vital role) that describes
the displacement and deformation of each layer, whereas analogous
models are not provided for Euler or IPM.

\subsection{Section overview}

As easy and straightforward the steps presented in subsection \ref{subsec:basic idea of the construction}
may appear, the real construction is not as simple. Actually, the
steps \ref{enu:stream function} to \ref{enu:optimize parameters}
will not be done sequentially but, actually, in parallel. In section
\ref{sec:construction}, we will first reformulate the Boussinesq
system \eqref{eq:Boussinesq system vorticity formulation} to obtain
a system of equations for each layer. Then, we will pick up the thread
of point \ref{enu:stream function} of subsection \ref{subsec:basic idea of the construction},
bearing in mind what we have learned in subsection \ref{subsec:vorticity growth mechanism}
and introducing the stream function of the solution. Afterwards, we
will jump to point \ref{enu:optimize parameters}, slowly finding
properties of the parameters $B_{n}\left(t\right)$, $a_{n}\left(t\right)$,
$b_{n}\left(t\right)$, $c_{1}^{\left(n\right)}\left(t\right)$, $c_{2}^{\left(n\right)}\left(t\right)$
and $\lambda_{n}$ and of the change of variables introduced in subsection
\ref{subsec:vorticity growth mechanism}. We will start by obtaining
the ODEs that govern the temporal evolution of the parameters, but
we will find that these are too complicated to treat directly, which
will force us to construct ideal models, which are approximations
of the real ODEs. Asking for closeness between the solutions of the
ideal models and the solutions of the real ODEs will give us some
hints about how some parameters should be selected. Moreover, the
inverted degenerate half-pendula will have a prominent presence in
these ideal models. 

Somewhere in the middle of section \ref{sec:construction}, we will
choose the form of the density, accomplishing points \ref{enu:force leading order}
and \ref{enu:force non-leading order}. Throughout the section, to
obtain some guidance in choosing the parameters, we shall constantly
present properties we wish our solution to have and, from those properties,
we will deduce restrictions on the value of the parameters. At the
end of the section, the value of some parameters will still be unknown
but the main ideas of the mechanism will be completely determined.
At this point, we will still have no idea about the maximum regularity
of the forces.

In section \ref{sec:time evolution}, we will rigorously prove that
the ideal models presented in section \ref{sec:construction} are
close to their real counterparts. We will need this to be able to
use the properties of the ideal models when bounding the forces of
the density and vorticity equations. Computing these bounds will be
the tasks tackled in sections \ref{sec:bounds for density force}
and \ref{sec:bounds for vorticity force}. With the bounds of the
forces available, in section \ref{sec:final parameter optimization and closing arguments},
we will finally be able to compute the maximum regularity of the forces
and to obtain the necessary values of the remaining parameters for
that maximum regularity to be achieved, finishing point \ref{enu:optimize parameters}.

\subsection{Functional spaces and notation}

Let $p\in\left[1,\infty\right]$, $d\in\mathbb{N}$ and $w\in\mathbb{R}^{d}$.
We will use the notation $\left|\left|w\right|\right|_{p}$ to refer
to the $p$-norm of a vector in the Euclidean space $\mathbb{R}^{d}$,
i.e., 
\[
\left|\left|w\right|\right|_{p}\coloneqq\left(\sum_{n=1}^{d}\left|x_{n}\right|^{p}\right)^{\frac{1}{p}}.
\]
There will be no confusion with norms in the Lebesgue spaces $L^{p}\left(\mathbb{R}^{d}\right)$,
because we will write these as $\left|\left|\cdot\right|\right|_{L^{p}\left(\mathbb{R}^{d}\right)}$.

Let $d\in\mathbb{N}$ and $F\subseteq\mathbb{R}^{d}$ be closed in
$\mathbb{R}^{d}$. We recall that a function $f:F\to\mathbb{R}$ is
said to be Hölder of parameter $\alpha\in\left(0,1\right)$ (denoted
$f\in C^{\alpha}\left(F\right)$) if
\[
\left|\left|f\right|\right|_{\dot{C}^{\alpha}\left(F\right)}\coloneqq\sup_{x,y\in F}\frac{\left|f\left(x\right)-f\left(y\right)\right|}{\left|\left|x-y\right|\right|_{2}^{\alpha}}<\infty.
\]
$\left|\left|\cdot\right|\right|_{\dot{C}^{\alpha}\left(F\right)}$
is a seminorm. The set $C^{\alpha}\left(F\right)$ can be made into
a Banach space by considering the norm
\[
\left|\left|f\right|\right|_{C^{\alpha}\left(F\right)}\coloneqq\max\left\{ \left|\left|f\right|\right|_{L^{\infty}\left(F\right)},\left|\left|f\right|\right|_{\dot{C}^{\alpha}\left(F\right)}\right\} .
\]
Likewise, we will work with the Hölder spaces 
\[
C^{k,\alpha}\left(F\right)=\left\{ f:F\to\mathbb{R}\text{ of class }C^{k}\text{ and with all derivatives of order }k\text{ belonging to }C^{\alpha}\left(F\right)\right\} .
\]
We will equip this space with the norm
\[
\left|\left|f\right|\right|_{C^{k,\alpha}\left(F\right)}\coloneqq\max\left\{ \left|\left|f\right|\right|_{C^{k}\left(F\right)},\left|\left|f\right|\right|_{\dot{C}^{k,\alpha}\left(F\right)}\right\} ,
\]
where
\[
\left|\left|f\right|\right|_{C^{k}\left(F\right)}=\max_{0\le m\le k}\left|\left|f\right|\right|_{\dot{C}^{m}\left(F\right)},\quad\left|\left|f\right|\right|_{\dot{C}^{k}\left(F\right)}=\max_{\left|\beta\right|=k}\left|\left|\mathrm{D}^{\beta}f\right|\right|_{L^{\infty}\left(F\right)},
\]
\[
\left|\left|f\right|\right|_{\dot{C}^{k,\alpha}\left(F\right)}=\max_{\left|\beta\right|=k}\left|\left|\mathrm{D}^{\beta}f\right|\right|_{\dot{C}^{\alpha}\left(F\right)},\quad\mathrm{D}^{\beta}f\coloneqq\frac{\partial^{\left|\beta\right|}f}{\partial x_{1}^{\beta_{1}}\dots\partial x_{d}^{\beta_{d}}},\quad\beta\in\mathbb{N}_{0}^{\beta},\quad\left|\beta\right|=\beta_{1}+\dots+\beta_{d}.
\]

We will make use of the following two properties of the $\left|\left|\cdot\right|\right|_{\dot{C}^{\alpha}}$-seminorm:
\begin{enumerate}
\item $\forall f,g\in C^{\alpha}\left(F\right)$, we have
\begin{equation}
\left|\left|fg\right|\right|_{\dot{C}^{\alpha}\left(F\right)}\le\left|\left|f\right|\right|_{L^{\infty}\left(F\right)}\left|\left|g\right|\right|_{\dot{C}^{\alpha}\left(F\right)}+\left|\left|f\right|\right|_{\dot{C}^{\alpha}\left(F\right)}\left|\left|g\right|\right|_{L^{\infty}\left(F\right)},\label{eq:property Calpha multiplication}
\end{equation}
\item $\forall f\in C^{\alpha}\left(F\right)$ and $\forall\varphi\in C^{1}\left(F;\varphi\left(F\right)\right)$,
we have
\begin{equation}
\left|\left|f\circ\varphi\right|\right|_{\dot{C}^{\alpha}\left(F\right)}\leq K_{0}\left(d\right)\left|\left|\varphi\right|\right|_{\dot{C}^{1}\left(F;\varphi\left(F\right)\right)}^{\alpha}\left|\left|f\right|\right|_{\dot{C}^{\alpha}\left(\varphi\left(F\right)\right)}.\label{eq:property Calpha composition}
\end{equation}
\end{enumerate}
Similarly, if $k\neq0$,
\begin{enumerate}
\item $\forall f,g\in C^{k,\alpha}\left(F\right)$, we have
\begin{equation}
\left|\left|fg\right|\right|_{C^{k,\alpha}\left(F\right)}\leq K_{1}\left(d\right)\left|\left|f\right|\right|_{C^{k,\alpha}\left(F\right)}\left|\left|g\right|\right|_{C^{k,\alpha}\left(F\right)},\label{eq:property Ckalpha multiplication}
\end{equation}
\item $\forall f\in C^{k,\alpha}\left(F\right)$ and $\forall\varphi\in C^{k-1,\alpha}\left(F;\varphi\left(F\right)\right)$,
we have
\begin{equation}
\left|\left|f\circ\varphi\right|\right|_{C^{k,\alpha}\left(F\right)}\le K_{2}\left(d\right)\left(1+\max_{\left|\beta\right|=1}\left|\left|\mathrm{D}^{\beta}\varphi\right|\right|_{C^{k-1,\alpha}\left(F;\varphi\left(F\right)\right)}\right)^{k+\alpha}\left|\left|f\right|\right|_{C^{k,\alpha}\left(\varphi\left(F\right)\right)}.\label{eq:property Ckalpha composition}
\end{equation}
\end{enumerate}
Lastly, when considering vector-valued functions $f\in C^{\alpha}\left(F;\mathbb{R}^{q}\right)$
(with $q\in\mathbb{N}$), we will take
\[
\begin{aligned}\left|\left|f\right|\right|_{L^{\infty}\left(F;\mathbb{R}^{q}\right)} & =\max_{i=1,\dots,q}\left|\left|f_{i}\right|\right|_{L^{\infty}\left(F\right)},\\
\left|\left|f\right|\right|_{\dot{C}^{\alpha}\left(F;\mathbb{R}^{q}\right)} & =\max_{i=1,\dots,q}\left|\left|f_{i}\right|\right|_{\dot{C}^{\alpha}\left(F\right)},\\
\left|\left|f\right|\right|_{C^{\alpha}\left(F;\mathbb{R}^{q}\right)} & =\max_{i=1,\dots,q}\left|\left|f_{i}\right|\right|_{C^{\alpha}\left(F\right)},\\
\left|\left|f\right|\right|_{C^{k}\left(F;\mathbb{R}^{q}\right)} & =\max_{i=1,\dots,q}\left|\left|f_{i}\right|\right|_{C^{k}\left(F\right)},\\
\left|\left|f\right|\right|_{\dot{C}^{k}\left(F;\mathbb{R}^{q}\right)} & =\max_{i=1,\dots,q}\left|\left|f_{i}\right|\right|_{\dot{C}^{k}\left(F\right)},\\
\left|\left|f\right|\right|_{C^{k,\alpha}\left(F;\mathbb{R}^{q}\right)} & =\max_{i=1,\dots,q}\left|\left|f_{i}\right|\right|_{C^{k,\alpha}\left(F\right)}.
\end{aligned}
\]

\subsubsection{Marginal Hölder serminorms}

Another concept that will come in handy is that of marginal Hölder
seminorms. Let $F=F_{1}\times F_{2}$, with $F_{1},F_{2}\subseteq\mathbb{R}$
closed. Given $f:F\to\mathbb{R}$, we define
\begin{equation}
\left|\left|f\left(\cdot\right)\right|\right|_{\dot{C}_{1}^{\alpha}\left(F\right)}\coloneqq\sup_{x_{2}\in F_{2}}\left|\left|f\left(\cdot,x_{2}\right)\right|\right|_{\dot{C}^{\alpha}\left(F_{1}\right)},\quad\left|\left|f\left(\cdot\right)\right|\right|_{\dot{C}_{2}^{\alpha}\left(F\right)}\coloneqq\sup_{x_{1}\in F_{1}}\left|\left|f\left(x_{1},\cdot\right)\right|\right|_{\dot{C}^{\alpha}\left(F_{2}\right)}.\label{eq:marginal seminorms}
\end{equation}
It is clear that, by definition, these seminorms are smaller or equal
then the ``joint'' Hölder seminorm $\left|\left|f\right|\right|_{\dot{C}^{\alpha}\left(F\right)}$.
Moreover, it is possible to recover the ``joint'' Hölder seminorm
from the marginal Hölder seminorms as follows:
\begin{equation}
\begin{aligned}\left|\left|f\right|\right|_{\dot{C}^{\alpha}\left(F\right)} & =\sup_{{\footnotesize \begin{matrix}x,y\in F\\
x\neq y
\end{matrix}}}\frac{\left|f\left(x\right)-f\left(y\right)\right|}{\left|\left|x-y\right|\right|_{2}^{\alpha}}=\sup_{{\footnotesize \begin{matrix}x,y\in F\\
x\neq y
\end{matrix}}}\frac{\left|f\left(x_{1},x_{2}\right)-f\left(x_{1},y_{2}\right)+f\left(x_{1},y_{2}\right)-f\left(y_{1},y_{2}\right)\right|}{\left|\left|x-y\right|\right|_{2}^{\alpha}}\leq\\
 & \leq\sup_{{\footnotesize \begin{matrix}x,y\in F\\
x\neq y
\end{matrix}}}\frac{\left|f\left(x_{1},x_{2}\right)-f\left(x_{1},y_{2}\right)\right|}{\left|\left|x-y\right|\right|_{2}^{\alpha}}+\sup_{{\footnotesize \begin{matrix}x,y\in F\\
x\neq y
\end{matrix}}}\frac{\left|f\left(x_{1},y_{2}\right)-f\left(y_{1},y_{2}\right)\right|}{\left|\left|x-y\right|\right|_{2}^{\alpha}}\leq\\
 & \leq\sup_{{\footnotesize \begin{matrix}x,y\in F\\
x\neq y
\end{matrix}}}\left|\left|f\right|\right|_{\dot{C}_{2}^{\alpha}\left(F\right)}\underbrace{\frac{\left|x_{2}-y_{2}\right|^{\alpha}}{\left|\left|x-y\right|\right|_{2}^{\alpha}}}_{\leq1}+\sup_{{\footnotesize \begin{matrix}x,y\in F\\
x\neq y
\end{matrix}}}\left|\left|f\right|\right|_{\dot{C}_{1}^{\alpha}\left(F\right)}\underbrace{\frac{\left|x_{1}-y_{1}\right|^{\alpha}}{\left|\left|x-y\right|\right|_{2}^{\alpha}}}_{\leq1}\leq\left|\left|f\right|\right|_{\dot{C}_{2}^{\alpha}\left(F\right)}+\left|\left|f\right|\right|_{\dot{C}_{1}^{\alpha}\left(F\right)}.
\end{aligned}
\label{eq:Holder seminorm by marginals}
\end{equation}

\section{\label{sec:construction}Construction}

Instead of providing all the details of the construction up front
and proving that it leads to the main result of the paper, we will
introduce the ingredients slowly, trying to motivate every step we
do and every choice we make. Actually, we will just start from these
two ideas:
\begin{enumerate}
\item Under an appropriate change of variables, the velocity field of a
layer should look like
\[
u\sim\left(\sin\left(x_{1}\right)\cos\left(x_{2}\right),-\cos\left(x_{1}\right)\sin\left(x_{2}\right)\right).
\]
\item Under an appropriate change of variables, the density of a layer should
look like
\[
\rho\sim\sin\left(x_{1}\right)\cos\left(x_{2}\right).
\]
\end{enumerate}
These two ideas come from equations \eqref{eq:intro u1 with transport}
and \eqref{eq:intro form rho}, respectively. 

As we will be choosing the value of certain parameters as we progress
through the proof, in order to make it more clear when we are doing
so, we will use a special environment just for this purpose: the environment
``Choice''. Before we start, let us provide two pieces of advise: 
\begin{enumerate}
\item Bear in mind that, when reading a certain result, all the Choices
that appear before in the paper act as additional implicit assumptions.
\item In you wish to read the proof in detail, we strongly recommend to
print section \ref{sec:useful formulae} in paper and to have it on
hand to consult when necessary. By doing this, we believe that a lot
of computations will be easier to follow.
\end{enumerate}

\subsection{The layer equations}

We depart from Boussinesq's system, which, in vorticity formulation
(see equation \eqref{eq:Boussinesq system vorticity formulation}),
is given by
\begin{equation}
\begin{aligned}\frac{\partial\omega}{\partial t}+u\cdot\nabla\omega & =\frac{\partial\rho}{\partial x_{2}}+f_{\omega},\\
\frac{\partial\rho}{\partial t}+u\cdot\nabla\rho & =f_{\rho},
\end{aligned}
\label{eq:construction Bousinessq vorticity}
\end{equation}
where $\omega$ is the fluid vorticity, $u=\nabla^{\perp}\Delta^{-1}\omega$
is the fluid velocity, $\rho$ are the fluid density fluctuations,
$f_{\omega}$ are the external forces of the vorticity equation and
$f_{\rho}$ are the external forces of the density equation.

We shall arrange our fluid in layers (indexed by $\mathbb{N}$), where
each layer will have a much smaller size than the previous one. Every
physical quantity (velocity, vorticity, density, ...) will be assigned
a capital letter ($U$, $\Omega$, $P$, ...)\footnote{We remind the reader of the fact that $P$ is also ``capital rho''.
There will be no possibility of misunderstanding it with the pressure,
since the pressure will not appear in the construction.} and a lower case letter ($u$, $\omega$, $\rho$, ...). To denote
the contribution of the $n$-th layer to a certain physical quantity,
we will use the small case letter along with the super-index $\left(n\right)$,
whereas we will employ the capital letter with the same super-index
to refer to the value of the physical quantity when the first $n$
layers are taken into account. In this way,
\begin{itemize}
\item $U^{\left(n\right)}$ denotes the combined velocity of the first $n$
layers, while $u^{\left(n\right)}$ refers to the velocity of the
$n$-th layer;
\item $\Omega^{\left(n\right)}$ denotes the combined vorticity of the first
$n$ layers, while $\omega^{\left(n\right)}$ refers to the vorticity
of the $n$-th layer;
\item $P^{\left(n\right)}$ denotes the combined density of the first $n$
layers, while $\rho^{\left(n\right)}$ refers to the density of the
$n$-th layer;
\item $F_{\omega}^{\left(n\right)}$ denotes the combined force in the vorticity
equation when the first $n$ layers are considered, whereas $f_{\omega}^{\left(n\right)}$
refers to the force in the vorticity equation associated to the $n$-th
layer;
\item $F_{\rho}^{\left(n\right)}$ denotes the combined force in the density
equation when the first $n$ layers are considered, whereas $f_{\rho}^{\left(n\right)}$
refers to the force in the density equation associated to the $n$-th
layer.
\end{itemize}
Furthermore, clearly, we have
\[
U^{\left(n\right)}=\sum_{m=1}^{n}u^{\left(m\right)},\quad\Omega^{\left(n\right)}=\sum_{m=1}^{n}\omega^{\left(m\right)},\quad P^{\left(n\right)}=\sum_{m=1}^{n}\rho^{\left(m\right)},\quad F_{\omega}^{\left(n\right)}=\sum_{m=1}^{n}f_{\omega}^{\left(m\right)},\quad F_{\rho}^{\left(n\right)}=\sum_{m=1}^{n}f_{\rho}^{\left(m\right)}.
\]

Since we are going to work with layers, it will be helpful to rewrite
equation \eqref{eq:construction Bousinessq vorticity} in terms of
the vorticity, velocity and density of each layer. This is exactly
what we achieve in the following Proposition.
\begin{prop}[Boussinesq in layers]
\label{prop:Bousinessq in layers}With the aforementioned notation,
the following are equivalent:
\begin{enumerate}
\item $\forall n\in\mathbb{N}$, we have
\[
\begin{aligned}\frac{\partial\Omega^{\left(n\right)}}{\partial t}+U^{\left(n\right)}\cdot\nabla\Omega^{\left(n\right)} & =\frac{\partial P^{\left(n\right)}}{\partial x_{2}}+F_{\omega}^{\left(n\right)},\\
\frac{\partial P^{\left(n\right)}}{\partial t}+U^{\left(n\right)}\cdot\nabla P^{\left(n\right)} & =F_{\rho}^{\left(n\right)},
\end{aligned}
\]
\item $\forall n\in\mathbb{N}$, taking $\Omega^{\left(0\right)}=0=P^{\left(0\right)}$,
we have
\[
\begin{aligned}\frac{\partial\omega^{\left(n\right)}}{\partial t}+U^{\left(n-1\right)}\cdot\nabla\omega^{\left(n\right)}+u^{\left(n\right)}\cdot\nabla\Omega^{\left(n-1\right)}+u^{\left(n\right)}\cdot\nabla\omega^{\left(n\right)} & =\frac{\partial\rho^{\left(n\right)}}{\partial x_{2}}+f_{\omega}^{\left(n\right)},\\
\frac{\partial\rho^{\left(n\right)}}{\partial t}+U^{\left(n-1\right)}\cdot\nabla\rho^{\left(n\right)}+u^{\left(n\right)}\cdot\nabla P^{\left(n-1\right)}+u^{\left(n\right)}\cdot\nabla\rho^{\left(n\right)} & =f_{\rho}^{\left(n\right)}.
\end{aligned}
\]
\end{enumerate}
\end{prop}
\begin{rem}
This Proposition provides us with the equation that the physical quantities
of each layer must satisfy so that the total contributions of the
physical quantities when the first $n$ layers are considered be always
a solution of the Boussinesq system, independently of the value of
$n\in\mathbb{N}$. Notice that $U^{\left(n-1\right)}\cdot\nabla\omega^{\left(n\right)}$
represents the transport of the ``new'' vorticity due to the ``old''
velocity, while $u^{\left(n\right)}\cdot\nabla\Omega^{\left(n-1\right)}$
depicts the transport of the ``old'' vorticity due to the ``new''
velocity, and $u^{\left(n\right)}\cdot\nabla\omega^{\left(n\right)}$
encodes the transport of the ``new'' vorticity by the ``new''
velocity. We will call this last term ``the pure quadratic term''.
We have the same analogous behavior in the density equation.
\end{rem}
\begin{proof}
~
\begin{itemize}
\item $1\implies2$. We depart from statement 1 and we write $\Omega^{\left(n\right)}=\Omega^{\left(n-1\right)}+\omega^{\left(n\right)}$,
acting analogously with the other physical quantities. In this way,
we obtain that
\[
\begin{aligned}\frac{\partial\Omega^{\left(n-1\right)}}{\partial t}+\frac{\partial\omega^{\left(n\right)}}{\partial t}+\\
+\left(U^{\left(n-1\right)}+u^{\left(n\right)}\right)\cdot\left(\nabla\Omega^{\left(n-1\right)}+\nabla\omega^{\left(n\right)}\right) & =\frac{\partial P^{\left(n-1\right)}}{\partial x_{2}}+\frac{\partial\rho^{\left(n\right)}}{\partial x_{2}}+\\
 & \quad+F_{\omega}^{\left(n-1\right)}+f_{\omega}^{\left(n\right)}.\\
\frac{\partial P^{\left(n-1\right)}}{\partial t}+\frac{\partial\rho^{\left(n\right)}}{\partial t}+\\
+\left(U^{\left(n-1\right)}+u^{\left(n\right)}\right)\cdot\left(\nabla P^{\left(n-1\right)}+\nabla\rho^{\left(n\right)}\right) & =F_{\rho}^{\left(n-1\right)}+f_{\rho}^{\left(n\right)}.
\end{aligned}
\]
Utilizing the equations of statement 1 for $n-1$, some terms of the
equations above cancel each other, which leads to the equations of
statement 2.
\item $2\implies1$. We depart from statement 2 and we sum the equations
from $n=1$ to $m$. In other words, we consider
\[
\begin{aligned}\sum_{n=1}^{m}\left[\frac{\partial\omega^{\left(n\right)}}{\partial t}+U^{\left(n-1\right)}\cdot\nabla\omega^{\left(n\right)}+u^{\left(n\right)}\cdot\nabla\Omega^{\left(n-1\right)}+u^{\left(n\right)}\cdot\nabla\omega^{\left(n\right)}\right] & =\sum_{n=1}^{m}\left[\frac{\partial\rho^{\left(n\right)}}{\partial x_{2}}+f_{\omega}^{\left(n\right)}\right],\\
\sum_{n=1}^{m}\left[\frac{\partial\rho^{\left(n\right)}}{\partial t}+U^{\left(n-1\right)}\cdot\nabla\rho^{\left(n\right)}+u^{\left(n\right)}\cdot\nabla P^{\left(n-1\right)}+u^{\left(n\right)}\cdot\nabla\rho^{\left(n\right)}\right] & =\sum_{n=1}^{m}f_{\rho}^{\left(n\right)}.
\end{aligned}
\]
Firstly, we combine the quadratic terms $U^{\left(n-1\right)}\cdot\nabla\omega^{\left(n\right)}+u^{\left(n\right)}\cdot\nabla\omega^{\left(n\right)}=U^{\left(n\right)}\cdot\nabla\omega^{\left(n\right)}$,
which leads to 
\[
\begin{aligned}\sum_{n=1}^{m}\left[\frac{\partial\omega^{\left(n\right)}}{\partial t}+U^{\left(n\right)}\cdot\nabla\omega^{\left(n\right)}+u^{\left(n\right)}\cdot\nabla\Omega^{\left(n-1\right)}\right] & =\sum_{n=1}^{m}\left[\frac{\partial\rho^{\left(n\right)}}{\partial x_{2}}+f_{\omega}^{\left(n\right)}\right],\\
\sum_{n=1}^{m}\left[\frac{\partial\rho^{\left(n\right)}}{\partial t}+U^{\left(n\right)}\cdot\nabla\rho^{\left(n\right)}+u^{\left(n\right)}\cdot\nabla P^{\left(n-1\right)}\right] & =\sum_{n=1}^{m}f_{\rho}^{\left(n\right)}.
\end{aligned}
\]
Now, using the relation between the capital and small case letters,
it is straightforward that
\[
\begin{aligned}\frac{\partial\Omega^{\left(m\right)}}{\partial t}+\sum_{n=1}^{m}\left[U^{\left(n\right)}\cdot\nabla\omega^{\left(n\right)}+u^{\left(n\right)}\cdot\nabla\Omega^{\left(n-1\right)}\right] & =\frac{\partial P^{\left(m\right)}}{\partial x_{2}}+F_{\omega}^{\left(m\right)},\\
\frac{\partial P^{\left(m\right)}}{\partial t}+\sum_{n=1}^{m}\left[U^{\left(n\right)}\cdot\nabla\rho^{\left(n\right)}+u^{\left(n\right)}\cdot\nabla P^{\left(n-1\right)}\right] & =F_{\rho}^{\left(m\right)}.
\end{aligned}
\]
Nevertheless, at first glance, it is not clear what happens to the
quadratic terms. However, recalling that
\[
U^{\left(n\right)}=\sum_{l=1}^{n}u^{\left(l\right)},\quad\Omega^{\left(n\right)}=\sum_{l=1}^{n}\omega^{\left(l\right)},
\]
one can deduce
\[
\begin{aligned}\sum_{n=1}^{m}\left[U^{\left(n\right)}\cdot\nabla\omega^{\left(n\right)}+u^{\left(n\right)}\cdot\nabla\Omega^{\left(n-1\right)}\right] & =\sum_{n=1}^{m}\left[\left(\sum_{l=1}^{n}u^{\left(l\right)}\right)\cdot\nabla\omega^{\left(n\right)}+u^{\left(n\right)}\cdot\nabla\left(\sum_{l=1}^{n-1}\omega^{\left(l\right)}\right)\right]=\\
 & =\sum_{n=1}^{m}\left[\sum_{l=1}^{n}\left(u^{\left(l\right)}\cdot\nabla\omega^{\left(n\right)}\right)+\sum_{l=1}^{n-1}\left(u^{\left(n\right)}\cdot\nabla\omega^{\left(l\right)}\right)\right].
\end{aligned}
\]
Interchanging the order of summation in the first sum, we get
\[
\sum_{n=1}^{m}\left[U^{\left(n\right)}\cdot\nabla\omega^{\left(n\right)}+u^{\left(n\right)}\cdot\nabla\Omega^{\left(n-1\right)}\right]=\left[\sum_{l=1}^{m}\sum_{n=l}^{m}\left(u^{\left(l\right)}\cdot\nabla\omega^{\left(n\right)}\right)+\sum_{n=1}^{m}\sum_{l=1}^{n-1}\left(u^{\left(n\right)}\cdot\nabla\omega^{\left(l\right)}\right)\right].
\]
Exchanging the labels $n$ and $l$ in the second sum, we arrive to
\[
\begin{aligned}\sum_{n=1}^{m}\left[U^{\left(n\right)}\cdot\nabla\omega^{\left(n\right)}+u^{\left(n\right)}\cdot\nabla\Omega^{\left(n-1\right)}\right] & =\left[\sum_{l=1}^{m}\sum_{n=l}^{m}\left(u^{\left(l\right)}\cdot\nabla\omega^{\left(n\right)}\right)+\sum_{l=1}^{m}\sum_{n=1}^{l-1}\left(u^{\left(l\right)}\cdot\nabla\omega^{\left(n\right)}\right)\right]=\\
 & =\sum_{l=1}^{m}\sum_{n=1}^{m}\left(u^{\left(l\right)}\cdot\nabla\omega^{\left(n\right)}\right)=\sum_{l=1}^{m}\left[u^{\left(l\right)}\cdot\nabla\left(\sum_{n=1}^{m}\omega^{\left(n\right)}\right)\right]=\\
 & =\left(\sum_{l=1}^{m}u^{\left(l\right)}\right)\cdot\nabla\Omega^{\left(m\right)}=U^{\left(m\right)}\cdot\nabla\Omega^{\left(m\right)}.
\end{aligned}
\]
The same argument can be analogously applied to the density equation.
Hence, it is clear that we can prove statement 1 from statement 2.
\end{itemize}
\end{proof}

\subsection{Change of variables}

In subsection \ref{subsec:vorticity growth mechanism} we argued that
the velocity of the background layer and that of the new layer should
have a ``similar'' functional dependence. ``Similar'' meant that,
under an appropriate change of variables, they should look the same.
We will extrapolate this idea to all layers, i.e., under appropriate
changes of variables, the expression of a physical quantity should
have the same functional form across all layers. To exploit this functional
``similarity'' between layers, it will be useful to rewrite the
equations of the second statement of Proposition \ref{prop:Bousinessq in layers}
under an arbitrary change of variables $\phi^{\left(n\right)}\left(t,x\right)$.
This is exactly the purpose of the next Proposition.
\begin{prop}
\label{prop:change of variables}Let $n\in\mathbb{N}$. The layer
equations of Proposition \ref{prop:Bousinessq in layers}, under the
change of variables $x\leftarrow\phi^{\left(n\right)}\left(t,x\right)$,
become
\[
\begin{aligned} & \frac{\partial\widetilde{\omega^{\left(n\right)}}^{n}}{\partial t}\left(t,x\right)+\\
 & +\left(\widetilde{U^{\left(n-1\right)}}^{n}\left(t,x\right)-\frac{\partial\phi^{\left(n\right)}}{\partial t}\left(t,x\right)\right)\cdot\left[\nabla\left(\phi^{\left(n\right)}\right)\left(t,x\right)\right]^{-1}\cdot\nabla\widetilde{\omega^{\left(n\right)}}^{n}\left(t,x\right)+\\
 & +\widetilde{u^{\left(n\right)}}^{n}\left(t,x\right)\cdot\left[\nabla\left(\phi^{\left(n\right)}\right)\left(t,x\right)\right]^{-1}\cdot\nabla\widetilde{\Omega^{\left(n-1\right)}}^{n}\left(t,x\right)+\\
 & +\widetilde{u^{\left(n\right)}}^{n}\left(t,x\right)\cdot\left[\nabla\left(\phi^{\left(n\right)}\right)\left(t,x\right)\right]^{-1}\cdot\nabla\widetilde{\omega^{\left(n\right)}}^{n}\left(t,x\right)=\\
= & \left(\begin{matrix}0 & 1\end{matrix}\right)\cdot\left[\nabla\left(\phi^{\left(n\right)}\right)\left(t,x\right)\right]^{-1}\cdot\nabla\widetilde{\rho^{\left(n\right)}}^{n}\left(t,x\right)+\widetilde{f_{\omega}^{\left(n\right)}}^{n}\left(t,x\right),
\end{aligned}
\]
\[
\begin{aligned} & \frac{\partial\widetilde{\rho^{\left(n\right)}}^{n}}{\partial t}\left(t,x\right)+\\
 & +\left(\widetilde{U^{\left(n-1\right)}}^{n}\left(t,x\right)-\frac{\partial\phi^{\left(n\right)}}{\partial t}\left(t,x\right)\right)\cdot\left[\nabla\left(\phi^{\left(n\right)}\right)\left(t,x\right)\right]^{-1}\cdot\nabla\widetilde{\rho^{\left(n\right)}}^{n}\left(t,x\right)+\\
 & +\widetilde{u^{\left(n\right)}}^{n}\left(t,x\right)\cdot\left[\nabla\left(\phi^{\left(n\right)}\right)\left(t,x\right)\right]^{-1}\cdot\nabla\widetilde{P^{\left(n-1\right)}}^{n}\left(t,x\right)+\\
 & +\widetilde{u^{\left(n\right)}}^{n}\left(t,x\right)\cdot\left[\nabla\left(\phi^{\left(n\right)}\right)\left(t,x\right)\right]^{-1}\cdot\nabla\widetilde{\rho^{\left(n\right)}}^{n}\left(t,x\right)=\\
= & \widetilde{f_{\rho}^{\left(n\right)}}^{n}\left(t,x\right),
\end{aligned}
\]
where $\widetilde{g}^{n}\left(t,x\right)=g\left(t,\phi^{\left(n\right)}\left(t,x\right)\right)$
and $\mathrm{J}^{-1}$ denotes the inverse of the jacobian matrix.
\end{prop}
\begin{rem}
\label{rem:change of variables should be approximation of transport}As
we shall explain later, we would like the summands that appear in
these equations to be small in a certain $\left|\left|\cdot\right|\right|_{C^{k,\alpha}\left(\mathbb{R}^{2}\right)}$
norm. In both equations, there is a transport term (the one that has
$\widetilde{U^{\left(n-1\right)}}^{n}\left(t,x\right)$). By the form
of this term, it is evident that, the closer $\frac{\partial\phi^{\left(n\right)}}{\partial t}\left(t,x\right)$
is to $\widetilde{U^{\left(n-1\right)}}^{n}\left(t,x\right)$, the
smaller the transport term will be. In fact, if $\phi^{\left(n\right)}\left(t,x\right)$
were the change of variables associated to particle transport, we
would have $\frac{\partial\phi^{\left(n\right)}}{\partial t}\left(t,x\right)=U^{\left(n\right)}\left(t,\phi^{\left(n\right)}\left(t,x\right)\right)=\widetilde{U^{\left(n\right)}}^{n}\left(t,x\right)$.
This would not completely eliminate the transport term but it would
make it so that it exactly cancels the pure quadratic term. From this
fact, one can clearly see that it is of our interest that $\phi^{\left(n\right)}\left(t,x\right)$
be a certain approximation of the true particle trajectories of the
equation. 
\end{rem}
\begin{proof}
Let $g\in C^{1}\left(\left[0,1\right]\times\mathbb{R}^{2};\mathbb{R}\right)$
be a generic function. We would like to know how both the spatial
and the temporal derivatives change under the proposed change of variables.
We will begin with the spatial derivatives. The chain rule assures
us that
\begin{equation}
\begin{aligned} & \nabla\left(g\left(t,\phi^{\left(n\right)}\left(t,x\right)\right)\right)=\nabla\phi^{\left(n\right)}\left(t,x\right)\cdot\nabla g\left(t,\phi^{\left(n\right)}\left(t,x\right)\right)\iff\\
\iff & \nabla g\left(t,\phi^{\left(n\right)}\left(t,x\right)\right)=\left[\nabla\left(\phi^{\left(n\right)}\right)\left(t,x\right)\right]^{-1}\cdot\nabla\left(g\left(t,\phi^{\left(n\right)}\left(t,x\right)\right)\right)=\left[\nabla\left(\phi^{\left(n\right)}\right)\left(t,x\right)\right]^{-1}\cdot\nabla\widetilde{g}^{n}\left(t,x\right).
\end{aligned}
\label{eq:cv der esp}
\end{equation}
We continue with the temporal derivatives. The chain rules imposes
that
\[
\frac{\partial\widetilde{g}^{n}}{\partial t}\left(t,x\right)=\begin{aligned}\frac{\partial}{\partial t}\left(g\left(t,\phi^{\left(n\right)}\left(t,x\right)\right)\right) & =\frac{\partial g}{\partial t}\left(t,\phi^{\left(n\right)}\left(t,x\right)\right)+\frac{\partial\phi^{\left(n\right)}}{\partial t}\left(t,x\right)\cdot\nabla g\left(t,\phi^{\left(n\right)}\left(t,x\right)\right).\end{aligned}
\]
Making use of equation \eqref{eq:cv der esp}, it can be deduced that
\begin{equation}
\begin{aligned} & \frac{\partial\widetilde{g}^{n}}{\partial t}\left(t,x\right)=\frac{\partial g}{\partial t}\left(t,\phi^{\left(n\right)}\left(t,x\right)\right)+\frac{\partial\phi^{\left(n\right)}}{\partial t}\left(t,x\right)\cdot\left[\nabla\left(\phi^{\left(n\right)}\right)\left(t,x\right)\right]^{-1}\cdot\nabla\widetilde{g}^{n}\left(t,x\right)\iff\\
\iff & \frac{\partial g}{\partial t}\left(t,\phi^{\left(n\right)}\left(t,x\right)\right)=\frac{\partial\widetilde{g}^{n}}{\partial t}\left(t,x\right)-\frac{\partial\phi^{\left(n\right)}}{\partial t}\left(t,x\right)\cdot\left[\nabla\left(\phi^{\left(n\right)}\right)\left(t,x\right)\right]^{-1}\cdot\nabla\widetilde{g}^{n}\left(t,x\right).
\end{aligned}
\label{eq:cv der temp}
\end{equation}

Evaluating the equations of the second statement of Proposition \ref{prop:Bousinessq in layers}
in $x\leftarrow\phi^{\left(n\right)}\left(t,x\right)$ and employing
equations \eqref{eq:cv der esp} and \eqref{eq:cv der temp}, we are
able to write
\[
\begin{aligned} & \frac{\partial\widetilde{\omega^{\left(n\right)}}^{n}}{\partial t}\left(t,x\right)-\frac{\partial\phi^{\left(n\right)}}{\partial t}\left(t,x\right)\cdot\left[\nabla\left(\phi^{\left(n\right)}\right)\left(t,x\right)\right]^{-1}\cdot\nabla\widetilde{\omega^{\left(n\right)}}^{n}\left(t,x\right)+\\
 & +\widetilde{U^{\left(n-1\right)}}^{n}\left(t,x\right)\cdot\left[\nabla\left(\phi^{\left(n\right)}\right)\left(t,x\right)\right]^{-1}\cdot\nabla\widetilde{\omega^{\left(n\right)}}^{n}\left(t,x\right)+\\
 & +\widetilde{u^{\left(n\right)}}^{n}\left(t,x\right)\cdot\left[\nabla\left(\phi^{\left(n\right)}\right)\left(t,x\right)\right]^{-1}\cdot\nabla\widetilde{\Omega^{\left(n-1\right)}}^{n}\left(t,x\right)+\\
 & +\widetilde{u^{\left(n\right)}}^{n}\left(t,x\right)\cdot\left[\nabla\left(\phi^{\left(n\right)}\right)\left(t,x\right)\right]^{-1}\cdot\nabla\widetilde{\omega^{\left(n\right)}}^{n}\left(t,x\right)=\\
= & \left(\begin{matrix}0 & 1\end{matrix}\right)\cdot\left[\nabla\left(\phi^{\left(n\right)}\right)\left(t,x\right)\right]^{-1}\cdot\nabla\widetilde{\rho^{\left(n\right)}}^{n}\left(t,x\right)+\widetilde{f_{\omega}^{\left(n\right)}}^{n}\left(t,x\right).
\end{aligned}
\]
\[
\begin{aligned} & \frac{\partial\widetilde{\rho^{\left(n\right)}}^{n}}{\partial t}\left(t,x\right)-\frac{\partial\phi^{\left(n\right)}}{\partial t}\left(t,x\right)\cdot\left[\nabla\left(\phi^{\left(n\right)}\right)\left(t,x\right)\right]^{-1}\cdot\nabla\widetilde{\rho^{\left(n\right)}}^{n}\left(t,x\right)+\\
 & +\widetilde{U^{\left(n-1\right)}}^{n}\left(t,x\right)\cdot\left[\nabla\left(\phi^{\left(n\right)}\right)\left(t,x\right)\right]^{-1}\cdot\nabla\widetilde{\rho^{\left(n\right)}}^{n}\left(t,x\right)+\\
 & +\widetilde{u^{\left(n\right)}}^{n}\left(t,x\right)\cdot\left[\nabla\left(\phi^{\left(n\right)}\right)\left(t,x\right)\right]^{-1}\cdot\nabla\widetilde{P^{\left(n-1\right)}}^{n}+\\
 & +\widetilde{u^{\left(n\right)}}^{n}\left(t,x\right)\cdot\left[\nabla\left(\phi^{\left(n\right)}\right)\left(t,x\right)\right]^{-1}\cdot\nabla\widetilde{\rho^{\left(n\right)}}^{n}\left(t,x\right)=\\
= & \widetilde{f_{\rho}^{\left(n\right)}}^{n}\left(t,x\right).
\end{aligned}
\]
Joining the first two lines of each equation, we arrive to the result
presented in the statement.
\end{proof}

\subsection{First choices of the construction}

\subsubsection{Time}

As we explained in subsection \ref{subsec:basic idea of the construction},
as time progresses, we will introduce layer after layer of density.
The following Choice is there just to make sure that the first layer
is introduced at $t=0$ and that the time of the blow-up is $t=1$.
\begin{choice}
\label{choice:time (initial)}Let $\left(t_{n}\right)_{n\in\mathbb{N}}\subseteq\left[0,1\right]$
be a sequence, where $t_{n}$ represents the moment when the $n$-th
layer is introduced. $\left(t_{n}\right)_{n\in\mathbb{N}}$ is an
increasing sequence and satisfies $t_{1}=0$ and $\lim_{n\to\infty}t_{n}=1$.
\end{choice}

\subsubsection{The change of variables $\phi^{\left(n\right)}\left(t,x\right)$}

Recall that, by Remark \ref{rem:change of variables should be approximation of transport},
$\phi^{\left(n\right)}\left(t,x\right)$ should resemble the change
of variables that eliminates transport in the equations given in Proposition
\ref{prop:change of variables}. Our goal is to choose $\phi^{\left(n\right)}\left(t,x\right)$
as simple as possible. Since the spatial expanse of each layer should
decrease with $n$, the size of layer $n$ should be much smaller
then the size of layer $n-1$. Thence, a linear approximation of the
true particle trajectories should suffice to account for the transport
of layer $n$ due to its preceding layers. In order words, it is reasonable
to take $\phi^{\left(n\right)}\left(t,x\right)$ as the spatially
linear approximation of the true particle trajectories (as we did
in subsection \ref{subsec:vorticity growth mechanism}). Nonetheless,
the same approach cannot be repeated in the time variable \footnote{Albeit it is true that, as we add new layers, the times considered
will become smaller and smaller, the characteristic time of each layer
will also decrease at the same rate. This implies that, if we focus
on one layer, the evolution times we will work with will be of the
same order than the characteristic time of the layer and, consequently,
will not be small enough for a Taylor expansion.}, where we will maintain all generality. 
\begin{choice}
\label{choice:phin}In light of the above, we take
\[
\phi^{\left(n\right)}\left(t,x\right)\coloneqq\phi^{\left(n\right)}\left(t,0\right)+\left(\frac{1}{a_{n}\left(t\right)}x_{1},\frac{1}{b_{n}\left(t\right)}x_{2}\right),
\]
where $\phi^{\left(n\right)}\left(t,0\right)$, $\frac{1}{a_{n}\left(t\right)}$
and $\frac{1}{b_{n}\left(t\right)}$ are functions to be determined.
Notice that we are choosing $\phi^{\left(n\right)}\left(t,x\right)$
in such a way that the center of the $n$-th layer corresponds to
$x=0$. We remark that this does not imply $\phi^{\left(n\right)}\left(t_{n},0\right)=0$.
In fact, we will choose $\phi^{\left(n\right)}\left(t_{n},0\right)=c^{\left(n\right)}$.
In this way, $c^{\left(n\right)}$ will be the initial position of
the center of the $n$-th layer.
\end{choice}
\begin{rem}
The domain of the change of variables defined in Choice \ref{choice:phin}
consists of the coordinates where the expressions for the physical
quantities look the same across all layers, i.e., the domain is independent
of $n\in\mathbb{N}$ and places the centers of all layers at $x=0$.
On the other hand, the image of the change of variables defined in
Choice \ref{choice:phin} is the real plane $\mathbb{R}^{2}$, where
the solution is going to be constructed. Thus, $\phi^{\left(n\right)}\left(t,x\right)$
maps the $n$-th layer from ``ideal'' coordinates to its real position.
\end{rem}
From Choice \ref{choice:phin}, we can easily derive the inverse change
of variables:
\begin{equation}
\left(\phi^{\left(n\right)}\right)^{-1}\left(t,x\right)=\left(a_{n}\left(t\right)\left(x_{1}-\phi_{1}^{\left(n\right)}\left(t,0\right)\right),b_{n}\left(t\right)\left(x_{2}-\phi_{2}^{\left(n\right)}\left(t,0\right)\right)\right).\label{eq:phin inverse}
\end{equation}
Then, we automatically obtain that
\begin{equation}
\mathrm{J}^{-1}\phi^{\left(n\right)}\left(t,x\right)=\mathrm{J}\left(\phi^{\left(n\right)}\right)^{-1}\left(t,\phi^{\left(n\right)}\left(t,x\right)\right)=\left(\begin{matrix}a_{n}\left(t\right) & 0\\
0 & b_{n}\left(t\right)
\end{matrix}\right).\label{eq:jacobian inverse}
\end{equation}

\subsubsection{Tilde operators}

This simple expression for the jacobian of the inverse change of variables
has the advantage that we can define a nabla tilde operator $\widetilde{\nabla}^{n}$,
which will allow us to simplify the equations of Proposition \ref{prop:change of variables}.
This $\widetilde{\nabla}^{n}$ operator is given by
\begin{equation}
\widetilde{\nabla}^{n}\coloneqq\left(a_{n}\left(t\right)\frac{\partial}{\partial x_{1}},b_{n}\left(t\right)\frac{\partial}{\partial x_{2}}\right).\label{eq:def tilde nabla}
\end{equation}
Consequently, we can define ``tilde super $n$'' versions of the
gradient, divergence and laplacian operators. Explicitly,
\[
\begin{aligned}\widetilde{\nabla}^{n}g & =\left(a_{n}\left(t\right)\frac{\partial g}{\partial x_{1}},b_{n}\left(t\right)\frac{\partial g}{\partial x_{2}}\right),\quad\text{ for }g:\mathbb{R}^{2}\to\mathbb{R},\\
\widetilde{\nabla}^{n}\cdot g & =a_{n}\left(t\right)\frac{\partial g_{1}}{\partial x_{1}}+b_{n}\left(t\right)\frac{\partial g_{2}}{\partial x_{2}},\quad\text{ for }g:\mathbb{R}^{2}\to\mathbb{R}^{2},\\
\widetilde{\Delta}^{n}g & =a_{n}\left(t\right)^{2}\frac{\partial^{2}g}{\partial x_{1}^{2}}+b_{n}\left(t\right)^{2}\frac{\partial^{2}g}{\partial x_{2}^{2}},\quad\text{ for }g:\mathbb{R}^{2}\to\mathbb{R}.
\end{aligned}
\]
Using the language of these operators and taking into account how
we have chosen $\phi^{\left(n\right)}\left(t,x\right)$ in Choice
\ref{choice:phin}, as
\begin{equation}
\widetilde{\nabla}^{n}g=\left(\begin{matrix}a_{n}\left(t\right) & 0\\
0 & b_{n}\left(t\right)
\end{matrix}\right)\left(\begin{matrix}\frac{\partial g}{\partial x_{1}} & \frac{\partial g}{\partial x_{2}}\end{matrix}\right)=\left[\nabla\left(\phi^{\left(n\right)}\right)\right]^{-1}\cdot\nabla g\quad\forall g\in C^{1}\left(\mathbb{R}^{2};\mathbb{R}\right)\label{eq:gradient tilde phi}
\end{equation}
by equation \eqref{eq:jacobian inverse}, we are able to rewrite the
equations of Proposition \ref{prop:change of variables} as follows:
\begin{equation}
\begin{aligned} & \frac{\partial\widetilde{\omega^{\left(n\right)}}^{n}}{\partial t}\left(t,x\right)+\\
 & +\left(\widetilde{U^{\left(n-1\right)}}^{n}\left(t,x\right)-\frac{\partial\phi^{\left(n\right)}}{\partial t}\left(t,0\right)-\left(\begin{matrix}-\frac{1}{a_{n}\left(t\right)^{2}}\frac{\mathrm{d}a_{n}}{\mathrm{d}t}\left(t\right)x_{1}\\
-\frac{1}{b_{n}\left(t\right)^{2}}\frac{\mathrm{d}b_{n}}{\mathrm{d}t}\left(t\right)x_{2}
\end{matrix}\right)\right)\cdot\widetilde{\nabla}^{n}\widetilde{\omega^{\left(n\right)}}^{n}\left(t,x\right)+\\
 & +\widetilde{u^{\left(n\right)}}^{n}\left(t,x\right)\cdot\widetilde{\nabla}^{n}\widetilde{\Omega^{\left(n-1\right)}}^{n}+\widetilde{u^{\left(n\right)}}^{n}\left(t,x\right)\cdot\widetilde{\nabla}^{n}\widetilde{\omega^{\left(n\right)}}^{n}\left(t,x\right)=\\
= & \left(\begin{matrix}0 & 1\end{matrix}\right)\cdot\left[\widetilde{\nabla}^{n}\widetilde{\rho^{\left(n\right)}}^{n}\left(t,x\right)\right]+\widetilde{f_{\omega}^{\left(n\right)}}^{n}\left(t,x\right),
\end{aligned}
\label{eq:Boussinesq 1 con tildes}
\end{equation}
\begin{equation}
\begin{aligned} & \frac{\partial\widetilde{\rho^{\left(n\right)}}^{n}}{\partial t}\left(t,x\right)+\\
 & +\left(\widetilde{U^{\left(n-1\right)}}^{n}\left(t,x\right)-\frac{\partial\phi^{\left(n\right)}}{\partial t}\left(t,0\right)-\left(\begin{matrix}-\frac{1}{a_{n}\left(t\right)^{2}}\frac{\mathrm{d}a_{n}}{\mathrm{d}t}\left(t\right)x_{1}\\
-\frac{1}{b_{n}\left(t\right)^{2}}\frac{\mathrm{d}b_{n}}{\mathrm{d}t}\left(t\right)x_{2}
\end{matrix}\right)\right)\cdot\widetilde{\nabla}^{n}\widetilde{\rho^{\left(n\right)}}^{n}\left(t,x\right)+\\
 & +\widetilde{u^{\left(n\right)}}^{n}\left(t,x\right)\cdot\widetilde{\nabla}^{n}\widetilde{P^{\left(n-1\right)}}^{n}\left(t,x\right)+\widetilde{u^{\left(n\right)}}^{n}\left(t,x\right)\cdot\widetilde{\nabla}^{n}\widetilde{\rho^{\left(n\right)}}^{n}\left(t,x\right)=\\
= & \widetilde{f_{\rho}^{\left(n\right)}}^{n}\left(t,x\right).
\end{aligned}
\label{eq:Boussinesq 2 con tildes}
\end{equation}
Furthermore, these ``tilde'' operators satisfy properties similar
to those fulfilled by the original operators.
\begin{lem}
\label{lem:operadores tilde}~
\begin{enumerate}
\item Given $f,g:\mathbb{R}^{2}\to\mathbb{R}$, we have
\[
\widetilde{\nabla}^{n}\left(fg\right)=g\widetilde{\nabla}^{n}f+f\widetilde{\nabla}^{n}g.
\]
\item Given $f:\mathbb{R}^{2}\to\mathbb{R}$ and $u:\mathbb{R}^{2}\to\mathbb{R}^{2}$,
we have
\[
\widetilde{\nabla}^{n}\cdot\left(fu\right)=f\widetilde{\nabla}^{n}\cdot u+\widetilde{\nabla}^{n}f\cdot u.
\]
\item Given $f,g:\mathbb{R}^{2}\to\mathbb{R}$, we have
\[
\widetilde{\Delta}^{n}\left(fg\right)=f\widetilde{\Delta}^{n}g+g\widetilde{\Delta}^{n}f+2\widetilde{\nabla}^{n}f\cdot\widetilde{\nabla}^{n}g.
\]
\end{enumerate}
\end{lem}
\begin{proof}
~
\begin{enumerate}
\item Using the definition of $\widetilde{\nabla}^{n}$ and the Leibniz
rule, we deduce that
\[
\begin{aligned}\widetilde{\nabla}^{n}\left(fg\right) & =\left(a_{n}\left(t\right)\frac{\partial}{\partial x_{1}}\left(fg\right),b_{n}\left(t\right)\frac{\partial}{\partial x_{2}}\left(fg\right)\right)=\\
 & =\left(a_{n}\left(t\right)f\left(\frac{\partial g}{\partial x_{1}}+g\frac{\partial f}{\partial x_{1}}\right),b_{n}\left(t\right)\left(f\frac{\partial g}{\partial x_{2}}+g\frac{\partial f}{\partial x_{2}}\right)\right)=\\
 & =f\left(a_{n}\left(t\right)\frac{\partial g}{\partial x_{1}},b_{n}\left(t\right)\frac{\partial g}{\partial x_{2}}\right)+g\left(a_{n}\left(t\right)\frac{\partial f}{\partial x_{1}},b_{n}\left(t\right)\frac{\partial f}{\partial x_{2}}\right)=\\
 & =f\widetilde{\nabla}^{n}g+g\widetilde{\nabla}^{n}f.
\end{aligned}
\]
\item Employing the definition of $\widetilde{\nabla}^{n}$ and the Leibniz
rule, we arrive to
\[
\begin{aligned}\widetilde{\nabla}^{n}\cdot\left(fu\right) & =a_{n}\left(t\right)\frac{\partial}{\partial x_{1}}\left(fu_{1}\right)+b_{n}\left(t\right)\frac{\partial}{\partial x_{2}}\left(fu_{2}\right)=\\
 & =a_{n}\left(t\right)\left(f\frac{\partial u_{1}}{\partial x_{1}}+u_{1}\frac{\partial f}{\partial x_{1}}\right)+b_{n}\left(t\right)\left(f\frac{\partial u_{2}}{\partial x_{2}}+u_{2}\frac{\partial f}{\partial x_{2}}\right)=\\
 & =f\left[a_{n}\left(t\right)\frac{\partial u_{1}}{\partial x_{1}}+b_{n}\left(t\right)\frac{\partial u_{2}}{\partial x_{2}}\right]+\left(u_{1},u_{2}\right)\cdot\left(a_{n}\left(t\right)\frac{\partial f}{\partial x_{1}},b_{n}\left(t\right)\frac{\partial f}{\partial x_{2}}\right)=\\
 & =f\widetilde{\nabla}^{n}\cdot u+u\cdot\widetilde{\nabla}^{n}f.
\end{aligned}
\]
\item Making use of points $1$ and $2$, we obtain
\[
\begin{aligned}\widetilde{\Delta}^{n}\left(fg\right) & =\widetilde{\nabla}^{n}\cdot\left(\widetilde{\nabla}^{n}\left(fg\right)\right)=\widetilde{\nabla}^{n}\cdot\left(g\widetilde{\nabla}^{n}f+f\widetilde{\nabla}^{n}g\right)=\\
 & =\widetilde{\nabla}^{n}g\cdot\widetilde{\nabla}^{n}f+g\widetilde{\Delta}^{n}f+\widetilde{\nabla}^{n}f\cdot\widetilde{\nabla}^{n}g+f\widetilde{\Delta}^{n}g=\\
 & =g\widetilde{\Delta}^{n}f+f\widetilde{\Delta}^{n}g+2\widetilde{\nabla}^{n}f\cdot\widetilde{\nabla}^{n}g.
\end{aligned}
\]
\end{enumerate}
\end{proof}
\begin{prop}
\label{prop:relations psi n tilde derivatives}If $\psi^{\left(n\right)}$
is the stream function of the $n$-th layer, we have $\widetilde{u^{\left(n\right)}}^{n}\left(t,x\right)=-\widetilde{\nabla}^{n\perp}\widetilde{\psi^{\left(n\right)}}^{n}\left(t,x\right)$
and $\widetilde{\omega^{\left(n\right)}}^{n}\left(t,x\right)=-\widetilde{\Delta}^{n}\widetilde{\psi^{\left(n\right)}}^{n}\left(t,x\right)$.
\end{prop}
\begin{proof}
Because $\psi^{\left(n\right)}$ is the stream function of the $n$-th
layer, we have
\[
\begin{aligned}\widetilde{u^{\left(n\right)}}^{n}\left(t,x\right) & =u^{\left(n\right)}\left(t,\phi^{\left(n\right)}\left(t,x\right)\right)=-\nabla^{\perp}\psi^{\left(n\right)}\left(t,\phi^{\left(n\right)}\left(t,x\right)\right)=\\
 & =\left(\begin{matrix}\frac{\partial\psi^{\left(n\right)}}{\partial x_{2}}\left(t,\phi^{\left(n\right)}\left(t,x\right)\right)\\
-\frac{\partial\psi^{\left(n\right)}}{\partial x_{1}}\left(t,\phi^{\left(n\right)}\left(t,x\right)\right)
\end{matrix}\right).
\end{aligned}
\]
Recalling equations \eqref{eq:cv der esp} and \eqref{eq:jacobian inverse},
we obtain that
\[
\begin{aligned}\frac{\partial\psi^{\left(n\right)}}{\partial x_{2}}\left(t,\phi^{\left(n\right)}\left(t,x\right)\right) & =b_{n}\left(t\right)\frac{\partial}{\partial x_{2}}\left(\psi^{\left(n\right)}\left(t,\phi^{\left(n\right)}\left(t,x\right)\right)\right)=b_{n}\left(t\right)\frac{\partial\widetilde{\psi^{\left(n\right)}}}{\partial x_{2}}\left(t,x\right),\\
\frac{\partial\psi^{\left(n\right)}}{\partial x_{1}}\left(t,\phi^{\left(n\right)}\left(t,x\right)\right) & =a_{n}\left(t\right)\frac{\partial}{\partial x_{1}}\left(\psi^{\left(n\right)}\left(t,\phi^{\left(n\right)}\left(t,x\right)\right)\right)=a_{n}\left(t\right)\frac{\partial\widetilde{\psi^{\left(n\right)}}}{\partial x_{1}}\left(t,x\right).
\end{aligned}
\]
Hence,
\[
\widetilde{u^{\left(n\right)}}^{n}\left(t,x\right)=\left(\begin{matrix}b_{n}\left(t\right)\frac{\partial\widetilde{\psi^{\left(n\right)}}}{\partial x_{2}}\left(t,x\right)\\
-a_{n}\left(t\right)\frac{\partial\widetilde{\psi^{\left(n\right)}}}{\partial x_{1}}\left(t,x\right)
\end{matrix}\right)=-\widetilde{\nabla}^{n\perp}\widetilde{\psi^{\left(n\right)}}^{n}\left(t,x\right),
\]
where we have used definition \eqref{eq:def tilde nabla} in the last
step.

Concerning the second statement, again, as $\psi^{\left(n\right)}$
is the stream function of the $n$-th layer, we must have
\[
\begin{aligned}\widetilde{\omega^{\left(n\right)}}^{n}\left(t,x\right) & =\omega^{\left(n\right)}\left(t,\phi^{\left(n\right)}\left(t,x\right)\right)=-\Delta\psi^{\left(n\right)}\left(t,\phi^{\left(n\right)}\left(t,x\right)\right)=\\
 & =-\frac{\partial^{2}\psi^{\left(n\right)}}{\partial x_{1}^{2}}\left(t,\phi^{\left(n\right)}\left(t,x\right)\right)-\frac{\partial^{2}\psi^{\left(n\right)}}{\partial x_{2}^{2}}\left(t,\phi^{\left(n\right)}\left(t,x\right)\right)=\\
 & =-\frac{\partial}{\partial x_{1}}\left(\frac{\partial\psi^{\left(n\right)}}{\partial x_{1}}\right)\left(t,\phi^{\left(n\right)}\left(t,x\right)\right)-\frac{\partial}{\partial x_{2}}\left(\frac{\partial\psi^{\left(n\right)}}{\partial x_{2}}\right)\left(t,\phi^{\left(n\right)}\left(t,x\right)\right).
\end{aligned}
\]
Thanks to equations \eqref{eq:cv der esp} and \eqref{eq:jacobian inverse},
we may write
\[
\begin{aligned}\widetilde{\omega^{\left(n\right)}}^{n}\left(t,x\right) & =-a_{n}\left(t\right)\frac{\partial}{\partial x_{1}}\left(\frac{\partial\psi^{\left(n\right)}}{\partial x_{1}}\left(t,\phi^{\left(n\right)}\left(t,x\right)\right)\right)-b_{n}\left(t\right)\frac{\partial}{\partial x_{2}}\left(\frac{\partial\psi^{\left(n\right)}}{\partial x_{2}}\left(t,\phi^{\left(n\right)}\left(t,x\right)\right)\right)=\\
 & =-a_{n}\left(t\right)^{2}\frac{\partial^{2}}{\partial x_{1}^{2}}\left(\psi^{\left(n\right)}\left(t,\phi^{\left(n\right)}\left(t,x\right)\right)\right)-b_{n}\left(t\right)^{2}\frac{\partial^{2}}{\partial x_{2}^{2}}\left(\psi^{\left(n\right)}\left(t,\phi^{\left(n\right)}\left(t,x\right)\right)\right)=\\
 & =-a_{n}\left(t\right)^{2}\frac{\partial^{2}\widetilde{\psi^{\left(n\right)}}^{n}}{\partial x_{1}^{2}}\left(t,x\right)-b_{n}\left(t\right)^{2}\frac{\partial^{2}\widetilde{\psi^{\left(n\right)}}^{n}}{\partial x_{2}^{2}}\left(t,x\right)=-\widetilde{\Delta}^{n}\widetilde{\psi^{\left(n\right)}}^{n}\left(t,x\right),
\end{aligned}
\]
where we have used definition \eqref{eq:def tilde nabla} in the last
step.
\end{proof}

\subsubsection{Stream function}

As happens in the 2D Euler equations, we can condense the information
about the velocity field and the vorticity of the fluid in the stream
function, which is a scalar function. In other words, if we prescribe
the stream function of each layer, we can automatically compute, via
differentiation, all other physical quantities that appear in the
equation of the vorticity except for $\frac{\partial\rho^{\left(n\right)}}{\partial x_{2}}$
and $f_{\omega}^{\left(n\right)}$. Actually, as one can check by
applying Proposition \ref{prop:relations psi n tilde derivatives},
taking the stream function as
\begin{equation}
\widetilde{\psi^{\left(n\right)}}^{n}\left(t,x\right)\coloneqq B_{n}\left(t\right)\sin\left(x_{1}\right)\sin\left(x_{2}\right)\label{eq:stream function}
\end{equation}
generates the velocity field
\[
\widetilde{u^{\left(n\right)}}^{n}\left(t,x\right)\coloneqq\left(\begin{matrix}B_{n}\left(t\right)b_{n}\left(t\right)\sin\left(x_{1}\right)\cos\left(x_{2}\right)\\
-B_{n}\left(t\right)a_{n}\left(t\right)\cos\left(x_{1}\right)\sin\left(x_{2}\right)
\end{matrix}\right),
\]
which looks very similar to the velocity field \eqref{eq:intro u1 with transport}
we saw in subsection \eqref{subsec:vorticity growth mechanism}. In
this way, the shortness of expression \eqref{eq:stream function}
is a clear incentive to work with the stream function. Moreover, it
is much easier to add cutoffs to a stream function than to add them
to a velocity field. This is because, by construction, every velocity
field that comes from a stream function will automatically be divergence
free. Thereby, it is enough to add the cutoff function multiplying
the expression \eqref{eq:stream function}, as we shall do shortly.
Additionally, this comes with the advantage that all the physical
quantities involved in our construction (velocity, vorticity, density
and both external forces) will be compactly supported (to see this,
recall that $u^{\left(n\right)}=-\nabla^{\perp}\psi^{\left(n\right)}$
and $\omega^{\left(n\right)}=-\Delta\psi^{\left(n\right)}$ and take
a look at the second statement of Proposition \ref{prop:Bousinessq in layers}).
\begin{choice}
\label{choice:psin}Based on all the aforementioned remarks, we shall
construct a flux whose stream function is
\[
\Psi^{\left(\infty\right)}\left(t,x\right)=\sum_{n=1}^{\infty}\psi^{\left(n\right)}\left(t,x\right),\quad\widetilde{\psi^{\left(n\right)}}^{n}\left(t,x\right)=B_{n}\left(t\right)\eta^{\left(n\right)}\left(x\right)f\left(x\right),
\]
where
\[
f\left(x\right)=\sin\left(x_{1}\right)\sin\left(x_{2}\right),\quad\eta^{\left(n\right)}\left(x\right)=\varphi\left(\lambda_{n}x_{1}\right)\varphi\left(\lambda_{n}x_{2}\right),
\]
$\varphi\in C_{c}^{\infty}\left(\mathbb{R}\right)$ is even, satisfies
$\varphi\equiv1$ in the interval $\left[-8\pi,8\pi\right]$ and $\varphi\equiv0$
in $\mathbb{R}\setminus\left[-16\pi,16\pi\right]$, $\left(\lambda_{n}\right)_{n\in\mathbb{N}}\subseteq\left(0,1\right)$
is a sequence whose values will be chosen later and $B_{n}\left(t\right)$
is a free parameter that will also be chosen later.
\end{choice}
As we will see, the flux associated to this stream function is not
an exact solution of the Boussinesq system, i.e., error terms appear.
We will use the force to compensate these error terms. Consequently,
the flux we present will indeed be a solution of the forced Boussinesq
system. The main appeal of this construction is that the force will
be substantially more regular then the solution at the blow-up time.
This means that the external force is not the mechanism behind the
finite-time singularity. Besides, our construction will be completely
explicit in the spatial variables and compactly supported, as we will
see.

\subsection{Velocity and vorticity}

In this subsection, we will provide explicit expressions for the velocity
and the vorticity and we will obtain some bounds that will be needed
later on.
\begin{prop}
\label{prop:computations vorticity}We have
\[
\begin{aligned}\widetilde{u^{\left(n\right)}}^{n}\left(t,x\right) & =B_{n}\left(t\right)\varphi\left(\lambda_{n}x_{1}\right)\varphi\left(\lambda_{n}x_{2}\right)\left(\begin{matrix}b_{n}\left(t\right)\sin\left(x_{1}\right)\cos\left(x_{2}\right)\\
-a_{n}\left(t\right)\cos\left(x_{1}\right)\sin\left(x_{2}\right)
\end{matrix}\right)+\\
 & \quad+\lambda_{n}B_{n}\left(t\right)\left(\begin{matrix}b_{n}\left(t\right)\varphi\left(\lambda_{n}x_{1}\right)\varphi'\left(\lambda_{n}x_{2}\right)\\
-a_{n}\left(t\right)\varphi'\left(\lambda_{n}x_{1}\right)\varphi\left(\lambda_{n}x_{2}\right)
\end{matrix}\right)\sin\left(x_{1}\right)\sin\left(x_{2}\right).
\end{aligned}
\]
\[
\begin{aligned}\widetilde{\omega^{\left(n\right)}}^{n}\left(t,x\right) & =B_{n}\left(t\right)\left(a_{n}\left(t\right)^{2}+b_{n}\left(t\right)^{2}\right)\varphi\left(\lambda_{n}x_{1}\right)\varphi\left(\lambda_{n}x_{2}\right)\sin\left(x_{1}\right)\sin\left(x_{2}\right)+\\
 & \quad-\lambda_{n}^{2}B_{n}\left(t\right)a_{n}\left(t\right)^{2}\varphi''\left(\lambda_{n}x_{1}\right)\varphi\left(\lambda_{n}x_{2}\right)\sin\left(x_{1}\right)\sin\left(x_{2}\right)+\\
 & \quad-\lambda_{n}^{2}B_{n}\left(t\right)b_{n}\left(t\right)^{2}\varphi\left(\lambda_{n}x_{1}\right)\varphi''\left(\lambda_{n}x_{2}\right)\sin\left(x_{1}\right)\sin\left(x_{2}\right)+\\
 & \quad-2\lambda_{n}B_{n}\left(t\right)a_{n}\left(t\right)^{2}\varphi'\left(\lambda_{n}x_{1}\right)\varphi\left(\lambda_{n}x_{2}\right)\cos\left(x_{1}\right)\sin\left(x_{2}\right)+\\
 & \quad-2\lambda_{n}B_{n}\left(t\right)b_{n}\left(t\right)^{2}\varphi\left(\lambda_{n}x_{1}\right)\varphi'\left(\lambda_{n}x_{2}\right)\sin\left(x_{1}\right)\cos\left(x_{2}\right).
\end{aligned}
\]
\end{prop}
\begin{rem}
The Proposition above is telling us that, due to the cutoff functions,
both the velocity and the vorticity differ from the ideal objective
introduced in subsection \ref{subsec:vorticity growth mechanism}.
Nevertheless, in both cases, we can express the physical quantity
as the sum of one term that has the desired structure (and has zero
order in $\lambda_{n}$) and some other terms which are proportional
to $\lambda_{n}$. Consequently, if $\lambda_{n}$ is ``small'',
both the velocity and the vorticity should behave ``similarly''
to the case without cutoff functions.
\end{rem}
\begin{proof}
The expression of the velocity follows from applying Proposition \ref{prop:relations psi n tilde derivatives}
to the choice of $\psi^{\left(n\right)}$ taken in Choice \ref{choice:psin}.
For the vorticity, by Proposition \ref{prop:relations psi n tilde derivatives}
and Lemma \ref{lem:operadores tilde}, we have
\[
\begin{aligned}\widetilde{\omega^{\left(n\right)}}^{n}\left(t,x\right) & =-\widetilde{\Delta}^{n}\widetilde{\psi^{\left(n\right)}}^{n}\left(t,x\right)=-B_{n}\left(t\right)\eta^{\left(n\right)}\left(x\right)\widetilde{\Delta}^{n}f\left(x\right)-B_{n}\left(t\right)f\left(x\right)\widetilde{\Delta}^{n}\eta^{\left(n\right)}\left(x\right)-2\widetilde{\nabla}^{n}\eta^{\left(n\right)}\left(x\right)\cdot\widetilde{\nabla}^{n}f\left(x\right)=\\
 & =B_{n}\left(t\right)\varphi\left(\lambda_{n}x_{1}\right)\varphi\left(\lambda_{n}x_{2}\right)\left(a_{n}\left(t\right)^{2}+b_{n}\left(t\right)^{2}\right)\sin\left(x_{1}\right)\sin\left(x_{2}\right)+\\
 & \quad-\lambda_{n}^{2}B_{n}\left(t\right)\left(a_{n}\left(t\right)^{2}\varphi''\left(\lambda_{n}x_{1}\right)\varphi\left(\lambda_{n}x_{2}\right)+b_{n}\left(t\right)^{2}\varphi\left(\lambda_{n}x_{1}\right)\varphi''\left(\lambda_{n}x_{2}\right)\right)\sin\left(x_{1}\right)\sin\left(x_{2}\right)+\\
 & \quad-2\lambda_{n}B_{n}\left(t\right)\left(\begin{matrix}a_{n}\left(t\right)\varphi'\left(\lambda_{n}x_{1}\right)\varphi\left(\lambda_{n}x_{2}\right)\\
b_{n}\left(t\right)\varphi\left(\lambda_{n}x_{1}\right)\varphi'\left(\lambda_{n}x_{2}\right)
\end{matrix}\right)\cdot\left(\begin{matrix}a_{n}\left(t\right)\cos\left(x_{1}\right)\sin\left(x_{2}\right)\\
b_{n}\left(t\right)\sin\left(x_{1}\right)\cos\left(x_{2}\right)
\end{matrix}\right).
\end{aligned}
\]
Expanding the expressions above, we obtain the statement.
\end{proof}
\begin{prop}
\label{prop:form of the velocity}There are functions $V^{\left(n\right)},W^{\left(n\right)}:\left[0,1\right]\times\mathbb{R}^{2}\to\mathbb{R}^{2}$
such that
\[
\begin{aligned}\widetilde{u^{\left(n\right)}}^{n}\left(t,x\right) & =\widetilde{V^{\left(n\right)}}^{n}\left(t,x\right)+\lambda_{n}\widetilde{W^{\left(n\right)}}^{n}\left(t,x\right),\\
\widetilde{V^{\left(n\right)}}^{n}\left(t,x\right) & =B_{n}\left(t\right)\varphi\left(\lambda_{n}x_{1}\right)\varphi\left(\lambda_{n}x_{2}\right)\left(\begin{matrix}b_{n}\left(t\right)\sin\left(x_{1}\right)\cos\left(x_{2}\right)\\
-a_{n}\left(t\right)\cos\left(x_{1}\right)\sin\left(x_{2}\right)
\end{matrix}\right).
\end{aligned}
\]
Furthermore, given $\alpha\in\left(0,1\right)$,
\[
{\small \begin{matrix}\begin{aligned}\left|\left|V_{1}^{\left(n\right)}\left(t,\cdot\right)\right|\right|_{L^{\infty}\left(\mathbb{R}^{2}\right)} & \le B_{n}\left(t\right)b_{n}\left(t\right),\\
\left|\left|W_{1}^{\left(n\right)}\left(t,\cdot\right)\right|\right|_{L^{\infty}\left(\mathbb{R}^{2}\right)} & \lesssim_{\varphi}B_{n}\left(t\right)b_{n}\left(t\right),\\
\left|\left|V_{1}^{\left(n\right)}\left(t,\cdot\right)\right|\right|_{\dot{C}^{\alpha}\left(\mathbb{R}^{2}\right)} & \lesssim_{\varphi}B_{n}\left(t\right)b_{n}\left(t\right)\max\left\{ a_{n}\left(t\right)^{\alpha},b_{n}\left(t\right)^{\alpha}\right\} ,\\
\left|\left|W_{1}^{\left(n\right)}\left(t,\cdot\right)\right|\right|_{\dot{C}^{\alpha}\left(\mathbb{R}^{2}\right)} & \lesssim_{\varphi}B_{n}\left(t\right)b_{n}\left(t\right)\max\left\{ a_{n}\left(t\right)^{\alpha},b_{n}\left(t\right)^{\alpha}\right\} ,
\end{aligned}
 &  & \begin{aligned}\left|\left|V_{2}^{\left(n\right)}\left(t,\cdot\right)\right|\right|_{L^{\infty}\left(\mathbb{R}^{2}\right)} & \le B_{n}\left(t\right)a_{n}\left(t\right),\\
\left|\left|W_{2}^{\left(n\right)}\left(t,\cdot\right)\right|\right|_{L^{\infty}\left(\mathbb{R}^{2}\right)} & \lesssim_{\varphi}B_{n}\left(t\right)a_{n}\left(t\right),\\
\left|\left|V_{2}^{\left(n\right)}\left(t,\cdot\right)\right|\right|_{\dot{C}^{\alpha}\left(\mathbb{R}^{2}\right)} & \lesssim_{\varphi}B_{n}\left(t\right)a_{n}\left(t\right)\max\left\{ a_{n}\left(t\right)^{\alpha},b_{n}\left(t\right)^{\alpha}\right\} ,\\
\left|\left|W_{2}^{\left(n\right)}\left(t,\cdot\right)\right|\right|_{\dot{C}^{\alpha}\left(\mathbb{R}^{2}\right)} & \lesssim_{\varphi}B_{n}\left(t\right)a_{n}\left(t\right)\max\left\{ a_{n}\left(t\right)^{\alpha},b_{n}\left(t\right)^{\alpha}\right\} ,
\end{aligned}
\end{matrix}}
\]
where $\lesssim_{\varphi}$ means that there is $M\left(\varphi\right)>0$
such that
\[
\text{left hand side}\leq M\left(\varphi\right)\cdot\text{right hand side}.
\]
The constants $M\left(\varphi\right)$ are different in each case.
\end{prop}
\begin{rem}
Notice that, although we define the functions $V^{\left(n\right)}$
and $W^{\left(n\right)}$ via the change of variables $\phi^{\left(n\right)}\left(t,x\right)$,
the bounds we present are for the functions $V^{\left(n\right)}$
and $W^{\left(n\right)}$ themselves and not for $\widetilde{V^{\left(n\right)}}^{n}$
and $\widetilde{W^{\left(n\right)}}^{n}$. We do this, because, even
though the expressions are much simpler to write with the change of
variables $\phi^{\left(n\right)}\left(t,x\right)$, we are not trying
to make the force under this change of variables be bounded; what
we need is the force itself (without any change of variables) to be
bounded.
\end{rem}
\begin{proof}
Taking
\[
\widetilde{W^{\left(n\right)}}^{n}\left(t,x\right)=B_{n}\left(t\right)\left(\begin{matrix}b_{n}\left(t\right)\varphi\left(\lambda_{n}x_{1}\right)\varphi'\left(\lambda_{n}x_{2}\right)\\
-a_{n}\left(t\right)\varphi'\left(\lambda_{n}x_{1}\right)\varphi\left(\lambda_{n}x_{2}\right)
\end{matrix}\right)\sin\left(x_{1}\right)\sin\left(x_{2}\right),
\]
the decomposition given in the statement is clearly true.

Next, we proceed to compute the required $\left|\left|\cdot\right|\right|_{L^{\infty}\left(\mathbb{R}^{2}\right)}$
and $\left|\left|\cdot\right|\right|_{\dot{C}^{\alpha}\left(\mathbb{R}^{2}\right)}$
norms. Using that $\left|\left|\varphi\right|\right|_{L^{\infty}\left(\mathbb{R}^{2}\right)}=1$
(see Choice \ref{choice:psin}) and recalling that $\left|\left|\cdot\right|\right|_{L^{\infty}\left(\mathbb{R}^{2}\right)}$
is invariant under diffeomorphisms of the domain, we immediately deduce
the bounds in $\left|\left|\cdot\right|\right|_{L^{\infty}\left(\mathbb{R}^{2}\right)}$
for $V^{\left(n\right)}$ and $W^{\left(n\right)}$ given in the statement.
Indeed, schematically, we have
\[
\begin{aligned}\widetilde{V^{\left(n\right)}}^{n}\left(t,x\right) & =B_{n}\left(t\right)\underbrace{\varphi\left(\lambda_{n}x_{1}\right)}_{\le1}\underbrace{\varphi\left(\lambda_{n}x_{2}\right)}_{\le1}\left(\begin{matrix}b_{n}\left(t\right)\overbrace{\sin\left(x_{1}\right)}^{\le1}\overbrace{\cos\left(x_{2}\right)}^{\le1}\\
-a_{n}\left(t\right)\underbrace{\cos\left(x_{1}\right)}_{\le1}\underbrace{\sin\left(x_{2}\right)}_{\le1}
\end{matrix}\right),\\
\widetilde{W^{\left(n\right)}}^{n}\left(t,x\right) & =B_{n}\left(t\right)\left(\begin{matrix}b_{n}\left(t\right)\overbrace{\varphi\left(\lambda_{n}x_{1}\right)}^{\le1}\overbrace{\varphi'\left(\lambda_{n}x_{2}\right)}^{\le\left|\left|\varphi\right|\right|_{\dot{C}^{1}\left(\mathbb{R}\right)}}\\
-a_{n}\left(t\right)\underbrace{\varphi'\left(\lambda_{n}x_{1}\right)}_{\le\left|\left|\varphi\right|\right|_{\dot{C}^{1}\left(\mathbb{R}\right)}}\underbrace{\varphi\left(\lambda_{n}x_{2}\right)}_{\le1}
\end{matrix}\right)\underbrace{\sin\left(x_{1}\right)}_{\le1}\underbrace{\sin\left(x_{2}\right)}_{\le1}.
\end{aligned}
\]

For the $\left|\left|\cdot\right|\right|_{\dot{C}^{\alpha}\left(\mathbb{R}^{2}\right)}$
seminorm, a little more work is required. Equation \eqref{eq:property Calpha multiplication}
tells us that, when computing the $\left|\left|\cdot\right|\right|_{\dot{C}^{\alpha}\left(\mathbb{R}^{2}\right)}$
seminorm of $\widetilde{V^{\left(n\right)}}^{n}$ or $\widetilde{W^{\left(n\right)}}^{n}$,
we will obtain one summand for every factor that has some space dependence.
In each of those summands, one factor will be a $\left|\left|\cdot\right|\right|_{\dot{C}^{\alpha}\left(\mathbb{R}^{2}\right)}$
seminorm and the others will be $\left|\left|\cdot\right|\right|_{L^{\infty}\left(\mathbb{R}^{2}\right)}$
norms. Furthermore, thanks to equation \eqref{eq:property Calpha composition},
we can bound
\[
\left|\left|\varphi\left(\lambda_{n}\cdot\right)\right|\right|_{\dot{C}^{\alpha}\left(\mathbb{R}\right)}\le K_{0}\left(1\right)\lambda_{n}^{\alpha}\left|\left|\varphi\right|\right|_{\dot{C}^{\alpha}\left(\mathbb{R}\right)}.
\]
As $\varphi\in C_{c}^{\infty}\left(\mathbb{R}\right)$, we have $\left|\left|\varphi\right|\right|_{\dot{C}^{\alpha}\left(\mathbb{R}\right)}\le\left|\left|\varphi\right|\right|_{\dot{C}^{1}\left(\mathbb{R}\right)}$.
Similarly, we have $\left|\left|\sin\left(\cdot\right)\right|\right|_{\dot{C}^{\alpha}\left(\mathbb{R}\right)}\le\left|\left|\sin\left(\cdot\right)\right|\right|_{\dot{C}^{1}\left(\mathbb{R}\right)}=1$
and $\left|\left|\cos\left(\cdot\right)\right|\right|_{\dot{C}^{\alpha}\left(\mathbb{R}\right)}\le\left|\left|\cos\left(\cdot\right)\right|\right|_{\dot{C}^{1}\left(\mathbb{R}\right)}=1$.
Thereby, schematically, our situation is the following:
\[
\widetilde{V_{1}^{\left(n\right)}}^{n}\left(t,x\right)=B_{n}\left(t\right)b_{n}\left(t\right)\underbrace{\overbrace{\varphi\left(\lambda_{n}x_{1}\right)}^{{\footnotesize \begin{matrix}\text{bounded in}\\
\dot{C}^{\alpha}\left(\mathbb{R}^{2}\right)\text{ by }\\
K_{0}\left(1\right)\lambda_{n}^{\alpha}\left|\left|\varphi\right|\right|_{\dot{C}^{1}\left(\mathbb{R}\right)}
\end{matrix}}}}_{{\footnotesize \begin{matrix}\text{bounded in}\\
L^{\infty}\left(\mathbb{R}^{2}\right)\text{ by }1
\end{matrix}}}\underbrace{\overbrace{\varphi\left(\lambda_{n}x_{2}\right)}^{{\footnotesize \begin{matrix}\text{bounded in}\\
\dot{C}^{\alpha}\left(\mathbb{R}^{2}\right)\text{ by }\\
K_{0}\left(1\right)\lambda_{n}^{\alpha}\left|\left|\varphi\right|\right|_{\dot{C}^{1}\left(\mathbb{R}\right)}
\end{matrix}}}}_{{\footnotesize \begin{matrix}\text{bounded in}\\
L^{\infty}\left(\mathbb{R}^{2}\right)\text{ by }1
\end{matrix}}}\underbrace{\overbrace{\sin\left(x_{1}\right)}^{{\footnotesize \begin{matrix}\text{bounded in}\\
\dot{C}^{\alpha}\left(\mathbb{R}^{2}\right)\text{ by }1
\end{matrix}}}}_{{\footnotesize \begin{matrix}\text{bounded in}\\
L^{\infty}\left(\mathbb{R}^{2}\right)\text{ by }1
\end{matrix}}}\underbrace{\overbrace{\cos\left(x_{2}\right)}^{{\footnotesize \begin{matrix}\text{bounded in}\\
\dot{C}^{\alpha}\left(\mathbb{R}^{2}\right)\text{ by }1
\end{matrix}}}}_{{\footnotesize \begin{matrix}\text{bounded in}\\
L^{\infty}\left(\mathbb{R}^{2}\right)\text{ by }1
\end{matrix}}}.
\]
Consequently, we arrive to
\[
\left|\left|\widetilde{V_{1}^{\left(n\right)}}^{n}\left(t,\cdot\right)\right|\right|_{\dot{C}^{\alpha}\left(\mathbb{R}^{2}\right)}\le B_{n}\left(t\right)b_{n}\left(t\right)\left[K_{0}\left(1\right)\lambda_{n}^{\alpha}\left|\left|\varphi\right|\right|_{\dot{C}^{1}\left(\mathbb{R}\right)}+K_{0}\left(1\right)\lambda_{n}^{\alpha}\left|\left|\varphi\right|\right|_{\dot{C}^{1}\left(\mathbb{R}\right)}+1+1\right].
\]
Since $\lambda_{n}\le1$ by Choice \ref{choice:psin}, we can write
\[
\left|\left|\widetilde{V_{1}^{\left(n\right)}}^{n}\left(t,\cdot\right)\right|\right|_{\dot{C}^{\alpha}\left(\mathbb{R}^{2}\right)}\le2\left(1+K\left(1\right)\left|\left|\varphi\right|\right|_{\dot{C}^{1}\left(\mathbb{R}\right)}\right)B_{n}\left(t\right)b_{n}\left(t\right).
\]
Acting analogously, we get
\[
\begin{aligned}\widetilde{V_{2}^{\left(n\right)}}^{n}\left(t,\cdot\right) & =-B_{n}\left(t\right)a_{n}\left(t\right)\underbrace{\overbrace{\varphi\left(\lambda_{n}x_{1}\right)}^{{\footnotesize \begin{matrix}\text{bounded in}\\
\dot{C}^{\alpha}\left(\mathbb{R}^{2}\right)\text{ by }\\
K_{0}\left(1\right)\lambda_{n}^{\alpha}\left|\left|\varphi\right|\right|_{\dot{C}^{1}\left(\mathbb{R}\right)}
\end{matrix}}}}_{{\footnotesize \begin{matrix}\text{bounded in}\\
L^{\infty}\left(\mathbb{R}^{2}\right)\text{ by }1
\end{matrix}}}\underbrace{\overbrace{\varphi\left(\lambda_{n}x_{2}\right)}^{{\footnotesize \begin{matrix}\text{bounded in}\\
\dot{C}^{\alpha}\left(\mathbb{R}^{2}\right)\text{ by }\\
K_{0}\left(1\right)\lambda_{n}^{\alpha}\left|\left|\varphi\right|\right|_{\dot{C}^{1}\left(\mathbb{R}\right)}
\end{matrix}}}}_{{\footnotesize \begin{matrix}\text{bounded in}\\
L^{\infty}\left(\mathbb{R}^{2}\right)\text{ by }1
\end{matrix}}}\underbrace{\overbrace{\cos\left(x_{1}\right)}^{{\footnotesize \begin{matrix}\text{bounded in}\\
\dot{C}^{\alpha}\left(\mathbb{R}^{2}\right)\text{ by }1
\end{matrix}}}}_{{\footnotesize \begin{matrix}\text{bounded in}\\
L^{\infty}\left(\mathbb{R}^{2}\right)\text{ by }1
\end{matrix}}}\underbrace{\overbrace{\sin\left(x_{2}\right)}^{{\footnotesize \begin{matrix}\text{bounded in}\\
\dot{C}^{\alpha}\left(\mathbb{R}^{2}\right)\text{ by }1
\end{matrix}}}}_{{\footnotesize \begin{matrix}\text{bounded in}\\
L^{\infty}\left(\mathbb{R}^{2}\right)\text{ by }1
\end{matrix}}},\\
\widetilde{W_{1}^{\left(n\right)}}^{n}\left(t,\cdot\right) & =B_{n}\left(t\right)b_{n}\left(t\right)\underbrace{\overbrace{\varphi\left(\lambda_{n}x_{1}\right)}^{{\footnotesize \begin{matrix}\text{bounded in}\\
\dot{C}^{\alpha}\left(\mathbb{R}^{2}\right)\text{ by }\\
K_{0}\left(1\right)\lambda_{n}^{\alpha}\left|\left|\varphi\right|\right|_{\dot{C}^{1}\left(\mathbb{R}\right)}
\end{matrix}}}}_{{\footnotesize \begin{matrix}\text{bounded in}\\
L^{\infty}\left(\mathbb{R}^{2}\right)\text{ by }1
\end{matrix}}}\underbrace{\overbrace{\varphi'\left(\lambda_{n}x_{2}\right)}^{{\footnotesize \begin{matrix}\text{bounded in}\\
\dot{C}^{\alpha}\left(\mathbb{R}^{2}\right)\text{ by }\\
K_{0}\left(1\right)\lambda_{n}^{\alpha}\left|\left|\varphi\right|\right|_{\dot{C}^{2}\left(\mathbb{R}\right)}
\end{matrix}}}}_{{\footnotesize \begin{matrix}\text{bounded in}\\
L^{\infty}\left(\mathbb{R}^{2}\right)\text{ by}\\
\left|\left|\varphi\right|\right|_{\dot{C}^{1}\left(\mathbb{R}\right)}
\end{matrix}}}\underbrace{\overbrace{\sin\left(x_{1}\right)}^{{\footnotesize \begin{matrix}\text{bounded in}\\
\dot{C}^{\alpha}\left(\mathbb{R}^{2}\right)\text{ by }1
\end{matrix}}}}_{{\footnotesize \begin{matrix}\text{bounded in}\\
L^{\infty}\left(\mathbb{R}^{2}\right)\text{ by }1
\end{matrix}}}\underbrace{\overbrace{\sin\left(x_{2}\right)}^{{\footnotesize \begin{matrix}\text{bounded in}\\
\dot{C}^{\alpha}\left(\mathbb{R}^{2}\right)\text{ by }1
\end{matrix}}}}_{{\footnotesize \begin{matrix}\text{bounded in}\\
L^{\infty}\left(\mathbb{R}^{2}\right)\text{ by }1
\end{matrix}}},\\
\widetilde{W_{2}^{\left(n\right)}}^{n}\left(t,\cdot\right) & =-B_{n}\left(t\right)a_{n}\left(t\right)\underbrace{\overbrace{\varphi'\left(\lambda_{n}x_{1}\right)}^{{\footnotesize \begin{matrix}\text{bounded in}\\
\dot{C}^{\alpha}\left(\mathbb{R}^{2}\right)\text{ by }\\
K_{0}\left(1\right)\lambda_{n}^{\alpha}\left|\left|\varphi\right|\right|_{\dot{C}^{2}\left(\mathbb{R}\right)}
\end{matrix}}}}_{{\footnotesize \begin{matrix}\text{bounded in}\\
L^{\infty}\left(\mathbb{R}^{2}\right)\text{ by}\\
\left|\left|\varphi\right|\right|_{\dot{C}^{1}\left(\mathbb{R}\right)}
\end{matrix}}}\underbrace{\overbrace{\varphi\left(\lambda_{n}x_{2}\right)}^{{\footnotesize \begin{matrix}\text{bounded in}\\
\dot{C}^{\alpha}\left(\mathbb{R}^{2}\right)\text{ by }\\
K_{0}\left(1\right)\lambda_{n}^{\alpha}\left|\left|\varphi\right|\right|_{\dot{C}^{1}\left(\mathbb{R}\right)}
\end{matrix}}}}_{{\footnotesize \begin{matrix}\text{bounded in}\\
L^{\infty}\left(\mathbb{R}^{2}\right)\text{ by }1
\end{matrix}}}\underbrace{\overbrace{\sin\left(x_{1}\right)}^{{\footnotesize \begin{matrix}\text{bounded in}\\
\dot{C}^{\alpha}\left(\mathbb{R}^{2}\right)\text{ by }1
\end{matrix}}}}_{{\footnotesize \begin{matrix}\text{bounded in}\\
L^{\infty}\left(\mathbb{R}^{2}\right)\text{ by }1
\end{matrix}}}\underbrace{\overbrace{\sin\left(x_{2}\right)}^{{\footnotesize \begin{matrix}\text{bounded in}\\
\dot{C}^{\alpha}\left(\mathbb{R}^{2}\right)\text{ by }1
\end{matrix}}}}_{{\footnotesize \begin{matrix}\text{bounded in}\\
L^{\infty}\left(\mathbb{R}^{2}\right)\text{ by }1
\end{matrix}}},
\end{aligned}
\]
and, as a consequence,
\[
\begin{aligned}\left|\left|\widetilde{V_{2}^{\left(n\right)}}^{n}\left(t,\cdot\right)\right|\right|_{\dot{C}^{\alpha}\left(\mathbb{R}^{2}\right)} & \lesssim_{\varphi}B_{n}\left(t\right)a_{n}\left(t\right),\\
\left|\left|\widetilde{W_{1}^{\left(n\right)}}^{n}\left(t,\cdot\right)\right|\right|_{\dot{C}^{\alpha}\left(\mathbb{R}^{2}\right)} & \lesssim_{\varphi}B_{n}\left(t\right)b_{n}\left(t\right),\\
\left|\left|\widetilde{W_{2}^{\left(n\right)}}^{n}\left(t,\cdot\right)\right|\right|_{\dot{C}^{\alpha}\left(\mathbb{R}^{2}\right)} & \lesssim_{\varphi}B_{n}\left(t\right)a_{n}\left(t\right).
\end{aligned}
\]

Lastly, to obtain the expressions of the statement, we need to write
$\left|\left|V_{1}^{\left(n\right)}\left(t,\cdot\right)\right|\right|_{\dot{C}^{\alpha}\left(\mathbb{R}^{2}\right)}$
in terms of $\left|\left|\widetilde{V_{1}^{\left(n\right)}}^{n}\left(t,x\right)\right|\right|_{\dot{C}^{\alpha}\left(\mathbb{R}^{2}\right)}$.
To achieve this, we will appeal to equation \eqref{eq:property Calpha composition}.
In this way,
\[
\left|\left|V_{1}^{\left(n\right)}\left(t,\cdot\right)\right|\right|_{\dot{C}^{\alpha}\left(\mathbb{R}^{2}\right)}=\left|\left|\widetilde{V_{1}^{\left(n\right)}}^{n}\left(t,\left(\phi^{\left(n\right)}\right)^{-1}\left(t,\cdot\right)\right)\right|\right|_{\dot{C}^{\alpha}\left(\mathbb{R}^{2}\right)}\lesssim\left|\left|\left(\phi^{\left(n\right)}\right)^{-1}\left(t,\cdot\right)\right|\right|_{\dot{C}^{1}\left(\mathbb{R}^{2};\mathbb{R}^{2}\right)}^{\alpha}\left|\left|\widetilde{V_{1}^{\left(n\right)}}^{n}\left(t,\cdot\right)\right|\right|_{\dot{C}^{\alpha}\left(\mathbb{R}^{2}\right)}.
\]
Making use of equation \eqref{eq:jacobian inverse}, we infer that
\[
\left|\left|\left(\phi^{\left(n\right)}\right)^{-1}\left(t,\cdot\right)\right|\right|_{\dot{C}^{1}\left(\mathbb{R}^{2};\mathbb{R}^{2}\right)}=\max\left\{ a_{n}\left(t\right),b_{n}\left(t\right)\right\} .
\]
Combining all the ingredients mentioned above, we arrive to the bound
of the statement for $\left|\left|V_{1}^{\left(n\right)}\left(t,\cdot\right)\right|\right|_{\dot{C}^{\alpha}\left(\mathbb{R}^{2}\right)}$.
The other bounds can be obtained analogously.
\end{proof}

\subsection{First look at the transport term}

Coming back to equations \eqref{eq:Boussinesq 1 con tildes} and \eqref{eq:Boussinesq 2 con tildes},
we see that the term
\[
\left(\widetilde{U^{\left(n-1\right)}}^{n}\left(t,x\right)-\frac{\partial\phi^{\left(n\right)}}{\partial t}\left(t,0\right)-\left(\begin{matrix}-\frac{1}{a_{n}\left(t\right)^{2}}\frac{\mathrm{d}a_{n}}{\mathrm{d}t}\left(t\right)x_{1}\\
-\frac{1}{b_{n}\left(t\right)^{2}}\frac{\mathrm{d}b_{n}}{\mathrm{d}t}\left(t\right)x_{2}
\end{matrix}\right)\right)
\]
can be made ``small'' if we take $\frac{\partial\phi^{\left(n\right)}}{\partial t}\left(t,0\right)$,
$a_{n}\left(t\right)$ and $b_{n}\left(t\right)$ such that
\[
\frac{\partial\phi^{\left(n\right)}}{\partial t}\left(t,0\right)+\left(\begin{matrix}-\frac{1}{a_{n}\left(t\right)^{2}}\frac{\mathrm{d}a_{n}}{\mathrm{d}t}\left(t\right)x_{1}\\
-\frac{1}{b_{n}\left(t\right)^{2}}\frac{\mathrm{d}b_{n}}{\mathrm{d}t}\left(t\right)x_{2}
\end{matrix}\right)
\]
be a first order spatial Taylor expansion of $\widetilde{U^{\left(n-1\right)}}^{n}$.
In this way, to fulfill our objective, we would need $\frac{\partial\phi^{\left(n\right)}}{\partial t}\left(t,0\right)$
to be the zeroth order term of the Taylor expansion of $\widetilde{U^{\left(n-1\right)}}^{n}$
and $\left(\begin{matrix}-\frac{1}{a_{n}\left(t\right)^{2}}\frac{\mathrm{d}a_{n}}{\mathrm{d}t}\left(t\right)x_{1}\\
-\frac{1}{b_{n}\left(t\right)^{2}}\frac{\mathrm{d}b_{n}}{\mathrm{d}t}\left(t\right)x_{2}
\end{matrix}\right)$ to be the first order term of the Taylor expansion of $\widetilde{U^{\left(n-1\right)}}^{n}$.
Nonetheless, it is not evident that this can be done out of the box,
since a generic first order spatial Taylor expansion of $\widetilde{U^{\left(n-1\right)}}^{n}\left(t,x\right)$
would have dependence on $x_{2}$ in the first component and dependence
on $x_{1}$ in the second component and these dependencies do not
appear in the term shown. Thankfully, the next Proposition tells us
that we can achieve our goal if we choose all layer centers on the
$x_{1}$ axis.
\begin{prop}
\label{prop:relation between anbn and jacobian}Let $n\in\mathbb{N}$
and consider times $t\in\left[t_{n},1\right]$, where $\left(t_{n}\right)_{n\in\mathbb{N}}$
is the sequence defined in Choice \ref{choice:time (initial)}. Suppose
that $\forall m\in\mathbb{N}$ with $m\le n$, we have $\left|\phi_{1}^{\left(n\right)}\left(t,0\right)-\phi_{1}^{\left(m\right)}\left(t,0\right)\right|\le\frac{8\pi}{a_{m}\left(t\right)}$
$\forall t\in\left[t_{n},1\right]$. Then, it is possible to have
\begin{equation}
\frac{\partial\phi^{\left(n\right)}}{\partial t}\left(t,0\right)+\left(\begin{matrix}-\frac{1}{a_{n}\left(t\right)^{2}}\frac{\mathrm{d}a_{n}}{\mathrm{d}t}\left(t\right)x_{1}\\
-\frac{1}{b_{n}\left(t\right)^{2}}\frac{\mathrm{d}b_{n}}{\mathrm{d}t}\left(t\right)x_{2}
\end{matrix}\right)=\widetilde{U^{\left(n-1\right)}}^{n}\left(t,0\right)+\mathrm{J}\widetilde{U^{\left(n-1\right)}}^{n}\left(t,0\right)\cdot\left(\begin{matrix}x_{1}\\
x_{2}
\end{matrix}\right)\quad\forall t\in\left[t_{n},1\right]\label{eq:Taylor development}
\end{equation}
provided that, $\forall t\in\left[t_{n},1\right]$,
\[
\begin{alignedat}{1}\phi_{2}^{\left(n\right)}\left(t,0\right) & =0,\\
\frac{\partial\phi_{1}^{\left(n\right)}}{\partial t}\left(t,0\right) & =\sum_{m=1}^{n-1}B_{m}\left(t\right)b_{m}\left(t\right)\sin\left(a_{m}\left(t\right)\left(\phi_{1}^{\left(n\right)}\left(t,0\right)-\phi_{1}^{\left(m\right)}\left(t,0\right)\right)\right),\\
\frac{\mathrm{d}}{\mathrm{d}t}\left(\ln\left(b_{n}\left(t\right)\right)\right) & =\sum_{m=1}^{n-1}B_{m}\left(t\right)a_{m}\left(t\right)b_{m}\left(t\right)\cos\left(a_{m}\left(t\right)\left(\phi_{1}^{\left(n\right)}\left(t,0\right)-\phi_{1}^{\left(m\right)}\left(t,0\right)\right)\right),\\
\frac{\mathrm{d}}{\mathrm{d}t}\left(a_{n}\left(t\right)b_{n}\left(t\right)\right) & =0.
\end{alignedat}
\]
\end{prop}
\begin{rem}
The assumption presented in Proposition \ref{prop:relation between anbn and jacobian}
is just a technicality that we need to guarantee that the center of
layer $n$ will never exit the zone where the cutoff functions of
past layers equal $1$. In Proposition \ref{prop:layer center never leafs cutoff area},
we will prove that the assumption is true.
\end{rem}
\begin{proof}
By definition of $U^{\left(n-1\right)}$,
\[
\begin{aligned}\widetilde{U^{\left(n-1\right)}}^{n}\left(t,x\right) & =\sum_{m=1}^{n-1}\widetilde{u^{\left(m\right)}}^{n}\left(t,x\right)=\sum_{m=1}^{n-1}u^{\left(m\right)}\left(t,\phi^{\left(n\right)}\left(t,x\right)\right)=\\
 & =\sum_{m=1}^{n-1}u^{\left(m\right)}\left(t,\phi^{\left(m\right)}\left(t,\left(\phi^{\left(m\right)}\right)^{-1}\left(t,\phi^{\left(n\right)}\left(t,x\right)\right)\right)\right)=\\
 & =\sum_{m=1}^{n-1}\widetilde{u^{\left(m\right)}}^{m}\left(t,\left(\phi^{\left(m\right)}\right)^{-1}\left(t,\phi^{\left(n\right)}\left(t,x\right)\right)\right).
\end{aligned}
\]
Using equation \eqref{eq:phin inverse}, we deduce that
\begin{equation}
\widetilde{U^{\left(n-1\right)}}^{n}\left(t,x\right)=\sum_{m=1}^{n-1}\widetilde{u^{\left(m\right)}}^{m}\left(t,\left(\begin{matrix}a_{m}\left(t\right)\left(\phi_{1}^{\left(n\right)}\left(t,x\right)-\phi_{1}^{\left(m\right)}\left(t,0\right)\right)\\
b_{m}\left(t\right)\left(\phi_{2}^{\left(n\right)}\left(t,x\right)-\phi_{2}^{\left(m\right)}\left(t,0\right)\right)
\end{matrix}\right)\right).\label{eq:phi as sum of u layers}
\end{equation}
By Proposition \ref{prop:computations vorticity}, 
\[
\begin{aligned}\widetilde{u^{\left(n\right)}}^{n}\left(t,x\right) & =B_{n}\left(t\right)\varphi\left(\lambda_{n}x_{1}\right)\varphi\left(\lambda_{n}x_{2}\right)\left(\begin{matrix}b_{n}\left(t\right)\sin\left(x_{1}\right)\cos\left(x_{2}\right)\\
-a_{n}\left(t\right)\cos\left(x_{1}\right)\sin\left(x_{2}\right)
\end{matrix}\right)+\\
 & \quad+\lambda_{n}B_{n}\left(t\right)\left(\begin{matrix}b_{n}\left(t\right)\varphi\left(\lambda_{n}x_{1}\right)\varphi'\left(\lambda_{n}x_{2}\right)\\
-a_{n}\left(t\right)\varphi'\left(\lambda_{n}x_{1}\right)\varphi\left(\lambda_{n}x_{2}\right)
\end{matrix}\right)\sin\left(x_{1}\right)\sin\left(x_{2}\right).
\end{aligned}
\]
As we are considering times $t\in\left[t_{n},1\right]$ and $\left|\phi_{1}^{\left(n\right)}\left(t,0\right)-\phi_{1}^{\left(m\right)}\left(t,0\right)\right|\le\frac{8\pi}{a_{m}\left(t\right)}$
$\forall t\in\left[t_{n},1\right]$ and $\forall m,n\in\mathbb{N}^{2}$
with $m\le n$, since $\lambda_{n}\le1$ and $\left.\varphi\right|_{\left[-8\pi,8\pi\right]}\equiv1$
by Choice \eqref{choice:psin}, when evaluating the velocities in
the expressions presented in equation \eqref{eq:phi as sum of u layers},
we will always be in the region of $\varphi$ where its value is identically
$1$. This means that, $\forall t\in\left[t_{n},1\right]$ and $\forall m\in\left\{ 1,\dots,n-1\right\} $,
we may write
\[
\widetilde{u^{\left(m\right)}}^{m}\left(t,x\right)=\left(\begin{matrix}B_{m}\left(t\right)b_{m}\left(t\right)\sin\left(x_{1}\right)\cos\left(x_{2}\right)\\
-B_{m}\left(t\right)a_{m}\left(t\right)\cos\left(x_{1}\right)\sin\left(x_{2}\right)
\end{matrix}\right).
\]
Consequently, equation \eqref{eq:phi as sum of u layers} becomes
\begin{equation}
\begin{aligned}\widetilde{U^{\left(n-1\right)}}^{n}\left(t,x\right) & =\sum_{m=1}^{n-1}\left(\begin{matrix}B_{m}\left(t\right)b_{m}\left(t\right)\sin\left(a_{m}\left(t\right)\left(\phi_{1}^{\left(n\right)}\left(t,x\right)-\phi_{1}^{\left(m\right)}\left(t,0\right)\right)\right)\cdot\\
\cdot\cos\left(b_{m}\left(t\right)\left(\phi_{2}^{\left(n\right)}\left(t,x\right)-\phi_{2}^{\left(m\right)}\left(t,0\right)\right)\right)\\
\\-B_{m}\left(t\right)a_{m}\left(t\right)\cos\left(a_{m}\left(t\right)\left(\phi_{1}^{\left(n\right)}\left(t,x\right)-\phi_{1}^{\left(m\right)}\left(t,0\right)\right)\right)\cdot\\
\cdot\sin\left(b_{m}\left(t\right)\left(\phi_{2}^{\left(n\right)}\left(t,x\right)-\phi_{2}^{\left(m\right)}\left(t,0\right)\right)\right)
\end{matrix}\right).\end{aligned}
\label{eq:velocity background layers}
\end{equation}
Now, let us compute $\mathrm{J}\widetilde{U^{\left(n-1\right)}}^{n}\left(t,x\right)$,
which is straightforward from \eqref{eq:velocity background layers}.
Taking into account that
\[
\frac{\partial\phi_{1}^{\left(n\right)}}{\partial x_{1}}\left(t,x\right)=\frac{1}{a_{n}\left(t\right)},\quad\frac{\partial\phi_{2}^{\left(n\right)}}{\partial x_{2}}=\frac{1}{b_{n}\left(t\right)},\quad\frac{\partial\phi_{1}^{\left(n\right)}}{\partial x_{2}}\left(t,x\right)=0=\frac{\partial\phi_{2}^{\left(n\right)}}{\partial x_{1}}\left(t,x\right)
\]
by Choice \ref{choice:phin}, we arrive to
\begin{equation}
\begin{aligned}\frac{\partial\widetilde{U_{1}^{\left(n-1\right)}}^{n}}{\partial x_{1}}\left(t,x\right) & =\sum_{m=1}^{n-1}\frac{B_{m}\left(t\right)b_{m}\left(t\right)a_{m}\left(t\right)}{a_{n}\left(t\right)}\cos\left(a_{m}\left(t\right)\left(\phi_{1}^{\left(n\right)}\left(t,x\right)-\phi_{1}^{\left(m\right)}\left(t,0\right)\right)\right)\cdot\\
 & \quad\cdot\cos\left(b_{m}\left(t\right)\left(\phi_{2}^{\left(n\right)}\left(t,x\right)-\phi_{2}^{\left(m\right)}\left(t,0\right)\right)\right),\\
\frac{\partial\widetilde{U_{1}^{\left(n-1\right)}}^{n}}{\partial x_{2}}\left(t,x\right) & =-\sum_{m=1}^{n-1}\frac{B_{m}\left(t\right)b_{m}\left(t\right)^{2}}{b_{n}\left(t\right)}\sin\left(a_{m}\left(t\right)\left(\phi_{1}^{\left(n\right)}\left(t,x\right)-\phi_{1}^{\left(m\right)}\left(t,0\right)\right)\right)\cdot\\
 & \quad\cdot\sin\left(b_{m}\left(t\right)\left(\phi_{2}^{\left(n\right)}\left(t,x\right)-\phi_{2}^{\left(m\right)}\left(t,0\right)\right)\right),\\
\frac{\partial\widetilde{U_{2}^{\left(n-1\right)}}^{n}}{\partial x_{1}}\left(t,x\right) & =\sum_{m=1}^{n-1}\frac{B_{m}\left(t\right)a_{m}\left(t\right)^{2}}{a_{n}\left(t\right)}\sin\left(a_{m}\left(t\right)\left(\phi_{1}^{\left(n\right)}\left(t,x\right)-\phi_{1}^{\left(m\right)}\left(t,0\right)\right)\right)\cdot\\
 & \quad\cdot\sin\left(b_{m}\left(t\right)\left(\phi_{2}^{\left(n\right)}\left(t,x\right)-\phi_{2}^{\left(m\right)}\left(t,0\right)\right)\right),\\
\frac{\partial\widetilde{U_{2}^{\left(n-1\right)}}^{n}}{\partial x_{2}}\left(t,x\right) & =-\sum_{m=1}^{n-1}\frac{B_{m}\left(t\right)a_{m}\left(t\right)b_{m}\left(t\right)}{b_{n}\left(t\right)}\cos\left(a_{m}\left(t\right)\left(\phi_{1}^{\left(n\right)}\left(t,x\right)-\phi_{1}^{\left(m\right)}\left(t,0\right)\right)\right)\cdot\\
 & \quad\cdot\cos\left(b_{m}\left(t\right)\left(\phi_{2}^{\left(n\right)}\left(t,x\right)-\phi_{2}^{\left(m\right)}\left(t,0\right)\right)\right).
\end{aligned}
\label{eq:jacobian complete}
\end{equation}
As we have remarked in the short text before this Proposition, in
order for equation \eqref{eq:Taylor development} to be true we need
the diagonal terms of $\mathrm{J}\widetilde{U^{\left(n-1\right)}}^{n}$
to vanish at $x=0$. How can we achieve that? Well, if all layer centers
were positioned in the $x_{1}$ axis, i.e., if we had $\phi_{2}^{\left(n\right)}\left(t,0\right)\equiv0$
$\forall n\in\mathbb{N}$, it is clear that the diagonal terms would
indeed vanish. Actually, we would get
\begin{equation}
\begin{aligned}\frac{\partial\widetilde{U_{1}^{\left(n-1\right)}}^{n}}{\partial x_{1}}\left(t,0\right) & =\sum_{m=1}^{n-1}\frac{B_{m}\left(t\right)b_{m}\left(t\right)a_{m}\left(t\right)}{a_{n}\left(t\right)}\cos\left(a_{m}\left(t\right)\left(\phi_{1}^{\left(n\right)}\left(t,0\right)-\phi_{1}^{\left(m\right)}\left(t,0\right)\right)\right),\\
\frac{\partial\widetilde{U_{1}^{\left(n-1\right)}}^{n}}{\partial x_{2}}\left(t,0\right) & =0,\\
\frac{\partial\widetilde{U_{2}^{\left(n-1\right)}}^{n}}{\partial x_{1}}\left(t,0\right) & =0,\\
\frac{\partial\widetilde{U_{2}^{\left(n-1\right)}}^{n}}{\partial x_{2}}\left(t,0\right) & =-\sum_{m=1}^{n-1}\frac{B_{m}\left(t\right)a_{m}\left(t\right)b_{m}\left(t\right)}{b_{n}\left(t\right)}\cos\left(a_{m}\left(t\right)\left(\phi_{1}^{\left(n\right)}\left(t,0\right)-\phi_{1}^{\left(m\right)}\left(t,0\right)\right)\right).
\end{aligned}
\label{eq:jacobian at x=00003D0}
\end{equation}
But, is it consistent to require $\phi_{2}^{\left(n\right)}\left(t,0\right)\equiv0$
$\forall n\in\mathbb{N}$? To answer this question we have to study
the zeroth order term of the expansion. If we want \eqref{eq:Taylor development}
to be satisfied, evaluating the equation at $x=0$, we deduce that
$\frac{\partial\phi^{\left(n\right)}}{\partial t}\left(t,x\right)=\widetilde{U^{\left(n-1\right)}}^{n}\left(t,0\right)$
and using equation \eqref{eq:velocity background layers} we arrive
to the following ODE system:
\[
\frac{\partial\phi^{\left(n\right)}}{\partial t}\left(t,0\right)=\begin{aligned}\widetilde{U^{\left(n-1\right)}}^{n}\left(t,0\right) & =\sum_{m=1}^{n-1}\left(\begin{matrix}B_{m}\left(t\right)b_{m}\left(t\right)\sin\left(a_{m}\left(t\right)\left(\phi_{1}^{\left(n\right)}\left(t,0\right)-\phi_{1}^{\left(m\right)}\left(t,0\right)\right)\right)\cdot\\
\cdot\cos\left(b_{m}\left(t\right)\left(\phi_{2}^{\left(n\right)}\left(t,0\right)-\phi_{2}^{\left(m\right)}\left(t,0\right)\right)\right)\\
\\-B_{m}\left(t\right)a_{m}\left(t\right)\cos\left(a_{m}\left(t\right)\left(\phi_{1}^{\left(n\right)}\left(t,0\right)-\phi_{1}^{\left(m\right)}\left(t,0\right)\right)\right)\cdot\\
\cdot\sin\left(b_{m}\left(t\right)\left(\phi_{2}^{\left(n\right)}\left(t,0\right)-\phi_{2}^{\left(m\right)}\left(t,0\right)\right)\right)
\end{matrix}\right).\end{aligned}
\]
Now, notice that $\phi_{2}^{\left(n\right)}\left(t,0\right)\equiv0$
$\forall n\in\mathbb{N}$ is indeed a solution of the second ODE.
Consequently, it is perfectly fine to demand $\phi_{2}^{\left(n\right)}\left(t,0\right)\equiv0$
$\forall n\in\mathbb{N}$. Under that assumption, the first ODE above
reduces to the one given in the statement. Concerning the jacobian
\eqref{eq:jacobian at x=00003D0}, after this choice, we have just
two non-zero components and we have two free parameters $a_{n}\left(t\right)$
and $b_{n}\left(t\right)$. Hence, if we want equation \eqref{eq:Taylor development}
to be satisfied we clearly need
\[
\begin{aligned}-\frac{1}{a_{n}\left(t\right)^{2}}\frac{\mathrm{d}a_{n}}{\mathrm{d}t}\left(t\right) & =\sum_{m=1}^{n-1}\frac{B_{m}\left(t\right)b_{m}\left(t\right)a_{m}\left(t\right)}{a_{n}\left(t\right)}\cos\left(a_{m}\left(t\right)\left(\phi_{1}^{\left(n\right)}\left(t,0\right)-\phi_{1}^{\left(m\right)}\left(t,0\right)\right)\right),\\
-\frac{1}{b_{n}\left(t\right)^{2}}\frac{\mathrm{d}b_{n}}{\mathrm{d}t}\left(t\right) & =-\sum_{m=1}^{n-1}\frac{B_{m}\left(t\right)b_{m}\left(t\right)a_{m}\left(t\right)}{b_{n}\left(t\right)}\cos\left(a_{m}\left(t\right)\left(\phi_{1}^{\left(n\right)}\left(t,0\right)-\phi_{1}^{\left(m\right)}\left(t,0\right)\right)\right).
\end{aligned}
\]
These conditions are equivalent to
\[
\begin{aligned}\frac{\mathrm{d}}{\mathrm{d}t}\left(\ln\left(a_{n}\left(t\right)\right)\right) & =-\sum_{m=1}^{n-1}B_{m}\left(t\right)b_{m}\left(t\right)a_{m}\left(t\right)\cos\left(a_{m}\left(t\right)\left(\phi_{1}^{\left(n\right)}\left(t,0\right)-\phi_{1}^{\left(m\right)}\left(t,0\right)\right)\right),\\
\frac{\mathrm{d}}{\mathrm{d}t}\left(\ln\left(b_{n}\left(t\right)\right)\right) & =\sum_{m=1}^{n-1}B_{m}\left(t\right)b_{m}\left(t\right)a_{m}\left(t\right)\cos\left(a_{m}\left(t\right)\left(\phi_{1}^{\left(n\right)}\left(t,0\right)-\phi_{1}^{\left(m\right)}\left(t,0\right)\right)\right).
\end{aligned}
\]
Summing both equations it is evident that
\[
\frac{\mathrm{d}}{\mathrm{d}t}\left(\ln\left(a_{n}\left(t\right)b_{n}\left(t\right)\right)\right)=0\iff\frac{\mathrm{d}}{\mathrm{d}t}\left(a_{n}\left(t\right)b_{n}\left(t\right)\right)=0.
\]
\end{proof}
\begin{choice}
\label{choice:anbncn}We will take $\phi^{\left(n\right)}\left(t,0\right)$,
$a_{n}\left(t\right)$ and $b_{n}\left(t\right)$ as Proposition \ref{prop:relation between anbn and jacobian}
dictates for equation \eqref{eq:Taylor development} to be satisfied.
\end{choice}
As a consequence of Proposition \ref{prop:relation between anbn and jacobian},
provided that its assumptions are true, equations \eqref{eq:Boussinesq 1 con tildes}
and \eqref{eq:Boussinesq 2 con tildes} become 
\begin{equation}
\begin{aligned} & \frac{\partial\widetilde{\omega^{\left(n\right)}}^{n}}{\partial t}\left(t,x\right)+\\
 & +\left(\widetilde{U^{\left(n-1\right)}}^{n}\left(t,x\right)-\widetilde{U^{\left(n-1\right)}}^{n}\left(t,0\right)-\mathrm{J}\widetilde{U^{\left(n-1\right)}}^{n}\left(t,0\right)\cdot\left(\begin{matrix}x_{1}\\
x_{2}
\end{matrix}\right)\right)\cdot\widetilde{\nabla}^{n}\widetilde{\omega^{\left(n\right)}}^{n}\left(t,x\right)+\\
 & +\widetilde{u^{\left(n\right)}}^{n}\left(t,x\right)\cdot\widetilde{\nabla}^{n}\widetilde{\Omega^{\left(n-1\right)}}^{n}+\widetilde{u^{\left(n\right)}}^{n}\left(t,x\right)\cdot\widetilde{\nabla}^{n}\widetilde{\omega^{\left(n\right)}}^{n}\left(t,x\right)=\\
= & \left(\begin{matrix}0 & 1\end{matrix}\right)\cdot\left[\widetilde{\nabla}^{n}\widetilde{\rho^{\left(n\right)}}^{n}\left(t,x\right)\right]+\widetilde{f_{\omega}^{\left(n\right)}}^{n}\left(t,x\right),
\end{aligned}
\label{eq:Boussinesq vorticity Taylor}
\end{equation}
\begin{equation}
\begin{aligned} & \frac{\partial\widetilde{\rho^{\left(n\right)}}^{n}}{\partial t}\left(t,x\right)+\\
 & +\left(\widetilde{U^{\left(n-1\right)}}^{n}\left(t,x\right)-\widetilde{U^{\left(n-1\right)}}^{n}\left(t,0\right)-\mathrm{J}\widetilde{U^{\left(n-1\right)}}^{n}\left(t,0\right)\cdot\left(\begin{matrix}x_{1}\\
x_{2}
\end{matrix}\right)\right)\cdot\widetilde{\nabla}^{n}\widetilde{\rho^{\left(n\right)}}^{n}\left(t,x\right)+\\
 & +\widetilde{u^{\left(n\right)}}^{n}\left(t,x\right)\cdot\widetilde{\nabla}^{n}\widetilde{P^{\left(n-1\right)}}^{n}\left(t,x\right)+\widetilde{u^{\left(n\right)}}^{n}\left(t,x\right)\cdot\widetilde{\nabla}^{n}\widetilde{\rho^{\left(n\right)}}^{n}\left(t,x\right)=\\
= & \widetilde{f_{\rho}^{\left(n\right)}}^{n}\left(t,x\right).
\end{aligned}
\label{eq:Boussinesq density Taylor}
\end{equation}

\subsection{\label{subsec:transport of the layer centers}Transport of the layer
centers}

In the last section, through Choice \ref{choice:anbncn}, we already
selected which ODEs satisfy $\phi^{\left(n\right)}\left(t,0\right)$,
$a_{n}\left(t\right)$ and $b_{n}\left(t\right)$. In this section,
we will find (without rigorous proof) the toy model for the behavior
of $\phi_{1}^{\left(n\right)}\left(t,0\right)$: a weighted sum of
degenerate half-pendula\footnote{In subsection \ref{subsec:vorticity growth mechanism} we had just
one pendulum because we were considering just two layers. When $n$
layers are considered, $n-1$ pendula appear.}.

By Choice \ref{choice:anbncn} and Proposition \ref{prop:relation between anbn and jacobian},
we have 
\begin{equation}
\frac{\partial\phi_{1}^{\left(n\right)}}{\partial t}\left(t,0\right)=\sum_{m=1}^{n-1}B_{m}\left(t\right)b_{m}\left(t\right)\sin\left(a_{m}\left(t\right)\left(\phi_{1}^{\left(n\right)}\left(t,0\right)-\phi_{1}^{\left(m\right)}\left(t,0\right)\right)\right)\quad\forall t\in\left[t_{n},1\right].\label{eq:ODE for phi1n}
\end{equation}
The following two intuitive ideas will come into play here.
\begin{enumerate}
\item For all summands $1\le m\le n-2$, we can actually replace $\phi_{1}^{\left(n\right)}\left(t,0\right)$
with $\phi_{1}^{\left(m+1\right)}\left(t,0\right)$ without changing
the solution much, because we do not expect the velocity of layer
$m$ to tell the difference between the layers $m+1$ and $n$. In
other words, we expect $a_{m}\left(t\right)\left(\phi^{\left(n\right)}\left(t,0\right)-\phi^{\left(m+1\right)}\left(t,0\right)\right)$
to be very small in comparison to $a_{m}\left(t\right)\phi_{1}^{\left(m+1\right)}\left(t,0\right)$.
\item Since we are considering times $t\in\left[t_{n},1\right]$, as we
did in subsection \ref{subsec:vorticity growth mechanism}, we should
be able to treat the parameters of past layers as constant in time.
That is, we expect to be able to approximate $B_{m}\left(t\right)\sim B_{m}\left(1\right)$,
$b_{m}\left(t\right)\sim b_{m}\left(1\right)$ and $a_{m}\left(t\right)\sim a_{m}\left(1\right)$.
\end{enumerate}
With these ideas clear, we define
\[
\Xi^{\left(n\right)}\left(t\right)\coloneqq\phi_{1}^{\left(n\right)}\left(t,0\right)-\phi_{1}^{\left(n-1\right)}\left(t,0\right)\quad\forall n\in\mathbb{N},
\]
where we are taking $\phi_{1}^{\left(0\right)}\left(t,0\right)\equiv0$.
Using equation \eqref{eq:ODE for phi1n}, we are able to deduce the
following equation for $\Xi^{\left(n\right)}$.
\begin{equation}
\begin{aligned}\frac{\mathrm{d}\Xi^{\left(n\right)}}{\mathrm{d}t}\left(t\right) & =\frac{\partial\phi_{1}^{\left(n\right)}}{\partial t}\left(t,0\right)-\frac{\partial\phi_{1}^{\left(n-1\right)}}{\partial t}\left(t,0\right)=\\
 & =B_{n-1}\left(t\right)b_{n-1}\left(t\right)\sin\left(a_{n-1}\left(t\right)\Xi^{\left(n\right)}\left(t\right)\right)+\\
 & \quad+\sum_{m=1}^{n-2}B_{m}\left(t\right)b_{m}\left(t\right)\left[\sin\left(a_{m}\left(t\right)\left(\phi_{1}^{\left(n\right)}\left(t,0\right)-\phi_{1}^{\left(m\right)}\left(t,0\right)\right)\right)+\right.\\
 & \qquad\left.-\sin\left(a_{m}\left(t\right)\left(\phi_{1}^{\left(n-1\right)}\left(t,0\right)-\phi_{1}^{\left(m\right)}\left(t,0\right)\right)\right)\right].
\end{aligned}
\label{eq:ODE JI}
\end{equation}
We shall denote the first summand by $I_{1}$ and the sum from $m=1$
to $n-2$ by $I_{2}$. Next, we will see that, if the intuitive ideas
exposed are to be true, $I_{2}$ should be small in comparison to
$I_{1}$. As $\sin\left(\cdot\right)$ is a Lipschitz function with
constant $1$, we deduce that
\[
I_{2}\sim\sum_{m=1}^{n-2}B_{m}\left(t\right)b_{m}\left(t\right)a_{m}\left(t\right)\left[\phi_{1}^{\left(n\right)}\left(t,0\right)-\phi_{1}^{\left(n-1\right)}\left(t,0\right)\right]=\sum_{m=1}^{n-2}B_{m}\left(t\right)b_{m}\left(t\right)a_{m}\left(t\right)\Xi^{\left(n\right)}\left(t\right).
\]
By looking at equation \eqref{eq:ODE JI}, notice that, once $\Xi^{\left(n\right)}\left(t\right)$
covers one period of $\sin\left(a_{n-1}\left(t\right)\Xi^{\left(n\right)}\left(t\right)\right)$,
we should have $I_{1}=0$. If $I_{1}$ really is to dominate $I_{2}$,
we should not expect $\Xi^{\left(n\right)}\left(t\right)$ to travel
much more than one period. Thereby, we expect $a_{n-1}\left(t\right)\Xi^{\left(n\right)}\left(t\right)\sim O\left(1\right)$,
which means that we should have $\Xi^{\left(n\right)}\left(t\right)\sim\frac{1}{a_{n-1}\left(t\right)}$.
With this in mind, an educated guess for the order of the sum is
\[
I_{2}\sim\frac{1}{a_{n-1}\left(t\right)}\sum_{m=1}^{n-2}B_{m}\left(t\right)a_{m}\left(t\right)b_{m}\left(t\right).
\]
On the other hand, since $\sin\left(a_{n-1}\left(t\right)\Xi^{\left(n\right)}\left(t\right)\right)\sim O\left(1\right)$,
another educated guess for the order of $I_{1}$ is
\[
I_{1}\sim B_{n-1}\left(t\right)b_{n-1}\left(t\right).
\]
Consequently, if we divide $\frac{I_{2}}{I_{1}}$, we obtain
\[
\frac{I_{2}}{I_{1}}\sim\frac{\sum_{m=1}^{n-2}B_{m}\left(t\right)a_{m}\left(t\right)b_{m}\left(t\right)}{B_{n-1}\left(t\right)a_{n-1}\left(t\right)b_{n-1}\left(t\right)}.
\]
As we can see, numerator and denominator are closely related. Suppose,
for a moment, that we are only considering $t=1$. Since we expect
the spatial scale of our layers to decrease with $n\in\mathbb{N}$,
by virtue of Choice \ref{choice:phin}, both $a_{n}\left(1\right)$
and $b_{n}\left(1\right)$ should increase in $n\in\mathbb{N}$. For
now, we have no idea about the behavior of $B_{n}\left(1\right)$,
so we do not know whether $B_{n}\left(1\right)a_{n}\left(1\right)b_{n}\left(1\right)$
increases or decreases in $\mathbb{N}$. Nevertheless, notice that
$I_{1}$ represents the effect of layer $n-1$, while $I_{2}$ represents
the effect of all past layers except for $n-1$. Our intuitive ideas
tell us that only layer $n-1$ should matter. This means that we should
have $\frac{I_{2}}{I_{1}}\ll1$. For that to happen, we clearly need,
at least, that $B_{n}\left(1\right)a_{n}\left(1\right)b_{n}\left(1\right)$
increases in $n\in\mathbb{N}$. Actually, it is even possible to deduce
some information about how fast $B_{n}\left(1\right)a_{n}\left(1\right)b_{n}\left(1\right)$
has to grow. If one takes $B_{n}\left(1\right)a_{n}\left(1\right)b_{n}\left(1\right)$
growing as a power of $n\in\mathbb{N}$, it is easy to check that
$\frac{\sum_{m=1}^{n-2}B_{m}\left(1\right)b_{m}\left(1\right)a_{m}\left(1\right)}{B_{n-1}\left(1\right)b_{n-1}\left(1\right)a_{n-1}\left(1\right)}$
actually grows in $n$. However, if we take $B_{n}\left(1\right)a_{n}\left(1\right)b_{n}\left(1\right)$
growing exponentially in $n\in\mathbb{N}$, we have $\sum_{m=1}^{n-2}B_{m}\left(1\right)b_{m}\left(1\right)a_{m}\left(1\right)\sim B_{n-2}\left(1\right)b_{n-2}\left(1\right)a_{n-2}\left(1\right)$
and, consequently, it is possible to make $\frac{\sum_{m=1}^{n-2}B_{m}\left(1\right)b_{m}\left(1\right)a_{m}\left(1\right)}{B_{n-1}\left(1\right)b_{n-1}\left(1\right)a_{n-1}\left(1\right)}$
as small as we wish. Thus, $B_{n}\left(1\right)a_{n}\left(1\right)b_{n}\left(1\right)$
should, at minimum, grow exponentially in $n\in\mathbb{N}$. This
motives the following Choice:
\begin{choice}
\label{choice:Bnanbn}We shall take $B_{n}\left(1\right)a_{n}\left(1\right)b_{n}\left(1\right)=M_{n}$
$\forall n\in\mathbb{N}$, where $\left(M_{n}\right)_{n\in\mathbb{N}}\subseteq\mathbb{R}^{+}$
grows, at least, exponentially in $n\in\mathbb{N}$.
\end{choice}
\begin{rem}
\label{rem:amplitude of vorticity explodes}According to Proposition
\ref{prop:computations vorticity}, the only zeroth order term in
$\lambda_{n}$ of the vorticity $\widetilde{\omega^{\left(n\right)}}^{n}\left(1,x\right)$
has amplitude
\[
B_{n}\left(1\right)\left(a_{n}\left(1\right)^{2}+b_{n}\left(1\right)^{2}\right)=B_{n}\left(1\right)a_{n}\left(1\right)b_{n}\left(1\right)\left(\frac{a_{n}\left(1\right)}{b_{n}\left(1\right)}+\frac{b_{n}\left(1\right)}{a_{n}\left(1\right)}\right).
\]
Using the fact that the function $z\to z+\frac{1}{z}$ is bounded
below by $2$, we deduce that
\[
B_{n}\left(1\right)\left(a_{n}\left(1\right)^{2}+b_{n}\left(1\right)^{2}\right)\ge2B_{n}\left(1\right)a_{n}\left(1\right)b_{n}\left(1\right).
\]
Now, notice that, under Choice \ref{choice:Bnanbn}, we obtain
\[
B_{n}\left(1\right)\left(a_{n}\left(1\right)^{2}+b_{n}\left(1\right)^{2}\right)\ge2M_{n},
\]
where $M_{n}$ grows, at least, exponentially in $n\in\mathbb{N}$.
Consequently, we infer that the amplitude of our vorticity blows up
at $t=1$ when $n\to\infty$, which is a good sign for a blow-up.
\end{rem}
Under Choice \ref{choice:Bnanbn}, we should be able to neglect the
term with the sum in ODE \eqref{eq:ODE for phi1n}, obtaining the
approximate ODE
\[
\frac{\mathrm{d}\Xi_{0}^{\left(n\right)}}{\mathrm{d}t}\left(t\right)=B_{n-1}\left(t\right)b_{n-1}\left(t\right)\sin\left(a_{n-1}\left(t\right)\Xi_{0}^{\left(n\right)}\left(t\right)\right).
\]
Furthermore, if we assume that $B_{n-1}\left(t\right)\sim B_{n-1}\left(1\right)$,
$b_{n-1}\left(t\right)\sim b_{n-1}\left(1\right)$ and $a_{n-1}\left(t\right)\sim b_{n-1}\left(1\right)$
(as hinted in subsection \ref{subsec:vorticity growth mechanism}),
we get to
\begin{equation}
\frac{\mathrm{d}\Xi_{0}^{\left(n\right)}}{\mathrm{d}t}\left(t\right)=B_{n-1}\left(1\right)b_{n-1}\left(1\right)\sin\left(a_{n-1}\left(1\right)\Xi_{0}^{\left(n\right)}\left(t\right)\right).\label{eq:ODE JI0}
\end{equation}
Now, let us see that, from this ODE, we can obtain the ODE of an inverted
degenerate half-pendulum (see equation \eqref{eq:degenerate half-pendulum}).
If we multiply at both sides by $a_{n-1}\left(1\right)$ and we undertake
the change of variables 
\begin{equation}
\tau=B_{n-1}\left(1\right)b_{n-1}\left(1\right)a_{n-1}\left(1\right)\left(t-t_{n}\right),\quad F_{n}\left(\tau\right)=a_{n-1}\left(1\right)\Xi_{0}^{\left(n\right)}\left(t\right),\label{eq:change of variables JI0 Fn}
\end{equation}
we arrive to the ODE
\[
\begin{aligned} & a_{n-1}\left(1\right)\frac{\mathrm{d}\Xi_{0}^{\left(n\right)}}{\mathrm{d}t}=B_{n-1}\left(1\right)b_{n-1}\left(1\right)a_{n-1}\left(1\right)\sin\left(a_{n-1}\left(1\right)\Xi_{0}^{\left(n\right)}\right)\iff\\
\iff & \frac{\mathrm{d}\left(a_{n-1}\left(1\right)\Xi_{0}^{\left(n\right)}\right)}{\mathrm{d}\left(B_{n-1}\left(1\right)b_{n-1}\left(1\right)a_{n-1}\left(1\right)t\right)}\left(t\right)=\sin\left(a_{n-1}\left(1\right)\Xi_{0}^{\left(n\right)}\left(t\right)\right)\iff\frac{\mathrm{d}F_{n}}{\mathrm{d}\tau}\left(\tau\right)=\sin\left(F_{n}\left(\tau\right)\right),
\end{aligned}
\]
which is the ODE that appears in \eqref{eq:degenerate half-pendulum}.
Moreover, $F_{n}\left(0\right)=a_{n-1}\left(1\right)\Xi_{0}^{\left(n\right)}\left(t_{n}\right)$.
Hence, we have indeed obtained that the toy model for the behavior
of $\Xi^{\left(n\right)}$ is, under an appropriate change of variables,
an inverted degenerate half-pendulum. Then, we can express $\phi_{1}^{\left(n\right)}\left(t\right)$
as a sum of $\left(\Xi_{0}^{\left(m\right)}\left(t\right)\right)_{m=1}^{n}$.
Indeed, consider
\[
\sum_{m=1}^{n}\Xi_{0}^{\left(m\right)}\left(t\right)\sim\sum_{m=1}^{n}\Xi^{\left(m\right)}\left(t\right)=\sum_{m=1}^{n}\left(\phi_{1}^{\left(m\right)}\left(t,0\right)-\phi_{1}^{\left(m-1\right)}\left(t,0\right)\right)=\phi_{1}^{\left(n\right)}\left(t,0\right)-\phi_{1}^{\left(0\right)}\left(t,0\right),
\]
where we have taken advantage of the fact that the sum is telescopic.
Recall that, by convention, we had $\phi_{1}^{\left(0\right)}\left(t,0\right)\equiv0$.
Therefore,
\begin{equation}
\phi_{1}^{\left(n\right)}\left(t,0\right)\sim\sum_{m=1}^{n}\Xi_{0}^{\left(m\right)}\left(t\right)\sim\sum_{m=1}^{n}\overbrace{\frac{1}{a_{m}\left(1\right)}}^{\text{weight}}\overbrace{\left(a_{m}\left(1\right)\Xi_{0}^{\left(m\right)}\left(t\right)\right)}^{{\footnotesize \begin{matrix}\text{inverted degenerate}\\
\text{half-pendulum}
\end{matrix}}}\label{eq:phi_1 weighted sum of half-pendula}
\end{equation}
and, consequently, $\phi_{1}^{\left(n\right)}\left(t,0\right)$ actually
behaves like a weighted sum of inverted degenerate half-pendula, provided
that $\Xi^{\left(n\right)}\left(t\right)$ and $\Xi_{0}^{\left(n\right)}\left(t\right)$
are close.

After explaining how the degenerate half-pendula appear in the toy
model of $\phi_{1}^{\left(n\right)}\left(t,0\right)$, we now dedicate
some time to studying equation \eqref{eq:degenerate half-pendulum}
in detail.
\begin{lem}
\label{lem:the good ODE}Any solution to the initial value problem
\[
\frac{\mathrm{d}F}{\mathrm{d}t}\left(t\right)=\sin\left(F\left(t\right)\right),\quad F\left(0\right)\in\left[0,\pi\right],
\]
satisfies:
\begin{enumerate}
\item $F\left(t\right)\in\left[0,\pi\right]$ $\forall t\in\mathbb{R}$.
\item ~
\[
\frac{\sin\left(F\left(t\right)\right)}{\sin\left(F\left(0\right)\right)}=2\frac{1+\cos\left(F\left(0\right)\right)}{\sin^{2}\left(F\left(0\right)\right)}\frac{1}{\mathrm{e}^{t}+\left(\frac{1+\cos\left(F\left(0\right)\right)}{\sin\left(F\left(0\right)\right)}\right)^{2}\mathrm{e}^{-t}}\quad\forall t\in\mathbb{R}.
\]
\item The time of maximum growth of $\frac{\sin\left(F\left(t\right)\right)}{\sin\left(F\left(0\right)\right)}$
is given by 
\[
t_{\max}=\ln\left(\frac{1+\cos\left(F\left(0\right)\right)}{\sin\left(F\left(0\right)\right)}\right)
\]
 and that maximum growth is
\[
\max_{t\in\mathbb{R}}\frac{\sin\left(F\left(t\right)\right)}{\sin\left(F\left(0\right)\right)}=\frac{1}{\sin\left(F\left(0\right)\right)}.
\]
\item The equation
\[
\frac{\sin\left(F\left(t\right)\right)}{\sin\left(F\left(0\right)\right)}=1
\]
has two solutions
\[
t=0,\quad t=2t_{\max}=2\ln\left(\frac{1+\cos\left(F\left(0\right)\right)}{\sin\left(F\left(0\right)\right)}\right).
\]
\item Actually, using the $t_{\max}$ introduced in point 3, we have
\[
\sin\left(F\left(t\right)\right)=\frac{1}{\cosh\left(t_{\max}-t\right)}.
\]
In figure \ref{fig:graph ideal model}, we represent $\sin\left(F\left(t\right)\right)$
for different values of $t_{\max}$.
\end{enumerate}
\end{lem}
\begin{figure}[h]
\begin{centering}
\includegraphics[width=0.8\linewidth]{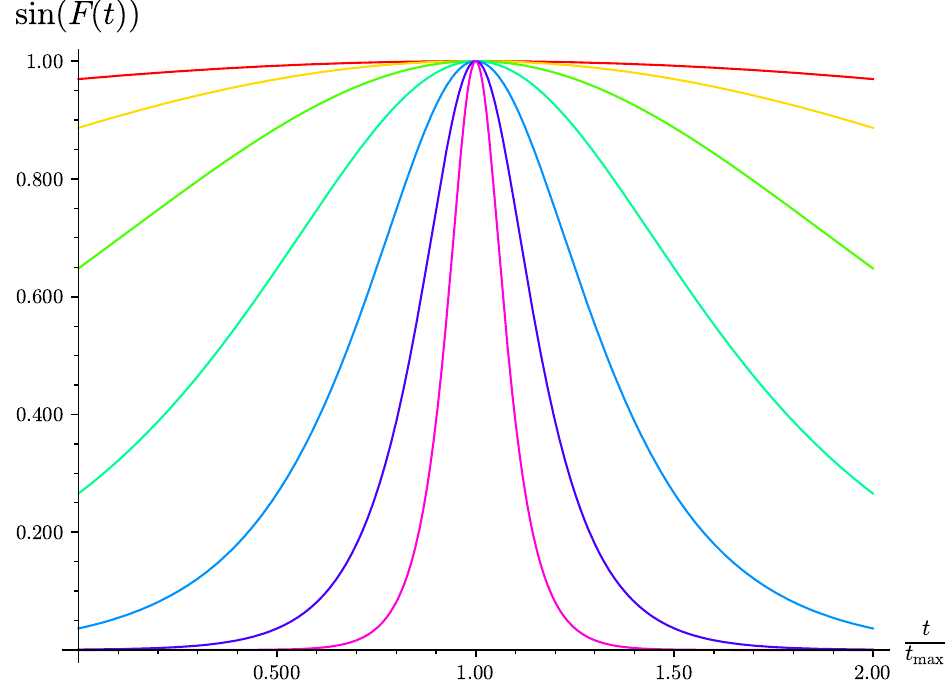}
\par\end{centering}
\caption{\label{fig:graph ideal model}Profile of $\sin\left(F\left(t\right)\right)$
for the following values of $t_{\max}$ (from red to violet): $\frac{1}{4}$,
$\frac{1}{2}$, $1$, $2$, $4$, $8$, $16$. As we will see later,
as $n$ advances, $t_{\max}$ will also grow. This means that, as
we advance through the layers, the profile of the ideal model will
become ever pointier.}

\end{figure}

\begin{proof}
~
\begin{enumerate}
\item First note that, as $\sin\left(\cdot\right)\in C^{\infty}\left(\mathbb{R}\right)$,
by the Existence and Uniqueness Theorem, the solution to the ODE must
also be $C^{\infty}$ in its interval of existence. As $\left|\frac{\mathrm{d}F}{\mathrm{d}t}\left(t\right)\right|\le1$,
we know that our solution must also exist globally in time. Uniqueness
of solutions guarantees that the flow map is one-to-one, i.e., two
solutions for different initial conditions may never acquire the same
value at the same time instant. As $F\left(t\right)\equiv0$ and $F\left(t\right)\equiv\pi$
are stationary solutions, we conclude that any solution that starts
with an initial condition $F\left(0\right)\in\left(0,\pi\right)$
can never escape the open interval $\left(0,\pi\right)$ in finite
time.
\item We ``rearrange'' the ODE of the statement to obtain
\[
\frac{\mathrm{d}F}{\sin\left(F\right)}=\mathrm{d}t.
\]
Integrating at both sides leads to
\[
\left[-\ln\left(\left|\frac{1+\cos\left(F\left(s\right)\right)}{\sin\left(F\left(s\right)\right)}\right|\right)\right]_{s=0}^{s=t}=t-0.
\]
Equivalently,
\[
\ln\left(\left|\frac{\frac{1+\cos\left(F\left(0\right)\right)}{\sin\left(F\left(0\right)\right)}}{\frac{1+\cos\left(F\left(t\right)\right)}{\sin\left(F\left(t\right)\right)}}\right|\right)=t.
\]
We may rewrite this as
\[
\ln\left(\left|\frac{\sin\left(F\left(t\right)\right)}{\sin\left(F\left(0\right)\right)}\right|\right)-\ln\left(\left|\frac{1+\cos\left(F\left(t\right)\right)}{1+\cos\left(F\left(0\right)\right)}\right|\right)=t.
\]
In particular, point 1 is telling us that $\sin\left(F\left(t\right)\right)\ge0$
$\forall t\in\mathbb{R}$. Thus, we may eliminate the $\left|\cdot\right|$
in the first logarithm. Furthermore, we always have $1+\cos\left(\cdot\right)\ge0$
and, as a consequence, we can also eliminate the $\left|\cdot\right|$
in the second logarithm. Taking exponentials at both sides, we obtain
\begin{equation}
\frac{\sin\left(F\left(t\right)\right)}{\sin\left(F\left(0\right)\right)}=\frac{1+\cos\left(F\left(t\right)\right)}{1+\cos\left(F\left(0\right)\right)}\mathrm{e}^{t}.\label{eq:f first eq}
\end{equation}
In what follows, we will do some rearrangements to express $1+\cos\left(F\left(t\right)\right)$
as a function of $\sin\left(F\left(t\right)\right)$. Squaring both
sides, we get to
\[
\begin{aligned} & \frac{\sin^{2}\left(F\left(t\right)\right)}{\sin^{2}\left(F\left(0\right)\right)}=\frac{\left(1+\cos\left(F\left(t\right)\right)\right)^{2}}{\left(1+\cos^ {}\left(F\left(0\right)\right)\right)^{2}}\mathrm{e}^{2t}\iff\\
\iff & \left(\frac{1+\cos\left(F\left(0\right)\right)}{\sin\left(F\left(0\right)\right)}\right)^{2}=\frac{1+\cos^{2}\left(F\left(t\right)\right)+2\cos\left(F\left(t\right)\right)}{\sin^{2}\left(F\left(t\right)\right)}\mathrm{e}^{2t}=\\
 & =\frac{-1+\cos^{2}\left(F\left(t\right)\right)+2+2\cos\left(F\left(t\right)\right)}{\sin^{2}\left(F\left(t\right)\right)}\mathrm{e}^{2t}.
\end{aligned}
\]
Using the fundamental trigonometric identity, we are allowed to write
\[
\begin{aligned} & \left(\frac{1+\cos\left(F\left(0\right)\right)}{\sin\left(F\left(0\right)\right)}\right)^{2}=\frac{-\sin^{2}\left(F\left(t\right)\right)+2+2\cos\left(F\left(t\right)\right)}{\sin^{2}\left(F\left(t\right)\right)}\mathrm{e}^{2t}\iff\\
\iff & \left(\frac{1+\cos\left(F\left(0\right)\right)}{\sin\left(F\left(0\right)\right)}\right)^{2}=\left(-1+2\frac{1+\cos\left(F\left(t\right)\right)}{\sin^{2}\left(F\left(t\right)\right)}\right)\mathrm{e}^{2t}.
\end{aligned}
\]
From here, we can solve for $1+\cos\left(F\left(t\right)\right)$.
Indeed, we obtain
\[
\begin{aligned} & \left(\frac{1+\cos\left(F\left(0\right)\right)}{\sin\left(F\left(0\right)\right)}\right)^{2}\mathrm{e}^{-2t}=-1+2\frac{1+\cos\left(F\left(t\right)\right)}{\sin^{2}\left(F\left(t\right)\right)}\iff\\
\iff & 1+\cos\left(F\left(t\right)\right)=\frac{1}{2}\sin^{2}\left(F\left(t\right)\right)\left[1+\left(\frac{1+\cos\left(F\left(0\right)\right)}{\sin\left(F\left(0\right)\right)}\right)^{2}\mathrm{e}^{-2t}\right].
\end{aligned}
\]
Introducing this find back in equation \eqref{eq:f first eq}, we
arrive to
\[
\begin{aligned}\frac{\sin\left(F\left(t\right)\right)}{\sin\left(F\left(0\right)\right)} & =\frac{1}{2}\sin^{2}\left(F\left(t\right)\right)\frac{1+\left(\frac{1+\cos\left(F\left(0\right)\right)}{\sin\left(F\left(0\right)\right)}\right)^{2}\mathrm{e}^{-2t}}{1+\cos\left(F\left(0\right)\right)}\mathrm{e}^{t}.\end{aligned}
\]
Lastly, we can solve for $\sin\left(F\left(t\right)\right)$:
\[
\begin{aligned}\frac{\sin\left(F\left(t\right)\right)}{\sin\left(F\left(0\right)\right)} & =2\frac{1+\cos\left(F\left(0\right)\right)}{\sin^{2}\left(F\left(0\right)\right)}\frac{1}{1+\left(\frac{1+\cos\left(F\left(0\right)\right)}{\sin\left(F\left(0\right)\right)}\right)^{2}\mathrm{e}^{-2t}}\mathrm{e}^{-t}=\\
 & =2\frac{1+\cos\left(F\left(0\right)\right)}{\sin^{2}\left(F\left(0\right)\right)}\frac{1}{\mathrm{e}^{t}+\left(\frac{1+\cos\left(F\left(0\right)\right)}{\sin\left(F\left(0\right)\right)}\right)^{2}\mathrm{e}^{-t}}.
\end{aligned}
\]
\item We want to find the point where the function
\begin{equation}
g\left(t\right)\coloneqq\mathrm{e}^{t}+\left(\frac{1+\cos\left(F\left(0\right)\right)}{\sin\left(F\left(0\right)\right)}\right)^{2}\mathrm{e}^{-t}\label{eq:defg}
\end{equation}
attains its minimum. This will clearly coincide with the point where
$\frac{\sin\left(F\left(t\right)\right)}{\sin\left(F\left(0\right)\right)}$
attains its maximum. Differentiating, we obtain
\[
\frac{\mathrm{d}g}{\mathrm{d}t}\left(t\right)=\mathrm{e}^{t}-\left(\frac{1+\cos\left(F\left(0\right)\right)}{\sin\left(F\left(0\right)\right)}\right)^{2}\mathrm{e}^{-t}.
\]
Solving for $\frac{\mathrm{d}g}{\mathrm{d}t}\left(t\right)=0$ yields
\[
\frac{\mathrm{d}g}{\mathrm{d}t}\left(t\right)=0\iff\mathrm{e}^{t}=\left(\frac{1+\cos\left(F\left(0\right)\right)}{\sin\left(F\left(0\right)\right)}\right)^{2}\mathrm{e}^{-t}\iff\mathrm{e}^{2t}=\left(\frac{1+\cos\left(F\left(0\right)\right)}{\sin\left(F\left(0\right)\right)}\right)^{2}.
\]
Taking logarithms, we deduce that
\[
t=\ln\left(\frac{1+\cos\left(F\left(0\right)\right)}{\sin\left(F\left(0\right)\right)}\right).
\]
It is clear that this critical point must be a minimum of $g$ since
it is the only critical point and $\lim_{t\to\infty}g\left(t\right)=\infty=\lim_{t\to-\infty}g\left(t\right)$.
Hence,
\[
\min_{t\in\mathbb{R}}g\left(t\right)=\frac{1+\cos\left(F\left(0\right)\right)}{\sin\left(F\left(0\right)\right)}+\left(\frac{1+\cos\left(F\left(0\right)\right)}{\sin\left(F\left(0\right)\right)}\right)^{2}\frac{1}{\frac{1+\cos\left(F\left(0\right)\right)}{\sin\left(F\left(0\right)\right)}}=2\frac{1+\cos\left(F\left(0\right)\right)}{\sin\left(F\left(0\right)\right)}.
\]
Consequently,
\[
\max_{t\in\mathbb{R}}\frac{\sin\left(F\left(t\right)\right)}{\sin\left(F\left(0\right)\right)}=2\frac{1+\cos\left(F\left(0\right)\right)}{\sin^{2}\left(F\left(0\right)\right)}\frac{1}{2\frac{1+\cos\left(F\left(0\right)\right)}{\sin\left(F\left(0\right)\right)}}=\frac{1}{\sin\left(F\left(0\right)\right)}.
\]
\item It is trivial that $t=0$ is a solution of the equation. To find the
other one, recall the function $g\left(t\right)$ we defined in equation
\eqref{eq:defg}. Looking at the expression of point 2, it is clear
that the equation $\frac{\sin\left(F\left(t\right)\right)}{\sin\left(F\left(0\right)\right)}=1$
will have multiple solutions if and only if there is a $t\neq0$ such
that $g\left(t\right)=g\left(0\right)$. Let us solve this equation.
\[
g\left(t\right)=g\left(0\right)\iff\mathrm{e}^{t}+\left(\frac{1+\cos\left(F\left(0\right)\right)}{\sin\left(F\left(0\right)\right)}\right)^{2}\mathrm{e}^{-t}=1+\left(\frac{1+\cos\left(F\left(0\right)\right)}{\sin\left(F\left(0\right)\right)}\right)^{2}.
\]
We multiply by $\mathrm{e}^{t}$ at both sides, which leads to
\[
\mathrm{e}^{2t}-\left[1+\left(\frac{1+\cos\left(F\left(0\right)\right)}{\sin\left(F\left(0\right)\right)}\right)^{2}\right]\mathrm{e}^{t}+\left(\frac{1+\cos\left(F\left(0\right)\right)}{\sin\left(F\left(0\right)\right)}\right)^{2}=0.
\]
This is a second degree equation for $\mathrm{e}^{t}$. Making use
of the solution formula, we deduce
\[
\begin{aligned}\mathrm{e}^{t} & =\frac{1+\left(\frac{1+\cos\left(F\left(0\right)\right)}{\sin\left(F\left(0\right)\right)}\right)^{2}\pm\sqrt{\left[1+\left(\frac{1+\cos\left(F\left(0\right)\right)}{\sin\left(F\left(0\right)\right)}\right)^{2}\right]^{2}-4\left(\frac{1+\cos\left(F\left(0\right)\right)}{\sin\left(F\left(0\right)\right)}\right)^{2}}}{2}=\\
 & =\frac{1+\left(\frac{1+\cos\left(F\left(0\right)\right)}{\sin\left(F\left(0\right)\right)}\right)^{2}\pm\sqrt{\left[1-\left(\frac{1+\cos\left(F\left(0\right)\right)}{\sin\left(F\left(0\right)\right)}\right)^{2}\right]^{2}}}{2}=\\
 & =\frac{1+\left(\frac{1+\cos\left(F\left(0\right)\right)}{\sin\left(F\left(0\right)\right)}\right)^{2}\pm\left[1-\left(\frac{1+\cos\left(F\left(0\right)\right)}{\sin\left(F\left(0\right)\right)}\right)^{2}\right]}{2}=\left\{ \begin{matrix}1,\\
\left(\frac{1+\cos\left(F\left(0\right)\right)}{\sin\left(F\left(0\right)\right)}\right)^{2}.
\end{matrix}\right.
\end{aligned}
\]
Taking logarithms, we conclude that
\[
t=0\quad\lor\quad t=2\ln\left(\frac{1+\cos\left(F\left(0\right)\right)}{\sin\left(F\left(0\right)\right)}\right)=2t_{\max}.
\]
\item We depart from point 2:
\[
\frac{\sin\left(F\left(t\right)\right)}{\sin\left(F\left(0\right)\right)}=2\frac{1+\cos\left(F\left(0\right)\right)}{\sin^{2}\left(F\left(0\right)\right)}\frac{1}{\mathrm{e}^{t}+\left(\frac{1+\cos\left(F\left(0\right)\right)}{\sin\left(F\left(0\right)\right)}\right)^{2}\mathrm{e}^{-t}}.
\]
Multiplying by $\sin\left(F\left(0\right)\right)$ at both sides,
we arrive to
\[
\sin\left(F\left(t\right)\right)=2\left(\frac{1+\cos\left(F\left(0\right)\right)}{\sin\left(F\left(0\right)\right)}\right)\frac{1}{\mathrm{e}^{t}+\left(\frac{1+\cos\left(F\left(0\right)\right)}{\sin\left(F\left(0\right)\right)}\right)^{2}\mathrm{e}^{-t}}.
\]
Because
\[
t_{\max}=\ln\left(\frac{1+\cos\left(F\left(0\right)\right)}{\sin\left(F\left(0\right)\right)}\right)\iff\mathrm{e}^{t_{\max}}=\frac{1+\cos\left(F\left(0\right)\right)}{\sin\left(F\left(0\right)\right)},
\]
we infer that
\[
\sin\left(F\left(t\right)\right)=2\mathrm{e}^{t_{\max}}\frac{1}{\mathrm{e}^{t}+\mathrm{e}^{2t_{\max}}\mathrm{e}^{-t}}=\frac{2}{\mathrm{e}^{\left(t-t_{\max}\right)}+\mathrm{e}^{\left(t_{\max}-t\right)}}=\frac{2}{2\cosh\left(t_{\max}-t\right)}=\frac{1}{\cosh\left(t_{\max}-t\right)}.
\]
\end{enumerate}
\end{proof}
\begin{rem}
\label{rem:relation delta M}Taking a look at the change of variables
\eqref{eq:change of variables JI0 Fn}, we can see that
\[
\sin\left(a_{n-1}\left(1\right)\Xi_{0}^{\left(n\right)}\left(t\right)\right)=\sin\left(F_{n}\left(B_{n-1}\left(1\right)a_{n-1}\left(1\right)b_{n-1}\left(1\right)\left(t-t_{n}\right)\right)\right)\quad\forall t\in\left[t_{n},1\right].
\]
If we express the equality above via the change of variables
\[
\hat{\hat{t}}=\frac{t-t_{n}}{1-t_{n}}\iff t=t_{n}+\left(1-t_{n}\right)\hat{\hat{t}},
\]
which maps the interval $\left[t_{n},1\right]$ into the interval
$\left[0,1\right]$, we obtain
\[
\sin\left(a_{n-1}\left(1\right)\Xi_{0}^{\left(n\right)}\left(t_{n}+\left(1-t_{n}\right)\hat{\hat{t}}\right)\right)=\sin\left(F_{n}\left(B_{n-1}\left(1\right)a_{n-1}\left(1\right)b_{n-1}\left(1\right)\left(1-t_{n}\right)\hat{\hat{t}}\right)\right).
\]
Using Choice \ref{choice:Bnanbn}, we deduce that
\begin{equation}
\sin\left(a_{n-1}\left(1\right)\Xi_{0}^{\left(n\right)}\left(t_{n}+\left(1-t_{n}\right)\hat{\hat{t}}\right)\right)=\sin\left(F_{n}\left(M_{n-1}\left(1-t_{n}\right)\hat{\hat{t}}\right)\right)=\frac{1}{\cosh\left(\hat{\hat{t}}_{\max}^{\left(n\right)}-M_{n-1}\left(1-t_{n}\right)\hat{\hat{t}}\right)},\label{eq:a lot of sines}
\end{equation}
where we have applied Lemma \ref{lem:the good ODE} in the last equality
and
\[
\hat{\hat{t}}_{\max}^{\left(n\right)}=\ln\left(\frac{1+\cos\left(a_{n-1}\left(1\right)\Xi_{0}^{\left(n\right)}\left(t_{n}\right)\right)}{\sin\left(a_{n-1}\left(1\right)\Xi_{0}^{\left(n\right)}\left(t_{n}\right)\right)}\right).
\]
The relationship between $\hat{\hat{t}}_{\max}^{\left(n\right)}$
and $M_{n-1}\left(1-t_{n}\right)$ is important. We distinguish five
different possible behaviors:
\begin{itemize}
\item If $M_{n-1}\left(1-t_{n}\right)<\hat{\hat{t}}_{\max}^{\left(n\right)}$,
we have $\hat{\hat{t}}_{\max}^{\left(n\right)}-M_{n-1}\left(1-t_{n}\right)\hat{\hat{t}}>0$
$\forall\hat{\hat{t}}\in\left[0,1\right]$ and, consequently,\\
$\sin\left(F_{n}\left(M_{n-1}\left(1-t_{n}\right)\hat{\hat{t}}\right)\right)$
is an increasing function that never attains the value $\sin\left(F_{n}\left(M_{n-1}\left(1-t_{n}\right)\hat{\hat{t}}\right)\right)=1$.
\item If $M_{n-1}\left(1-t_{n}\right)=\hat{\hat{t}}_{\max}^{\left(n\right)}$,
$\sin\left(F_{n}\left(M_{n-1}\left(1-t_{n}\right)\hat{\hat{t}}\right)\right)$
is an increasing function that attains its maximum value $\sin\left(F_{n}\left(M_{n-1}\left(1-t_{n}\right)\hat{\hat{t}}\right)\right)=1$
at $\hat{\hat{t}}=1$.
\item If $\hat{\hat{t}}_{\max}^{\left(n\right)}<M_{n-1}\left(1-t_{n}\right)<2\hat{\hat{t}}_{\max}^{\left(n\right)}$,
$\sin\left(F_{n}\left(M_{n-1}\left(1-t_{n}\right)\hat{\hat{t}}\right)\right)$
first increases and then decreases, attains its maximum value $\sin\left(F_{n}\left(M_{n-1}\left(1-t_{n}\right)\hat{\hat{t}}\right)\right)=1$
at a certain point in the interval $\left[0,1\right]$ and satisfies
$\sin\left(F_{n}\left(M_{n-1}\left(1-t_{n}\right)\hat{\hat{t}}\right)\right)\neq\sin\left(F_{n}\left(0\right)\right)$
if $\hat{\hat{t}}\neq0$.
\item If $M_{n-1}\left(1-t_{n}\right)=2\hat{\hat{t}}_{\max}^{\left(n\right)}$,
$\sin\left(F_{n}\left(M_{n-1}\left(1-t_{n}\right)\hat{\hat{t}}\right)\right)$
first increases and then decreases, attains its maximum value $\sin\left(F_{n}\left(M_{n-1}\left(1-t_{n}\right)\hat{\hat{t}}\right)\right)=1$
at the point $\hat{\hat{t}}=\frac{1}{2}$ and $\sin\left(F_{n}\left(0\right)\right)=\sin\left(F_{n}\left(M_{n-1}\left(1-t_{n}\right)\right)\right)$.
\item If $M_{n-1}\left(1-t_{n}\right)>2\hat{\hat{t}}_{\max}^{\left(n\right)}$,
$\sin\left(F_{n}\left(M_{n-1}\left(1-t_{n}\right)\hat{\hat{t}}\right)\right)$
first increases and then decreases, attains its maximum value $\sin\left(F_{n}\left(M_{n-1}\left(1-t_{n}\right)\hat{\hat{t}}\right)\right)=1$
at a certain point in the interval $\left[0,\frac{1}{2}\right]$ and
there exists $\hat{\hat{t}}^{*}\in\left(0,1\right]$ such that $\sin\left(F_{n}\left(0\right)\right)=\sin\left(F_{n}\left(M_{n-1}\left(1-t_{n}\right)\hat{\hat{t}}^{*}\right)\right)$.
\end{itemize}
In figure \ref{fig:relation delta M}, we can see a graphical representation
of each of these cases. When we bound the force of the density equation
in section \ref{sec:bounds for density force}, we will find out which
behavior interests us the most. This casuistry clearly suggests that,
if we wish for all layers to have the same behavior, as $M_{n-1}\left(1-t_{n}\right)$
depends on $n\in\mathbb{N}$, $\hat{\hat{t}}_{\max}^{\left(n\right)}$
should also depend on $n\in\mathbb{N}$ in the same way. Nonetheless,
without knowing the dependence of $M_{n}$ and $1-t_{n}$ on $n\in\mathbb{N}$
we cannot progress further.

\begin{figure}[H]
\begin{centering}
\includegraphics[width=0.8\linewidth]{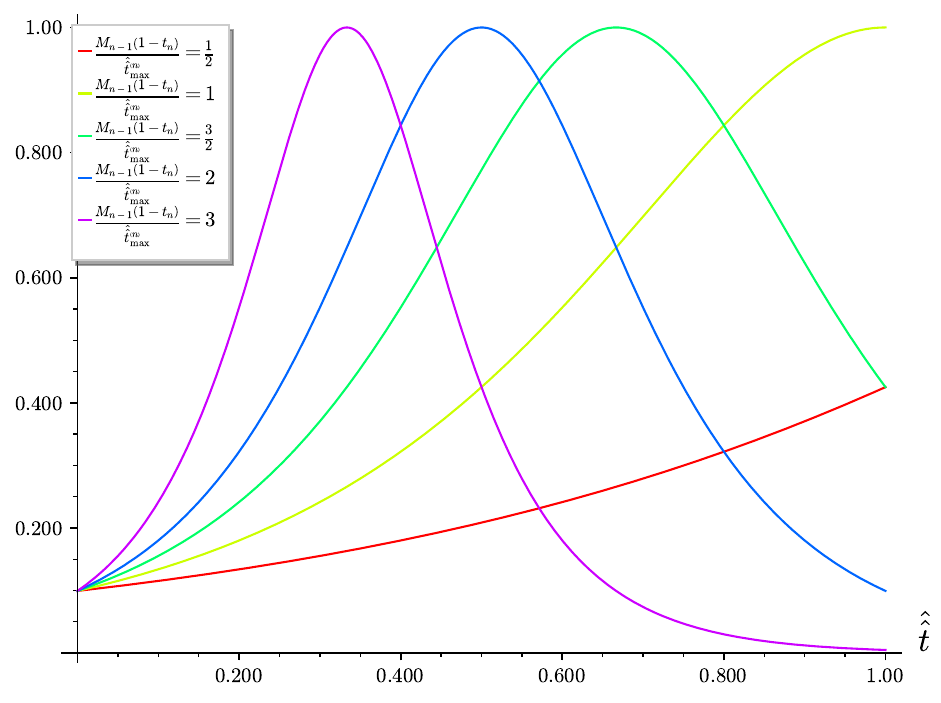}
\par\end{centering}
\caption{\label{fig:relation delta M}Profile of $\sin\left(a_{n-1}\left(1\right)\Xi_{0}^{\left(n\right)}\left(t_{n}+\left(1-t_{n}\right)\hat{\hat{t}}\right)\right)$
for $\hat{\hat{t}}_{\max}^{\left(n\right)}=3$ and different values
of the quotient $\frac{M_{n-1}\left(1-t_{n}\right)}{\hat{\hat{t}}_{\max}^{\left(n\right)}}$
as specified in the legend.}

\end{figure}
\end{rem}

\subsection{\label{subsec:cons transport term revisited}Transport term revisited}

Our next objective will be to estimate the $\left|\left|\cdot\right|\right|_{C^{\alpha}\left(\mathbb{R}^{2}\right)}$
norm of the transport term of the vorticity equation (see \eqref{eq:Boussinesq vorticity Taylor})
at $t=1$ and to see what conditions we need on $a_{n}\left(1\right)$,
$b_{n}\left(1\right)$ and $B_{n}\left(1\right)$ for the $\left|\left|\cdot\right|\right|_{C^{\alpha}\left(\mathbb{R}^{2}\right)}$
norm to be bounded. Nonetheless, first we need the following result
about $\widetilde{\nabla}^{n}\widetilde{\omega^{\left(n\right)}}^{n}\left(t,x\right)$.
\begin{prop}
\label{prop:form gradient omega}There are functions $\Gamma^{\left(n\right)},G^{\left(n\right)}:\left[0,1\right]\times\mathbb{R}^{2}\to\mathbb{R}^{2}$
such that
\[
\begin{aligned}\widetilde{\nabla}^{n}\widetilde{\omega^{\left(n\right)}}^{n}\left(t,x\right) & =\widetilde{\Gamma^{\left(n\right)}}^{n}\left(t,x\right)+\lambda_{n}\widetilde{G^{\left(n\right)}}^{n}\left(t,x\right),\\
\widetilde{\Gamma^{\left(n\right)}}^{n}\left(t,x\right) & =B_{n}\left(t\right)\left(a_{n}\left(t\right)^{2}+b_{n}\left(t\right)^{2}\right)\varphi\left(\lambda_{n}x_{1}\right)\varphi\left(\lambda_{n}x_{2}\right)\left(\begin{matrix}a_{n}\left(t\right)\cos\left(x_{1}\right)\sin\left(x_{2}\right)\\
b_{n}\left(t\right)\sin\left(x_{1}\right)\cos\left(x_{2}\right)
\end{matrix}\right).
\end{aligned}
\]
Moreover, given $\alpha\in\left(0,1\right)$, we have
\[
{\small \begin{matrix}\begin{aligned}\left|\left|\Gamma_{1}^{\left(n\right)}\left(t,\cdot\right)\right|\right|_{L^{\infty}\left(\mathbb{R}^{2}\right)} & \le B_{n}\left(t\right)a_{n}\left(t\right)\left(a_{n}\left(t\right)^{2}+b_{n}\left(t\right)^{2}\right),\\
\left|\left|G_{1}^{\left(n\right)}\left(t,\cdot\right)\right|\right|_{L^{\infty}\left(\mathbb{R}^{2}\right)} & \lesssim_{\varphi}B_{n}\left(t\right)a_{n}\left(t\right)\max\left\{ a_{n}\left(t\right)^{2},b_{n}\left(t\right)^{2}\right\} ,\\
\left|\left|\widetilde{\Gamma_{1}^{\left(n\right)}}^{n}\left(t,\cdot\right)\right|\right|_{\dot{C}^{\alpha}\left(\mathbb{R}^{2}\right)} & \lesssim_{\varphi}B_{n}\left(t\right)a_{n}\left(t\right)\left(a_{n}\left(t\right)^{2}+b_{n}\left(t\right)^{2}\right),\\
\left|\left|\widetilde{G_{1}^{\left(n\right)}}^{n}\left(t,\cdot\right)\right|\right|_{\dot{C}^{\alpha}\left(\mathbb{R}^{2}\right)} & \lesssim_{\varphi}B_{n}\left(t\right)a_{n}\left(t\right)\max\left\{ a_{n}\left(t\right)^{2},b_{n}\left(t\right)^{2}\right\} ,\\
\left|\left|\Gamma_{1}^{\left(n\right)}\left(t,\cdot\right)\right|\right|_{\dot{C}^{\alpha}\left(\mathbb{R}^{2}\right)} & \lesssim_{\varphi}B_{n}\left(t\right)a_{n}\left(t\right)\left(a_{n}\left(t\right)^{2}+b_{n}\left(t\right)^{2}\right)\cdot\\
 & \quad\cdot\max\left\{ a_{n}\left(t\right)^{\alpha},b_{n}\left(t\right)^{\alpha}\right\} ,\\
\left|\left|G_{1}^{\left(n\right)}\left(t,\cdot\right)\right|\right|_{\dot{C}^{\alpha}\left(\mathbb{R}^{2}\right)} & \lesssim_{\varphi}B_{n}\left(t\right)a_{n}\left(t\right)\cdot\\
 & \quad\cdot\max\left\{ a_{n}\left(t\right)^{2+\alpha},b_{n}\left(t\right)^{2+\alpha}\right\} ,
\end{aligned}
 &  & \begin{aligned}\left|\left|\Gamma_{2}^{\left(n\right)}\left(t,\cdot\right)\right|\right|_{L^{\infty}\left(\mathbb{R}^{2}\right)} & \le B_{n}\left(t\right)b_{n}\left(t\right)\left(a_{n}\left(t\right)^{2}+b_{n}\left(t\right)^{2}\right),\\
\left|\left|G_{2}^{\left(n\right)}\left(t,\cdot\right)\right|\right|_{L^{\infty}\left(\mathbb{R}^{2}\right)} & \lesssim_{\varphi}B_{n}\left(t\right)b_{n}\left(t\right)\max\left\{ a_{n}\left(t\right)^{2},b_{n}\left(t\right)^{2}\right\} ,\\
\left|\left|\widetilde{\Gamma_{2}^{\left(n\right)}}^{n}\left(t,\cdot\right)\right|\right|_{\dot{C}^{\alpha}\left(\mathbb{R}^{2}\right)} & \lesssim_{\varphi}B_{n}\left(t\right)b_{n}\left(t\right)\left(a_{n}\left(t\right)^{2}+b_{n}\left(t\right)^{2}\right),\\
\left|\left|\widetilde{G_{2}^{\left(n\right)}}^{n}\left(t,\cdot\right)\right|\right|_{\dot{C}^{\alpha}\left(\mathbb{R}^{2}\right)} & \lesssim_{\varphi}B_{n}\left(t\right)b_{n}\left(t\right)\max\left\{ a_{n}\left(t\right)^{2},b_{n}\left(t\right)^{2}\right\} ,\\
\left|\left|\Gamma_{2}^{\left(n\right)}\left(t,\cdot\right)\right|\right|_{\dot{C}^{\alpha}\left(\mathbb{R}^{2}\right)} & \lesssim_{\varphi}B_{n}\left(t\right)b_{n}\left(t\right)\left(a_{n}\left(t\right)^{2}+b_{n}\left(t\right)^{2}\right)\cdot\\
 & \quad\cdot\max\left\{ a_{n}\left(t\right)^{\alpha},b_{n}\left(t\right)^{\alpha}\right\} ,\\
\left|\left|G_{2}^{\left(n\right)}\left(t,\cdot\right)\right|\right|_{\dot{C}^{\alpha}\left(\mathbb{R}^{2}\right)} & \lesssim_{\varphi}B_{n}\left(t\right)b_{n}\left(t\right)\cdot\\
 & \quad\cdot\max\left\{ a_{n}\left(t\right)^{2+\alpha},b_{n}\left(t\right)^{2+\alpha}\right\} .
\end{aligned}
\end{matrix}}
\]
\end{prop}
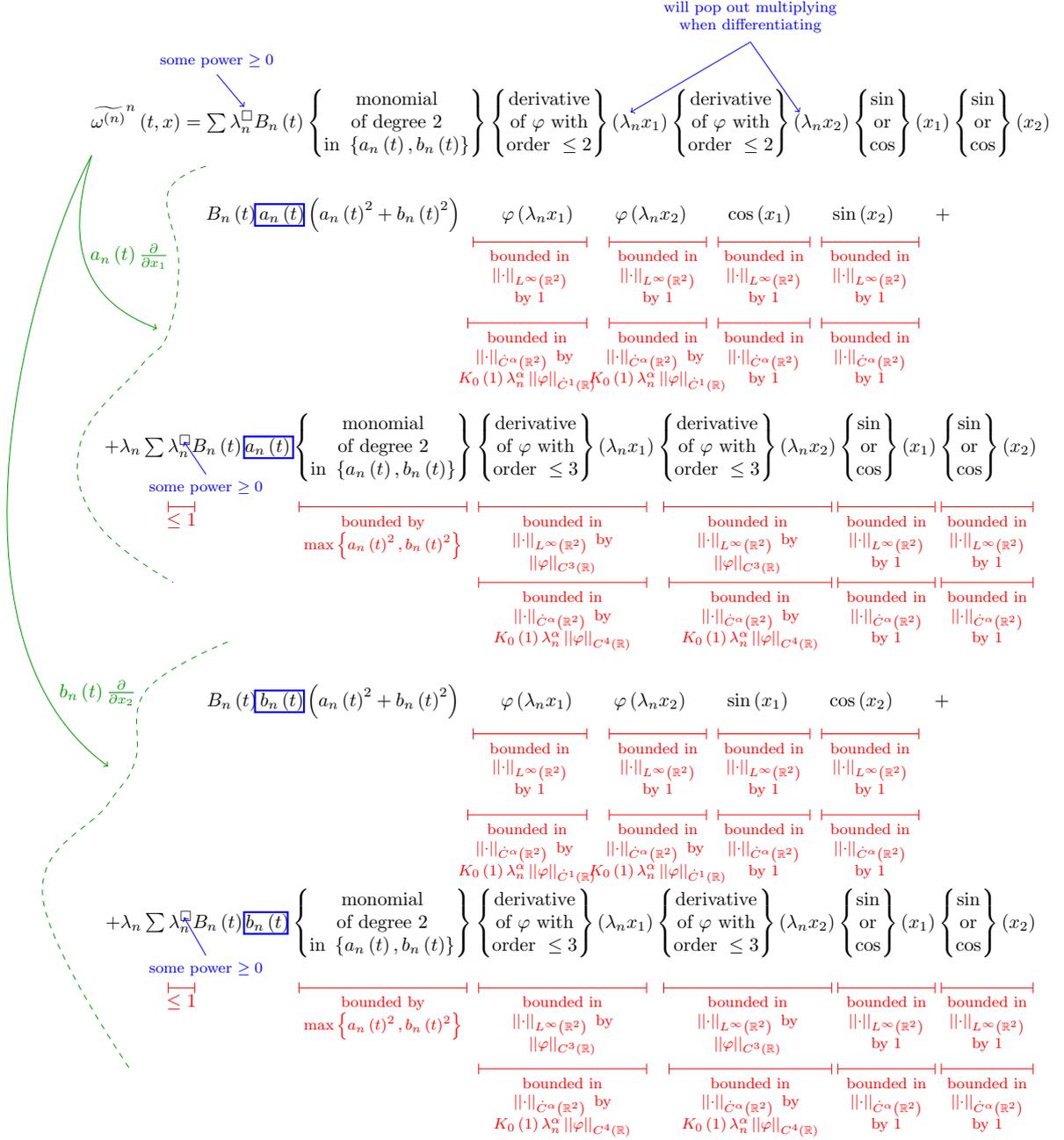
\begin{figure}[H]
\centering
\resizebox{\columnwidth}{!}{%
\begin{tikzpicture}
\node at (0,0) {$\widetilde{\omega^{\left(n\right)}}^{n}\left(t,x\right)=\sum\lambda_{n}^{\square}B_{n}\left(t\right)\left\{ \begin{matrix}\text{monomial}\\
\text{of degree }2\\
\text{in }\left\{ a_{n}\left(t\right),b_{n}\left(t\right)\right\} 
\end{matrix}\right\} \left\{ \begin{matrix}\text{derivative}\\
\text{of }\varphi\text{ with}\\
\text{order }\le2
\end{matrix}\right\} \left(\lambda_{n}x_{1}\right)\left\{ \begin{matrix}\text{derivative}\\
\text{of }\varphi\text{ with}\\
\text{order }\le2
\end{matrix}\right\} \left(\lambda_{n}x_{2}\right)\left\{ \begin{matrix}\sin\\
\text{or}\\
\cos
\end{matrix}\right\} \left(x_{1}\right)\left\{ \begin{matrix}\sin\\
\text{or}\\
\cos
\end{matrix}\right\} \left(x_{2}\right)$};
\draw [blue, ->] (-6.5,0.9) node [above] {\footnotesize some power $\ge0$}-- (-6,0.3);
\draw [blue, ->] (3.3,1.5) node (v1) {} node [above] {\footnotesize $\begin{matrix}\text{will pop out multiplying}\\
\text{when differentiating}
\end{matrix}$}-- (1.1,0.2);
\draw [blue, ->] (v1.center) -- (4.2,0.2);

\node at (0,-1.7) {$\phantom{-}B_{n}\left(t\right)a_{n}\left(t\right)\left(a_{n}\left(t\right)^{2}+b_{n}\left(t\right)^{2}\right)\qquad\varphi\left(\lambda_{n}x_{1}\right)\qquad\varphi\left(\lambda_{n}x_{2}\right)\qquad\cos\left(x_{1}\right)\qquad\sin\left(x_{2}\right)\qquad+$};
\draw [green!60!black, ->] (-8.8,-0.6) .. controls (-9.2,-1.9) and (-9.2,-3.2) .. (-7.6,-3.8) node [midway, right] {$a_{n}\left(t\right)\frac{\partial}{\partial x_{1}}$};
\draw [green!60!black, dashed] plot[smooth, tension=.7] coordinates {(-6.7,-0.8) (-7.4,-1.3) (-7.2,-2.6) (-7.4,-3.9) (-7.9,-4.5) (-9,-5.9) (-8.7,-7.4) (-7.3,-8.5)};
\draw [blue, line width = 1] (-5.8,-1.5) rectangle (-4.9,-1.9);
\draw [|-|, red](-1.8,-2.2) -- (0.3,-2.2) node [midway, below] {\footnotesize $\begin{matrix}\text{bounded in}\\
\left|\left|\cdot\right|\right|_{L^{\infty}\left(\mathbb{R}^{2}\right)}\\
\text{by }1
\end{matrix}$};
\draw [|-|, red](0.7,-2.2) -- (2.5,-2.2) node [midway, below] {\footnotesize $\begin{matrix}\text{bounded in}\\
\left|\left|\cdot\right|\right|_{L^{\infty}\left(\mathbb{R}^{2}\right)}\\
\text{by }1
\end{matrix}$};
\draw [|-|, red](2.7,-2.2) -- (4.4,-2.2) node [midway, below] {\footnotesize $\begin{matrix}\text{bounded in}\\
\left|\left|\cdot\right|\right|_{L^{\infty}\left(\mathbb{R}^{2}\right)}\\
\text{by }1
\end{matrix}$};
\draw [|-|, red](4.6,-2.2) -- (6.4,-2.2) node [midway, below] {\footnotesize $\begin{matrix}\text{bounded in}\\
\left|\left|\cdot\right|\right|_{L^{\infty}\left(\mathbb{R}^{2}\right)}\\
\text{by }1
\end{matrix}$};
\draw [|-|, red](-1.9,-3.7) -- (0.3,-3.7) node [midway, below] {\footnotesize $\begin{matrix}\text{bounded in}\\
\left|\left|\cdot\right|\right|_{\dot{C}^{\alpha}\left(\mathbb{R}^{2}\right)}\text{ by }\\
K_0\left(1\right)\lambda_{n}^{\alpha}\left|\left|\varphi\right|\right|_{\dot{C}^{1}\left(\mathbb{R}\right)}
\end{matrix}$};
\draw [|-|, red](0.7,-3.7) -- (2.5,-3.7) node [midway, below] {\footnotesize $\begin{matrix}\text{bounded in}\\
\left|\left|\cdot\right|\right|_{\dot{C}^{\alpha}\left(\mathbb{R}^{2}\right)}\text{ by }\\
K_0\left(1\right)\lambda_{n}^{\alpha}\left|\left|\varphi\right|\right|_{\dot{C}^{1}\left(\mathbb{R}\right)}
\end{matrix}$};
\draw [|-|, red](2.7,-3.7) -- (4.4,-3.7) node [midway, below] {\footnotesize $\begin{matrix}\text{bounded in}\\
\left|\left|\cdot\right|\right|_{\dot{C}^{\alpha}\left(\mathbb{R}^{2}\right)}\\
\text{by }1
\end{matrix}$};
\draw [|-|, red](4.6,-3.7) -- (6.4,-3.7) node [midway, below] {\footnotesize $\begin{matrix}\text{bounded in}\\
\left|\left|\cdot\right|\right|_{\dot{C}^{\alpha}\left(\mathbb{R}^{2}\right)}\\
\text{by }1
\end{matrix}$};

\node at (0,-6) {$+\lambda_{n}\sum\lambda_{n}^{\square}B_{n}\left(t\right)a_{n}\left(t\right)\left\{ \begin{matrix}\text{monomial}\\
\text{of degree }2\\
\text{in }\left\{ a_{n}\left(t\right),b_{n}\left(t\right)\right\} 
\end{matrix}\right\} \left\{ \begin{matrix}\text{derivative}\\
\text{of }\varphi\text{ with}\\
\text{order }\le3
\end{matrix}\right\} \left(\lambda_{n}x_{1}\right)\left\{ \begin{matrix}\text{derivative}\\
\text{of }\varphi\text{ with}\\
\text{order }\le3
\end{matrix}\right\} \left(\lambda_{n}x_{2}\right)\left\{ \begin{matrix}\sin\\
\text{or}\\
\cos
\end{matrix}\right\} \left(x_{1}\right)\left\{ \begin{matrix}\sin\\
\text{or}\\
\cos
\end{matrix}\right\} \left(x_{2}\right)$};
\draw [blue, ->] (-6.7,-6.5) node (v2) {} node [below] {\footnotesize some power $\ge0$}-- (-7.1,-5.9);
\draw [|-|, red] (-5,-7.1) -- (-1.9,-7.1) node [midway, below] {\footnotesize $\begin{matrix}\text{bounded by}\\
\max\left\{ a_{n}\left(t\right)^{2},b_{n}\left(t\right)^{2}\right\} 
\end{matrix}$};
\draw [|-|, red](-1.7,-7.1) -- (1.4,-7.1) node [midway, below] {\footnotesize $\begin{matrix}\text{bounded in}\\
\left|\left|\cdot\right|\right|_{L^{\infty}\left(\mathbb{R}^{2}\right)}\text{ by}\\
\left|\left|\varphi\right|\right|_{C^{3}\left(\mathbb{R}\right)}
\end{matrix}$};
\draw [|-|, red](1.7,-7.1) -- (4.8,-7.1) node [midway, below] {\footnotesize $\begin{matrix}\text{bounded in}\\
\left|\left|\cdot\right|\right|_{L^{\infty}\left(\mathbb{R}^{2}\right)}\text{ by}\\
\left|\left|\varphi\right|\right|_{C^{3}\left(\mathbb{R}\right)}
\end{matrix}$};
\draw [|-|, red](4.9,-7.1) -- (6.7,-7.1) node [midway, below] {\footnotesize $\begin{matrix}\text{bounded in}\\
\left|\left|\cdot\right|\right|_{L^{\infty}\left(\mathbb{R}^{2}\right)}\\
\text{by }1
\end{matrix}$};
\draw [|-|, red](6.8,-7.1) -- (8.5,-7.1) node [midway, below] {\footnotesize $\begin{matrix}\text{bounded in}\\
\left|\left|\cdot\right|\right|_{L^{\infty}\left(\mathbb{R}^{2}\right)}\\
\text{by }1
\end{matrix}$};
\draw [|-|, red](-1.7,-8.5) -- (1.4,-8.5) node [midway, below] {\footnotesize $\begin{matrix}\text{bounded in}\\
\left|\left|\cdot\right|\right|_{\dot{C}^{\alpha}\left(\mathbb{R}^{2}\right)}\text{ by}\\
K_0\left(1\right)\lambda_{n}^{\alpha}\left|\left|\varphi\right|\right|_{C^{4}\left(\mathbb{R}\right)}
\end{matrix}$};
\draw [|-|, red](1.8,-8.5) -- (4.8,-8.5) node [midway, below] {\footnotesize $\begin{matrix}\text{bounded in}\\
\left|\left|\cdot\right|\right|_{\dot{C}^{\alpha}\left(\mathbb{R}^{2}\right)}\text{ by}\\
K_0\left(1\right)\lambda_{n}^{\alpha}\left|\left|\varphi\right|\right|_{C^{4}\left(\mathbb{R}\right)}
\end{matrix}$};
\draw [|-|, red](4.9,-8.5) -- (6.7,-8.5) node [midway, below] {\footnotesize $\begin{matrix}\text{bounded in}\\
\left|\left|\cdot\right|\right|_{\dot{C}^{\alpha}\left(\mathbb{R}^{2}\right)}\\
\text{by }1
\end{matrix}$};
\draw [|-|, red](6.8,-8.5) -- (8.5,-8.5) node [midway, below] {\footnotesize $\begin{matrix}\text{bounded in}\\
\left|\left|\cdot\right|\right|_{\dot{C}^{\alpha}\left(\mathbb{R}^{2}\right)}\\
\text{by }1
\end{matrix}$};
\draw [blue, line width = 1] (-6,-5.8) rectangle (-5.1,-6.2);
\draw [|-|, red] (-7.4,-7.1) -- (-6.9,-7.1) node [midway, below] {$\le 1$};

\node at (0, -10.7) {$\phantom{-}B_{n}\left(t\right)b_{n}\left(t\right)\left(a_{n}\left(t\right)^{2}+b_{n}\left(t\right)^{2}\right)\qquad\varphi\left(\lambda_{n}x_{1}\right)\qquad\varphi\left(\lambda_{n}x_{2}\right)\qquad\sin\left(x_{1}\right)\qquad\cos\left(x_{2}\right)\qquad+$};
\draw [green!60!black, ->] (-8.8,-0.6) .. controls (-9.9,-2.6) and (-11.8,-9.5) .. (-8.5,-11.9) node [pos=0.85, right=4] {$b_{n}\left(t\right)\frac{\partial}{\partial x_{2}}$};
\draw [green!60!black, dashed] plot[smooth, tension=.7] coordinates {(-6.3,-9.6) (-7.8,-10.5) (-8,-11.8) (-8.5,-12.6) (-9.7,-14.3) (-8.1,-17.5)};
\draw [blue, line width = 1] (-5.8,-10.5) rectangle (-4.9,-10.9);
\draw [|-|, red](-1.8,-11.3) -- (0.3,-11.3) node [midway, below] {\footnotesize $\begin{matrix}\text{bounded in}\\
\left|\left|\cdot\right|\right|_{L^{\infty}\left(\mathbb{R}^{2}\right)}\\
\text{by }1
\end{matrix}$};
\draw [|-|, red](0.7,-11.3) -- (2.5,-11.3) node [midway, below] {\footnotesize $\begin{matrix}\text{bounded in}\\
\left|\left|\cdot\right|\right|_{L^{\infty}\left(\mathbb{R}^{2}\right)}\\
\text{by }1
\end{matrix}$};
\draw [|-|, red](2.7,-11.3) -- (4.4,-11.3) node [midway, below] {\footnotesize $\begin{matrix}\text{bounded in}\\
\left|\left|\cdot\right|\right|_{L^{\infty}\left(\mathbb{R}^{2}\right)}\\
\text{by }1
\end{matrix}$};
\draw [|-|, red](4.6,-11.3) -- (6.4,-11.3) node [midway, below] {\footnotesize $\begin{matrix}\text{bounded in}\\
\left|\left|\cdot\right|\right|_{L^{\infty}\left(\mathbb{R}^{2}\right)}\\
\text{by }1
\end{matrix}$};
\draw [|-|, red](-1.9,-12.8) -- (0.3,-12.8) node [midway, below] {\footnotesize $\begin{matrix}\text{bounded in}\\
\left|\left|\cdot\right|\right|_{\dot{C}^{\alpha}\left(\mathbb{R}^{2}\right)}\text{ by }\\
K_0\left(1\right)\lambda_{n}^{\alpha}\left|\left|\varphi\right|\right|_{\dot{C}^{1}\left(\mathbb{R}\right)}
\end{matrix}$};
\draw [|-|, red](0.7,-12.8) -- (2.5,-12.8) node [midway, below] {\footnotesize $\begin{matrix}\text{bounded in}\\
\left|\left|\cdot\right|\right|_{\dot{C}^{\alpha}\left(\mathbb{R}^{2}\right)}\text{ by }\\
K_0\left(1\right)\lambda_{n}^{\alpha}\left|\left|\varphi\right|\right|_{\dot{C}^{1}\left(\mathbb{R}\right)}
\end{matrix}$};
\draw [|-|, red](2.7,-12.8) -- (4.4,-12.8) node [midway, below] {\footnotesize $\begin{matrix}\text{bounded in}\\
\left|\left|\cdot\right|\right|_{\dot{C}^{\alpha}\left(\mathbb{R}^{2}\right)}\\
\text{by }1
\end{matrix}$};
\draw [|-|, red](4.6,-12.8) -- (6.4,-12.8) node [midway, below] {\footnotesize $\begin{matrix}\text{bounded in}\\
\left|\left|\cdot\right|\right|_{\dot{C}^{\alpha}\left(\mathbb{R}^{2}\right)}\\
\text{by }1
\end{matrix}$};

\node at (0,-14.8) {$+\lambda_{n}\sum\lambda_{n}^{\square}B_{n}\left(t\right)b_{n}\left(t\right)\left\{ \begin{matrix}\text{monomial}\\
\text{of degree }2\\
\text{in }\left\{ a_{n}\left(t\right),b_{n}\left(t\right)\right\} 
\end{matrix}\right\} \left\{ \begin{matrix}\text{derivative}\\
\text{of }\varphi\text{ with}\\
\text{order }\le3
\end{matrix}\right\} \left(\lambda_{n}x_{1}\right)\left\{ \begin{matrix}\text{derivative}\\
\text{of }\varphi\text{ with}\\
\text{order }\le3
\end{matrix}\right\} \left(\lambda_{n}x_{2}\right)\left\{ \begin{matrix}\sin\\
\text{or}\\
\cos
\end{matrix}\right\} \left(x_{1}\right)\left\{ \begin{matrix}\sin\\
\text{or}\\
\cos
\end{matrix}\right\} \left(x_{2}\right)$};
\draw [blue, ->] (-6.7,-15.4) node [below] {\footnotesize some power $\ge0$}-- (-7.1,-14.7);
\draw [|-|, red] (-5,-16) -- (-1.9,-16) node [midway, below] {\footnotesize $\begin{matrix}\text{bounded by}\\
\max\left\{ a_{n}\left(t\right)^{2},b_{n}\left(t\right)^{2}\right\} 
\end{matrix}$};
\draw [|-|, red](-1.7,-16) -- (1.4,-16) node [midway, below] {\footnotesize $\begin{matrix}\text{bounded in}\\
\left|\left|\cdot\right|\right|_{L^{\infty}\left(\mathbb{R}^{2}\right)}\text{ by}\\
\left|\left|\varphi\right|\right|_{C^{3}\left(\mathbb{R}\right)}
\end{matrix}$};
\draw [|-|, red](1.7,-16) -- (4.8,-16) node [midway, below] {\footnotesize $\begin{matrix}\text{bounded in}\\
\left|\left|\cdot\right|\right|_{L^{\infty}\left(\mathbb{R}^{2}\right)}\text{ by}\\
\left|\left|\varphi\right|\right|_{C^{3}\left(\mathbb{R}\right)}
\end{matrix}$};
\draw [|-|, red](4.9,-16) -- (6.7,-16) node [midway, below] {\footnotesize $\begin{matrix}\text{bounded in}\\
\left|\left|\cdot\right|\right|_{L^{\infty}\left(\mathbb{R}^{2}\right)}\\
\text{by }1
\end{matrix}$};
\draw [|-|, red](6.8,-16) -- (8.5,-16) node [midway, below] {\footnotesize $\begin{matrix}\text{bounded in}\\
\left|\left|\cdot\right|\right|_{L^{\infty}\left(\mathbb{R}^{2}\right)}\\
\text{by }1
\end{matrix}$};
\draw [|-|, red](-1.7,-17.5) -- (1.4,-17.5) node [midway, below] {\footnotesize $\begin{matrix}\text{bounded in}\\
\left|\left|\cdot\right|\right|_{\dot{C}^{\alpha}\left(\mathbb{R}^{2}\right)}\text{ by}\\
K_0\left(1\right)\lambda_{n}^{\alpha}\left|\left|\varphi\right|\right|_{C^{4}\left(\mathbb{R}\right)}
\end{matrix}$};
\draw [|-|, red](1.8,-17.5) -- (4.8,-17.5) node [midway, below] {\footnotesize $\begin{matrix}\text{bounded in}\\
\left|\left|\cdot\right|\right|_{\dot{C}^{\alpha}\left(\mathbb{R}^{2}\right)}\text{ by}\\
K_0\left(1\right)\lambda_{n}^{\alpha}\left|\left|\varphi\right|\right|_{C^{4}\left(\mathbb{R}\right)}
\end{matrix}$};
\draw [|-|, red](4.9,-17.5) -- (6.7,-17.5) node [midway, below] {\footnotesize $\begin{matrix}\text{bounded in}\\
\left|\left|\cdot\right|\right|_{\dot{C}^{\alpha}\left(\mathbb{R}^{2}\right)}\\
\text{by }1
\end{matrix}$};
\draw [|-|, red](6.8,-17.5) -- (8.5,-17.5) node [midway, below] {\footnotesize $\begin{matrix}\text{bounded in}\\
\left|\left|\cdot\right|\right|_{\dot{C}^{\alpha}\left(\mathbb{R}^{2}\right)}\\
\text{by }1
\end{matrix}$};
\draw [blue, line width = 1] (-6,-15) rectangle (-5.2,-14.6);
\draw [|-|, red] (-7.4,-16) -- (-6.9,-16) node [midway, below] {$\le 1$};
\end{tikzpicture}
}

\caption{\label{fig:proof form gradient omega}Diagrammatic summary of the
proof of Proposition \eqref{prop:form gradient omega}.}

\end{figure}

\begin{proof}
Figure \ref{fig:proof form gradient omega} contains a diagrammatic
summary of this proof. Consider the expression given in Proposition
\ref{prop:computations vorticity} for $\widetilde{\omega^{\left(n\right)}}^{n}\left(t,x\right)$.
It contains five summands and each one of them contains four parts:
a power of $\lambda_{n}$, a time dependent factor, a cutoff factor
(formed by a product of derivatives\footnote{Here, the word ``derivative'' is used in a broad sense, i.e., we
understand the function itself as a zeroth order derivative.} of $\varphi$) and a main factor (formed by a product of sines and
cosines). For example,
\[
\underbrace{\lambda_{n}^{2}}_{\text{power of }\lambda_{n}}\underbrace{B_{n}\left(t\right)a_{n}\left(t\right)^{2}}_{\text{time dependent factor}}\underbrace{\varphi''\left(\lambda_{n}x_{1}\right)\varphi\left(\lambda_{n}x_{2}\right)}_{\text{cutoff factor}}\underbrace{\sin\left(x_{1}\right)\sin\left(x_{2}\right)}_{\text{main factor}}.
\]
When applying the tilde gradient $\widetilde{\nabla}^{n}$ to one
of the summands, Leibniz's rule tells us that $\widetilde{\nabla}^{n}$
will hit either the cutoff factor or the main factor. Moreover, since
the cutoff function and its derivatives are always evaluated at $\lambda_{n}x_{1}$
or $\lambda_{n}x_{2}$, an extra $\lambda_{n}$ will ``pop out''
multiplying when $\widetilde{\nabla}^{n}$ acts on the cutoff factor.
This means that the only zeroth order term on $\lambda_{n}$ is obtained
when $\widetilde{\nabla}^{n}$ is applied to the main factor of the
first summand (the only one which contains no $\lambda_{n}$) and,
consequently, must be
\[
\widetilde{\Gamma^{\left(n\right)}}^{n}\left(t,x\right)=B_{n}\left(t\right)\left(a_{n}\left(t\right)^{2}+b_{n}\left(t\right)^{2}\right)\varphi\left(\lambda_{n}x_{1}\right)\varphi\left(\lambda_{n}x_{2}\right)\left(\begin{matrix}a_{n}\left(t\right)\cos\left(x_{1}\right)\sin\left(x_{2}\right)\\
b_{n}\left(t\right)\sin\left(x_{1}\right)\cos\left(x_{2}\right)
\end{matrix}\right).
\]
Hence, all other terms will contain at least one $\lambda_{n}$. This
means that the decomposition provided in the statement is correct
and that $\widetilde{G^{\left(n\right)}}^{n}\left(t,x\right)$ contains
no negative powers of $\lambda_{n}$. The following diagram, where
we have used equation \eqref{eq:property Calpha composition} and
the fact $\left|\left|\varphi\right|\right|_{\dot{C}^{\alpha}\left(\mathbb{R}\right)}\le\left|\left|\varphi\right|\right|_{\dot{C}^{1}\left(\mathbb{R}\right)}$
to estimate $\left|\left|\varphi\left(\lambda_{n}\cdot\right)\right|\right|_{\dot{C}^{\alpha}\left(\mathbb{R}\right)}$,
makes the computation of the $\left|\left|\cdot\right|\right|_{L^{\infty}\left(\mathbb{R}^{2}\right)}$
norm and $\left|\left|\cdot\right|\right|_{\dot{C}^{\alpha}\left(\mathbb{R}^{2}\right)}$
seminorm of $\widetilde{\Gamma^{\left(n\right)}}^{n}\left(t,x\right)$
immediate:
\[
\begin{aligned}\widetilde{\Gamma_{1}^{\left(n\right)}}^{n}\left(t,x\right) & =B_{n}\left(t\right)\left(a_{n}\left(t\right)^{2}+b_{n}\left(t\right)^{2}\right)a_{n}\left(t\right)\underbrace{\overbrace{\varphi\left(\lambda_{n}x_{1}\right)}^{{\footnotesize \begin{matrix}\text{bounded in}\\
\dot{C}^{\alpha}\left(\mathbb{R}^{2}\right)\text{ by}\\
K_{0}\left(1\right)\lambda_{n}^{\alpha}\left|\left|\varphi\right|\right|_{\dot{C}^{1}\left(\mathbb{R}\right)}
\end{matrix}}}}_{{\footnotesize {\footnotesize \begin{matrix}\text{bounded in}\\
L^{\infty}\left(\mathbb{R}^{2}\right)\text{ by }1
\end{matrix}}}}\underbrace{\overbrace{\varphi\left(\lambda_{n}x_{2}\right)}^{{\footnotesize \begin{matrix}\text{bounded in}\\
\dot{C}^{\alpha}\left(\mathbb{R}^{2}\right)\text{ by}\\
K_{0}\left(1\right)\lambda_{n}^{\alpha}\left|\left|\varphi\right|\right|_{\dot{C}^{1}\left(\mathbb{R}\right)}
\end{matrix}}}}_{{\footnotesize {\footnotesize \begin{matrix}\text{bounded in}\\
L^{\infty}\left(\mathbb{R}^{2}\right)\text{ by }1
\end{matrix}}}}\underbrace{\overbrace{\cos\left(x_{1}\right)}^{{\footnotesize \begin{matrix}\text{bounded in}\\
\dot{C}^{\alpha}\left(\mathbb{R}^{2}\right)\text{ by }1
\end{matrix}}}}_{{\footnotesize \begin{matrix}\text{bounded in}\\
L^{\infty}\left(\mathbb{R}^{2}\right)\text{ by }1
\end{matrix}}}\underbrace{\overbrace{\sin\left(x_{2}\right)}^{{\footnotesize \begin{matrix}\text{bounded in}\\
\dot{C}^{\alpha}\left(\mathbb{R}^{2}\right)\text{ by }1
\end{matrix}}}}_{{\footnotesize \begin{matrix}\text{bounded in}\\
L^{\infty}\left(\mathbb{R}^{2}\right)\text{ by }1
\end{matrix}}},\\
\widetilde{\Gamma_{2}^{\left(n\right)}}^{n}\left(t,x\right) & =B_{n}\left(t\right)\left(a_{n}\left(t\right)^{2}+b_{n}\left(t\right)^{2}\right)b_{n}\left(t\right)\underbrace{\overbrace{\varphi\left(\lambda_{n}x_{1}\right)}^{{\footnotesize \begin{matrix}\text{bounded in}\\
\dot{C}^{\alpha}\left(\mathbb{R}^{2}\right)\text{ by}\\
K_{0}\left(1\right)\lambda_{n}^{\alpha}\left|\left|\varphi\right|\right|_{\dot{C}^{1}\left(\mathbb{R}\right)}
\end{matrix}}}}_{{\footnotesize \begin{matrix}\text{bounded in}\\
L^{\infty}\left(\mathbb{R}^{2}\right)\text{ by }1
\end{matrix}}}\underbrace{\overbrace{\varphi\left(\lambda_{n}x_{2}\right)}^{{\footnotesize \begin{matrix}\text{bounded in}\\
\dot{C}^{\alpha}\left(\mathbb{R}^{2}\right)\text{ by}\\
K_{0}\left(1\right)\lambda_{n}^{\alpha}\left|\left|\varphi\right|\right|_{\dot{C}^{1}\left(\mathbb{R}\right)}
\end{matrix}}}}_{{\footnotesize \begin{matrix}\text{bounded in}\\
L^{\infty}\left(\mathbb{R}^{2}\right)\text{ by }1
\end{matrix}}}\underbrace{\overbrace{\sin\left(x_{1}\right)}^{{\footnotesize \begin{matrix}\text{bounded in}\\
\dot{C}^{\alpha}\left(\mathbb{R}^{2}\right)\text{ by }1
\end{matrix}}}}_{{\footnotesize \begin{matrix}\text{bounded in}\\
L^{\infty}\left(\mathbb{R}^{2}\right)\text{ by }1
\end{matrix}}}\underbrace{\overbrace{\cos\left(x_{2}\right)}^{{\footnotesize \begin{matrix}\text{bounded in}\\
\dot{C}^{\alpha}\left(\mathbb{R}^{2}\right)\text{ by }1
\end{matrix}}}}_{{\footnotesize \begin{matrix}\text{bounded in}\\
L^{\infty}\left(\mathbb{R}^{2}\right)\text{ by }1
\end{matrix}}}.
\end{aligned}
\]
As the $\left|\left|\cdot\right|\right|_{L^{\infty}\left(\mathbb{R}^{2}\right)}$
norm is invariant under diffeomorphisms, with the scheme above, we
already obtain the $\left|\left|\cdot\right|\right|_{L^{\infty}\left(\mathbb{R}^{2}\right)}$
norms we sought after. To obtain $\left|\left|\Gamma_{1}^{\left(n\right)}\left(t,\cdot\right)\right|\right|_{\dot{C}^{\alpha}\left(\mathbb{R}^{2}\right)}$
and $\left|\left|\Gamma_{2}^{\left(n\right)}\left(t,\cdot\right)\right|\right|_{\dot{C}^{\alpha}\left(\mathbb{R}^{2}\right)}$
from $\left|\left|\widetilde{\Gamma_{1}^{\left(n\right)}}^{n}\left(t,\cdot\right)\right|\right|_{\dot{C}^{\alpha}\left(\mathbb{R}^{2}\right)}$
and $\left|\left|\widetilde{\Gamma_{2}^{\left(n\right)}}^{n}\left(t,\cdot\right)\right|\right|_{\dot{C}^{\alpha}\left(\mathbb{R}^{2}\right)}$,
we turn to equation \eqref{eq:property Calpha composition}, which
implies that
\[
\left|\left|\Gamma_{1}^{\left(n\right)}\left(t,\cdot\right)\right|\right|_{\dot{C}^{\alpha}\left(\mathbb{R}^{2}\right)}=\left|\left|\widetilde{\Gamma_{1}^{\left(n\right)}}^{n}\left(t,\left(\phi^{\left(n\right)}\right)^{-1}\left(t,\cdot\right)\right)\right|\right|_{\dot{C}^{\alpha}\left(\mathbb{R}^{2}\right)}\lesssim\left|\left|\left(\phi^{\left(n\right)}\right)^{-1}\left(t,\cdot\right)\right|\right|_{\dot{C}^{1}\left(\mathbb{R}^{2};\mathbb{R}^{2}\right)}^{\alpha}\left|\left|\widetilde{\Gamma_{1}^{\left(n\right)}}^{n}\left(t,\cdot\right)\right|\right|_{\dot{C}^{\alpha}\left(\mathbb{R}^{2}\right)},
\]
where, thanks to equation \eqref{eq:jacobian inverse}, we know that
\[
\left|\left|\left(\phi^{\left(n\right)}\right)^{-1}\left(t,\cdot\right)\right|\right|_{\dot{C}^{1}\left(\mathbb{R}^{2};\mathbb{R}^{2}\right)}=\max\left\{ a_{n}\left(t\right),b_{n}\left(t\right)\right\} .
\]
With all these ingredients, we arrive to the bound of the statement.
The case of $\left|\left|\Gamma_{2}^{\left(n\right)}\left(t,\cdot\right)\right|\right|_{\dot{C}^{\alpha}\left(\mathbb{R}^{2}\right)}$
is completely analogous.

As is the case with $\widetilde{\omega^{\left(n\right)}}^{n}\left(t,x\right)$,
$\widetilde{G_{1}^{\left(n\right)}}^{n}\left(t,x\right)$ and $\widetilde{G_{2}^{\left(n\right)}}^{n}\left(t,x\right)$
will contain multiple summands and each one of them may be split in
four factors: a power of $\lambda_{n}$, a time dependent factor,
a cutoff factor and a main factor. Since $\lambda_{n}\le1$, the $\lambda_{n}$
factor can always be bounded by $1$. Concerning the time factors,
because $\widetilde{\nabla}^{n}$ multiplies the derivative with respect
to $x_{1}$ by $a_{n}\left(t\right)$ and the derivative with respect
to $x_{2}$ by $b_{n}\left(t\right)$, the time factors of the summands
of $\widetilde{G_{1}^{\left(n\right)}}^{n}\left(t,x\right)$ will
be $B_{n}\left(t\right)a_{n}\left(t\right)$ times a monomial of degree
$2$ in the variables $a_{n}\left(t\right)$ and $b_{n}\left(t\right)$.
Consequently, we can bound the time factor of any summand of $\widetilde{G_{1}^{\left(n\right)}}^{n}\left(t,x\right)$
by $B_{n}\left(t\right)a_{n}\left(t\right)\max\left\{ a_{n}\left(t\right)^{2},b_{n}\left(t\right)^{2}\right\} $.
Analogously, the time factor of one summand of $\widetilde{G_{2}^{\left(n\right)}}^{n}\left(t,x\right)$
will be $B_{n}\left(t\right)b_{n}\left(t\right)$ times a monomial
of degree $2$ in the variables $a_{n}\left(t\right)$ and $b_{n}\left(t\right)$.
Therefore, we can bound the time factor of any summand of $\widetilde{G_{2}^{\left(n\right)}}^{n}\left(t,x\right)$
by $B_{n}\left(t\right)b_{n}\left(t\right)\max\left\{ a_{n}\left(t\right)^{2},b_{n}\left(t\right)^{2}\right\} $.
Regarding the main factor, we know that, when $\widetilde{\nabla}^{n}$
is applied to it, we will still obtain a product of sines and cosines.
Furthermore, since the function $\varphi$ and its derivatives of
order $\le3$ (those are the ones that appear in $\widetilde{G^{\left(n\right)}}^{n}\left(t,x\right)$)
are bounded in $L^{\infty}\left(\mathbb{R}^{2}\right)$, we deduce
that the cutoff factor is bounded in $L^{\infty}\left(\mathbb{R}^{2}\right)$
by $\left|\left|\varphi\right|\right|_{C^{3}\left(\mathbb{R}\right)}^{2}$,
as this term contains two derivatives of $\varphi$. On the other
hand, we know that a product of sines and cosines is bounded in $L^{\infty}\left(\mathbb{R}^{2}\right)$
by $1$ and, consequently, the main factor is bounded by $1$. Thus,
we have the following picture:
\[
\begin{aligned} & \left|\left|\text{one summand of }\widetilde{G_{1}^{\left(n\right)}}^{n}\left(t,\cdot\right)\text{ or }\widetilde{G_{2}^{\left(n\right)}}^{n}\left(t,\cdot\right)\right|\right|_{L^{\infty}\left(\mathbb{R}^{2}\right)}\le\\
\le & \underbrace{\left(\text{power of }\lambda_{n}\right)}_{\le1}\underbrace{\left(\text{time factor}\right)}_{{\footnotesize \begin{matrix}\le B_{n}\left(t\right)a_{n}\left(t\right)\max\left\{ a_{n}\left(t\right)^{2},b_{n}\left(t\right)^{2}\right\} \text{ for }\widetilde{G_{1}^{\left(n\right)}}^{n}\\
\le B_{n}\left(t\right)b_{n}\left(t\right)\max\left\{ a_{n}\left(t\right)^{2},b_{n}\left(t\right)^{2}\right\} \text{ for }\widetilde{G_{2}^{\left(n\right)}}^{n}
\end{matrix}}}\underbrace{\left|\left|\text{cutoff factor}\right|\right|_{L^{\infty}\left(\mathbb{R}^{2}\right)}}_{\le\left|\left|\varphi\right|\right|_{C^{3}\left(\mathbb{R}\right)}^{2}}\underbrace{\left|\left|\text{main factor}\right|\right|_{L^{\infty}\left(\mathbb{R}^{2}\right)}}_{\le1}.
\end{aligned}
\]
With this information, we may conclude that
\[
\begin{aligned}\left|\left|\widetilde{G_{1}^{\left(n\right)}}^{n}\left(t,\cdot\right)\right|\right|_{L^{\infty}\left(\mathbb{R}^{2}\right)} & \lesssim_{\varphi}B_{n}\left(t\right)a_{n}\left(t\right)\max\left\{ a_{n}\left(t\right)^{2},b_{n}\left(t\right)^{2}\right\} ,\\
\left|\left|\widetilde{G_{2}^{\left(n\right)}}^{n}\left(t,\cdot\right)\right|\right|_{L^{\infty}\left(\mathbb{R}^{2}\right)} & \lesssim_{\varphi}B_{n}\left(t\right)b_{n}\left(t\right)\max\left\{ a_{n}\left(t\right)^{2},b_{n}\left(t\right)^{2}\right\} .
\end{aligned}
\]
Given that the $\left|\left|\cdot\right|\right|_{L^{\infty}\left(\mathbb{R}^{2}\right)}$
norm is invariant under diffeomorphisms of the domain, the same bounds
must be true for $\left|\left|G_{1}^{\left(n\right)}\left(t,\cdot\right)\right|\right|_{L^{\infty}\left(\mathbb{R}^{2}\right)}$
and $\left|\left|G_{2}^{\left(n\right)}\left(t,\cdot\right)\right|\right|_{L^{\infty}\left(\mathbb{R}^{2}\right)}$,
respectively.

Let us continue with the $\left|\left|\cdot\right|\right|_{\dot{C}^{\alpha}\left(\mathbb{R}^{2}\right)}$
norm. By equation \eqref{eq:property Calpha multiplication}, when
trying to compute $\left|\left|\widetilde{G_{1}^{\left(n\right)}}^{n}\left(t,\cdot\right)\right|\right|_{\dot{C}^{\alpha}\left(\mathbb{R}^{2}\right)}$
or $\left|\left|\widetilde{G_{2}^{\left(n\right)}}^{n}\left(t,\cdot\right)\right|\right|_{\dot{C}^{\alpha}\left(\mathbb{R}^{2}\right)}$,
we will obtain two summands for each summand of $\widetilde{G_{1}^{\left(n\right)}}\left(t,x\right)$
or $\widetilde{G_{2}^{\left(n\right)}}\left(t,x\right)$. In one of
them, the $\left|\left|\cdot\right|\right|_{\dot{C}^{\alpha}\left(\mathbb{R}^{d}\right)}$
seminorm will be applied to the cutoff factor and, in the other, it
will be applied to the main factor. In the cutoff factor, the derivatives
of $\varphi$ will be evaluated in $\lambda_{n}x_{1}$ or $\lambda_{n}x_{2}$.
Since $\varphi\in C_{c}^{\infty}\left(\mathbb{R}\right)$, we know
that the $\left|\left|\cdot\right|\right|_{\dot{C}^{\gamma}\left(\mathbb{R}\right)}$
seminorm of $\varphi$ and all its derivatives of order $\le3$ are
uniformly bounded in $\gamma\in\left(0,1\right)$ by $\left|\left|\varphi\right|\right|_{C^{4}\left(\mathbb{R}\right)}$.
(Note that the cutoff factor contains derivatives of $\varphi$ with
maximum order $3$). This means that, when calculating the $\left|\left|\cdot\right|\right|_{\dot{C}^{\alpha}\left(\mathbb{R}\right)}$
seminorm of the cutoff factor, we will obtain two summands (because
the cutoff factor is the product of two derivatives of $\varphi$),
where, making use of how the $\left|\left|\cdot\right|\right|_{\dot{C}^{\alpha}\left(\mathbb{R}\right)}$
changes under morphisms of the domain (see equation \eqref{eq:property Calpha composition}),
each one may be bounded by $\lambda_{n}^{\alpha}K_{0}\left(1\right)\left|\left|\varphi\right|\right|_{C^{4}\left(\mathbb{R}\right)}\left|\left|\varphi\right|\right|_{C^{3}\left(\mathbb{R}\right)}$.
As $\lambda_{n}\le1$, it can be assured that the $\left|\left|\cdot\right|\right|_{\dot{C}^{\alpha}\left(\mathbb{R}\right)}$
seminorm of the cutoff factor is bounded by $2K_{0}\left(1\right)\left|\left|\varphi\right|\right|_{C^{4}\left(\mathbb{R}\right)}\left|\left|\varphi\right|\right|_{C^{3}\left(\mathbb{R}\right)}$.
On the other hand, since $\left|\left|\sin\left(\cdot\right)\right|\right|_{\dot{C}^{\gamma}\left(\mathbb{R}\right)}=1=\left|\left|\cos\left(\cdot\right)\right|\right|_{\dot{C}^{\gamma}\left(\mathbb{R}\right)}=\left|\left|\sin\left(\cdot\right)\right|\right|_{L^{\infty}\left(\mathbb{R}\right)}=\left|\left|\cos\left(\cdot\right)\right|\right|_{L^{\infty}\left(\mathbb{R}\right)}$
$\forall\gamma\in\left(0,1\right)$, following a similar argument,
we may affirm that the $\left|\left|\cdot\right|\right|_{\dot{C}^{\alpha}\left(\mathbb{R}\right)}$
seminorm of the main factor is bounded by $2$. To summarize, we have
\[
\begin{aligned} & \left|\left|\text{one summand of }\widetilde{G_{1}^{\left(n\right)}}\left(t,x\right)\text{ or }\widetilde{G_{2}^{\left(n\right)}}\left(t,x\right)\right|\right|_{\dot{C}^{\alpha}\left(\mathbb{R}^{2}\right)}\le\\
\le & \underbrace{\left(\text{power of }\lambda_{n}\right)}_{\le1}\underbrace{\left(\text{time factor}\right)}_{{\footnotesize \begin{matrix}\le B_{n}\left(t\right)a_{n}\left(t\right)\max\left\{ a_{n}\left(t\right)^{2},b_{n}\left(t\right)^{2}\right\} \text{ for }\widetilde{G_{1}^{\left(n\right)}}\\
\le B_{n}\left(t\right)b_{n}\left(t\right)\max\left\{ a_{n}\left(t\right)^{2},b_{n}\left(t\right)^{2}\right\} \text{ for }\widetilde{G_{2}^{\left(n\right)}}
\end{matrix}}}\underbrace{\left|\left|\text{cutoff factor}\right|\right|_{\dot{C}^{\alpha}\left(\mathbb{R}^{2}\right)}}_{\le2K_{0}\left(1\right)\left|\left|\varphi\right|\right|_{C^{4}\left(\mathbb{R}\right)}\left|\left|\varphi\right|\right|_{C^{3}\left(\mathbb{R}\right)}}\underbrace{\left|\left|\text{main factor}\right|\right|_{L^{\infty}\left(\mathbb{R}^{2}\right)}}_{\le1}+\\
 & +\underbrace{\left(\text{power of }\lambda_{n}\right)}_{\le1}\underbrace{\left(\text{time factor}\right)}_{{\footnotesize \begin{matrix}\le B_{n}\left(t\right)a_{n}\left(t\right)\max\left\{ a_{n}\left(t\right)^{2},b_{n}\left(t\right)^{2}\right\} \text{ for }\widetilde{G_{1}^{\left(n\right)}}\\
\le B_{n}\left(t\right)b_{n}\left(t\right)\max\left\{ a_{n}\left(t\right)^{2},b_{n}\left(t\right)^{2}\right\} \text{ for }\widetilde{G_{2}^{\left(n\right)}}
\end{matrix}}}\underbrace{\left|\left|\text{cutoff factor}\right|\right|_{L^{\infty}\left(\mathbb{R}^{2}\right)}}_{\le\left|\left|\varphi\right|\right|_{C^{3}\left(\mathbb{R}\right)}^{2}}\underbrace{\left|\left|\text{main factor}\right|\right|_{\dot{C}^{\alpha}\left(\mathbb{R}^{2}\right)}.}_{\le2}
\end{aligned}
\]
 Consequently, 
\[
\begin{aligned}\left|\left|\widetilde{G_{1}^{\left(n\right)}}\left(t,\cdot\right)\right|\right|_{\dot{C}^{\alpha}\left(\mathbb{R}^{2}\right)} & \lesssim_{\varphi}B_{n}\left(t\right)a_{n}\left(t\right)\max\left\{ a_{n}\left(t\right)^{2},b_{n}\left(t\right)^{2}\right\} ,\\
\left|\left|\widetilde{G_{2}^{\left(n\right)}}\left(t,\cdot\right)\right|\right|_{\dot{C}^{\alpha}\left(\mathbb{R}^{2}\right)} & \lesssim_{\varphi}B_{n}\left(t\right)b_{n}\left(t\right)\max\left\{ a_{n}\left(t\right)^{2},b_{n}\left(t\right)^{2}\right\} .
\end{aligned}
\]
Lastly, taking into account how the $\left|\left|\cdot\right|\right|_{\dot{C}^{\alpha}\left(\mathbb{R}^{2}\right)}$
changes under morphisms of the domain (see equation \eqref{eq:property Calpha composition}),
we deduce that
\[
\left|\left|G_{1}^{\left(n\right)}\left(t,\cdot\right)\right|\right|_{\dot{C}^{\alpha}\left(\mathbb{R}^{2}\right)}=\left|\left|\widetilde{G_{1}^{\left(n\right)}}^{n}\left(t,\left(\phi^{\left(n\right)}\right)^{-1}\left(t,\cdot\right)\right)\right|\right|_{\dot{C}^{\alpha}\left(\mathbb{R}^{2}\right)}\lesssim\left|\left|\left(\phi^{\left(n\right)}\right)^{-1}\left(t,\cdot\right)\right|\right|_{\dot{C}^{1}\left(\mathbb{R}^{2};\mathbb{R}^{2}\right)}^{\alpha}\left|\left|\widetilde{G_{1}^{\left(n\right)}}^{n}\left(t,\cdot\right)\right|\right|_{\dot{C}^{\alpha}\left(\mathbb{R}^{2}\right)}.
\]
Employing equation \eqref{eq:jacobian inverse}, we arrive to
\[
\left|\left|\left(\phi^{\left(n\right)}\right)^{-1}\left(t,\cdot\right)\right|\right|_{\dot{C}^{1}\left(\mathbb{R}^{2};\mathbb{R}^{2}\right)}=\max\left\{ a_{n}\left(t\right),b_{n}\left(t\right)\right\} 
\]
and, as a consequence,
\[
\left|\left|G_{1}^{\left(n\right)}\left(t,\cdot\right)\right|\right|_{\dot{C}^{\alpha}\left(\mathbb{R}^{2}\right)}\lesssim_{\varphi}B_{n}\left(t\right)a_{n}\left(t\right)\max\left\{ a_{n}\left(t\right)^{2+\alpha},b_{n}\left(t\right)^{2+\alpha}\right\} .
\]
The case of $\left|\left|G_{2}^{\left(n\right)}\left(t,\cdot\right)\right|\right|_{\dot{C}^{\alpha}\left(\mathbb{R}^{2}\right)}$
is completely analogous.
\end{proof}
Next in line would be to calculate some bounds of the Taylor expansion
of $\widetilde{U^{\left(n-1\right)}}^{n}\left(t,x\right)$ that appears
in equations \eqref{eq:Boussinesq vorticity Taylor} and \eqref{eq:Boussinesq density Taylor}.
Nonetheless, doing this fully rigorously can prove to be very cumbersome
without knowing more about the specific choice of our parameters.
Therefore, we will undertake a simplified calculation for the moment.
The expression for the remainder of the Multivariate Taylor Theorem
assures us that
\[
\begin{aligned} & \widetilde{U^{\left(n-1\right)}}^{n}\left(t,x\right)-\widetilde{U^{\left(n-1\right)}}^{n}\left(t,0\right)-\mathrm{J}\widetilde{U^{\left(n-1\right)}}^{n}\left(t,0\right)\cdot\left(\begin{matrix}x_{1}\\
x_{2}
\end{matrix}\right)=\\
= & \frac{1}{2}\left[\int_{0}^{1}\left(1-s\right)\frac{\partial^{2}\widetilde{U^{\left(n-1\right)}}^{n}}{\partial x_{1}^{2}}\left(sx\right)\mathrm{d}s\right]x_{1}^{2}+\\
 & +\left[\int_{0}^{1}\left(1-s\right)\frac{\partial^{2}\widetilde{U^{\left(n-1\right)}}^{n}}{\partial x_{1}\partial x_{2}}\left(sx\right)\mathrm{d}s\right]x_{1}x_{2}+\\
 & +\frac{1}{2}\left[\int_{0}^{1}\left(1-s\right)\frac{\partial^{2}\widetilde{U^{\left(n-1\right)}}^{n}}{\partial x_{2}^{2}}\left(sx\right)\mathrm{d}s\right]x_{2}^{2}.
\end{aligned}
\]
To simplify things, instead of considering the whole expression above,
we will work just with the term
\[
\widetilde{I^{\left(n\right)}}^{n}\left(t,x\right)\coloneqq\left[\int_{0}^{1}\left(1-s\right)\frac{\partial^{2}\widetilde{U_{1}^{\left(n-1\right)}}^{n}}{\partial x_{1}^{2}}\left(sx\right)\mathrm{d}s\right]x_{1}^{2}.
\]
Furthermore, suppose we will evaluate the expression above only in
the zone where the cutoff factors that appear in the expression given
in Proposition \ref{prop:computations vorticity} for the velocity
are identically 1. Then, to calculate $\frac{\partial^{2}\widetilde{U_{1}^{\left(n-1\right)}}^{n}}{\partial x_{1}^{2}}\left(sx\right)$,
we can depart from equation \eqref{eq:jacobian complete} in the proof
of Proposition \ref{prop:relation between anbn and jacobian}. Taking
into account that
\[
\frac{\partial\phi_{1}^{\left(n\right)}}{\partial x_{1}}\left(t,x\right)=\frac{1}{a_{n}\left(t\right)},\quad\frac{\partial\phi_{2}^{\left(n\right)}}{\partial x_{2}}=\frac{1}{b_{n}\left(t\right)},\quad\frac{\partial\phi_{1}^{\left(n\right)}}{\partial x_{2}}\left(t,x\right)=0=\frac{\partial\phi_{2}^{\left(n\right)}}{\partial x_{1}}\left(t,x\right)
\]
by Choice \ref{choice:phin}, we arrive to
\[
\begin{aligned}\frac{\partial^{2}\widetilde{U_{1}^{\left(n-1\right)}}^{n}}{\partial x_{1}^{2}}\left(t,x\right) & =-\sum_{m=1}^{n-1}\frac{B_{m}\left(t\right)b_{m}\left(t\right)a_{m}\left(t\right)^{2}}{a_{n}\left(t\right)^{2}}\sin\left(a_{m}\left(t\right)\left(\phi_{1}^{\left(n\right)}\left(t,x\right)-\phi_{1}^{\left(m\right)}\left(t,0\right)\right)\right)\cdot\\
 & \quad\cdot\cos\left(b_{m}\left(t\right)\left(\phi_{2}^{\left(n\right)}\left(t,x\right)-\phi_{2}^{\left(m\right)}\left(t,0\right)\right)\right).
\end{aligned}
\]
We wish to calculate $\left|\left|I^{\left(n\right)}\left(t,\cdot\right)\right|\right|_{L^{\infty}\left(D^{\left(n\right)}\left(t\right)\right)}$,
where
\[
D^{\left(n\right)}\left(t\right)\coloneqq\left(\phi^{\left(n\right)}\right)^{-1}\left(t,\left[-\frac{16\pi}{\lambda_{n}},\frac{16\pi}{\lambda_{n}}\right]^{2}\right)
\]
and $\left(\phi^{\left(n\right)}\right)^{-1}\left(t,\cdot\right)$
should be understood as the preimage at fixed $t$. As the $\left|\left|\cdot\right|\right|_{L^{\infty}}$-norm
is invariant under diffeomorphisms, we have
\[
\left|\left|I^{\left(n\right)}\left(t,\cdot\right)\right|\right|_{L^{\infty}\left(D^{\left(n\right)}\left(t\right)\right)}=\left|\left|\widetilde{I^{\left(n\right)}}^{n}\left(t,\cdot\right)\right|\right|_{L^{\infty}\left(\left[-\frac{16\pi}{\lambda_{n}},\frac{16\pi}{\lambda_{n}}\right]^{2}\right)}.
\]
Bounding $\left|\sin\left(\cdot\right)\right|\le1$, $\left|\cos\left(\cdot\right)\right|\le1$
and $s\le1$ and $1-s\le1$ in the integral, we conclude that
\begin{equation}
\left|\left|I^{\left(n\right)}\left(t,\cdot\right)\right|\right|_{L^{\infty}\left(D^{\left(n\right)}\left(t\right)\right)}\lesssim\frac{\sum_{m=1}^{n-1}B_{m}\left(t\right)b_{m}\left(t\right)a_{m}\left(t\right)^{2}}{\lambda_{n}^{2}a_{n}\left(t\right)^{2}}.\label{eq:bound I}
\end{equation}

Now that we have obtained this bound for $I^{\left(n\right)}\left(t,x\right)$,
let us study the transport term
\[
\widetilde{T_{\omega}^{\left(n\right)}}^{n}\coloneqq\left(\widetilde{U^{\left(n-1\right)}}^{n}\left(t,x\right)-\widetilde{U^{\left(n-1\right)}}^{n}\left(t,0\right)-\mathrm{J}\widetilde{U^{\left(n-1\right)}}^{n}\left(t,0\right)\cdot\left(\begin{matrix}x_{1}\\
x_{2}
\end{matrix}\right)\right)\cdot\widetilde{\nabla}^{n}\widetilde{\omega^{\left(n\right)}}^{n}\left(t,x\right).
\]
Thanks to equation \eqref{eq:bound I} and Proposition \ref{prop:form gradient omega},
we know that a bound of $\left|\left|\widetilde{T_{\omega}^{\left(n\right)}}^{n}\right|\right|_{\dot{C}^{\alpha}\left(D^{\left(n\right)}\left(t\right)\right)}$
will contain the summand
\[
S\coloneqq\frac{\sum_{m=1}^{n-1}B_{m}\left(t\right)b_{m}\left(t\right)a_{m}\left(t\right)^{2}}{\lambda_{n}^{2}a_{n}\left(t\right)^{2}}B_{n}\left(t\right)a_{n}\left(t\right)\max\left\{ a_{n}\left(t\right)^{2+\alpha},b_{n}\left(t\right)^{2+\alpha}\right\} .
\]
To simplify things, consider the expression above only at $t=1$.
To simplify further, suppose that $a_{m}\left(1\right)\sim b_{m}\left(1\right)$.
In that case, the expression above takes the form
\[
S\sim\frac{\sum_{m=1}^{n-1}B_{m}\left(1\right)b_{m}\left(1\right)^{3}}{\lambda_{n}^{2}b_{n}\left(1\right)^{2}}B_{n}\left(1\right)b_{n}\left(1\right)^{3+\alpha}\sim\left(\sum_{m=1}^{n-1}B_{m}\left(1\right)b_{m}\left(1\right)^{3}\right)\frac{B_{n}\left(1\right)b_{n}\left(1\right)^{1+\alpha}}{\lambda_{n}^{2}}.
\]
Recall that, by Choice \ref{choice:Bnanbn}, we have $B_{n}\left(1\right)a_{n}\left(1\right)b_{n}\left(1\right)=M_{n}$.
Hence, we expect $B_{m}\left(1\right)b_{m}\left(1\right)^{3}$ to
grow even faster in $m$ than $M_{m}$. This means that $\sum_{m=1}^{n-1}B_{m}\left(1\right)b_{m}\left(1\right)^{3}$
should grow at least exponentially and, consequently, it should have
the same order of its last term. This leads to
\begin{equation}
S\sim B_{n-1}\left(1\right)b_{n-1}\left(1\right)^{3}\frac{B_{n}\left(1\right)b_{n}\left(1\right)^{1+\alpha}}{\lambda_{n}^{2}}.\label{eq:first bound transport}
\end{equation}
Let us ignore the $\lambda_{n}$ for now. If we wish for the expression
above to be bounded in $n$, we need $b_{n}\left(1\right)$ to grow
at least superexponentially in $n$, because an exponential growth
is not enough. Indeed, suppose $b_{n}\left(1\right)\sim L^{n},$ $B_{n}\left(1\right)\sim\frac{1}{L^{n}}$
with $L>1$. Then, we would have $B_{n}\left(1\right)b_{n}\left(1\right)^{2}\sim L^{n}$
(which grows exponentially, satisfying Choice \ref{choice:Bnanbn}),
but
\[
S\sim B_{n-1}\left(1\right)b_{n-1}\left(1\right)^{3}B_{n}\left(1\right)b_{n}\left(1\right)^{1+\alpha}\sim\frac{1}{L^{n-1}}L^{3\left(n-1\right)}\frac{1}{L^{n}}L^{n\left(1+\alpha\right)}=L^{\left(2+\alpha\right)n-2}
\]
and this expression is never bounded in $n\in\mathbb{N}$.

On the other hand, if $b_{n}\left(1\right)$ grows superexponentially
in $n$, then $\frac{b_{n-1}\left(1\right)}{b_{n}\left(1\right)}$
can be made as small as we want. To fix ideas, assume that $\frac{b_{n-1}\left(1\right)}{b_{n}\left(1\right)}\sim\frac{1}{b_{n}\left(1\right)^{\gamma}}$,
where $\gamma\in\left(0,1\right)$ is very close to $1$. Then, we
can write
\[
\begin{aligned}S & \sim B_{n-1}\left(1\right)b_{n-1}\left(1\right)^{3}B_{n}\left(1\right)b_{n}\left(1\right)^{1+\alpha}\sim B_{n-1}\left(1\right)b_{n-1}\left(1\right)^{3}\frac{B_{n}\left(1\right)b_{n}\left(1\right)^{2}}{b_{n}\left(1\right)^{1-\alpha}}\sim\\
 & \sim B_{n-1}\left(1\right)b_{n-1}\left(1\right)^{2}B_{n}\left(1\right)b_{n}\left(1\right)^{2}b_{n}\left(1\right)^{\alpha}\frac{b_{n-1}\left(1\right)}{b_{n}\left(1\right)}\sim\\
 & \sim\underbrace{B_{n-1}\left(1\right)b_{n-1}\left(1\right)^{2}}_{\sim M_{n-1}}\underbrace{B_{n}\left(1\right)b_{n}\left(1\right)^{2}}_{\sim M_{n}}b_{n}\left(1\right)^{\alpha}\frac{1}{b_{n}\left(1\right)^{\gamma}}\sim M_{n}M_{n-1}\frac{1}{b_{n}\left(1\right)^{\gamma-\alpha}}.
\end{aligned}
\]
If $M_{n}$ grows exponentially in $n\in\mathbb{N}$, as $b_{n}\left(1\right)$
grows superexponentially, as long as $\gamma>\alpha$, the factor
$M_{n}M_{n-1}$ will be negligible and the expression above will be
decreasing in $n$. If $M_{n}$ grows superexponentially in $n\in\mathbb{N}$,
but $M_{n}M_{n-1}$ grows slower than $b_{n}\left(1\right)^{\gamma-\alpha}$,
the expression above will still be decreasing in $n\in\mathbb{N}$.
Thereby, in any case, it would be adviceful to choose $b_{n}\left(1\right)$
such that
\[
\frac{b_{n-1}\left(1\right)}{b_{n}\left(1\right)}\sim\frac{1}{b_{n}\left(1\right)^{\gamma}}.
\]
Assuming the equality, one obtains
\[
b_{n-1}\left(1\right)=b_{n}\left(1\right)^{1-\gamma}\iff b_{n}\left(1\right)=b_{n-1}\left(1\right)^{\frac{1}{1-\gamma}}.
\]
Clearly, this recursion admits the solution
\[
b_{n}\left(1\right)=C^{\left(\frac{1}{1-\gamma}\right)^{n}},
\]
where $C>0$ is some constant.
\begin{choice}
\label{choice:anbn}With all the information presented above, it make
sense to take
\[
b_{n}\left(t\right)=C^{\left(1+k_{n}\left(t\right)\right)\left(\frac{1}{1-\gamma}\right)^{n}},\quad a_{n}\left(t\right)=C^{\left(1-k_{n}\left(t\right)\right)\left(\frac{1}{1-\gamma}\right)^{n}},
\]
where $C>1$, $\gamma\in\left(0,1\right)$ and $k_{n}\left(t\right)$
is a function that captures the time evolution of $b_{n}\left(t\right)$
and $a_{n}\left(t\right)$. Note that $\frac{\mathrm{d}}{\mathrm{d}t}\left(a_{n}\left(t\right)b_{n}\left(t\right)\right)=0$,
as required by Choice \ref{choice:anbncn}. Indeed, $\frac{\mathrm{d}}{\mathrm{d}t}\left(a_{n}\left(t\right)b_{n}\left(t\right)\right)=0$
is the reason why we only need one function to account for the time
evolution of both parameters $a_{n}\left(t\right)$ and $b_{n}\left(t\right)$.
Furthermore, notice that $k_{n}\left(t\right)$ is not signed, i.e.,
could be positive or negative.

Nevertheless, by Choice \ref{choice:phin}, it is clear that layer
$n$ has size of order $\frac{1}{a_{n}\left(t\right)}$ in the $x_{1}$
direction and $\frac{1}{b_{n}\left(t\right)}$ in the $x_{2}$ direction.
Consequently, if we want the support of our solution to be uniformly
bounded in time, we need $a_{n}\left(t\right)$ and $b_{n}\left(t\right)$
to remain greater than unity at all times. For that purpose, we need
to take $\sup_{n\in\mathbb{N}}\max_{t\in\left[t_{n},1\right]}\left|k_{n}\left(t\right)\right|<1$.
\end{choice}

\subsection{\label{subsec:cons pure quadratic term}Pure quadratic term}

To obtain more information about how we should choose $\lambda_{n}$,
it will be useful to study the pure quadratic term of the vorticity
equation.
\begin{prop}
\label{prop:bounds pure quadratic term}Consider the pure quadratic
term
\[
\widetilde{Q_{\omega}^{\left(n\right)}}^{n}\left(t,x\right)\coloneqq\widetilde{u^{\left(n\right)}}^{n}\left(t,x\right)\cdot\widetilde{\nabla}^{n}\widetilde{\omega^{\left(n\right)}}^{n}\left(t,x\right).
\]
Then,
\[
\begin{aligned}\left|\left|Q_{\omega}^{\left(n\right)}\left(t,\cdot\right)\right|\right|_{L^{\infty}\left(\mathbb{R}^{2}\right)} & \lesssim_{\varphi}\lambda_{n}B_{n}\left(t\right)^{2}a_{n}\left(t\right)b_{n}\left(t\right)\max\left\{ a_{n}\left(t\right)^{2},b_{n}\left(t\right)^{2}\right\} ,\\
\left|\left|Q_{\omega}^{\left(n\right)}\left(t,\cdot\right)\right|\right|_{\dot{C}^{\alpha}\left(\mathbb{R}^{2}\right)} & \lesssim_{\varphi}\lambda_{n}B_{n}\left(t\right)^{2}a_{n}\left(t\right)b_{n}\left(t\right)\max\left\{ a_{n}\left(t\right)^{2+\alpha},b_{n}\left(t\right)^{2+\alpha}\right\} .
\end{aligned}
\]
\end{prop}
\begin{proof}
Taking a look at Propositions \ref{prop:form of the velocity} and
\ref{prop:form gradient omega}, notice that the only zeroth order
term in $\lambda_{n}$ that we could obtain form the dot product $\widetilde{u^{\left(n\right)}}^{n}\left(t,x\right)\cdot\widetilde{\nabla}^{n}\widetilde{\omega^{\left(n\right)}}^{n}\left(t,x\right)$
vanishes because of orthogonality, i.e., $\widetilde{V^{\left(n\right)}}^{n}\left(t,x\right)\cdot\widetilde{\Gamma^{\left(n\right)}}^{n}\left(t,x\right)=0$.
Using the aforementioned Propositions, we can write:
\[
\begin{aligned}Q_{\omega}^{\left(n\right)}\left(t,x\right) & =\left(V^{\left(n\right)}\left(t,x\right)+\lambda_{n}W^{\left(n\right)}\left(t,x\right)\right)\cdot\left(\Gamma^{\left(n\right)}\left(t,x\right)+\lambda_{n}G^{\left(n\right)}\left(t,x\right)\right)=\\
 & =\underbrace{V^{\left(n\right)}\left(t,x\right)\cdot\Gamma^{\left(n\right)}\left(t,x\right)}_{=0}+\lambda_{n}\left[\sum_{i=1}^{2}W_{i}^{\left(n\right)}\left(t,x\right)\cdot\Gamma_{i}^{\left(n\right)}\left(t,x\right)+\sum_{i=1}^{2}V_{i}^{\left(n\right)}\left(t,x\right)\cdot G_{i}^{\left(n\right)}\left(t,x\right)\right]+\\
 & \quad+\lambda_{n}^{2}\sum_{i=1}^{2}W_{i}^{\left(n\right)}\left(t,x\right)\cdot G_{i}^{\left(n\right)}\left(t,x\right).
\end{aligned}
\]
Next, using the bounds given in Propositions \ref{prop:form of the velocity}
and \ref{prop:form gradient omega}, we deduce that
\[
\begin{aligned}\left|\left|Q_{\omega}^{\left(n\right)}\left(t,\cdot\right)\right|\right|_{L^{\infty}\left(\mathbb{R}^{2}\right)} & \lesssim_{\varphi}\lambda_{n}\left[B_{n}\left(t\right)b_{n}\left(t\right)B_{n}\left(t\right)a_{n}\left(t\right)\left(a_{n}\left(t\right)^{2}+b_{n}\left(t\right)^{2}\right)+\right.\\
 & \qquad+B_{n}\left(t\right)a_{n}\left(t\right)B_{n}\left(t\right)b_{n}\left(t\right)\left(a_{n}\left(t\right)^{2}+b_{n}\left(t\right)^{2}\right)+\\
 & \qquad+B_{n}\left(t\right)b_{n}\left(t\right)B_{n}\left(t\right)a_{n}\left(t\right)\max\left\{ a_{n}\left(t\right)^{2},b_{n}\left(t\right)^{2}\right\} +\\
 & \qquad\left.+B_{n}\left(t\right)a_{n}\left(t\right)B_{n}\left(t\right)b_{n}\left(t\right)\max\left\{ a_{n}\left(t\right)^{2},b_{n}\left(t\right)^{2}\right\} \right]+\\
 & \quad+\lambda_{n}^{2}\left[B_{n}\left(t\right)b_{n}\left(t\right)B_{n}\left(t\right)a_{n}\left(t\right)\max\left\{ a_{n}\left(t\right)^{2},b_{n}\left(t\right)^{2}\right\} +\right.\\
 & \qquad\left.+B_{n}\left(t\right)a_{n}\left(t\right)B_{n}\left(t\right)b_{n}\left(t\right)\max\left\{ a_{n}\left(t\right)^{2},b_{n}\left(t\right)^{2}\right\} \right].
\end{aligned}
\]
Employing the fact that $\lambda_{n}^{2}\le\lambda_{n}$ as $\lambda_{n}\le1$
by Choice \ref{choice:psin}, we arrive to
\[
\left|\left|Q_{\omega}^{\left(n\right)}\left(t,\cdot\right)\right|\right|_{L^{\infty}\left(\mathbb{R}^{2}\right)}\lesssim_{\varphi}\lambda_{n}B_{n}\left(t\right)^{2}a_{n}\left(t\right)b_{n}\left(t\right)\max\left\{ a_{n}\left(t\right)^{2},b_{n}\left(t\right)^{2}\right\} .
\]

Similarly, bearing in mind how the $\left|\left|\cdot\right|\right|_{\dot{C}^{\alpha}\left(\mathbb{R}^{2}\right)}$
seminorm of a product is computed (see equation \eqref{eq:property Calpha multiplication}),
we obtain
\[
\begin{aligned}\left|\left|Q_{\omega}^{\left(n\right)}\left(t,\cdot\right)\right|\right|_{\dot{C}^{\alpha}\left(\mathbb{R}^{2}\right)} & \lesssim_{\varphi}\lambda_{n}\left[B_{n}\left(t\right)b_{n}\left(t\right)\max\left\{ a_{n}\left(t\right)^{\alpha},b_{n}\left(t\right)^{\alpha}\right\} B_{n}\left(t\right)a_{n}\left(t\right)\left(a_{n}\left(t\right)^{2}+b_{n}\left(t\right)^{2}\right)+\right.\\
 & \qquad+B_{n}\left(t\right)b_{n}\left(t\right)B_{n}\left(t\right)a_{n}\left(t\right)\left(a_{n}\left(t\right)^{2}+b_{n}\left(t\right)^{2}\right)\max\left\{ a_{n}\left(t\right)^{\alpha},b_{n}\left(t\right)^{\alpha}\right\} +\\
 & \qquad+B_{n}\left(t\right)a_{n}\left(t\right)\max\left\{ a_{n}\left(t\right)^{\alpha},b_{n}\left(t\right)^{\alpha}\right\} B_{n}\left(t\right)b_{n}\left(t\right)\left(a_{n}\left(t\right)^{2}+b_{n}\left(t\right)^{2}\right)+\\
 & \qquad+B_{n}\left(t\right)a_{n}\left(t\right)B_{n}\left(t\right)b_{n}\left(t\right)\left(a_{n}\left(t\right)^{2}+b_{n}\left(t\right)^{2}\right)\max\left\{ a_{n}\left(t\right)^{\alpha},b_{n}\left(t\right)^{\alpha}\right\} +\\
 & \qquad+B_{n}\left(t\right)b_{n}\left(t\right)\max\left\{ a_{n}\left(t\right)^{\alpha},b_{n}\left(t\right)^{\alpha}\right\} B_{n}\left(t\right)a_{n}\left(t\right)\max\left\{ a_{n}\left(t\right)^{2},b_{n}\left(t\right)^{2}\right\} +\\
 & \qquad+B_{n}\left(t\right)b_{n}\left(t\right)B_{n}\left(t\right)a_{n}\left(t\right)\max\left\{ a_{n}\left(t\right)^{2+\alpha},b_{n}\left(t\right)^{2+\alpha}\right\} +\\
 & \qquad+B_{n}\left(t\right)a_{n}\left(t\right)\max\left\{ a_{n}\left(t\right)^{\alpha},b_{n}\left(t\right)^{\alpha}\right\} B_{n}\left(t\right)b_{n}\left(t\right)\max\left\{ a_{n}\left(t\right)^{2},b_{n}\left(t\right)^{2}\right\} +\\
 & \qquad\left.+B_{n}\left(t\right)a_{n}\left(t\right)B_{n}\left(t\right)b_{n}\left(t\right)\max\left\{ a_{n}\left(t\right)^{2+\alpha},b_{n}\left(t\right)^{2+\alpha}\right\} \right]+\\
 & \quad+\lambda_{n}^{2}\left[B_{n}\left(t\right)b_{n}\left(t\right)\max\left\{ a_{n}\left(t\right)^{\alpha},b_{n}\left(t\right)^{\alpha}\right\} B_{n}\left(t\right)a_{n}\left(t\right)\max\left\{ a_{n}\left(t\right)^{2},b_{n}\left(t\right)^{2}\right\} +\right.\\
 & \qquad+B_{n}\left(t\right)b_{n}\left(t\right)B_{n}\left(t\right)a_{n}\left(t\right)\max\left\{ a_{n}\left(t\right)^{2+\alpha},b_{n}\left(t\right)^{2+\alpha}\right\} +\\
 & \qquad+B_{n}\left(t\right)a_{n}\left(t\right)\max\left\{ a_{n}\left(t\right)^{\alpha},b_{n}\left(t\right)^{\alpha}\right\} B_{n}\left(t\right)b_{n}\left(t\right)\max\left\{ a_{n}\left(t\right)^{2},b_{n}\left(t\right)^{2}\right\} +\\
 & \qquad\left.+B_{n}\left(t\right)a_{n}\left(t\right)B_{n}\left(t\right)b_{n}\left(t\right)\max\left\{ a_{n}\left(t\right)^{2+\alpha},b_{n}\left(t\right)^{2+\alpha}\right\} \right].
\end{aligned}
\]
As $\lambda_{n}^{2}\le\lambda_{n}$ by Choice \ref{choice:psin},
we can write
\[
\left|\left|Q_{\omega}^{\left(n\right)}\left(t,\cdot\right)\right|\right|_{\dot{C}^{\alpha}\left(\mathbb{R}^{2}\right)}\lesssim_{\varphi}\lambda_{n}B_{n}\left(t\right)^{2}a_{n}\left(t\right)b_{n}\left(t\right)\max\left\{ a_{n}\left(t\right)^{2+\alpha},b_{n}\left(t\right)^{2+\alpha}\right\} .
\]
\end{proof}
\begin{rem}
\label{rem:why we need lambdan}Assuming for simplicity that $a_{n}\left(t\right)\sim b_{n}\left(t\right)$,
thanks to Proposition \ref{prop:bounds pure quadratic term}, we obtain
\[
\begin{aligned}\left|\left|Q_{\omega}^{\left(n\right)}\left(t,\cdot\right)\right|\right|_{\dot{C}^{\alpha}\left(\mathbb{R}^{2}\right)} & \lesssim_{\varphi}\lambda_{n}\underbrace{B_{n}\left(t\right)a_{n}\left(t\right)b_{n}\left(t\right)}_{\sim M_{n}}\underbrace{B_{n}\left(t\right)b_{n}\left(t\right)^{2}}_{\sim M_{n}}b_{n}\left(t\right)^{\alpha}\sim\lambda_{n}M_{n}^{2}b_{n}\left(t\right)^{\alpha}.\end{aligned}
\]
As $b_{n}\left(t\right)$ grows superexponentially in $n$ according
to Choice \ref{choice:anbn}, the only way to prevent $\left|\left|Q_{\omega}^{\left(n\right)}\left(t,\cdot\right)\right|\right|_{\dot{C}^{\alpha}\left(\mathbb{R}^{2}\right)}$
from blowing up is to compensate with the factor $\lambda_{n}$, i.e.,
we have to make $\lambda_{n}$ decrease superexponentially in $n$.
\end{rem}
\begin{choice}
\label{choice:lambdan}In the light of the above, we take
\[
\lambda_{n}=C^{-\Lambda\left(\frac{1}{1-\gamma}\right)^{n}},
\]
where $\Lambda>0$ is to be determined.
\end{choice}
\begin{rem}
Choice \ref{choice:lambdan}, along with Remark \ref{rem:why we need lambdan}
and the explanation given before Choice \ref{choice:anbn}, reveals
that a compromise has to be made when choosing $\lambda_{n}$. Indeed,
on the one hand, in view of Remark \ref{rem:why we need lambdan},
we would want $\lambda_{n}$ to be as small as possible. However,
on the other hand, taking a look at the bound we found for one summand
of the transport term in equation \eqref{eq:first bound transport},
since $\lambda_{n}$ is in the denominator, it is easy to see that
we want $\lambda_{n}$ to be as close to $1$ as possible, since that
will make the bounds smaller.
\end{rem}

\subsection{Functional form of the density}

The next thing we are going to do is to choose which error terms of
equation \eqref{eq:Boussinesq vorticity Taylor} we are going to cancel
out with $\frac{\partial\rho}{\partial x_{2}}$. To do this, we will
turn to the introduction (see equation \eqref{eq:intro form rho}),
where we used a density of the form $\sin\left(x_{1}\right)\cos\left(x_{2}\right)$.
Thereby, we should search for a term in the equation of the vorticity
that has order zero in $\lambda_{n}$ and looks like $\sin\left(x_{1}\right)\sin\left(x_{2}\right)$
(note that the derivative with respect to $x_{2}$ of $\sin\left(x_{1}\right)\cos\left(x_{2}\right)$
is precisely $-\sin\left(x_{1}\right)\sin\left(x_{2}\right)$). We
have already studied the transport term in subsection \ref{subsec:cons transport term revisited}
and the pure quadratic term in subsection \ref{subsec:cons pure quadratic term}
and we have seen nothing similar to what we are looking for. This
is because all those terms come from a scalar product of two quantities
and, consequently, in each summand we will always have four trigonometric
factors. By the same reason, we should not expect to find the wanted
term in $\widetilde{u^{\left(n\right)}}^{n}\left(t,x\right)\cdot\widetilde{\nabla}^{n}\widetilde{\Omega^{\left(n-1\right)}}^{n}\left(t,x\right)$.
Then, the only term that remains is the temporal derivative $\frac{\partial\widetilde{\omega^{\left(n\right)}}^{n}}{\partial t}$
and, according to Proposition \ref{prop:computations vorticity},
it will only contain one zeroth order term in $\lambda_{n}$, which
is
\[
\frac{\mathrm{d}}{\mathrm{d}t}\left[B_{n}\left(t\right)\left(a_{n}\left(t\right)^{2}+b_{n}\left(t\right)^{2}\right)\right]\varphi\left(\lambda_{n}x_{1}\right)\varphi\left(\lambda_{n}x_{2}\right)\sin\left(x_{1}\right)\sin\left(x_{2}\right).
\]
Notice that, in view of equation \eqref{eq:intro derivative rho in transport},
leaving the cutoff factors aside, this term has the form we are looking
for. Nevertheless, there is a slight inconvenience. If we were to
set
\begin{equation}
\left(0,1\right)\cdot\widetilde{\nabla}^{n}\widetilde{\rho^{\left(n\right)}}^{n}\left(t,x\right)=\frac{\mathrm{d}}{\mathrm{d}t}\left[B_{n}\left(t\right)\left(a_{n}\left(t\right)^{2}+b_{n}\left(t\right)^{2}\right)\right]\varphi\left(\lambda_{n}x_{1}\right)\varphi\left(\lambda_{n}x_{2}\right)\sin\left(x_{1}\right)\sin\left(x_{2}\right),\label{eq:first attempt density}
\end{equation}
then, integrating with respect to $x_{2}$, we would get
\begin{equation}
\widetilde{\rho^{\left(n\right)}}^{n}\left(t,x\right)=\frac{1}{b_{n}\left(t\right)}\frac{\mathrm{d}}{\mathrm{d}t}\left[B_{n}\left(t\right)\left(a_{n}\left(t\right)^{2}+b_{n}\left(t\right)^{2}\right)\right]\varphi\left(\lambda_{n}x_{1}\right)\sin\left(x_{1}\right)\int_{0}^{x_{2}}\varphi\left(\lambda_{n}\xi\right)\sin\left(\xi\right)\mathrm{d}\xi,\label{eq:first try of rho}
\end{equation}
which would make our density non-explicit. As we will see shortly,
we can get an explicit $\widetilde{\rho^{\left(n\right)}}^{n}\left(t,x\right)$
by adding a term of order $1$ in $\lambda_{n}$. If we integrate
by parts in \eqref{eq:first try of rho}, we obtain
\[
\begin{aligned}\int_{0}^{x_{2}}\varphi\left(\lambda_{n}\xi\right)\sin\left(\xi\right)\mathrm{d}\xi & =-\varphi\left(\lambda_{n}x_{2}\right)\cos\left(x_{2}\right)+\lambda_{n}\int_{0}^{x_{2}}\varphi'\left(\lambda_{n}\xi\right)\cos\left(\xi\right)\mathrm{d}\xi.\end{aligned}
\]
Consequently,
\[
\begin{aligned}\widetilde{\rho^{\left(n\right)}}^{n}\left(t,x\right) & =\overbrace{-\frac{1}{b_{n}\left(t\right)}\frac{\mathrm{d}}{\mathrm{d}t}\left[B_{n}\left(t\right)\left(a_{n}\left(t\right)^{2}+b_{n}\left(t\right)^{2}\right)\right]\varphi\left(\lambda_{n}x_{1}\right)\varphi\left(\lambda_{n}x_{2}\right)\sin\left(x_{1}\right)\cos\left(x_{2}\right)}^{\text{this is the density we would like to have}}+\\
 & \quad+\underbrace{\lambda_{n}\frac{1}{b_{n}\left(t\right)}\frac{\mathrm{d}}{\mathrm{d}t}\left[B_{n}\left(t\right)\left(a_{n}\left(t\right)^{2}+b_{n}\left(t\right)^{2}\right)\right]\varphi\left(\lambda_{n}x_{1}\right)\sin\left(x_{1}\right)\int_{0}^{x_{2}}\varphi'\left(\lambda_{n}\xi\right)\cos\left(\xi\right)\mathrm{d}\xi}_{\eqqcolon\widetilde{E^{\left(n\right)}}^{n}\left(t,x\right)}.
\end{aligned}
\]
If we compute $\left(0,1\right)\cdot\widetilde{\nabla}^{n}\widetilde{E^{\left(n\right)}}^{n}\left(t,x\right)$,
we will get
\[
\left(0,1\right)\cdot\widetilde{\nabla}^{n}\widetilde{E^{\left(n\right)}}^{n}\left(t,x\right)=\lambda_{n}\frac{\mathrm{d}}{\mathrm{d}t}\left[B_{n}\left(t\right)\left(a_{n}\left(t\right)^{2}+b_{n}\left(t\right)^{2}\right)\right]\varphi\left(\lambda_{n}x_{1}\right)\varphi'\left(\lambda_{n}x_{2}\right)\sin\left(x_{1}\right)\cos\left(x_{2}\right).
\]
Hence, if instead of setting \eqref{eq:first attempt density}, we
impose
\begin{equation}
\begin{aligned}\left(0,1\right)\cdot\widetilde{\nabla}^{n}\widetilde{\rho^{\left(n\right)}}^{n}\left(t,x\right) & =\frac{\mathrm{d}}{\mathrm{d}t}\left[B_{n}\left(t\right)\left(a_{n}\left(t\right)^{2}+b_{n}\left(t\right)^{2}\right)\right]\varphi\left(\lambda_{n}x_{1}\right)\varphi\left(\lambda_{n}x_{2}\right)\sin\left(x_{1}\right)\sin\left(x_{2}\right)+\\
 & \quad-\lambda_{n}\frac{\mathrm{d}}{\mathrm{d}t}\left[B_{n}\left(t\right)\left(a_{n}\left(t\right)^{2}+b_{n}\left(t\right)^{2}\right)\right]\varphi\left(\lambda_{n}x_{1}\right)\varphi'\left(\lambda_{n}x_{2}\right)\sin\left(x_{1}\right)\cos\left(x_{2}\right),
\end{aligned}
\label{eq:second attempt density}
\end{equation}
we will obtain
\[
\widetilde{\rho^{\left(n\right)}}^{n}\left(t,x\right)=-\frac{1}{b_{n}\left(t\right)}\frac{\mathrm{d}}{\mathrm{d}t}\left[B_{n}\left(t\right)\left(a_{n}\left(t\right)^{2}+b_{n}\left(t\right)^{2}\right)\right]\varphi\left(\lambda_{n}x_{1}\right)\varphi\left(\lambda_{n}x_{2}\right)\sin\left(x_{1}\right)\cos\left(x_{2}\right),
\]
as we wanted. Thus, we will enforce \eqref{eq:second attempt density}
instead of \eqref{eq:first attempt density}, as we clarify in the
following Choice.
\begin{choice}
\label{choice:density}As a result of the above, we choose
\[
\begin{aligned}\left(0,1\right)\cdot\widetilde{\nabla}^{n}\widetilde{\rho^{\left(n\right)}}^{n}\left(t,x\right) & =\frac{\mathrm{d}}{\mathrm{d}t}\left[B_{n}\left(t\right)\left(a_{n}\left(t\right)^{2}+b_{n}\left(t\right)^{2}\right)\right]\varphi\left(\lambda_{n}x_{1}\right)\varphi\left(\lambda_{n}x_{2}\right)\sin\left(x_{1}\right)\sin\left(x_{2}\right)+\\
 & \quad-\lambda_{n}\frac{\mathrm{d}}{\mathrm{d}t}\left[B_{n}\left(t\right)\left(a_{n}\left(t\right)^{2}+b_{n}\left(t\right)^{2}\right)\right]\varphi\left(\lambda_{n}x_{1}\right)\varphi'\left(\lambda_{n}x_{2}\right)\sin\left(x_{1}\right)\cos\left(x_{2}\right),
\end{aligned}
\]
which leads to
\[
\widetilde{\rho^{\left(n\right)}}^{n}\left(t,x\right)=-\frac{1}{b_{n}\left(t\right)}\frac{\mathrm{d}}{\mathrm{d}t}\left[B_{n}\left(t\right)\left(a_{n}\left(t\right)^{2}+b_{n}\left(t\right)^{2}\right)\right]\varphi\left(\lambda_{n}x_{1}\right)\varphi\left(\lambda_{n}x_{2}\right)\sin\left(x_{1}\right)\cos\left(x_{2}\right).
\]
\end{choice}

\subsection{Choosing of $B_{n}\left(t\right)$}

If we take a look at the expression given for the density in Choice
\ref{choice:density}, we can see that the amplitude has two factors:
$\frac{1}{b_{n}\left(t\right)}$, which decreases superexponentially
with $n\in\mathbb{N}$, and $\frac{\mathrm{d}}{\mathrm{d}t}\left[B_{n}\left(t\right)\left(a_{n}\left(t\right)^{2}+b_{n}\left(t\right)^{2}\right)\right]$,
which we do not know too much about. By Choice \ref{choice:Bnanbn},
we know that $B_{n}\left(1\right)a_{n}\left(1\right)b_{n}\left(1\right)$
grows, at least, exponentially in $n\in\mathbb{N}$. If $a_{n}\left(t\right)\sim b_{n}\left(t\right)\sim b_{n}\left(1\right)$,
it makes sense for $B_{n}\left(t\right)\left(a_{n}\left(t\right)^{2}+b_{n}\left(t\right)^{2}\right)$
to increase in $n\in\mathbb{N}$ in the same manner, but we cannot
say much about $\frac{\mathrm{d}}{\mathrm{d}t}\left[B_{n}\left(t\right)\left(a_{n}\left(t\right)^{2}+b_{n}\left(t\right)^{2}\right)\right]$,
because we do not know anything about our time scale yet. Furthermore,
recalling the ideas we presented in the introduction, we wanted to
have a way to turn the density on and off. As $a_{n}\left(t\right)$
and $b_{n}\left(t\right)$ are already chosen, the only way to do
this is through $B_{n}\left(t\right)$. Moreover, notice that the
amplitude of the density would be much easier to study if we could
somehow ``factor out'' the dependence on $n$ with respect to the
dependence on $t$. With all these ideas in mind, it makes sense to
try
\[
\frac{1}{b_{n}\left(t\right)}\frac{\mathrm{d}}{\mathrm{d}t}\left[B_{n}\left(t\right)\left(a_{n}\left(t\right)^{2}+b_{n}\left(t\right)^{2}\right)\right]=z_{n}H^{\left(n\right)}\left(t\right),
\]
where $\left(z_{n}\right)_{n\in\mathbb{N}}$ is a certain sequence
we know nothing about (for the moment) and $H^{\left(n\right)}\left(t\right)$
is a bounded function such that $H^{\left(n\right)}\left(t\right)=0$
$\forall t\le t_{n}$. If this were the case, as we will have $B_{n}\left(0\right)=0$
$\forall n\in\mathbb{N}$, we would obtain
\[
B_{n}\left(t\right)\left(a_{n}\left(t\right)^{2}+b_{n}\left(t\right)^{2}\right)=z_{n}\int_{0}^{t}H^{\left(n\right)}\left(s\right)b_{n}\left(s\right)\mathrm{d}s.
\]
Bear in mind that, by Proposition \ref{prop:computations vorticity},
this term above is the amplitude of the leading order term of the
vorticity. Changing the sequence $\left(z_{n}\right)_{n\in\mathbb{N}}$
if necessary, we may assume, for the sake of simplicity, that $\int_{0}^{1}H^{\left(n\right)}\left(s\right)b_{n}\left(s\right)\mathrm{d}s=1$.
One way to do this is to take
\[
H^{\left(n\right)}\left(t\right)=\frac{h^{\left(n\right)}\left(t\right)}{\int_{0}^{1}h^{\left(n\right)}\left(s\right)b_{n}\left(s\right)\mathrm{d}s},
\]
where $h^{\left(n\right)}\left(t\right)$ is a function of $O\left(1\right)$
such that $\left.h^{\left(n\right)}\left(t\right)\right|_{\left[0,t_{n}\right]}\equiv0$.
\begin{choice}
\label{choice:amplitude density}All the aforementioned arguments
justify that we choose the amplitude of the density as
\[
\frac{1}{b_{n}\left(t\right)}\frac{\mathrm{d}}{\mathrm{d}t}\left[B_{n}\left(t\right)\left(a_{n}\left(t\right)^{2}+b_{n}\left(t\right)^{2}\right)\right]=z_{n}\frac{h^{\left(n\right)}\left(t\right)}{\int_{0}^{1}h^{\left(n\right)}\left(s\right)b_{n}\left(s\right)\mathrm{d}s}=z_{n}\frac{h^{\left(n\right)}\left(t\right)}{\int_{t_{n}}^{1}h^{\left(n\right)}\left(s\right)b_{n}\left(s\right)\mathrm{d}s},
\]
where $\left(z_{n}\right)_{n\in\mathbb{N}}\subseteq\mathbb{R}^{+}$
is a certain sequence yet to be determined and $\left(h^{\left(n\right)}\right)_{n\in\mathbb{N}}$
is a sequence of functions $h^{\left(n\right)}:\left[0,1\right]\to\left[0,1\right]$
such that $\left.h^{\left(n\right)}\right|_{\left[0,t_{n}\right]}=0$
$\forall n\in\mathbb{N}$.
\end{choice}
As a consequence of Choice \ref{choice:amplitude density}, we arrive
to
\[
\begin{aligned}B_{n}\left(t\right)\left(a_{n}\left(t\right)^{2}+b_{n}\left(t\right)^{2}\right) & =z_{n}\frac{\int_{t_{n}}^{t}h^{\left(n\right)}\left(s\right)b_{n}\left(s\right)\mathrm{d}s}{\int_{t_{n}}^{1}h^{\left(n\right)}\left(s\right)b_{n}\left(s\right)\mathrm{d}s},\\
B_{n}\left(t\right) & =\frac{z_{n}}{a_{n}\left(t\right)^{2}+b_{n}\left(t\right)^{2}}\frac{\int_{t_{n}}^{t}h^{\left(n\right)}\left(s\right)b_{n}\left(s\right)\mathrm{d}s}{\int_{t_{n}}^{1}h^{\left(n\right)}\left(s\right)b_{n}\left(s\right)\mathrm{d}s}.
\end{aligned}
\]
Furthermore, the sequences $\left(z_{n}\right)_{n\in\mathbb{N}}$
of Choice \ref{choice:amplitude density} and $\left(M_{n}\right)_{n\in\mathbb{N}}$
of Choice \ref{choice:Bnanbn} are related. Indeed, since
\[
\frac{\int_{t_{n}}^{t}h^{\left(n\right)}\left(s\right)b_{n}\left(s\right)\mathrm{d}s}{\int_{t_{n}}^{1}h^{\left(n\right)}\left(s\right)b_{n}\left(s\right)\mathrm{d}s}=1\quad\text{ at }t=1,
\]
we deduce that
\[
M_{n}=B_{n}\left(1\right)a_{n}\left(1\right)b_{n}\left(1\right)=\frac{z_{n}}{a_{n}\left(1\right)^{2}+b_{n}\left(1\right)^{2}}a_{n}\left(1\right)b_{n}\left(1\right).
\]
Using Choice \ref{choice:anbn}, we arrive to
\begin{equation}
\begin{aligned}M_{n} & =\frac{z_{n}}{C^{2\left(1-k_{n}\left(1\right)\right)\left(\frac{1}{1-\gamma}\right)^{n}}+C^{2\left(1+k_{n}\left(1\right)\right)\left(\frac{1}{1-\gamma}\right)^{n}}}C^{2\left(\frac{1}{1-\gamma}\right)^{n}}=\frac{z_{n}}{C^{-2k_{n}\left(1\right)\left(\frac{1}{1-\gamma}\right)^{n}}+C^{2k_{n}\left(1\right)\left(\frac{1}{1-\gamma}\right)^{n}}}=\\
 & =\frac{z_{n}}{2\cosh\left(2k_{n}\left(1\right)\left(\frac{1}{1-\gamma}\right)^{n}\right)}.
\end{aligned}
\label{eq:relation Mn and zn}
\end{equation}

\subsection{\label{subsec:bounding the time derivative of the density}Bounding
the time derivative of the density}

In this subsection, we will obtain the first pieces of information
concerning how we should choose our time scale and the profile of
$h^{\left(n\right)}\left(t\right)$.

Consider the time derivative of the density given in Choice \ref{choice:density}:
\[
\frac{\partial\widetilde{\rho^{\left(n\right)}}^{n}}{\partial t}\left(t,x\right)=-\frac{\mathrm{d}}{\mathrm{d}t}\left[\frac{1}{b_{n}\left(t\right)}\frac{\mathrm{d}}{\mathrm{d}t}\left[B_{n}\left(t\right)\left(a_{n}\left(t\right)^{2}+b_{n}\left(t\right)^{2}\right)\right]\right]\varphi\left(\lambda_{n}x_{1}\right)\varphi\left(\lambda_{n}x_{2}\right)\sin\left(x_{1}\right)\sin\left(x_{2}\right).
\]
Making use of Choice \ref{choice:amplitude density}, we can write
\[
\begin{aligned}\frac{\partial\widetilde{\rho^{\left(n\right)}}^{n}}{\partial t}\left(t,x\right) & =-\frac{\mathrm{d}}{\mathrm{d}t}\left[z_{n}\frac{h^{\left(n\right)}\left(t\right)}{\int_{t_{n}}^{1}h^{\left(n\right)}\left(s\right)b_{n}\left(s\right)\mathrm{d}s}\right]\varphi\left(\lambda_{n}x_{1}\right)\varphi\left(\lambda_{n}x_{2}\right)\sin\left(x_{1}\right)\sin\left(x_{2}\right)=\\
 & =-z_{n}\frac{\frac{\mathrm{d}h^{\left(n\right)}}{\mathrm{d}t}\left(t\right)}{\int_{t_{n}}^{1}h^{\left(n\right)}\left(s\right)b_{n}\left(s\right)\mathrm{d}s}\varphi\left(\lambda_{n}x_{1}\right)\varphi\left(\lambda_{n}x_{2}\right)\sin\left(x_{1}\right)\sin\left(x_{2}\right).
\end{aligned}
\]
We wish to calculate the $\left|\left|\cdot\right|\right|_{C^{1,\alpha}\left(\mathbb{R}^{2}\right)}$
norm of the expression above. Thanks to equation \eqref{eq:property Ckalpha composition},
we obtain that
\[
\left|\left|\varphi\left(\lambda_{n}\cdot\right)\right|\right|_{C^{1,\alpha}\left(\mathbb{R}\right)}\lesssim\left(1+\lambda_{n}\right)^{1+\alpha}\left|\left|\varphi\right|\right|_{C^{1,\alpha}\left(\mathbb{R}\right)}.
\]
Thereby, using equation \eqref{eq:property Ckalpha multiplication}
and employing the fact that $\left|\left|\sin\left(\cdot\right)\right|\right|_{C^{1,\alpha}\left(\mathbb{R}\right)}=1$,
we get to
\[
\left|\left|\frac{\partial\widetilde{\rho^{\left(n\right)}}^{n}}{\partial t}\left(t,\cdot\right)\right|\right|_{C^{1,\alpha}\left(\mathbb{R}^{2}\right)}\lesssim z_{n}\frac{\frac{\mathrm{d}h^{\left(n\right)}}{\mathrm{d}t}\left(t\right)}{\int_{t_{n}}^{1}h^{\left(n\right)}\left(s\right)b_{n}\left(s\right)\mathrm{d}s}\left(1+\lambda_{n}\right)^{2\left(1+\alpha\right)}\left|\left|\varphi\right|\right|_{C^{1,\alpha}\left(\mathbb{R}\right)}^{2}.
\]
Bounding $\alpha<1$ and $\lambda_{n}\le1$, we arrive to
\[
\left|\left|\frac{\partial\widetilde{\rho^{\left(n\right)}}^{n}}{\partial t}\left(t,\cdot\right)\right|\right|_{C^{1,\alpha}\left(\mathbb{R}^{2}\right)}\lesssim_{\varphi}z_{n}\frac{\frac{\mathrm{d}h^{\left(n\right)}}{\mathrm{d}t}\left(t\right)}{\int_{t_{n}}^{1}h^{\left(n\right)}\left(s\right)b_{n}\left(s\right)\mathrm{d}s}.
\]
Applying equation \eqref{eq:property Ckalpha composition}, we obtain
\[
\begin{aligned}\left|\left|\frac{\partial\widetilde{\rho^{\left(n\right)}}^{n}}{\partial t}\left(t,\left(\phi^{\left(n\right)}\right)^{-1}\left(t,\cdot\right)\right)\right|\right|_{C^{1,\alpha}\left(\mathbb{R}^{2}\right)} & \lesssim_{\varphi}\left(1+\max\left\{ \left|\left|\left(\phi^{\left(n\right)}\right)^{-1}\left(t,\cdot\right)\right|\right|_{\dot{C}^{1}\left(\mathbb{R}^{2};\mathbb{R}^{2}\right)},\left|\left|\left(\phi^{\left(n\right)}\right)^{-1}\left(t,\cdot\right)\right|\right|_{\dot{C}^{1,\alpha}\left(\mathbb{R}^{2};\mathbb{R}^{2}\right)}\right\} \right)^{1+\alpha}\cdot\\
 & \quad\cdot z_{n}\frac{\frac{\mathrm{d}h^{\left(n\right)}}{\mathrm{d}t}\left(t\right)}{\int_{t_{n}}^{1}h^{\left(n\right)}\left(s\right)b_{n}\left(s\right)\mathrm{d}s}.
\end{aligned}
\]
Now, taking into account equation \eqref{eq:jacobian inverse}, we
deduce that
\begin{equation}
\begin{aligned}\left|\left|\frac{\partial\widetilde{\rho^{\left(n\right)}}^{n}}{\partial t}\left(t,\left(\phi^{\left(n\right)}\right)^{-1}\left(t,\cdot\right)\right)\right|\right|_{C^{1,\alpha}\left(\mathbb{R}^{2}\right)} & \lesssim_{\varphi}\left(1+\max\left\{ a_{n}\left(t\right),b_{n}\left(t\right)\right\} \right)^{1+\alpha}z_{n}\frac{\frac{\mathrm{d}h^{\left(n\right)}}{\mathrm{d}t}\left(t\right)}{\int_{t_{n}}^{1}h^{\left(n\right)}\left(s\right)b_{n}\left(s\right)\mathrm{d}s}.\end{aligned}
\label{eq:rigorous bound time derivative density}
\end{equation}
As $h^{\left(n\right)}\left(s\right)$ is of order $1$ by Choice
\ref{choice:amplitude density}, it make sense to postulate that
\[
\int_{t_{n}}^{1}h^{\left(n\right)}\left(s\right)b_{n}\left(s\right)\mathrm{d}s\sim\left(1-t_{n}\right)\max_{s\in\left[t_{n},1\right]}b_{n}\left(s\right),\quad\frac{\mathrm{d}h^{\left(n\right)}}{\mathrm{d}t}\left(t\right)\sim\frac{1}{1-t_{n}},
\]
where $1-t_{n}$ is the time scale of layer $n$. Then, we would have
\[
\begin{aligned}\left|\left|\frac{\partial\widetilde{\rho^{\left(n\right)}}^{n}}{\partial t}\left(t,\left(\phi^{\left(n\right)}\right)^{-1}\left(t,\cdot\right)\right)\right|\right|_{C^{1,\alpha}\left(\mathbb{R}^{2}\right)} & \sim\frac{z_{n}}{\left(1-t_{n}\right)^{2}}\frac{\left(1+\max\left\{ a_{n}\left(t\right),b_{n}\left(t\right)\right\} \right)^{1+\alpha}}{\max_{s\in\left[t_{n},1\right]}b_{n}\left(s\right)}\sim\\
 & \sim\frac{z_{n}}{\left(1-t_{n}\right)^{2}}\frac{\max\left\{ 1,a_{n}\left(t\right),b_{n}\left(t\right)\right\} ^{1+\alpha}}{\max_{s\in\left[t_{n},1\right]}b_{n}\left(s\right)}.
\end{aligned}
\]
Resorting to Choice \ref{choice:anbn}, we can write
\[
\begin{aligned}\left|\left|\frac{\partial\widetilde{\rho^{\left(n\right)}}^{n}}{\partial t}\left(t,\left(\phi^{\left(n\right)}\right)^{-1}\left(t,\cdot\right)\right)\right|\right|_{C^{1,\alpha}\left(\mathbb{R}^{2}\right)} & \sim\frac{z_{n}}{\left(1-t_{n}\right)^{2}}\max\left\{ C^{-\left[1+\max_{s\in\left[t_{n},1\right]}k_{n}\left(s\right)\right]\left(\frac{1}{1-\gamma}\right)^{n}},\right.\\
 & \quad C^{\left[\alpha-\left(1+\alpha\right)k_{n}\left(t\right)-\max_{s\in\left[t_{n},1\right]}k_{n}\left(s\right)\right]\left(\frac{1}{1-\gamma}\right)^{n}},\\
 & \quad\left.C^{\left[\alpha+\left(1+\alpha\right)k_{n}\left(t\right)-\max_{s\in\left[t_{n},1\right]}k_{n}\left(s\right)\right]\left(\frac{1}{1-\gamma}\right)^{n}}\right\} .
\end{aligned}
\]
Let us study each term that appears independently.
\begin{itemize}
\item Since $t_{n}\xrightarrow{n\to\infty}1$, because $t=1$ is the time
of the explosion, $1-t_{n}$ has to decrease in $n\in\mathbb{N}$.
In other words, $\frac{1}{1-t_{n}}$ must grow in $n\in\mathbb{N}$.
\item What happens with $z_{n}$? By virtue of Choice \ref{choice:amplitude density},
we have
\[
B_{n}\left(t\right)\left(a_{n}\left(t\right)^{2}+b_{n}\left(t\right)^{2}\right)=z_{n}\frac{\int_{t_{n}}^{t}h^{\left(n\right)}\left(s\right)b_{n}\left(s\right)\mathrm{d}s}{\int_{t_{n}}^{1}h^{\left(n\right)}\left(s\right)b_{n}\left(s\right)\mathrm{d}s}.
\]
On the other hand, recall that $B_{n}\left(t\right)\left(a_{n}\left(t\right)^{2}+b_{n}\left(t\right)^{2}\right)$
is the amplitude of the leading term of the vorticity (see Proposition
\ref{prop:computations vorticity} and Remark \ref{rem:amplitude of vorticity explodes}).
Furthermore, by Remark \ref{rem:amplitude of vorticity explodes},
we know that $B_{n}\left(t\right)\left(a_{n}\left(t\right)^{2}+b_{n}\left(t\right)^{2}\right)$
grows, at least, exponentially in $n\in\mathbb{N}$. Since
\[
\frac{\int_{t_{n}}^{t}h^{\left(n\right)}\left(s\right)b_{n}\left(s\right)\mathrm{d}s}{\int_{t_{n}}^{1}h^{\left(n\right)}\left(s\right)b_{n}\left(s\right)\mathrm{d}s}=1\quad\text{ at }t=1,
\]
we conclude that $z_{n}$ grows, at least, exponentially in $n\in\mathbb{N}$.
\item In this manner, we know that $\frac{1}{1-t_{n}}$ grows in $n\in\mathbb{N}$
and that $z_{n}$ grows, at least, exponentially in $n\in\mathbb{N}$.
So, if we wish for the $\left|\left|\cdot\right|\right|_{C^{1,\alpha}\left(\mathbb{R}^{2}\right)}$
norm of the time derivative of the density to be bounded, we need
the $\max\left\{ \right\} $ term to be decreasing in $n\in\mathbb{N}$.
By its nature, either it grows superexponentially or it decreases
superexponentially. What are the conditions for it to shrink? As $\sup_{n\in\mathbb{N}}\max_{t\in\left[t_{n},1\right]}\left|k_{n}\left(t\right)\right|<1$
by Choice \ref{choice:anbn}, we infer that there is $\varepsilon>0$
such that
\[
1+\max_{s\in\left[t_{n},1\right]}k_{n}\left(s\right)\ge\varepsilon\quad\forall n\in\mathbb{N},
\]
so
\[
C^{-\left[1+\max_{s\in\left[t_{n},1\right]}k_{n}\left(s\right)\right]\left(\frac{1}{1-\gamma}\right)^{n}}\xrightarrow[n\to\infty]{}0
\]
superexponentially. The behavior of the other arguments of the max
depends on the signs of $k_{n}\left(t\right)$ and of $\max_{s\in\left[t_{n},1\right]}k_{n}\left(s\right)$.
If $k_{n}\left(t\right)$ and $\max_{s\in\left[t_{n},1\right]}k_{n}\left(s\right)$
are both negative, then 
\[
\alpha-\left(1+\alpha\right)k_{n}\left(t\right)-\max_{s\in\left[t_{n},1\right]}k_{n}\left(s\right)>0
\]
and the second term grows superexponentially in $n\in\mathbb{N}$,
which is not what we want.

Let us suppose, then, that $\max_{s\in\left[t_{n},1\right]}k_{n}\left(s\right)$
is nonnegative. Then, the second term will decrease with $n\in\mathbb{N}$
as long as
\[
\alpha-\left(1+\alpha\right)k_{n}\left(t\right)-\max_{s\in\left[t_{n},1\right]}k_{n}\left(s\right)<0\iff\alpha\left(1-k_{n}\left(t\right)\right)<\max_{s\in\left[t_{n},1\right]}k_{n}\left(s\right)+k_{n}\left(t\right).
\]
Since $\left|k_{n}\left(t\right)\right|<1$, $1-k_{n}\left(t\right)>0$
and, consequently, the condition above is equivalent to
\[
\alpha<\frac{\max_{s\in\left[t_{n},1\right]}k_{n}\left(s\right)+k_{n}\left(t\right)}{1-k_{n}\left(t\right)}.
\]
Studying the function
\[
f\left(x,y\right)\coloneqq\frac{y+x}{1-x}
\]
in the domain $D\coloneqq\left\{ \left(x,y\right)\in\left(0,1\right)\times\left(-1,1\right)\right\} $,
we see that $\frac{\partial f}{\partial x},\frac{\partial f}{\partial y}>0$
in $D$. This means that, so as to allow $\alpha$ to be as big as
possible, we must make $\max_{s\in\left[t_{n},1\right]}k_{n}\left(s\right)$
as large as we can and we should also try to make $k_{n}\left(t\right)$
large and close to $1$ for all times.

What happens with the third term? It will decrease superexponentially
in $n\in\mathbb{N}$ as long as
\begin{equation}
\begin{aligned} & \alpha+\left(1+\alpha\right)k_{n}\left(t\right)-\max_{s\in\left[t_{n},1\right]}k_{n}\left(s\right)<0\iff\alpha\left(1+k_{n}\left(t\right)\right)+k_{n}\left(t\right)<\max_{s\in\left[t_{n},1\right]}k_{n}\left(s\right)\iff\\
\iff & \alpha<\frac{\max_{s\in\left[t_{n},1\right]}k_{n}\left(s\right)-k_{n}\left(t\right)}{1+k_{n}\left(t\right)}.
\end{aligned}
\label{eq:alpha small for bn time derivative density}
\end{equation}
Studying the function
\[
f\left(x,y\right)\coloneqq\frac{y-x}{1+x}
\]
in the domain $D\coloneqq\left\{ \left(x,y\right)\in\left(0,1\right)\times\left(-1,1\right)\right\} $,
we see that $\frac{\partial f}{\partial x}<0$ and $\frac{\partial f}{\partial y}>0$
in $D$. This means that, like for the second term, it is of our interest
to choose $\max_{s\in\left[t_{n},1\right]}k_{n}\left(s\right)$ as
big as we can, but, simultaneously, we should not take $k_{n}\left(t\right)$
too big. This might seem paradoxical: we need to take $k_{n}\left(t\right)$
big and small at the same time... How can we escape this apparent
paradox? The answer is to look back at the rigorous bound \eqref{eq:rigorous bound time derivative density}.
To continue after that equation, we assumed that $\frac{\mathrm{d}h^{\left(n\right)}}{\mathrm{d}t}\left(t\right)\sim\frac{1}{1-t_{n}}$.
However, what would happen if we could make $\frac{\mathrm{d}h^{\left(n\right)}}{\mathrm{d}t}\left(t\right)=0$
when $k_{n}\left(t\right)$ is big? Then, we would not have to worry
about satisfying \eqref{eq:alpha small for bn time derivative density}
when $k_{n}\left(t\right)$ is big and we would just have to satisfy
it when $k_{n}\left(t\right)$ is small. This is a key idea of the
construction.

In any case, we see that it would be disadvantageous to let $z_{n}$
or $\frac{1}{1-t_{n}}$ grow superexponentially in $n\in\mathbb{N}$,
because, if it were the case, we would have to compensate that growth
with the $\max\left\{ \right\} $, further restricting the values
of $\alpha$ for which $\left|\left|\frac{\partial\widetilde{\rho^{\left(n\right)}}^{n}}{\partial t}\left(t,\left(\phi^{\left(n\right)}\right)^{-1}\left(t,\cdot\right)\right)\right|\right|_{C^{1,\alpha}\left(\mathbb{R}^{2}\right)}$
is bounded. Nevertheless, we will see in subsection \ref{subsec:time convergence of bn(t)}
that it is unavoidable and we will have to deal with this fact.
\end{itemize}
\begin{summary}
\label{sum:ideas time derivative density}In this subsection, we have
introduced the following ideas:
\begin{enumerate}
\item We need to make $\max_{s\in\left[t_{n},1\right]}k_{n}\left(s\right)$
big (big means close to $1$, because we have the restriction $\max_{s\in\left[t_{n},1\right]}\left|k_{n}\left(s\right)\right|<1$
by Choice \ref{choice:anbn}).
\item We should have $\frac{\mathrm{d}h^{\left(n\right)}}{\mathrm{d}t}\left(t\right)=0$
when $k_{n}\left(t\right)$ is big.
\item $z_{n}$ and $\frac{1}{1-t_{n}}$ should grow in $n\in\mathbb{N}$
as slowly as possible.
\end{enumerate}
\end{summary}

\subsection{Equation for $k_{n}\left(t\right)$}

In subsection \ref{subsec:transport of the layer centers}, we introduced
the toy model that governs the time dynamics of the position of the
layer centers. In this subsection, we will present another toy model:
the one that describes the time evolution of $k_{n}\left(t\right)$.
Understanding the dynamics of $k_{n}\left(t\right)$ will be important
to know when $k_{n}\left(t\right)$ is big and when it is small, which,
according to summary \ref{sum:ideas time derivative density}, is
a key requirement to bound the time derivative of the density.

Coming back to Choice \ref{choice:anbncn}, we have
\[
\frac{\mathrm{d}}{\mathrm{d}t}\left(\ln\left(b_{n}\left(t\right)\right)\right)=\sum_{m=1}^{n-1}B_{m}\left(t\right)b_{m}\left(t\right)a_{m}\left(t\right)\cos\left(a_{m}\left(t\right)\left(\phi_{1}^{\left(n\right)}\left(t,0\right)-\phi_{1}^{\left(m\right)}\left(t,0\right)\right)\right)\quad\forall t\in\left[t_{n},1\right].
\]
Using that 
\[
b_{n}\left(t\right)=C^{\left(\frac{1}{1-\gamma}\right)^{n}\left(1+k_{n}\left(t\right)\right)}
\]
by Choice \ref{choice:anbn} leads to
\[
\begin{aligned}\ln\left(b_{n}\left(t\right)\right) & =\left(1+k_{n}\left(t\right)\right)\left(\frac{1}{1-\gamma}\right)^{n}\ln\left(C\right),\\
\frac{\mathrm{d}}{\mathrm{d}t}\left(\ln\left(b_{n}\left(t\right)\right)\right) & =\left(\frac{1}{1-\gamma}\right)^{n}\ln\left(C\right)\frac{\mathrm{d}k_{n}}{\mathrm{d}t}\left(t\right).
\end{aligned}
\]
Consequently, we have
\begin{equation}
\frac{\mathrm{d}k_{n}}{\mathrm{d}t}\left(t\right)=\frac{1}{\ln\left(C\right)\left(\frac{1}{1-\gamma}\right)^{n}}\sum_{m=1}^{n-1}B_{m}\left(t\right)b_{m}\left(t\right)a_{m}\left(t\right)\cos\left(a_{m}\left(t\right)\left(\phi_{1}^{\left(n\right)}\left(t,0\right)-\phi_{1}^{\left(m\right)}\left(t,0\right)\right)\right)\quad\forall t\in\left[t_{n},1\right].\label{eq:real dkntdt}
\end{equation}
If we apply the same logic as in subsection \ref{subsec:transport of the layer centers},
the dominant term in the sum should be the one with $m=n-1$ and we
should be able to approximate $B_{n-1}\left(t\right)\sim B_{n-1}\left(1\right)$,
$b_{n-1}\left(t\right)\sim b_{n-1}\left(1\right)$ and $a_{n-1}\left(t\right)\sim a_{n-1}\left(1\right)$,
because we are considering times $t\in\left[t_{n},1\right]$. This
means that our toy model for $k_{n}\left(t\right)$ solves
\begin{equation}
\frac{\mathrm{d}\overline{k}_{n}\left(t\right)}{\mathrm{d}t}=\frac{B_{n-1}\left(1\right)b_{n-1}\left(1\right)a_{n-1}\left(1\right)}{\ln\left(C\right)\left(\frac{1}{1-\gamma}\right)^{n}}\cos\left(a_{n-1}\left(1\right)\Xi_{0}^{\left(n\right)}\left(t\right)\right)\quad\forall t\in\left[t_{n},1\right],\label{eq:ideal dkntdt}
\end{equation}
where $\Xi_{0}^{\left(n\right)}\left(t\right)$ is the toy model of
$\Xi^{\left(n\right)}\left(t\right)$ introduced in subsection \ref{subsec:transport of the layer centers}.
To continue, we need the following Lemma.
\begin{lem}
\label{lem:integral cos of the good ODE}Let $F\left(t\right)$ be
a solution to the initial value problem given in Lemma \ref{lem:the good ODE}.
Then,
\[
\int_{0}^{t}\cos\left(F\left(\tau\right)\right)\mathrm{d}\tau=\ln\left(\frac{\sin\left(F\left(t\right)\right)}{\sin\left(F\left(0\right)\right)}\right)=\ln\left(\frac{\cosh\left(t_{\max}\right)}{\cosh\left(t_{\max}-t\right)}\right).
\]
\end{lem}
\begin{proof}
By Lemma \ref{lem:the good ODE}, $F\left(t\right)$ satisfies
\begin{equation}
\frac{\mathrm{d}F}{\mathrm{d}t}\left(t\right)=\sin\left(F\left(t\right)\right).\label{eq:ODE of F}
\end{equation}
If we differentiate again, we obtain
\[
\frac{\mathrm{d}^{2}F}{\mathrm{d}t^{2}}\left(t\right)=\cos\left(F\left(t\right)\right)\frac{\mathrm{d}F}{\mathrm{d}t}\left(t\right),
\]
which can be rewritten as
\[
\frac{\mathrm{d}}{\mathrm{d}t}\left(\ln\left(\frac{\mathrm{d}F}{\mathrm{d}t}\left(t\right)\right)\right)=\cos\left(F\left(t\right)\right).
\]
Substituting this into the integral of the statement, we arrive to
\[
\int_{0}^{t}\cos\left(F\left(\tau\right)\right)\mathrm{d}\tau=\ln\left(\frac{\mathrm{d}F}{\mathrm{d}t}\left(t\right)\right)-\ln\left(\frac{\mathrm{d}F}{\mathrm{d}t}\left(0\right)\right)=\ln\left(\frac{\frac{\mathrm{d}F}{\mathrm{d}t}\left(t\right)}{\frac{\mathrm{d}F}{\mathrm{d}t}\left(0\right)}\right).
\]
Using \eqref{eq:ODE of F}, we obtain the first equality of the statement.
To obtain the second, we use point 5 of Lemma \ref{lem:the good ODE}.
\end{proof}
Now, with this information, coming back to equation \eqref{eq:ideal dkntdt},
we can integrate it explicitly to obtain
\[
\overline{k}_{n}\left(t\right)-\overline{k}_{n}\left(t_{n}\right)=\frac{B_{n-1}\left(1\right)b_{n-1}\left(1\right)a_{n-1}\left(1\right)}{\ln\left(C\right)\left(\frac{1}{1-\gamma}\right)^{n}}\int_{t_{n}}^{t}\cos\left(a_{n-1}\left(1\right)\Xi_{0}^{\left(n\right)}\left(s\right)\right)\mathrm{d}s.
\]
Resorting to the change of variables that appeared in equation \eqref{eq:change of variables JI0 Fn},
we get
\[
\overline{k}_{n}\left(t\right)-\overline{k}_{n}\left(t_{n}\right)=\frac{B_{n-1}\left(1\right)b_{n-1}\left(1\right)a_{n-1}\left(1\right)}{\ln\left(C\right)\left(\frac{1}{1-\gamma}\right)^{n}}\int_{t_{n}}^{t}\cos\left(F_{n}\left(B_{n-1}\left(1\right)a_{n-1}\left(1\right)b_{n-1}\left(1\right)\left(s-t_{n}\right)\right)\right)\mathrm{d}s,
\]
where $F_{n}$ satisfies the ODE presented in Lemma \ref{lem:the good ODE}.
Undertaking the change of variables
\[
\begin{aligned}w & =B_{n-1}\left(1\right)a_{n-1}\left(1\right)b_{n-1}\left(1\right)\left(s-t_{n}\right),\\
\mathrm{d}w & =B_{n-1}\left(1\right)a_{n-1}\left(1\right)b_{n-1}\left(1\right)\mathrm{d}s,\\
s=t_{n} & \iff w=0,\\
s=t & \iff w=B_{n-1}\left(1\right)a_{n-1}\left(1\right)b_{n-1}\left(1\right)\left(t-t_{n}\right),
\end{aligned}
\]
we arrive to
\[
\overline{k}_{n}\left(t\right)-\overline{k}_{n}\left(t_{n}\right)=\frac{1}{\ln\left(C\right)\left(\frac{1}{1-\gamma}\right)^{n}}\int_{0}^{B_{n-1}\left(1\right)a_{n-1}\left(1\right)b_{n-1}\left(1\right)\left(t-t_{n}\right)}\cos\left(F_{n}\left(w\right)\right)\mathrm{d}w.
\]
Using Lemma \ref{lem:integral cos of the good ODE}, we deduce that
\begin{equation}
\begin{aligned}\overline{k}_{n}\left(t\right)-\overline{k}_{n}\left(t_{n}\right) & =\frac{1}{\ln\left(C\right)\left(\frac{1}{1-\gamma}\right)^{n}}\ln\left(\frac{\sin\left(F_{n}\left(B_{n-1}\left(1\right)a_{n-1}\left(1\right)b_{n-1}\left(1\right)\left(t-t_{n}\right)\right)\right)}{\sin\left(F_{n}\left(0\right)\right)}\right)=\\
 & =\frac{1}{\ln\left(C\right)\left(\frac{1}{1-\gamma}\right)^{n}}\ln\left(\frac{\sin\left(a_{n-1}\left(1\right)\Xi_{0}^{\left(n\right)}\left(t\right)\right)}{\sin\left(a_{n-1}\left(1\right)\Xi_{0}^{\left(n\right)}\left(t_{n}\right)\right)}\right).
\end{aligned}
\label{eq:first version kn}
\end{equation}
Although it may not seem so at first glance, equation \eqref{eq:first version kn}
is very important to our construction, because it relates $\max_{s\in\left[0,1\right]}k_{n}\left(s\right)$
with $\sin\left(a_{n-1}\left(1\right)\Xi_{0}^{\left(n\right)}\left(t_{n}\right)\right)$,
which is connected to the $\hat{\hat{t}}_{\max}^{\left(n\right)}$
that appears in Remark \ref{rem:relation delta M}. Furthermore, we
see that the behavior of $\overline{k}_{n}\left(t\right)$ is deeply
connected to that of an inverted degenerate half-pendulum. Lastly,
the oscillatory behavior of $\sin\left(a_{n-1}\left(1\right)\Xi_{0}^{\left(n\right)}\left(t\right)\right)$
means that, if needed, we can make $\overline{k}_{n}\left(t\right)$
grow and then shrink.

\subsection{\label{subsec:peak into convergence of Bnt}Brief peak into the time
convergence of $B_{n}\left(t\right)$}

Recall that, in the introduction (see subsection \ref{subsec:vorticity growth mechanism}),
we argued that we could consider the parameters of old layers as constant
in time. Moreover, we have used this assumption multiple times when
deriving the toy models for $\Xi^{\left(n\right)}\left(t\right)$
and $k_{n}\left(t\right)$. Nonetheless, is it really true? Or, more
concretely, when is it true for $B_{n}\left(t\right)$?

Consider the simplest case, which is $n=1$. By Choice \ref{choice:anbncn},
we must have
\[
\frac{\mathrm{d}}{\mathrm{d}t}\left(\ln\left(b_{1}\left(t\right)\right)\right)=\sum_{m=1}^{1-1}B_{m}\left(t\right)a_{m}\left(t\right)b_{m}\left(t\right)\cos\left(a_{m}\left(t\right)\left(\phi_{1}^{\left(1\right)}\left(t,0\right)-\phi_{1}^{\left(m\right)}\left(t,0\right)\right)\right)=0.
\]
Consequently, $b_{1}\left(t\right)$ is constant in time. By Choice
\ref{choice:anbncn}, we know that
\[
\frac{\mathrm{d}}{\mathrm{d}t}\left(a_{n}\left(t\right)b_{n}\left(t\right)\right)=0
\]
and, as a consequence, $a_{1}\left(t\right)$ is also constant in
time. Let us check what happens with $B_{1}\left(t\right)$. By Choice
\ref{choice:amplitude density}, 
\[
B_{1}\left(t\right)=\frac{z_{1}}{a_{1}\left(t\right)^{2}+b_{1}\left(t\right)^{2}}\frac{\int_{t_{1}}^{t}h^{\left(1\right)}\left(s\right)b_{1}\left(s\right)\mathrm{d}s}{\int_{t_{1}}^{1}h^{\left(1\right)}\left(s\right)b_{1}\left(s\right)\mathrm{d}s}.
\]
As $b_{1}\left(t\right)$ and $a_{1}\left(t\right)$ are constant
in time, we obtain
\[
B_{1}\left(t\right)=\frac{z_{1}}{a_{1}\left(1\right)^{2}+b_{1}\left(1\right)^{2}}\frac{\int_{t_{1}}^{t}h^{\left(1\right)}\left(s\right)\mathrm{d}s}{\int_{t_{1}}^{1}h^{\left(1\right)}\left(s\right)\mathrm{d}s}.
\]
In this way,
\[
\frac{B_{1}\left(t\right)}{B_{1}\left(1\right)}=\frac{\int_{t_{1}}^{t}h^{\left(1\right)}\left(s\right)\mathrm{d}s}{\int_{t_{1}}^{1}h^{\left(1\right)}\left(s\right)\mathrm{d}s}.
\]
We would like the quantity above very close to $1$ for times $t\in\left[t_{2},1\right]$.
In fact, we can aim higher: we may take $\left.h^{\left(1\right)}\right|_{\left[t_{2},1\right]}\equiv0$,
which eliminates the problem entirely, as, in that case, $B_{1}\left(t\right)=B_{1}\left(1\right)$
$\forall t\in\left[t_{2},1\right]$. Furthermore, notice that, by
Choice \ref{choice:amplitude density}, the amplitude of each density
layer is proportional to $h^{\left(n\right)}$. Taking $\left.h^{\left(1\right)}\right|_{\left[t_{2},1\right]}\equiv0$
means that the amplitude of $\widetilde{\rho^{\left(1\right)}}^{1}\left(t,x\right)$
would be zero $\forall t\in\left[t_{2},1\right]$, i.e., at the time
when we turn on the second layer of density, the first layer of density
will already have disappeared. This, as one can imagine, will make
computations easier. The same argument works if we extend it to all
$h^{\left(n\right)}$'s by analogy: taking $\left.h^{\left(n\right)}\right|_{\left[t_{n+1},0\right]}\equiv0$
$\forall n\in\mathbb{N}$ guarantees that, as we already know that
$\left.h^{\left(n\right)}\right|_{\left[0,t_{n}\right]}\equiv0$ $\forall n\in\mathbb{N}$,
$\widetilde{\rho^{\left(n\right)}}^{n}\left(t,x\right)$ will only
be non-zero $\forall t\in\left[t_{n},t_{n+1}\right]$. Hence, at any
given time instant, there will only be one density layer which is
non-zero.
\begin{choice}
\label{choice:no two simultanous densities}Considering the above,
it makes sense to choose $\left.h^{\left(n\right)}\right|_{\left[t_{n+1},1\right]}\equiv0$,
where $h^{\left(n\right)}$ is the function introduced in Choice \ref{choice:amplitude density}.
Thereby, only one density layer is non-zero at any given moment.
\end{choice}

\subsection{\label{subsec:completion of the toy model}Completion of the toy
model}

In this subsection, we will use all the information we have acquired
until this moment to choose the profiles of $\overline{k}_{n}\left(t\right)$
and $h^{\left(n\right)}\left(t\right)$. By Choice \ref{choice:amplitude density},
we have $\left.h^{\left(n\right)}\right|_{\left[0,t_{n}\right]}\equiv0$.
By Choice \ref{choice:no two simultanous densities}, we know that
$\left.h^{\left(n\right)}\right|_{\left[t_{n+1},1\right]}\equiv0$.
Moreover, by Summary \ref{sum:ideas time derivative density}, we
know we should choose $\frac{\mathrm{d}h^{\left(n\right)}}{\mathrm{d}t}\left(t\right)=0$,
at least, when $\overline{k}_{n}\left(t\right)$ is big. How do we
reconcile these three requirements? It is clear that we cannot simply
take $h^{\left(n\right)}\equiv0$ in $\left[0,1\right]$, because,
in that case, in view of Choices \ref{choice:density} and \ref{choice:amplitude density},
we would have no density. Hence, we need $h^{\left(n\right)}$ to
increase and decrease again in the interval $\left[t_{n},t_{n+1}\right]$,
assuring that $h^{\left(n\right)}\equiv\text{constant}$ when $\overline{k}_{n}\left(t\right)$
is big. One way to this is the following:
\begin{enumerate}
\item Start with $\overline{k}_{n}\left(t_{n}\right)=0$ and $h^{\left(n\right)}\left(t_{n}\right)=0$.
\item Rapidly increase $h^{\left(n\right)}$ from $0$ to a certain value,
let us say, $1$. During this process, $\overline{k}_{n}\left(t\right)$
remains small, so, even if $\frac{\mathrm{d}h^{\left(n\right)}}{\mathrm{d}t}\sim\frac{1}{1-t_{n}}$,
as long as $\frac{1}{1-t_{n}}$ does not grow too fast (as pointed
out in Summary \ref{sum:ideas time derivative density}), $\left|\left|\frac{\partial\widetilde{\rho^{\left(n\right)}}^{n}}{\partial t}\left(t,\left(\phi^{\left(n\right)}\right)^{-1}\left(t,\cdot\right)\right)\right|\right|_{C^{1,\alpha}\left(\mathbb{R}^{2}\right)}$
should remain bounded in $n\in\mathbb{N}$.
\item While we maintain $h^{\left(n\right)}=1$, we let $\overline{k}_{n}\left(t\right)$
grow until it reaches its maximum. During this phase, as $\frac{\mathrm{d}h^{\left(n\right)}}{\mathrm{d}t}\left(t\right)=0$,
we have $\left|\left|\frac{\partial\widetilde{\rho^{\left(n\right)}}^{n}}{\partial t}\left(t,\left(\phi^{\left(n\right)}\right)^{-1}\left(t,\cdot\right)\right)\right|\right|_{C^{1,\alpha}\left(\mathbb{R}^{2}\right)}=0$.
\item Then, still maintaining $h^{\left(n\right)}=1$, we let $\overline{k}_{n}\left(t\right)$
decrease until it is almost zero. Again, during this phase, we have
$\left|\left|\frac{\partial\widetilde{\rho^{\left(n\right)}}^{n}}{\partial t}\left(t,\left(\phi^{\left(n\right)}\right)^{-1}\left(t,\cdot\right)\right)\right|\right|_{C^{1,\alpha}\left(\mathbb{R}^{2}\right)}=0$.
\item At this point, we rapidly decrease $h^{\left(n\right)}$ from $1$
to $0$, achieving $h^{\left(n\right)}\left(t_{n+1}\right)=0$. Again,
during this process, $\overline{k}_{n}\left(t\right)$ remains small,
so, even if $\frac{\mathrm{d}h^{\left(n\right)}}{\mathrm{d}t}\sim\frac{1}{1-t_{n}}$,
as long as $\frac{1}{1-t_{n}}$ does not grow too fast (as pointed
out in Summary \ref{sum:ideas time derivative density}), $\left|\left|\frac{\partial\widetilde{\rho^{\left(n\right)}}^{n}}{\partial t}\left(t,\left(\phi^{\left(n\right)}\right)^{-1}\left(t,\cdot\right)\right)\right|\right|_{C^{1,\alpha}\left(\mathbb{R}^{2}\right)}$
should remain bounded in $n\in\mathbb{N}$.
\item We end with $\overline{k}_{n}\left(1\right)=0$.
\end{enumerate}
Figure \ref{fig:time picture} displays this information schematically.

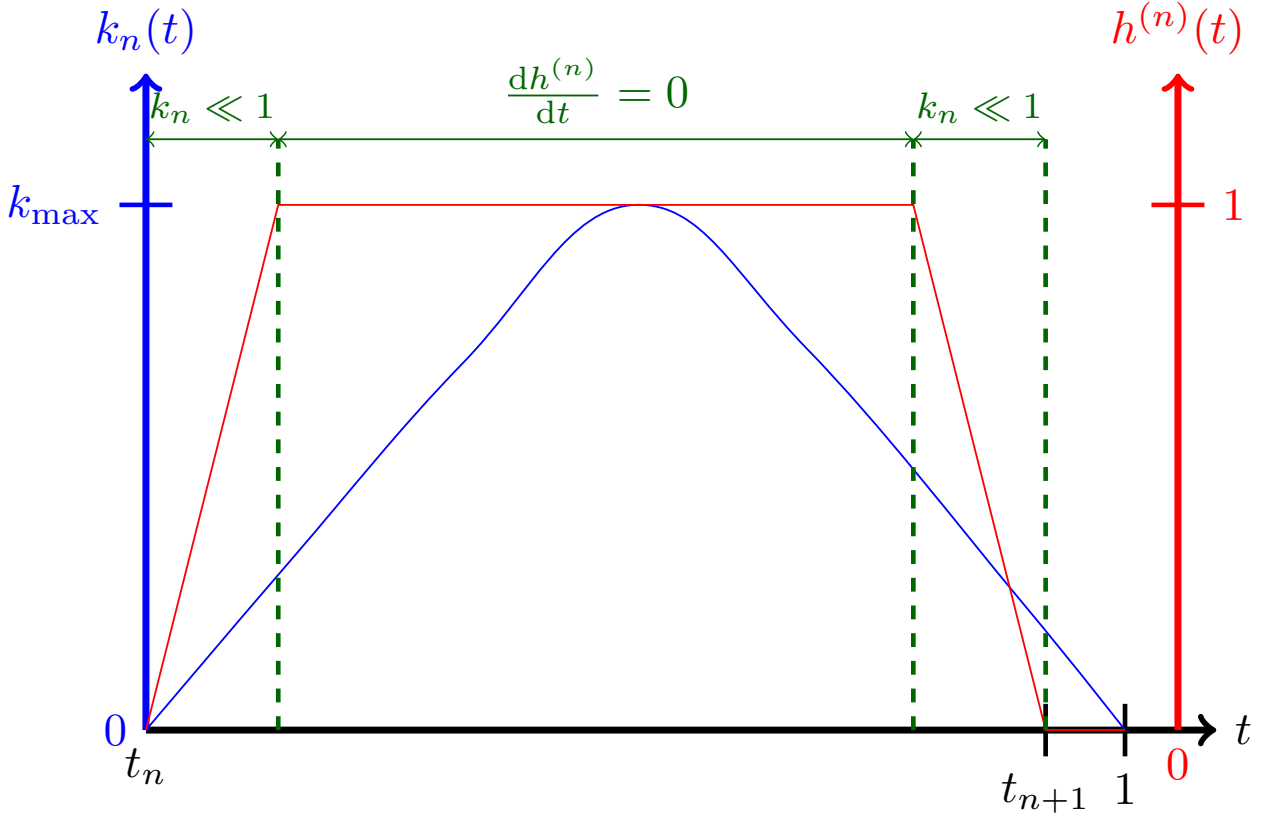
\begin{figure}[h]
\resizebox{\columnwidth}{!}{%
\begin{tikzpicture}
	\draw [->, line width = 1.5] (0,0) node [below] {$t_n$} --(5*1.618,0) node [right] {$t$};
	\draw [->, line width = 1.5, blue] (0,0) node [left] {$0$} --(0,5) node [above] {$k_n(t)$};
	\draw [->, line width = 1.5, red] (7.8,0) node [below] {$0$}--(7.8,5) node [above] {$h^{(n)}(t)$};
	\draw [line width = 1] (7.4,0.2)--++(0,-0.4) node [below] {$1$};
	\draw [line width = 1] (6.8,0.2)--++(0,-0.4) node [below] {$t_{n+1}$};
	
\draw [blue] plot[smooth, tension=0.7] coordinates {(0,0) (1.1,1.3) (2.4,2.8) (3.7,4) (5,2.9) (6.6,1) (7.4,0)};
\draw [blue, line width=1] (-0.2, 4) node [left] {$k_{\mathrm{max}}$} --++(0.4,0);
\draw [red, line width=1] (7.8-0.2, 4) --++(0.4,0) node [right] {$1$};

\draw [red] (0,0) -- (1,4) -- (5.8,4) -- (6.8,0) -- (7.4,0);
\draw [green!40!black, dashed, line width= 1] (1,0)--++(0,4.5);
\draw [green!40!black, dashed, line width= 1] (5.8,0)--++(0,4.5);
\draw [green!40!black, dashed, line width= 1] (6.8,0)--++(0,4.5);

\draw [green!40!black, <->] (1,4.5)--(5.8,4.5) node [midway, above] {$\frac{\mathrm{d}h^{(n)}}{\mathrm{d}t}=0$};
\draw [green!40!black, <->] (0,4.5)--(1,4.5) node [midway, above] {\footnotesize$k_n\ll1$};
\draw [green!40!black, <->] (5.8,4.5)--(6.8,4.5) node [midway, above] {\footnotesize$k_n\ll1$};
\end{tikzpicture}
}

\caption{\label{fig:time picture}Graphical representation of the evolution
of $k_{n}\left(t\right)$ and $h^{\left(n\right)}\left(t\right)$.}

\end{figure}

\begin{rem}
One reasonable question the reader might ask is why do we wish to
finish with $\overline{k}_{n}\left(1\right)=0$ instead of $\overline{k}_{n}\left(t_{n+1}\right)=0$.
The answer is that, because we expect $\frac{1-t_{n+1}}{1-t_{n}}$
to be small, both choices should work, but the computations will be
a little bit easier if we take $\overline{k}_{n}\left(1\right)=0$.
\end{rem}
Can we actually choose our parameters so that the construction presented
in Figure \ref{fig:time picture} is realized? First of all, notice
that we need $\overline{k}_{n}\left(t\right)$ to grow and then decrease.
By equation \eqref{eq:first version kn}, for this to happen, we require
$\sin\left(a_{n-1}\left(1\right)\Xi_{0}^{\left(n\right)}\left(t\right)\right)$
to increase and, then, decrease. Besides, as $\overline{k}_{n}\left(t_{n}\right)=0$,
we need $\overline{k}_{n}\left(1\right)=\overline{k}_{n}\left(t_{n}\right)$,
i.e.,
\[
\sin\left(a_{n-1}\left(1\right)\Xi_{0}^{\left(n\right)}\left(t_{n}\right)\right)=\sin\left(a_{n-1}\left(1\right)\Xi_{0}^{\left(n\right)}\left(1\right)\right).
\]
Using the change of variables provided in equation \eqref{eq:change of variables JI0 Fn},
we arrive to
\[
\sin\left(F_{n}\left(0\right)\right)=\sin\left(F_{n}\left(B_{n-1}\left(1\right)a_{n-1}\left(1\right)b_{n-1}\left(1\right)\left(1-t_{n}\right)\right)\right).
\]
By Choice \ref{choice:Bnanbn}, we can write
\[
\sin\left(F_{n}\left(0\right)\right)=\sin\left(F_{n}\left(M_{n-1}\left(1-t_{n}\right)\right)\right).
\]
Using Lemma \ref{lem:the good ODE} like in Remark \ref{rem:relation delta M}
provides
\begin{equation}
\frac{1}{\cosh\left(\hat{\hat{t}}_{\max}^{\left(n\right)}\right)}=\frac{1}{\cosh\left(\hat{\hat{t}}_{\max}^{\left(n\right)}-M_{n-1}\left(1-t_{n}\right)\right)},\label{eq:equality of cosh}
\end{equation}
where
\[
\hat{\hat{t}}_{\max}^{\left(n\right)}=\ln\left(\frac{1+\cos\left(a_{n-1}\left(1\right)\Xi_{0}^{\left(n\right)}\left(t_{n}\right)\right)}{\sin\left(a_{n-1}\left(1\right)\Xi_{0}^{\left(n\right)}\left(t_{n}\right)\right)}\right).
\]
Actually, employing Lemma \ref{lem:the good ODE}, we may rewrite
$\hat{\hat{t}}_{\max}^{\left(n\right)}$ a little bit differently:
\[
\sin\left(a_{n-1}\left(1\right)\Xi_{0}^{\left(n\right)}\left(t_{n}\right)\right)=\sin\left(F_{n}\left(0\right)\right)=\frac{1}{\cosh\left(\hat{\hat{t}}_{\max}^{\left(n\right)}\right)}\iff\hat{\hat{t}}_{\max}^{\left(n\right)}=\mathrm{arccosh}\left(\frac{1}{\sin\left(a_{n-1}\left(1\right)\Xi_{0}^{\left(n\right)}\left(t_{n}\right)\right)}\right).
\]
\eqref{eq:equality of cosh} can only be satisfied if
\begin{equation}
M_{n-1}\left(1-t_{n}\right)=2\hat{\hat{t}}_{\max}^{\left(n\right)}=2\mathrm{arccosh}\left(\frac{1}{\sin\left(a_{n-1}\left(1\right)\Xi_{0}^{\left(n\right)}\left(t_{n}\right)\right)}\right).\label{eq:periodicity of the evolution}
\end{equation}
Notice that, by choosing the relation between $M_{n-1}$ and $\hat{\hat{t}}_{\max}^{\left(n\right)}$
this way, we are in the next-to-last case of Remark \ref{rem:relation delta M}
(in other words, our choice is represented by the blue line of figure
\ref{fig:relation delta M}). This means that there is a time $t_{n}^{*}$
when $\sin\left(a_{n-1}\left(1\right)\Xi_{0}^{\left(n\right)}\left(t\right)\right)$
attains its maximum value $\sin\left(a_{n-1}\left(1\right)\Xi_{0}^{\left(n\right)}\left(t_{n}^{*}\right)\right)=1$.
By point 3 of Lemma \ref{lem:the good ODE}, this time must satisfy
the equation
\[
M_{n-1}\left(t_{n}^{*}-t_{n}\right)=\hat{\hat{t}}_{\max}^{\left(n\right)}.
\]
Consequently, by equation \eqref{eq:periodicity of the evolution},
it is exactly
\[
t_{n}^{*}=t_{n}+\frac{1-t_{n}}{2}.
\]

On the other hand, by equation \eqref{eq:first version kn}, taking
into account that $\overline{k}_{n}\left(t_{n}\right)=0$, we obtain
\begin{equation}
\max_{s\in\left[t_{n},1\right]}\overline{k}_{n}\left(s\right)=\overline{k}_{n}\left(t_{n}^{*}\right)=\frac{1}{\ln\left(C\right)\left(\frac{1}{1-\gamma}\right)^{n}}\ln\left(\frac{1}{\sin\left(a_{n-1}\left(1\right)\Xi_{0}^{\left(n\right)}\left(t_{n}\right)\right)}\right),\label{eq:max(kn)}
\end{equation}
since $\sin\left(a_{n-1}\left(1\right)\Xi_{0}^{\left(n\right)}\left(t_{n}^{*}\right)\right)=1$.
This means that $\max_{s\in\left[t_{n},1\right]}\overline{k}_{n}\left(s\right)$
is directly related to $\sin\left(a_{n-1}\left(1\right)\Xi_{0}^{\left(n\right)}\left(t_{n}\right)\right)$.
Indeed,
\[
\sin\left(a_{n-1}\left(1\right)\Xi_{0}^{\left(n\right)}\left(t_{n}\right)\right)=C^{-\left(\max_{s\in\left[t_{n},1\right]}\overline{k}_{n}\left(s\right)\right)\left(\frac{1}{1-\gamma}\right)^{n}}.
\]
Then, equation \eqref{eq:periodicity of the evolution} allows us
to solve for the time scale $1-t_{n}$:
\begin{equation}
1-t_{n}=\frac{2}{M_{n-1}}\mathrm{arccosh}\left(C^{\left(\max_{s\in\left[t_{n},1\right]}\overline{k}_{n}\left(s\right)\right)\left(\frac{1}{1-\gamma}\right)^{n}}\right).\label{eq:time scale}
\end{equation}
For such big numbers, we can approximate $\mathrm{arccosh}\left(x\right)\sim\ln\left(2x\right)$
(see Lemma \ref{lem:arccosh by ln}), which leads to
\[
1-t_{n}\sim\frac{2}{M_{n-1}}\left[\ln\left(C\right)\left(\frac{1}{1-\gamma}\right)^{n}\left(\max_{s\in\left[t_{n},1\right]}\overline{k}_{n}\left(s\right)\right)+\ln\left(2\right)\right].
\]
Since $\left(\frac{1}{1-\gamma}\right)^{n}$ grows in $n\in\mathbb{N}$,
for $n$ big enough, we may neglect $\ln\left(2\right)$ in comparison
to the other summand, which provides
\[
1-t_{n}\sim\frac{2}{M_{n-1}}\ln\left(C\right)\left(\frac{1}{1-\gamma}\right)^{n}\left(\max_{s\in\left[t_{n},1\right]}\overline{k}_{n}\left(s\right)\right).
\]
As $\left(\frac{1}{1-\gamma}\right)^{n}$ grows exponentially and
the time scale $1-t_{n}$ should decrease in $n\in\mathbb{N}$, we
deduce that $M_{n-1}$ should grow, at least, exponentially, which
is consistent with Choice \ref{choice:Bnanbn}. Moreover, if $M_{n-1}$
grows superexponentially, $1-t_{n}$ decreases superexponentially.
Furthermore, by Choice \ref{choice:time (initial)}, we should have
$t_{1}=0$. For this to happen, we would need
\[
1=1-0=1-t_{1}=\frac{2}{M_{0}}\ln\left(C\right)\left(\frac{1}{1-\gamma}\right)\left(\max_{s\in\left[0,1\right]}\overline{k}_{1}\left(s\right)\right).
\]
This can be achieved by taking $M_{0}$ appropriately.

In addition, this choice of parameters presents an interesting situation:
if we choose $\max_{s\in\left[t_{n},1\right]}\overline{k}_{n}\left(s\right)=k_{\max}$
$\forall n\in\mathbb{N}$, where $k_{\max}\in\left(0,1\right)$, under
an appropriate affine time transformation, $\overline{k}_{n}\left(t\right)$
has a well defined pointwise limit when $n\to\infty$. To see this,
recall that, by equation \eqref{eq:first version kn},
\[
\begin{aligned}\overline{k}_{n}\left(t\right) & =\frac{1}{\ln\left(C\right)\left(\frac{1}{1-\gamma}\right)^{n}}\ln\left(\frac{\sin\left(a_{n-1}\left(1\right)\Xi_{0}^{\left(n\right)}\left(t\right)\right)}{\sin\left(a_{n-1}\left(1\right)\Xi_{0}^{\left(n\right)}\left(t_{n}\right)\right)}\right)=\\
 & =\frac{1}{\ln\left(C\right)\left(\frac{1}{1-\gamma}\right)^{n}}\left[\ln\left(\sin\left(a_{n-1}\left(1\right)\Xi_{0}^{\left(n\right)}\left(t\right)\right)\right)+\ln\left(\frac{1}{\sin\left(a_{n-1}\left(1\right)\Xi_{0}^{\left(n\right)}\left(t_{n}\right)\right)}\right)\right].
\end{aligned}
\]
By equation \eqref{eq:max(kn)}, we obtain
\[
\overline{k}_{n}\left(t\right)=k_{\max}+\frac{\ln\left(\sin\left(a_{n-1}\left(1\right)\Xi_{0}^{\left(n\right)}\left(t\right)\right)\right)}{\ln\left(C\right)\left(\frac{1}{1-\gamma}\right)^{n}}.
\]
Resorting to the change of variables \eqref{eq:change of variables JI0 Fn}
and to Choice \ref{choice:Bnanbn}, we arrive to
\[
\overline{k}_{n}\left(t\right)=k_{\max}+\frac{\ln\left(\sin\left(F_{n}\left(M_{n-1}\left(t-t_{n}\right)\right)\right)\right)}{\ln\left(C\right)\left(\frac{1}{1-\gamma}\right)^{n}}.
\]
Using Lemma \ref{lem:the good ODE}, we deduce that
\[
\overline{k}_{n}\left(t\right)=k_{\max}+\frac{1}{\ln\left(C\right)\left(\frac{1}{1-\gamma}\right)^{n}}\ln\left(\frac{1}{\cosh\left(\hat{\hat{t}}_{\max}^{\left(n\right)}-M_{n-1}\left(t-t_{n}\right)\right)}\right).
\]
By equation \eqref{eq:periodicity of the evolution}, we have
\[
\overline{k}_{n}\left(t\right)=k_{\max}+\frac{1}{\ln\left(C\right)\left(\frac{1}{1-\gamma}\right)^{n}}\ln\left(\frac{1}{\cosh\left(\frac{1}{2}M_{n-1}\left(1-t_{n}\right)-M_{n-1}\left(t-t_{n}\right)\right)}\right).
\]
Introducing the change of variables
\begin{equation}
\hat{\hat{t}}=\frac{t-t_{n}}{1-t_{n}}\iff t=t_{n}+\left(1-t_{n}\right)\hat{\hat{t}},\label{eq:change thathat}
\end{equation}
we infer that
\[
\begin{aligned}\overline{k}_{n}\left(t_{n}+\left(1-t_{n}\right)\hat{\hat{t}}\right) & =k_{\max}+\frac{1}{\ln\left(C\right)\left(\frac{1}{1-\gamma}\right)^{n}}\ln\left(\frac{1}{\cosh\left(\frac{1}{2}M_{n-1}\left(1-t_{n}\right)-M_{n-1}\left(1-t_{n}\right)\hat{\hat{t}}\right)}\right)=\\
 & =k_{\max}+\frac{1}{\ln\left(C\right)\left(\frac{1}{1-\gamma}\right)^{n}}\ln\left(\frac{1}{\cosh\left(\frac{1}{2}M_{n-1}\left(1-t_{n}\right)\left(1-2\hat{\hat{t}}\right)\right)}\right).
\end{aligned}
\]
Using equation \eqref{eq:time scale} leads us to
\begin{equation}
\overline{k}_{n}\left(t_{n}+\left(1-t_{n}\right)\hat{\hat{t}}\right)=k_{\max}+\frac{1}{\ln\left(C\right)\left(\frac{1}{1-\gamma}\right)^{n}}\ln\left(\frac{1}{\cosh\left(\mathrm{arccosh}\left(C^{k_{\max}\left(\frac{1}{1-\gamma}\right)^{n}}\right)\left(1-2\hat{\hat{t}}\right)\right)}\right).\label{eq:kn second version exact}
\end{equation}
Now, for big enough $n$, we can approximate 
\[
\mathrm{arccosh}\left(C^{k_{\max}\left(\frac{1}{1-\gamma}\right)^{n}}\right)\sim\ln\left(2C^{k_{\max}\left(\frac{1}{1-\gamma}\right)^{n}}\right)=\ln\left(2\right)+k_{\max}\left(\frac{1}{1-\gamma}\right)^{n}\ln\left(C\right)\sim k_{\max}\left(\frac{1}{1-\gamma}\right)^{n}\ln\left(C\right).
\]
Therefore,
\[
\overline{k}_{n}\left(t_{n}+\left(1-t_{n}\right)\hat{\hat{t}}\right)\sim k_{\max}+\frac{1}{\ln\left(C\right)\left(\frac{1}{1-\gamma}\right)^{n}}\ln\left(\frac{1}{\cosh\left(k_{\max}\left(\frac{1}{1-\gamma}\right)^{n}\ln\left(C\right)\left(1-2\hat{\hat{t}}\right)\right)}\right).
\]
As long as $k_{\max}\left(\frac{1}{1-\gamma}\right)^{n}\left(1-2\hat{\hat{t}}\right)$
is big (which will happen for every $\hat{\hat{t}}\in\left[0,1\right]\setminus\left\{ \frac{1}{2}\right\} $
provided that $n$ is large enough), we can approximate $\cosh\left(x\right)\sim\frac{1}{2}\mathrm{e}^{\left|x\right|}$,
leading to
\begin{equation}
\begin{aligned}\overline{k}_{n}\left(t_{n}+\left(1-t_{n}\right)\hat{\hat{t}}\right) & \sim k_{\max}+\frac{1}{\ln\left(C\right)\left(\frac{1}{1-\gamma}\right)^{n}}\ln\left(\frac{1}{\frac{1}{2}\mathrm{e}^{k_{\max}\left(\frac{1}{1-\gamma}\right)^{n}\ln\left(C\right)\left|1-2\hat{\hat{t}}\right|}}\right)=\\
 & \sim k_{\max}+\underbrace{\frac{1}{\ln\left(C\right)\left(\frac{1}{1-\gamma}\right)^{n}}\ln\left(2\right)}_{\xrightarrow[n\to\infty]{}0}-\frac{k_{\max}\left(\frac{1}{1-\gamma}\right)^{n}\ln\left(C\right)}{\ln\left(C\right)\left(\frac{1}{1-\gamma}\right)^{n}}\left|1-2\hat{\hat{t}}\right|\sim\\
 & \sim k_{\max}-k_{\max}\left|1-2\hat{\hat{t}}\right|=k_{\max}\left(1-\left|1-2\hat{\hat{t}}\right|\right).
\end{aligned}
\label{eq:kn second version approximate}
\end{equation}
Recall that this pointwise convergence works $\forall\hat{\hat{t}}\in\left[0,1\right]\setminus\left\{ \frac{1}{2}\right\} $.
Nevertheless, it is easy to check that both \eqref{eq:kn second version exact}
and \eqref{eq:kn second version approximate} return the same value
for $\hat{\hat{t}}=\frac{1}{2}$, which means that \eqref{eq:kn second version approximate}
is valid (as a limit when $n\to\infty$) $\forall\hat{\hat{t}}\in\left[0,1\right]$.
Figure \ref{fig:limit profile kn} shows the profile of $\overline{k}_{n}\left(t\right)$
given in equation \eqref{eq:kn second version approximate}.

\begin{figure}[h]
\begin{centering}
\includegraphics[width=0.7\linewidth]{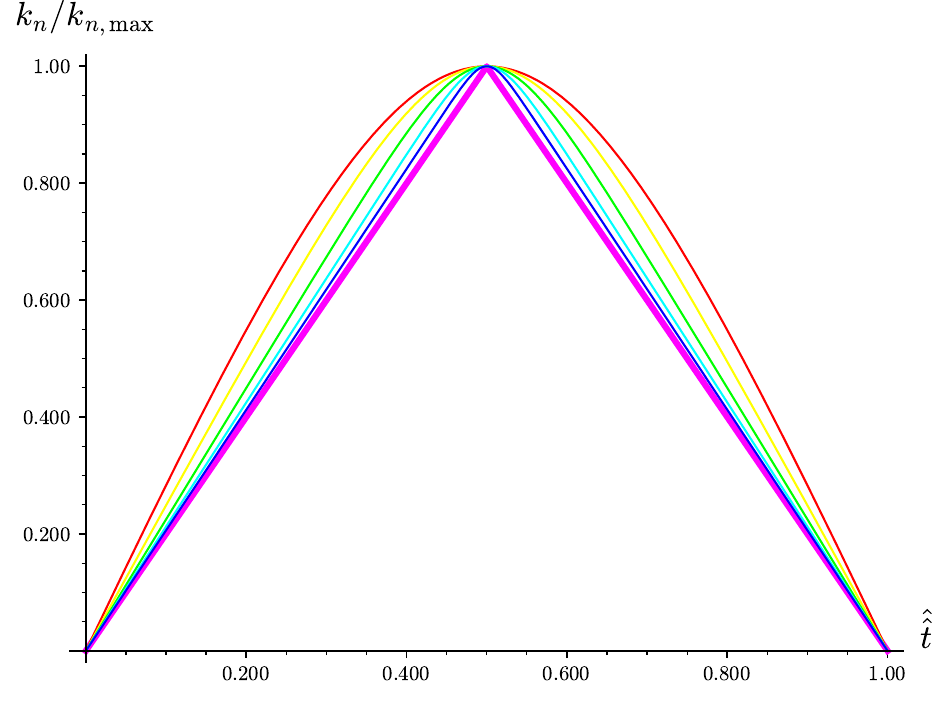}
\par\end{centering}
\caption{\label{fig:limit profile kn}The thick pink line represents the limit
profile of $\overline{k}_{n}\left(t_{n}+\left(1-t_{n}\right)\hat{\hat{t}}\right)$
when $n\to\infty$. The other lines correspond to  $n=1$ (red), $n=2$
(yellow), $n=3$ (green), $n=4$ (light blue) and $n=5$ (dark blue).}

\end{figure}

Another important property of $\overline{k}_{n}\left(t_{n}+\left(1-t_{n}\right)\hat{\hat{t}}\right)$
which is visible in figure \ref{fig:limit profile kn} is that its
convergence to $k_{\max}\left(1-\left|1-2\hat{\hat{t}}\right|\right)$
is monotone in $n\in\mathbb{N}$. In other words, we have the following
Lemma.
\begin{lem}
\label{lem:monotone convergence kn}$\forall n\in\mathbb{N}$ and
$\forall\hat{\hat{t}}\in\left[0,1\right]$, we have
\[
\overline{k}_{n+1}\left(t_{n+1}+\left(1-t_{n+1}\right)\hat{\hat{t}}\right)\le\overline{k}_{n}\left(t_{n}+\left(1-t_{n}\right)\hat{\hat{t}}\right).
\]
\end{lem}
\begin{proof}
Take a look at equation \eqref{eq:kn second version exact}:
\[
\overline{k}_{n}\left(t_{n}+\left(1-t_{n}\right)\hat{\hat{t}}\right)=k_{\max}+\frac{1}{\ln\left(C\right)\left(\frac{1}{1-\gamma}\right)^{n}}\ln\left(\frac{1}{\cosh\left(\mathrm{arccosh}\left(C^{k_{\max}\left(\frac{1}{1-\gamma}\right)^{n}}\right)\left(1-2\hat{\hat{t}}\right)\right)}\right).
\]
We shall consider three cases:
\begin{itemize}
\item $0\leq\hat{\hat{t}}<\frac{1}{2}$. Then, $1-2\hat{\hat{t}}>0$. In
this way, the argument of $\cosh$ is always positive and becomes
more positive as $n\in\mathbb{N}$ grows. Since $\cosh$ is increasing
for positive arguments, we deduce that $\cosh\left(\dots\right)$
(for fixed $\hat{\hat{t}}\in\left[0,\frac{1}{2}\right)$) grows in
$n\in\mathbb{N}$. Hence, $\frac{1}{\cosh\left(\dots\right)}$ is
decreasing in $n\in\mathbb{N}$ and, consequently, $\ln\left(\frac{1}{\cosh\left(\dots\right)}\right)$
is also decreasing in $n\in\mathbb{N}$. As $\frac{1}{\ln\left(C\right)\left(\frac{1}{1-\gamma}\right)^{n}}$
shrinks in $n\in\mathbb{N}$ as well, we conclude that $\overline{k}_{n}\left(t_{n}+\left(1-t_{n}\right)\hat{\hat{t}}\right)$
is decreasing in $n\in\mathbb{N}$.
\item $\hat{\hat{t}}=\frac{1}{2}$. In this case, the argument of $\cosh$
is zero and, consequently, the $\ln$ vanishes, leading to
\[
\overline{k}_{n}\left(t_{n}+\left(1-t_{n}\right)\hat{\hat{t}}\right)=k_{\max}\quad\forall n\in\mathbb{N}.
\]
Thus, the result is true for $\hat{\hat{t}}=\frac{1}{2}$.
\item $\frac{1}{2}<\hat{\hat{t}}\leq1$. Then, $1-2\hat{\hat{t}}<0$. In
this manner, the argument of $\cosh$ is always negative and becomes
more negative as $n\in\mathbb{N}$ grows. Since $\cosh$ is decreasing
for negative arguments, we deduce that $\cosh\left(\dots\right)$
(for fixed $\hat{\hat{t}}\in\left(\frac{1}{2},1\right]$) increases
in $n\in\mathbb{N}$. From here, the argument follows like in the
$0\le\hat{\hat{t}}<\frac{1}{2}$ case.
\end{itemize}
\end{proof}
\begin{cor}
\label{cor:ideal kn non negative}$\overline{k}_{n}\left(t\right)\ge0$
$\forall t\in\left[t_{n},1\right]$. 
\end{cor}
\begin{proof}
Combining Lemma \ref{lem:monotone convergence kn} and equation \eqref{eq:kn second version approximate}
we can easily deduce that, $\forall n\in\mathbb{N}$ and $\forall\hat{\hat{t}}\in\left[0,1\right]$,
\[
\overline{k}_{n}\left(t_{n}+\left(1-t_{n}\right)\hat{\hat{t}}\right)\ge\lim_{m\to\infty}\overline{k}_{m}\left(t_{m}+\left(1-t_{m}\right)\hat{\hat{t}}\right)=k_{\max}\left(1-\left|1-2\hat{\hat{t}}\right|\right)\ge0.
\]
Recalling the change of variables given in \eqref{eq:change thathat},
we arrive to the statement.
\end{proof}
In this way, we have shown that $\overline{k}_{n}\left(t\right)$
actually behaves as portrayed in figure \ref{fig:time picture}. Now,
let us see that we can also take $h^{\left(n\right)}$ in agreement
with the profile given in that picture. The simplest way to do it
is to choose
\begin{equation}
h^{\left(n\right)}\left(t\right)=\left\{ \begin{matrix}0 &  & t\le t_{n},\\
\frac{t-t_{n}}{\left(1-t_{n}\right)\zeta} &  & t_{n}\le t\le t_{n}+\left(1-t_{n}\right)\zeta,\\
1 &  & t_{n}+\left(1-t_{n}\right)\zeta\le t\le t_{n+1}-\left(1-t_{n}\right)\zeta,\\
1-\frac{t-t_{n+1}+\left(1-t_{n}\right)\zeta}{\left(1-t_{n}\right)\zeta} &  & t_{n+1}-\left(1-t_{n}\right)\zeta\le t\le t_{n+1},\\
0 &  & t_{n+1}\le t\le1,
\end{matrix}\right.\label{eq:def hn first way}
\end{equation}
where $0<\zeta<\frac{1}{4}$ is a parameter that controls how fast
$h^{\left(n\right)}$ changes from $0$ to $1$ and viceversa. Notice
that, actually, $\zeta$ must be smaller than a number that lies a
little to the left of $\frac{1}{2}$ for the definition above to make
sense, i.e., for the time intervals no to overlap. As we will take
$\zeta>0$ very small, \eqref{eq:def hn first way} will always be
well-defined. Unfortunately, this choice of $h^{\left(n\right)}$
makes it continuous but not differentiable in time. We can make $h^{\left(n\right)}\left(t\right)$
of class $C^{\infty}$ in time if we, instead, take
\begin{equation}
\frac{\mathrm{d}h_{\varepsilon}^{\left(n\right)}}{\mathrm{d}t}\left(t\right)=\left\{ \begin{matrix}0 &  & t\le t_{n},\\
\frac{1}{\left(1-t_{n}\right)\zeta}\vartheta\left(\frac{t-t_{n}}{\varepsilon\left(1-t_{n}\right)\zeta}\right)\vartheta\left(-\frac{t-t_{n}-\left(1-t_{n}\right)\zeta}{\varepsilon\left(1-t_{n}\right)\zeta}\right) &  & t_{n}\le t\le t_{n}+\left(1-t_{n}\right)\zeta,\\
0 &  & t_{n}+\left(1-t_{n}\right)\zeta\le t\le t_{n+1}-\left(1-t_{n}\right)\zeta,\\
-\frac{1}{\left(1-t_{n}\right)\zeta}\vartheta\left(\frac{t-t_{n+1}+\left(1-t_{n}\right)\zeta}{\varepsilon\left(1-t_{n}\right)\zeta}\right)\vartheta\left(-\frac{t-t_{n+1}}{\varepsilon\left(1-t_{n}\right)\zeta}\right) &  & t_{n+1}-\left(1-t_{n}\right)\zeta\le t\le t_{n+1},\\
0 &  & t_{n+1}\le t\le1,
\end{matrix}\right.\label{eq:def dhndt second way}
\end{equation}
where $\vartheta:\mathbb{R}\to\left[0,1\right]$ is a $C^{\infty}$
transition function such that $\left.\vartheta\right|_{\left(-\infty,0\right]}\equiv0$
and $\left.\vartheta\right|_{\left[1,\infty\right)}\equiv1$ and $\varepsilon>0$
is a parameter that controls in how much time the sharp corners of
\eqref{eq:def hn first way} are smoothen out. One option is to choose
\[
\vartheta\left(x\right)=\frac{\mathrm{e}^{-\frac{1}{x}}\chi_{\left\{ x>0\right\} }\left(x\right)}{\mathrm{e}^{-\frac{1}{x}}\chi_{\left\{ x>0\right\} }\left(x\right)+\mathrm{e}^{-\frac{1}{1-x}}\chi_{\left\{ x<1\right\} }\left(x\right)}\quad\forall x\in\mathbb{R}.
\]
Integrating \eqref{eq:def dhndt second way} in time, we get that
$\forall t\in\left[t_{n}+\left(1-t_{n}\right)\zeta,t_{n+1}-\left(1-t_{n}\right)\zeta\right]$,
\begin{equation}
\begin{aligned}h_{\varepsilon}^{\left(n\right)}\left(t\right) & =h_{\varepsilon}\left(t_{n}+\left(1-t_{n}\right)\zeta\right)=\int_{t_{n}}^{t_{n}+\left(1-t_{n}\right)\zeta}\frac{1}{\left(1-t_{n}\right)\zeta}\overbrace{\vartheta\left(\frac{s-t_{n}}{\varepsilon\left(1-t_{n}\right)\zeta}\right)}^{\le1}\overbrace{\vartheta\left(-\frac{s-t_{n}-\left(1-t_{n}\right)\zeta}{\varepsilon\left(1-t_{n}\right)\zeta}\right)}^{\le1}\mathrm{d}s\le\\
 & \le\int_{t_{n}}^{t_{n}+\left(1-t_{n}\right)\zeta}\frac{1}{\left(1-t_{n}\right)\zeta}\mathrm{d}s=1,\\
h_{\varepsilon}^{\left(n\right)}\left(t\right) & =\int_{t_{n}}^{t_{n}+\left(1-t_{n}\right)\zeta}\overbrace{\frac{1}{\left(1-t_{n}\right)\zeta}\vartheta\left(\frac{s-t_{n}}{\varepsilon\left(1-t_{n}\right)\zeta}\right)\vartheta\left(-\frac{s-t_{n}-\left(1-t_{n}\right)\zeta}{\varepsilon\left(1-t_{n}\right)\zeta}\right)}^{\ge0}\mathrm{d}s\ge\\
 & \ge\int_{t_{n}+\left(1-t_{n}\right)\zeta\varepsilon}^{t_{n}+\left(1-t_{n}\right)\zeta-\left(1-t_{n}\right)\zeta\varepsilon}\frac{1}{\left(1-t_{n}\right)\zeta}\underbrace{\vartheta\left(\frac{s-t_{n}}{\varepsilon\left(1-t_{n}\right)\zeta}\right)}_{=1}\underbrace{\vartheta\left(-\frac{s-t_{n}-\left(1-t_{n}\right)\zeta}{\varepsilon\left(1-t_{n}\right)\zeta}\right)}_{=1}\mathrm{d}s=\\
 & \ge\int_{t_{n}+\left(1-t_{n}\right)\zeta\varepsilon}^{t_{n}+\left(1-t_{n}\right)\zeta-\left(1-t_{n}\right)\zeta\varepsilon}\frac{1}{\left(1-t_{n}\right)\zeta}\mathrm{d}s=1-2\varepsilon.
\end{aligned}
\label{eq:h_eps near one}
\end{equation}
Furthermore, by the symmetry behind definition \eqref{eq:def dhndt second way},
it is easy to see that
\begin{equation}
h_{\varepsilon}^{\left(n\right)}\left(t\right)=0\quad\forall t\in\left[t_{n+1},1\right].\label{eq:h_eps after tn+1}
\end{equation}

Thereby, it is indeed possible to reproduce the picture that appears
in figure \ref{fig:time picture}, both for $\overline{k}_{n}\left(t\right)$
and for $h^{\left(n\right)}\left(t\right)$. Next, we apply what we
have learned throughout this subsection to state a series of Choices
whose purpose is to select our parameters in such a way that the desired
construction is realized.
\begin{choice}
\label{choice:time picture} $\forall n\in\mathbb{N}$, we choose
$k_{n}\left(1\right)=\overline{k}_{n}\left(1\right)=0$ and $h^{\left(n\right)}\left(t\right)=h_{\varepsilon}^{\left(n\right)}\left(t\right)$
as given by equation \eqref{eq:def dhndt second way} with $0<\varepsilon<\frac{1}{4}$
and $0<\zeta<\frac{1}{4}$. The precise values of $\varepsilon$ and
$\zeta$ will be determined later.
\end{choice}
\begin{rem}
Notice that, thanks to equation \eqref{eq:h_eps after tn+1}, Choice
\ref{choice:time picture} is perfectly compatible with Choices \ref{choice:amplitude density}
and \ref{choice:no two simultanous densities}.
\end{rem}
\begin{rem}
\label{rem:effect of choice of kn(t) on an(t) and bn(t)}Assuming
that $k_{n}\left(t\right)\sim\overline{k}_{n}\left(t\right)$, if
we choose $\overline{k}_{n}\left(t\right)$ as portrayed by equation
\eqref{eq:kn second version approximate}, according to Choice \ref{choice:anbn},
$b_{n}\left(t\right)$ will be increasing in $t$ for $\hat{\hat{t}}\in\left[0,\frac{1}{2}\right]$
and decreasing in $t$ for $\hat{\hat{t}}\in\left[\frac{1}{2},1\right]$.
Undoing the change of variables \eqref{eq:change of variables JI0 Fn},
we obtain that $b_{n}\left(t\right)$ will be increasing in $t$ for
$t\in\left[t_{n},t_{n}+\frac{1-t_{n}}{2}\right]$ and decreasing in
$t$ for $t\in\left[t_{n}+\frac{1-t_{n}}{2},1\right]$. $a_{n}\left(t\right)$
has the exactly opposite behavior.
\end{rem}
\begin{choice}
\label{choice:ideal kn}We choose
\[
\max_{s\in\left[t_{n},1\right]}\overline{k}_{n}\left(s\right)=k_{\max}\quad\forall n\in\mathbb{N},
\]
where $0<k_{\max}<1$ and $\overline{k}_{n}\left(t\right)$ is given
by equation \eqref{eq:kn second version exact}.
\end{choice}
\begin{choice}
\label{choice:wnmax and Ji(tn)}We choose $\sin\left(a_{n-1}\left(1\right)\Xi^{\left(n\right)}\left(t_{n}\right)\right)$
so that
\[
M_{n-1}\left(1-t_{n}\right)=2\hat{\hat{t}}_{\max}^{\left(n\right)}=2\mathrm{arccosh}\left(\frac{1}{\sin\left(a_{n-1}\left(1\right)\Xi^{\left(n\right)}\left(t_{n}\right)\right)}\right),
\]
as dictated by equation \eqref{eq:periodicity of the evolution}.
Furthermore, in order to avoid the ambiguity of the arcsin function,
we select $\Xi^{\left(n\right)}\left(t_{n}\right)$ so that $0<a_{n-1}\left(1\right)\Xi^{\left(n\right)}\left(t_{n}\right)<\frac{\pi}{2}$.
\end{choice}

\subsection{\label{subsec:time convergence of bn(t)}Time convergence of $b_{n}\left(t\right)$}

After subsection \ref{subsec:completion of the toy model}, we have
already set the structure (although not the value) of most of the
parameters at our disposal. With ``structure'' we mean how a sequence
grows, what ODE does a time dependent parameter satisfy, etc. However,
there is one notorious exception: we have not chosen how $M_{n}$
(given in Choice \ref{choice:Bnanbn}) grows (we still have to decide
if we give it an exponential or superexponential growth) or, equivalently,
(thanks to equation \eqref{eq:time scale}), we have not chosen how
$1-t_{n}$ decreases. Nonetheless, we have seen in Summary \ref{sum:ideas time derivative density}
that an exponential growth might seem more advantageous. Unfortunately,
we will see in this section that exponential growth impedes being
able to consider the parameters of past layers as constants in time,
an assumption we have used multiple times throughout our construction.

By the arguments exposed in equation \eqref{eq:kn second version approximate},
assuming that $k_{n}\left(t\right)$ behaves like its ideal model
$\overline{k}_{n}\left(t\right)$, we expect to have
\[
k_{n}\left(t_{n}+\left(1-t_{n}\right)\hat{\hat{t}}\right)\xrightarrow[n\to\infty]{}k_{\max}\left(1-\left|1-2\hat{\hat{t}}\right|\right)\quad\forall\hat{\hat{t}}\in\left[0,1\right].
\]
As the only obstacle we found when taking the limit $n\to\infty$
to get from \eqref{eq:kn second version exact} to \eqref{eq:kn second version approximate}
was for $\hat{\hat{t}}=\frac{1}{2}$, we expect to have uniform convergence
far from $\hat{\hat{t}}=\frac{1}{2}$. This means that we expect
\[
k_{n}\left(t\right)\overset{n\to\infty}{\sim}k_{\max}\left(1-\left|1-2\frac{t-t_{n}}{1-t_{n}}\right|\right)
\]
uniformly in time near $t=1$. Near $t=1$, the contents of the absolute
value are always negative. In this way, we can write
\[
k_{n}\left(t\right)\overset{n\to\infty}{\sim}k_{\max}\left(1+1-2\frac{t-t_{n}}{1-t_{n}}\right)=2k_{\max}\left(1-\frac{t-t_{n}}{1-t_{n}}\right)=2k_{\max}\frac{1-t}{1-t_{n}}.
\]
We want $k_{n}\left(t\right)$ to be very close to zero when we start
the next layer, i.e., we want $k_{n}\left(t\right)\approx0$ $\forall t\in\left[t_{n+1},t\right]$.
Notice that we can bound
\[
k_{n}\left(t\right)=2k_{\max}\frac{1-t}{1-t_{n}}\le2k_{\max}\frac{1-t_{n+1}}{1-t_{n}}\quad\forall t\in\left[t_{n+1},1\right].
\]
This means that we should study
\begin{equation}
k_{n}\left(t_{n+1}\right)=2k_{\max}\frac{1-t_{n+1}}{1-t_{n}}.\label{eq:kn(tn+1)}
\end{equation}
If $1-t_{n}$ decreases exponentially, say, $1-t_{n}=T^{n}$ with
$T<1$, we have
\begin{equation}
k_{n}\left(t_{n+1}\right)=2k_{\max}\frac{T^{n+1}}{T^{n}}=2k_{\max}T,\label{eq:bound kn tn+1 exponential}
\end{equation}
which is a bound for $k_{n}\left(t_{n+1}\right)$ that is uniform
in $n\in\mathbb{N}$. As we will see, this is a problem. We wish for
$\left|\frac{b_{n}\left(t\right)}{b_{n}\left(1\right)}-1\right|$
to be small $\forall t\in\left[t_{n+1},1\right]$. Using Choices \ref{choice:anbn}
and \ref{choice:time picture}, we obtain
\begin{equation}
\frac{b_{n}\left(t\right)}{b_{n}\left(1\right)}=\frac{C^{\left(1+k_{n}\left(t\right)\right)\left(\frac{1}{1-\gamma}\right)^{n}}}{C^{\left(1+k_{n}\left(1\right)\right)\left(\frac{1}{1-\gamma}\right)^{n}}}=C^{k_{n}\left(t\right)\left(\frac{1}{1-\gamma}\right)^{n}}.\label{eq:quotient bn}
\end{equation}
Even if we know that $k_{n}\left(t\right)\le2k_{\max}T$ (see equation
\eqref{eq:bound kn tn+1 exponential}), the bound
\[
\frac{b_{n}\left(t\right)}{b_{n}\left(1\right)}\le C^{2k_{\max}T\left(\frac{1}{1-\gamma}\right)^{n}}
\]
is of no use for us because it is not uniform in $n\in\mathbb{N}$.
In other words, with this choice of $1-t_{n}$ we cannot make $\left|\frac{b_{n}\left(t\right)}{b_{n}\left(1\right)}-1\right|$
small $\forall n\in\mathbb{N}$.

On the other hand, if we choose $1-t_{n}=C^{-\delta\left(\frac{1}{1-\gamma}\right)^{n}}$
(with $\delta>0$ small)\footnote{We will actually choose $1-t_{n}$ a little differently, but, for
the purposes of this argument, this expression suffices to convey
what we wish to explain.}, coming back to equation \eqref{eq:kn(tn+1)}, we get to
\[
\begin{aligned}k_{n}\left(t_{n+1}\right) & =2k_{\max}\frac{C^{-\delta\left(\frac{1}{1-\gamma}\right)^{n+1}}}{C^{-\delta\left(\frac{1}{1-\gamma}\right)^{n}}}=2k_{\max}C^{-\delta\left(\frac{1}{1-\gamma}\right)^{n+1}+\delta\left(\frac{1}{1-\gamma}\right)^{n}}=\\
 & =2k_{\max}C^{-\delta\left(\frac{1}{1-\gamma}\right)^{n+1}\left[1-\left(1-\gamma\right)\right]}=2k_{\max}C^{-\delta\gamma\left(\frac{1}{1-\gamma}\right)^{n+1}}.
\end{aligned}
\]
In this case, the converge of $k_{n}\left(t_{n+1}\right)$ improves
superexponentially with $n\in\mathbb{N}$. Substituting in equation
\eqref{eq:quotient bn} leads to
\[
\frac{b_{n}\left(t\right)}{b_{n}\left(1\right)}\le\exp\left(\ln\left(C\right)2k_{\max}C^{-\delta\gamma\left(\frac{1}{1-\gamma}\right)^{n+1}}\left(\frac{1}{1-\gamma}\right)^{n}\right)\quad\forall t\in\left[t_{n+1},1\right].
\]
Since the superexponential decrease compensates the growth of $\left(\frac{1}{1-\gamma}\right)^{n}$,
we obtain that the approximation $b_{n}\left(t\right)\sim b_{n}\left(1\right)$
actually improves with $n\in\mathbb{N}$.

\begin{choice}
\label{choice:Mn}Taking all the above into account, we shall take
\[
M_{n}\coloneqq YC^{\delta\left(\frac{1}{1-\gamma}\right)^{n}}\quad\forall n\in\mathbb{N},
\]
where $\delta>0$ is ``small'' and $Y>0$. Under this choice (see
equation \eqref{eq:time scale}), the time scale $1-t_{n}$ becomes
\[
1-t_{n}=\frac{1}{Y}C^{-\delta\left(\frac{1}{1-\gamma}\right)^{n-1}}\mathrm{arccosh}\left(C^{k_{\max}\left(\frac{1}{1-\gamma}\right)^{n}}\right),
\]
where $Y$ is chosen so that $t_{1}=0$. Note: this $Y$ depends on
$C$ and $\delta$.
\end{choice}
\begin{rem}
\label{rem:initial condition}Under Choices \ref{choice:ideal kn},
\ref{choice:wnmax and Ji(tn)} and \ref{choice:Mn}, equation \eqref{eq:time scale}
is satisfied and becomes
\[
\sin\left(a_{n-1}\left(1\right)\Xi^{\left(n\right)}\left(t_{n}\right)\right)=C^{-k_{\max}\left(\frac{1}{1-\gamma}\right)^{n}}.
\]
Moreover, as $0<a_{n-1}\left(1\right)\Xi^{\left(n\right)}\left(t_{n}\right)<\frac{\pi}{2}$
by Choice \ref{choice:wnmax and Ji(tn)}, we may assure that
\[
a_{n-1}\left(1\right)\Xi^{\left(n\right)}\left(t_{n}\right)=\arcsin\left(C^{-k_{\max}\left(\frac{1}{1-\gamma}\right)^{n}}\right).
\]
\end{rem}

\subsection{Summary of the construction}

Our solution is given by
\[
\begin{aligned}\psi\left(t,x\right) & =\sum_{n=1}^{\infty}\psi^{\left(n\right)}\left(t,x\right),\\
\widetilde{\psi^{\left(n\right)}}^{n}\left(t,x\right) & =B_{n}\left(t\right)\varphi\left(\lambda_{n}x_{1}\right)\varphi\left(\lambda_{n}x_{2}\right)\sin\left(x_{1}\right)\sin\left(x_{2}\right),\\
\rho\left(t,x\right) & =\sum_{n=1}^{\infty}\rho^{\left(n\right)}\left(t,x\right),\\
\widetilde{\rho^{\left(n\right)}}^{n}\left(t,x\right) & =-\frac{1}{b_{n}\left(t\right)}\frac{\mathrm{d}}{\mathrm{d}t}\left[B_{n}\left(t\right)\left(a_{n}\left(t\right)^{2}+b_{n}\left(t\right)^{2}\right)\right]\varphi\left(\lambda_{n}x_{1}\right)\varphi\left(\lambda_{n}x_{2}\right)\sin\left(x_{1}\right)\cos\left(x_{2}\right),
\end{aligned}
\]
where
\begin{enumerate}
\item $\varphi\in C_{c}^{\infty}\left(\mathbb{R}\right)$ is any compactly
supported even smooth function that satisfies $\varphi\equiv1$ in
the interval $\left[-8\pi,8\pi\right]$ and $\varphi\equiv0$ in $\mathbb{R}\setminus\left[-16\pi,16\pi\right]$.
\item $\lambda_{n}=C^{-\Lambda\left(\frac{1}{1-\gamma}\right)^{n}}$ as
given by Choice \ref{choice:lambdan}. We still do not know the value
of $\Lambda$.
\item $\widetilde{\psi^{\left(n\right)}}^{n}\left(t,x\right)=\psi^{\left(n\right)}\left(t,\phi^{\left(n\right)}\left(t,x\right)\right)$,
where $\phi^{\left(n\right)}\left(t,x\right)$ is the change of variables
given in Choice \ref{choice:phin}:
\[
\phi^{\left(n\right)}\left(t,x\right)=\phi^{\left(n\right)}\left(t,0\right)+\left(\frac{1}{a_{n}\left(t\right)}x_{1},\frac{1}{b_{n}\left(t\right)}x_{2}\right).
\]
\item The temporal dynamics of $a_{n}\left(t\right)$, $b_{n}\left(t\right)$,
$B_{n}\left(t\right)$ and $\phi^{\left(n\right)}\left(t,0\right)$
are given by Choice \ref{choice:anbncn} and Proposition \ref{prop:relation between anbn and jacobian},
i.e., $\forall t\in\left[t_{n},1\right]$,
\[
\begin{alignedat}{1}\phi_{2}^{\left(n\right)}\left(t,0\right) & =0,\\
\frac{\partial\phi_{1}^{\left(n\right)}}{\partial t}\left(t,0\right) & =\sum_{m=1}^{n-1}B_{m}\left(t\right)b_{m}\left(t\right)\sin\left(a_{m}\left(t\right)\left(\phi_{1}^{\left(n\right)}\left(t,0\right)-\phi_{1}^{\left(m\right)}\left(t,0\right)\right)\right),\\
\frac{\mathrm{d}}{\mathrm{d}t}\left(\ln\left(b_{n}\left(t\right)\right)\right) & =\sum_{m=1}^{n-1}B_{m}\left(t\right)a_{m}\left(t\right)b_{m}\left(t\right)\cos\left(a_{m}\left(t\right)\left(\phi_{1}^{\left(n\right)}\left(t,0\right)-\phi_{1}^{\left(m\right)}\left(t,0\right)\right)\right),\\
\frac{\mathrm{d}}{\mathrm{d}t}\left(a_{n}\left(t\right)b_{n}\left(t\right)\right) & =0.
\end{alignedat}
\]
$\phi_{1}^{\left(n\right)}\left(t,0\right)$ evolves in time like
a weighted sum of inverted degenerate half-pendula (see equations
\eqref{eq:phi_1 weighted sum of half-pendula}, \eqref{eq:ODE JI0},
\eqref{eq:a lot of sines}, and Choices \ref{choice:wnmax and Ji(tn)},
\ref{choice:Mn}):
\[
\phi_{1}^{\left(n\right)}\left(t,0\right)\sim\sum_{m=1}^{n}\Xi_{0}^{\left(m\right)}\left(t\right)\sim\sum_{m=1}^{n}\overbrace{\frac{1}{a_{m}\left(1\right)}}^{\text{weight}}\overbrace{\left(a_{m}\left(1\right)\Xi_{0}^{\left(m\right)}\left(t\right)\right)}^{{\footnotesize \begin{matrix}\text{inverted degenerate}\\
\text{half-pendulum}
\end{matrix}}},
\]
where
\[
\begin{aligned}\frac{\mathrm{d}\Xi_{0}^{\left(n\right)}}{\mathrm{d}t}\left(t\right) & =B_{n-1}\left(1\right)b_{n-1}\left(1\right)\sin\left(a_{n-1}\left(1\right)\Xi_{0}^{\left(n\right)}\left(t\right)\right)\quad\forall t\in\left[t_{n},1\right],\\
\sin\left(a_{n-1}\left(1\right)\Xi_{0}^{\left(n\right)}\left(t_{n}+\left(1-t_{n}\right)\hat{\hat{t}}\right)\right) & =\frac{1}{\cosh\left(\mathrm{arccosh}\left(C^{k_{\max}\left(\frac{1}{1-\gamma}\right)^{n}}\right)\left(1-2\hat{\hat{t}}\right)\right)}\quad\forall\hat{\hat{t}}\in\left[0,1\right]
\end{aligned}
\]
and $\sin\left(a_{n-1}\left(1\right)\Xi_{0}^{\left(n\right)}\left(t_{n}+\left(1-t_{n}\right)\hat{\hat{t}}\right)\right)$
behaves like shown in figure \ref{fig:graph ideal model}.
\item The time scales $\left(t_{n}\right)_{n\in\mathbb{N}}$ are given by
Choices \ref{choice:time (initial)} and \ref{choice:Mn}:
\[
1-t_{n}=\frac{1}{Y}C^{-\delta\left(\frac{1}{1-\gamma}\right)^{n-1}}\mathrm{arccosh}\left(C^{k_{\max}\left(\frac{1}{1-\gamma}\right)^{n}}\right),
\]
where we have not fixed the value of $Y$, $\delta$ and $k_{\max}$
yet. Our construction corresponds to the blue line of figure \ref{fig:relation delta M}.
\item The initial conditions for $a_{n}\left(t\right)$ and $b_{n}\left(t\right)$
are specified through Choices \ref{choice:anbn} and \ref{choice:time picture}:
\[
b_{n}\left(t\right)=C^{\left(1+k_{n}\left(t\right)\right)\left(\frac{1}{1-\gamma}\right)^{n}},\quad a_{n}\left(t\right)=C^{\left(1-k_{n}\left(t\right)\right)\left(\frac{1}{1-\gamma}\right)^{n}},\quad k_{n}\left(1\right)=0,
\]
were $C$ is a parameter we still have not assigned a value to. By
equation \eqref{eq:kn second version approximate},
\[
k_{n}\left(t_{n}+\left(1-t_{n}\right)\hat{\hat{t}}\right)\overset{n\to\infty}{\sim}k_{\max}\left(1-\left|1-2\hat{\hat{t}}\right|\right)\quad\forall\hat{\hat{t}}\in\left[0,1\right],
\]
which is displayed graphically in figure \ref{fig:limit profile kn}.
\item The initial conditions $\phi^{\left(n\right)}\left(t_{n},0\right)=c^{\left(n\right)}$
are given indirectly by Proposition \ref{prop:relation between anbn and jacobian}
and Remark \ref{rem:initial condition}:
\[
\begin{aligned}a_{n-1}\left(1\right)\Xi^{\left(n\right)}\left(t_{n}\right) & =\arcsin\left(C^{-k_{\max}\left(\frac{1}{1-\gamma}\right)^{n}}\right),\\
\Xi^{\left(n\right)}\left(t\right) & =\phi_{1}^{\left(n\right)}\left(t,0\right)-\phi_{1}^{\left(n-1\right)}\left(t,0\right),\\
\phi_{2}^{\left(n\right)}\left(t,0\right) & \equiv0,
\end{aligned}
\]
where we take $\phi^{\left(0\right)}\left(t,0\right)\equiv0$.
\item $B_{n}\left(t\right)$ is given by Choice \ref{choice:amplitude density},
equation \eqref{eq:relation Mn and zn} and Choice \ref{choice:time picture}:
\[
B_{n}\left(t\right)=\frac{2M_{n}}{a_{n}\left(t\right)^{2}+b_{n}\left(t\right)^{2}}\frac{\int_{t_{n}}^{t}h^{\left(n\right)}\left(s\right)b_{n}\left(s\right)\mathrm{d}s}{\int_{t_{n}}^{1}h^{\left(n\right)}\left(s\right)b_{n}\left(s\right)\mathrm{d}s}.
\]
Here, $h^{\left(n\right)}\left(s\right)$ is a function given by Choice
\ref{choice:time picture} and equation \eqref{eq:def dhndt second way}
(see figure \ref{fig:time picture} for a schematic representation).
We still have to choose values for $\varepsilon$ and $\zeta$.
\item $M_{n}$ is given by Choice \ref{choice:Mn}:
\[
M_{n}\coloneqq YC^{\delta\left(\frac{1}{1-\gamma}\right)^{n}}.
\]
\end{enumerate}
Notice that, actually, because of how $B_{n}\left(t\right)$ and $t_{n}$
have been defined, for every $t\in\left[0,1\right)$, only a finite
number of summands in $\psi\left(t,\infty\right)=\sum_{n=1}^{\infty}\psi^{\left(n\right)}\left(t,x\right)$
are non zero. Thereby, for fixed $t<1$, this is always a finite sum,
i.e., the intrinsic limit associated to the series is fictitious.
In the case of the density $\rho\left(t,x\right)=\sum_{n=1}^{\infty}\rho^{\left(n\right)}\left(t,x\right)$,
because of Choice \ref{choice:no two simultanous densities}, only
one summand is non-zero, i.e., not only is the sum finite, it actually
encompasses one single element.

\section{\label{sec:time evolution}Time evolution}

The main objective of this section is to prove that the toy models
introduced for $\Xi^{\left(n\right)}\left(t\right)$ and $k_{n}\left(t\right)$
are close to their real counterparts, i.e., we will prove that $\Xi^{\left(n\right)}\left(t\right)$
converges to $\Xi_{0}^{\left(n\right)}\left(t\right)$ and that $k_{n}\left(t\right)$
converges to $\overline{k}_{n}\left(t\right)$. Moreover, we will
see that the assumption that appears in Proposition \ref{prop:relation between anbn and jacobian}
is superfluous.

In multiple results of this section and the ones that follow, we will
need some assumptions on the value of the parameters $C$, $\gamma$
and $k_{\max}$. In order not to repeat them in every proposition,
we state those requirements now in a Choice.
\begin{choice}
\label{choice:min requirements on C gamma and kmax}We will choose
$C$, $\gamma$ and $k_{\max}$ so that $C>2$, $\gamma\in\left(\frac{1}{2},1\right)$
and $k_{\max}\ge\frac{1}{100}$.
\end{choice}

\subsection{\label{subsec:time convergence}Time convergence of the parameters}

Before undertaking the main goal of this subsection, we will prove
rigorously that, thanks to Choice \ref{choice:Mn}, we can treat the
parameters $a_{n}\left(t\right)$, $b_{n}\left(t\right)$ and $B_{n}\left(t\right)$
of past layers as constant in time. To do this, we will need three
Lemmas. Lemma \ref{lem:estimate sum superexponential} basically tells
us that the sum of a superexponential sequence has the same order
of its last term, whereas Lemma \ref{lem:arccosh by ln} states that
$\mathrm{arccosh}\left(x\right)\sim\ln\left(2x\right)$ for big $x$.
On the other hand, Lemma \ref{lem:exponetial superexponential bound}
says that the product of an exponential increasing sequence times
a superexponential decreasing sequence is bounded by the superexponential
sequence with a slight change in the exponent.
\begin{lem}
\label{lem:estimate sum superexponential}Let $\gamma\in\left(\frac{1}{2},1\right)$,
$\delta>0$, $C>2$ and $n\in\mathbb{N}$. We have
\[
\sum_{m=1}^{n-1}C^{\delta\left(\frac{1}{1-\gamma}\right)^{m}}\lesssim_{\delta}C^{\delta\left(\frac{1}{1-\gamma}\right)^{n-1}}.
\]
\end{lem}
\begin{proof}
Notice that we can write
\begin{equation}
\begin{aligned}\sum_{m=1}^{n-1}C^{\delta\left(\frac{1}{1-\gamma}\right)^{m}} & =C^{\delta\left(\frac{1}{1-\gamma}\right)^{n-1}}\sum_{m=1}^{n-1}C^{\left[\delta\left(\frac{1}{1-\gamma}\right)^{m}-\delta\left(\frac{1}{1-\gamma}\right)^{n-1}\right]}=\\
 & =C^{\delta\left(\frac{1}{1-\gamma}\right)^{n-1}}\sum_{m=1}^{n-1}C^{-\delta\left(\frac{1}{1-\gamma}\right)^{n-1}\left[1-\left(1-\gamma\right)^{n-1-m}\right]}=\\
 & =C^{\delta\left(\frac{1}{1-\gamma}\right)^{n-1}}\left(1+\sum_{m=1}^{n-2}C^{-\delta\left(\frac{1}{1-\gamma}\right)^{n-1}\left[1-\left(1-\gamma\right)^{n-1-m}\right]}\right).
\end{aligned}
\label{eq:sum decomposed}
\end{equation}
As $\gamma\in\left(0,1\right)$, we have
\[
\left(1-\gamma\right)^{n-1-m}\le1-\gamma\quad\forall m\in\left\{ 1,\dots,n-2\right\} .
\]
Consequently, we can bound
\[
1-\left(1-\gamma\right)^{n-1-m}\ge1-\left(1-\gamma\right)=\gamma
\]
and, in this way,
\[
\sum_{m=1}^{n-2}C^{-\delta\left(\frac{1}{1-\gamma}\right)^{n-1}\left[1-\left(1-\gamma\right)^{n-1-m}\right]}\le\sum_{m=1}^{n-2}C^{-\gamma\delta\left(\frac{1}{1-\gamma}\right)^{n-1}}\le\left(n-2\right)C^{-\gamma\delta\left(\frac{1}{1-\gamma}\right)^{n-1}}.
\]
As $\gamma\ge\frac{1}{2}$ and $C>2$, we deduce that $\frac{1}{1-\gamma}\ge2$
and, as a consequence, we can write
\[
\sum_{m=1}^{n-2}C^{-\delta\left(\frac{1}{1-\gamma}\right)^{n-1}\left[1-\left(1-\gamma\right)^{n-1-m}\right]}\le\left(n-2\right)2^{-\frac{\delta}{2}2^{n-1}}=\left(n-2\right)2^{-\delta2^{n-2}}.
\]
Notice that this last expression is bounded in $n\in\mathbb{N}$,
since
\[
\lim_{n\to\infty}\left(n-2\right)2^{-\delta2^{n-2}}=0.
\]
Therefore, there is a constant $c\left(\delta\right)$ such that
\[
\sum_{m=1}^{n-2}C^{-\delta\left(\frac{1}{1-\gamma}\right)^{n-1}\left[1-\left(1-\gamma\right)^{n-1-m}\right]}\le c\left(\delta\right)\quad\forall n\in\mathbb{N}.
\]
Turning back to equation \eqref{eq:sum decomposed}, we conclude that
\[
\sum_{m=1}^{n-1}C^{\delta\left(\frac{1}{1-\gamma}\right)^{m}}\le C^{\delta\left(\frac{1}{1-\gamma}\right)^{n-1}}\left(1+c\left(\delta\right)\right),
\]
which proves the result.
\end{proof}
\begin{lem}
\label{lem:arccosh by ln}We have
\[
\mathrm{arccosh}\left(x\right)\le\ln\left(2x\right)\quad\forall x\in\left[1,\infty\right).
\]
\end{lem}
\begin{proof}
Recalling that
\[
\mathrm{arccosh}\left(x\right)=\ln\left(x+\sqrt{x^{2}-1}\right),
\]
we consider the function
\[
F\left(x\right)\coloneqq\frac{\mathrm{arccosh}\left(x\right)}{\ln\left(2x\right)}=\frac{\ln\left(x+\sqrt{x^{2}-1}\right)}{\ln\left(2x\right)}.
\]
Differentiating, we get
\[
\begin{aligned}F'\left(x\right) & =\frac{\frac{1+\frac{2x}{2\,\sqrt{x^{2}-1}}}{x+\sqrt{x^{2}-1}}\ln\left(2x\right)-\ln\left(x+\sqrt{x^{2}-1}\right)\frac{2}{2x}}{\left(\ln\left(2x\right)\right)^{2}}=\\
 & =\frac{\left(1+\frac{2x}{2\,\sqrt{x^{2}-1}}\right)2x\ln\left(2x\right)-2\ln\left(x+\sqrt{x^{2}-1}\right)\left(x+\sqrt{x^{2}-1}\right)}{\left(\ln\left(2x\right)\right)^{2}\left(x+\sqrt{x^{2}-1}\right)2x}.
\end{aligned}
\]
Since $\forall x\in\left[1,\infty\right)$,
\[
x+\sqrt{x^{2}-1}\le x+\sqrt{x^{2}}=2x
\]
and the function $z\to z\ln z$ is increasing $\forall x\in\left[1,\infty\right)$,
we obtain that
\[
2x\ln\left(2x\right)\ge2\ln\left(x+\sqrt{x^{2}-1}\right)\left(x+\sqrt{x^{2}-1}\right).
\]
Consequently,
\[
F'\left(x\right)\ge\frac{\frac{2x}{2\,\sqrt{x^{2}-1}}2x\ln\left(2x\right)}{\left(\ln\left(2x\right)\right)^{2}\left(x+\sqrt{x^{2}-1}\right)2x}=\frac{x}{\sqrt{x^{2}-1}\left(x+\sqrt{x^{2}-1}\right)\ln\left(2x\right)}>0\quad\forall x\in\left(1,\infty\right).
\]
Hence, $F$ is an increasing function. Moreover, it is easy to see
that $\lim_{x\to\infty}F\left(x\right)=1$. Thus, we conclude that
$F\left(x\right)\le1$ $\forall x\in\left[1,\infty\right)$, which
proves the result.
\end{proof}
\begin{lem}
\label{lem:exponetial superexponential bound}Let $n\in\mathbb{N}$,
$a\ge1$, $b>0$ and $\beta\in\left(0,1\right)$. Then,
\[
a^{n+1}C^{-ba^{n}}\le\frac{4\mathrm{e}^{-2}}{\left(1-\beta\right)^{2}b^{2}\ln\left(C\right)^{2}}C^{-\beta ba^{n}}.
\]
\end{lem}
\begin{proof}
Notice that we can write
\[
a^{n+1}C^{-ba^{n}}=a^{n+1}C^{-\left(1-\beta\right)ba^{n}}C^{-\beta ba^{n}}.
\]
Let us call $z=a^{n}$. Then, as $n\ge1$ and $a\ge1$, we deduce
that $z=a^{n}\ge1$ and, in this way, $z^{\frac{n+1}{n}}\le z^{2}$.
Hence, we have
\begin{equation}
a^{n+1}C^{-ba^{n}}=z^{\frac{n+1}{n}}C^{-\left(1-\beta\right)bz}C^{-\beta ba^{n}}\le\underbrace{z^{2}C^{-\left(1-\beta\right)bz}}_{\eqqcolon F\left(z\right)}C^{-\beta ba^{n}}.\label{eq:first split}
\end{equation}
Clearly, $F\ge0$, $F\left(0\right)=0$ and $\lim_{z\to\infty}F\left(z\right)=0$.
Furthermore,
\[
F'\left(z\right)=2zC^{-\left(1-\beta\right)bz}+z^{2}C^{-\left(1-\beta\right)bz}\left(-\left(1-\beta\right)b\ln\left(C\right)\right)
\]
and the equation $F'\left(z\right)=0$ admits the only solution
\[
z=\frac{2}{\left(1-\beta\right)b\ln\left(C\right)}.
\]
This means that $\frac{2}{\left(1-\beta\right)b\ln\left(C\right)}$
must be the point where the absolute maximum of $F$ is attained and,
consequently,
\[
F\left(z\right)\le F\left(\frac{2}{\left(1-\beta\right)b\ln\left(C\right)}\right)=\frac{4}{\left(1-\beta\right)^{2}b^{2}\ln\left(C\right)^{2}}\mathrm{e}^{-2}\quad\forall z\in\left(0,\infty\right).
\]
Returning back to equation \eqref{eq:first split}, we arrive to
\[
a^{n+1}C^{-ba^{n}}\le\frac{4\mathrm{e}^{-2}}{\left(1-\beta\right)^{2}b^{2}\ln\left(C\right)^{2}}C^{-\beta ba^{n}}.
\]
\end{proof}
\begin{prop}
\label{prop:time convergence}Let $\beta\in\left(0,1\right)$. Provided
that $C$ is sufficiently big (let us say $C\ge\Upsilon\left(\beta,\delta\right)$),
then, $\forall n\in\mathbb{N}$ and $\forall t\in\left[t_{n+1},1\right]$,
\[
\left|\frac{b_{n}\left(t\right)}{b_{n}\left(1\right)}-1\right|\leq C^{-\beta\delta\gamma\left(\frac{1}{1-\gamma}\right)^{n}},\quad\left|\frac{a_{n}\left(t\right)}{a_{n}\left(1\right)}-1\right|\leq C^{-\beta\delta\gamma\left(\frac{1}{1-\gamma}\right)^{n}},\quad\left|\frac{B_{n}\left(t\right)}{B_{n}\left(1\right)}-1\right|\leq C^{-\beta\delta\gamma\left(\frac{1}{1-\gamma}\right)^{n}}.
\]
\end{prop}
\begin{proof}
We shall make use of an induction argument. For the case $n=1$, as
we have explained at the beginning of subsection \ref{subsec:peak into convergence of Bnt},
$a_{1}\left(t\right)$, $b_{1}\left(t\right)$ and $B_{1}\left(t\right)$
are all constant in the interval $\left[t_{2},1\right]$ (the latter
thanks to Choice \ref{choice:no two simultanous densities}). Therefore,
the statement is true for $n=1$.

Now, assume the statement is true for $n-1$ and lower and we shall
prove it for $n$. By Choice \ref{choice:anbncn}, we have
\[
\frac{\mathrm{d}}{\mathrm{d}t}\left(\ln\left(b_{n}\left(t\right)\right)\right)=\sum_{m=1}^{n-1}B_{m}\left(t\right)a_{m}\left(t\right)b_{m}\left(t\right)\cos\left(a_{m}\left(t\right)\left(\phi_{1}^{\left(n\right)}\left(t,0\right)-\phi_{1}^{\left(m\right)}\left(t,0\right)\right)\right).
\]
Integrating the equation above in the interval $\left[t,1\right]$
with $t\in\left[t_{n+1},1\right]$ provides
\[
\ln\left(\frac{b_{n}\left(1\right)}{b_{n}\left(t\right)}\right)=\sum_{m=1}^{n-1}\int_{t}^{1}B_{m}\left(s\right)a_{m}\left(s\right)b_{m}\left(s\right)\cos\left(a_{m}\left(s\right)\left(\phi_{1}^{\left(n\right)}\left(s,0\right)-\phi_{1}^{\left(m\right)}\left(s,0\right)\right)\right)\mathrm{d}s.
\]
Multiplying by $-1$ at both sides and bounding $\left|\cos\left(\cdot\right)\right|\le1$,
we arrive to
\begin{equation}
\left|\ln\left(\frac{b_{n}\left(t\right)}{b_{n}\left(1\right)}\right)\right|\le\sum_{m=1}^{n-1}\int_{t}^{1}B_{m}\left(s\right)a_{m}\left(s\right)b_{m}\left(s\right)\mathrm{d}s.\label{eq:ln(bs) first bound}
\end{equation}
Now, by the induction hypothesis, which we can apply as $t_{n+1}\ge t_{m}$
$\forall m\in\left\{ 1,\dots,n-1\right\} $ by Choice \ref{choice:time (initial)},
we know that, $\forall t\in\left[t_{n+1},1\right]$,
\[
\begin{aligned} & \left|\frac{b_{m}\left(t\right)}{b_{m}\left(1\right)}-1\right|\leq C^{-\beta\delta\gamma\left(\frac{1}{1-\gamma}\right)^{m}}\iff-C^{-\beta\delta\gamma\left(\frac{1}{1-\gamma}\right)^{m}}\leq\frac{b_{m}\left(t\right)}{b_{m}\left(1\right)}-1\leq C^{-\beta\delta\gamma\left(\frac{1}{1-\gamma}\right)^{m}}\iff\\
\iff & 1-C^{-\beta\delta\gamma\left(\frac{1}{1-\gamma}\right)^{m}}\le\frac{b_{m}\left(t\right)}{b_{m}\left(1\right)}\le1+C^{-\beta\delta\gamma\left(\frac{1}{1-\gamma}\right)^{m}}.
\end{aligned}
\]
In particular, we have
\[
b_{m}\left(t\right)\le\left(1+C^{-\beta\delta\gamma\left(\frac{1}{1-\gamma}\right)^{m}}\right)b_{m}\left(1\right)\le2b_{m}\left(1\right).
\]
Clearly, the same argument can be used for $a_{m}\left(t\right)$
and $B_{m}\left(t\right)$. Consequently, we may transform \eqref{eq:ln(bs) first bound}
into
\[
\begin{aligned}\left|\ln\left(\frac{b_{n}\left(t\right)}{b_{n}\left(1\right)}\right)\right| & \lesssim\sum_{m=1}^{n-1}B_{m}\left(1\right)a_{m}\left(1\right)b_{m}\left(1\right)\int_{t}^{1}\mathrm{d}s=\\
 & \lesssim\left(1-t\right)\sum_{m=1}^{n-1}B_{m}\left(1\right)a_{m}\left(1\right)b_{m}\left(1\right).
\end{aligned}
\]
Since we are considering times $t\in\left[t_{n+1},1\right]$, we have
\[
\left|\ln\left(\frac{b_{n}\left(t\right)}{b_{n}\left(1\right)}\right)\right|\lesssim\left(1-t_{n+1}\right)\sum_{m=1}^{n-1}B_{m}\left(1\right)a_{m}\left(1\right)b_{m}\left(1\right).
\]
Thereby, recalling Choices \ref{choice:Bnanbn} and \ref{choice:Mn},
we can write
\[
\left|\ln\left(\frac{b_{n}\left(t\right)}{b_{n}\left(1\right)}\right)\right|\lesssim\frac{1}{Y}C^{-\delta\left(\frac{1}{1-\gamma}\right)^{n}}\mathrm{arccosh}\left(C^{k_{\max}\left(\frac{1}{1-\gamma}\right)^{n+1}}\right)\sum_{m=1}^{n-1}YC^{\delta\left(\frac{1}{1-\gamma}\right)^{m}}.
\]
Recurring to Lemmas \ref{lem:estimate sum superexponential} and \ref{lem:arccosh by ln},
we find that
\begin{equation}
\begin{aligned}\left|\ln\left(\frac{b_{n}\left(t\right)}{b_{n}\left(1\right)}\right)\right| & \lesssim_{\delta}C^{-\delta\left(\frac{1}{1-\gamma}\right)^{n}}\ln\left(2C^{k_{\max}\left(\frac{1}{1-\gamma}\right)^{n+1}}\right)C^{\delta\left(\frac{1}{1-\gamma}\right)^{n-1}}=\\
 & \lesssim_{\delta}C^{-\delta\left(\frac{1}{1-\gamma}\right)^{n}}\left(\ln\left(2\right)+k_{\max}\left(\frac{1}{1-\gamma}\right)^{n+1}\ln\left(C\right)\right)C^{\delta\left(\frac{1}{1-\gamma}\right)^{n-1}}.
\end{aligned}
\label{eq:time convergence the log 1}
\end{equation}
Since $C>2$, $\gamma\ge\frac{1}{2}$, $k_{\max}\ge\frac{1}{100}$
by Choice \ref{choice:min requirements on C gamma and kmax} and $n\ge1$,
we have
\[
k_{\max}\left(\frac{1}{1-\gamma}\right)^{n+1}\ln\left(C\right)\ge\frac{1}{100}2^{2}\ln\left(2\right).
\]
This means that we can bound
\[
\ln\left(2\right)+k_{\max}\left(\frac{1}{1-\gamma}\right)^{n+1}\ln\left(C\right)\lesssim k_{\max}\left(\frac{1}{1-\gamma}\right)^{n+1}\ln\left(C\right).
\]
Consequently, equation \eqref{eq:time convergence the log 1} becomes
\[
\left|\ln\left(\frac{b_{n}\left(t\right)}{b_{n}\left(1\right)}\right)\right|\lesssim_{\delta}k_{\max}\ln\left(C\right)\left(\frac{1}{1-\gamma}\right)^{n+1}C^{-\delta\left(\frac{1}{1-\gamma}\right)^{n}\left[1-\left(1-\gamma\right)\right]}.
\]
As $k_{\max}\le1$ by Choice \ref{choice:ideal kn}, we deduce that
\begin{equation}
\begin{aligned}\left|\ln\left(\frac{b_{n}\left(t\right)}{b_{n}\left(1\right)}\right)\right| & \lesssim_{\delta}\ln\left(C\right)\left(\frac{1}{1-\gamma}\right)^{n+1}C^{-\delta\gamma\left(\frac{1}{1-\gamma}\right)^{n}}.\end{aligned}
\label{eq:time convergence the log 2}
\end{equation}
The next step is to show that, taking $C$ sufficiently large, we
can make the expression above as small as we want. By Lemma \ref{lem:exponetial superexponential bound},
taking $a\leftarrow\frac{1}{1-\gamma}\ge\frac{1}{1-\frac{1}{2}}\ge2$
(by Choice \ref{choice:min requirements on C gamma and kmax}) and
$b\leftarrow\delta\gamma$, we obtain that
\[
\left|\ln\left(\frac{b_{n}\left(t\right)}{b_{n}\left(1\right)}\right)\right|\lesssim_{\delta}\ln\left(C\right)\frac{4\mathrm{e}^{-2}}{\left(1-\beta\right)^{2}\delta^{2}\gamma^{2}\ln\left(C\right)^{2}}C^{-\beta\delta\gamma\left(\frac{1}{1-\gamma}\right)^{n}}.
\]
Since $\gamma\ge\frac{1}{2}$ by Choice \ref{choice:min requirements on C gamma and kmax},
we have
\[
\left|\ln\left(\frac{b_{n}\left(t\right)}{b_{n}\left(1\right)}\right)\right|\lesssim_{\delta}\frac{C^{-\beta\delta\gamma\left(\frac{1}{1-\gamma}\right)^{n}}}{\left(1-\beta\right)^{2}\delta^{2}\ln\left(C\right)}\lesssim_{\beta,\delta}\frac{C^{-\beta\delta\gamma\left(\frac{1}{1-\gamma}\right)^{n}}}{\ln\left(C\right)}.
\]
By the definition of $\lesssim_{\beta,\delta}$, there is $K_{0}\left(\beta,\delta\right)>0$
such that
\[
\left|\ln\left(\frac{b_{n}\left(t\right)}{b_{n}\left(1\right)}\right)\right|\leq K_{0}\left(\beta,\delta\right)\frac{C^{-\beta\delta\gamma\left(\frac{1}{1-\gamma}\right)^{n}}}{\ln\left(C\right)}.
\]
Taking $C$ sufficiently large, we can compensate the constant $K_{0}\left(\beta,\delta\right)$
with the factor $\ln\left(C\right)$. Indeed, provided that $C\ge\exp\left(384K_{0}\left(\beta,\delta\right)\right)\eqqcolon\Upsilon\left(\beta,\delta\right)$,
we can guarantee that
\begin{equation}
\left|\ln\left(\frac{b_{n}\left(t\right)}{b_{n}\left(1\right)}\right)\right|\leq\frac{1}{384}C^{-\beta\delta\gamma\left(\frac{1}{1-\gamma}\right)^{n}}.\label{eq:bound ln bn}
\end{equation}
Since the function $f_{1}\left(x\right)=\mathrm{e}^{x}-1$ is Lipschitz
in the interval $\left[-1,1\right]$ with constant $3$ and satisfies
$f_{1}\left(0\right)=0$, we can assure that
\begin{equation}
\begin{aligned}\left|\frac{b_{n}\left(t\right)}{b_{n}\left(1\right)}-1\right| & =\left|f_{1}\left(\ln\left(\frac{b_{n}\left(t\right)}{b_{n}\left(1\right)}\right)\right)-\overbrace{f_{1}\left(0\right)}^{=0}\right|\le\\
 & \le3\left|\ln\left(\frac{b_{n}\left(t\right)}{b_{n}\left(1\right)}\right)-0\right|\le\frac{1}{128}C^{-\beta\delta\gamma\left(\frac{1}{1-\gamma}\right)^{n}}
\end{aligned}
\label{eq:convergence bn statement improved}
\end{equation}
as long as $\left|\ln\left(\frac{b_{n}\left(t\right)}{b_{n}\left(1\right)}\right)\right|\le1$,
which we know to be always true thanks to equation \eqref{eq:bound ln bn}.
As the computations above are valid $\forall t\in\left[t_{n+1},1\right]$,
this proves the first inequality of the statement.

Let us continue with $a_{n}\left(t\right)$. By Choice \ref{choice:anbncn},
we know that
\[
a_{n}\left(1\right)b_{n}\left(1\right)=a_{n}\left(t\right)b_{n}\left(t\right)\quad\forall t\in\left[t_{n+1},1\right].
\]
Therefore,
\[
\left|\frac{a_{n}\left(t\right)}{a_{n}\left(1\right)}-1\right|=\left|\frac{a_{n}\left(1\right)b_{n}\left(1\right)}{b_{n}\left(t\right)a_{n}\left(1\right)}-1\right|=\left|\frac{b_{n}\left(1\right)}{b_{n}\left(t\right)}-1\right|=\left|\left(\frac{b_{n}\left(t\right)}{b_{n}\left(1\right)}-1+1\right)^{-1}-1\right|.
\]
Because the function
\[
f_{2}\left(x\right)=\left|\left(x+1\right)^{-1}-1\right|
\]
is locally Lipschitz in the interval $\left[-\frac{1}{2},\frac{1}{2}\right]$
with constant $4$ and $f_{2}\left(0\right)=0$, it can be deduced
that
\begin{equation}
\begin{aligned}\left|\frac{a_{n}\left(t\right)}{a_{n}\left(1\right)}-1\right| & =\left|f_{2}\left(\frac{b_{n}\left(t\right)}{b_{n}\left(1\right)}-1\right)-\overbrace{f_{2}\left(0\right)}^{=0}\right|\le\\
 & \le4\left|\frac{b_{n}\left(t\right)}{b_{n}\left(1\right)}-1\right|\le\frac{1}{32}C^{-\beta\delta\gamma\left(\frac{1}{1-\gamma}\right)^{n}},
\end{aligned}
\label{eq:convergence an improved}
\end{equation}
where we have made use of equation \eqref{eq:convergence bn statement improved}
in the last step, as long as $\left|\frac{b_{n}\left(t\right)}{b_{n}\left(1\right)}-1\right|\le\frac{1}{2}$,
which is true by equation \eqref{eq:convergence bn statement improved}. 

Lastly, we need to deal with $B_{n}\left(t\right)$. By Choice \ref{choice:amplitude density},
\[
B_{n}\left(t\right)=\frac{z_{n}}{a_{n}\left(t\right)^{2}+b_{n}\left(t\right)^{2}}\frac{\int_{t_{n}}^{t}h^{\left(n\right)}\left(s\right)b_{n}\left(s\right)\mathrm{d}s}{\int_{t_{n}}^{1}h^{\left(n\right)}\left(s\right)b_{n}\left(s\right)\mathrm{d}s}.
\]
As $t\in\left[t_{n+1},t\right]$, thanks to Choice \ref{choice:no two simultanous densities},
we know that $\left.h^{\left(n\right)}\right|_{\left[t_{n+1},1\right]}\equiv0$
and, consequently,
\[
\int_{t_{n}}^{t}h^{\left(n\right)}\left(s\right)b_{n}\left(s\right)\mathrm{d}s=\int_{t_{n}}^{1}h^{\left(n\right)}\left(s\right)b_{n}\left(s\right)\mathrm{d}s\quad\forall t\in\left[t_{n+1},1\right].
\]
Hence,
\[
B_{n}\left(t\right)=\frac{z_{n}}{a_{n}\left(t\right)^{2}+b_{n}\left(t\right)^{2}}\quad\forall t\in\left[t_{n+1},1\right].
\]
In this way,
\[
\begin{aligned}\frac{B_{n}\left(1\right)}{B_{n}\left(t\right)}-1 & =\frac{a_{n}\left(t\right)^{2}+b_{n}\left(t\right)^{2}}{a_{n}\left(1\right)^{2}+b_{n}\left(1\right)^{2}}-1=\frac{a_{n}\left(t\right)^{2}-a_{n}\left(1\right)^{2}}{a_{n}\left(1\right)^{2}+b_{n}\left(1\right)^{2}}+\frac{b_{n}\left(t\right)^{2}-b_{n}\left(1\right)^{2}}{a_{n}\left(1\right)^{2}+b_{n}\left(1\right)^{2}}=\\
 & =\underbrace{\frac{a_{n}\left(1\right)^{2}}{a_{n}\left(1\right)^{2}+b_{n}\left(1\right)^{2}}}_{\le1}\left(\frac{a_{n}\left(t\right)^{2}}{a_{n}\left(1\right)^{2}}-1\right)+\underbrace{\frac{b_{n}\left(1\right)^{2}}{a_{n}\left(1\right)^{2}+b_{n}\left(1\right)^{2}}}_{\le1}\left(\frac{b_{n}\left(t\right)^{2}}{b_{n}\left(1\right)^{2}}-1\right).
\end{aligned}
\]
Then,
\[
\begin{aligned}\left|\frac{B_{n}\left(1\right)}{B_{n}\left(t\right)}-1\right| & \le\left|\frac{a_{n}\left(t\right)^{2}}{a_{n}\left(1\right)^{2}}-1\right|+\left|\frac{b_{n}\left(t\right)^{2}}{b_{n}\left(1\right)^{2}}-1\right|=\\
 & \le\left|-2+2\frac{a_{n}\left(t\right)}{a_{n}\left(1\right)}+\left(\frac{a_{n}\left(t\right)}{a_{n}\left(1\right)}-1\right)^{2}\right|+\left|-2+2\frac{b_{n}\left(t\right)}{b_{n}\left(1\right)}+\left(\frac{b_{n}\left(t\right)}{b_{n}\left(1\right)}-1\right)^{2}\right|\le\\
 & \le2\left|\frac{a_{n}\left(t\right)}{a_{n}\left(1\right)}-1\right|+\left|\frac{a_{n}\left(t\right)}{a_{n}\left(1\right)}-1\right|^{2}+2\left|\frac{b_{n}\left(t\right)}{b_{n}\left(1\right)}-1\right|+\left|\frac{b_{n}\left(t\right)}{b_{n}\left(1\right)}-1\right|^{2}.
\end{aligned}
\]
Since the function
\[
f_{3}\left(x,y\right)=2x+x^{2}+2y+y^{2}
\]
is Lipschitz in the interval $\left[-1,1\right]^{2}$, satisfying
\[
\left|f_{3}\left(x_{2},y_{2}\right)-f_{3}\left(x_{1},y_{1}\right)\right|\le4\left(\left|x_{2}-x_{1}\right|+\left|y_{2}-y_{1}\right|\right),
\]
and $f_{3}\left(0,0\right)=0$, we deduce that
\[
\begin{aligned}\left|\frac{B_{n}\left(1\right)}{B_{n}\left(t\right)}-1\right| & \leq\left|f_{3}\left(\frac{a_{n}\left(t\right)}{a_{n}\left(1\right)}-1,\frac{b_{n}\left(t\right)}{b_{n}\left(1\right)}-1\right)-\overbrace{f_{3}\left(0,0\right)}^{=0}\right|\leq\\
 & \leq4\left(\left|\frac{a_{n}\left(t\right)}{a_{n}\left(1\right)}-1\right|+\left|\frac{b_{n}\left(t\right)}{b_{n}\left(1\right)}-1\right|\right),
\end{aligned}
\]
as long as 
\[
\left|\frac{a_{n}\left(t\right)}{a_{n}\left(1\right)}-1\right|\leq1,\quad\left|\frac{b_{n}\left(t\right)}{b_{n}\left(1\right)}-1\right|\leq1,
\]
which are true by equations \eqref{eq:convergence bn statement improved}
and \eqref{eq:convergence an improved}. By the aforementioned equations,
we also arrive to
\[
\left|\frac{B_{n}\left(1\right)}{B_{n}\left(t\right)}-1\right|\le\frac{4}{32}C^{-\beta\delta\gamma\left(\frac{1}{1-\gamma}\right)^{n}}+\frac{4}{128}C^{-\beta\delta\gamma\left(\frac{1}{1-\gamma}\right)^{n}}\le\frac{1}{4}C^{-\beta\delta\gamma\left(\frac{1}{1-\gamma}\right)^{n}}.
\]
Notice that, although we wanted a bound for $\left|\frac{B_{n}\left(t\right)}{B_{n}\left(1\right)}-1\right|$,
we have actually found a bound for $\left|\frac{B_{n}\left(1\right)}{B_{n}\left(t\right)}-1\right|$.
This is basically the same situation we faced when we had a bound
for $\left|\frac{b_{n}\left(t\right)}{b_{n}\left(1\right)}-1\right|$
and wanted to find a bound for $\left|\frac{a_{n}\left(t\right)}{a_{n}\left(1\right)}-1\right|$.
Thus, following the same argument, we may conclude that
\[
\left|\frac{B_{n}\left(t\right)}{B_{n}\left(1\right)}-1\right|\le C^{-\beta\delta\gamma\left(\frac{1}{1-\gamma}\right)^{n}}.
\]
\end{proof}

\subsection{Convergence of $\Xi^{\left(n\right)}\left(t\right)$ to $\Xi_{0}^{\left(n\right)}\left(t\right)$}

We shall continue with the main result of this section: the convergence
of $\Xi^{\left(n\right)}\left(t\right)$ to the ideal model $\Xi_{0}^{\left(n\right)}\left(t\right)$.
To accomplish this task, we will need some Lemmas.
\begin{lem}
\label{lem:equivalence sin with x(x-pi)}In the interval $\left[0,\pi\right]$,
we have
\[
\sin\left(x\right)\approx x\left(x-\pi\right),
\]
where $\approx$ denotes $\lesssim$ and $\gtrsim$ simultaneously.
\end{lem}
\begin{proof}
Let us consider the quotient
\[
f\left(x\right)\coloneqq\frac{\sin\left(x\right)}{x\left(\pi-x\right)}.
\]
Notice that $f$ can be continuously extended at $x=0$ and $x=\pi$
with values $f\left(0\right)=f\left(\pi\right)=\frac{1}{\pi}$. As
the only zeroes of $\sin\left(x\right)$ in the interval $\left[0,\pi\right]$
are $x=0$ and $x=\pi$, this means that $f$ does not change sign
in $\left[0,\pi\right]$. By Weierstrass' Theorem, as $\left[0,\pi\right]$
is compact and $f$ is continuous, we deduce that $f$ attains a maximum
and a minimum value in $\left[0,\pi\right]$. Furthermore, we know
that this minimum value must be strictly positive because $f$ has
no zeros in $\left[0,\pi\right]$. This proves the result.
\end{proof}
\begin{lem}
\label{lem:bound sines 2}Let $\varepsilon\in\left[-1,1\right]$.
We have
\[
\left|\frac{\sin\left(\left(1+\varepsilon\right)x\right)}{\sin\left(x\right)}-1\right|\lesssim\frac{\left|\varepsilon\right|}{\pi-x}\quad\forall x\in\left[0,\pi\right).
\]
\end{lem}
\begin{proof}
Notice that we can write
\begin{equation}
\frac{\sin\left(\left(1+\varepsilon\right)x\right)}{\sin\left(x\right)}-1=\frac{\sin\left(\left(1+\varepsilon\right)x\right)-\sin\left(x\right)}{\sin\left(x\right)}=\underbrace{\frac{\sin\left(\left(1+\varepsilon\right)x\right)-\sin\left(x\right)}{x}}_{\eqqcolon f\left(x\right)}\underbrace{\frac{x}{\sin\left(x\right)}}_{\eqqcolon g\left(x\right)}.\label{eq:decomposition sines}
\end{equation}
To continue, we will bound both functions separately. To find a bound
for $f\left(x\right)$, we shall express $\sin\left(\cdot\right)$
via its Taylor series. Thereby,
\begin{equation}
f\left(x\right)=\frac{\sum_{n=0}^{\infty}\frac{\left(-1\right)^{n}}{\left(2n+1\right)!}\left(1+\varepsilon\right)^{2n+1}x^{2n+1}-\sum_{n=0}^{\infty}\frac{\left(-1\right)^{n}}{\left(2n+1\right)!}x^{2n+1}}{x}=\sum_{n=0}^{\infty}\frac{\left(-1\right)^{n}}{\left(2n+1\right)!}x^{2n}\left[\left(1+\varepsilon\right)^{2n+1}-1\right].\label{eq:f}
\end{equation}
Next, we shall take advantage of the fact that
\[
z^{n}-1=\left(1+z+\dots+z^{n-1}\right)\left(z-1\right).
\]
Evaluating at $z=1+\varepsilon$, we arrive to
\begin{equation}
\left(1+\varepsilon\right)^{n}-1=\varepsilon\sum_{m=0}^{n-1}\left(1+\varepsilon\right)^{m}.\label{eq:1-eps}
\end{equation}
Making use of \eqref{eq:1-eps} in \eqref{eq:f} provides
\[
f\left(x\right)=\varepsilon\sum_{n=0}^{\infty}\frac{\left(-1\right)^{n}}{\left(2n+1\right)!}x^{2n}\left(\sum_{m=0}^{2n}\left(1+\varepsilon\right)^{m}\right).
\]
Furthermore, clearly, as we are considering $x\in\left[0,\pi\right)$,
\begin{equation}
\left|f\left(x\right)\right|\le\left|\varepsilon\right|\sum_{n=0}^{\infty}\underbrace{\frac{1}{\left(2n+1\right)!}\pi^{2n}\left(\sum_{m=0}^{2n}\left(1+\varepsilon\right)^{m}\right)}_{\eqqcolon s_{n}}.\label{eq:def sn}
\end{equation}
Now, we will prove that this series is convergent. To do this, we
will apply D'Alambert's ratio test:
\[
\frac{s_{n+1}}{s_{n}}=\frac{\frac{1}{\left(2n+3\right)!}\pi^{2n+2}\left(\sum_{m=0}^{2n+2}\left(1+\varepsilon\right)^{m}\right)}{\frac{1}{\left(2n+1\right)!}\pi^{2n}\left(\sum_{m=0}^{2n}\left(1+\varepsilon\right)^{m}\right)}.
\]
We compute
\[
\sum_{m=0}^{2n}\left(1+\varepsilon\right)^{m}=\frac{\left(1+\varepsilon\right)^{2n+1}-1}{1+\varepsilon-1},\quad\sum_{m=0}^{2n+2}\left(1+\varepsilon\right)^{m}=\frac{\left(1+\varepsilon\right)^{2n+3}-1}{1+\varepsilon-1}.
\]
Thus,
\begin{equation}
\frac{s_{n+1}}{s_{n}}=\frac{\pi^{2}\left[\left(1+\varepsilon\right)^{2n+3}-1\right]}{\left(2n+3\right)\left(2n+2\right)\left[\left(1+\varepsilon\right)^{2n+1}-1\right]}.\label{eq:quotient sn}
\end{equation}
To continue, consider the function
\begin{equation}
\begin{aligned}h\left(\varepsilon\right) & \coloneqq\frac{\left(1+\varepsilon\right)^{2n+3}-1}{\left(1+\varepsilon\right)^{2n+1}-1}=\frac{\left(1+\varepsilon\right)^{2n+3}-\left(1+\varepsilon\right)^{2}+\left(1+\varepsilon\right)^{2}-1}{\left(1+\varepsilon\right)^{2n+1}-1}=\\
 & =\left(1+\varepsilon\right)^{2}\frac{\left(1+\varepsilon\right)^{2n+1}-1}{\left(1+\varepsilon\right)^{2n+1}-1}+\frac{\left(1+\varepsilon\right)^{2}-1}{\left(1+\varepsilon\right)^{2n+1}-1}=\\
 & =\left(1+\varepsilon\right)^{2}+\frac{\left(1+\varepsilon\right)^{2}-1}{\left(1+\varepsilon\right)^{2n+1}-1}.
\end{aligned}
\label{eq:h(eps)}
\end{equation}
It is evident that $\left(1+\varepsilon\right)^{2}$ is bounded for
$\varepsilon\in\left[-1,1\right]$. Now, for the other summand, let
us distinguish two cases. If $\varepsilon>0$, we have $1+\varepsilon>0$
and, consequently, as $n\ge1$,
\begin{equation}
\left(1+\varepsilon\right)^{2}\le\left(1+\varepsilon\right)^{2n+1}\iff\left(1+\varepsilon\right)^{2}-1\le\left(1+\varepsilon\right)^{2n+1}-1\iff\frac{\left(1+\varepsilon\right)^{2}-1}{\left(1+\varepsilon\right)^{2n+1}-1}\le1.\label{eq:h(eps) positive}
\end{equation}
On the other hand, if $\varepsilon<0$, we have $0\le1+\varepsilon<1$
and, as a consequence, as $n\ge1$,
\begin{equation}
\left(1+\varepsilon\right)^{2}\ge\left(1+\varepsilon\right)^{2n+1}\iff1-\left(1+\varepsilon\right)^{2}\le1-\left(1+\varepsilon\right)^{2n+1}\iff\frac{1-\left(1+\varepsilon\right)^{2}}{1-\left(1+\varepsilon\right)^{2n+1}}\le1.\label{eq:h(eps) negative}
\end{equation}
In other words, equations \eqref{eq:h(eps) positive} and \eqref{eq:h(eps) negative}
prove that
\[
\frac{\left(1+\varepsilon\right)^{2}-1}{\left(1+\varepsilon\right)^{2n+1}-1}\le1\quad\forall\varepsilon\in\left[-1,1\right]\setminus\left\{ 0\right\} .
\]
Taking the limit $\varepsilon\to0$, we deduce that the bound must
also be true for $\varepsilon=0$. Hence, by equation \eqref{eq:h(eps)},
it is true that $h\left(\varepsilon\right)\leq\left(1+1\right)^{2}+1=5$
$\forall\varepsilon\in\left[-1,1\right]$. Making use of this information
in equation \eqref{eq:quotient sn}, we infer that
\[
\frac{s_{n+1}}{s_{n}}\leq\frac{5\pi^{2}}{\left(2n+3\right)\left(2n+2\right)}\xrightarrow[n\to\infty]{}0.
\]
Consequently, the ratio test assures that the series $\sum_{n=0}^{\infty}s_{n}$
converges (absolutely) and, in this way, in the view of \eqref{eq:def sn},
it can be concluded that
\begin{equation}
\left|f\left(x\right)\right|\lesssim\left|\varepsilon\right|\quad\forall x\in\left[0,\pi\right].\label{eq:bound f}
\end{equation}
As for $g\left(x\right)$, by Lemma \ref{lem:equivalence sin with x(x-pi)},
we have
\begin{equation}
g\left(x\right)=\frac{x}{\sin\left(x\right)}\lesssim\frac{x}{x\left(\pi-x\right)}=\frac{1}{\pi-x}.\label{eq:bound g}
\end{equation}
Substituting \eqref{eq:bound f} and \eqref{eq:bound g} in \eqref{eq:decomposition sines}
provides the result.
\end{proof}
\begin{lem}
\label{lem:inverse of primitive of 1/sin(x)}Let $F$ denote a primitive
of $\frac{1}{\sin\left(x\right)}$. Then, if $x,y\in\left[0,\pi\right]$
and $\left|F\left(y\right)-F\left(x\right)\right|<1$, we have
\[
\left|y-x\right|\le\frac{\sin\left(x\right)\left|F\left(y\right)-F\left(x\right)\right|}{1-\left|F\left(y\right)-F\left(x\right)\right|}.
\]
In figure \ref{fig:primitive 1/sin(x)}, we can see the graph of $F$.
\end{lem}
\begin{figure}[h]
\begin{centering}
\includegraphics[width=0.7\linewidth]{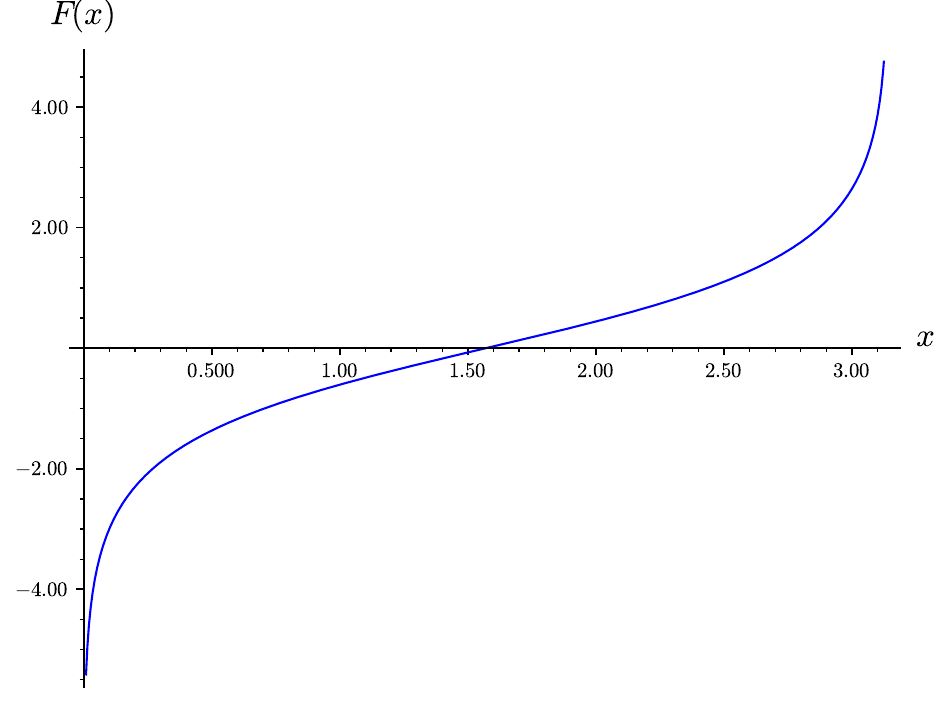}
\par\end{centering}
\caption{\label{fig:primitive 1/sin(x)}Graph of a primitive of $\frac{1}{\sin\left(x\right)}$.
Notice that the inverse of $F$ is well defined and has small derivative
near $0$ and $\pi$. This is the main idea behind the proof of Lemma
\ref{lem:inverse of primitive of 1/sin(x)}.}

\end{figure}

\begin{proof}
Consider $\frac{\mathrm{d}F}{\mathrm{d}x}=\frac{1}{\sin\left(x\right)}$
in the interval $x\in\left[0,\pi\right]$. As $\left|\sin\left(x\right)\right|$
is bounded above, we deduce that $\frac{\mathrm{d}F}{\mathrm{d}x}$
is never zero. Furthermore, $F$ is only ill-defined at $x=0$ and
$x=\pi$ and, in both cases, we have $\lim_{x\to0^{+}}\frac{\mathrm{d}F}{\mathrm{d}x}\left(x\right)=+\infty=\lim_{x\to\pi^{-}}\frac{\mathrm{d}F}{\mathrm{d}x}\left(x\right)$.
Therefore, we deduce that $\frac{\mathrm{d}F}{\mathrm{d}x}$ is always
strictly positive. In this way, $F$ is strictly increasing and, consequently,
has a well defined inverse $F^{-1}$ in the interval $\left[0,\pi\right]$.
By the Inverse Function Theorem, we have
\[
\frac{\mathrm{d}F^{-1}}{\mathrm{d}x}\left(F\left(x\right)\right)=\frac{1}{\frac{\mathrm{d}F}{\mathrm{d}x}\left(x\right)}=\sin\left(x\right).
\]
Let $x,x+h\in\left[0,\pi\right]$. The Mean Value Theorem assures
us that there exists 
\[
\xi\in\left[\min\left\{ F\left(x\right),F\left(x+h\right)\right\} ,\max\left\{ F\left(x\right),F\left(x+h\right)\right\} \right]
\]
such that
\[
\begin{aligned}h & =x+h-x=F^{-1}\left(F\left(x+h\right)\right)-F^{-1}\left(F\left(x\right)\right)=\frac{\mathrm{d}F^{-1}}{\mathrm{d}x}\left(\xi\right)\left(F\left(x+h\right)-F\left(x\right)\right)=\\
 & =\sin\left(F^{-1}\left(\xi\right)\right)\left(F\left(x+h\right)-F\left(x\right)\right).
\end{aligned}
\]
Since $\sin\left(\cdot\right)$ is a Lipschitz function with constant
$1$ and $F^{-1}$ is increasing (because $\frac{\mathrm{d}F^{-1}}{\mathrm{d}x}\left(x\right)>0$
$\forall x\in\left(0,\pi\right)$), we may write
\[
\begin{aligned}\left|h\right| & \le\left|\sin\left(F^{-1}\left(F\left(x\right)\right)\right)+\left|F^{-1}\left(\xi\right)-F^{-1}\left(F\left(x\right)\right)\right|\right|\left|F\left(x+h\right)-F\left(x\right)\right|\le\\
 & \le\left|\sin\left(F^{-1}\left(F\left(x\right)\right)\right)+\left|F^{-1}\left(F\left(x+h\right)\right)-F^{-1}\left(F\left(x\right)\right)\right|\right|\left|F\left(x+h\right)-F\left(x\right)\right|=\\
 & \le\left(\sin\left(x\right)+\left|h\right|\right)\left|F\left(x+h\right)-F\left(x\right)\right|.
\end{aligned}
\]
Solving for $\left|h\right|$ leads to
\[
\left|h\right|\le\frac{\sin\left(x\right)\left|F\left(x+h\right)-F\left(x\right)\right|}{1-\left|F\left(x+h\right)-F\left(x\right)\right|},
\]
which makes sense as long as $\left|F\left(x+h\right)-F\left(x\right)\right|<1$.
Rewriting $y=x+h$, we obtain the result of the statement.
\end{proof}
\begin{lem}
\label{lem:trap}Let $T>0$, $\left[a,c\right]\subseteq\mathbb{R}$
and $x\left(t\right)$ be the solution of the initial value problem
\[
\frac{\mathrm{d}x}{\mathrm{d}t}\left(t\right)=F\left(t,x\left(t\right)\right),\quad x\left(0\right)=a,
\]
where $F:\left[0,\infty\right)\times\left[a,c\right]\to\mathbb{R}$
is locally Lipschitz in its spatial variable. If
\begin{enumerate}
\item $F\left(t,a\right)>0$ $\forall t\in\left[0,T\right]$,
\item $F\left(t,z\right)\le\varepsilon$ $\forall t\in\left[0,T\right]$
and $\forall z\in\left[b,c\right]$, where $b\in\left(a,c\right)$,
\end{enumerate}
then, $x\left(t\right)\in\left[a,b+\varepsilon t\right]$ $\forall t\in\left[0,T\right]$
as long as $b+\varepsilon t\le c$.
\end{lem}
\begin{proof}
As $F\left(t,a\right)>0$ $\forall t\in\left[0,T\right]$, any solution
that attains the value $a$ must be increasing when it attains it.
This means that, if a solution has crossed this value once, it may
never cross it again. Consequently, $x\left(t\right)\ge a$ $\forall t\in\left[0,T\right]$.
To prove that $x\left(t\right)\le b+\varepsilon t$ $\forall t\in\left[0,T\right]$,
consider two cases:
\begin{itemize}
\item $x\left(t\right)<b$ $\forall t\in\left[0,T\right]$. Then, we are
done.
\item $x\left(t^{*}\right)=b$ for some $t^{*}\in\left[0,T\right]$. If
the equation $x\left(t^{*}\right)=b$ has multiple solutions, let
$t^{*}$ be the first one. Then, we can integrate to obtain
\[
x\left(t\right)-b=x\left(t\right)-x\left(t^{*}\right)=\int_{t^{*}}^{t}\frac{\mathrm{d}x}{\mathrm{d}t}\left(s\right)\mathrm{d}s=\int_{t^{*}}^{t}F\left(s,x\left(s\right)\right)\mathrm{d}s\le\varepsilon\left(t-t^{*}\right)\le\varepsilon t.
\]
This argument is only valid as long as $b\le x\left(t\right)\le c$.
We can guarantee this last inequality if $b+\varepsilon t\le c$.
Concerning the first one ($b\le x\left(t\right)$), if, for some time
$t^{**}\in\left[t^{*},T\right]$, $x\left(t\right)$ crosses $b$
from the right to the left, then, we will trivially have $x\left(t\right)\le b\le b+\varepsilon t$
for the time that $x\left(t\right)$ remains to the left of $b$.
If $x\left(t\right)$ were to cross $b$ from the left to the right
again (let us say, at a certain $t^{***}\in\left[t^{**},T\right]$),
we could repeat the same argument to show that $x\left(t\right)\le b+\varepsilon t$
for the time that $x\left(t\right)$ remains to the right of $b\left(t\right)$.
In other words, iterating these arguments as needed, we can see that
the worst case is that $x\left(t\right)$ crosses $b$ from the left
to the right at $t=0$ and remains in that region growing at the maximum
speed $\varepsilon$, which is exactly what the term $b+\varepsilon t$
represents.
\end{itemize}
\end{proof}
\begin{prop}
\label{prop:JI convergence to ideal}Let $n\in\mathbb{N}$ with $n\ge2$
and $\beta,\beta',\beta''\in\left(0,1\right)$. Moreover, let
\begin{itemize}
\item $\Xi^{\left(n\right)}\left(t\right)\coloneqq\phi_{1}^{\left(n\right)}\left(t,0\right)-\phi_{1}^{\left(n-1\right)}\left(t,0\right)$
(for the definition of $\Xi^{\left(1\right)}\left(t\right)$, take
$\phi_{1}^{\left(0\right)}\left(t,0\right)\equiv0$),
\item $F_{n}\left(\tau\right)$ be the solution of the equation presented
in Lemma \ref{lem:the good ODE} with initial condition $F_{n}\left(0\right)=a_{n-1}\left(1\right)\Xi^{\left(n\right)}\left(t_{n}\right)$,
\item $\Xi_{0}^{\left(n\right)}\left(t\right)\coloneqq\frac{1}{a_{n-1}\left(1\right)}F_{n}\left(B_{n-1}\left(1\right)a_{n-1}\left(1\right)b_{n-1}\left(1\right)\left(t-t_{n}\right)\right)$.
\end{itemize}
Provided that $C$ is big enough (let us say $C\ge\Upsilon\left(\beta,\beta',\beta'',\delta\right)$),
there is $t_{n}^{*}\in\left[t_{n}+\frac{1-t_{n}}{2},1\right]$ (which
depends on $\beta,\beta',\beta''$ and $\delta$) such that
\begin{enumerate}
\item ~
\[
a_{n-1}\left(1\right)\left|\Xi^{\left(n\right)}\left(t\right)-\Xi_{0}^{\left(n\right)}\left(t\right)\right|\leq C^{-\beta\beta'\delta\gamma\left(\frac{1}{1-\gamma}\right)^{n-1}}\sin\left(a_{n-1}\left(1\right)\Xi_{0}^{\left(n\right)}\left(t\right)\right)\quad\forall t\in\left[t_{n},t_{n}+\frac{1-t_{n}}{2}\right],
\]
\item ~
\[
a_{n-1}\left(1\right)\left|\Xi^{\left(n\right)}\left(t\right)-\Xi_{0}^{\left(n\right)}\left(t\right)\right|\leq C^{-\beta\beta'\left(1-\beta''\right)\delta\gamma\left(\frac{1}{1-\gamma}\right)^{n-1}}\sin\left(a_{n-1}\left(1\right)\Xi_{0}^{\left(n\right)}\left(t\right)\right)\quad\forall t\in\left[t_{n}+\frac{1-t_{n}}{2},t_{n}^{*}\right],
\]
\item ~
\[
a_{n-1}\left(1\right)\left|\Xi^{\left(n\right)}\left(t\right)-\Xi_{0}^{\left(n\right)}\left(t\right)\right|\leq14C^{-\beta\beta'\min\left\{ \beta'',1-\beta''\right\} \delta\gamma\left(\frac{1}{1-\gamma}\right)^{n-1}}\quad\forall t\in\left[t_{n}^{*},1\right].
\]
\end{enumerate}
Consequently, taking $\beta''=\frac{1}{2}$,
\[
a_{n-1}\left(1\right)\left|\Xi^{\left(n\right)}\left(t\right)-\Xi_{0}^{\left(n\right)}\left(t\right)\right|\leq14C^{-\frac{1}{2}\beta\beta'\delta\gamma\left(\frac{1}{1-\gamma}\right)^{n-1}}\quad\forall t\in\left[t_{n},1\right].
\]
\end{prop}
\begin{rem}
By the nature of equation \eqref{eq:ODE JIn}, its closeness to the
ideal model varies throughout the interval $\left[t_{n},1\right]$.
In the first half of the interval, where $\Xi_{0}^{\left(n\right)}\left(t\right)$
is increasing, we have a very good convergence to the ideal model.
However, as $\Xi_{0}^{\left(n\right)}\left(t\right)$ begins to decrease,
this convergence worsens. Finally, when $a_{n-1}\left(1\right)\Xi^{\left(n\right)}\left(t\right)$
nears $\pi$ enough, the leading term $B_{n-1}\left(t\right)b_{n-1}\left(t\right)\sin\left(a_{n-1}\left(1\right)\Xi^{\left(n\right)}\left(t\right)\right)$
becomes of the same order as the other terms in \eqref{eq:ODE JIn}
and we completely lose control of the distance $\left|\Xi^{\left(n\right)}\left(t\right)-\Xi_{0}^{\left(n\right)}\left(t\right)\right|$.
Nevertheless, we can still guarantee that $a_{n-1}\left(1\right)\Xi^{\left(n\right)}\left(t\right)$
remains close to $\pi$. As the ideal model also remains close to
$\pi$, they must be close to each other.
\end{rem}
\begin{proof}
Throughout the proof, we will present various requirements on $C$
of the form $C\ge\Upsilon_{i}\left(\text{parameters}\right)$. For
the proof to work, we need to take $C$ as the maximum of all the
$\Upsilon_{i}$. Moreover, observe that the $\Xi^{\left(n\right)}\left(t\right)$
and $\Xi_{0}^{\left(n\right)}\left(t\right)$ defined here are exactly
the same we introduced in subsection \ref{subsec:transport of the layer centers}.
There, we found that
\begin{equation}
\begin{aligned}\frac{\mathrm{d}\Xi^{\left(n\right)}}{\mathrm{d}t}\left(t\right) & =B_{n-1}\left(t\right)b_{n-1}\left(t\right)\sin\left(a_{n-1}\left(t\right)\Xi^{\left(n\right)}\left(t\right)\right)+\\
 & \quad+\sum_{m=1}^{n-2}B_{m}\left(t\right)b_{m}\left(t\right)\left[\sin\left(a_{m}\left(t\right)\left(\phi_{1}^{\left(n\right)}\left(t,0\right)-\phi_{1}^{\left(m\right)}\left(t,0\right)\right)\right)+\right.\\
 & \qquad\left.-\sin\left(a_{m}\left(t\right)\left(\phi_{1}^{\left(n-1\right)}\left(t,0\right)-\phi_{1}^{\left(m\right)}\left(t,0\right)\right)\right)\right]
\end{aligned}
\label{eq:ODE JIn}
\end{equation}
and
\begin{equation}
\frac{\mathrm{d}\Xi_{0}^{\left(n\right)}}{\mathrm{d}t}\left(t\right)=B_{n-1}\left(1\right)b_{n-1}\left(1\right)\sin\left(a_{n-1}\left(1\right)\Xi_{0}^{\left(n\right)}\left(t\right)\right).\label{eq:ODE Jin0}
\end{equation}
We will consider times $t\in\left[t_{n},1\right]$. Let us see that
$\Xi^{\left(n\right)}\left(t\right)$ and $\Xi_{0}^{\left(n\right)}\left(t\right)$
are actually close. To do this, one would normally study $\frac{\mathrm{d}}{\mathrm{d}t}\left(\left|\Xi^{\left(n\right)}\left(t\right)-\Xi_{0}^{\left(n\right)}\left(t\right)\right|\right)$
and, then, try to apply Grönwall's inequality. Nevertheless, in our
case, this approach fails to provide a bound that is not increasing
in $n\in\mathbb{N}$. Indeed, if one tries to apply Grönwall's inequality,
one obtains a bound of the form
\[
a_{n-1}\left(1\right)\left|\Xi^{\left(n\right)}\left(t\right)-\Xi_{0}^{\left(n\right)}\left(t\right)\right|\lesssim_{\delta}C^{-\beta\gamma\delta\left(\frac{1}{1-\gamma}\right)^{n-1}}C^{k_{\max}\left(\frac{1}{1-\gamma}\right)^{n}},
\]
which is increasing in $n\in\mathbb{N}$.

Hence, we have to try another idea. Inspiration comes from the following
fact: let $F\left(x\right)$ denote a primitive of $\frac{1}{\sin\left(x\right)}$
(see the proof of the first point of Lemma \ref{lem:the good ODE}
to find an expression for $F\left(x\right)$) and consider
\[
\frac{\mathrm{d}}{\mathrm{d}t}\left(F\left(a_{n-1}\left(1\right)\Xi_{0}^{\left(n\right)}\left(t\right)\right)\right).
\]
By equation \eqref{eq:ODE Jin0}, we deduce that
\[
\begin{alignedat}{1}\frac{\mathrm{d}}{\mathrm{d}t}\left(F\left(a_{n-1}\left(1\right)\Xi_{0}^{\left(n\right)}\left(t\right)\right)\right) & =\frac{a_{n-1}\left(1\right)\frac{\mathrm{d}\Xi_{0}^{\left(n\right)}}{\mathrm{d}t}\left(t\right)}{\sin\left(a_{n-1}\left(1\right)\Xi_{0}^{\left(n\right)}\left(t\right)\right)}=\frac{a_{n-1}\left(1\right)B_{n-1}\left(1\right)b_{n-1}\left(1\right)}{\sin\left(a_{n-1}\left(1\right)\Xi_{0}^{\left(n\right)}\left(t\right)\right)}\sin\left(a_{n-1}\left(1\right)\Xi_{0}^{\left(n\right)}\left(t\right)\right)=\\
 & =B_{n-1}\left(1\right)a_{n-1}\left(1\right)b_{n-1}\left(1\right)
\end{alignedat}
\]
is constant in $t$. This means that we expect $\frac{\mathrm{d}}{\mathrm{d}t}\left(F\left(a_{n-1}\left(1\right)\Xi^{\left(n\right)}\left(t\right)\right)\right)$
to be almost constant in $t$ as well. Consequently, it makes sense
to study
\[
\frac{\mathrm{d}}{\mathrm{d}t}\left(F\left(a_{n-1}\left(1\right)\Xi^{\left(n\right)}\left(t\right)\right)-F\left(a_{n-1}\left(1\right)\Xi_{0}^{\left(n\right)}\left(t\right)\right)\right).
\]
By the definition of $F$ and equations \eqref{eq:ODE JIn} and \eqref{eq:ODE Jin0},
we get
\[
\begin{aligned} & \frac{\mathrm{d}}{\mathrm{d}t}\left(F\left(a_{n-1}\left(1\right)\Xi^{\left(n\right)}\left(t\right)\right)-F\left(a_{n-1}\left(1\right)\Xi_{0}^{\left(n\right)}\left(t\right)\right)\right)=\\
= & a_{n-1}\left(1\right)\frac{B_{n-1}\left(t\right)b_{n-1}\left(t\right)\sin\left(a_{n-1}\left(t\right)\Xi^{\left(n\right)}\left(t\right)\right)}{\sin\left(a_{n-1}\left(1\right)\Xi^{\left(n\right)}\left(t\right)\right)}+\\
 & \quad+\frac{a_{n-1}\left(1\right)}{\sin\left(a_{n-1}\left(1\right)\Xi^{\left(n\right)}\left(t\right)\right)}\sum_{m=1}^{n-2}B_{m}\left(t\right)b_{m}\left(t\right)\left[\sin\left(a_{m}\left(t\right)\left(\phi_{1}^{\left(n\right)}\left(t,0\right)-\phi_{1}^{\left(m\right)}\left(t,0\right)\right)\right)+\right.\\
 & \qquad\left.-\sin\left(a_{m}\left(t\right)\left(\phi_{1}^{\left(n-1\right)}\left(t,0\right)-\phi_{1}^{\left(m\right)}\left(t,0\right)\right)\right)\right]+\\
 & \quad-\underbrace{a_{n-1}\left(1\right)\frac{B_{n-1}\left(1\right)b_{n-1}\left(1\right)\sin\left(a_{n-1}\left(1\right)\Xi_{0}^{\left(n\right)}\left(t\right)\right)}{\sin\left(a_{n-1}\left(1\right)\Xi_{0}^{\left(n\right)}\left(t\right)\right)}}_{=B_{n-1}\left(1\right)a_{n-1}\left(1\right)b_{n-1}\left(1\right)}.
\end{aligned}
\]
If we call
\[
\begin{aligned}I_{2}\coloneqq & \frac{a_{n-1}\left(1\right)}{\sin\left(a_{n-1}\left(1\right)\Xi^{\left(n\right)}\left(t\right)\right)}\sum_{m=1}^{n-2}B_{m}\left(t\right)b_{m}\left(t\right)\left[\sin\left(a_{m}\left(t\right)\left(\phi_{1}^{\left(n\right)}\left(t,0\right)-\phi_{1}^{\left(m\right)}\left(t,0\right)\right)\right)+\right.\\
 & \quad\left.-\sin\left(a_{m}\left(t\right)\left(\phi_{1}^{\left(n-1\right)}\left(t,0\right)-\phi_{1}^{\left(m\right)}\left(t,0\right)\right)\right)\right],
\end{aligned}
\]
we obtain
\begin{equation}
\begin{aligned} & \frac{\mathrm{d}}{\mathrm{d}t}\left(F\left(a_{n-1}\left(1\right)\Xi^{\left(n\right)}\left(t\right)\right)-F\left(a_{n-1}\left(1\right)\Xi_{0}^{\left(n\right)}\left(t\right)\right)\right)=\\
= & \underbrace{a_{n-1}\left(1\right)\frac{B_{n-1}\left(t\right)b_{n-1}\left(t\right)\sin\left(a_{n-1}\left(t\right)\Xi^{\left(n\right)}\left(t\right)\right)-B_{n-1}\left(1\right)b_{n-1}\left(1\right)\sin\left(a_{n-1}\left(1\right)\Xi^{\left(n\right)}\left(t\right)\right)}{\sin\left(a_{n-1}\left(1\right)\Xi^{\left(n\right)}\left(t\right)\right)}}_{\eqqcolon I_{1}}+I_{2}.
\end{aligned}
\label{eq:time derivative F}
\end{equation}
To continue, we will bound $I_{1}$ and $I_{2}$ separately.
\begin{itemize}
\item $I_{1}$. First of all, summing and subtracting some terms, we arrive
to
\[
\begin{aligned} & B_{n-1}\left(t\right)b_{n-1}\left(t\right)\sin\left(a_{n-1}\left(t\right)\Xi^{\left(n\right)}\left(t\right)\right)-B_{n-1}\left(1\right)b_{n-1}\left(1\right)\sin\left(a_{n-1}\left(1\right)\Xi^{\left(n\right)}\left(t\right)\right)=\\
= & B_{n-1}\left(t\right)b_{n-1}\left(t\right)\sin\left(a_{n-1}\left(t\right)\Xi^{\left(n\right)}\left(t\right)\right)-B_{n-1}\left(1\right)b_{n-1}\left(1\right)\sin\left(a_{n-1}\left(t\right)\Xi^{\left(n\right)}\left(t\right)\right)+\\
 & \quad+B_{n-1}\left(1\right)b_{n-1}\left(1\right)\sin\left(a_{n-1}\left(t\right)\Xi^{\left(n\right)}\left(t\right)\right)-B_{n-1}\left(1\right)b_{n-1}\left(1\right)\sin\left(a_{n-1}\left(1\right)\Xi^{\left(n\right)}\left(t\right)\right)=\\
= & B_{n-1}\left(1\right)b_{n-1}\left(1\right)\left(\frac{B_{n-1}\left(t\right)b_{n-1}\left(t\right)}{B_{n-1}\left(1\right)b_{n-1}\left(1\right)}-1\right)\sin\left(a_{n-1}\left(t\right)\Xi^{\left(n\right)}\left(t\right)\right)+\\
 & \quad+B_{n-1}\left(1\right)b_{n-1}\left(1\right)\left(\sin\left(a_{n-1}\left(t\right)\Xi^{\left(n\right)}\left(t\right)\right)-\sin\left(a_{n-1}\left(1\right)\Xi^{\left(n\right)}\left(t\right)\right)\right).
\end{aligned}
\]
Consequently,
\begin{equation}
\begin{aligned}I_{1} & =B_{n-1}\left(1\right)b_{n-1}\left(1\right)a_{n-1}\left(1\right)\left(\frac{B_{n-1}\left(t\right)b_{n-1}\left(t\right)}{B_{n-1}\left(1\right)b_{n-1}\left(1\right)}-1\right)\frac{\sin\left(a_{n-1}\left(t\right)\Xi^{\left(n\right)}\left(t\right)\right)}{\sin\left(a_{n-1}\left(1\right)\Xi^{\left(n\right)}\left(t\right)\right)}+\\
 & \quad+B_{n-1}\left(1\right)b_{n-1}\left(1\right)a_{n-1}\left(1\right)\left(\frac{\sin\left(a_{n-1}\left(t\right)\Xi^{\left(n\right)}\left(t\right)\right)}{\sin\left(a_{n-1}\left(1\right)\Xi^{\left(n\right)}\left(t\right)\right)}-1\right).
\end{aligned}
\label{eq:I1}
\end{equation}
To continue, we need a bound for each factor.
\begin{itemize}
\item By Choice \ref{choice:Bnanbn}, we have
\begin{equation}
B_{n-1}\left(1\right)b_{n-1}\left(1\right)a_{n-1}\left(1\right)=M_{n-1}.\label{eq:I11}
\end{equation}
\item Since
\[
\begin{aligned} & B_{n-1}\left(t\right)b_{n-1}\left(t\right)-B_{n-1}\left(1\right)b_{n-1}\left(1\right)=\\
= & \frac{1}{2}\left(B_{n-1}\left(t\right)-B_{n-1}\left(1\right)\right)\left(b_{n-1}\left(t\right)+b_{n-1}\left(1\right)\right)+\frac{1}{2}\left(B_{n-1}\left(t\right)+B_{n-1}\left(1\right)\right)\left(b_{n-1}\left(t\right)-b_{n-1}\left(1\right)\right),
\end{aligned}
\]
we can write
\[
\begin{aligned}\left(\frac{B_{n-1}\left(t\right)b_{n-1}\left(t\right)}{B_{n-1}\left(1\right)b_{n-1}\left(1\right)}-1\right) & =\frac{1}{2}\left[\left(\frac{B_{n-1}\left(t\right)}{B_{n-1}\left(1\right)}-1\right)\left(\frac{b_{n-1}\left(t\right)}{b_{n-1}\left(1\right)}+1\right)+\left(\frac{B_{n-1}\left(t\right)}{B_{n-1}\left(1\right)}+1\right)\left(\frac{b_{n-1}\left(t\right)}{b_{n-1}\left(1\right)}-1\right)\right]=\\
 & =\frac{1}{2}\left[\left(\frac{B_{n-1}\left(t\right)}{B_{n-1}\left(1\right)}-1\right)\left(\frac{b_{n-1}\left(t\right)}{b_{n-1}\left(1\right)}-1+2\right)+\right.\\
 & \qquad\left.+\left(\frac{B_{n-1}\left(t\right)}{B_{n-1}\left(1\right)}-1+2\right)\left(\frac{b_{n-1}\left(t\right)}{b_{n-1}\left(1\right)}-1\right)\right].
\end{aligned}
\]
As we are considering times $t\in\left[t_{n},1\right]$, provided
that $C$ is big enough, let us say $C\ge\Upsilon_{1}\left(\beta,\delta\right)$,
Proposition \ref{prop:time convergence} allows us to bound
\begin{equation}
\begin{aligned}\left|\frac{B_{n-1}\left(t\right)b_{n-1}\left(t\right)}{B_{n-1}\left(1\right)b_{n-1}\left(1\right)}-1\right| & \lesssim\left[\overbrace{\left|\frac{B_{n-1}\left(t\right)}{B_{n-1}\left(1\right)}-1\right|}^{\leq C^{-\beta\delta\gamma\left(\frac{1}{1-\gamma}\right)^{n-1}}}\overbrace{\left(\left|\frac{b_{n-1}\left(t\right)}{b_{n-1}\left(1\right)}-1\right|+2\right)}^{\leq2}+\right.\\
 & \quad\left.+\underbrace{\left(\left|\frac{B_{n-1}\left(t\right)}{B_{n-1}\left(1\right)}-1\right|+2\right)}_{\le2}\underbrace{\left|\frac{b_{n-1}\left(t\right)}{b_{n-1}\left(1\right)}-1\right|}_{\leq C^{-\beta\delta\gamma\left(\frac{1}{1-\gamma}\right)^{n-1}}}\right]\lesssim\\
 & \lesssim C^{-\beta\delta\gamma\left(\frac{1}{1-\gamma}\right)^{n-1}}.
\end{aligned}
\label{eq:I12}
\end{equation}
\item Notice that
\[
\frac{\sin\left(a_{n-1}\left(t\right)\Xi^{\left(n\right)}\left(t\right)\right)}{\sin\left(a_{n-1}\left(1\right)\Xi^{\left(n\right)}\left(t\right)\right)}=\frac{\sin\left(a_{n-1}\left(1\right)\left(\frac{a_{n-1}\left(t\right)}{a_{n-1}\left(1\right)}-1+1\right)\Xi^{\left(n\right)}\left(t\right)\right)}{\sin\left(a_{n-1}\left(1\right)\Xi^{\left(n\right)}\left(t\right)\right)}.
\]
As $\left|\frac{a_{n-1}\left(t\right)}{a_{n-1}\left(t\right)}-1\right|\le1$
by Proposition \ref{prop:time convergence} because $C\ge\Upsilon_{1}\left(\beta,\delta\right)$,
Lemma \ref{lem:bound sines 2} assures us that, taking
\[
x\leftarrow a_{n-1}\left(1\right)\Xi^{\left(n\right)}\left(t\right),\quad\varepsilon\leftarrow\frac{a_{n-1}\left(t\right)}{a_{n-1}\left(1\right)}-1,
\]
we must have
\[
\left|\frac{\sin\left(a_{n-1}\left(t\right)\Xi^{\left(n\right)}\left(t\right)\right)}{\sin\left(a_{n-1}\left(1\right)\Xi^{\left(n\right)}\left(t\right)\right)}\right|\lesssim\frac{\left|\frac{a_{n-1}\left(t\right)}{a_{n-1}\left(1\right)}-1\right|}{\pi-a_{n-1}\left(1\right)\Xi^{\left(n\right)}\left(t\right)}
\]
as long as
\[
0<a_{n-1}\left(1\right)\Xi^{\left(n\right)}\left(t\right)<\pi.
\]
Since $C\ge\Upsilon_{1}\left(\beta,\delta\right)$, Proposition \ref{prop:time convergence}
tells us, then, that
\begin{equation}
\left|\frac{\sin\left(a_{n-1}\left(t\right)\Xi^{\left(n\right)}\left(t\right)\right)}{\sin\left(a_{n-1}\left(1\right)\Xi^{\left(n\right)}\left(t\right)\right)}\right|\leq\frac{C^{-\beta\delta\gamma\left(\frac{1}{1-\gamma}\right)^{n-1}}}{\pi-a_{n-1}\left(1\right)\Xi^{\left(n\right)}\left(t\right)}.\label{eq:I13}
\end{equation}
\end{itemize}
Combining \eqref{eq:I11}, \eqref{eq:I12} and \eqref{eq:I13} in
\eqref{eq:I1} leads to
\[
\begin{aligned}\left|I_{1}\right| & \lesssim M_{n-1}C^{-\beta\delta\gamma\left(\frac{1}{1-\gamma}\right)^{n-1}}\left(1+\frac{C^{-\beta\delta\gamma\left(\frac{1}{1-\gamma}\right)^{n-1}}}{\pi-a_{n-1}\left(1\right)\Xi^{\left(n\right)}\left(t\right)}\right)+\\
 & \quad+M_{n-1}\frac{C^{-\beta\delta\gamma\left(\frac{1}{1-\gamma}\right)^{n-1}}}{\pi-a_{n-1}\left(1\right)\Xi^{\left(n\right)}\left(t\right)}.
\end{aligned}
\]
This expression is valid for the subset of times
\begin{equation}
\mathcal{T}_{n}\coloneqq\left\{ t\in\left[t_{n},1\right]:a_{n-1}\left(1\right)\Xi^{\left(n\right)}\left(t\right)<\pi\right\} .\label{eq:def bigTn}
\end{equation}
As $C^{-\beta\delta\gamma\left(\frac{1}{1-\gamma}\right)^{n-1}}\le1$
$\forall n\in\mathbb{N}$, we have
\[
\left(C^{-\beta\delta\gamma\left(\frac{1}{1-\gamma}\right)^{n-1}}\right)^{2}\leq C^{-\beta\delta\gamma\left(\frac{1}{1-\gamma}\right)^{n-1}}
\]
and, as the function $\frac{1}{\pi-x}$ is bounded below in the interval
$x\in\left(0,\pi\right)$, we may conclude that
\begin{equation}
\left|I_{1}\right|\lesssim M_{n-1}\frac{C^{-\beta\delta\gamma\left(\frac{1}{1-\gamma}\right)^{n-1}}}{\pi-a_{n-1}\left(1\right)\Xi^{\left(n\right)}\left(t\right)}\quad\forall t\in\mathcal{T}_{n}.\label{eq:the bound for I1 v0}
\end{equation}

\item $I_{2}$. Since $\sin\left(\cdot\right)$ is a Lipschitz function
of constant $1$, we can write
\[
\begin{aligned}\left|I_{2}\right| & \le\frac{a_{n-1}\left(1\right)}{\sin\left(a_{n-1}\left(1\right)\Xi^{\left(n\right)}\left(t\right)\right)}\sum_{m=1}^{n-2}B_{m}\left(t\right)b_{m}\left(t\right)a_{m}\left(t\right)\left|\overbrace{\phi_{1}^{\left(n\right)}\left(t,0\right)-\phi_{1}^{\left(n-1\right)}\left(t\right)}^{=\Xi^{\left(n\right)}\left(t\right)}\right|\le\\
 & \le\frac{a_{n-1}\left(1\right)\Xi^{\left(n\right)}\left(t\right)}{\sin\left(a_{n-1}\left(1\right)\Xi^{\left(n\right)}\left(t\right)\right)}\sum_{m=1}^{n-2}B_{m}\left(t\right)b_{m}\left(t\right)a_{m}\left(t\right)\quad\forall t\in\mathcal{T}_{n}.
\end{aligned}
\]
By Choice \ref{choice:anbncn} and Proposition \ref{prop:time convergence},
which we can apply because $C\ge\Upsilon_{1}\left(\beta,\delta\right)$
and because we are considering times $t\in\mathcal{T}_{n}\subseteq\left[t_{n},1\right]$,
repeating the same argument we presented in the proof of Proposition
\ref{prop:time convergence} to obtain $b_{m}\left(t\right)\le2b_{m}\left(1\right)$,
but applied to $B_{m}\left(t\right)$ instead of $b_{m}\left(t\right)$,
we arrive to
\[
\begin{aligned}\left|I_{2}\right| & \le\frac{a_{n-1}\left(1\right)\Xi^{\left(n\right)}\left(t\right)}{\sin\left(a_{n-1}\left(1\right)\Xi^{\left(n\right)}\left(t\right)\right)}\sum_{m=1}^{n-2}B_{m}\left(t\right)b_{m}\left(1\right)a_{m}\left(1\right)\lesssim\\
 & \lesssim\frac{a_{n-1}\left(1\right)\Xi^{\left(n\right)}\left(t\right)}{\sin\left(a_{n-1}\left(1\right)\Xi^{\left(n\right)}\left(t\right)\right)}\sum_{m=1}^{n-2}B_{m}\left(1\right)b_{m}\left(1\right)a_{m}\left(1\right)\quad\forall t\in\mathcal{T}_{n}.
\end{aligned}
\]
By Lemma \ref{lem:equivalence sin with x(x-pi)} and Choices \ref{choice:Bnanbn}
and \ref{choice:Mn}, we have
\[
\left|I_{2}\right|\lesssim\frac{1}{\pi-a_{n-1}\left(1\right)\Xi^{\left(n\right)}\left(t\right)}\sum_{m=1}^{n-2}YC^{\delta\left(\frac{1}{1-\gamma}\right)^{m}}\quad\forall t\in\mathcal{T}_{n}.
\]
By virtue of Lemma \ref{lem:estimate sum superexponential} (which
we can apply thanks to Choice \ref{choice:min requirements on C gamma and kmax})
and Choice \ref{choice:Mn}, we obtain
\begin{equation}
\left|I_{2}\right|\lesssim_{\delta}\frac{1}{\pi-a_{n-1}\left(1\right)\Xi^{\left(n\right)}\left(t\right)}YC^{\delta\left(\frac{1}{1-\gamma}\right)^{n-2}}=\frac{M_{n-2}}{\pi-a_{n-1}\left(1\right)\Xi^{\left(n\right)}\left(t\right)}\quad\forall t\in\mathcal{T}_{n}.\label{eq:the bound for I2}
\end{equation}
\end{itemize}
Unifying \eqref{eq:the bound for I1 v0} and \eqref{eq:the bound for I2}
back into \eqref{eq:time derivative F} leads to
\begin{equation}
\begin{aligned} & \left|\frac{\mathrm{d}}{\mathrm{d}t}\left(F\left(a_{n-1}\left(1\right)\Xi^{\left(n\right)}\left(t\right)\right)-F\left(a_{n-1}\left(1\right)\Xi_{0}^{\left(n\right)}\left(t\right)\right)\right)\right|\lesssim_{\delta}\\
\lesssim_{\delta} & \frac{1}{\pi-a_{n-1}\left(1\right)\Xi^{\left(n\right)}\left(t\right)}\left[M_{n-1}C^{-\beta\delta\gamma\left(\frac{1}{1-\gamma}\right)^{n-1}}+M_{n-2}\right]\quad\forall t\in\mathcal{T}_{n}.
\end{aligned}
\label{eq:bound time derivative primitive 1/sin(x)}
\end{equation}
Let $\mathcal{I}_{n}\subseteq\mathcal{T}_{n}\subseteq\left[t_{n},1\right]$
be the largest interval such that $t_{n}\in\mathcal{I}_{n}$ and
\begin{equation}
\frac{1}{\pi-a_{n-1}\left(1\right)\Xi^{\left(n\right)}\left(t\right)}\le K_{n}\quad\forall t\in\mathcal{I}_{n},\label{eq:Kn}
\end{equation}
where $\left(K_{n}\right)_{n\in\mathbb{N}}\subseteq\left(0,\infty\right)$.
In this manner, equation \eqref{eq:bound time derivative primitive 1/sin(x)}
becomes
\begin{equation}
\begin{aligned} & \left|\frac{\mathrm{d}}{\mathrm{d}t}\left(F\left(a_{n-1}\left(1\right)\Xi^{\left(n\right)}\left(t\right)\right)-F\left(a_{n-1}\left(1\right)\Xi_{0}^{\left(n\right)}\left(t\right)\right)\right)\right|\lesssim_{\delta}\\
\lesssim_{\delta} & K_{n}\left[M_{n-1}C^{-\beta\delta\gamma\left(\frac{1}{1-\gamma}\right)^{n-1}}+M_{n-2}\right]\quad\forall t\in\mathcal{I}_{n}.
\end{aligned}
\label{eq:bound time derivative primitive 1/sin(x) v2}
\end{equation}
Then, integrating in time at both sides of equation \eqref{eq:bound time derivative primitive 1/sin(x) v2},
we arrive to
\[
\begin{aligned} & \left|F\left(a_{n-1}\left(1\right)\Xi^{\left(n\right)}\left(t\right)\right)-F\left(a_{n-1}\left(1\right)\Xi_{0}^{\left(n\right)}\left(t\right)\right)\right|+\\
 & \quad-\overbrace{\left|F\left(a_{n-1}\left(1\right)\Xi^{\left(n\right)}\left(t_{n}\right)\right)-F\left(a_{n-1}\left(1\right)\Xi_{0}^{\left(n\right)}\left(t_{n}\right)\right)\right|}^{=0}\lesssim_{\delta}\\
\lesssim_{\delta} & \left(t-t_{n}\right)K_{n}\left[M_{n-1}C^{-\beta\delta\gamma\left(\frac{1}{1-\gamma}\right)^{n-1}}+M_{n-2}\right]\leq\left(1-t_{n}\right)K_{n}\left[M_{n-1}C^{-\beta\delta\gamma\left(\frac{1}{1-\gamma}\right)^{n-1}}+M_{n-2}\right]
\end{aligned}
\]
$\forall t\in\mathcal{I}_{n}$. Choice \ref{choice:Mn} implies that
\[
\begin{aligned} & \left|F\left(a_{n-1}\left(1\right)\Xi^{\left(n\right)}\left(t\right)\right)-F\left(a_{n-1}\left(1\right)\Xi_{0}^{\left(n\right)}\left(t\right)\right)\right|\lesssim_{\delta}\\
\lesssim_{\delta} & K_{n}\mathrm{arccosh}\left(C^{k_{\max}\left(\frac{1}{1-\gamma}\right)^{n}}\right)\left[C^{-\beta\delta\gamma\left(\frac{1}{1-\gamma}\right)^{n-1}}+\frac{M_{n-2}}{M_{n-1}}\right]=\\
\lesssim_{\delta} & K_{n}\mathrm{arccosh}\left(C^{k_{\max}\left(\frac{1}{1-\gamma}\right)^{n}}\right)\left[C^{-\beta\delta\gamma\left(\frac{1}{1-\gamma}\right)^{n-1}}+C^{\delta\left(\frac{1}{1-\gamma}\right)^{n-2}}C^{-\delta\left(\frac{1}{1-\gamma}\right)^{n-1}}\right]=\\
\lesssim_{\delta} & K_{n}\mathrm{arccosh}\left(C^{k_{\max}\left(\frac{1}{1-\gamma}\right)^{n}}\right)\left[C^{-\beta\delta\gamma\left(\frac{1}{1-\gamma}\right)^{n-1}}+C^{-\delta\left(\frac{1}{1-\gamma}\right)^{n-1}\left[1-\left(1-\gamma\right)\right]}\right]=\\
\lesssim_{\delta} & K_{n}\mathrm{arccosh}\left(C^{k_{\max}\left(\frac{1}{1-\gamma}\right)^{n}}\right)\left[C^{-\beta\delta\gamma\left(\frac{1}{1-\gamma}\right)^{n-1}}+C^{-\delta\gamma\left(\frac{1}{1-\gamma}\right)^{n-1}}\right]\quad\forall t\in\mathcal{I}_{n}.
\end{aligned}
\]
On the other hand, Lemma \ref{lem:arccosh by ln} guarantees that
\[
\begin{aligned} & \left|F\left(a_{n-1}\left(1\right)\Xi^{\left(n\right)}\left(t\right)\right)-F\left(a_{n-1}\left(1\right)\Xi_{0}^{\left(n\right)}\left(t\right)\right)\right|\lesssim_{\delta}\\
\lesssim_{\delta} & K_{n}\left(\ln\left(2\right)+\ln\left(C\right)k_{\max}\left(\frac{1}{1-\gamma}\right)^{n}\right)\left[C^{-\beta\delta\gamma\left(\frac{1}{1-\gamma}\right)^{n-1}}+C^{-\delta\gamma\left(\frac{1}{1-\gamma}\right)^{n-1}}\right]\quad\forall t\in\mathcal{I}_{n}.
\end{aligned}
\]
Arguing like in the proof of Proposition \ref{prop:time convergence}
from equation \eqref{eq:time convergence the log 1} to equation \eqref{eq:time convergence the log 2},
we deduce that
\begin{equation}
\begin{aligned} & \left|F\left(a_{n-1}\left(1\right)\Xi^{\left(n\right)}\left(t\right)\right)-F\left(a_{n-1}\left(1\right)\Xi_{0}^{\left(n\right)}\left(t\right)\right)\right|\lesssim_{\delta}\\
\lesssim_{\delta} & K_{n}\left(\frac{1}{1-\gamma}\right)^{n}\ln\left(C\right)\left[C^{-\beta\delta\gamma\left(\frac{1}{1-\gamma}\right)^{n-1}}+C^{-\delta\gamma\left(\frac{1}{1-\gamma}\right)^{n-1}}\right]\quad\forall t\in\mathcal{I}_{n}.
\end{aligned}
\label{eq:bound Fn v0}
\end{equation}
As $\beta<1$, we can bound
\[
C^{-\delta\gamma\left(\frac{1}{1-\gamma}\right)^{n-1}}\le C^{-\beta\delta\gamma\left(\frac{1}{1-\gamma}\right)^{n-1}}.
\]
Hence, we may transform \eqref{eq:bound Fn v0} into
\[
\left|F\left(a_{n-1}\left(1\right)\Xi^{\left(n\right)}\left(t\right)\right)-F\left(a_{n-1}\left(1\right)\Xi_{0}^{\left(n\right)}\left(t\right)\right)\right|\lesssim_{\delta}K_{n}\left(\frac{1}{1-\gamma}\right)^{n}\ln\left(C\right)C^{-\beta\delta\gamma\left(\frac{1}{1-\gamma}\right)^{n-1}}\quad\forall t\in\mathcal{I}_{n}.
\]
Considering that $\frac{1}{1-\gamma}\ge2$ by Choice \ref{choice:min requirements on C gamma and kmax},
we may apply Lemma \ref{lem:exponetial superexponential bound} taking
$a\leftarrow\frac{1}{1-\gamma}$, $b\leftarrow\beta\delta\gamma$,
$n\leftarrow n-1$ and $\beta\leftarrow\beta'$, achieving
\[
\left|F\left(a_{n-1}\left(1\right)\Xi^{\left(n\right)}\left(t\right)\right)-F\left(a_{n-1}\left(1\right)\Xi_{0}^{\left(n\right)}\left(t\right)\right)\right|\lesssim_{\delta}K_{n}\ln\left(C\right)\frac{4\mathrm{e}^{-2}}{\left(1-\beta'\right)^{2}\left(\beta\gamma\delta\right)^{2}\ln\left(C\right)^{2}}C^{-\beta\beta'\delta\gamma\left(\frac{1}{1-\gamma}\right)^{n-1}}\quad\forall t\in\mathcal{I}_{n}.
\]
Since $\gamma\ge\frac{1}{2}$ by Choice \ref{choice:min requirements on C gamma and kmax},
we deduce that
\begin{equation}
\begin{aligned}\left|F\left(a_{n-1}\left(1\right)\Xi^{\left(n\right)}\left(t\right)\right)-F\left(a_{n-1}\left(1\right)\Xi_{0}^{\left(n\right)}\left(t\right)\right)\right| & \lesssim_{\delta}K_{n}\frac{4e^{-2}}{\left(1-\beta'\right)^{2}\beta^{2}\delta^{2}\left(\frac{1}{2}\right)^{2}\ln\left(C\right)}C^{-\beta\beta'\delta\gamma\left(\frac{1}{1-\gamma}\right)^{n-1}}\lesssim_{\beta,\beta',\delta}\\
 & \lesssim_{\beta,\beta',\delta}K_{n}\frac{C^{-\beta\beta'\delta\gamma\left(\frac{1}{1-\gamma}\right)^{n-1}}}{\ln\left(C\right)}\quad\forall t\in\mathcal{I}_{n}.
\end{aligned}
\label{eq:bound Fn v2}
\end{equation}

Let us think a little bit about how we could choose $K_{n}$. Thanks
to Choices \ref{choice:wnmax and Ji(tn)} and \ref{choice:Mn}, we
are in the second to last case of Remark \ref{rem:relation delta M}.
This means that $\sin\left(a_{n-1}\left(1\right)\Xi_{0}^{\left(n\right)}\left(t\right)\right)$
is increasing $\forall t\in\left[t_{n},t_{n}+\frac{1-t_{n}}{2}\right]$
and achieves its maximum value $1$ at $t_{n}+\frac{1-t_{n}}{2}$.
Thanks to Choice \ref{choice:wnmax and Ji(tn)}, we can ensure that
$a_{n-1}\left(1\right)\Xi_{0}^{\left(n\right)}\left(t\right)\le\frac{\pi}{2}$
$\forall t\in\left[t_{n},t_{n}+\frac{1-t_{n}}{2}\right]$. If $\Xi^{\left(n\right)}\left(t\right)$
and $\Xi_{0}^{\left(n\right)}\left(t\right)$ are close, the same
thing should happen to $\Xi^{\left(n\right)}\left(t\right)$. Then,
taking a look at equation \eqref{eq:Kn}, we should be able to choose
$K_{n}=\frac{1}{\frac{\pi}{2}-\varepsilon}$ for some $\varepsilon>0$
small. Under this choice, \eqref{eq:bound Fn v2} becomes
\begin{equation}
\begin{aligned} & \left|F\left(a_{n-1}\left(1\right)\Xi^{\left(n\right)}\left(t\right)\right)-F\left(a_{n-1}\left(1\right)\Xi_{0}^{\left(n\right)}\left(t\right)\right)\right|\lesssim_{\beta,\beta',\delta,\varepsilon}\\
\lesssim_{\beta,\beta',\delta,\varepsilon} & \frac{C^{-\beta\beta'\delta\gamma\left(\frac{1}{1-\gamma}\right)^{n-1}}}{\ln\left(C\right)}\quad\forall t\in\left[t_{n},t_{n}+\frac{1-t_{n}}{2}\right].
\end{aligned}
\label{eq:bound first zone}
\end{equation}
We can compensate the constant that appears in the inequality above
via the factor $\ln\left(C\right)$ by taking $C$ big enough. Let
us say that we need $C\ge\Upsilon_{2}\left(\beta,\beta',\delta,\varepsilon\right)$
for the constant to be $\frac{1}{2}$ (or smaller). Then,
\[
\left|F\left(a_{n-1}\left(1\right)\Xi^{\left(n\right)}\left(t\right)\right)-F\left(a_{n-1}\left(1\right)\Xi_{0}^{\left(n\right)}\left(t\right)\right)\right|\le\frac{1}{2}C^{-\beta\beta'\delta\gamma\left(\frac{1}{1-\gamma}\right)^{n-1}}.
\]
As $\frac{1}{2}C^{-\beta\beta'\delta\gamma\left(\frac{1}{1-\gamma}\right)^{n-1}}\le\frac{1}{2}$,
we may apply Lemma \ref{lem:inverse of primitive of 1/sin(x)}, which
leads us to
\begin{equation}
a_{n-1}\left(1\right)\left|\Xi^{\left(n\right)}\left(t\right)-\Xi_{0}^{\left(n\right)}\left(t\right)\right|\leq C^{-\beta\beta'\delta\gamma\left(\frac{1}{1-\gamma}\right)^{n-1}}\sin\left(a_{n-1}\left(1\right)\Xi_{0}^{\left(n\right)}\left(t\right)\right)\quad\forall t\in\left[t_{n},t_{n}+\frac{1-t_{n}}{2}\right].\label{eq:bound first zone v1}
\end{equation}
This proves the first statement, provided that our choice of $K_{n}$
is consistent. Notice that it is indeed the case, since, thanks to
equation \eqref{eq:bound first zone v1} and the bound
\[
C^{-\beta\beta'\delta\gamma\left(\frac{1}{1-\gamma}\right)^{n-1}}\sin\left(a_{n-1}\left(1\right)\Xi_{0}^{\left(n\right)}\left(t\right)\right)\leq C^{-\beta\beta'\delta\frac{\gamma}{1-\gamma}}\le C^{-\beta\beta'\delta\frac{\frac{1}{2}}{1-\frac{1}{2}}}\le C^{-\beta\beta'\delta},
\]
which is satisfied because $n\ge2$, $\gamma\ge\frac{1}{2}$ (by Choice
\ref{choice:min requirements on C gamma and kmax}) and the function
$x\to\frac{x}{1-x}$ is increasing, taking $C$ big enough (let us
say $C\ge\Upsilon_{3}\left(\beta,\beta',\delta,\varepsilon\right)$),
we can assure that
\[
a_{n-1}\left(1\right)\left|\Xi^{\left(n\right)}\left(t\right)-\Xi_{0}^{\left(n\right)}\left(t\right)\right|\leq\varepsilon\quad\forall t\in\left[t_{n},t_{n}+\frac{1-t_{n}}{2}\right]\;\land\;\forall n\in\mathbb{N}
\]
and, as a consequence,
\[
\frac{1}{\pi-a_{n-1}\left(1\right)\Xi^{\left(n\right)}\left(t\right)}\le\frac{1}{\pi-\underbrace{a_{n-1}\left(1\right)\Xi_{0}^{\left(n\right)}\left(t\right)}_{\le\frac{\pi}{2}}-\varepsilon}\le\frac{1}{\pi-\frac{\pi}{2}-\varepsilon}\le\frac{1}{\frac{\pi}{2}-\varepsilon},
\]
which is exactly the choice of $K_{n}$ we have made.

Now, consider times $t\in\left[t_{n}+\frac{1-t_{n}}{2},1\right]\cap\mathcal{I}_{n}\eqqcolon\mathcal{I}_{n}^{\rightarrow}$
in \eqref{eq:bound Fn v2}. Taking a look at equation \eqref{eq:bound Fn v2},
it is clear that, to go beyond $t_{n}+\frac{1-t_{n}}{2}$, we will
have to take $K_{n}$ bigger than $\frac{1}{\frac{\pi}{2}-\varepsilon}$.
In fact, let us try to take $K_{n}$ of the same order than $C^{\beta\beta'\delta\gamma\left(\frac{1}{1-\gamma}\right)^{n-1}}$,
more precisely, we shall take
\[
K_{n}=C^{\beta\beta'\beta''\delta\gamma\left(\frac{1}{1-\gamma}\right)^{n-1}},
\]
where we recall that $\beta''\in\left(0,1\right)$. Then, \eqref{eq:bound Fn v2}
becomes
\[
\left|F\left(a_{n-1}\left(1\right)\Xi^{\left(n\right)}\left(t\right)\right)-F\left(a_{n-1}\left(1\right)\Xi_{0}^{\left(n\right)}\left(t\right)\right)\right|\lesssim_{\beta,\beta',\delta}\frac{C^{-\beta\beta'\left(1-\beta''\right)\delta\gamma\left(\frac{1}{1-\gamma}\right)^{n-1}}}{\ln\left(C\right)^ {}}\quad\forall t\in\mathcal{I}_{n}^{\rightarrow}.
\]
Again, we can compensate the constant of the inequality above with
the factor $\ln\left(C\right)$ choosing $C$ big enough. Assume $C\ge\Upsilon_{4}\left(\beta,\beta',\delta\right)$
is required for the constant to be $\frac{1}{2}$ (or smaller). Then,
\[
\left|F\left(a_{n-1}\left(1\right)\Xi^{\left(n\right)}\left(t\right)\right)-F\left(a_{n-1}\left(1\right)\Xi_{0}^{\left(n\right)}\left(t\right)\right)\right|\leq\frac{1}{2}C^{-\beta\beta'\left(1-\beta''\right)\delta\gamma\left(\frac{1}{1-\gamma}\right)^{n-1}}.
\]
As $\frac{1}{2}C^{-\beta\beta'\left(1-\beta''\right)\delta\gamma\left(\frac{1}{1-\gamma}\right)^{n-1}}\leq\frac{1}{2}$,
we may apply Lemma \ref{lem:inverse of primitive of 1/sin(x)}, which
leads us to
\begin{equation}
a_{n-1}\left(1\right)\left|\Xi^{\left(n\right)}\left(t\right)-\Xi_{0}^{\left(n\right)}\left(t\right)\right|\leq C^{-\beta\beta'\left(1-\beta''\right)\delta\gamma\left(\frac{1}{1-\gamma}\right)^{n-1}}\sin\left(a_{n-1}\left(1\right)\Xi_{0}^{\left(n\right)}\left(t\right)\right)\quad\forall t\in\mathcal{I}_{n}^{\rightarrow}.\label{eq:bound intermidiate}
\end{equation}
In other words, taking $C$ big enough, we can ensure that $\Xi^{\left(n\right)}\left(t\right)$
and $\Xi_{0}^{\left(n\right)}\left(t\right)$ remain as close as we
want, but only for $t\in\mathcal{I}_{n}^{\rightarrow}$ (recall that
$\mathcal{I}_{n}^{\rightarrow}$ is defined in terms of $K_{n}$ through
equation \eqref{eq:Kn}). Which interval is $\mathcal{I}_{n}^{\rightarrow}$?
Can we say something about it? On the one hand, it is clear, by definition,
that it must begin at $t=t_{n}+\frac{1-t_{n}}{2}$. On the other hand,
we do not know when it ends. Nonetheless, we can distinguish two cases:
$1\in\mathcal{I}_{n}^{\rightarrow}$ and $1\notin\mathcal{I}_{n}^{\rightarrow}$.
In the first one, we take $t_{n}^{*}=1$. In the second one, there
is $t_{n}^{*}\in\left[t_{n}+\frac{1-t_{n}}{2},1\right)$ such that
$\mathcal{I}_{n}^{\rightarrow}=\left[t_{n}+\frac{1-t_{n}}{2},t_{n}^{*}\right]$
(notice that we know $\mathcal{I}_{n}^{\rightarrow}$ must be closed
because of the non-strict inequality that appears in its definition
{[}see equation \eqref{eq:Kn}{]}). In any case, this proves the second
statement.

Now, only the proof of the third statement remains. This will require
a different approach. Our objective is to make use of Lemma \ref{lem:trap}.
If $t_{n}^{*}=1$, there is nothing to prove. Assume that $t_{n}^{*}\neq1$.
By the definitions of $t_{n}^{*}$ and $\mathcal{I}_{n}^{\rightarrow}$
and equation \eqref{eq:Kn}, we know that
\begin{equation}
\pi-a_{n-1}\left(1\right)\Xi^{\left(n\right)}\left(t_{n}^{*}\right)=C^{-\beta\beta'\beta''\delta\gamma\left(\frac{1}{1-\gamma}\right)^{n-1}}.\label{eq:third argument closeness to pi}
\end{equation}
What can we say about $\frac{\mathrm{d}}{\mathrm{d}t}\left(a_{n-1}\left(1\right)\Xi^{\left(n\right)}\left(t_{n}^{*}\right)\right)$?
By equation \eqref{eq:ODE JIn}, we have
\begin{equation}
\begin{aligned}\frac{\mathrm{d}}{\mathrm{d}t}\left(a_{n-1}\left(1\right)\Xi^{\left(n\right)}\left(t\right)\right)\left(t\right) & =B_{n-1}\left(t\right)a_{n-1}\left(1\right)b_{n-1}\left(t\right)\sin\left(a_{n-1}\left(t\right)\Xi^{\left(n\right)}\left(t\right)\right)+\\
 & \quad+a_{n-1}\left(1\right)\sum_{m=1}^{n-2}B_{m}\left(t\right)b_{m}\left(t\right)\left[\sin\left(a_{m}\left(t\right)\left(\phi_{1}^{\left(n\right)}\left(t,0\right)-\phi_{1}^{\left(m\right)}\left(t,0\right)\right)\right)+\right.\\
 & \qquad\left.-\sin\left(a_{m}\left(t\right)\left(\phi_{1}^{\left(n-1\right)}\left(t,0\right)-\phi_{1}^{\left(m\right)}\left(t,0\right)\right)\right)\right]\quad\forall t\in\left[t_{n}^{*},1\right].
\end{aligned}
\label{eq:derivative an-1Jin}
\end{equation}
We will call
\begin{equation}
\text{LT}\eqref{eq:derivative an-1Jin}\left(t\right)\coloneqq B_{n-1}\left(t\right)a_{n-1}\left(1\right)b_{n-1}\left(t\right)\sin\left(a_{n-1}\left(t\right)\Xi^{\left(n\right)}\left(t\right)\right)\label{eq:third argument definition leading term}
\end{equation}
and we will denote the other summand that appears in equation \eqref{eq:derivative an-1Jin}
by $\text{OT}\eqref{eq:derivative an-1Jin}\left(t\right)$. LT stands
for ``leading term'' and OT stands for ``other term''. We want
to prove that $\text{LT}\eqref{eq:derivative an-1Jin}\left(t\right)$
is the dominant term when $a_{n-1}\left(1\right)\Xi^{\left(n\right)}\left(t\right)=a_{n-1}\left(1\right)\Xi^{\left(n\right)}\left(t_{n}^{*}\right)$.
If it were the case, since we expect $a_{n-1}\left(t_{n}^{*}\right)\approx a_{n-1}\left(1\right)$
by Proposition \ref{prop:time convergence}, we should have $\sin\left(a_{n-1}\left(t\right)\Xi^{\left(n\right)}\left(t\right)\right)>0$
because $a_{n-1}\left(t\right)\Xi^{\left(n\right)}\left(t\right)\approx a_{n-1}\left(1\right)\Xi^{\left(n\right)}\left(t\right)=a_{n-1}\left(1\right)\Xi^{\left(n\right)}\left(t_{n}^{*}\right)$
and $\frac{\pi}{2}\le a_{n-1}\left(1\right)\Xi^{\left(n\right)}\left(t_{n}^{*}\right)<\pi$
by definition of $\mathcal{I}_{n}^{\rightarrow}$ and equation \eqref{eq:Kn},
which would prove that $\frac{\mathrm{d}}{\mathrm{d}t}\left(a_{n-1}\left(1\right)\Xi^{\left(n\right)}\left(t\right)\right)\left(t\right)>0$.
Moreover, to apply Lemma \ref{lem:trap}, we wish to see that, if
$a_{n-1}\left(1\right)\Xi^{\left(n\right)}\left(t\right)$ is greater
than $\pi$, assuming again $a_{n-1}\left(t\right)\approx a_{n-1}\left(1\right)$,
the leading term becomes negative and, consequently, $\frac{\mathrm{d}}{\mathrm{d}t}\left(a_{n-1}\left(1\right)\Xi^{\left(n\right)}\left(t\right)\right)\left(t\right)\leq\text{OT}\eqref{eq:derivative an-1Jin}\left(t\right)$,
which will be small compared to the time scale $1-t_{n}^{*}$. To
prove all this rigorously, we will need multiple steps.
\begin{enumerate}
\item \label{enu:uniform bounds}Uniform bounds between $a_{n-1}\left(1\right)\Xi^{\left(n\right)}\left(t\right)$
and $a_{n-1}\left(t\right)\Xi^{\left(n\right)}\left(t\right)$. By
Proposition \ref{prop:time convergence} (which we can apply because
$C\ge\Upsilon_{1}\left(\beta,\delta\right)$), we know that, $\forall t\in\left[t_{n}^{*},1\right]$,
\begin{equation}
\frac{\left|a_{n-1}\left(t\right)\Xi^{\left(n\right)}\left(t\right)-a_{n-1}\left(1\right)\Xi^{\left(n\right)}\left(t\right)\right|}{a_{n-1}\left(1\right)\Xi^{\left(n\right)}\left(t\right)}\le\left|\frac{a_{n-1}\left(t\right)}{a_{n-1}\left(1\right)}-1\right|\le C^{-\beta\delta\gamma\left(\frac{1}{1-\gamma}\right)^{n-1}}.\label{eq:third argument uniform bounds an-1}
\end{equation}
\item Bound for $a_{n-1}\left(t_{n}^{*}\right)\Xi^{\left(n\right)}\left(t_{n}^{*}\right)$.
By equations \eqref{eq:third argument closeness to pi} and \eqref{eq:third argument uniform bounds an-1},
we deduce that
\begin{equation}
\begin{aligned}a_{n-1}\left(t_{n}^{*}\right)\Xi^{\left(n\right)}\left(t_{n}^{*}\right) & \le a_{n-1}\left(1\right)\Xi^{\left(n\right)}\left(t_{n}^{*}\right)\left(1+C^{-\beta\delta\gamma\left(\frac{1}{1-\gamma}\right)^{n-1}}\right)=\\
 & \leq\left(\pi-C^{-\beta\beta'\beta''\delta\gamma\left(\frac{1}{1-\gamma}\right)^{n-1}}\right)\left(1+C^{-\beta\delta\gamma\left(\frac{1}{1-\gamma}\right)^{n-1}}\right),\\
a_{n-1}\left(t_{n}^{*}\right)\Xi^{\left(n\right)}\left(t_{n}^{*}\right) & \ge a_{n-1}\left(1\right)\Xi^{\left(n\right)}\left(t_{n}^{*}\right)\left(1-C^{-\beta\delta\gamma\left(\frac{1}{1-\gamma}\right)^{n-1}}\right)=\\
 & \ge\left(\pi-C^{-\beta\beta'\beta''\delta\gamma\left(\frac{1}{1-\gamma}\right)^{n-1}}\right)\left(1-C^{-\beta\delta\gamma\left(\frac{1}{1-\gamma}\right)^{n-1}}\right).
\end{aligned}
\label{eq:third argument bounds time tn*}
\end{equation}
Thereby, as we have the bounds
\begin{equation}
\begin{aligned}C^{-\beta\beta'\beta''\delta\gamma\left(\frac{1}{1-\gamma}\right)^{n-1}} & \le C^{-\beta\beta'\beta''\delta\frac{\gamma}{1-\gamma}}\le C^{-\beta\beta'\beta''\delta},\\
C^{-\beta\delta\gamma\left(\frac{1}{1-\gamma}\right)^{n-1}} & \le C^{-\beta\delta\frac{\gamma}{1-\gamma}}\le C^{-\beta\delta}
\end{aligned}
\label{eq:third argument bounds}
\end{equation}
because $n\ge2$, $\gamma\ge\frac{1}{2}$ by Choice \ref{choice:min requirements on C gamma and kmax}
and the function $x\to\frac{x}{1-x}$ is increasing in $\left(0,1\right)$,
thanks to the lower bound given in \eqref{eq:third argument bounds time tn*},
choosing $C$ big enough, let us say $C\ge\Upsilon_{4}\left(\beta,\beta',\beta",\delta\right)$,
we can easily ensure that
\[
a_{n-1}\left(t_{n}^{*}\right)\Xi^{\left(n\right)}\left(t_{n}^{*}\right)\ge\frac{\pi}{2}.
\]
Furthermore, taking a look at the upper bound given in \eqref{eq:third argument bounds time tn*}
and expanding the terms that appear, we see that
\[
\begin{aligned}a_{n-1}\left(t_{n}^{*}\right)\Xi^{\left(n\right)}\left(t_{n}^{*}\right) & \le\pi-C^{-\beta\beta'\beta''\delta\gamma\left(\frac{1}{1-\gamma}\right)^{n-1}}+\pi C^{-\beta\delta\gamma\left(\frac{1}{1-\gamma}\right)^{n-1}}=\\
 & \leq\pi-C^{-\beta\beta'\beta''\delta\gamma\left(\frac{1}{1-\gamma}\right)^{n-1}}\left(1-\pi C^{-\left[1-\beta'\beta''\right]\beta\delta\gamma\left(\frac{1}{1-\gamma}\right)^{n-1}}\right).
\end{aligned}
\]
Since we have the bound
\begin{equation}
C^{-\left[1-\beta'\beta''\right]\beta\delta\gamma\left(\frac{1}{1-\gamma}\right)^{n-1}}\le C^{-\left[1-\beta\beta''\right]\beta\delta\frac{\gamma}{1-\gamma}}\le C^{-\left[1-\beta\beta''\right]\beta\delta}\label{eq:third argument bound 1-betabeta thingy}
\end{equation}
because $n\ge2$, $\gamma\ge\frac{1}{2}$ by Choice \ref{choice:min requirements on C gamma and kmax}
and the function $x\to\frac{x}{1-x}$ is increasing in $\left(0,1\right)$,
and $1-\beta'\beta''>0$, taking $C$ large enough (let us say $C\ge\Upsilon_{5}\left(\beta,\beta',\beta'',\delta\right)$),
we can assure that
\[
a_{n-1}\left(t_{n}^{*}\right)\Xi^{\left(n\right)}\left(t_{n}^{*}\right)\leq\pi-\frac{1}{2}C^{-\beta\beta'\beta''\delta\gamma\left(\frac{1}{1-\gamma}\right)^{n-1}}.
\]
In other words, we have proven that
\begin{equation}
\frac{\pi}{2}\leq a_{n-1}\left(t_{n}^{*}\right)\Xi^{\left(n\right)}\left(t_{n}^{*}\right)\leq\pi-\frac{1}{2}C^{-\beta\beta'\beta''\delta\gamma\left(\frac{1}{1-\gamma}\right)^{n-1}}.\label{eq:third argument bounds modified anjin tn*}
\end{equation}
\item The time intervals $\mathcal{J}_{n}$. Let $\mathcal{J}_{n}$ be the
largest time interval that begins at $t_{n}^{*}$, is contained in
$\left[t_{n}^{*},1\right]$ and satisfies 
\begin{equation}
a_{n-1}\left(t\right)\Xi^{\left(n\right)}\left(t\right),a_{n-1}\left(1\right)\Xi^{\left(n\right)}\left(t\right)\leq\frac{3\pi}{2}\quad\forall t\in\mathcal{J}_{n}.\label{eq:third argument defining feature of Jn}
\end{equation}
By equations \eqref{eq:third argument closeness to pi} and \eqref{eq:third argument bounds modified anjin tn*}
and the continuity of both $a_{n-1}\left(t\right)\Xi^{\left(n\right)}\left(t\right)$
and $a_{n-1}\left(1\right)\Xi^{\left(n\right)}\left(t\right)$, we
deduce that the interior of $\mathcal{J}_{n}$ is non-empty. In fact,
we will see that, actually, taking $C$ appropriately, we will have
$\mathcal{J}_{n}=\left[t_{n}^{*},1\right]$.
\item Estimates for $\text{LT}\eqref{eq:derivative an-1Jin}$ when $a_{n-1}\left(1\right)\Xi^{\left(n\right)}\left(\tau\right)=a_{n-1}\left(1\right)\Xi^{\left(n\right)}\left(t_{n}^{*}\right)$
and $\tau\in\mathcal{J}_{n}$. By Choice \ref{choice:anbncn} and
Proposition \ref{prop:time convergence}, we have
\[
\begin{aligned}a_{n-1}\left(\tau\right)\Xi^{\left(n\right)}\left(\tau\right) & =a_{n-1}\left(\tau\right)\Xi^{\left(n\right)}\left(t_{n}^{*}\right)=\frac{a_{n-1}\left(\tau\right)}{a_{n-1}\left(1\right)}\frac{a_{n-1}\left(1\right)}{a_{n-1}\left(t_{n}^{*}\right)}a_{n-1}\left(t_{n}^{*}\right)\Xi^{\left(n\right)}\left(t_{n}^{*}\right)=\\
 & =\frac{a_{n-1}\left(\tau\right)}{a_{n-1}\left(1\right)}\frac{a_{n-1}\left(1\right)}{a_{n-1}\left(1\right)b_{n-1}\left(1\right)\frac{1}{b_{n-1}\left(t_{n}^{*}\right)}}a_{n-1}\left(t_{n}^{*}\right)\Xi^{\left(n\right)}\left(t_{n}^{*}\right)=\\
 & =\left(\frac{a_{n-1}\left(\tau\right)}{a_{n-1}\left(1\right)}-1+1\right)\left(\frac{b_{n-1}\left(t_{n}^{*}\right)}{b_{n-1}\left(1\right)}-1+1\right)a_{n-1}\left(t_{n}^{*}\right)\Xi^{\left(n\right)}\left(t_{n}^{*}\right)\leq\\
 & \leq\left(1+C^{-\beta\delta\gamma\left(\frac{1}{1-\gamma}\right)^{n-1}}\right)^{2}a_{n-1}\left(t_{n}^{*}\right)\Xi^{\left(n\right)}\left(t_{n}^{*}\right).
\end{aligned}
\]
As $C^{-\beta\delta\gamma\left(\frac{1}{1-\gamma}\right)^{n-1}}\leq1$,
we can easily bound
\[
a_{n-1}\left(\tau\right)\Xi^{\left(n\right)}\left(\tau\right)\leq\left(1+3C^{-\beta\delta\gamma\left(\frac{1}{1-\gamma}\right)^{n-1}}\right)a_{n-1}\left(t_{n}^{*}\right)\Xi^{\left(n\right)}\left(t_{n}^{*}\right).
\]
 Then, in view of equation \eqref{eq:third argument bounds modified anjin tn*},
since $\sin\left(x\right)\ge\frac{2}{\pi}\left(\pi-x\right)$ $\forall x\in\left[\frac{\pi}{2},\pi\right]$,
we infer that
\[
\begin{aligned}\sin\left(a_{n-1}\left(\tau\right)\Xi^{\left(n\right)}\left(\tau\right)\right) & \gtrsim\pi-a_{n-1}\left(\tau\right)\Xi^{\left(n\right)}\left(\tau\right)\ge\pi-\left(1+3C^{-\beta\delta\gamma\left(\frac{1}{1-\gamma}\right)^{n-1}}\right)a_{n-1}\left(t_{n}^{*}\right)\Xi^{\left(n\right)}\left(t_{n}^{*}\right)\ge\\
 & \gtrsim\pi-\left(1+3C^{-\beta\delta\gamma\left(\frac{1}{1-\gamma}\right)^{n-1}}\right)\left(\pi-\frac{1}{2}C^{-\beta\beta'\beta''\delta\gamma\left(\frac{1}{1-\gamma}\right)^{n-1}}\right)\ge\\
 & \gtrsim\frac{1}{2}C^{-\beta\beta'\beta''\delta\gamma\left(\frac{1}{1-\gamma}\right)^{n-1}}-3\pi C^{-\beta\delta\gamma\left(\frac{1}{1-\gamma}\right)^{n-1}}\gtrsim\\
 & \gtrsim C^{-\beta\beta'\beta''\delta\gamma\left(\frac{1}{1-\gamma}\right)^{n-1}}\left[1-6\pi C^{-\left[1-\beta'\beta''\right]\beta\delta\gamma\left(\frac{1}{1-\gamma}\right)^{n-1}}\right].
\end{aligned}
\]
Thanks to bound \eqref{eq:third argument bound 1-betabeta thingy},
as long as $C$ is taken large enough (let us say $C\ge\Upsilon_{6}\left(\beta,\beta',\beta'',\delta\right)$),
we may ensure that
\[
1-6\pi C^{-\left[1-\beta'\beta''\right]\beta\delta\gamma\left(\frac{1}{1-\gamma}\right)^{n-1}}\ge\frac{1}{2}
\]
and, consequently,
\begin{equation}
\sin\left(a_{n-1}\left(\tau\right)\Xi^{\left(n\right)}\left(\tau\right)\right)\gtrsim C^{-\beta\beta'\beta''\delta\gamma\left(\frac{1}{1-\gamma}\right)^{n-1}}.\label{eq:third argument lower bound sine}
\end{equation}
Now, observe that, by Proposition \ref{prop:time convergence} (which
we can apply because $C\ge\Upsilon_{1}\left(\beta,\delta\right)$
and $\mathcal{J}_{n}\subseteq\left[t_{n},1\right]$), we can guarantee
that $B_{n-1}\left(\tau\right)\gtrsim B_{n-1}\left(1\right)$ and
$b_{n-1}\left(\tau\right)\gtrsim b_{n-1}\left(1\right)$. In view
of \eqref{eq:third argument definition leading term}, this, along
with equation \eqref{eq:third argument lower bound sine}, ensures
that
\[
\begin{aligned}\text{LT}\eqref{eq:derivative an-1Jin}\left(\tau\right) & \gtrsim B_{n-1}\left(1\right)a_{n-1}\left(1\right)b_{n-1}\left(1\right)C^{-\beta\beta'\beta''\delta\gamma\left(\frac{1}{1-\gamma}\right)^{n-1}}\end{aligned}
\]
provided that $a_{n-1}\left(1\right)\Xi^{\left(n\right)}\left(\tau\right)=a_{n-1}\left(1\right)\Xi^{\left(n\right)}\left(t_{n}^{*}\right)$.
Lastly, by Choices \ref{choice:Bnanbn} and \ref{choice:Mn}, we conclude
that
\begin{equation}
\text{LT}\eqref{eq:derivative an-1Jin}\left(\tau\right)\gtrsim YC^{\delta\left(\frac{1}{1-\gamma}\right)^{n-1}}C^{-\beta\beta'\beta''\delta\gamma\left(\frac{1}{1-\gamma}\right)^{n-1}}\gtrsim YC^{\delta\left(1-\beta\beta'\beta''\gamma\right)\left(\frac{1}{1-\gamma}\right)^{n-1}}\label{eq:third argument bound dominant term}
\end{equation}
as long as $a_{n-1}\left(1\right)\Xi^{\left(n\right)}\left(\tau\right)=a_{n-1}\left(1\right)\Xi^{\left(n\right)}\left(t_{n}^{*}\right)$.
Notice that the exponent $\delta\left(1-\beta\beta'\beta''\gamma\right)$
is strictly positive because $\beta,\beta',\beta'',\gamma<1$. 
\item Estimates for $\text{OT}\eqref{eq:derivative an-1Jin}$ at times $t\in\mathcal{J}_{n}$.
Making use of the fact that $\sin\left(\cdot\right)$ is a Lipschitz
function of constant $1$ and employing equation \eqref{eq:third argument defining feature of Jn},
Proposition \ref{prop:time convergence}, Choices \ref{choice:anbncn},
\ref{choice:Bnanbn} and \ref{choice:Mn} and Lemma \ref{lem:estimate sum superexponential},
we can write
\begin{equation}
\begin{aligned}\left|\text{OT}\eqref{eq:derivative an-1Jin}\left(t\right)\right|= & \left|a_{n-1}\left(1\right)\sum_{m=1}^{n-2}B_{m}\left(t\right)b_{m}\left(t\right)\left[\sin\left(a_{m}\left(t\right)\left(\phi_{1}^{\left(n\right)}\left(t,0\right)-\phi_{1}^{\left(m\right)}\left(t,0\right)\right)\right)+\right.\right.\\
 & \quad\left.\left.-\sin\left(a_{m}\left(t\right)\left(\phi_{1}^{\left(n-1\right)}\left(t,0\right)-\phi_{1}^{\left(m\right)}\left(t,0\right)\right)\right)\right]\right|\le\\
 & \le a_{n-1}\left(1\right)\underbrace{\sum_{m=1}^{n-2}\underbrace{B_{m}\left(t\right)b_{m}\left(t\right)a_{m}\left(t\right)}_{\lesssim B_{m}\left(1\right)b_{m}\left(1\right)a_{m}\left(1\right)=YC^{\delta\left(\frac{1}{1-\gamma}\right)^{m}}}}_{\lesssim_{\delta}YC^{\delta\left(\frac{1}{1-\gamma}\right)^{n-2}}}\quad\underbrace{\Xi^{\left(n\right)}\left(t\right)}_{=\underbrace{a_{n-1}\left(1\right)\Xi^{\left(n\right)}\left(t\right)}_{\leq\frac{3\pi}{2}}\frac{1}{a_{n-1}\left(1\right)}}\lesssim_{\delta}\\
 & \lesssim_{\delta}YC^{\delta\left(\frac{1}{1-\gamma}\right)^{n-2}}\quad\forall t\in\mathcal{J}_{n}.
\end{aligned}
\label{eq:third argument bound other term}
\end{equation}
\item $\frac{\mathrm{d}}{\mathrm{d}t}\left(a_{n-1}\left(1\right)\Xi^{\left(n\right)}\left(t\right)\right)\left(\tau\right)>0$
when $a_{n-1}\left(1\right)\Xi^{\left(n\right)}\left(\tau\right)=a_{n-1}\left(1\right)\Xi^{\left(n\right)}\left(t_{n}^{*}\right)$
and $\tau\in\mathcal{J}_{n}$. Equations \eqref{eq:third argument bound dominant term}
and \eqref{eq:third argument bound other term} imply that
\[
\frac{\left|\text{OT}\eqref{eq:derivative an-1Jin}\left(\tau\right)\right|}{\left|\text{LT}\eqref{eq:derivative an-1Jin}\left(\tau\right)\right|}\lesssim_{\delta}\frac{YC^{\delta\left(\frac{1}{1-\gamma}\right)^{n-2}}}{YC^{\delta\left(1-\beta\beta'\beta''\gamma\right)\left(\frac{1}{1-\gamma}\right)^{n-1}}}=C^{-\delta\left(1-\beta\beta'\beta"\gamma-\left(1-\gamma\right)\right)\left(\frac{1}{1-\gamma}\right)^{n-1}}=C^{-\delta\gamma\left(1-\beta\beta'\beta''\right)\left(\frac{1}{1-\gamma}\right)^{n-1}}.
\]
Observe that the exponent is strictly negative because $\beta,\beta,\beta''<1$.
Moreover, since $n\ge2$, $\gamma\ge\frac{1}{2}$ (by Choice \ref{choice:min requirements on C gamma and kmax})
and the function $x\to\frac{x}{1-x}$ is increasing in the interval
$\left(0,1\right)$, we may bound
\[
C^{-\delta\gamma\left(1-\beta\beta'\beta''\right)\left(\frac{1}{1-\gamma}\right)^{n-1}}\le C^{-\delta\frac{\gamma}{1-\gamma}\left(1-\beta\beta'\beta''\right)}\le C^{-\delta\left(1-\beta\beta'\beta''\right)}.
\]
In this manner, taking $C$ big enough (let us say $C\ge\Upsilon_{7}\left(\beta,\beta',\beta'',\delta\right)$),
we can make $\text{OT}\eqref{eq:derivative an-1Jin}\left(\tau\right)$
negligible in comparison to $\text{LT}\eqref{eq:derivative an-1Jin}\left(\tau\right)$
$\forall n\in\mathbb{N}$. As $\text{LT}\eqref{eq:derivative an-1Jin}\left(\tau\right)$
is positive by equation , we conclude that
\begin{equation}
\frac{\mathrm{d}}{\mathrm{d}t}\left(a_{n-1}\left(1\right)\Xi^{\left(n\right)}\left(t\right)\right)\left(\tau\right)>0\quad\text{when \ensuremath{a_{n-1}\left(1\right)\Xi^{\left(n\right)}\left(\tau\right)}=\ensuremath{a_{n-1}\left(1\right)\Xi^{\left(n\right)}\left(t_{n}^{*}\right)} and }\tau\in\mathcal{J}_{n}.\label{eq:third argument positive derivative at tn*}
\end{equation}
\item The trap. Now, suppose that $\exists s\in\left[t_{n}^{*},1\right]$
such that $a_{n-1}\left(s\right)\Xi^{\left(n\right)}\left(s\right)=\pi$.
Then, it makes sense to study equation \eqref{eq:derivative an-1Jin}
when $\pi\le a_{n-1}\left(t\right)\Xi^{\left(n\right)}\left(t\right)\le\frac{3\pi}{2}$.
Recall that this upper bound is satisfied $\forall t\in\mathcal{J}_{n}$.
Under this assumption, it is evident that the term $\text{LT}\eqref{eq:derivative an-1Jin}\left(t\right)$
is now negative. Thus, equation \eqref{eq:third argument bound other term}
tells us that
\begin{equation}
\frac{\mathrm{d}}{\mathrm{d}t}\left(a_{n-1}\left(1\right)\Xi^{\left(n\right)}\left(t\right)\right)\left(t\right)\leq\left|\text{OT}(\ref{eq:derivative an-1Jin})\left(t\right)\right|\lesssim_{\delta}YC^{\delta\left(\frac{1}{1-\gamma}\right)^{n-2}}\quad\text{if }\pi\le a_{n-1}\left(t\right)\Xi^{\left(n\right)}\left(t\right)\le\frac{3\pi}{2}.\label{eq:third argument bounded derivative after pi v0}
\end{equation}
Notice that, by Choice \ref{choice:anbncn} and Proposition \ref{prop:time convergence},
\[
\begin{aligned}a_{n-1}\left(t\right)\Xi^{\left(n\right)}\left(t\right) & =\frac{a_{n-1}\left(t\right)}{a_{n-1}\left(1\right)}a_{n-1}\left(1\right)\Xi^{\left(n\right)}\left(t\right)=\left(\frac{a_{n-1}\left(t\right)}{a_{n-1}\left(1\right)}-1+1\right)a_{n-1}\left(1\right)\Xi^{\left(n\right)}\left(t\right)\ge\\
 & \ge\left(1-C^{-\beta\delta\gamma\left(\frac{1}{1-\gamma}\right)^{n-1}}\right)a_{n-1}\left(1\right)\Xi^{\left(n\right)}\left(t\right).
\end{aligned}
\]
Thanks to the second equation of \eqref{eq:third argument bounds},
choosing $C$ large enough (let us say $C\ge\Upsilon_{8}\left(\beta,\delta\right)$),
we can ensure that $C^{-\beta\delta\gamma\left(\frac{1}{1-\gamma}\right)^{n-1}}\leq\frac{1}{2}$
$\forall n\in\mathbb{N}$. In this way, as we have $1-x\ge\frac{1}{1+2x}$
$\forall x\in\left[0,\frac{1}{2}\right]$, we can write
\[
a_{n-1}\left(t\right)\Xi^{\left(n\right)}\left(t\right)\ge\frac{1}{1+2C^{-\beta\delta\gamma\left(\frac{1}{1-\gamma}\right)^{n-1}}}a_{n-1}\left(1\right)\Xi^{\left(n\right)}\left(t\right).
\]
Hence, we can guarantee that $a_{n-1}\left(t\right)\Xi^{\left(n\right)}\left(t\right)\ge\pi$
as long as 
\[
a_{n-1}\left(1\right)\Xi^{\left(n\right)}\left(t\right)\ge\pi\left(1+2C^{-\beta\delta\gamma\left(\frac{1}{1-\gamma}\right)^{n-1}}\right).
\]
Thereby, equation \eqref{eq:third argument bounded derivative after pi v0}
implies that
\begin{equation}
\begin{aligned}\frac{\mathrm{d}}{\mathrm{d}t}\left(a_{n-1}\left(1\right)\Xi^{\left(n\right)}\left(t\right)\right)\left(t\right) & \leq\left|\text{OT}(\ref{eq:derivative an-1Jin})\left(t\right)\right|\lesssim_{\delta}YC^{\delta\left(\frac{1}{1-\gamma}\right)^{n-2}}\\
 & \quad\text{if }\pi\left(1+2C^{-\beta\delta\gamma\left(\frac{1}{1-\gamma}\right)^{n-1}}\right)\le a_{n-1}\left(1\right)\Xi^{\left(n\right)}\left(t\right)\le\frac{3\pi}{2}.
\end{aligned}
\label{eq:third argument bounded derivative after pi}
\end{equation}
If, on the other hand, $\nexists s\in\left[t_{n}^{*},1\right]$ such
that $a_{n-1}\left(s\right)\Xi^{\left(n\right)}\left(s\right)=\pi$,
by the continuity of $a_{n-1}\left(t\right)\Xi^{\left(n\right)}\left(t\right)$
and equation \eqref{eq:third argument bounds modified anjin tn*},
we deduce that $a_{n-1}\left(t\right)\Xi^{\left(n\right)}\left(t\right)<\pi$
$\forall t\in\left[t_{n}^{*},1\right]$. Then, by Choice \ref{choice:anbncn}
and Proposition \ref{prop:time convergence}, we can bound
\[
\begin{aligned}a_{n-1}\left(1\right)\Xi^{\left(n\right)}\left(t\right) & =\frac{a_{n-1}\left(1\right)}{a_{n-1}\left(t\right)}a_{n-1}\left(t\right)\Xi^{\left(n\right)}\left(t\right)=\frac{a_{n-1}\left(1\right)}{a_{n-1}\left(1\right)b_{n-1}\left(1\right)\frac{1}{b_{n-1}\left(t\right)}}a_{n-1}\left(t\right)\Xi^{\left(n\right)}\left(t\right)=\\
 & =\frac{b_{n-1}\left(t\right)}{b_{n-1}\left(1\right)}a_{n-1}\left(t\right)\Xi^{\left(n\right)}\left(t\right)=\left(\frac{b_{n-1}\left(t\right)}{b_{n-1}\left(1\right)}-1+1\right)a_{n-1}\left(t\right)\Xi^{\left(n\right)}\left(t\right)<\\
 & <\pi\left(1+C^{-\beta\delta\gamma\left(\frac{1}{1-\gamma}\right)^{n-1}}\right)<\pi\left(1+2C^{-\beta\delta\gamma\left(\frac{1}{1-\gamma}\right)^{n-1}}\right).
\end{aligned}
\]
Consequently, equation \eqref{eq:third argument bounded derivative after pi}
is trivially true in this case as its assumption is never satisfied.
Thereby, as we are considering $t\in\mathcal{J}_{n}$, thanks to equations
\eqref{eq:third argument positive derivative at tn*} and \eqref{eq:third argument bounded derivative after pi},
we can apply Lemma \ref{lem:trap} with 
\[
a\leftarrow a_{n-1}\left(1\right)\Xi^{\left(n\right)}\left(t_{n}^{*}\right),\quad b\leftarrow\pi\left(1+2C^{-\beta\delta\gamma\left(\frac{1}{1-\gamma}\right)^{n-1}}\right),\quad c\leftarrow\frac{3\pi}{2},\quad t\leftarrow t-t_{n}^{*},
\]
which assures us that
\begin{equation}
a_{n-1}\left(1\right)\Xi^{\left(n\right)}\left(t\right)\in\left[a_{n-1}\left(1\right)\Xi^{\left(n\right)}\left(t_{n}^{*}\right),\pi\left(1+2C^{-\beta\delta\gamma\left(\frac{1}{1-\gamma}\right)^{n-1}}\right)+L\left(\delta\right)YC^{\delta\left(\frac{1}{1-\gamma}\right)^{n-2}}\left(t-t_{n}^{*}\right)\right]\label{eq:result of trapping lemma}
\end{equation}
as long as $t\in\mathcal{J}_{n}$ and
\[
\pi\left(1+2C^{-\beta\delta\gamma\left(\frac{1}{1-\gamma}\right)^{n-1}}\right)+L\left(\delta\right)YC^{\delta\left(\frac{1}{1-\gamma}\right)^{n-2}}\left(t-t_{n}^{*}\right)\le\frac{3\pi}{2},
\]
where $L\left(\delta\right)$ is the constant of inequality \eqref{eq:third argument bounded derivative after pi}.
As $t\in\mathcal{J}_{n}\subseteq\left[t_{n}^{*},1\right]$, in particular,
we know that $t\le1$ and, in this way, thanks to Choice \ref{choice:Mn}
and the definition of $t_{n}^{*}$ , we have
\[
t-t_{n}^{*}\le1-t_{n}^{*}\le1-t_{n}=\frac{1}{Y}C^{-\delta\left(\frac{1}{1-\gamma}\right)^{n-1}}\mathrm{arccosh}\left(C^{k_{\max}\left(\frac{1}{1-\gamma}\right)^{n}}\right)
\]
and, as a consequence,
\[
\begin{aligned}L\left(\delta\right)YC^{\delta\left(\frac{1}{1-\gamma}\right)^{n-2}}\left(t-t_{n}^{*}\right) & \leq L\left(\delta\right)C^{-\delta\left(\frac{1}{1-\gamma}\right)^{n-1}\left[1-\left(1-\gamma\right)\right]}\mathrm{arccosh}\left(C^{k_{\max}\left(\frac{1}{1-\gamma}\right)^{n}}\right)\leq\\
 & \leq L\left(\delta\right)C^{-\delta\gamma\left(\frac{1}{1-\gamma}\right)^{n-1}}\mathrm{arccosh}\left(C^{k_{\max}\left(\frac{1}{1-\gamma}\right)^{n}}\right).
\end{aligned}
\]
By Lemma \ref{lem:arccosh by ln},
\[
L\left(\delta\right)YC^{\delta\left(\frac{1}{1-\gamma}\right)^{n-2}}\left(t-t_{n}^{*}\right)\le L\left(\delta\right)C^{-\delta\gamma\left(\frac{1}{1-\gamma}\right)^{n-1}}\left[\ln\left(2\right)+\ln\left(C\right)k_{\max}\left(\frac{1}{1-\gamma}\right)^{n}\right].
\]
Proceeding like in the proof of Proposition \ref{prop:time convergence}
from equation \eqref{eq:time convergence the log 1} to equation \eqref{eq:time convergence the log 2},
we deduce that
\[
L\left(\delta\right)YC^{\delta\left(\frac{1}{1-\gamma}\right)^{n-2}}\left(t-t_{n}^{*}\right)\lesssim_{\delta}C^{-\delta\gamma\left(\frac{1}{1-\gamma}\right)^{n-1}}\ln\left(C\right)\left(\frac{1}{1-\gamma}\right)^{n}.
\]
Since $\frac{1}{1-\gamma}\ge2$ by Choice \ref{choice:min requirements on C gamma and kmax},
we are in position to apply Lemma \ref{lem:exponetial superexponential bound}
with $a\leftarrow\frac{1}{1-\gamma}$, $n\leftarrow n-1$ and $b\leftarrow\delta\gamma$,
which lets us bound
\[
L\left(\delta\right)YC^{\delta\left(\frac{1}{1-\gamma}\right)^{n-2}}\left(t-t_{n}^{*}\right)\lesssim_{\delta}L\left(\delta\right)\frac{4\mathrm{e}^{-2}}{\left(1-\beta\right)^{2}\gamma^{2}\delta^{2}\ln\left(C\right)}C^{-\beta\gamma\delta\left(\frac{1}{1-\gamma}\right)^{n-1}}.
\]
As $\gamma\ge\frac{1}{2}$ by Choice \ref{choice:min requirements on C gamma and kmax},
we obtain
\begin{equation}
L\left(\delta\right)YC^{\delta\left(\frac{1}{1-\gamma}\right)^{n-2}}\left(t-t_{n}^{*}\right)\lesssim_{\beta,\delta}\frac{C^{-\beta\gamma\delta\left(\frac{1}{1-\gamma}\right)^{n-1}}}{\ln\left(C\right)}.\label{eq:third argument small over pi}
\end{equation}
Then, taking $C$ big enough, we can compensate the constant of the
inequality above through the factor $\ln\left(C\right)$. Suppose
we need $C\ge\Upsilon_{9}\left(\beta,\delta\right)$ to make this
constant $1$ (or smaller). Thereby,
\begin{equation}
L\left(\delta\right)YC^{\delta\left(\frac{1}{1-\gamma}\right)^{n-2}}\left(t-t_{n}^{*}\right)\le C^{-\beta\gamma\delta\left(\frac{1}{1-\gamma}\right)^{n-1}}.\label{eq:third argument bound L(delta)blah}
\end{equation}
Using this result in equation \eqref{eq:result of trapping lemma},
along with equation \eqref{eq:third argument closeness to pi}, provides
\begin{equation}
a_{n-1}\left(1\right)\Xi^{\left(n\right)}\left(t\right)\in\left[\pi-C^{-\beta\beta'\beta''\delta\gamma\left(\frac{1}{1-\gamma}\right)^{n-1}},\pi\left(1+2C^{-\beta\delta\gamma\left(\frac{1}{1-\gamma}\right)^{n-1}}\right)+C^{-\beta\gamma\delta\left(\frac{1}{1-\gamma}\right)^{n-1}}\right]\label{eq:third argument trap v0}
\end{equation}
as long as $t\in\mathcal{J}_{n}$ and
\begin{equation}
\pi\left(1+2C^{-\beta\delta\gamma\left(\frac{1}{1-\gamma}\right)^{n-1}}\right)+C^{-\beta\gamma\delta\left(\frac{1}{1-\gamma}\right)^{n-1}}\leq\frac{3\pi}{2}.\label{eq:third argument as long as condition}
\end{equation}
Thanks to the second bound exposed in \eqref{eq:third argument bounds},
it is evident that, choosing $C$ large enough, let us say $C\ge\Upsilon_{10}\left(\beta,\delta\right)$,
we can ensure that
\[
\pi\left(1+2C^{-\beta\delta\gamma\left(\frac{1}{1-\gamma}\right)^{n-1}}\right)+C^{-\beta\gamma\delta\left(\frac{1}{1-\gamma}\right)^{n-1}}\le\frac{5}{4}\pi<\frac{3\pi}{2},
\]
which makes condition \eqref{eq:third argument as long as condition}
superfluous. Bounding $\beta,\beta'<1$ and $1\leq\pi$, we can convert
\eqref{eq:third argument trap v0} into
\begin{equation}
\left|\pi-a_{n-1}\left(1\right)\Xi^{\left(n\right)}\left(t\right)\right|\le3\pi C^{-\beta\beta'\beta''\delta\gamma\left(\frac{1}{1-\gamma}\right)^{n-1}}\quad\forall t\in\mathcal{J}_{n}.\label{eq:third argument trap v1}
\end{equation}
Recall that, when we defined the interval $\mathcal{J}_{n}$ in equation
\eqref{eq:third argument defining feature of Jn}, we did not know
its length. Equation \eqref{eq:third argument trap v1}, along with
the first bound of equation \eqref{eq:third argument bounds}, tells
us that, provided that $C$ is selected large enough, let us say $C\ge\Upsilon_{11}\left(\beta,\beta',\beta'',\delta\right)$,
we can guarantee that, $\forall t\in\mathcal{J}_{n}$, $a_{n-1}\left(1\right)\Xi^{\left(n\right)}\left(t\right)$
is as close to $\pi$ as we want. Similarly, thanks to point \ref{enu:uniform bounds},
we have
\[
\begin{aligned}\left|\pi-a_{n-1}\left(t\right)\Xi^{\left(n\right)}\left(t\right)\right| & \le\left|\pi-a_{n-1}\left(1\right)\Xi^{\left(n\right)}\left(t\right)\right|+\left|a_{n-1}\left(1\right)\Xi^{\left(n\right)}\left(t\right)-a_{n-1}\left(t\right)\Xi^{\left(n\right)}\left(t\right)\right|\leq\\
 & \leq3\pi C^{-\beta\beta'\beta''\delta\gamma\left(\frac{1}{1-\gamma}\right)^{n-1}}+C^{-\beta\delta\gamma\left(\frac{1}{1-\gamma}\right)^{n-1}}a_{n-1}\left(1\right)\Xi^{\left(n\right)}\left(t\right)\leq\\
 & \leq3\pi C^{-\beta\beta'\beta''\delta\gamma\left(\frac{1}{1-\gamma}\right)^{n-1}}+C^{-\beta\delta\gamma\left(\frac{1}{1-\gamma}\right)^{n-1}}\left(1+3\pi C^{-\beta\beta'\beta''\delta\gamma\left(\frac{1}{1-\gamma}\right)^{n-1}}\right).
\end{aligned}
\]
In this manner, by the bounds presented in equation \eqref{eq:third argument bounds},
choosing $C$ big enough, let us say $C\ge\Upsilon_{12}\left(\beta,\beta',\beta'',\delta\right)$,
we may assure that, $\forall t\in\mathcal{J}_{n}$, $a_{n-1}\left(t\right)\Xi^{\left(n\right)}\left(t\right)$
is also as close to $\pi$ as we want. In other words, we have proven
that, in the conditions employed in the definition of $\mathcal{J}_{n}$,
the conditions are always fulfilled with a ``large'' margin $\forall t\in\mathcal{J}_{n}$,
in other words, they are not really restrictions! This means we must
have $\mathcal{J}_{n}=\left[t_{n}^{*},1\right]$. Consequently, equation
\eqref{eq:third argument trap v1} becomes
\begin{equation}
\left|\pi-a_{n-1}\left(1\right)\Xi^{\left(n\right)}\left(t\right)\right|\le3\pi C^{-\beta\beta'\beta''\delta\gamma\left(\frac{1}{1-\gamma}\right)^{n-1}}\quad\forall t\in\left[t_{n}^{*},1\right].\label{eq:third argument trap}
\end{equation}
\item The sandwich. By the second statement of this Proposition,
\[
a_{n-1}\left(1\right)\left|\Xi^{\left(n\right)}\left(t_{n}^{*}\right)-\Xi_{0}^{\left(n\right)}\left(t_{n}^{*}\right)\right|\leq C^{-\beta\beta'\left(1-\beta''\right)\delta\gamma\left(\frac{1}{1-\gamma}\right)^{n-1}}\sin\left(a_{n-1}\left(1\right)\Xi_{0}^{\left(n\right)}\left(t_{n}^{*}\right)\right).
\]
Bounding $\sin\left(x\right)\le1$, we obtain that
\[
\left|a_{n-1}\left(1\right)\Xi^{\left(n\right)}\left(t_{n}^{*}\right)-a_{n-1}\left(1\right)\Xi_{0}^{\left(n\right)}\left(t_{n}^{*}\right)\right|\le C^{-\beta\beta'\left(1-\beta''\right)\delta\gamma\left(\frac{1}{1-\gamma}\right)^{n-1}}.
\]
Thereby, the triangular inequality, along with equation \eqref{eq:third argument closeness to pi}
assures us that
\begin{equation}
\begin{aligned}\left|\pi-a_{n-1}\left(1\right)\Xi_{0}^{\left(n\right)}\left(t_{n}^{*}\right)\right| & \le\left|\pi-a_{n-1}\left(1\right)\Xi^{\left(n\right)}\left(t_{n}^{*}\right)\right|+\left|a_{n-1}\left(1\right)\Xi^{\left(n\right)}\left(t_{n}^{*}\right)-a_{n-1}\left(1\right)\Xi_{0}^{\left(n\right)}\left(t_{n}^{*}\right)\right|\le\\
 & \le C^{-\beta\beta'\beta''\delta\gamma\left(\frac{1}{1-\gamma}\right)^{n-1}}+C^{-\beta\beta'\left(1-\beta''\right)\delta\gamma\left(\frac{1}{1-\gamma}\right)^{n-1}}\le2C^{-\beta\beta'\min\left\{ \beta'',1-\beta''\right\} \delta\gamma\left(\frac{1}{1-\gamma}\right)^{n-1}}.
\end{aligned}
\label{eq:third argument JI0tn* also close to pi}
\end{equation}
Since $a_{n-1}\left(1\right)\Xi_{0}^{\left(n\right)}\left(t\right)=F_{n}\left(B_{n-1}\left(1\right)a_{n-1}\left(1\right)b_{n-1}\left(1\right)\left(t-t_{n}\right)\right)$
and $F_{n}$ is a solution of the ODE presented in Lemma \ref{lem:the good ODE},
by point 1 of the mentioned Lemma, we know that $0\le a_{n-1}\left(1\right)\Xi_{0}^{\left(n\right)}\left(t\right)\le\pi$
$\forall t\in\left[t_{n}^{*},1\right]$. Consequently, equation \eqref{eq:ODE Jin0}
guarantees that
\[
\frac{\mathrm{d}}{\mathrm{d}t}\left(a_{n-1}\left(1\right)\Xi_{0}^{\left(n\right)}\left(t\right)\right)=B_{n-1}\left(1\right)a_{n-1}\left(1\right)b_{n-1}\left(1\right)\sin\left(a_{n-1}\left(1\right)\Xi_{0}^{\left(n\right)}\left(t\right)\right)\ge0
\]
and, in this manner, we conclude that $a_{n-1}\left(1\right)\Xi_{0}^{\left(n\right)}\left(t\right)$
is a non-decreasing function of time. Hence, thanks to equation \eqref{eq:third argument JI0tn* also close to pi},
we now know that
\begin{equation}
\pi-a_{n-1}\left(1\right)\Xi_{0}^{\left(n\right)}\left(t\right)\le\pi-a_{n-1}\left(1\right)\Xi_{0}^{\left(n\right)}\left(t_{n}^{*}\right)\le2C^{-\beta\beta'\min\left\{ \beta'',1-\beta''\right\} \delta\gamma\left(\frac{1}{1-\gamma}\right)^{n-1}}.\label{eq:third argument ideal close to pi}
\end{equation}
Lastly, from equations \eqref{eq:third argument trap} and \eqref{eq:third argument ideal close to pi},
by the triangular inequality, it easily follows that
\[
\begin{aligned}a_{n-1}\left(1\right)\left|\Xi^{\left(n\right)}\left(t\right)-\Xi_{0}^{\left(n\right)}\left(t\right)\right| & \le\left|a_{n-1}\left(1\right)\Xi^{\left(n\right)}\left(t\right)-\pi\right|+\left|\pi-a_{n-1}\left(1\right)\Xi_{0}^{\left(n\right)}\left(t\right)\right|\leq\\
 & \le3\pi C^{-\beta\beta'\beta''\delta\gamma\left(\frac{1}{1-\gamma}\right)^{n-1}}+2C^{-\beta\beta'\min\left\{ \beta'',1-\beta''\right\} \delta\gamma\left(\frac{1}{1-\gamma}\right)^{n-1}}\le\\
 & \le14C^{-\beta\beta'\min\left\{ \beta'',1-\beta''\right\} \delta\gamma\left(\frac{1}{1-\gamma}\right)^{n-1}},
\end{aligned}
\]
where we have bounded $\pi\le4$. This proves the third statement.
\end{enumerate}
Obtaining the uniform bound in the case $\beta''=\frac{1}{2}$ is
immediate from the previous results and the bounds $1\le14$, $\sin\left(x\right)\le x$,
$\beta''=\frac{1}{2}\le1$.
\end{proof}

\subsection{Convergence of $k_{n}\left(t\right)$}

We finish our proofs of uniform convergence in time to the ideal model
by showing that $k_{n}\left(t\right)$ also lies close to our idealized
$\overline{k}_{n}\left(t\right)$, whose limit in $n\in\mathbb{N}$
was explicitly computed in subsection \ref{subsec:completion of the toy model}.
\begin{prop}
\label{prop:convergence kn to ideal model}Let $n\in\mathbb{N}$ with
$n\ge2$ and $\beta,\beta'\in\left(0,1\right)$. Moreover, let $k_{n}\left(t\right)$
be as defined in Choice \ref{choice:anbn} and $\overline{k}_{n}\left(t\right)$
be the ideal model of $k_{n}\left(t\right)$, which is given as the
solution of ODE \eqref{eq:ideal dkntdt}. Then, provided that $C$
is big enough (let us say $C\ge\Upsilon\left(\beta,\beta',\delta\right)$),
\[
\left|k_{n}\left(t\right)-\overline{k}_{n}\left(t\right)\right|\lesssim_{\delta}C^{-\frac{1}{2}\beta\beta'\delta\gamma\left(\frac{1}{1-\gamma}\right)^{n-1}}\quad\forall t\in\left[t_{n},1\right].
\]
\end{prop}
\begin{proof}
We depart from equations \eqref{eq:real dkntdt} and \eqref{eq:ideal dkntdt}.
Subtracting both equations leads to
\begin{equation}
\begin{aligned}\frac{\mathrm{d}}{\mathrm{d}t}\left(k_{n}\left(t\right)-\overline{k}_{n}\left(t\right)\right) & =\frac{1}{\ln\left(C\right)\left(\frac{1}{1-\gamma}\right)^{n}}\sum_{m=1}^{n-1}B_{m}\left(t\right)b_{m}\left(t\right)a_{m}\left(t\right)\cos\left(a_{m}\left(t\right)\left(\phi_{1}^{\left(n\right)}\left(t,0\right)-\phi_{1}^{\left(m\right)}\left(t,0\right)\right)\right)+\\
 & \quad-\frac{B_{n-1}\left(1\right)b_{n-1}\left(1\right)a_{n-1}\left(1\right)}{\ln\left(C\right)\left(\frac{1}{1-\gamma}\right)^{n}}\cos\left(a_{n-1}\left(1\right)\Xi_{0}^{\left(n\right)}\left(t\right)\right)=\\
 & =\frac{1}{\ln\left(C\right)\left(\frac{1}{1-\gamma}\right)^{n}}\left[B_{n-1}\left(t\right)b_{n-1}\left(t\right)a_{n-1}\left(t\right)\cos\left(a_{n-1}\left(t\right)\Xi^{\left(n\right)}\left(t\right)\right)+\right.\\
 & \qquad\left.-B_{n-1}\left(1\right)b_{n-1}\left(1\right)a_{n-1}\left(1\right)\cos\left(a_{n-1}\left(1\right)\Xi_{0}^{\left(n\right)}\left(t\right)\right)\right]+\\
 & \quad+\frac{1}{\ln\left(C\right)\left(\frac{1}{1-\gamma}\right)^{n}}\sum_{m=1}^{n-2}B_{m}\left(t\right)b_{m}\left(t\right)a_{m}\left(t\right)\cos\left(a_{m}\left(t\right)\left(\phi_{1}^{\left(n\right)}\left(t,0\right)-\phi_{1}^{\left(m\right)}\left(t,0\right)\right)\right).
\end{aligned}
\label{eq:convergence of kn ODE difference}
\end{equation}
Let us denote the first summand by $I_{1}$ and the summand that contains
the sum from $m=1$ to $n-2$ by $I_{2}$. We shall bound each term
independently.
\begin{itemize}
\item $I_{1}$. Using Choice \ref{choice:anbncn} to write $b_{n-1}\left(t\right)a_{n-1}\left(t\right)=b_{n-1}\left(1\right)a_{n-1}\left(1\right)$,
adding and subtracting some terms and employing the triangular inequality,
we arrive to
\[
\begin{aligned}\left|I_{1}\right|\leq & \frac{B_{n-1}\left(1\right)a_{n-1}\left(1\right)b_{n-1}\left(1\right)}{\ln\left(C\right)\left(\frac{1}{1-\gamma}\right)^{n}}\left[\left|\frac{B_{n-1}\left(t\right)}{B_{n-1}\left(1\right)}-1\right|\cos\left(a_{n-1}\left(t\right)\Xi^{\left(n\right)}\left(t\right)\right)+\right.\\
 & \quad\left.+\left|\cos\left(a_{n-1}\left(t\right)\Xi^{\left(n\right)}\left(t\right)\right)-\cos\left(a_{n-1}\left(1\right)\Xi_{0}^{\left(n\right)}\left(t\right)\right)\right|\right].
\end{aligned}
\]
As $\left|\cos\left(\cdot\right)\right|$ is Lipschitz of constant
$1$ and $\left|\cos\left(\cdot\right)\right|\le1$, we deduce that
\[
\begin{aligned}\left|I_{1}\right| & \le\frac{B_{n-1}\left(1\right)a_{n-1}\left(1\right)b_{n-1}\left(1\right)}{\ln\left(C\right)\left(\frac{1}{1-\gamma}\right)^{n}}\left[\left|\frac{B_{n-1}\left(t\right)}{B_{n-1}\left(1\right)}-1\right|+\left|a_{n-1}\left(t\right)\Xi^{\left(n\right)}\left(t\right)-a_{n-1}\left(1\right)\Xi_{0}^{\left(n\right)}\left(t\right)\right|\right]\le\\
 & \le\frac{B_{n-1}\left(1\right)a_{n-1}\left(1\right)b_{n-1}\left(1\right)}{\ln\left(C\right)\left(\frac{1}{1-\gamma}\right)^{n}}\left[\left|\frac{B_{n-1}\left(t\right)}{B_{n-1}\left(1\right)}-1\right|+a_{n-1}\left(t\right)\left|\Xi^{\left(n\right)}\left(t\right)-\Xi_{0}^{\left(n\right)}\left(t\right)\right|+\right.\\
 & \quad\left.+\left|a_{n-1}\left(t\right)-a_{n-1}\left(1\right)\right|\Xi_{0}^{\left(n\right)}\left(t\right)\right]\le\\
 & \le\frac{B_{n-1}\left(1\right)a_{n-1}\left(1\right)b_{n-1}\left(1\right)}{\ln\left(C\right)\left(\frac{1}{1-\gamma}\right)^{n}}\left[\left|\frac{B_{n-1}\left(t\right)}{B_{n-1}\left(1\right)}-1\right|+\right.\\
 & \quad\left.+\left(\frac{a_{n-1}\left(t\right)}{a_{n-1}\left(1\right)}-1+1\right)a_{n-1}\left(1\right)\left|\Xi^{\left(n\right)}\left(t\right)-\Xi_{0}^{\left(n\right)}\left(t\right)\right|+a_{n-1}\left(1\right)\Xi_{0}^{\left(n\right)}\left(t\right)\left|\frac{a_{n-1}\left(t\right)}{a_{n-1}\left(1\right)}-1\right|\right].
\end{aligned}
\]
Taking $C$ big enough (let us say $C\ge\Upsilon_{}\left(\beta,\beta',\delta\right)$)
so that we can make use of Propositions \ref{prop:time convergence}
and \ref{prop:JI convergence to ideal} (with $\beta''=\frac{1}{2}$)
and recalling that $a_{n-1}\left(1\right)\Xi_{0}^{\left(n\right)}\left(t\right)\leq\pi$
by Lemma \ref{lem:the good ODE}, we infer that
\[
\begin{aligned}\left|I_{1}\right| & \leq\frac{B_{n-1}\left(1\right)a_{n-1}\left(1\right)b_{n-1}\left(1\right)}{\ln\left(C\right)\left(\frac{1}{1-\gamma}\right)^{n}}\left[C^{-\beta\delta\gamma\left(\frac{1}{1-\gamma}\right)^{n-1}}+\left(C^{-\beta\delta\gamma\left(\frac{1}{1-\gamma}\right)^{n-1}}+1\right)14C^{-\frac{1}{2}\beta\beta'\delta\gamma\left(\frac{1}{1-\gamma}\right)^{n-1}}+\right.\\
 & \quad\left.+\pi C^{-\beta\delta\gamma\left(\frac{1}{1-\gamma}\right)^{n-1}}\right].
\end{aligned}
\]
As $\beta,\beta',\gamma,\delta<1$, the term that decreases the slowest
is the one that contains $C^{-\frac{1}{2}\beta\beta'\delta\gamma\left(\frac{1}{1-\gamma}\right)^{n-1}}$.
Thus, employing Choices \ref{choice:Bnanbn} and \ref{choice:Mn},
\begin{equation}
\left|I_{1}\right|\lesssim\frac{B_{n-1}\left(1\right)a_{n-1}\left(1\right)b_{n-1}\left(1\right)}{\ln\left(C\right)\left(\frac{1}{1-\gamma}\right)^{n}}C^{-\frac{1}{2}\beta\beta'\delta\gamma\left(\frac{1}{1-\gamma}\right)^{n-1}}=\frac{YC^{\delta\left(\frac{1}{1-\gamma}\right)^{n-1}}}{\ln\left(C\right)\left(\frac{1}{1-\gamma}\right)^{n}}C^{-\frac{1}{2}\beta\beta'\delta\gamma\left(\frac{1}{1-\gamma}\right)^{n-1}}.\label{eq:convergence of kn bound I1}
\end{equation}
\item $I_{2}$. Bounding $\left|\cos\left(\cdot\right)\right|\le1$ and
using Choice \ref{choice:anbncn} and Proposition \ref{prop:time convergence}
provides
\[
\begin{aligned}\left|I_{2}\right| & \leq\frac{1}{\ln\left(C\right)\left(\frac{1}{1-\gamma}\right)^{n}}\sum_{m=1}^{n-2}B_{m}\left(t\right)b_{m}\left(t\right)a_{m}\left(t\right)=\frac{1}{\ln\left(C\right)\left(\frac{1}{1-\gamma}\right)^{n}}\sum_{m=1}^{n-2}B_{m}\left(t\right)b_{m}\left(1\right)a_{m}\left(1\right)\lesssim\\
 & \lesssim\frac{1}{\ln\left(C\right)\left(\frac{1}{1-\gamma}\right)^{n}}\sum_{m=1}^{n-2}B_{m}\left(1\right)b_{m}\left(1\right)a_{m}\left(1\right).
\end{aligned}
\]
Now, by Choices \ref{choice:Bnanbn} and \ref{choice:Mn}, we obtain
\[
\left|I_{2}\right|\lesssim\frac{1}{\ln\left(C\right)\left(\frac{1}{1-\gamma}\right)^{n}}\sum_{m=1}^{n-2}YC^{\delta\left(\frac{1}{1-\gamma}\right)^{m}}.
\]
Finally, Lemma \ref{lem:estimate sum superexponential} allows us
to bound
\begin{equation}
\left|I_{2}\right|\lesssim_{\delta}\frac{Y}{\ln\left(C\right)\left(\frac{1}{1-\gamma}\right)^{n}}C^{\delta\left(\frac{1}{1-\gamma}\right)^{n-2}}.\label{eq:convergence of kn bound I2}
\end{equation}
\end{itemize}
Combining equations \eqref{eq:convergence of kn bound I1} and \eqref{eq:convergence of kn bound I2}
in \eqref{eq:convergence of kn ODE difference} and recalling that
$-\frac{\mathrm{d}\left|f\right|}{\mathrm{d}t}\leq\left|\frac{\mathrm{d}\left|f\right|}{\mathrm{d}t}\right|\leq\left|\frac{\mathrm{d}f}{\mathrm{d}t}\right|$,
one obtains
\[
-\frac{\mathrm{d}}{\mathrm{d}t}\left(\left|k_{n}\left(t\right)-\overline{k}_{n}\left(t\right)\right|\right)\left(t\right)\lesssim_{\delta}\frac{Y}{\ln\left(C\right)\left(\frac{1}{1-\gamma}\right)^{n}}\left[C^{\delta\left(\frac{1}{1-\gamma}\right)^{n-1}}C^{-\frac{1}{2}\beta\beta'\delta\gamma\left(\frac{1}{1-\gamma}\right)^{n-1}}+C^{\delta\left(\frac{1}{1-\gamma}\right)^{n-2}}\right].
\]
Integrating from a general $t$ to $t=1$ at both sides provides,
since $k_{n}\left(1\right)=0=\overline{k}_{n}\left(1\right)$ by Choice
\ref{choice:time picture}, 
\[
\left|k_{n}\left(t\right)-\overline{k}_{n}\left(t\right)\right|-\underbrace{\left|k_{n}\left(1\right)-\overline{k}_{n}\left(1\right)\right|}_{=0}\lesssim_{\delta}\frac{Y\left(1-t\right)}{\ln\left(C\right)\left(\frac{1}{1-\gamma}\right)^{n}}\left[C^{\delta\left(\frac{1}{1-\gamma}\right)^{n-1}}C^{-\frac{1}{2}\beta\beta'\delta\gamma\left(\frac{1}{1-\gamma}\right)^{n-1}}+C^{\delta\left(\frac{1}{1-\gamma}\right)^{n-2}}\right].
\]
Since $t\in\left[t_{n},1\right]$, we can bound $1-t\le1-t_{n}$.
Thereby, Choice \ref{choice:Mn} allows us to write
\[
\begin{aligned}\left|k_{n}\left(t\right)-\overline{k}_{n}\left(t\right)\right| & \lesssim_{\delta}\frac{C^{-\delta\left(\frac{1}{1-\gamma}\right)^{n-1}}\mathrm{arccosh}\left(C^{k_{\max}\left(\frac{1}{1-\gamma}\right)^{n}}\right)}{\ln\left(C\right)\left(\frac{1}{1-\gamma}\right)^{n}}\left[C^{\delta\left(\frac{1}{1-\gamma}\right)^{n-1}}C^{-\frac{1}{2}\beta\beta'\delta\gamma\left(\frac{1}{1-\gamma}\right)^{n-1}}+C^{\delta\left(\frac{1}{1-\gamma}\right)^{n-2}}\right]=\\
 & \lesssim_{\delta}\frac{\mathrm{arccosh}\left(C^{k_{\max}\left(\frac{1}{1-\gamma}\right)^{n}}\right)}{\ln\left(C\right)\left(\frac{1}{1-\gamma}\right)^{n}}\left[C^{-\frac{1}{2}\beta\beta'\delta\gamma\left(\frac{1}{1-\gamma}\right)^{n-1}}+C^{-\delta\left(1-\left(1-\gamma\right)\right)\left(\frac{1}{1-\gamma}\right)^{n-1}}\right]=\\
 & \lesssim_{\delta}\frac{\mathrm{arccosh}\left(C^{k_{\max}\left(\frac{1}{1-\gamma}\right)^{n}}\right)}{\ln\left(C\right)\left(\frac{1}{1-\gamma}\right)^{n}}\left[C^{-\frac{1}{2}\beta\beta'\delta\gamma\left(\frac{1}{1-\gamma}\right)^{n-1}}+C^{-\delta\gamma\left(\frac{1}{1-\gamma}\right)^{n-1}}\right].
\end{aligned}
\]
As $\beta,\beta'<1$, we deduce that
\[
\left|k_{n}\left(t\right)-\overline{k}_{n}\left(t\right)\right|\lesssim_{\delta}\frac{\mathrm{arccosh}\left(C^{k_{\max}\left(\frac{1}{1-\gamma}\right)^{n}}\right)}{\ln\left(C\right)\left(\frac{1}{1-\gamma}\right)^{n}}C^{-\frac{1}{2}\beta\beta'\delta\gamma\left(\frac{1}{1-\gamma}\right)^{n-1}}.
\]
Employing Lemma \ref{lem:arccosh by ln} and proceeding like in the
proof of Proposition \ref{prop:time convergence} from equation \eqref{eq:time convergence the log 1}
to \eqref{eq:time convergence the log 2} leads to
\[
\begin{aligned}\left|k_{n}\left(t\right)-\overline{k}_{n}\left(t\right)\right| & \lesssim_{\delta}\frac{\ln\left(2\right)+k_{\max}\left(\frac{1}{1-\gamma}\right)^{n}\ln\left(C\right)}{\ln\left(C\right)\left(\frac{1}{1-\gamma}\right)^{n}}C^{-\frac{1}{2}\beta\beta'\delta\gamma\left(\frac{1}{1-\gamma}\right)^{n-1}}\lesssim C^{-\frac{1}{2}\beta\beta'\delta\gamma\left(\frac{1}{1-\gamma}\right)^{n-1}}.\end{aligned}
\]
\end{proof}
\begin{cor}
\label{cor:bound for an(t), bn(t)}Let $n\in\mathbb{N}$ with $n\ge2$
and $\mu>0$. Provided that $C$ is big enough (let us say $C\ge\Upsilon\left(\delta,\mu\right)$),
then
\[
\begin{aligned} & C^{\left(1-\overline{k}_{n}\left(t\right)-\mu\right)\left(\frac{1}{1-\gamma}\right)^{n}}\le a_{n}\left(t\right)\leq C^{\left(1-\overline{k}_{n}\left(t\right)+\mu\right)\left(\frac{1}{1-\gamma}\right)^{n}}\le C^{\left(1+\mu\right)\left(\frac{1}{1-\gamma}\right)^{n}},\\
 & C^{\left(1-\mu\right)\left(\frac{1}{1-\gamma}\right)^{n}}\le C^{\left(1+\overline{k}_{n}\left(t\right)-\mu\right)\left(\frac{1}{1-\gamma}\right)^{n}}\le b_{n}\left(t\right)\leq C^{\left(1+\overline{k}_{n}\left(t\right)+\mu\right)\left(\frac{1}{1-\gamma}\right)^{n}}.
\end{aligned}
\]
Furthermore,
\[
\begin{aligned}C^{\left(1+\overline{k}_{n}\left(t\right)-\mu\right)\left(\frac{1}{1-\gamma}\right)^{n}}\le\max\left\{ a_{n}\left(t\right),b_{n}\left(t\right)\right\}  & \leq C^{\left(1+\overline{k}_{n}\left(t\right)+\mu\right)\left(\frac{1}{1-\gamma}\right)^{n}}\le b_{n}\left(t\right)C^{2\mu\left(\frac{1}{1-\gamma}\right)^{n}},\\
\min\left\{ a_{n}\left(t\right),b_{n}\left(t\right)\right\}  & \ge C^{\left(1-\overline{k}_{n}\left(t\right)-\mu\right)\left(\frac{1}{1-\gamma}\right)^{n}}\ge a_{n}\left(t\right)C^{-2\mu\left(\frac{1}{1-\gamma}\right)^{n}}.
\end{aligned}
\]
\end{cor}
\begin{proof}
By Choice \ref{choice:anbn}, we can write
\[
a_{n}\left(t\right)=C^{\left(1-k_{n}\left(t\right)\right)\left(\frac{1}{1-\gamma}\right)^{n}},\quad b_{n}\left(t\right)=C^{\left(1+k_{n}\left(t\right)\right)\left(\frac{1}{1-\gamma}\right)^{n}}.
\]
To continue, recall that, by Corollary \ref{cor:ideal kn non negative},
we have $\overline{k}_{n}\left(t\right)\ge0$ $\forall t\in\left[t_{n},1\right]$
and $\forall n\in\mathbb{N}$. Intuitively, since $k_{n}\left(t\right)$
and $\overline{k}_{n}\left(t\right)$ are close, we expect $k_{n}\left(t\right)$
to be almost non-negative as well. If $C$ is big enough (let us say
$C\ge\Upsilon_{1}\left(\frac{1}{2},\frac{1}{2},\delta\right)$), choosing
$\beta=\beta'=\frac{1}{2}$, we may make use of Proposition \ref{prop:convergence kn to ideal model},
which provides
\[
\left|k_{n}\left(t\right)-\overline{k}_{n}\left(t\right)\right|\lesssim_{\delta}C^{-\frac{1}{2}\beta\beta'\delta\gamma\left(\frac{1}{1-\gamma}\right)^{n-1}}\quad\forall t\in\left[t_{n},1\right].
\]
 Since
\[
C^{-\frac{1}{2}\beta\beta'\delta\gamma\left(\frac{1}{1-\gamma}\right)^{n-1}}=C^{-\frac{1}{8}\delta\gamma\left(\frac{1}{1-\gamma}\right)^{n-1}}\le C^{-\frac{1}{8}\delta\frac{\gamma}{1-\gamma}}\le C^{-\frac{1}{8}\delta}
\]
because $\gamma\ge\frac{1}{2}$ by Choice \ref{choice:min requirements on C gamma and kmax}
and the function $x\to\frac{x}{1-x}$ is increasing, choosing $C$
large enough (let us say $C\ge\Upsilon_{2}\left(\delta,\mu\right)$),
we can ensure that
\[
\left|k_{n}\left(t\right)-\overline{k}_{n}\left(t\right)\right|\le\mu\quad\forall t\in\left[t_{n},1\right].
\]
Thus, 
\begin{equation}
\begin{aligned} & C^{\left(1-\overline{k}_{n}\left(t\right)-\mu\right)\left(\frac{1}{1-\gamma}\right)^{n}}\le a_{n}\left(t\right)\le C^{\left(1-\overline{k}_{n}\left(t\right)+\mu\right)\left(\frac{1}{1-\gamma}\right)^{n}},\\
 & C^{\left(1+\overline{k}_{n}\left(t\right)-\mu\right)\left(\frac{1}{1-\gamma}\right)^{n}}\leq b_{n}\left(t\right)\le C^{\left(1+\overline{k}_{n}\left(t\right)+\mu\right)\left(\frac{1}{1-\gamma}\right)^{n}}.
\end{aligned}
\label{eq:bounds bn an}
\end{equation}
As $\overline{k}_{n}\left(t\right)\ge0$ (see Corollary \ref{cor:ideal kn non negative}),
we deduce that
\[
a_{n}\left(t\right)\le C^{\left(1+\mu\right)\left(\frac{1}{1-\gamma}\right)^{n}},\quad b_{n}\left(t\right)\ge C^{\left(1-\mu\right)\left(\frac{1}{1-\gamma}\right)^{n}}
\]
and, consequently,
\[
\begin{aligned}\max\left\{ a_{n}\left(t\right),b_{n}\left(t\right)\right\}  & \leq C^{\left(1+\overline{k}_{n}\left(t\right)+\mu\right)\left(\frac{1}{1-\gamma}\right)^{n}}\le b_{n}\left(t\right)C^{2\mu\left(\frac{1}{1-\gamma}\right)^{n}},\\
\min\left\{ a_{n}\left(t\right),b_{n}\left(t\right)\right\}  & \ge C^{\left(1-\overline{k}_{n}\left(t\right)-\mu\right)\left(\frac{1}{1-\gamma}\right)^{n}}\ge a_{n}\left(t\right)C^{-2\mu\left(\frac{1}{1-\gamma}\right)^{n}}.
\end{aligned}
\]
Furthermore, from \eqref{eq:bounds bn an}, because $\overline{k}_{n}\left(t\right)\ge0$
by Corollary \ref{cor:ideal kn non negative}, it also follows that
\[
\max\left\{ a_{n}\left(t\right),b_{n}\left(t\right)\right\} \ge\max\left\{ C^{\left(1-\overline{k}_{n}\left(t\right)-\mu\right)\left(\frac{1}{1-\gamma}\right)^{n}},C^{\left(1+\overline{k}_{n}\left(t\right)-\mu\right)\left(\frac{1}{1-\gamma}\right)^{n}}\right\} =C^{\left(1+\overline{k}_{n}\left(t\right)-\mu\right)\left(\frac{1}{1-\gamma}\right)^{n}}.
\]
\end{proof}

\subsection{Transport of the layer centers revisited}

The objective of this last subsection is to prove that the assumption
of Proposition \ref{prop:relation between anbn and jacobian} is superfluous.
\begin{lem}
\label{lem:exponential x2}$\forall x\in\left[0,\infty\right)$ one
has $\mathrm{e}^{x}\ge x^{2}$.
\end{lem}
\begin{proof}
Let $f\left(x\right)=\mathrm{e}^{x}-x^{2}$. Differentiating, we get
\[
f'\left(x\right)=\mathrm{e}^{x}-2x.
\]
And, differentiating again,
\[
f''\left(x\right)=\mathrm{e}^{x}-2.
\]
Notice that
\[
f''\left(x\right)=0\iff x=\ln\left(2\right).
\]
Furthermore,
\[
f'''\left(x\right)=\mathrm{e}^{x}\ge1\quad\forall x\in\left[0,\infty\right).
\]
All this information is telling us that $x=\ln\left(2\right)$ is
the only critical point of $f'$ and that it is a local minimum. As
$f''\left(0\right)=-1<0$, $\lim_{x\to\infty}f'\left(x\right)=\infty$
and $x=\ln\left(2\right)$ is the only critical point of $f'$, by
Weierstrass' Theorem, we deduce that $x=\ln\left(2\right)$ is the
absolute minimum of $f'$. As $f'\left(\ln\left(2\right)\right)=2-2\ln\left(2\right)>0$,
we conclude that $f'\left(x\right)>0$ $\forall x\in\left[0,\infty\right)$.
Thus, $f$ is an increasing function and, in this way, $f\left(x\right)\ge f\left(0\right)=1>0$
$\forall x\in\left[0,\infty\right)$.
\end{proof}
\begin{prop}
\label{prop:layer center never leafs cutoff area}Let $n\in\mathbb{N}$.
Then, provided that $C$ is big enough (let us say $C\ge\Upsilon\left(\delta\right)$),
$\forall m\in\mathbb{N}$ with $m\le n$, we have 
\[
\left|\phi_{1}^{\left(n\right)}\left(t,0\right)-\phi_{1}^{\left(m\right)}\left(t,0\right)\right|\le\frac{8\pi}{a_{m}\left(t\right)}\quad\forall t\in\left[t_{n},1\right].
\]
\end{prop}
\begin{proof}
First of all, notice that the statement is trivial for $n=1$ (which
implies $m=1$). From this point on, we will assume that $n\ge2$.
The proof proceeds by expressing $\phi_{1}^{\left(n\right)}\left(t,0\right)-\phi_{1}^{\left(m\right)}\left(t,0\right)$
as a sum of $\Xi^{\left(k\right)}\left(t\right)$. Notice that, since
$\Xi^{\left(k\right)}\left(t\right)=\phi_{1}^{\left(k\right)}\left(t,0\right)-\phi_{1}^{\left(k-1\right)}\left(t,0\right)$,
\begin{equation}
\begin{aligned}\phi_{1}^{\left(n\right)}\left(t,0\right)-\phi_{1}^{\left(m\right)}\left(t,0\right) & =\sum_{k=m+1}^{n}\Xi^{\left(k\right)}\left(t\right)=\sum_{k=m+1}^{n}\left[\left(\Xi^{\left(k\right)}\left(t\right)-\Xi_{0}^{\left(k\right)}\left(t\right)\right)+\Xi_{0}^{\left(k\right)}\left(t\right)\right]=\\
 & =\sum_{k=m+1}^{n}\frac{1}{a_{k-1}\left(1\right)}\left[a_{k-1}\left(1\right)\left(\Xi^{\left(k\right)}\left(t\right)-\Xi_{0}^{\left(k\right)}\left(t\right)\right)+a_{k-1}\left(1\right)\Xi_{0}^{\left(k\right)}\left(t\right)\right],
\end{aligned}
\label{eq:difference of phis}
\end{equation}
where $\Xi_{0}^{\left(k\right)}\left(t\right)$ is as defined in Proposition
\ref{prop:JI convergence to ideal}. On the one hand, by Lemma \ref{lem:the good ODE},
we can bound $a_{k-1}\left(1\right)\Xi_{0}^{\left(k\right)}\left(t\right)\le\pi$.
On the other hand, by Proposition \ref{prop:JI convergence to ideal}
with $\beta=\beta'=\beta''=\frac{1}{2}$, which we can apply as long
as $C\ge\Upsilon_{1}\left(\frac{1}{2},\frac{1}{2},\frac{1}{2},\delta\right)$,
we have
\[
a_{k-1}\left(1\right)\left|\Xi^{\left(k\right)}\left(t\right)-\Xi_{0}^{\left(k\right)}\left(t\right)\right|\le14C^{-\frac{1}{8}\delta\gamma\left(\frac{1}{1-\gamma}\right)^{n-1}}.
\]
As
\[
C^{-\frac{1}{8}\delta\gamma\left(\frac{1}{1-\gamma}\right)^{n-1}}\le C^{-\frac{1}{8}\delta\frac{\gamma}{1-\gamma}}\le C^{-\frac{1}{8}\delta}
\]
because $n\ge2$ and $\gamma\ge\frac{1}{2}$ by Choice \ref{choice:min requirements on C gamma and kmax},
provided that $C$ is large enough (let us say $C\ge\Upsilon_{2}\left(\delta\right)$),
we may ensure that
\[
C^{-\frac{1}{8}\delta\gamma\left(\frac{1}{1-\gamma}\right)^{n-1}}\le\frac{\pi}{14}
\]
and, thereby,
\[
a_{k-1}\left(1\right)\left|\Xi^{\left(k\right)}\left(t\right)-\Xi_{0}^{\left(k\right)}\left(t\right)\right|\le\pi.
\]
Consequently, equation \eqref{eq:difference of phis} implies that
\begin{equation}
\begin{aligned}\left|\phi_{1}^{\left(n\right)}\left(t,0\right)-\phi_{1}^{\left(m\right)}\left(t,0\right)\right| & \le\sum_{k=m+1}^{n}\frac{2\pi}{a_{k-1}\left(1\right)}=\frac{2\pi}{a_{m}\left(1\right)}\sum_{k=m+1}^{n}\frac{a_{m}\left(1\right)}{a_{k-1}\left(1\right)}=\\
 & \le\frac{2\pi}{a_{m}\left(t\right)}\left(\frac{a_{m}\left(t\right)}{a_{m}\left(1\right)}-1+1\right)\sum_{k=m+1}^{n}\frac{a_{m}\left(1\right)}{a_{k-1}\left(1\right)}.
\end{aligned}
\label{eq:layer centers general bound}
\end{equation}
On the one hand, by Proposition \ref{prop:time convergence},
\begin{equation}
\frac{a_{m}\left(t\right)}{a_{m}\left(1\right)}-1+1\le1+C^{-\frac{1}{2}\delta\gamma\left(\frac{1}{1-\gamma}\right)^{n}}\le2.\label{eq:layer centers bound 1}
\end{equation}
On the other hand, by Choices \ref{choice:anbn} and \ref{choice:time picture},
\[
\begin{aligned}\sum_{k=m+1}^{n}\frac{a_{m}\left(1\right)}{a_{k-1}\left(1\right)} & =\sum_{k=m+1}^{n}\frac{C^{\left(\frac{1}{1-\gamma}\right)^{m}}}{C^{\left(\frac{1}{1-\gamma}\right)^{k-1}}}=1+\sum_{k=m+2}^{n}C^{\left(\frac{1}{1-\gamma}\right)^{m}}C^{-\left(\frac{1}{1-\gamma}\right)^{k-1}}=\\
 & =1+\sum_{k=m+2}^{n}C^{-\left(\frac{1}{1-\gamma}\right)^{k-1}\left[1-\left(1-\gamma\right)^{k-1-m}\right]}.
\end{aligned}
\]
Since $k-1-m\ge1$ and $\gamma\in\left(0,1\right)$, we can bound
\[
1-\left(1-\gamma\right)^{k-1-m}\ge1-\left(1-\gamma\right)=\gamma
\]
and, as a consequence, we obtain
\[
\sum_{k=m+1}^{n}\frac{a_{m}\left(1\right)}{a_{k-1}\left(1\right)}\le1+\sum_{k=m+2}^{n}C^{-\gamma\left(\frac{1}{1-\gamma}\right)^{k-1}}=1+\sum_{k=m+1}^{n-1}C^{-\gamma\left(\frac{1}{1-\gamma}\right)^{k}}=1+\sum_{k=m+1}^{n-1}\exp\left(-\gamma\left(\frac{1}{1-\gamma}\right)^{k}\ln\left(C\right)\right).
\]
As $\mathrm{e}^{x}\ge x^{2}$ $\forall x\in\left[0,\infty\right)$
by Lemma \ref{lem:exponential x2} and, equivalently, $\mathrm{e}^{-x}\le\frac{1}{x^{2}}$
$\forall x\in\left[0,\infty\right)$, we infer that
\[
\begin{aligned}\sum_{k=m+1}^{n}\frac{a_{m}\left(1\right)}{a_{k-1}\left(1\right)} & \le1+\sum_{k=m+1}^{n-1}\frac{\left(1-\gamma\right)^{2k}}{\gamma^{2}\ln\left(C\right)^{2}}\le1+\frac{1}{\gamma^{2}\ln\left(C\right)^{2}}\sum_{k=0}^{\infty}\left[\left(1-\gamma\right)^{2}\right]^{k}=\\
 & \le1+\frac{1}{\gamma^{2}\ln\left(C\right)^{2}}\frac{1}{1-\left(1-\gamma\right)^{2}}\le1+\frac{1}{\gamma^{3}\left(2-\gamma\right)\ln\left(C\right)^{2}}.
\end{aligned}
\]
Since $\gamma\ge\frac{1}{2}$ and $\gamma<1$ by Choice \ref{choice:min requirements on C gamma and kmax},
we deduce that
\[
\sum_{k=m+1}^{n}\frac{a_{m}\left(1\right)}{a_{k-1}\left(1\right)}\le1+\frac{8}{\ln\left(C\right)^{2}}.
\]
As long as $C\ge\exp\left(2\,\sqrt{2}\right)\eqqcolon\Upsilon_{3}$,
we will have
\[
C\ge\exp\left(2\,\sqrt{2}\right)\iff\ln\left(C\right)\ge2\,\sqrt{2}=\sqrt{8}\iff\ln\left(C\right)^{2}\ge8,
\]
which means that
\begin{equation}
\sum_{k=m+1}^{n}\frac{a_{m}\left(1\right)}{a_{k-1}\left(1\right)}\le1+1=2.\label{eq:layer centers bound 2}
\end{equation}
Combining \eqref{eq:layer centers general bound}, \eqref{eq:layer centers bound 1}
and \eqref{eq:layer centers bound 2} provides the result.
\end{proof}

\section{\label{sec:bounds for density force}Bounds for density force}

Equation \eqref{eq:Boussinesq density Taylor} tells us exactly how
to compute the force $\widetilde{f_{\rho}^{\left(n\right)}}^{n}\left(t,x\right)$.
We distinguish the following terms:
\begin{equation}
\begin{aligned} & \underbrace{\frac{\partial\widetilde{\rho^{\left(n\right)}}^{n}}{\partial t}\left(t,x\right)}_{\text{time derivative}}+\\
 & +\underbrace{\left(\widetilde{U^{\left(n-1\right)}}^{n}\left(t,x\right)-\widetilde{U^{\left(n-1\right)}}^{n}\left(t,0\right)-\mathrm{J}\widetilde{U^{\left(n-1\right)}}^{n}\left(t,0\right)\cdot\left(\begin{matrix}x_{1}\\
x_{2}
\end{matrix}\right)\right)\cdot\widetilde{\nabla}^{n}\widetilde{\rho^{\left(n\right)}}^{n}\left(t,x\right)}_{\text{transport term}}+\\
 & +\underbrace{\widetilde{u^{\left(n\right)}}^{n}\left(t,x\right)\cdot\widetilde{\nabla}^{n}\widetilde{P^{\left(n-1\right)}}^{n}\left(t,x\right)}_{\text{new transport of old density}}+\underbrace{\widetilde{u^{\left(n\right)}}^{n}\left(t,x\right)\cdot\widetilde{\nabla}^{n}\widetilde{\rho^{\left(n\right)}}^{n}\left(t,x\right)}_{\text{pure cuadratic term}}=\\
= & \widetilde{f_{\rho}^{\left(n\right)}}^{n}\left(t,x\right).
\end{aligned}
\label{eq:decomposition force density}
\end{equation}
We will study each of these terms in a different subsection. In most
bounds, we will need an additional small parameter which will control
how good some approximations are. We will call this parameter $\mu$.

\subsection{\label{subsec:bound density time derivative}Time derivative}

In subsection \ref{subsec:bounding the time derivative of the density},
we already studied the time derivative of the density, albeit perfunctorily.
Actually, the steps we followed until equation \eqref{eq:rigorous bound time derivative density}
are rigorous. In this subsection, we will improve the bounds we found.

In the following Lemma, we will formalize the intuition that, because
$h^{\left(n\right)}\sim O\left(1\right)$, we should have
\[
\int_{t_{n}}^{1}h^{\left(n\right)}\left(s\right)b_{n}\left(s\right)\mathrm{d}s\sim\left(1-t_{n}\right)\max_{s\in\left[t_{n},1\right]}b_{n}\left(s\right),
\]
which was developed during subsection \ref{subsec:bounding the time derivative of the density}.
\begin{lem}
\label{lem:estimate of integral hn bn}Let $n\in\mathbb{N}$ with
$n\ge2$. For every $\mu>0$, provided that $C$ is big enough (let
us say $C\ge\Upsilon\left(\delta,\mu\right)$), we have
\[
\int_{t_{n}}^{1}h_{\varepsilon}^{\left(n\right)}\left(s\right)b_{n}\left(s\right)\mathrm{d}s\gtrsim_{\mu}\frac{1}{Y}C^{\left(1+k_{\max}-\mu-\delta\left(1-\gamma\right)\right)\left(\frac{1}{1-\gamma}\right)^{n}}.
\]
\end{lem}
\begin{proof}
The first step is to restrict the domain of integration to the region
where $h_{\varepsilon}^{\left(n\right)}$ is almost one. By virtue
of Choice \ref{choice:time picture} and equations \eqref{eq:def dhndt second way}
and \eqref{eq:h_eps near one}, we have
\[
\begin{aligned}\int_{t_{n}}^{1}h_{\varepsilon}^{\left(n\right)}b_{n}\left(s\right)\mathrm{d}s & \ge\end{aligned}
\int_{t_{n}+\left(1-t_{n}\right)\zeta}^{t_{n+1}-\left(1-t_{n}\right)\zeta}\underbrace{h_{\varepsilon}^{\left(n\right)}\left(s\right)}_{\ge1-2\varepsilon}b_{n}\left(s\right)\mathrm{d}s\ge\left(1-2\varepsilon\right)\int_{t_{n}+\left(1-t_{n}\right)\zeta}^{t_{n+1}-\left(1-t_{n}\right)\zeta}b_{n}\left(s\right)\mathrm{d}s.
\]
Since $\varepsilon<\frac{1}{4}$ by Choice \ref{choice:time picture},
we get
\[
\int_{t_{n}}^{1}h_{\varepsilon}^{\left(n\right)}b_{n}\left(s\right)\mathrm{d}s\gtrsim\int_{t_{n}+\left(1-t_{n}\right)\zeta}^{t_{n+1}-\left(1-t_{n}\right)\zeta}b_{n}\left(s\right)\mathrm{d}s.
\]
Now, by Choice \ref{choice:anbn},
\[
\int_{t_{n}}^{1}h_{\varepsilon}^{\left(n\right)}b_{n}\left(s\right)\mathrm{d}s\gtrsim\int_{t_{n}+\left(1-t_{n}\right)\zeta}^{t_{n+1}-\left(1-t_{n}\right)\zeta}C^{\left(1+k_{n}\left(s\right)\right)\left(\frac{1}{1-\gamma}\right)^{n}}\mathrm{d}s.
\]
Manipulating the exponent, adding and subtracting multiple terms,
we arrive to
\begin{equation}
\int_{t_{n}}^{1}h_{\varepsilon}^{\left(n\right)}b_{n}\left(s\right)\mathrm{d}s\gtrsim\int_{t_{n}+\left(1-t_{n}\right)\zeta}^{t_{n+1}-\left(1-t_{n}\right)\zeta}C^{\left(1+k_{n}\left(s\right)-\overline{k}_{n}\left(s\right)+\overline{k}_{n}\left(s\right)-k_{\max}\left(1-\left|1-2\frac{s-t_{n}}{1-t_{n}}\right|\right)+k_{\max}\left(1-\left|1-2\frac{s-t_{n}}{1-t_{n}}\right|\right)\right)\left(\frac{1}{1-\gamma}\right)^{n}}\mathrm{d}s.\label{eq:bound integral hb 1}
\end{equation}
On the one hand, by Lemma \ref{lem:monotone convergence kn}, the
convergence of $\overline{k}_{n}\left(s\right)$ to $k_{\max}\left(1-\left|1-2\frac{s-t_{n}}{1-t_{n}}\right|\right)$
is monotonously decreasing in $n\in\mathbb{N}$. (Notice that we know
that $\overline{k}_{n}\left(s\right)\xrightarrow[n\to\infty]{}k_{\max}\left(1-\left|1-2\frac{s-t_{n}}{1-t_{n}}\right|\right)$
by the computations we developed in subsection \ref{subsec:completion of the toy model}:
see equation \eqref{eq:kn second version approximate}). Consequently,
\[
\overline{k}_{n}\left(s\right)-k_{\max}\left(1-\left|1-2\frac{s-t_{n}}{1-t_{n}}\right|\right)\ge0.
\]
On the other hand, by Proposition \ref{prop:convergence kn to ideal model}
taking $\beta=\beta'=\frac{1}{2}$, which we can apply as long as
$C$ is big enough (let us say $C\ge\Upsilon_{1}\left(\frac{1}{2},\frac{1}{2},\delta\right)$),
\[
\left|k_{n}\left(s\right)-\overline{k}_{n}\left(s\right)\right|\lesssim_{\delta}C^{-\frac{1}{8}\delta\gamma\left(\frac{1}{1-\gamma}\right)^{n-1}}\quad\forall s\in\left[t_{n},1\right].
\]
As we can bound
\[
C^{-\frac{1}{8}\delta\gamma\left(\frac{1}{1-\gamma}\right)^{n-1}}\le C^{-\frac{1}{8}\delta\frac{\gamma}{1-\gamma}}\le C^{-\frac{1}{8}\delta}
\]
because $n\ge2$, $\gamma\ge\frac{1}{2}$ (by Choice \ref{choice:min requirements on C gamma and kmax})
and the function $x\to\frac{x}{1-x}$ is increasing, taking $C$ bigger
than some $\Upsilon_{2}\left(\delta,\mu\right)>0$ we can ensure that
\[
\left|k_{n}\left(s\right)-\overline{k}_{n}\left(s\right)\right|\le\frac{\mu}{2}\quad\forall s\in\left[t_{n},1\right].
\]

Thereby, \eqref{eq:bound integral hb 1} becomes
\[
\int_{t_{n}}^{1}h_{\varepsilon}^{\left(n\right)}b_{n}\left(s\right)\mathrm{d}s\gtrsim\int_{t_{n}+\left(1-t_{n}\right)\zeta}^{t_{n+1}-\left(1-t_{n}\right)\zeta}C^{\left(1+k_{\max}\left(1-\left|1-2\frac{s-t_{n}}{1-t_{n}}\right|\right)-\frac{\mu}{2}\right)\left(\frac{1}{1-\gamma}\right)^{n}}\mathrm{d}s.
\]
Let us restrict the integration domain to the region where
\[
1+k_{\max}\left(1-\left|1-2\frac{s-t_{n}}{1-t_{n}}\right|\right)\ge1+k_{\max}-\frac{\mu}{2}.
\]
Solving for $s$ provides
\[
\begin{aligned} & -\left|1-2\frac{s-t_{n}}{1-t_{n}}\right|\ge-\frac{\mu}{2}\iff\left|1-2\frac{s-t_{n}}{1-t_{n}}\right|\le\frac{\mu}{2}\iff\left|\frac{1}{2}-\frac{s-t_{n}}{1-t_{n}}\right|\le\frac{\mu}{4}\iff\\
\iff & s\in\left[t_{n}+\left(1-t_{n}\right)\left(\frac{1}{2}-\frac{\mu}{4}\right),t_{n}+\left(1-t_{n}\right)\left(\frac{1}{2}+\frac{\mu}{4}\right)\right].
\end{aligned}
\]
In this way,
\[
\begin{aligned}\int_{t_{n}}^{1}h_{\varepsilon}^{\left(n\right)}b_{n}\left(s\right)\mathrm{d}s & \gtrsim\int_{t_{n}+\left(1-t_{n}\right)\left(\frac{1}{2}-\frac{\mu}{4}\right)}^{t_{n}+\left(1-t_{n}\right)\left(\frac{1}{2}+\frac{\mu}{4}\right)}C^{\left(\overbrace{1+k_{\max}\left(1-\left|1-2\frac{s-t_{n}}{1-t_{n}}\right|\right)}^{\ge1+k_{\max}-\frac{\mu}{2}}-\frac{\mu}{2}\right)\left(\frac{1}{1-\gamma}\right)^{n}}\mathrm{d}s\ge\\
 & \gtrsim\left(1-t_{n}\right)\frac{\mu}{2}C^{\left(1+k_{\max}-\mu\right)\left(\frac{1}{1-\gamma}\right)^{n}}.
\end{aligned}
\]
Employing Choice \ref{choice:Mn}, we deduce that
\[
\int_{t_{n}}^{1}h_{\varepsilon}^{\left(n\right)}b_{n}\left(s\right)\mathrm{d}s\gtrsim_{\mu}\frac{1}{Y}C^{-\delta\left(\frac{1}{1-\gamma}\right)^{n-1}}\mathrm{arccosh}\left(C^{k_{\max}\left(\frac{1}{1-\gamma}\right)^{n}}\right)C^{\left(1+k_{\max}-\mu\right)\left(\frac{1}{1-\gamma}\right)^{n}}.
\]
Notice that
\[
C^{k_{\max}\left(\frac{1}{1-\gamma}\right)^{n}}\ge C^{k_{\max}}\ge2^{\frac{1}{100}}>1
\]
because $n>0$ and $k_{\max}\ge\frac{1}{100}$, $C>2$ and $\gamma\ge\frac{1}{2}$
by Choice \ref{choice:min requirements on C gamma and kmax}. Since
$\mathrm{arccosh}\left(x\right)=1\iff x=0$ and $\mathrm{arccosh}$
is increasing for positive arguments, we can bound
\[
\int_{t_{n}}^{1}h_{\varepsilon}^{\left(n\right)}b_{n}\left(s\right)\mathrm{d}s\gtrsim_{\mu}\frac{1}{Y}C^{\left(1+k_{\max}-\mu-\delta\left(1-\gamma\right)\right)\left(\frac{1}{1-\gamma}\right)^{n}},
\]
which proves the result of the statement.
\end{proof}
\begin{prop}
\label{prop:bound density time derivative}Let $\mu>0$, $\alpha\in\left(0,1\right)$
and $n\ge n_{0}\left(\mu\right)\in\mathbb{N}$. Then, provided that
$C$ is taken sufficiently big (let us say $C\ge\Upsilon\left(\delta,\mu\right)$),
\[
\left|\left|\frac{\partial\widetilde{\rho^{\left(n\right)}}^{n}}{\partial t}\left(t,\left(\phi^{\left(n\right)}\right)^{-1}\left(t,\cdot\right)\right)\right|\right|_{C^{1,\alpha}\left(\mathbb{R}^{2}\right)}\lesssim_{\varphi,\mu}YC^{\left[\alpha-k_{\max}+4\zeta+7\mu+3\delta\right]\left(\frac{1}{1-\gamma}\right)^{n}}.
\]
This expression remains decreasing in $n\in\mathbb{N}$ as long as
\[
\alpha<k_{\max}-4\zeta-7\mu-3\delta.
\]
\end{prop}
\begin{rem}
Taking $k_{\max}$ close to one and $\delta,\zeta$ and $\mu$ small,
which is something we can always do, we may choose $\alpha$ as close
to $1$ as we wish. As we will see later, this means that the time
derivative of the density is not the limiting factor of the regularity
of the force.
\end{rem}
\begin{proof}
All the steps followed in the deduction of equation \eqref{eq:rigorous bound time derivative density}
are rigorous. Therefore, we may start from the aforementioned equation:
\[
\begin{aligned}\left|\left|\frac{\partial\widetilde{\rho^{\left(n\right)}}^{n}}{\partial t}\left(t,\left(\phi^{\left(n\right)}\right)^{-1}\left(t,\cdot\right)\right)\right|\right|_{C^{1,\alpha}\left(\mathbb{R}^{2}\right)} & \lesssim_{\varphi}\left(1+\max\left\{ a_{n}\left(t\right),b_{n}\left(t\right)\right\} \right)^{1+\alpha}z_{n}\frac{\frac{\mathrm{d}h^{\left(n\right)}}{\mathrm{d}t}\left(t\right)}{\int_{t_{n}}^{1}h^{\left(n\right)}\left(s\right)b_{n}\left(s\right)\mathrm{d}s}.\end{aligned}
\]
Since we wrote equation \eqref{eq:rigorous bound time derivative density}
down, we have learned a lot of things about our construction. As a
consequence, we are now able to bound every term that appears in the
equation above.
\begin{enumerate}
\item By Corollary \ref{cor:bound for an(t), bn(t)}, as long as $C$ is
big enough (let us say $C\ge\Upsilon_{1}\left(\delta,\mu\right)$)
and $n_{0}\ge2$, we may assure that
\[
\max\left\{ a_{n}\left(t\right),b_{n}\left(t\right)\right\} \le C^{\left(1+\overline{k}_{n}\left(t\right)+\mu\right)\left(\frac{1}{1-\gamma}\right)^{n}}.
\]
Since
\[
C^{\left(1+\overline{k}_{n}\left(t\right)+\mu\right)\left(\frac{1}{1-\gamma}\right)^{n}}\ge C^{1+\overline{k}_{n}\left(t\right)+\mu}\ge C\ge2
\]
as $n\ge0$, $\overline{k}_{n}\left(t\right)\ge0$ by Corollary \ref{cor:ideal kn non negative}
and $C>2$ by Choice \ref{choice:min requirements on C gamma and kmax},
we deduce that
\[
1+\max\left\{ a_{n}\left(t\right),b_{n}\left(t\right)\right\} \lesssim C^{\left(1+\overline{k}_{n}\left(t\right)+\mu\right)\left(\frac{1}{1-\gamma}\right)^{n}}.
\]
\item By equation \eqref{eq:relation Mn and zn} and Choices \ref{choice:time picture}
and \ref{choice:Mn}, we can write
\[
z_{n}=2M_{n}\cosh\left(2\overbrace{k_{n}\left(1\right)}^{=0}\left(\frac{1}{1-\gamma}\right)^{n}\right)=2M_{n}=2YC^{\delta\left(\frac{1}{1-\gamma}\right)^{n}}.
\]
\item By Choice \ref{choice:time picture} and Lemma \ref{lem:estimate of integral hn bn},
provided that $C$ is chosen large enough (let us say $C\ge\Upsilon_{2}\left(\delta,\mu\right)$),
\[
\int_{t_{n}}^{1}h_{\varepsilon}^{\left(n\right)}\left(s\right)b_{n}\left(s\right)\mathrm{d}s\gtrsim_{\mu}\frac{1}{Y}C^{\left(1+k_{\max}-\mu-\delta\left(1-\gamma\right)\right)\left(\frac{1}{1-\gamma}\right)^{n}}.
\]
\item By equation \eqref{eq:def dhndt second way} and Choice \ref{choice:Mn},
\[
\begin{aligned}\left|\frac{\mathrm{d}h^{\left(n\right)}}{\mathrm{d}t}\left(t\right)\right| & \le\frac{\chi_{\left[t_{n},t_{n}+\left(1-t_{n}\right)\zeta\right]\cup\left[t_{n+1}-\left(1-t_{n}\right)\zeta,t_{n+1}\right]}\left(t\right)}{\left(1-t_{n}\right)\zeta}=\\
 & =\frac{\chi_{\left[t_{n},t_{n}+\left(1-t_{n}\right)\zeta\right]\cup\left[t_{n+1}-\left(1-t_{n}\right)\zeta,t_{n+1}\right]}\left(t\right)}{Y}\frac{C^{\delta\left(\frac{1}{1-\gamma}\right)^{n-1}}}{\mathrm{arccosh}\left(C^{k_{\max}\left(\frac{1}{1-\gamma}\right)^{n}}\right)}.
\end{aligned}
\]
Since $\mathrm{arccosh}$ is increasing for positive arguments, $n>0$
and $C>2$ and $k_{\max}\ge\frac{1}{100}$ by Choice \ref{choice:min requirements on C gamma and kmax},
we can bound
\begin{equation}
\mathrm{arccosh}\left(C^{k_{\max}\left(\frac{1}{1-\gamma}\right)^{n}}\right)\ge\mathrm{arccosh}\left(C^{\frac{1}{100}}\right)\ge\mathrm{arccosh}\left(2^{\frac{1}{100}}\right).\label{eq:argument arccosh}
\end{equation}
Thereby,
\[
\left|\frac{\mathrm{d}h^{\left(n\right)}}{\mathrm{d}t}\left(t\right)\right|\lesssim\frac{\chi_{\left[t_{n},t_{n}+\left(1-t_{n}\right)\zeta\right]\cup\left[t_{n+1}-\left(1-t_{n}\right)\zeta,t_{n+1}\right]}\left(t\right)}{Y}C^{\delta\left(\frac{1}{1-\gamma}\right)^{n-1}}.
\]
\end{enumerate}
Combining all these factors, we arrive to
\[
\begin{aligned}\left|\left|\frac{\partial\widetilde{\rho^{\left(n\right)}}^{n}}{\partial t}\left(t,\left(\phi^{\left(n\right)}\right)^{-1}\left(t,\cdot\right)\right)\right|\right|_{C^{1,\alpha}\left(\mathbb{R}^{2}\right)} & \lesssim_{\varphi,\mu}C^{\left(1+\alpha\right)\left(1+\overline{k}_{n}\left(t\right)+\mu\right)\left(\frac{1}{1-\gamma}\right)^{n}}YC^{\delta\left(\frac{1}{1-\gamma}\right)^{n}}\cdot\\
 & \quad\cdot\frac{\chi_{\left[t_{n},t_{n}+\left(1-t_{n}\right)\zeta\right]\cup\left[t_{n+1}-\left(1-t_{n}\right)\zeta,t_{n+1}\right]}\left(t\right)\frac{1}{Y}C^{\delta\left(\frac{1}{1-\gamma}\right)^{n-1}}}{\frac{1}{Y}C^{\left(1+k_{\max}-\mu-\delta\left(1-\gamma\right)\right)\left(\frac{1}{1-\gamma}\right)^{n}}}=\\
 & \lesssim_{\varphi,\mu}YC^{\left[\left(1+\alpha\right)\left(1+\overline{k}_{n}\left(t\right)+\mu\right)+\mu+2\delta\left(1-\gamma\right)+\delta-1-k_{\max}\right]\left(\frac{1}{1-\gamma}\right)^{n}}\cdot\\
 & \quad\cdot\chi_{\left[t_{n},t_{n}+\left(1-t_{n}\right)\zeta\right]\cup\left[t_{n+1}-\left(1-t_{n}\right)\zeta,t_{n+1}\right]}\left(t\right).
\end{aligned}
\]
Given that $\gamma\in\left(0,1\right)$, we may write
\begin{equation}
\begin{aligned}\left|\left|\frac{\partial\widetilde{\rho^{\left(n\right)}}^{n}}{\partial t}\left(t,\left(\phi^{\left(n\right)}\right)^{-1}\left(t,\cdot\right)\right)\right|\right|_{C^{1,\alpha}\left(\mathbb{R}^{2}\right)} & \lesssim_{\varphi,\mu}YC^{\left[\left(1+\alpha\right)\left(1+\overline{k}_{n}\left(t\right)+\mu\right)+\mu+3\delta-1-k_{\max}\right]\left(\frac{1}{1-\gamma}\right)^{n}}\cdot\\
 & \quad\cdot\chi_{\left[t_{n},t_{n}+\left(1-t_{n}\right)\zeta\right]\cup\left[t_{n+1}-\left(1-t_{n}\right)\zeta,t_{n+1}\right]}\left(t\right).
\end{aligned}
\label{eq:time derivative v0}
\end{equation}

Now, we can use the time characteristic function to bound the size
of $\overline{k}_{n}\left(t\right)$. Indeed, recall that, as we explained
in subsection \ref{subsec:completion of the toy model}, we expect
$k_{n}\left(t\right)$ to be small when $\frac{\mathrm{d}h^{\left(n\right)}}{\mathrm{d}t}\neq0$.
In subsection \ref{subsec:completion of the toy model}, we saw that
the convergence of $\overline{k}_{n}\left(t\right)$ to $k_{\max}\left(1-\left|1-2\frac{t-t_{n}}{1-t_{n}}\right|\right)$
was uniform in any closed time interval contained in $\left[t_{n},1\right]$
that did not contain $t_{n}+\frac{1-t_{n}}{2}$. As $\zeta<\frac{1}{4}$
by Choice \ref{choice:time picture}, we can assure that
\[
t_{n}+\frac{1-t_{n}}{2}\notin\left[t_{n},t_{n}+\left(1-t_{n}\right)\zeta\right]\cup\left[t_{n+1}-\left(1-t_{n}\right)\zeta,t_{n+1}\right]\eqqcolon\mathcal{T}_{n}.
\]
Therefore, taking $n$ big enough, let us say $n\ge n_{0}\left(\mu\right)$,
we can guarantee that
\begin{equation}
\overline{k}_{n}\left(t\right)\le k_{\max}\left(1-\left|1-2\frac{t-t_{n}}{1-t_{n}}\right|\right)+\mu\quad\forall t\in\mathcal{T}_{n}.\label{eq:time derivative bound kn(t)}
\end{equation}
Moreover, $\forall t\in\left[t_{n},t_{n}+\left(1-t_{n}\right)\zeta\right]$,
\begin{equation}
1-\left|1-2\frac{t-t_{n}}{1-t_{n}}\right|=1-\left[1-2\frac{t-t_{n}}{1-t_{n}}\right]=2\frac{t-t_{n}}{1-t_{n}}\le2\zeta\label{eq:time derivative bound kn(t) sub 1}
\end{equation}
and, $\forall t\in\left[t_{n+1}-\left(1-t_{n}\right)\zeta,t_{n+1}\right]$,
\[
\begin{aligned}1-\left|1-2\frac{t-t_{n}}{1-t_{n}}\right| & =1+1-2\frac{t-t_{n}}{1-t_{n}}\le2-2\frac{t_{n+1}-t_{n}-\left(1-t_{n}\right)\zeta}{1-t_{n}}=\\
 & \leq2-2\frac{1-t_{n}-\left(1-t_{n}\right)\zeta}{1-t_{n}}+2\frac{1-t_{n+1}}{1-t_{n}}=2\zeta+\frac{1-t_{n+1}}{1-t_{n}}.
\end{aligned}
\]
Making use of Choice \ref{choice:Mn}, we conclude that, $\forall t\in\left[t_{n+1}-\left(1-t_{n}\right)\zeta,t_{n+1}\right]$,
\begin{equation}
\begin{aligned}1-\left|1-2\frac{t-t_{n}}{1-t_{n}}\right| & \leq2\zeta+\frac{\frac{1}{Y}C^{-\delta\left(\frac{1}{1-\gamma}\right)^{n}}\mathrm{arccosh}\left(C^{k_{\max}\left(\frac{1}{1-\gamma}\right)^{n+1}}\right)}{\frac{1}{Y}C^{-\delta\left(\frac{1}{1-\gamma}\right)^{n-1}}\mathrm{arccosh}\left(C^{k_{\max}\left(\frac{1}{1-\gamma}\right)^{n}}\right)}=\\
 & \leq2\zeta+C^{-\delta\left(1-\left(1-\gamma\right)\right)\left(\frac{1}{1-\gamma}\right)^{n}}\frac{\mathrm{arccosh}\left(C^{k_{\max}\left(\frac{1}{1-\gamma}\right)^{n+1}}\right)}{\mathrm{arccosh}\left(C^{k_{\max}\left(\frac{1}{1-\gamma}\right)^{n}}\right)}=\\
 & \leq2\zeta+\underbrace{C^{-\delta\gamma\left(\frac{1}{1-\gamma}\right)^{n}}\frac{\mathrm{arccosh}\left(C^{k_{\max}\left(\frac{1}{1-\gamma}\right)^{n+1}}\right)}{\mathrm{arccosh}\left(C^{k_{\max}\left(\frac{1}{1-\gamma}\right)^{n}}\right)}}_{\eqqcolon I}.
\end{aligned}
\label{eq:time derivative bound kn(t) sub 2 v0}
\end{equation}
Applying Lemma \ref{lem:arccosh by ln} in the numerator and repeating
the same argument as in equation \eqref{eq:argument arccosh} for
the denominator, we may bound
\[
I\lesssim C^{-\delta\gamma\left(\frac{1}{1-\gamma}\right)^{n}}\left[\ln\left(2\right)+k_{\max}\left(\frac{1}{1-\gamma}\right)^{n+1}\ln\left(C\right)\right].
\]
Following the same procedure as in the proof of Proposition \ref{prop:time convergence}
from equation \eqref{eq:time convergence the log 1} to equation \eqref{eq:time convergence the log 2},
we can write
\[
I\lesssim\left(\frac{1}{1-\gamma}\right)^{n+1}\ln\left(C\right)C^{-\delta\gamma\left(\frac{1}{1-\gamma}\right)^{n}}.
\]
Now, using the fact that $\gamma\ge\frac{1}{2}$ (by Choice \ref{choice:min requirements on C gamma and kmax})
and recurring to Lemma \ref{lem:exponetial superexponential bound}
with $a\leftarrow\frac{1}{1-\gamma}$, $b\leftarrow\delta\gamma$,
we obtain that
\[
I\lesssim_{\beta,\delta}\frac{C^{-\beta\delta\gamma\left(\frac{1}{1-\gamma}\right)^{n}}}{\ln\left(C\right)}\le\frac{1}{\ln\left(C\right)}.
\]
Then, taking $\beta=\frac{1}{2}$, as long as $C\ge\Upsilon_{3}\left(\delta,\mu\right)$,
we can make $I\leq\mu$. Introducing this finding back into equation
\eqref{eq:time derivative bound kn(t) sub 2 v0} and combining equations
\eqref{eq:time derivative bound kn(t)}, \eqref{eq:time derivative bound kn(t) sub 1}
and \eqref{eq:time derivative bound kn(t) sub 2 v0} provides
\[
\overline{k}_{n}\left(t\right)\le k_{\max}\left(2\zeta+\mu\right)+\mu\quad\forall t\in\mathcal{T}_{n}.
\]
As $k_{\max}\le1$ by Choice \ref{choice:ideal kn}, we conclude that
\begin{equation}
\overline{k}_{n}\left(t\right)\le2\zeta+2\mu\quad\forall t\in\mathcal{T}_{n}.\label{eq:time derivative bound kn(t) v1}
\end{equation}

Thanks to equation \eqref{eq:time derivative bound kn(t) v1}, \eqref{eq:time derivative v0}
becomes
\[
\begin{aligned}\left|\left|\frac{\partial\widetilde{\rho^{\left(n\right)}}^{n}}{\partial t}\left(t,\left(\phi^{\left(n\right)}\right)^{-1}\left(t,\cdot\right)\right)\right|\right|_{C^{1,\alpha}\left(\mathbb{R}^{2}\right)} & \lesssim_{\varphi,\mu}YC^{\left[\left(1+\alpha\right)\left(1+2\zeta+3\mu\right)+\mu+3\delta-1-k_{\max}\right]\left(\frac{1}{1-\gamma}\right)^{n}}=\\
 & \lesssim_{\varphi,\mu}YC^{\left[\alpha-k_{\max}+\left(1+\alpha\right)\left(2\zeta+3\mu\right)+\mu+3\delta\right]\left(\frac{1}{1-\gamma}\right)^{n}}.
\end{aligned}
\]
As $\alpha\le1$, we can write
\[
\left|\left|\frac{\partial\widetilde{\rho^{\left(n\right)}}^{n}}{\partial t}\left(t,\left(\phi^{\left(n\right)}\right)^{-1}\left(t,\cdot\right)\right)\right|\right|_{C^{1,\alpha}\left(\mathbb{R}^{2}\right)}\lesssim_{\varphi,\mu}YC^{\left[\alpha-k_{\max}+4\zeta+7\mu+3\delta\right]\left(\frac{1}{1-\gamma}\right)^{n}}.
\]
\end{proof}

\subsection{Pure quadratic term}
\begin{prop}
\label{prop:form of gradient of rho}Let $n\in\mathbb{N}$. The tilde
gradient of the density of layer $n$ is given by
\[
\begin{aligned}\widetilde{\nabla}^{n}\widetilde{\rho^{\left(n\right)}}^{n}\left(t,x\right)= & -2M_{n}\frac{h^{\left(n\right)}\left(t\right)}{\int_{t_{n}}^{1}h^{\left(n\right)}\left(s\right)b_{n}\left(s\right)\mathrm{d}s}\left[\varphi\left(\lambda_{n}x_{1}\right)\varphi\left(\lambda_{n}x_{2}\right)\left(\begin{matrix}a_{n}\left(t\right)\cos\left(x_{1}\right)\cos\left(x_{2}\right)\\
-b_{n}\left(t\right)\sin\left(x_{1}\right)\sin\left(x_{2}\right)
\end{matrix}\right)+\right.\\
 & \quad\left.+\lambda_{n}\left(\begin{matrix}a_{n}\left(t\right)\varphi'\left(\lambda_{n}x_{1}\right)\varphi\left(\lambda_{n}x_{2}\right)\\
b_{n}\left(t\right)\varphi\left(\lambda_{n}x_{1}\right)\varphi'\left(\lambda_{n}x_{2}\right)
\end{matrix}\right)\sin\left(x_{1}\right)\cos\left(x_{2}\right)\right].
\end{aligned}
\]
Furthermore, let $\mu>0$ and $n\ge2$. Provided that $C$ is large
enough (lets us say $C\ge\Upsilon\left(\delta,\mu\right)$), we can
guarantee that
\[
\begin{aligned}\left|\left|\widetilde{\nabla}^{n}\widetilde{\rho^{\left(n\right)}}^{n}\left(t,\left(\phi^{\left(n\right)}\right)^{-1}\left(t,\cdot\right)\right)\right|\right|_{L^{\infty}\left(\mathbb{R}^{2};\mathbb{R}^{2}\right)} & \lesssim_{\varphi,\mu}YC^{\left[2\delta+2\mu\right]\left(\frac{1}{1-\gamma}\right)^{n}},\\
\left|\left|\widetilde{\nabla}^{n}\widetilde{\rho^{\left(n\right)}}^{n}\left(t,\left(\phi^{\left(n\right)}\right)^{-1}\left(t,\cdot\right)\right)\right|\right|_{\dot{C}^{\alpha}\left(\mathbb{R}^{2};\mathbb{R}^{2}\right)} & \lesssim_{\varphi,\mu}YC^{\left[\alpha\left(1+k_{\max}\right)+2\delta+3\mu\right]\left(\frac{1}{1-\gamma}\right)^{n}},\\
\left|\left|\widetilde{\nabla}^{n}\widetilde{\rho^{\left(n\right)}}^{n}\left(t,\left(\phi^{\left(n\right)}\right)^{-1}\left(t,\cdot\right)\right)\right|\right|_{\dot{C}^{1}\left(\mathbb{R}^{2};\mathbb{R}^{2}\right)} & \lesssim_{\varphi,\mu}YC^{\left[\left(1+k_{\max}\right)+2\delta+3\mu\right]\left(\frac{1}{1-\gamma}\right)^{n}},\\
\left|\left|\widetilde{\nabla}^{n}\widetilde{\rho^{\left(n\right)}}^{n}\left(t,\left(\phi^{\left(n\right)}\right)^{-1}\left(t,\cdot\right)\right)\right|\right|_{\dot{C}^{1,\alpha}\left(\mathbb{R}^{2};\mathbb{R}^{2}\right)} & \lesssim_{\varphi,\mu}YC^{\left[\left(1+\alpha\right)\left(1+k_{\max}\right)+2\delta+4\mu\right]\left(\frac{1}{1-\gamma}\right)^{n}}.
\end{aligned}
\]
\end{prop}
\begin{proof}
By Choices \ref{choice:density} and \ref{choice:amplitude density},
we have
\[
\widetilde{\rho^{\left(n\right)}}^{n}\left(t,x\right)=-z_{n}\frac{h^{\left(n\right)}\left(t\right)}{\int_{t_{n}}^{1}h^{\left(n\right)}\left(s\right)b_{n}\left(s\right)\mathrm{d}s}\varphi\left(\lambda_{n}x_{1}\right)\varphi\left(\lambda_{n}x_{2}\right)\sin\left(x_{1}\right)\cos\left(x_{2}\right).
\]
Next, we eliminate $z_{n}$ via equation \eqref{eq:relation Mn and zn}
and Choice \ref{choice:time picture}, leading to
\[
\widetilde{\rho^{\left(n\right)}}^{n}\left(t,x\right)=-2M_{n}\frac{h^{\left(n\right)}\left(t\right)}{\int_{t_{n}}^{1}h^{\left(n\right)}\left(s\right)b_{n}\left(s\right)\mathrm{d}s}\varphi\left(\lambda_{n}x_{1}\right)\varphi\left(\lambda_{n}x_{2}\right)\sin\left(x_{1}\right)\cos\left(x_{2}\right).
\]
Differentiating the expression above, we obtain
\begin{equation}
\begin{aligned}\widetilde{\nabla}^{n}\widetilde{\rho^{\left(n\right)}}^{n}\left(t,x\right)= & -2M_{n}\frac{h^{\left(n\right)}\left(t\right)}{\int_{t_{n}}^{1}h^{\left(n\right)}\left(s\right)b_{n}\left(s\right)\mathrm{d}s}\left[\varphi\left(\lambda_{n}x_{1}\right)\varphi\left(\lambda_{n}x_{2}\right)\left(\begin{matrix}a_{n}\left(t\right)\cos\left(x_{1}\right)\cos\left(x_{2}\right)\\
-b_{n}\left(t\right)\sin\left(x_{1}\right)\sin\left(x_{2}\right)
\end{matrix}\right)+\right.\\
 & \quad\left.+\lambda_{n}\left(\begin{matrix}a_{n}\left(t\right)\varphi'\left(\lambda_{n}x_{1}\right)\varphi\left(\lambda_{n}x_{2}\right)\\
b_{n}\left(t\right)\varphi\left(\lambda_{n}x_{1}\right)\varphi'\left(\lambda_{n}x_{2}\right)
\end{matrix}\right)\sin\left(x_{1}\right)\cos\left(x_{2}\right)\right],
\end{aligned}
\label{eq:gradient rho}
\end{equation}
which is the result of the statement.

As $\lambda_{n}\leq1$ by Choice \ref{choice:psin}, inspecting equation
\eqref{eq:gradient rho}, we obtain that
\[
\left|\left|\widetilde{\nabla}^{n}\widetilde{\rho^{\left(n\right)}}^{n}\left(t,\cdot\right)\right|\right|_{L^{\infty}\left(\mathbb{R}^{2};\mathbb{R}^{2}\right)}\lesssim_{\varphi}M_{n}\frac{h^{\left(n\right)}\left(t\right)}{\int_{t_{n}}^{1}h^{\left(n\right)}\left(s\right)b_{n}\left(s\right)\mathrm{d}s}\max\left\{ a_{n}\left(t\right),b_{n}\left(t\right)\right\} .
\]
As long as $C$ is big enough (let us say $C\ge\Upsilon_{1}\left(\delta,\gamma\right)$)
and $n\ge2$, we can apply Corollary \ref{cor:bound for an(t), bn(t)}
and Lemma \ref{lem:estimate of integral hn bn}. These two results,
along with equation \eqref{eq:h_eps near one} and Choice \ref{choice:Mn}
provide
\[
\left|\left|\widetilde{\nabla}^{n}\widetilde{\rho^{\left(n\right)}}^{n}\left(t,\cdot\right)\right|\right|_{L^{\infty}\left(\mathbb{R}^{2};\mathbb{R}^{2}\right)}\lesssim_{\varphi,\mu}YC^{\left[\delta-\left(1+k_{\max}-\mu-\delta\left(1-\gamma\right)\right)+\left(1+\overline{k}_{n}\left(t\right)+\mu\right)\right]\left(\frac{1}{1-\gamma}\right)^{n}}.
\]
Bounding $\overline{k}_{n}\left(t\right)\le k_{\max}$ (see Choice
\ref{choice:ideal kn}) and $1-\gamma\le1$ (see Choice \ref{choice:anbn}),
we arrive to
\[
\left|\left|\widetilde{\nabla}^{n}\widetilde{\rho^{\left(n\right)}}^{n}\left(t,\cdot\right)\right|\right|_{L^{\infty}\left(\mathbb{R}^{2};\mathbb{R}^{2}\right)}\lesssim_{\varphi,\mu}YC^{\left[2\delta+2\mu\right]\left(\frac{1}{1-\gamma}\right)^{n}}.
\]

In view of equation \eqref{eq:gradient rho}, bearing in mind equations
\eqref{eq:property Calpha multiplication} and \eqref{eq:property Calpha composition}
and the fact that $\lambda_{n}\le1$ by Choice \ref{choice:psin},
we deduce that
\[
\left|\left|\widetilde{\nabla}^{n}\widetilde{\rho^{\left(n\right)}}^{n}\left(t,\cdot\right)\right|\right|_{\dot{C}^{\alpha}\left(\mathbb{R}^{2};\mathbb{R}^{2}\right)}\lesssim_{\varphi}M_{n}\frac{h^{\left(n\right)}\left(t\right)}{\int_{t_{n}}^{1}h^{\left(n\right)}\left(s\right)b_{n}\left(s\right)\mathrm{d}s}\max\left\{ a_{n}\left(t\right),b_{n}\left(t\right)\right\} .
\]
Now, equation \eqref{eq:property Calpha composition} implies that
\[
\begin{aligned}\left|\left|\widetilde{\nabla}^{n}\widetilde{\rho^{\left(n\right)}}^{n}\left(t,\left(\phi^{\left(n\right)}\right)^{-1}\left(t,\cdot\right)\right)\right|\right|_{\dot{C}^{\alpha}\left(\mathbb{R}^{2};\mathbb{R}^{2}\right)} & \lesssim\left|\left|\left(\phi^{\left(n\right)}\right)^{-1}\left(t,\cdot\right)\right|\right|_{\dot{C}^{1}\left(\mathbb{R}^{2};\mathbb{R}^{2}\right)}^{\alpha}\left|\left|\widetilde{\nabla}^{n}\widetilde{\rho^{\left(n\right)}}^{n}\left(t,\cdot\right)\right|\right|_{\dot{C}^{\alpha}\left(\mathbb{R}^{2};\mathbb{R}^{2}\right)}\lesssim_{\varphi}\\
 & \lesssim_{\varphi}M_{n}\frac{h^{\left(n\right)}\left(t\right)}{\int_{t_{n}}^{1}h^{\left(n\right)}\left(s\right)b_{n}\left(s\right)\mathrm{d}s}\max\left\{ a_{n}\left(t\right)^{1+\alpha},b_{n}\left(t\right)^{1+\alpha}\right\} .
\end{aligned}
\]
Corollary \ref{cor:bound for an(t), bn(t)}, Lemma \ref{lem:estimate of integral hn bn},
equation \eqref{eq:h_eps near one} and Choice \ref{choice:Mn} provide
\[
\begin{aligned}\left|\left|\widetilde{\nabla}^{n}\widetilde{\rho^{\left(n\right)}}^{n}\left(t,\left(\phi^{\left(n\right)}\right)^{-1}\left(t,\cdot\right)\right)\right|\right|_{\dot{C}^{\alpha}\left(\mathbb{R}^{2};\mathbb{R}^{2}\right)} & \lesssim_{\varphi,\mu}YC^{\left[\delta-\left(1+k_{\max}-\mu-\delta\left(1-\gamma\right)\right)+\left(1+\alpha\right)\left(1+\overline{k}_{n}\left(t\right)+\mu\right)\right]\left(\frac{1}{1-\gamma}\right)^{n}}\end{aligned}
.
\]
Bounding $\overline{k}_{n}\left(t\right)\le k_{\max}$ (see Choice
\ref{choice:ideal kn}), $1-\gamma\le1$ (see Choice \ref{choice:anbn})
and $\alpha\le1$, we infer that
\[
\left|\left|\widetilde{\nabla}^{n}\widetilde{\rho^{\left(n\right)}}^{n}\left(t,\left(\phi^{\left(n\right)}\right)^{-1}\left(t,\cdot\right)\right)\right|\right|_{\dot{C}^{\alpha}\left(\mathbb{R}^{2};\mathbb{R}^{2}\right)}\lesssim_{\varphi,\mu}YC^{\left[\alpha\left(1+k_{\max}\right)+2\delta+3\mu\right]\left(\frac{1}{1-\gamma}\right)^{n}}.
\]

Next, we proceed with the $\left|\left|\cdot\right|\right|_{\dot{C}^{1}\left(\mathbb{R}^{2};\mathbb{R}^{2}\right)}$
norm. On the one hand, differentiating \eqref{eq:gradient rho}, we
obtain
\begin{equation}
{\small \begin{aligned}\nabla\left(\widetilde{\nabla}^{n}\widetilde{\rho^{\left(n\right)}}^{n}\left(t,x\right)\right) & =-2M_{n}\frac{h^{\left(n\right)}\left(t\right)}{\int_{t_{n}}^{1}h^{\left(n\right)}\left(s\right)b_{n}\left(s\right)\mathrm{d}s}\\
 & \quad\left[\left(\begin{matrix}\lambda_{n}\varphi'\left(\lambda_{n}x_{1}\right)\varphi\left(\lambda_{n}x_{2}\right)a_{n}\left(t\right)\cos\left(x_{1}\right)\cos\left(x_{2}\right) & -\lambda_{n}\varphi'\left(\lambda_{n}x_{1}\right)\varphi\left(\lambda_{n}x_{2}\right)b_{n}\left(t\right)\sin\left(x_{1}\right)\sin\left(x_{2}\right)\\
\varphi\left(\lambda_{n}x_{1}\right)\lambda_{n}\varphi'\left(\lambda_{n}x_{2}\right)a_{n}\left(t\right)\cos\left(x_{1}\right)\cos\left(x_{2}\right) & -\varphi\left(\lambda_{n}x_{1}\right)\lambda_{n}\varphi'\left(\lambda_{n}x_{2}\right)b_{n}\left(t\right)\sin\left(x_{1}\right)\sin\left(x_{2}\right)
\end{matrix}\right)+\right.\\
 & \quad+\varphi\left(\lambda_{n}x_{1}\right)\varphi\left(\lambda_{n}x_{2}\right)\left(\begin{matrix}-a_{n}\left(t\right)\sin\left(x_{1}\right)\cos\left(x_{2}\right) & -b_{n}\left(t\right)\cos\left(x_{1}\right)\sin\left(x_{2}\right)\\
-a_{n}\left(t\right)\cos\left(x_{1}\right)\sin\left(x_{2}\right) & -b_{n}\left(t\right)\sin\left(x_{1}\right)\cos\left(x_{2}\right)
\end{matrix}\right)+\\
 & \quad+\lambda_{n}\left(\begin{matrix}a_{n}\left(t\right)\lambda_{n}\varphi''\left(\lambda_{n}x_{1}\right)\varphi\left(\lambda_{n}x_{2}\right) & b_{n}\left(t\right)\lambda_{n}\varphi'\left(\lambda_{n}x_{1}\right)\varphi'\left(\lambda_{n}x_{2}\right)\\
a_{n}\left(t\right)\varphi'\left(\lambda_{n}x_{1}\right)\lambda_{n}\varphi'\left(\lambda_{n}x_{2}\right) & b_{n}\left(t\right)\varphi\left(\lambda_{n}x_{1}\right)\lambda_{n}\varphi''\left(\lambda_{n}x_{2}\right)
\end{matrix}\right)\sin\left(x_{1}\right)\cos\left(x_{2}\right)+\\
 & \quad\left.+\left(\begin{matrix}a_{n}\left(t\right)\varphi'\left(\lambda_{n}x_{1}\right)\varphi\left(\lambda_{n}x_{2}\right)\cos\left(x_{1}\right)\cos\left(x_{2}\right) & b_{n}\left(t\right)\varphi\left(\lambda_{n}x_{1}\right)\varphi'\left(\lambda_{n}x_{2}\right)\cos\left(x_{1}\right)\cos\left(x_{2}\right)\\
-a_{n}\left(t\right)\varphi'\left(\lambda_{n}x_{1}\right)\varphi\left(\lambda_{n}x_{2}\right)\sin\left(x_{1}\right)\sin\left(x_{2}\right) & -b_{n}\left(t\right)\varphi\left(\lambda_{n}x_{1}\right)\varphi'\left(\lambda_{n}x_{2}\right)\sin\left(x_{1}\right)\sin\left(x_{2}\right)
\end{matrix}\right)\right].
\end{aligned}
}\label{eq:gradient gradient rho}
\end{equation}
On the other hand, by the chain rule and equation \eqref{eq:jacobian inverse},
\begin{equation}
\begin{aligned}\nabla\left(\widetilde{\nabla}^{n}\widetilde{\rho^{\left(n\right)}}^{n}\left(t,\left(\phi^{\left(n\right)}\right)^{-1}\left(t,x\right)\right)\right) & =\nabla\left(\phi^{\left(n\right)}\right)^{-1}\left(t,x\right)\cdot\nabla\left(\widetilde{\nabla}^{n}\widetilde{\rho^{\left(n\right)}}^{n}\right)\left(t,\left(\phi^{\left(n\right)}\right)^{-1}\left(t,x\right)\right)=\\
 & =\left(\begin{matrix}a_{n}\left(t\right) & 0\\
0 & b_{n}\left(t\right)
\end{matrix}\right)\cdot\nabla\left(\widetilde{\nabla}^{n}\widetilde{\rho^{\left(n\right)}}^{n}\right)\left(t,\left(\phi^{\left(n\right)}\right)^{-1}\left(t,x\right)\right).
\end{aligned}
\label{eq:chain rule gradient gradient rho}
\end{equation}
Equations \eqref{eq:gradient gradient rho} and \eqref{eq:chain rule gradient gradient rho}
and Choice \ref{choice:psin} imply
\[
\left|\left|\widetilde{\nabla}^{n}\widetilde{\rho^{\left(n\right)}}^{n}\left(t,\left(\phi^{\left(n\right)}\right)^{-1}\left(t,\cdot\right)\right)\right|\right|_{\dot{C}^{1}\left(\mathbb{R}^{2};\mathbb{R}^{2}\right)}\lesssim_{\varphi}M_{n}\frac{h^{\left(n\right)}\left(t\right)}{\int_{t_{n}}^{1}h^{\left(n\right)}\left(s\right)b_{n}\left(s\right)\mathrm{d}s}\max\left\{ a_{n}\left(t\right)^{2},b_{n}\left(t\right)^{2}\right\} .
\]
Then, we may proceed formally like in the $\left|\left|\cdot\right|\right|_{\dot{C}^{\alpha}\left(\mathbb{R}^{2};\mathbb{R}^{2}\right)}$
case with $\alpha=1$, which provides the bound of the statement.

Lastly, we consider the $\left|\left|\cdot\right|\right|_{\dot{C}^{1,\alpha}\left(\mathbb{R}^{2};\mathbb{R}^{2}\right)}$
seminorm. In view of equation \eqref{eq:gradient gradient rho}, employing
Choice \ref{choice:psin}, it is clear that
\[
\left|\left|\nabla\left(\widetilde{\nabla}^{n}\widetilde{\rho^{\left(n\right)}}^{n}\right)\left(t,\cdot\right)\right|\right|_{\dot{C}^{1,\alpha}\left(\mathbb{R}^{2};\mathbb{R}^{2}\right)}\lesssim_{\varphi}2M_{n}\frac{h^{\left(n\right)}\left(t\right)}{\int_{t_{n}}^{1}h^{\left(n\right)}\left(s\right)b_{n}\left(s\right)\mathrm{d}s}\max\left\{ a_{n}\left(t\right),b_{n}\left(t\right)\right\} .
\]
This, along with equations \eqref{eq:property Calpha composition},
\eqref{eq:chain rule gradient gradient rho} and \eqref{eq:jacobian inverse},
allows us to write
\[
\begin{aligned}\left|\left|\widetilde{\nabla}^{n}\widetilde{\rho^{\left(n\right)}}^{n}\left(t,\left(\phi^{\left(n\right)}\right)^{-1}\left(t,\cdot\right)\right)\right|\right|_{\dot{C}^{1,\alpha}\left(\mathbb{R}^{2};\mathbb{R}^{2}\right)} & =\left|\left|\nabla\left(\widetilde{\nabla}^{n}\widetilde{\rho^{\left(n\right)}}^{n}\left(t,\left(\phi^{\left(n\right)}\right)^{-1}\left(t,\cdot\right)\right)\right)\right|\right|_{\dot{C}^{\alpha}\left(\mathbb{R}^{2};\mathbb{R}^{2\times2}\right)}=\\
 & =\left|\left|\left(\begin{matrix}a_{n}\left(t\right) & 0\\
0 & b_{n}\left(t\right)
\end{matrix}\right)\cdot\nabla\left(\widetilde{\nabla}^{n}\widetilde{\rho^{\left(n\right)}}^{n}\right)\left(t,\left(\phi^{\left(n\right)}\right)^{-1}\left(t,x\right)\right)\right|\right|_{\dot{C}^{\alpha}\left(\mathbb{R}^{2};\mathbb{R}^{2\times2}\right)}\lesssim\\
 & \lesssim\left|\left|\left(\phi^{\left(n\right)}\right)^{-1}\left(t,\cdot\right)\right|\right|_{\dot{C}^{1}\left(\mathbb{R}^{2};\mathbb{R}^{2}\right)}^{\alpha}\cdot\\
 & \quad\cdot\left|\left|\left(\begin{matrix}a_{n}\left(t\right) & 0\\
0 & b_{n}\left(t\right)
\end{matrix}\right)\cdot\nabla\left(\widetilde{\nabla}^{n}\widetilde{\rho^{\left(n\right)}}^{n}\right)\left(t,\cdot\right)\right|\right|_{\dot{C}^{\alpha}\left(\mathbb{R}^{2};\mathbb{R}^{2\times2}\right)}\lesssim_{\varphi,\mu}\\
 & \lesssim_{\varphi,\mu}2M_{n}\frac{h^{\left(n\right)}\left(t\right)}{\int_{t_{n}}^{1}h^{\left(n\right)}\left(s\right)b_{n}\left(s\right)\mathrm{d}s}\max\left\{ a_{n}\left(t\right)^{2+\alpha},b_{n}\left(t\right)^{2+\alpha}\right\} .
\end{aligned}
\]
From here, me may proceed like in the $\left|\left|\cdot\right|\right|_{\dot{C}^{\alpha}\left(\mathbb{R}^{2};\mathbb{R}^{2}\right)}$
case, but with $\alpha\leftarrow1+\alpha$. In this way, we obtain
the expression of the statement.
\end{proof}
\begin{prop}
\label{prop:bound density pure quadratic term}Let $\mu>0$, $\alpha\in\left(0,1\right)$
and $n\in\mathbb{N}$ with $n\ge2$. As long as $C$ is taken large
enough (let us say $C\ge\Upsilon\left(\delta,\mu\right)$), the pure
quadratic term
\[
\widetilde{Q_{\rho}^{\left(n\right)}}^{n}\left(t,x\right)\coloneqq\widetilde{u^{\left(n\right)}}^{n}\left(t,x\right)\cdot\widetilde{\nabla}^{n}\widetilde{\rho^{\left(n\right)}}^{n}\left(t,x\right)
\]
satisfies the bounds
\[
\begin{aligned}\left|\left|Q_{\rho}^{\left(n\right)}\left(t,\cdot\right)\right|\right|_{C^{1,\alpha}\left(\mathbb{R}^{2}\right)} & \lesssim_{\varphi,\mu}Y^{3}C^{\left(\frac{1}{1-\gamma}\right)^{n}\left[\alpha-\min\left\{ k_{\max},\left(1-\alpha\right)k_{\max}+\alpha\Lambda,\Lambda-\alpha k_{\max}\right\} +3\delta+3\mu\right]}.\end{aligned}
\]
This expression remains decreasing in $n\in\mathbb{N}$ as long as
\[
\alpha<\min\left\{ k_{\max},\left(1-\alpha\right)k_{\max}+\alpha\Lambda,\Lambda-\alpha k_{\max}\right\} -3\delta-3\mu.
\]
\end{prop}
\begin{proof}
By Propositions \ref{prop:computations vorticity} and \ref{prop:form of gradient of rho},
we obtain
\begin{equation}
\begin{aligned}\widetilde{Q_{\rho}^{\left(n\right)}}^{n}\left(t,x\right) & =-\overbrace{2M_{n}\frac{h^{\left(n\right)}\left(t\right)}{\int_{t_{n}}^{1}h^{\left(n\right)}\left(s\right)b_{n}\left(s\right)\mathrm{d}s}B_{n}\left(t\right)a_{n}\left(t\right)b_{n}\left(t\right)}^{\eqqcolon T^{\left(n\right)}\left(t\right)}\left[\varphi\left(\lambda_{n}x_{1}\right)^{2}\varphi\left(\lambda_{n}x_{2}\right)^{2}\cdot\right.\\
 & \qquad\left.\cdot\left(\sin\left(x_{1}\right)\cos\left(x_{2}\right)\cos\left(x_{1}\right)\cos\left(x_{2}\right)+\cos\left(x_{1}\right)\sin\left(x_{2}\right)\sin\left(x_{1}\right)\sin\left(x_{2}\right)\right)+\lambda_{n}\widetilde{\epsilon^{\left(n\right)}}^{n}\left(t,x\right)\right]=\\
 & =-T^{\left(n\right)}\left(t\right)\left[\varphi\left(\lambda_{n}x_{1}\right)^{2}\varphi\left(\lambda_{n}x_{2}\right)^{2}\sin\left(2x_{1}\right)+\lambda_{n}\widetilde{\epsilon^{\left(n\right)}}^{n}\left(t,x\right)\right],
\end{aligned}
\label{eq:expression Qrho}
\end{equation}
where $\widetilde{\epsilon^{\left(n\right)}}^{n}\left(t,x\right)$
denotes some other terms we do not need to calculate. Nevertheless,
attending to Propositions \ref{prop:computations vorticity} and \ref{prop:form of gradient of rho}
and recurring to equations \eqref{eq:property Ckalpha multiplication}
and \eqref{eq:property Ckalpha composition}, it is not difficult
to see that
\begin{equation}
\left|\left|\widetilde{\epsilon^{\left(n\right)}}^{n}\left(t,\cdot\right)\right|\right|_{C^{1,\alpha}\left(\mathbb{R}^{2}\right)}\lesssim_{\varphi}\left(1+\lambda_{n}\right)^{1+\alpha}\lesssim1,\label{eq:bounds epsilon}
\end{equation}
where we have applied the fact that $\lambda_{n}\leq1$ (see Choice
\ref{choice:psin}) in the last step. Thus, as the $\left|\left|\cdot\right|\right|_{L^{\infty}\left(\mathbb{R}^{2}\right)}$
norm is invariant under diffeomorphisms, we deduce that
\begin{equation}
\left|\left|Q_{\rho}^{\left(n\right)}\left(t,\cdot\right)\right|\right|_{L^{\infty}\left(\mathbb{R}^{2}\right)}=\left|\left|\widetilde{Q_{\rho}^{\left(n\right)}}^{n}\left(t,\cdot\right)\right|\right|_{L^{\infty}\left(\mathbb{R}^{2}\right)}\lesssim_{\varphi}T^{\left(n\right)}\left(t\right).\label{eq:Qrho Linfty}
\end{equation}

One fundamental fact of equation \eqref{eq:expression Qrho} that
must be emphasized is that the main factor (we are using the same
terminology as in the proof of Proposition \ref{prop:form gradient omega})
of the leading order term in $\lambda_{n}$ has no $x_{2}$ dependence.
This will prove to be crucial, because the only zeroth order term
in $\lambda_{n}$ of the $\left|\left|\cdot\right|\right|_{C^{1,\alpha}}$
norm of \eqref{eq:expression Qrho} will depend exclusively on $a_{n}\left(t\right)$
and not on $b_{n}\left(t\right)$, which will lead to an unexpected
``improvement'' in regularity. To compute the $\left|\left|\cdot\right|\right|_{C^{1,\alpha}\left(\mathbb{R}\right)}$
norm of \eqref{eq:expression Qrho} precisely it will be convenient
to have an expression for the derivatives. Clearly,
\[
\begin{aligned}\frac{\partial\widetilde{Q_{\rho}^{\left(n\right)}}^{n}}{\partial x_{1}}\left(t,x\right) & =-T^{\left(n\right)}\left(t\right)\left[2\lambda_{n}\varphi\left(\lambda_{n}x_{1}\right)\varphi'\left(\lambda_{n}x_{1}\right)\varphi\left(\lambda_{n}x_{2}\right)^{2}\sin\left(2x_{1}\right)+2\varphi\left(\lambda_{n}x_{1}\right)^{2}\varphi\left(\lambda_{n}x_{2}\right)^{2}\cos\left(2x_{1}\right)\right.+\\
 & \quad\left.+\lambda_{n}\frac{\partial\widetilde{\epsilon^{\left(n\right)}}^{n}}{\partial x_{1}}\left(t,x\right)\right],\\
\frac{\partial\widetilde{Q_{\rho}^{\left(n\right)}}^{n}}{\partial x_{2}}\left(t,x\right) & =-T^{\left(n\right)}\left(t\right)\left[2\lambda_{n}\varphi\left(\lambda_{n}x_{1}\right)^{2}\varphi\left(\lambda_{n}x_{2}\right)\varphi'\left(\lambda_{n}x_{2}\right)\sin\left(2x_{1}\right)+\lambda_{n}\frac{\partial\widetilde{\epsilon^{\left(n\right)}}^{n}}{\partial x_{2}}\left(t,x\right)\right].
\end{aligned}
\]
It will prove advantageous to work with marginal Hölder seminorms
(see equation \eqref{eq:marginal seminorms}). From the equation above,
bearing in mind equations \eqref{eq:bounds epsilon}, \eqref{eq:property Calpha multiplication}
and \eqref{eq:property Calpha composition} and the fact that $\lambda_{n}\le1$
by Choice \ref{choice:psin}, it follows that
\begin{equation}
\begin{matrix}\begin{aligned}\left|\left|\frac{\partial\widetilde{Q_{\rho}^{\left(n\right)}}^{n}}{\partial x_{1}}\left(t,\cdot\right)\right|\right|_{L^{\infty}\left(\mathbb{R}^{2}\right)} & \lesssim_{\varphi}T^{\left(n\right)}\left(t\right),\\
\left|\left|\frac{\partial\widetilde{Q_{\rho}^{\left(n\right)}}^{n}}{\partial x_{1}}\left(t,\cdot\right)\right|\right|_{\dot{C}_{1}^{\alpha}\left(\mathbb{R}^{2}\right)} & \lesssim_{\varphi}T^{\left(n\right)}\left(t\right),\\
\left|\left|\frac{\partial\widetilde{Q_{\rho}^{\left(n\right)}}^{n}}{\partial x_{1}}\left(t,\cdot\right)\right|\right|_{\dot{C}_{2}^{\alpha}\left(\mathbb{R}^{2}\right)} & \lesssim_{\varphi}\lambda_{n}^{\alpha}T^{\left(n\right)}\left(t\right),
\end{aligned}
 &  & \begin{aligned}\left|\left|\frac{\partial\widetilde{Q_{\rho}^{\left(n\right)}}^{n}}{\partial x_{2}}\left(t,\cdot\right)\right|\right|_{L^{\infty}\left(\mathbb{R}^{2}\right)} & \lesssim_{\varphi}\lambda_{n}T^{\left(n\right)}\left(t\right),\\
\left|\left|\frac{\partial\widetilde{Q_{\rho}^{\left(n\right)}}^{n}}{\partial x_{2}}\left(t,\cdot\right)\right|\right|_{\dot{C}_{1}^{\alpha}\left(\mathbb{R}^{2}\right)} & \lesssim_{\varphi}\lambda_{n}T^{\left(n\right)}\left(t\right),\\
\left|\left|\frac{\partial\widetilde{Q_{\rho}^{\left(n\right)}}^{n}}{\partial x_{2}}\left(t,\cdot\right)\right|\right|_{\dot{C}_{2}^{\alpha}\left(\mathbb{R}^{2}\right)} & \lesssim_{\varphi}\lambda_{n}T^{\left(n\right)}\left(t\right).
\end{aligned}
\end{matrix}\label{eq:Qrho C0}
\end{equation}
Next in line would be to use \eqref{eq:Qrho C0} to find bounds for
$\frac{\partial Q^{\left(n\right)}}{\partial x_{1}}$ and $\frac{\partial Q^{\left(n\right)}}{\partial x_{2}}$.
To achieve this, first, we need to relate $\frac{\partial\widetilde{Q_{\rho}^{\left(n\right)}}^{n}}{\partial x_{1}}$
and $\frac{\partial Q_{\rho}^{\left(n\right)}}{\partial x_{1}}$.
We may do this through equations \eqref{eq:cv der esp} and \eqref{eq:jacobian inverse},
which allow us to write
\begin{equation}
\frac{\partial Q_{\rho}^{\left(n\right)}}{\partial x_{1}}\left(t,x\right)=a_{n}\left(t\right)\frac{\partial\widetilde{Q_{\rho}^{\left(n\right)}}^{n}}{\partial x_{1}}\left(t,\left(\phi^{\left(n\right)}\right)^{-1}\left(t,x\right)\right),\quad\frac{\partial Q_{\rho}^{\left(n\right)}}{\partial x_{2}}\left(t,x\right)=b_{n}\left(t\right)\frac{\partial\widetilde{Q_{\rho}^{\left(n\right)}}^{n}}{\partial x_{2}}\left(t,\left(\phi^{\left(n\right)}\right)^{-1}\left(t,x\right)\right).\label{eq:relation Qrho tilde and not tilde}
\end{equation}
Therefore, as the $\left|\left|\cdot\right|\right|_{L^{\infty}\left(\mathbb{R}^{2}\right)}$
is invariant under diffeomorphisms, equations \eqref{eq:Qrho C0}
and \eqref{eq:relation Qrho tilde and not tilde} provide
\begin{equation}
\begin{aligned}\left|\left|\frac{\partial Q_{\rho}^{\left(n\right)}}{\partial x_{1}}\left(t,\cdot\right)\right|\right|_{L^{\infty}\left(\mathbb{R}^{2}\right)} & \lesssim_{\varphi}T^{\left(n\right)}\left(t\right)a_{n}\left(t\right)^ {},\\
\left|\left|\frac{\partial Q_{\rho}^{\left(n\right)}}{\partial x_{2}}\left(t,\cdot\right)\right|\right|_{L^{\infty}\left(\mathbb{R}^{2}\right)} & \lesssim_{\varphi}T^{\left(n\right)}\left(t\right)\lambda_{n}b_{n}\left(t\right)^ {}.
\end{aligned}
\label{eq:Qrho C1}
\end{equation}
As for $\left|\left|\frac{\partial Q^{\left(n\right)}}{\partial x_{1}}\left(t,\cdot\right)\right|\right|_{\dot{C}^{\alpha}\left(\mathbb{R}^{2}\right)}$
and $\left|\left|\frac{\partial Q^{\left(n\right)}}{\partial x_{2}}\left(t,\cdot\right)\right|\right|_{\dot{C}^{\alpha}\left(\mathbb{R}^{2}\right)}$,
recall that $\left(\phi_{1}^{\left(n\right)}\right)^{-1}\left(t,x\right)$
only depends on $x_{1}$ and that $\left(\phi_{2}^{\left(n\right)}\right)^{-1}\left(t,x\right)$
only depends on $x_{2}$. Then, equations \eqref{eq:relation Qrho tilde and not tilde}
and \eqref{eq:property Calpha composition} lead to
\[
\begin{aligned}\left|\left|\frac{\partial Q_{\rho}^{\left(n\right)}}{\partial x_{1}}\left(t,\cdot\right)\right|\right|_{\dot{C}_{1}^{\alpha}\left(\mathbb{R}^{2}\right)} & \leq a_{n}\left(t\right)\left|\left|\left(\phi_{1}^{\left(n\right)}\right)^{-1}\left(t,\cdot\right)\right|\right|_{\dot{C}^{1}\left(\mathbb{R}\right)}^{\alpha}\left|\left|\frac{\partial\widetilde{Q_{\rho}^{\left(n\right)}}^{n}}{\partial x_{1}}\left(t,\cdot\right)\right|\right|_{\dot{C}_{1}^{\alpha}\left(\mathbb{R}\right)},\\
\left|\left|\frac{\partial Q_{\rho}^{\left(n\right)}}{\partial x_{1}}\left(t,\cdot\right)\right|\right|_{\dot{C}_{2}^{\alpha}\left(\mathbb{R}^{2}\right)} & \leq a_{n}\left(t\right)\left|\left|\left(\phi_{2}^{\left(n\right)}\right)^{-1}\left(t,\cdot\right)\right|\right|_{\dot{C}^{1}\left(\mathbb{R}\right)}^{\alpha}\left|\left|\frac{\partial\widetilde{Q_{\rho}^{\left(n\right)}}^{n}}{\partial x_{1}}\left(t,\cdot\right)\right|\right|_{\dot{C}_{2}^{\alpha}\left(\mathbb{R}\right)},\\
\left|\left|\frac{\partial Q_{\rho}^{\left(n\right)}}{\partial x_{2}}\left(t,\cdot\right)\right|\right|_{\dot{C}_{1}^{\alpha}\left(\mathbb{R}^{2}\right)} & \leq b_{n}\left(t\right)\left|\left|\left(\phi_{1}^{\left(n\right)}\right)^{-1}\left(t,\cdot\right)\right|\right|_{\dot{C}^{1}\left(\mathbb{R}\right)}^{\alpha}\left|\left|\frac{\partial\widetilde{Q_{\rho}^{\left(n\right)}}^{n}}{\partial x_{2}}\left(t,\cdot\right)\right|\right|_{\dot{C}_{1}^{\alpha}\left(\mathbb{R}\right)}.\\
\left|\left|\frac{\partial Q_{\rho}^{\left(n\right)}}{\partial x_{2}}\left(t,\cdot\right)\right|\right|_{\dot{C}_{2}^{\alpha}\left(\mathbb{R}^{2}\right)} & \leq b_{n}\left(t\right)\left|\left|\left(\phi_{2}^{\left(n\right)}\right)^{-1}\left(t,\cdot\right)\right|\right|_{\dot{C}^{1}\left(\mathbb{R}\right)}^{\alpha}\left|\left|\frac{\partial\widetilde{Q_{\rho}^{\left(n\right)}}^{n}}{\partial x_{2}}\left(t,\cdot\right)\right|\right|_{\dot{C}_{2}^{\alpha}\left(\mathbb{R}\right)}.
\end{aligned}
\]
Exploiting the fact that the marginal Hölder seminorms are bounded
by the ``joint'' Hölder seminorms, equations \eqref{eq:Qrho C0}
and \eqref{eq:jacobian inverse} guarantee that
\[
\begin{aligned}\left|\left|\frac{\partial Q_{\rho}^{\left(n\right)}}{\partial x_{1}}\left(t,\cdot\right)\right|\right|_{\dot{C}_{1}^{\alpha}\left(\mathbb{R}^{2}\right)} & \lesssim_{\varphi}a_{n}\left(t\right)^{1+\alpha}T^{\left(n\right)}\left(t\right),\\
\left|\left|\frac{\partial Q_{\rho}^{\left(n\right)}}{\partial x_{1}}\left(t,\cdot\right)\right|\right|_{\dot{C}_{1}^{\alpha}\left(\mathbb{R}^{2}\right)} & \lesssim_{\varphi}a_{n}\left(t\right)b_{n}\left(t\right)^{\alpha}\lambda_{n}^{\alpha}T^{\left(n\right)}\left(t\right),\\
\left|\left|\frac{\partial Q_{\rho}^{\left(n\right)}}{\partial x_{2}}\left(t,\cdot\right)\right|\right|_{\dot{C}_{1}^{\alpha}\left(\mathbb{R}^{2}\right)} & \lesssim_{\varphi}b_{n}\left(t\right)a_{n}\left(t\right)^{\alpha}\lambda_{n}T^{\left(n\right)}\left(t\right).\\
\left|\left|\frac{\partial Q_{\rho}^{\left(n\right)}}{\partial x_{2}}\left(t,\cdot\right)\right|\right|_{\dot{C}_{2}^{\alpha}\left(\mathbb{R}^{2}\right)} & \lesssim_{\varphi}b_{n}\left(t\right)^{1+\alpha}\lambda_{n}T^{\left(n\right)}\left(t\right).
\end{aligned}
\]
Now, we are able to combine the estimates above into bounds for the
``joint'' Hölder seminorms via equation \eqref{eq:Holder seminorm by marginals}.
Indeed,
\begin{equation}
\begin{aligned}\left|\left|\frac{\partial Q_{\rho}^{\left(n\right)}}{\partial x_{1}}\left(t,\cdot\right)\right|\right|_{\dot{C}^{\alpha}\left(\mathbb{R}^{2}\right)} & \lesssim_{\varphi}T^{\left(n\right)}\left(t\right)a_{n}\left(t\right)\max\left\{ a_{n}\left(t\right)^{\alpha},\lambda_{n}^{\alpha}b_{n}\left(t\right)^{\alpha}\right\} ,\\
\left|\left|\frac{\partial Q_{\rho}^{\left(n\right)}}{\partial x_{2}}\left(t,\cdot\right)\right|\right|_{\dot{C}^{\alpha}\left(\mathbb{R}^{2}\right)} & \lesssim_{\varphi}T^{\left(n\right)}\left(t\right)\lambda_{n}b_{n}\left(t\right)\max\left\{ a_{n}\left(t\right)^{\alpha},b_{n}\left(t\right)^{\alpha}\right\} .
\end{aligned}
\label{eq:Qrho C1alpha}
\end{equation}
Blending together equations \eqref{eq:Qrho Linfty}, \eqref{eq:Qrho C1}
and \eqref{eq:Qrho C1alpha}, we deduce that
\begin{equation}
\begin{aligned}\left|\left|Q_{\rho}^{\left(n\right)}\left(t,\cdot\right)\right|\right|_{C^{1,\alpha}\left(\mathbb{R}^{2}\right)} & \lesssim_{\varphi}T^{\left(n\right)}\left(t\right)\max\left\{ 1,a_{n}\left(t\right),\lambda_{n}b_{n}\left(t\right),a_{n}\left(t\right)^{1+\alpha},a_{n}\left(t\right)\lambda_{n}^{\alpha}b_{n}\left(t\right)^{\alpha},\lambda_{n}b_{n}\left(t\right)a_{n}\left(t\right)^{\alpha},\lambda_{n}b_{n}\left(t\right)^{1+\alpha}\right\} .\end{aligned}
\label{eq:bound quadratic term v0}
\end{equation}
The ``crucial fact'' we introduced at the beginning of this paragraph
is directly responsible for the following: the term $b_{n}\left(t\right)^{1+\alpha}$
(alone) does not appear inside the $\max$. It is also behind the
fact that the term $a_{n}\left(t\right)\lambda_{n}^{\alpha}b_{n}\left(t\right)^{\alpha}$
has a $\lambda_{n}^{\alpha}$ multiplying.

We shall study each term separately. 
\begin{enumerate}
\item As long as $C$ is big enough (let us say $C\ge\Upsilon_{1}\left(\delta,\mu\right)$),
since $n\ge2$, by Corollary \ref{cor:bound for an(t), bn(t)} and
Choice \ref{choice:lambdan}, we obtain
\[
\begin{aligned}\max\left\{ \dots\right\}  & \lesssim C^{\left(\frac{1}{1-\gamma}\right)^{n}\text{exponent}},\\
\text{exponent} & =\max\left\{ 0,1-\overline{k}_{n}\left(t\right)+\mu,1+\overline{k}_{n}\left(t\right)+\mu-\Lambda,\left(1+\alpha\right)\left(1-\overline{k}_{n}\left(t\right)+\mu\right),\right.\\
 & \quad\left(1-\overline{k}_{n}\left(t\right)+\mu\right)+\alpha\left(1+\overline{k}_{n}\left(t\right)+\mu\right)-\Lambda\alpha,\left(1+\overline{k}_{n}\left(t\right)+\mu\right)+\alpha\left(1-\overline{k}_{n}\left(t\right)+\mu\right)-\Lambda,\\
 & \quad\left.\left(1+\alpha\right)\left(1+\overline{k}_{n}\left(t\right)+\mu\right)-\Lambda\right\} .
\end{aligned}
\]
Bounding $0\le\overline{k}_{n}\left(t\right)\le k_{\max}$ (which
we can do thanks to Corollary \ref{cor:ideal kn non negative} and
Choice \ref{choice:ideal kn}), we arrive to
\[
\begin{aligned}\text{exponent} & \le\max\left\{ 0,1+\mu,1+k_{\max}+\mu-\Lambda,\left(1+\alpha\right)\left(1+\mu\right),\right.\\
 & \quad\left(1+\mu\right)+\alpha\left(1+k_{\max}+\mu\right)-\Lambda\alpha,\left(1+k_{\max}+\mu\right)+\alpha\left(1+\mu\right)-\Lambda,\\
 & \quad\left.\left(1+\alpha\right)\left(1+k_{\max}+\mu\right)-\Lambda\right\} .
\end{aligned}
\]
Employing the fact that $\alpha\in\left(0,1\right)$, we see that
some terms of the max are bounded by other terms of the max. Indeed,
\[
\begin{aligned} & 0\le1+\mu\le\left(1+\alpha\right)\left(1+\mu\right),\\
 & \left(1+k_{\max}+\mu\right)-\Lambda\le\left(1+\alpha\right)\left(1+k_{\max}+\mu\right)-\Lambda,\\
 & \left(1+k_{\max}+\mu\right)+\alpha\left(1+\mu\right)-\Lambda=\left(1+\alpha\right)+k_{\max}+\left(1+\alpha\right)\mu-\Lambda\le\\
 & \leq\left(1+\alpha\right)+\left(1+\alpha\right)k_{\max}+\left(1+\alpha\right)\mu-\Lambda=\left(1+\alpha\right)\left(1+k_{\max}+\mu\right)-\Lambda.
\end{aligned}
\]
Thereby,
\[
\text{exponent}\le\max\left\{ \left(1+\alpha\right)\left(1+\mu\right),1+\mu+\alpha\left(1+k_{\max}+\mu-\Lambda\right),\left(1+\alpha\right)\left(1+k_{\max}+\mu\right)-\Lambda\right\} .
\]
Besides, observe that every term contains the summand $\left(1+\alpha\right)\left(1+\mu\right)$,
so we may extract it, leading to
\[
\begin{aligned}\text{exponent} & \le\left(1+\alpha\right)\left(1+\mu\right)+\max\left\{ 0,\alpha\left(k_{\max}-\Lambda\right),\left(1+\alpha\right)k_{\max}-\Lambda\right\} .\end{aligned}
\]
\item By equations \eqref{eq:def dhndt second way} and \eqref{eq:h_eps near one},
we can bound $h^{\left(n\right)}\left(t\right)\le1$.
\item As $n\ge2$, by Lemma \ref{lem:estimate of integral hn bn}, if $C$
is taken large enough (let us say $C\ge\Upsilon_{2}\left(\delta,\mu\right)$),
we have
\[
\int_{t_{n}}^{1}h_{\varepsilon}^{\left(n\right)}\left(s\right)b_{n}\left(s\right)\mathrm{d}s\gtrsim_{\mu}\frac{1}{Y}C^{\left(1+k_{\max}-\mu-\delta\left(1-\gamma\right)\right)\left(\frac{1}{1-\gamma}\right)^{n}}.
\]
\item By Choice \ref{choice:anbncn} and Proposition \ref{prop:time convergence}
taking $\beta=\frac{1}{2}$, provided that $C$ is large enough (let
us say $C\ge\Upsilon_{3}\left(\frac{1}{2},\delta\right)$), we know
that
\[
B_{n}\left(t\right)a_{n}\left(t\right)b_{n}\left(t\right)=B_{n}\left(t\right)a_{n}\left(1\right)b_{n}\left(1\right)\lesssim B_{n}\left(1\right)a_{n}\left(1\right)b_{n}\left(1\right)=M_{n}.
\]
\item Lastly, by Choice \ref{choice:Mn}, it is
\[
M_{n}=YC^{\delta\left(\frac{1}{1-\gamma}\right)^{n}}.
\]
\end{enumerate}
Combining al these observations back into equation \eqref{eq:bound quadratic term v0}
provides
\[
\begin{aligned}\left|\left|Q_{\rho}^{\left(n\right)}\left(t,\cdot\right)\right|\right|_{C^{1,\alpha}\left(\mathbb{R}^{2}\right)} & \lesssim_{\varphi,\mu}\frac{Y^{2}C^{2\delta\left(\frac{1}{1-\gamma}\right)^{n}}}{\frac{1}{Y}C^{\left(1+k_{\max}-\mu-\delta\left(1-\gamma\right)\right)\left(\frac{1}{1-\gamma}\right)^{n}}}C^{\left[\left(1+\alpha\right)\left(1+\mu\right)+\max\left\{ 0,\alpha\left(k_{\max}-\Lambda\right),\left(1+\alpha\right)k_{\max}-\Lambda\right\} \right]\left(\frac{1}{1-\gamma}\right)^{n}}\leq\\
 & \lesssim_{\varphi,\mu}Y^{3}C^{\left(\frac{1}{1-\gamma}\right)^{n}\left[\delta\left(3-\gamma\right)+\mu+\left(1+\alpha\right)\left(1+\mu\right)-1+\max\left\{ -k_{\max},-\left(1-\alpha\right)k_{\max}-\alpha\Lambda,\alpha k_{\max}-\Lambda\right\} \right]}=\\
 & \lesssim_{\varphi,\mu}Y^{3}C^{\left(\frac{1}{1-\gamma}\right)^{n}\left[\delta\left(3-\gamma\right)+\mu+\left(1+\alpha\right)\left(1+\mu\right)-1-\min\left\{ k_{\max},\left(1-\alpha\right)k_{\max}+\alpha\Lambda,\Lambda-\alpha k_{\max}\right\} \right]}.
\end{aligned}
\]
Bounding $\alpha\leq1$ in the terms that contain $\mu$ and taking
into account that $\gamma\in\left(0,1\right)$, we arrive to
\[
\left|\left|Q_{\rho}^{\left(n\right)}\left(t,\cdot\right)\right|\right|_{C^{1,\alpha}\left(\mathbb{R}^{2}\right)}\lesssim_{\varphi,\mu}Y^{3}C^{\left(\frac{1}{1-\gamma}\right)^{n}\left[\alpha-\min\left\{ k_{\max},\left(1-\alpha\right)k_{\max}+\alpha\Lambda,\Lambda-\alpha k_{\max}\right\} +3\delta+3\mu\right]}.
\]
\end{proof}

\subsection{New transport of old density}
\begin{prop}
\label{prop:bound density new transport of old density}Let $n\in\mathbb{N}$.
$\forall t\in\left[t_{n},1\right]$, we have
\[
\widetilde{u^{\left(n\right)}}^{n}\left(t,x\right)\cdot\widetilde{\nabla}^{n}\widetilde{P^{\left(n-1\right)}}^{n}\left(t,x\right)=0.
\]
\end{prop}
\begin{proof}
By Choice \ref{choice:no two simultanous densities}, we have no two
simultaneous densities, i.e., $\widetilde{P^{\left(n-1\right)}}^{n}\left(t,x\right)\equiv0$
$\forall t\in\left[t_{n},1\right]$. This proves the result.
\end{proof}

\subsection{Transport term}

The first step in bounding the transport term is finding estimates
for the spatial second order derivatives of $\widetilde{U^{\left(n-1\right)}}^{n}$.
\begin{lem}
\label{lem:bounds second order derivatives}Let $n\in\mathbb{N}$
with $n\ge2$, $k\in\left\{ 0,1\right\} $ and $\alpha\in\left[0,1\right)$.
We have
\[
\begin{aligned}\left|\left|\frac{\partial^{2}\widetilde{U^{\left(n-1\right)}}^{n}}{\partial x_{1}^{2}}\left(t,\left(\phi^{\left(n\right)}\right)^{-1}\left(t,\cdot\right)\right)\right|\right|_{\dot{C}^{k,\alpha}\left(\mathbb{R}^{2};\mathbb{R}^{2}\right)} & \lesssim_{\delta,\varphi}\frac{F_{k,\alpha}^{\left(n\right)}}{a_{n}\left(t\right)^{2}},\\
\left|\left|\frac{\partial^{2}\widetilde{U^{\left(n-1\right)}}^{n}}{\partial x_{1}\partial x_{2}}\left(t,\left(\phi^{\left(n\right)}\right)^{-1}\left(t,\cdot\right)\right)\right|\right|_{\dot{C}^{k,\alpha}\left(\mathbb{R}^{2};\mathbb{R}^{2}\right)} & \lesssim_{\delta,\varphi}\frac{F_{k,\alpha}^{\left(n\right)}}{a_{n}\left(t\right)b_{n}\left(t\right)},\\
\left|\left|\frac{\partial^{2}\widetilde{U^{\left(n-1\right)}}^{n}}{\partial x_{2}^{2}}\left(t,\left(\phi^{\left(n\right)}\right)^{-1}\left(t,\cdot\right)\right)\right|\right|_{\dot{C}^{k,\alpha}\left(\mathbb{R}^{2};\mathbb{R}^{2}\right)} & \lesssim_{\delta,\varphi}\frac{F_{k,\alpha}^{\left(n\right)}}{b_{n}\left(t\right)^{2}},
\end{aligned}
\]
where
\[
\begin{aligned}F_{k,\alpha}^{\left(n\right)} & =YC^{\left(1+k+\alpha+\delta\right)\left(\frac{1}{1-\gamma}\right)^{n-1}}.\end{aligned}
\]
In the above, the case $k=0$ and $\alpha=0$ should be understood
as the $\left|\left|\cdot\right|\right|_{L^{\infty}\left(\mathbb{R}^{2}\right)}$
norm and the case $k=1$ and $\alpha=0$ should be interpreted as
the $\left|\left|\cdot\right|\right|_{\dot{C}^{1}\left(\mathbb{R}^{2}\right)}$
seminorm.
\end{lem}
\begin{proof}
As we did at the beginning of the proof of Proposition \ref{prop:relation between anbn and jacobian},
we can express $\widetilde{U^{\left(n-1\right)}}^{n}$ as given by
equation \eqref{eq:phi as sum of u layers}:
\[
\widetilde{U^{\left(n-1\right)}}^{n}\left(t,x\right)=\sum_{m=1}^{n-1}\widetilde{u^{\left(m\right)}}^{m}\left(t,\left(\begin{matrix}a_{m}\left(t\right)\left(\phi_{1}^{\left(n\right)}\left(t,x\right)-\phi_{2}^{\left(m\right)}\left(t,0\right)\right)\\
b_{m}\left(t\right)\left(\phi_{2}^{\left(n\right)}\left(t,x\right)-\phi_{2}^{\left(m\right)}\left(t,0\right)\right)
\end{matrix}\right)\right).
\]
On the other hand, by Proposition \ref{prop:computations vorticity},
we have
\[
\begin{aligned}\widetilde{u^{\left(n\right)}}^{n}\left(t,x\right) & =B_{n}\left(t\right)\varphi\left(\lambda_{n}x_{1}\right)\varphi\left(\lambda_{n}x_{2}\right)\left(\begin{matrix}b_{n}\left(t\right)\sin\left(x_{1}\right)\cos\left(x_{2}\right)\\
-a_{n}\left(t\right)\cos\left(x_{1}\right)\sin\left(x_{2}\right)
\end{matrix}\right)+\\
 & \quad+\lambda_{n}B_{n}\left(t\right)\left(\begin{matrix}b_{n}\left(t\right)\varphi\left(\lambda_{n}x_{1}\right)\varphi'\left(\lambda_{n}x_{2}\right)\\
-a_{n}\left(t\right)\varphi'\left(\lambda_{n}x_{1}\right)\varphi\left(\lambda_{n}x_{2}\right)
\end{matrix}\right)\sin\left(x_{1}\right)\sin\left(x_{2}\right).
\end{aligned}
\]
In this way, 
\begin{equation}
\begin{aligned}\widetilde{U_{1}^{\left(n-1\right)}}^{n}\left(t,x\right) & =\sum_{m=1}^{n-1}B_{m}\left(t\right)b_{m}\left(t\right)\left[\varphi\left(\lambda_{m}y_{1}^{\left(n,m\right)}\right)\varphi\left(\lambda_{m}y_{2}^{\left(n,m\right)}\right)\sin\left(y_{1}^{\left(n,m\right)}\right)\cos\left(y_{2}^{\left(n,m\right)}\right)+\right.\\
 & \quad\left.+\lambda_{m}\varphi\left(\lambda_{m}y_{1}^{\left(n,m\right)}\right)\varphi'\left(\lambda_{m}y_{2}^{\left(n,m\right)}\right)\sin\left(y_{1}^{\left(n,m\right)}\right)\sin\left(y_{2}^{\left(n,m\right)}\right)\right],\\
\widetilde{U_{2}^{\left(n-1\right)}}^{n}\left(t,x\right) & =-\sum_{m=1}^{n-1}B_{m}\left(t\right)a_{m}\left(t\right)\left[\varphi\left(\lambda_{m}y_{1}^{\left(n,m\right)}\right)\varphi\left(\lambda_{m}y_{2}^{\left(n,m\right)}\right)\cos\left(y_{1}^{\left(n,m\right)}\right)\sin\left(y_{2}^{\left(n,m\right)}\right)+\right.\\
 & \quad\left.+\lambda_{m}\varphi'\left(\lambda_{m}y_{1}^{\left(n,m\right)}\right)\varphi\left(\lambda_{m}y_{2}^{\left(n,m\right)}\right)\sin\left(y_{1}^{\left(n,m\right)}\right)\sin\left(y_{2}^{\left(n,m\right)}\right)\right],
\end{aligned}
\label{eq:velocity backgrond layers full expression}
\end{equation}
where
\begin{equation}
y^{\left(n,m\right)}=\left(\begin{matrix}a_{m}\left(t\right)\left(\phi_{1}^{\left(n\right)}\left(t,x\right)-\phi_{1}^{\left(m\right)}\left(t,0\right)\right)\\
b_{m}\left(t\right)\left(\phi_{2}^{\left(n\right)}\left(t,x\right)-\phi_{2}^{\left(m\right)}\left(t,0\right)\right)
\end{matrix}\right).\label{eq:definition y}
\end{equation}
Since
\[
\frac{\partial\phi_{1}^{\left(n\right)}}{\partial x_{1}}\left(t,x\right)=\frac{1}{a_{n}\left(t\right)},\quad\frac{\partial\phi_{2}^{\left(n\right)}}{\partial x_{2}}=\frac{1}{b_{n}\left(t\right)},\quad\frac{\partial\phi_{1}^{\left(n\right)}}{\partial x_{2}}\left(t,x\right)=0=\frac{\partial\phi_{2}^{\left(n\right)}}{\partial x_{1}}\left(t,x\right)
\]
by Choice \ref{choice:phin}, we deduce that
\begin{equation}
\frac{\partial y_{1}^{\left(n,m\right)}}{\partial x_{1}}=\frac{a_{m}\left(t\right)}{a_{n}\left(t\right)},\quad\frac{\partial y_{2}^{\left(n,m\right)}}{\partial x_{2}}=\frac{b_{m}\left(t\right)}{b_{n}\left(t\right)},\quad\frac{\partial y_{1}^{\left(n,m\right)}}{\partial x_{2}}=0=\frac{\partial y_{2}^{\left(n,m\right)}}{\partial x_{1}}.\label{eq:derivatives y}
\end{equation}

To continue, we wish to differentiate equation \eqref{eq:velocity backgrond layers full expression}.
To do this, it will be advantageous to study the framework behind
\eqref{eq:velocity backgrond layers full expression}. As we can see,
the first component of $\widetilde{U^{\left(n-1\right)}}^{n}$ is
the sum in $m\in\mathbb{N}$ of $B_{m}\left(t\right)b_{m}\left(t\right)$
times a sum of terms that have the following structure
\begin{equation}
\begin{matrix}\text{sign}\cdot\text{non-negative power of }\lambda_{n}\cdot\underbrace{\left(\text{derivative of }\varphi\right)\left(\lambda_{m}y_{1}^{\left(n,m\right)}\right)\cdot\left(\text{derivative of }\varphi\right)\left(\lambda_{m}y_{2}^{\left(n,m\right)}\right)}_{\text{cutoff factor}}\cdot\\
\cdot\underbrace{\left(\begin{matrix}\sin\\
\text{or}\\
\cos
\end{matrix}\right)\left(y_{1}^{\left(n,m\right)}\right)\left(\begin{matrix}\sin\\
\text{or}\\
\cos
\end{matrix}\right)\left(y_{2}^{\left(n,m\right)}\right)}_{\text{main factor}}
\end{matrix}\label{eq:SPCM}
\end{equation}
The same thing happens for the second component of $\widetilde{U^{\left(n-1\right)}}^{n}$,
except that we have $a_{m}\left(t\right)$ in the place of $b_{m}\left(t\right)$.
A term that has the structure exposed in equation \eqref{eq:SPCM}
will be denoted $\text{SPCM}\left(m,y^{\left(n,m\right)}\right)$,
which stands for ``sign, power, cutoff, main''. In this way, we
can rewrite \eqref{eq:velocity backgrond layers full expression}
as
\[
\begin{aligned}\widetilde{U_{1}^{\left(n-1\right)}}^{n}\left(t,x\right) & =\sum_{m=1}^{n-1}B_{m}\left(t\right)b_{m}\left(t\right)\sum_{2}\text{SPCM}\left(m,y^{\left(n,m\right)}\right),\\
\widetilde{U_{2}^{\left(n-1\right)}}^{n}\left(t,x\right) & =\sum_{m=1}^{n-1}B_{m}\left(t\right)a_{m}\left(t\right)\sum_{2}\text{SPCM}\left(m,y^{\left(n,m\right)}\right),
\end{aligned}
\]
where the subindex of $\sum$ denotes the number of summands. Furthermore,
notice that this structure is preserved by differentiation. Indeed,
since the cross derivatives of $y^{\left(n,m\right)}$ are zero by
equation \eqref{eq:derivatives y}, when differentiating a summand
of the form \eqref{eq:SPCM} with respect to $x_{1}$, we will obtain
two summands: one where the derivative hits the $y_{1}^{\left(n,m\right)}$-dependent
subfactor of the cutoff factor and one where the derivative acts on
the $y_{1}^{\left(n,m\right)}$-dependent subfactor of the main factor.
In any case, the structure of the cutoff factor or the main factor
remains unchanged, since the derivative of a derivative of $\varphi$
is another derivative of $\varphi$ and sines and cosines are obtained
from each other by differentiation. The only difference is that, when
the derivative hits the cutoff factor, we will obtain an extra $\lambda_{m}\frac{a_{m}\left(t\right)}{a_{n}\left(t\right)}$
multiplying, whereas, when the derivative acts on the main factor,
we will obtain just $\frac{a_{m}\left(t\right)}{a_{n}\left(t\right)}$.
A similar argument applies when differentiating with respect to $x_{2}$.
Thereby,
\[
\begin{aligned}\frac{\partial\widetilde{U_{1}^{\left(n-1\right)}}^{n}}{\partial x_{1}}\left(t,x\right) & =\sum_{m=1}^{n-1}B_{m}\left(t\right)\frac{a_{m}\left(t\right)b_{m}\left(t\right)}{a_{n}\left(t\right)}\sum_{4}\text{SPCM}\left(m,y^{\left(n,m\right)}\right),\\
\frac{\partial\widetilde{U_{1}^{\left(n-1\right)}}^{n}}{\partial x_{2}}\left(t,x\right) & =\sum_{m=1}^{n-1}B_{m}\left(t\right)\frac{b_{m}\left(t\right)^{2}}{b_{n}\left(t\right)}\sum_{4}\text{SPCM}\left(m,y^{\left(n,m\right)}\right),\\
\frac{\partial\widetilde{U_{2}^{\left(n-1\right)}}^{n}}{\partial x_{1}}\left(t,x\right) & =\sum_{m=1}^{n-1}B_{m}\left(t\right)\frac{a_{m}\left(t\right)^{2}}{a_{n}\left(t\right)}\sum_{4}\text{SPCM}\left(m,y^{\left(n,m\right)}\right),\\
\frac{\partial\widetilde{U_{2}^{\left(n-1\right)}}^{n}}{\partial x_{2}}\left(t,x\right) & =\sum_{m=1}^{n-1}B_{m}\left(t\right)\frac{a_{m}\left(t\right)b_{m}\left(t\right)}{b_{n}\left(t\right)}\sum_{4}\text{SPCM}\left(m,y^{\left(n,m\right)}\right).
\end{aligned}
\]
Following the same procedure, differentiating again, we obtain
\begin{equation}
\begin{aligned}\frac{\partial^{2}\widetilde{U_{1}^{\left(n-1\right)}}^{n}}{\partial x_{1}^{2}}\left(t,x\right) & =\sum_{m=1}^{n-1}B_{m}\left(t\right)\frac{a_{m}\left(t\right)^{2}b_{m}\left(t\right)}{a_{n}\left(t\right)^{2}}\sum_{8}\text{SPCM}\left(m,y^{\left(n,m\right)}\right),\\
\frac{\partial^{2}\widetilde{U_{1}^{\left(n-1\right)}}^{n}}{\partial x_{1}\partial x_{2}}\left(t,x\right) & =\sum_{m=1}^{n-1}B_{m}\left(t\right)\frac{a_{m}\left(t\right)b_{m}\left(t\right)^{2}}{a_{n}\left(t\right)b_{n}\left(t\right)}\sum_{8}\text{SPCM}\left(m,y^{\left(n,m\right)}\right),\\
\frac{\partial^{2}\widetilde{U_{1}^{\left(n-1\right)}}^{n}}{\partial x_{2}^{2}}\left(t,x\right) & =\sum_{m=1}^{n-1}B_{m}\left(t\right)\frac{b_{m}\left(t\right)^{3}}{b_{n}\left(t\right)^{2}}\sum_{8}\text{SPCM}\left(m,y^{\left(n,m\right)}\right),\\
\frac{\partial^{2}\widetilde{U_{2}^{\left(n-1\right)}}^{n}}{\partial x_{1}^{2}}\left(t,x\right) & =\sum_{m=1}^{n-1}B_{m}\left(t\right)\frac{a_{m}\left(t\right)^{3}}{a_{n}\left(t\right)^{2}}\sum_{8}\text{SPCM}\left(m,y^{\left(n,m\right)}\right),\\
\frac{\partial^{2}\widetilde{U_{2}^{\left(n-1\right)}}^{n}}{\partial x_{1}\partial x_{2}}\left(t,x\right) & =\sum_{m=1}^{n-1}B_{m}\left(t\right)\frac{a_{m}\left(t\right)^{2}b_{m}\left(t\right)}{a_{n}\left(t\right)b_{n}\left(t\right)}\sum_{8}\text{SPCM}\left(m,y^{\left(n,m\right)}\right),\\
\frac{\partial^{2}\widetilde{U_{2}^{\left(n-1\right)}}^{n}}{\partial x_{2}^{2}}\left(t,x\right) & =\sum_{m=1}^{n-1}B_{m}\left(t\right)\frac{a_{m}\left(t\right)b_{m}\left(t\right)^{2}}{b_{n}\left(t\right)^{2}}\sum_{8}\text{SPCM}\left(m,y^{\left(n,m\right)}\right).
\end{aligned}
\label{eq:second derivatives velocity past layers}
\end{equation}

The next step is to bound the expressions of equation \eqref{eq:second derivatives velocity past layers}
in $C^{k,\alpha}\left(\mathbb{R}^{2}\right)$. In view of equation
\eqref{eq:SPCM}, taking into account that $\lambda_{n}\le1$ by Choice
\ref{choice:psin}, we deduce that
\begin{equation}
\begin{aligned}\left|\left|\text{SPCM}\left(m,y^{\left(n,m\right)}\left(t,\cdot\right)\right)\right|\right|_{L^{\infty}\left(\mathbb{R}^{2}\right)} & =\left|\left|\text{SPCM}\left(m,\cdot\right)\right|\right|_{L^{\infty}\left(\mathbb{R}^{2}\right)}\lesssim_{\varphi}1.\end{aligned}
\label{eq:bounds SPCM Linfty}
\end{equation}
The fact that $y_{1}^{\left(n,m\right)}$ only depends on $x_{1}$
(and not $x_{2}$) and the analogous fact for $y_{2}^{\left(n,m\right)}$,
along with equations \eqref{eq:SPCM}, \eqref{eq:property Calpha composition},
\eqref{eq:marginal seminorms} and \eqref{eq:derivatives y} and Choice
\ref{choice:psin}, provide
\begin{equation}
\begin{aligned}\left|\left|\text{SPCM}\left(m,y^{\left(n,m\right)}\left(t,\cdot\right)\right)\right|\right|_{\dot{C}_{1}^{\alpha}\left(\mathbb{R}^{2}\right)} & \lesssim\left|\left|y_{1}^{\left(n,m\right)}\left(t,\cdot\right)\right|\right|_{\dot{C}^{1}\left(\mathbb{R}\right)}^{\alpha}\underbrace{\left|\left|\text{SPCM}\left(m,\cdot\right)\right|\right|_{\dot{C}_{1}^{\alpha}\left(\mathbb{R}^{2}\right)}}_{\lesssim_{\varphi}1}\lesssim_{\varphi}\left(\frac{a_{m}\left(t\right)}{a_{n}\left(t\right)}\right)^{\alpha},\\
\left|\left|\text{SPCM}\left(m,y^{\left(n,m\right)}\left(t,\cdot\right)\right)\right|\right|_{\dot{C}_{2}^{\alpha}\left(\mathbb{R}^{2}\right)} & \lesssim\left|\left|y_{2}^{\left(n,m\right)}\left(t,\cdot\right)\right|\right|_{\dot{C}^{1}\left(\mathbb{R}\right)}^{\alpha}\underbrace{\left|\left|\text{SPCM}\left(m,\cdot\right)\right|\right|_{\dot{C}_{2}^{\alpha}\left(\mathbb{R}^{2}\right)}}_{\lesssim_{\varphi}1}\lesssim_{\varphi}\left(\frac{b_{m}\left(t\right)}{b_{n}\left(t\right)}\right)^{\alpha}.
\end{aligned}
\label{eq:bounds SPCM v1}
\end{equation}

Thanks to equation \eqref{eq:bounds SPCM Linfty}, it is not difficult
to obtain the $\left|\left|\cdot\right|\right|_{L^{\infty}\left(\mathbb{R}^{2}\right)}$
norm of the second derivatives given in \eqref{eq:second derivatives velocity past layers}.
Indeed, making use of \eqref{eq:bounds SPCM Linfty} in \eqref{eq:second derivatives velocity past layers}
leads to
\[
\left|\left|\frac{\partial^{2}\widetilde{U_{1}^{\left(n-1\right)}}^{n}}{\partial x_{1}^{2}}\left(t,\cdot\right)\right|\right|_{L^{\infty}\left(\mathbb{R}^{2}\right)}\lesssim_{\varphi}\sum_{m=1}^{n-1}B_{m}\left(t\right)\frac{a_{m}\left(t\right)^{2}b_{m}\left(t\right)}{a_{n}\left(t\right)^{2}}.
\]
As $n\ge2$, we may apply Proposition \ref{prop:time convergence},
which assures us that, provided that $C$ is taken big enough (let
us say $C\ge\Upsilon_{1}\left(\beta,\delta\right)$), we may bound
\[
\left|\left|\frac{\partial^{2}\widetilde{U_{1}^{\left(n-1\right)}}^{n}}{\partial x_{1}^{2}}\left(t,\cdot\right)\right|\right|_{L^{\infty}\left(\mathbb{R}^{2}\right)}\lesssim_{\varphi}\frac{1}{a_{n}\left(t\right)^{2}}\sum_{m=1}^{n-1}B_{m}\left(1\right)a_{m}\left(1\right)b_{m}\left(1\right)a_{m}\left(1\right).
\]
Moreover, by Choices \ref{choice:time picture} and \ref{choice:anbn},
we infer that $a_{m}\left(1\right)=b_{m}\left(1\right)$, which means
that the bound we have obtained for the sum (we exclude the factor
$\frac{1}{a_{n}\left(t\right)^{2}}$) is valid for every second order
derivative of \eqref{eq:second derivatives velocity past layers}.
Resorting to Choices \ref{choice:anbn}, \ref{choice:Mn} and \ref{choice:time picture}
leads us to
\[
\left|\left|\frac{\partial^{2}\widetilde{U_{1}^{\left(n-1\right)}}^{n}}{\partial x_{1}^{2}}\left(t,\cdot\right)\right|\right|_{L^{\infty}\left(\mathbb{R}^{2}\right)}\lesssim_{\varphi}\frac{Y}{a_{n}\left(t\right)^{2}}\sum_{m=1}^{n-1}C^{\left(1+\delta\right)\left(\frac{1}{1-\gamma}\right)^{m}}.
\]
Now, Lemma \ref{lem:estimate sum superexponential} (which we can
apply by virtue of Choice \ref{choice:min requirements on C gamma and kmax})
allows us to write
\[
\left|\left|\frac{\partial^{2}\widetilde{U_{1}^{\left(n-1\right)}}^{n}}{\partial x_{1}^{2}}\left(t,\cdot\right)\right|\right|_{L^{\infty}\left(\mathbb{R}^{2}\right)}\lesssim_{\delta,\varphi}Y\frac{C^{\left(1+\delta\right)\left(\frac{1}{1-\gamma}\right)^{n-1}}}{a_{n}\left(t\right)^{2}}.
\]
As $\left|\left|\cdot\right|\right|_{L^{\infty}\left(\mathbb{R}^{2}\right)}$
is invariant under diffeomorphisms of the domain, this proves the
result. The same argument works for all the other second order derivatives,
providing
\begin{equation}
\begin{aligned}\left|\left|\frac{\partial^{2}\widetilde{U_{1}^{\left(n-1\right)}}^{n}}{\partial x_{1}^{2}}\left(t,\left(\phi^{\left(n\right)}\right)^{-1}\left(t,\cdot\right)\right)\right|\right|_{L^{\infty}\left(\mathbb{R}^{2}\right)} & \lesssim_{\delta,\varphi}Y\frac{C^{\left(1+\delta\right)\left(\frac{1}{1-\gamma}\right)^{n-1}}}{a_{n}\left(t\right)^{2}},\\
\left|\left|\frac{\partial^{2}\widetilde{U_{1}^{\left(n-1\right)}}^{n}}{\partial x_{1}\partial x_{2}}\left(t,\left(\phi^{\left(n\right)}\right)^{-1}\left(t,\cdot\right)\right)\right|\right|_{L^{\infty}\left(\mathbb{R}^{2}\right)} & \lesssim_{\delta,\varphi}Y\frac{C^{\left(1+\delta\right)\left(\frac{1}{1-\gamma}\right)^{n-1}}}{a_{n}\left(t\right)b_{n}\left(t\right)},\\
\left|\left|\frac{\partial^{2}\widetilde{U_{1}^{\left(n-1\right)}}^{n}}{\partial x_{2}^{2}}\left(t,\left(\phi^{\left(n\right)}\right)^{-1}\left(t,\cdot\right)\right)\right|\right|_{L^{\infty}\left(\mathbb{R}^{2}\right)} & \lesssim_{\delta,\varphi}Y\frac{C^{\left(1+\delta\right)\left(\frac{1}{1-\gamma}\right)^{n-1}}}{b_{n}\left(t\right)^{2}},\\
\left|\left|\frac{\partial^{2}\widetilde{U_{2}^{\left(n-1\right)}}^{n}}{\partial x_{1}^{2}}\left(t,\left(\phi^{\left(n\right)}\right)^{-1}\left(t,\cdot\right)\right)\right|\right|_{L^{\infty}\left(\mathbb{R}^{2}\right)} & \lesssim_{\delta,\varphi}Y\frac{C^{\left(1+\delta\right)\left(\frac{1}{1-\gamma}\right)^{n-1}}}{a_{n}\left(t\right)^{2}},\\
\left|\left|\frac{\partial^{2}\widetilde{U_{2}^{\left(n-1\right)}}^{n}}{\partial x_{1}\partial x_{2}}\left(t,\left(\phi^{\left(n\right)}\right)^{-1}\left(t,\cdot\right)\right)\right|\right|_{L^{\infty}\left(\mathbb{R}^{2}\right)} & \lesssim_{\delta,\varphi}Y\frac{C^{\left(1+\delta\right)\left(\frac{1}{1-\gamma}\right)^{n-1}}}{a_{n}\left(t\right)b_{n}\left(t\right)},\\
\left|\left|\frac{\partial^{2}\widetilde{U_{2}^{\left(n-1\right)}}^{n}}{\partial x_{2}^{2}}\left(t,\left(\phi^{\left(n\right)}\right)^{-1}\left(t,\cdot\right)\right)\right|\right|_{L^{\infty}\left(\mathbb{R}^{2}\right)} & \lesssim_{\delta,\varphi}Y\frac{C^{\left(1+\delta\right)\left(\frac{1}{1-\gamma}\right)^{n-1}}}{b_{n}\left(t\right)^{2}},
\end{aligned}
\label{eq:second derivatives bounds in Linfty}
\end{equation}
which is the result of the statement.

Now, we concentrate on the $\left|\left|\cdot\right|\right|_{\dot{C}^{\alpha}\left(\mathbb{R}^{2}\right)}$
seminorms. Again, we shall study these via marginal Hölder seminorms
(see equation \eqref{eq:marginal seminorms}). In view of equations
\eqref{eq:second derivatives velocity past layers} and \eqref{eq:bounds SPCM v1},
we deduce that
\[
\begin{aligned}\left|\left|\frac{\partial^{2}\widetilde{U_{1}^{\left(n-1\right)}}^{n}}{\partial x_{1}^{2}}\left(t,\cdot\right)\right|\right|_{\dot{C}_{1}^{\alpha}\left(\mathbb{R}^{2}\right)} & \lesssim_{\varphi}\sum_{m=1}^{n-1}B_{m}\left(t\right)\frac{a_{m}\left(t\right)^{2}b_{m}\left(t\right)}{a_{n}\left(t\right)^{2}}\left(\frac{a_{m}\left(t\right)}{a_{n}\left(t\right)}\right)^{\alpha},\\
\left|\left|\frac{\partial^{2}\widetilde{U_{1}^{\left(n-1\right)}}^{n}}{\partial x_{1}^{2}}\left(t,\cdot\right)\right|\right|_{\dot{C}_{2}^{\alpha}\left(\mathbb{R}^{2}\right)} & \lesssim_{\varphi}\sum_{m=1}^{n-1}B_{m}\left(t\right)\frac{a_{m}\left(t\right)^{2}b_{m}\left(t\right)}{a_{n}\left(t\right)^{2}}\left(\frac{b_{m}\left(t\right)}{b_{n}\left(t\right)}\right)^{\alpha}.
\end{aligned}
\]
In this manner, equations \eqref{eq:property Calpha composition}
and \eqref{eq:jacobian inverse}, along with the fact that $\left(\phi_{1}^{\left(n\right)}\right)^{-1}$
only depends on $x_{1}$ and that $\left(\phi_{2}^{\left(n\right)}\right)^{-2}$
only depends on $x_{2}$, allow us to write
\[
\begin{aligned} & \left|\left|\frac{\partial^{2}\widetilde{U_{1}^{\left(n-1\right)}}^{n}}{\partial x_{1}^{2}}\left(t,\left(\phi^{\left(n\right)}\right)^{-1}\left(t,\cdot\right)\right)\right|\right|_{\dot{C}_{1}^{\alpha}\left(\mathbb{R}^{2}\right)}\lesssim\\
\lesssim_{\phantom{\varphi}} & \left|\left|\left(\phi_{1}^{\left(n\right)}\right)^{-1}\left(t,\cdot\right)\right|\right|_{\dot{C}^{1}\left(\mathbb{R}\right)}^{\alpha}\left|\left|\frac{\partial^{2}\widetilde{U_{1}^{\left(n-1\right)}}^{n}}{\partial x_{1}^{2}}\left(t,\cdot\right)\right|\right|_{\dot{C}_{1}^{\alpha}\left(\mathbb{R}^{2}\right)}\lesssim_{\varphi}\\
\lesssim_{\varphi} & a_{n}\left(t\right)^{\alpha}\sum_{m=1}^{n-1}B_{m}\left(t\right)\frac{a_{m}\left(t\right)^{2}b_{m}\left(t\right)}{a_{n}\left(t\right)^{2}}\left(\frac{a_{m}\left(t\right)}{a_{n}\left(t\right)}\right)^{\alpha}=\sum_{m=1}^{n-1}B_{m}\left(t\right)\frac{a_{m}\left(t\right)^{2+\alpha}b_{m}\left(t\right)}{a_{n}\left(t\right)^{2}}.
\end{aligned}
\]
Analogously,
\[
\left|\left|\frac{\partial^{2}\widetilde{U_{1}^{\left(n-1\right)}}^{n}}{\partial x_{1}^{2}}\left(t,\left(\phi^{\left(n\right)}\right)^{-1}\left(t,\cdot\right)\right)\right|\right|_{\dot{C}_{2}^{\alpha}\left(\mathbb{R}^{2}\right)}\lesssim_{\varphi}\sum_{m=1}^{n-1}B_{m}\left(t\right)\frac{a_{m}\left(t\right)^{2}b_{m}\left(t\right)^{1+\alpha}}{a_{n}\left(t\right)^{2}}.
\]
To finish the argument, we need to recover the $\left|\left|\cdot\right|\right|_{\dot{C}^{\alpha}}$
seminorm form the marginal seminorms, which can be done through equation
\eqref{eq:Holder seminorm by marginals}. In this manner, we conclude
that
\[
\left|\left|\frac{\partial^{2}\widetilde{U_{1}^{\left(n-1\right)}}^{n}}{\partial x_{1}^{2}}\left(t,\left(\phi^{\left(n\right)}\right)^{-1}\left(t,\cdot\right)\right)\right|\right|_{\dot{C}^{\alpha}\left(\mathbb{R}^{2}\right)}\lesssim_{\varphi}\sum_{m=1}^{n-1}B_{m}\left(t\right)\frac{a_{m}\left(t\right)^{2}b_{m}\left(t\right)}{a_{n}\left(t\right)^{2}}\max\left\{ a_{m}\left(t\right)^{\alpha},b_{m}\left(t\right)^{\alpha}\right\} .
\]
Now, we can argue like we did in the $\left|\left|\cdot\right|\right|_{L^{\infty}\left(\mathbb{R}^{2}\right)}$
case to obtain that
\[
\left|\left|\frac{\partial^{2}\widetilde{U_{1}^{\left(n-1\right)}}^{n}}{\partial x_{1}^{2}}\left(t,\left(\phi^{\left(n\right)}\right)^{-1}\left(t,\cdot\right)\right)\right|\right|_{\dot{C}^{\alpha}\left(\mathbb{R}^{2}\right)}\lesssim_{\delta,\varphi}Y\frac{C^{\left(1+\alpha+\delta\right)\left(\frac{1}{1-\gamma}\right)^{n-1}}}{a_{n}\left(t\right)^{2}}.
\]
The same argument may be applied to the other second order derivatives
of \eqref{eq:second derivatives velocity past layers}, leading to
\[
\begin{aligned}\left|\left|\frac{\partial^{2}\widetilde{U_{1}^{\left(n-1\right)}}^{n}}{\partial x_{1}^{2}}\left(t,\left(\phi^{\left(n\right)}\right)^{-1}\left(t,\cdot\right)\right)\right|\right|_{\dot{C}^{\alpha}\left(\mathbb{R}^{2}\right)} & \lesssim_{\delta,\varphi}Y\frac{C^{\left(1+\alpha+\delta\right)\left(\frac{1}{1-\gamma}\right)^{n-1}}}{a_{n}\left(t\right)^{2}},\\
\left|\left|\frac{\partial^{2}\widetilde{U_{1}^{\left(n-1\right)}}^{n}}{\partial x_{1}\partial x_{2}}\left(t,\left(\phi^{\left(n\right)}\right)^{-1}\left(t,\cdot\right)\right)\right|\right|_{\dot{C}^{\alpha}\left(\mathbb{R}^{2}\right)} & \lesssim_{\delta,\varphi}Y\frac{C^{\left(1+\alpha+\delta\right)\left(\frac{1}{1-\gamma}\right)^{n-1}}}{a_{n}\left(t\right)b_{n}\left(t\right)},\\
\left|\left|\frac{\partial^{2}\widetilde{U_{1}^{\left(n-1\right)}}^{n}}{\partial x_{2}^{2}}\left(t,\left(\phi^{\left(n\right)}\right)^{-1}\left(t,\cdot\right)\right)\right|\right|_{\dot{C}^{\alpha}\left(\mathbb{R}^{2}\right)} & \lesssim_{\delta,\varphi}Y\frac{C^{\left(1+\alpha+\delta\right)\left(\frac{1}{1-\gamma}\right)^{n-1}}}{b_{n}\left(t\right)^{2}},\\
\left|\left|\frac{\partial^{2}\widetilde{U_{2}^{\left(n-1\right)}}^{n}}{\partial x_{1}^{2}}\left(t,\left(\phi^{\left(n\right)}\right)^{-1}\left(t,\cdot\right)\right)\right|\right|_{\dot{C}^{\alpha}\left(\mathbb{R}^{2}\right)} & \lesssim_{\delta,\varphi}Y\frac{C^{\left(1+\alpha+\delta\right)\left(\frac{1}{1-\gamma}\right)^{n-1}}}{a_{n}\left(t\right)^{2}},\\
\left|\left|\frac{\partial^{2}\widetilde{U_{2}^{\left(n-1\right)}}^{n}}{\partial x_{1}\partial x_{2}}\left(t,\left(\phi^{\left(n\right)}\right)^{-1}\left(t,\cdot\right)\right)\right|\right|_{\dot{C}^{\alpha}\left(\mathbb{R}^{2}\right)} & \lesssim_{\delta,\varphi}Y\frac{C^{\left(1+\alpha+\delta\right)\left(\frac{1}{1-\gamma}\right)^{n-1}}}{a_{n}\left(t\right)^ {}b_{n}\left(t\right)},\\
\left|\left|\frac{\partial^{2}\widetilde{U_{2}^{\left(n-1\right)}}^{n}}{\partial x_{2}^{2}}\left(t,\left(\phi^{\left(n\right)}\right)^{-1}\left(t,\cdot\right)\right)\right|\right|_{\dot{C}^{\alpha}\left(\mathbb{R}^{2}\right)} & \lesssim_{\delta,\varphi}Y\frac{C^{\left(1+\alpha+\delta\right)\left(\frac{1}{1-\gamma}\right)^{n-1}}}{b_{n}\left(t\right)^{2}}.
\end{aligned}
\]

What about the $\left|\left|\cdot\right|\right|_{\dot{C}^{1}\left(\mathbb{R}^{2}\right)}$
and $\left|\left|\cdot\right|\right|_{\dot{C}^{1,\alpha}\left(\mathbb{R}^{2}\right)}$
seminorms? Bearing in mind equations \eqref{eq:derivatives y} and
\eqref{eq:SPCM}, differentiating \eqref{eq:second derivatives velocity past layers}
provides something of the form
\begin{equation}
\frac{\partial^{3}\widetilde{U_{r}^{\left(n-1\right)}}^{n}}{\partial x_{i}\partial x_{j}\partial x_{l}}\left(t,x\right)=\sum_{m=1}^{n-1}B_{m}\left(t\right)\frac{\text{binomial in }\left\{ a_{m}\left(t\right),b_{m}\left(t\right)\right\} \text{ of order }4}{a_{n}\left(t\right)^{1_{\left\{ i=1\right\} }+1_{\left\{ j=1\right\} }+1_{\left\{ l=1\right\} }}b_{n}\left(t\right)^{1_{\left\{ i=2\right\} }+1_{\left\{ j=2\right\} }+1_{\left\{ l=2\right\} }}}\sum_{16}\text{SPCM}\left(m,y^{\left(n,m\right)}\right),\label{eq:form third derivatives in good variables}
\end{equation}
where $i,j,l,r\in\left\{ 1,2\right\} $ and $1_{\left\{ \text{set}\right\} }$
is the indicator function. Notice that $1_{\left\{ i=1\right\} }+1_{\left\{ j=1\right\} }+1_{\left\{ l=1\right\} }$
counts the number of times that $x_{1}$ appears in $\left\{ x_{i},x_{j},x_{r}\right\} $,
i.e., it counts the number of times we differentiate with respect
to $x_{1}$. Similarly, $1_{\left\{ i=2\right\} }+1_{\left\{ j=2\right\} }+l_{\left\{ l=2\right\} }$
counts the number of times we differentiate with respect to $x_{2}$.
In this way, the indicator functions of equation \eqref{eq:form third derivatives in good variables}
have the purpose of encoding the fact that we get one $a_{n}\left(t\right)$
in the denominator each time we differentiate with respect to $x_{1}$
and one $b_{n}\left(t\right)$ each time we differentiate with respect
to $b_{n}\left(t\right)$. By the chain rule and equation \eqref{eq:jacobian inverse},
as $\phi_{1}^{\left(n\right)}$ only depends on $x_{1}$ and $\phi_{2}^{\left(n\right)}$
only depends on $x_{2}$, we have
\begin{equation}
\begin{aligned} & \frac{\partial}{\partial x_{l}}\left(\frac{\partial^{2}\widetilde{U_{r}^{\left(n-1\right)}}^{n}}{\partial x_{i}\partial x_{j}}\left(t,\left(\phi^{\left(n\right)}\right)^{-1}\left(t,x\right)\right)\right)=\\
= & \nabla\left(\frac{\partial^{2}\widetilde{U_{r}^{\left(n-1\right)}}^{n}}{\partial x_{i}\partial x_{j}}\right)\left(t,\left(\phi^{\left(n\right)}\right)^{-1}\left(t,x\right)\right)\cdot\frac{\partial\left(\phi^{\left(n\right)}\right)^{-1}}{\partial x_{l}}\left(t,x\right)=\\
= & a_{n}\left(t\right)^{1_{\left\{ l=1\right\} }}b_{n}\left(t\right)^{1_{\left\{ l=2\right\} }}\frac{\partial^{3}\widetilde{U_{r}^{\left(n-1\right)}}^{n}}{\partial x_{i}\partial x_{j}\partial x_{l}}\left(t,\left(\phi^{\left(n\right)}\right)^{-1}\left(t,x\right)\right)
\end{aligned}
\label{eq:third derivatives in bad variables}
\end{equation}
$\forall\left\{ i,j,l,r\right\} \in\left\{ 1,2\right\} $. Since the
$\left|\left|\cdot\right|\right|_{L^{\infty}\left(\mathbb{R}^{2}\right)}$
norm is invariant under diffeomorphisms, employing equations \eqref{eq:bounds SPCM Linfty},
\eqref{eq:form third derivatives in good variables} and \eqref{eq:third derivatives in bad variables}
and the same techniques we have used before, we obtain
\[
\begin{aligned} & \left|\left|\frac{\partial^{2}\widetilde{U_{1}^{\left(n-1\right)}}^{n}}{\partial x_{1}^{2}}\left(t,\left(\phi^{\left(n\right)}\right)^{-1}\left(t,\cdot\right)\right)\right|\right|_{\dot{C}^{1}\left(\mathbb{R}^{2}\right)}=\\
=_{\phantom{\delta,\varphi}} & \max_{l\in\left\{ 1,2\right\} }\left\{ \left|\left|\frac{\partial}{\partial x_{l}}\left(\frac{\partial^{2}\widetilde{U_{1}^{\left(n-1\right)}}^{n}}{\partial x_{i}\partial x_{j}}\left(t,\left(\phi^{\left(n\right)}\right)^{-1}\left(t,\cdot\right)\right)\right)\right|\right|_{L^{\infty}\left(\mathbb{R}^{2}\right)}\right\} =\\
=_{\phantom{\delta,\varphi}} & \max\left\{ a_{n}\left(t\right)\left|\left|\frac{\partial^{3}\widetilde{U^{\left(n-1\right)}}^{n}}{\partial x_{1}^{3}}\left(t,\cdot\right)\right|\right|_{L^{\infty}\left(\mathbb{R}^{2}\right)},b_{n}\left(t\right)\left|\left|\frac{\partial^{3}\widetilde{U^{\left(n-1\right)}}^{n}}{\partial x_{1}^{2}\partial x_{2}}\left(t,\cdot\right)\right|\right|_{L^{\infty}\left(\mathbb{R}^{2}\right)}\right\} \lesssim_{\delta,\varphi}\\
\lesssim_{\delta,\varphi} & Y\frac{C^{\left(2+\delta\right)\left(\frac{1}{1-\gamma}\right)^{n-1}}}{a_{n}\left(t\right)^{2}}.
\end{aligned}
\]
Analogously, we can deduce the result for the $\left|\left|\cdot\right|\right|_{\dot{C}^{1}\left(\mathbb{R}^{2}\right)}$
seminorms of the other second order derivatives given in \eqref{eq:second derivatives velocity past layers}. 

Lastly, to compute the $\left|\left|\cdot\right|\right|_{\dot{C}^{1,\alpha}\left(\mathbb{R}^{2}\right)}$
seminorm, we throw ourselves into the hands of the marginal Hölder
seminorms again, repeating the argument we have exposed for the $\left|\left|\cdot\right|\right|_{\dot{C}^{\alpha}}$
case, but applied to equations \eqref{eq:third derivatives in bad variables}
and \eqref{eq:form third derivatives in good variables}.
\end{proof}
The next result we need has to do with bounding second order polynomials
evaluated in the inverse of $\phi^{\left(n\right)}\left(t,x\right)$.
\begin{lem}
\label{lem:bounds polynomials}Consider
\[
D^{\left(n\right)}\left(t\right)=\left(\phi^{\left(n\right)}\right)^{-1}\left(t,\left[-\frac{16\pi}{\lambda_{n}},\frac{16\pi}{\lambda_{n}}\right]^{2}\right),
\]
where $\left(\phi^{\left(n\right)}\right)^{-1}\left(t,\cdot\right)$
should be understood as the preimage at fixed $t$. Then,
\begin{center}
\begin{tabular}{|c|c|c|c|c|}
\hline 
 & $\left|\left|\cdot\right|\right|_{L^{\infty}\left(D^{\left(n\right)}\left(t\right)\right)}$ & $\left|\left|\cdot\right|\right|_{\dot{C}^{\alpha}\left(D^{\left(n\right)}\left(t\right)\right)}$ & $\left|\left|\cdot\right|\right|_{\dot{C}^{1}\left(D^{\left(n\right)}\left(t\right)\right)}$ & $\left|\left|\cdot\right|\right|_{\dot{C}^{1,\alpha}\left(D^{\left(n\right)}\left(t\right)\right)}$\tabularnewline
\hline 
\hline 
$\left[\left(\phi^{\left(n\right)}\right)_{1}^{-1}\left(t,\cdot\right)\right]^{2}$ & $\frac{1}{\lambda_{n}^{2}}$ & $\frac{a_{n}\left(t\right)^{\alpha}}{\lambda_{n}}$ & $\frac{a_{n}\left(t\right)^ {}}{\lambda_{n}}$ & $a_{n}\left(t\right)^{1+\alpha}$\tabularnewline
\hline 
$\left(\phi^{\left(n\right)}\right)_{1}^{-1}\left(t,\cdot\right)\left(\phi^{\left(n\right)}\right)_{2}^{-1}\left(t,\cdot\right)$ & $\frac{1}{\lambda_{n}^{2}}$ & $\frac{\max\left\{ a_{n}\left(t\right)^{\alpha},b_{n}\left(t\right)^{\alpha}\right\} }{\lambda_{n}}$ & $\frac{\max\left\{ a_{n}\left(t\right),b_{n}\left(t\right)\right\} }{\lambda_{n}}$ & $\begin{matrix}\max\left\{ a_{n}\left(t\right)b_{n}\left(t\right)^{\alpha},\right.\\
\left.a_{n}\left(t\right)^{\alpha}b_{n}\left(t\right)\right\} 
\end{matrix}$\tabularnewline
\hline 
$\left[\left(\phi^{\left(n\right)}\right)_{2}^{-1}\left(t,\cdot\right)\right]^{2}$ & $\frac{1}{\lambda_{n}^{2}}$ & $\frac{b_{n}\left(t\right)^{\alpha}}{\lambda_{n}}$ & $\frac{b_{n}\left(t\right)}{\lambda_{n}}$ & $b_{n}\left(t\right)^{1+\alpha}$\tabularnewline
\hline 
\end{tabular}
\par\end{center}

The table above has to be read as follows: the first column contains
the terms that we wish to bound, the headers of the other columns
specify the norm (or seminorm) we are using and the rest of the table
contains the (upper) bounds corresponding to each case. The bounds
are to be understood up to a constant.

\end{lem}
\begin{proof}
Since the $\left|\left|\cdot\right|\right|_{L^{\infty}\left(D^{\left(n\right)}\left(t\right)\right)}$
norm is invariant under diffemorphisms, we obtain
\[
\left|\left|\left[\left(\phi^{\left(n\right)}\right)_{1}^{-1}\left(t,\cdot\right)\right]^{2}\right|\right|_{L^{\infty}\left(D^{\left(n\right)}\left(t\right)\right)}=\left|\left|x_{1}^{2}\right|\right|_{L^{\infty}\left(\phi^{\left(n\right)}\left(t,D^{\left(n\right)}\left(t\right)\right)\right)}=\left|\left|x_{1}^{2}\right|\right|_{L^{\infty}\left(\left[-\frac{16\pi}{\lambda_{n}},\frac{16\pi}{\lambda_{n}}\right]^{2}\right)}\lesssim\frac{1}{\lambda_{n}^{2}}.
\]
The same argument applies for the other two elements of the column.

For the $\left|\left|\cdot\right|\right|_{\dot{C}^{\alpha}\left(D^{\left(n\right)}\left(t\right)\right)}$
seminorm, we apply equation \eqref{eq:property Calpha composition},
which provides
\[
\left|\left|\left[\left(\phi^{\left(n\right)}\right)_{1}^{-1}\left(t,\cdot\right)\right]^{2}\right|\right|_{\dot{C}^{\alpha}\left(D^{\left(n\right)}\left(t\right)\right)}\leq\left|\left|\left(\phi^{\left(n\right)}\right)_{1}^{-1}\left(t,\cdot\right)\right|\right|_{\dot{C}^{1}\left(\mathbb{R}\right)}^{\alpha}\left|\left|x_{1}^{2}\right|\right|_{\dot{C}^{\alpha}\left(\left[-\frac{16\pi}{\lambda_{n}},\frac{16\pi}{\lambda_{n}}\right]^{2}\right)}.
\]
On the one hand, equation \eqref{eq:jacobian inverse} gives us that
\[
\left|\left|\left(\phi^{\left(n\right)}\right)_{1}^{-1}\left(t,\cdot\right)\right|\right|_{\dot{C}^{1}\left(\mathbb{R}\right)}^{\alpha}=a_{n}\left(t\right)^{\alpha}.
\]
On the other hand,
\[
\left|\left|x_{1}^{2}\right|\right|_{\dot{C}^{\alpha}\left(\left[-\frac{16\pi}{\lambda_{n}},\frac{16\pi}{\lambda_{n}}\right]^{2}\right)}\leq\left|\left|x_{1}^{2}\right|\right|_{\dot{C}^{1}\left(\left[-\frac{16\pi}{\lambda_{n}},\frac{16\pi}{\lambda_{n}}\right]^{2}\right)}=\left|\left|2x_{1}\right|\right|_{L^{\infty}\left(\left[-\frac{16\pi}{\lambda_{n}},\frac{16\pi}{\lambda_{n}}\right]\right)}\lesssim\frac{1}{\lambda_{n}}.
\]
This proves the result for $\left|\left|\left[\left(\phi^{\left(n\right)}\right)_{1}^{-1}\left(t,\cdot\right)\right]^{2}\right|\right|_{\dot{C}^{\alpha}\left(D^{\left(n\right)}\left(t\right)\right)}$.
A similar argument may be applied to the remaining two elements of
the column.

Notice that
\[
\nabla\left(\left[\left(\phi^{\left(n\right)}\right)_{1}^{-1}\left(t,x\right)\right]^{2}\right)=2\left(\phi^{\left(n\right)}\right)_{1}^{-1}\left(t,x\right)\nabla\left[\left(\phi^{\left(n\right)}\right)_{1}^{-1}\left(t,x\right)\right].
\]
Hence, the chain rule and equation \eqref{eq:jacobian inverse} tell
us that
\begin{equation}
\nabla\left(\left[\left(\phi^{\left(n\right)}\right)_{1}^{-1}\left(t,x\right)\right]^{2}\right)=\left(2a_{n}\left(t\right)\left(\phi^{\left(n\right)}\right)_{1}^{-1}\left(t,x\right),0\right).\label{eq:polynomials the gradient 1}
\end{equation}
Similarly,
\begin{equation}
\begin{aligned}\nabla\left(\left(\phi^{\left(n\right)}\right)_{1}^{-1}\left(t,x\right)\left(\phi^{\left(n\right)}\right)_{2}^{-1}\left(t,x\right)\right) & =\left(a_{n}\left(t\right)\left(\phi^{\left(n\right)}\right)_{2}^{-1}\left(t,x\right),b_{n}\left(t\right)\left(\phi^{\left(n\right)}\right)_{1}^{-1}\left(t,x\right)\right),\\
\nabla\left(\left[\left(\phi^{\left(n\right)}\right)_{2}^{-1}\left(t,x\right)\right]^{2}\right) & =\left(0,2b_{n}\left(t\right)\left(\phi^{\left(n\right)}\right)_{2}^{-1}\left(t,x\right)\right).
\end{aligned}
\label{eq:polynomials the gradient 2}
\end{equation}
Here, it is of utmost importance to point out that the first component
of $\nabla\left(\left(\phi^{\left(n\right)}\right)_{1}^{-1}\left(t,x\right)\left(\phi^{\left(n\right)}\right)_{2}^{-1}\left(t,x\right)\right)$
only depends on $x_{2}$, whereas its second component only depends
on $x_{1}$. From equations \eqref{eq:polynomials the gradient 1}
and \eqref{eq:polynomials the gradient 2}, it is immediate to deduce
the results corresponding to the third column of the statement.

Employing the same arguments we have used to compute $\left|\left|\cdot\right|\right|_{L^{\infty}\left(D^{\left(n\right)}\left(t\right)\right)}$
and $\left|\left|\cdot\right|\right|_{\dot{C}^{\alpha}\left(D^{\left(n\right)}\left(t\right)\right)}$,
but applied to the gradients \eqref{eq:polynomials the gradient 1}
and \eqref{eq:polynomials the gradient 2} leads to the results about
the $\left|\left|\cdot\right|\right|_{\dot{C}^{1}\left(D^{\left(n\right)}\left(t\right)\right)}$
and $\left|\left|\cdot\right|\right|_{\dot{C}^{1,\alpha}\left(D^{\left(n\right)}\left(t\right)\right)}$
seminorms.
\end{proof}
Now that we have bounded the second order derivatives of $\widetilde{U^{\left(n-1\right)}}^{n}$
and have obtained estimates for second order polynomials evaluated
in $\left(\phi^{\left(n\right)}\right)^{-1}$, we have the required
tools to compute some bounds for the first factor of the transport
term. This is exactly the task we tackle in the next Proposition.
\begin{prop}
\label{prop:bounds for Taylor development of transport}Let $n\in\mathbb{N}$
with $n\ge2$ and 
\[
\begin{aligned}\widetilde{W^{\left(n\right)}}^{n}\left(t,x\right) & \coloneqq\widetilde{U^{\left(n-1\right)}}^{n}\left(t,x\right)-\widetilde{U^{\left(n-1\right)}}^{n}\left(t,0\right)-\mathrm{J}\widetilde{U^{\left(n-1\right)}}^{n}\left(t,0\right)\cdot\left(\begin{matrix}x_{1}\\
x_{2}
\end{matrix}\right),\\
D^{\left(n\right)}\left(t\right) & \coloneqq\left(\phi^{\left(n\right)}\right)^{-1}\left(t,\left[-\frac{16\pi}{\lambda_{n}},\frac{16\pi}{\lambda_{n}}\right]^{2}\right),
\end{aligned}
\]
where $\left(\phi^{\left(n\right)}\right)^{-1}\left(t,\cdot\right)$
should be understood as the preimage at fixed $t$. Then, as long
as $C$ is big enough (let us say $C\ge\Upsilon\left(\delta,\mu\right)$),
\[
\begin{aligned}\left|\left|W^{\left(n\right)}\left(t,\cdot\right)\right|\right|_{L^{\infty}\left(D^{\left(n\right)}\left(t\right);\mathbb{R}^{2}\right)} & \lesssim_{\delta,\varphi}YC^{\left[-2\left(1-\Lambda-k_{\max}\right)+2\mu+\left(1+\delta\right)\left(1-\gamma\right)\right]\left(\frac{1}{1-\gamma}\right)^{n}},\\
\left|\left|W^{\left(n\right)}\left(t,\cdot\right)\right|\right|_{\dot{C}^{\alpha}\left(D^{\left(n\right)}\left(t\right);\mathbb{R}^{2}\right)} & \lesssim_{\delta,\varphi}YC^{\left[-2\left(1-\Lambda-k_{\max}\right)+\max\left\{ \alpha\left(1-k_{\max}\right)-\Lambda,0\right\} +2\mu+\left(2+\delta\right)\left(1-\gamma\right)\right]\left(\frac{1}{1-\gamma}\right)^{n}},\\
\left|\left|W^{\left(n\right)}\left(t,\cdot\right)\right|\right|_{\dot{C}^{1}\left(D^{\left(n\right)}\left(t\right);\mathbb{R}^{2}\right)} & \lesssim_{\delta,\varphi}YC^{\left[-\min\left\{ 1-\Lambda-k_{\max},2\left(1-\Lambda-k_{\max}\right)\right\} +2\mu+\left(2+\delta\right)\left(1-\gamma\right)\right]\left(\frac{1}{1-\gamma}\right)^{n}},\\
\left|\left|W^{\left(n\right)}\left(t,\cdot\right)\right|\right|_{\dot{C}^{1,\alpha}\left(D^{\left(n\right)}\left(t\right);\mathbb{R}^{2}\right)} & \lesssim_{\delta,\varphi}YC^{\left[-\min\left\{ \left(1-\alpha\right)\left(1-k_{\max}\right),1-k_{\max}-\Lambda,2\left(1-k_{\max}-\Lambda\right)\right\} +2\mu+\left(3+\delta\right)\left(1-\gamma\right)\right]\left(\frac{1}{1-\gamma}\right)^{n}}.
\end{aligned}
\]
Under the hypothesis $1-k_{\max}-\Lambda>0$, which is necessary for
the convergence in $n\in\mathbb{N}$ of the bounds above, they simplify
to
\[
\begin{aligned}\left|\left|W^{\left(n\right)}\left(t,\cdot\right)\right|\right|_{L^{\infty}\left(D^{\left(n\right)}\left(t\right);\mathbb{R}^{2}\right)} & \lesssim_{\delta,\varphi}YC^{\left[-2\left(1-\Lambda-k_{\max}\right)+2\mu+\left(1+\delta\right)\left(1-\gamma\right)\right]\left(\frac{1}{1-\gamma}\right)^{n}},\\
\left|\left|W^{\left(n\right)}\left(t,\cdot\right)\right|\right|_{\dot{C}^{\alpha}\left(D^{\left(n\right)}\left(t\right);\mathbb{R}^{2}\right)} & \lesssim_{\delta,\varphi}YC^{\left[-2\left(1-\Lambda-k_{\max}\right)+\max\left\{ \alpha\left(1-k_{\max}\right)-\Lambda,0\right\} +2\mu+\left(2+\delta\right)\left(1-\gamma\right)\right]\left(\frac{1}{1-\gamma}\right)^{n}},\\
\left|\left|W^{\left(n\right)}\left(t,\cdot\right)\right|\right|_{\dot{C}^{1}\left(D^{\left(n\right)}\left(t\right);\mathbb{R}^{2}\right)} & \lesssim_{\delta,\varphi}YC^{\left[-\left(1-\Lambda-k_{\max}\right)+2\mu+\left(2+\delta\right)\left(1-\gamma\right)\right]\left(\frac{1}{1-\gamma}\right)^{n}},\\
\left|\left|W^{\left(n\right)}\left(t,\cdot\right)\right|\right|_{\dot{C}^{1,\alpha}\left(D^{\left(n\right)}\left(t\right);\mathbb{R}^{2}\right)} & \lesssim_{\delta,\varphi}YC^{\left[-\min\left\{ \left(1-\alpha\right)\left(1-k_{\max}\right),1-k_{\max}-\Lambda\right\} +2\mu+\left(3+\delta\right)\left(1-\gamma\right)\right]\left(\frac{1}{1-\gamma}\right)^{n}}.
\end{aligned}
\]
\end{prop}
\begin{proof}
The expression for the remainder of the Multivariate Taylor Theorem
assures us that
\[
\begin{aligned}\widetilde{W^{\left(n\right)}}^{n}\left(t,\cdot\right) & =\frac{1}{2}\left[\int_{0}^{1}\left(1-s\right)\frac{\partial^{2}\widetilde{U^{\left(n-1\right)}}^{n}}{\partial x_{1}^{2}}\left(sx\right)\mathrm{d}s\right]x_{1}^{2}+\\
 & \quad+\left[\int_{0}^{1}\left(1-s\right)\frac{\partial^{2}\widetilde{U^{\left(n-1\right)}}^{n}}{\partial x_{1}\partial x_{2}}\left(sx\right)\mathrm{d}s\right]x_{1}x_{2}+\\
 & \quad+\frac{1}{2}\left[\int_{0}^{1}\left(1-s\right)\frac{\partial^{2}\widetilde{U^{\left(n-1\right)}}^{n}}{\partial x_{2}^{2}}\left(sx\right)\mathrm{d}s\right]x_{2}^{2}.
\end{aligned}
\]
Let us rewrite this in ``normal'' variables, i.e.,
\begin{equation}
\begin{aligned}W^{\left(n\right)}\left(t,x\right) & =\frac{1}{2}\left[\int_{0}^{1}\left(1-s\right)\frac{\partial^{2}\widetilde{U^{\left(n-1\right)}}^{n}}{\partial x_{1}^{2}}\left(s\left(\phi^{\left(n\right)}\right)^{-1}\left(t,x\right)\right)\mathrm{d}s\right]\left[\left(\phi^{\left(n\right)}\right)_{1}^{-1}\left(t,x\right)\right]^{2}+\\
 & \quad+\left[\int_{0}^{1}\left(1-s\right)\frac{\partial^{2}\widetilde{U^{\left(n-1\right)}}^{n}}{\partial x_{1}\partial x_{2}}\left(s\left(\phi^{\left(n\right)}\right)^{-1}\left(t,x\right)\right)\mathrm{d}s\right]\left(\phi^{\left(n\right)}\right)_{1}^{-1}\left(t,x\right)\left(\phi^{\left(n\right)}\right)_{2}^{-1}\left(t,x\right)+\\
 & \quad+\frac{1}{2}\left[\int_{0}^{1}\left(1-s\right)\frac{\partial^{2}\widetilde{U^{\left(n-1\right)}}^{n}}{\partial x_{2}^{2}}\left(s\left(\phi^{\left(n\right)}\right)^{-1}\left(t,x\right)\right)\mathrm{d}s\right]\left[\left(\phi^{\left(n\right)}\right)_{2}^{-1}\left(t,x\right)\right]^{2}.
\end{aligned}
\label{eq:Taylor remainder real variables}
\end{equation}
Bounding $1-s\leq1$ in the integrals and using Lemmas \ref{lem:bounds second order derivatives}
(which we can use because $n\ge2$) and \ref{lem:bounds polynomials},
as the $\left|\left|\cdot\right|\right|_{L^{\infty}\left(\mathbb{R}^{2}\right)}$
remains invariant under diffeomorphisms of the domain, we deduce that
\[
\left|\left|W^{\left(n\right)}\left(t,\cdot\right)\right|\right|_{L^{\infty}\left(D^{\left(n\right)}\left(t\right);\mathbb{R}^{2}\right)}\lesssim_{\delta,\varphi}Y\frac{C^{\left(1+\delta\right)\left(\frac{1}{1-\gamma}\right)^{n-1}}}{a_{n}\left(t\right)^{2}}\frac{1}{\lambda_{n}^{2}}+Y\frac{C^{\left(1+\delta\right)\left(\frac{1}{1-\gamma}\right)^{n-1}}}{a_{n}\left(t\right)b_{n}\left(t\right)}\frac{1}{\lambda_{n}^{2}}+Y\frac{C^{\left(1+\delta\right)\left(\frac{1}{1-\gamma}\right)^{n-1}}}{b_{n}\left(t\right)^{2}}\frac{1}{\lambda_{n}^{2}}.
\]
Provided that $C$ is taken large enough (let us say $C\ge\Upsilon_{1}\left(\delta,\mu\right)$),
Corollary \ref{cor:bound for an(t), bn(t)} and Choice \ref{choice:lambdan}
provide
\[
\begin{aligned}\left|\left|W^{\left(n\right)}\left(t,\cdot\right)\right|\right|_{L^{\infty}\left(D^{\left(n\right)}\left(t\right);\mathbb{R}^{2}\right)} & \lesssim_{\delta,\gamma}Y\frac{C^{\left(1+\delta\right)\left(\frac{1}{1-\gamma}\right)^{n-1}}}{\lambda_{n}^{2}}\frac{1}{\min\left\{ a_{n}\left(t\right)^{2},b_{n}\left(t\right)^{2}\right\} }\leq\\
 & \lesssim_{\delta,\gamma}YC^{\left[-2+2\overline{k}_{n}\left(t\right)+2\Lambda+2\mu+\left(1+\delta\right)\left(1-\gamma\right)\right]\left(\frac{1}{1-\gamma}\right)^{n}}.
\end{aligned}
\]
By Choice \ref{choice:ideal kn}, we can bound
\[
\left|\left|W^{\left(n\right)}\left(t,\cdot\right)\right|\right|_{L^{\infty}\left(D^{\left(n\right)}\left(t\right);\mathbb{R}^{2}\right)}\lesssim_{\delta,\gamma}YC^{\left[-2\left(1-k_{\max}-\Lambda\right)+2\mu+\left(1+\delta\right)\left(1-\gamma\right)\right]\left(\frac{1}{1-\gamma}\right)^{n}}.
\]

Let us continue with the $\left|\left|\cdot\right|\right|_{\dot{C}^{\alpha}}$
seminorm. By equation \eqref{eq:property Calpha multiplication},
we will obtain two summands for each summand of \eqref{eq:Taylor remainder real variables}.
On the one hand, when the $\left|\left|\cdot\right|\right|_{\dot{C}^{\alpha}}$
acts on one of the second derivatives, we resort to Lemma \ref{lem:bounds second order derivatives}.
Nonetheless, notice that we may not apply this Lemma directly because
there is an $s$ multiplying $\left(\phi^{\left(n\right)}\right)^{-1}\left(t,x\right)$.
However, this is not a problem, since we can make use of equation
\eqref{eq:property Calpha composition} to deal with this $s$. On
the other hand, when $\left|\left|\cdot\right|\right|_{\dot{C}^{\alpha}}$
hits one of the polynomials in $\left(\phi^{\left(n\right)}\right)^{-1}$,
we employ Lemma \ref{lem:bounds polynomials}. Thereby, bounding $1-s\leq1$
and $s\leq1$, we obtain
\[
\begin{aligned}\left|\left|W^{\left(n\right)}\left(t,\cdot\right)\right|\right|_{\dot{C}^{\alpha}\left(D^{\left(n\right)}\left(t\right);\mathbb{R}^{2}\right)} & \lesssim_{\delta,\varphi}Y\frac{C^{\left(1+\delta\right)\left(\frac{1}{1-\gamma}\right)^{n-1}}}{a_{n}\left(t\right)^{2}}\frac{a_{n}\left(t\right)^{\alpha}}{\lambda_{n}}+\frac{C^{\left(1+\alpha+\delta\right)\left(\frac{1}{1-\gamma}\right)^{n-1}}}{a_{n}\left(t\right)^{2}}\frac{1}{\lambda_{n}^{2}}+\\
 & \quad+Y\frac{C^{\left(1+\delta\right)\left(\frac{1}{1-\gamma}\right)^{n-1}}}{a_{n}\left(t\right)b_{n}\left(t\right)}\frac{\max\left\{ a_{n}\left(t\right)^{\alpha},b_{n}\left(t\right)^{\alpha}\right\} }{\lambda_{n}}+\frac{C^{\left(1+\alpha+\delta\right)\left(\frac{1}{1-\gamma}\right)^{n-1}}}{a_{n}\left(t\right)b_{n}\left(t\right)}\frac{1}{\lambda_{n}^{2}}+\\
 & \quad+Y\frac{C^{\left(1+\delta\right)\left(\frac{1}{1-\gamma}\right)^{n-1}}}{b_{n}\left(t\right)^{2}}\frac{b_{n}\left(t\right)^{\alpha}}{\lambda_{n}}+\frac{C^{\left(1+\alpha+\delta\right)\left(\frac{1}{1-\gamma}\right)^{n-1}}}{b_{n}\left(t\right)^{2}}\frac{1}{\lambda_{n}^{2}}.
\end{aligned}
\]
Clearly,
\[
\begin{aligned}\left|\left|W^{\left(n\right)}\left(t,\cdot\right)\right|\right|_{\dot{C}^{\alpha}\left(D^{\left(n\right)}\left(t\right);\mathbb{R}^{2}\right)} & \lesssim_{\delta,\varphi}Y\frac{C^{\left(1+\delta\right)\left(\frac{1}{1-\gamma}\right)^{n-1}}}{\lambda_{n}}\max\left\{ \frac{1}{a_{n}\left(t\right)^{2-\alpha}},\frac{\max\left\{ a_{n}\left(t\right)^{\alpha},b_{n}\left(t\right)^{\alpha}\right\} }{a_{n}\left(t\right)b_{n}\left(t\right)},\frac{1}{b_{n}\left(t\right)^{2-\alpha}}\right\} +\\
 & \quad+Y\frac{C^{\left(1+\alpha+\delta\right)\left(\frac{1}{1-\gamma}\right)^{n-1}}}{\min\left\{ a_{n}\left(t\right)^{2},b_{n}\left(t\right)^{2}\right\} }\frac{1}{\lambda_{n}^{2}}.
\end{aligned}
\]
Corollary \ref{cor:bound for an(t), bn(t)} and Choices \ref{choice:anbn}
and \ref{choice:lambdan} provide
\[
\begin{aligned} & \left|\left|W^{\left(n\right)}\left(t,\cdot\right)\right|\right|_{\dot{C}^{\alpha}\left(D^{\left(n\right)}\left(t\right);\mathbb{R}^{2}\right)}\lesssim_{\delta,\varphi}\\
\lesssim_{\delta,\varphi} & YC^{\left[\Lambda+\max\left\{ -\left(2-\alpha\right)\left(1-\overline{k}_{n}\left(t\right)-\mu\right),-2+\alpha\left(1+\overline{k}_{n}\left(t\right)+\mu\right),-\left(2-\alpha\right)\left(1+\overline{k}_{n}\left(t\right)-\mu\right)\right\} +\left(1+\delta\right)\left(1-\gamma\right)\right]\left(\frac{1}{1-\gamma}\right)^{n}}+\\
 & +YC^{\left[2\Lambda-2\left(1-\overline{k}_{n}\left(t\right)-\mu\right)+\left(1+\alpha+\delta\right)\left(1-\gamma\right)\right]\left(\frac{1}{1-\gamma}\right)^{n}}.
\end{aligned}
\]
Notice that
\[
\begin{aligned}-\left(2-\alpha\right)\left(1-\overline{k}_{n}\left(t\right)-\mu\right) & =-2+\alpha+\left(2-\alpha\right)\overline{k}_{n}\left(t\right)+\left(2-\alpha\right)\mu,\\
-2+\alpha\left(1+\overline{k}_{n}\left(t\right)+\mu\right) & =-2+\alpha+\alpha\overline{k}_{n}\left(t\right)+\alpha\mu.
\end{aligned}
\]
As $\alpha\in\left(0,1\right)$, which implies that $\alpha\le2-\alpha$,
and $\overline{k}_{n}\left(t\right)\ge0$ by Corollary \ref{cor:ideal kn non negative},
we infer that
\[
\begin{aligned} & \left|\left|W^{\left(n\right)}\left(t,\cdot\right)\right|\right|_{\dot{C}^{\alpha}\left(D^{\left(n\right)}\left(t\right);\mathbb{R}^{2}\right)}\lesssim_{\delta,\varphi}\\
\lesssim_{\delta,\varphi} & YC^{\left[\Lambda-\left(2-\alpha\right)\left(1-\overline{k}_{n}\left(t\right)-\mu\right)+\left(1+\delta\right)\left(1-\gamma\right)\right]\left(\frac{1}{1-\gamma}\right)^{n}}+YC^{\left[2\Lambda-2\left(1-\overline{k}_{n}\left(t\right)-\mu\right)+\left(2+\delta\right)\left(1-\gamma\right)\right]\left(\frac{1}{1-\gamma}\right)^{n}}.
\end{aligned}
\]
Bounding $\overline{k}_{n}\left(t\right)\leq k_{\max}$ (see Choice
\ref{choice:ideal kn}), we obtain
\[
\begin{aligned}\left|\left|W^{\left(n\right)}\left(t,\cdot\right)\right|\right|_{\dot{C}^{\alpha}\left(D^{\left(n\right)}\left(t\right);\mathbb{R}^{2}\right)} & \lesssim_{\delta,\varphi}YC^{\left[-\min\left\{ \left(2-\alpha\right)\left(1-k_{\max}-\mu\right)-\Lambda,2\left(1-k_{\max}-\mu-\Lambda\right)\right\} +\left(2+\delta\right)\left(1-\gamma\right)\right]\left(\frac{1}{1-\gamma}\right)^{n}}\lesssim\\
 & \lesssim_{\delta,\varphi}YC^{\left[-2\left(1-\Lambda-k_{\max}\right)+\max\left\{ \alpha\left(1-k_{\max}\right)-\Lambda,0\right\} +2\mu+\left(2+\delta\right)\left(1-\gamma\right)\right]\left(\frac{1}{1-\gamma}\right)^{n}}.
\end{aligned}
\]

Next in line would be to compute $\left|\left|W^{\left(n\right)}\left(t,\cdot\right)\right|\right|_{\dot{C}^{1}\left(D^{\left(n\right)}\left(t\right);\mathbb{R}^{2}\right)}$.
Differentiating under the integral sign in \eqref{eq:Taylor remainder real variables},
we will obtain two summands for each term that appears in equation
\eqref{eq:Taylor remainder real variables}. Making use of Lemmas
\ref{lem:bounds second order derivatives} and \ref{lem:bounds polynomials},
formally, we can follow exactly the same procedure we introduced to
compute $\left|\left|W^{\left(n\right)}\left(t,\cdot\right)\right|\right|_{\dot{C}^{\alpha}\left(D^{\left(n\right)}\left(t\right);\mathbb{R}^{2}\right)}$
taking $\alpha=1$. Therefore,
\[
\begin{aligned}\left|\left|W^{\left(n\right)}\left(t,\cdot\right)\right|\right|_{\dot{C}^{1}\left(D^{\left(n\right)}\left(t\right);\mathbb{R}^{2}\right)} & \lesssim_{\delta,\varphi}YC^{\left[-2\left(1-\Lambda-k_{\max}\right)+\max\left\{ 1-k_{\max}-\Lambda,0\right\} +2\mu+\left(2+\delta\right)\left(1-\gamma\right)\right]\left(\frac{1}{1-\gamma}\right)^{n}}=\\
 & \lesssim_{\delta,\varphi}YC^{\left[-\min\left\{ 1-\Lambda-k_{\max},2\left(1-\Lambda-k_{\max}\right)\right\} +2\mu+\left(2+\delta\right)\left(1-\gamma\right)\right]\left(\frac{1}{1-\gamma}\right)^{n}}.
\end{aligned}
\]

Lastly, we have to calculate $\left|\left|W^{\left(n\right)}\left(t,\cdot\right)\right|\right|_{\dot{C}^{1,\alpha}\left(D^{\left(n\right)}\left(t\right);\mathbb{R}^{2}\right)}$.
To undertake this task, we will first differentiate under the integral
sign in \eqref{eq:Taylor remainder real variables}, duplicating the
number of summands that appear and, then, we will calculate the $\left|\left|\cdot\right|\right|_{\dot{C}^{\alpha}\left(D^{\left(n\right)}\left(t\right);\mathbb{R}^{2}\right)}$
norm of this derivative. By equation \eqref{eq:property Calpha multiplication},
this will again duplicate the number of summands. Basically, this
corresponds to applying a sort of ``Leibniz rule'' to each summand
of \eqref{eq:Taylor remainder real variables}. Schematically, for
each summand, we will have
\begin{equation}
\begin{aligned}\left|\left|F_{1}F_{2}\right|\right|_{\dot{C}^{1,\alpha}\left(D^{\left(n\right)}\left(t\right);\mathbb{R}^{2}\right)} & \lesssim\left|\left|F_{1}\right|\right|_{L^{\infty}\left(D^{\left(n\right)}\left(t\right);\mathbb{R}^{2}\right)}\left|\left|F_{2}\right|\right|_{\dot{C}^{1,\alpha}\left(D^{\left(n\right)}\left(t\right);\mathbb{R}^{2}\right)}+\\
 & \quad+\left|\left|F_{1}\right|\right|_{\dot{C}^{\alpha}\left(D^{\left(n\right)}\left(t\right);\mathbb{R}^{2}\right)}\left|\left|F_{2}\right|\right|_{\dot{C}^{1}\left(D^{\left(n\right)}\left(t\right);\mathbb{R}^{2}\right)}+\\
 & \quad+\left|\left|F_{1}\right|\right|_{\dot{C}^{1}\left(D^{\left(n\right)}\left(t\right);\mathbb{R}^{2}\right)}\left|\left|F_{2}\right|\right|_{\dot{C}^{\alpha}\left(D^{\left(n\right)}\left(t\right);\mathbb{R}^{2}\right)}+\\
 & \quad+\left|\left|F_{1}\right|\right|_{\dot{C}^{1,\alpha}\left(D^{\left(n\right)}\left(t\right);\mathbb{R}^{2}\right)}\left|\left|F_{2}\right|\right|_{L^{\infty}\left(D^{\left(n\right)}\left(t\right);\mathbb{R}^{2}\right)}.
\end{aligned}
\label{eq:informal leibniz rule}
\end{equation}
Thereby,
\[
\begin{aligned} & \left|\left|W^{\left(n\right)}\left(t,\cdot\right)\right|\right|_{\dot{C}^{1,\alpha}\left(D^{\left(n\right)}\left(t\right);\mathbb{R}^{2}\right)}\lesssim_{\delta,\varphi}\\
\lesssim_{\delta,\varphi} & Y\frac{C^{\left(1+\delta\right)\left(\frac{1}{1-\gamma}\right)^{n-1}}}{a_{n}\left(t\right)^{2}}a_{n}\left(t\right)^{1+\alpha}+Y\frac{C^{\left(1+\alpha+\delta\right)\left(\frac{1}{1-\gamma}\right)^{n-1}}}{a_{n}\left(t\right)^{2}}\frac{a_{n}\left(t\right)}{\lambda_{n}}+\\
 & \quad+Y\frac{C^{\left(2+\delta\right)\left(\frac{1}{1-\gamma}\right)^{n-1}}}{a_{n}\left(t\right)^{2}}\frac{a_{n}\left(t\right)^{\alpha}}{\lambda_{n}}+Y\frac{C^{\left(2+\alpha+\delta\right)\left(\frac{1}{1-\gamma}\right)^{n-1}}}{a_{n}\left(t\right)^{2}}\frac{1}{\lambda_{n}^{2}}+\\
 & +Y\frac{C^{\left(1+\delta\right)\left(\frac{1}{1-\gamma}\right)^{n-1}}}{a_{n}\left(t\right)b_{n}\left(t\right)}\max\left\{ a_{n}\left(t\right)b_{n}\left(t\right)^{\alpha},a_{n}\left(t\right)^{\alpha}b_{n}\left(t\right)\right\} +Y\frac{C^{\left(1+\alpha+\delta\right)\left(\frac{1}{1-\gamma}\right)^{n-1}}}{a_{n}\left(t\right)b_{n}\left(t\right)}\frac{\max\left\{ a_{n}\left(t\right),b_{n}\left(t\right)\right\} }{\lambda_{n}}+\\
 & \quad+Y\frac{C^{\left(2+\delta\right)\left(\frac{1}{1-\gamma}\right)^{n-1}}}{a_{n}\left(t\right)b_{n}\left(t\right)}\frac{\max\left\{ a_{n}\left(t\right)^{\alpha},b_{n}\left(t\right)^{\alpha}\right\} }{\lambda_{n}}+Y\frac{C^{\left(2+\alpha+\delta\right)\left(\frac{1}{1-\gamma}\right)^{n-1}}}{a_{n}\left(t\right)b_{n}\left(t\right)}\frac{1}{\lambda_{n}^{2}}+\\
 & +Y\frac{C^{\left(1+\delta\right)\left(\frac{1}{1-\gamma}\right)^{n-1}}}{b_{n}\left(t\right)^{2}}b_{n}\left(t\right)^{1+\alpha}+Y\frac{C^{\left(1+\alpha+\delta\right)\left(\frac{1}{1-\gamma}\right)^{n-1}}}{b_{n}\left(t\right)^{2}}\frac{b_{n}\left(t\right)}{\lambda_{n}}+\\
 & \quad+Y\frac{C^{\left(2+\delta\right)\left(\frac{1}{1-\gamma}\right)^{n-1}}}{b_{n}\left(t\right)^{2}}\frac{b_{n}\left(t\right)^{\alpha}}{\lambda_{n}}+Y\frac{C^{\left(2+\alpha+\delta\right)\left(\frac{1}{1-\gamma}\right)^{n-1}}}{b_{n}\left(t\right)^{2}}\frac{1}{\lambda_{n}^{2}}.
\end{aligned}
\]
Clearly,
\[
\begin{aligned} & \left|\left|W^{\left(n\right)}\left(t,\cdot\right)\right|\right|_{\dot{C}^{1,\alpha}\left(D^{\left(n\right)}\left(t\right);\mathbb{R}^{2}\right)}\lesssim_{\delta,\varphi}\\
\lesssim_{\delta,\varphi} & YC^{\left(1+\delta\right)\left(\frac{1}{1-\gamma}\right)^{n-1}}\max\left\{ \frac{1}{a_{n}\left(t\right)^{1-\alpha}},\frac{\max\left\{ a_{n}\left(t\right)b_{n}\left(t\right)^{\alpha},a_{n}\left(t\right)^{\alpha}b_{n}\left(t\right)\right\} }{a_{n}\left(t\right)b_{n}\left(t\right)},\frac{1}{b_{n}\left(t\right)^{1-\alpha}}\right\} +\\
 & \quad+Y\frac{C^{\left(1+\alpha+\delta\right)\left(\frac{1}{1-\gamma}\right)^{n-1}}}{\lambda_{n}}\max\left\{ \frac{1}{a_{n}\left(t\right)},\frac{\max\left\{ a_{n}\left(t\right),b_{n}\left(t\right)\right\} }{a_{n}\left(t\right)b_{n}\left(t\right)},\frac{1}{b_{n}\left(t\right)}\right\} +\\
 & \quad+Y\frac{C^{\left(2+\delta\right)\left(\frac{1}{1-\gamma}\right)^{n-1}}}{\lambda_{n}}\max\left\{ \frac{1}{a_{n}\left(t\right)^{2-\alpha}},\frac{\max\left\{ a_{n}\left(t\right)^{\alpha},b_{n}\left(t\right)^{\alpha}\right\} }{a_{n}\left(t\right)b_{n}\left(t\right)},\frac{1}{b_{n}\left(t\right)^{2-\alpha}}\right\} +\\
 & \quad+Y\frac{C^{\left(2+\alpha+\delta\right)\left(\frac{1}{1-\gamma}\right)^{n-1}}}{\lambda_{n}^{2}}\frac{1}{\min\left\{ a_{n}\left(t\right)^{2},b_{n}\left(t\right)^{2}\right\} }.
\end{aligned}
\]
Observe that
\[
\frac{\max\left\{ a_{n}\left(t\right)b_{n}\left(t\right)^{\alpha},a_{n}\left(t\right)^{\alpha}b_{n}\left(t\right)\right\} }{a_{n}\left(t\right)b_{n}\left(t\right)}=\max\left\{ \frac{1}{b_{n}\left(t\right)^{1-\alpha}},\frac{1}{a_{n}\left(t\right)^{1-\alpha}}\right\} =\frac{1}{\min\left\{ a_{n}\left(t\right)^{1-\alpha},b_{n}\left(t\right)^{1-\alpha}\right\} }.
\]
In this way,
\[
\begin{aligned}\left|\left|W^{\left(n\right)}\left(t,\cdot\right)\right|\right|_{\dot{C}^{1,\alpha}\left(D^{\left(n\right)}\left(t\right);\mathbb{R}^{2}\right)} & \lesssim_{\delta,\varphi}YC^{\left(1+\delta\right)\left(\frac{1}{1-\gamma}\right)^{n-1}}\frac{1}{\min\left\{ a_{n}\left(t\right)^{1-\alpha},b_{n}\left(t\right)^{1-\alpha}\right\} }+\\
 & \quad+Y\frac{C^{\left(1+\alpha+\delta\right)\left(\frac{1}{1-\gamma}\right)^{n-1}}}{\lambda_{n}}\max\left\{ \frac{1}{a_{n}\left(t\right)},\frac{1}{b_{n}\left(t\right)}\right\} +\\
 & \quad+Y\frac{C^{\left(2+\delta\right)\left(\frac{1}{1-\gamma}\right)^{n-1}}}{\lambda_{n}}\max\left\{ \frac{1}{a_{n}\left(t\right)^{2-\alpha}},\frac{\max\left\{ a_{n}\left(t\right)^{\alpha},b_{n}\left(t\right)^{\alpha}\right\} }{a_{n}\left(t\right)b_{n}\left(t\right)},\frac{1}{b_{n}\left(t\right)^{2-\alpha}}\right\} +\\
 & \quad+Y\frac{C^{\left(2+\alpha+\delta\right)\left(\frac{1}{1-\gamma}\right)^{n-1}}}{\lambda_{n}^{2}}\frac{1}{\min\left\{ a_{n}\left(t\right)^{2},b_{n}\left(t\right)^{2}\right\} }.
\end{aligned}
\]
We can apply the same procedure we have used for $\left|\left|W^{\left(n\right)}\left(t,\cdot\right)\right|\right|_{\dot{C}^{\alpha}\left(D^{\left(n\right)}\left(t\right);\mathbb{R}^{2}\right)}$
to bound the second and third terms. The fact that $\alpha\in\left(0,1\right)$,
Corollary \ref{cor:bound for an(t), bn(t)} and Choices \ref{choice:anbn}
and \ref{choice:lambdan} let us write
\[
\begin{aligned}\left|\left|W^{\left(n\right)}\left(t,\cdot\right)\right|\right|_{\dot{C}^{1,\alpha}\left(D^{\left(n\right)}\left(t\right);\mathbb{R}^{2}\right)} & \lesssim_{\delta,\varphi}YC^{\left[-\left(1-\alpha\right)\left(1-\overline{k}_{n}\left(t\right)-\mu\right)+\left(1+\delta\right)\left(1-\gamma\right)\right]\left(\frac{1}{1-\gamma}\right)^{n}}+\\
 & \quad+YC^{\left[\Lambda-\left(1-\overline{k}_{n}\left(t\right)-\mu\right)+\left(2+\delta\right)\left(1-\gamma\right)\right]\left(\frac{1}{1-\gamma}\right)^{n}}+\\
 & \quad+YC^{\left[\Lambda-\left(2-\alpha\right)\left(1-\overline{k}_{n}\left(t\right)-\mu\right)+\left(2+\delta\right)\left(1-\gamma\right)\right]\left(\frac{1}{1-\gamma}\right)^{n}}+\\
 & \quad+YC^{\left[2\Lambda-2\left(1-\overline{k}_{n}\left(t\right)-\mu\right)+\left(3+\delta\right)\left(1-\gamma\right)\right]\left(\frac{1}{1-\gamma}\right)^{n}}.
\end{aligned}
\]
Bounding $\overline{k}_{n}\left(t\right)\leq k_{\max}$ (see Choice
\ref{choice:ideal kn}), we obtain
\[
\begin{aligned} & \left|\left|W^{\left(n\right)}\left(t,\cdot\right)\right|\right|_{\dot{C}^{1,\alpha}\left(D^{\left(n\right)}\left(t\right);\mathbb{R}^{2}\right)}\lesssim_{\delta,\varphi}\\
\lesssim_{\delta,\varphi} & YC^{\left[-\min\left\{ \left(1-\alpha\right)\left(1-k_{\max}\right),1-k_{\max}-\Lambda,\left(2-\alpha\right)\left(1-k_{\max}\right)-\Lambda,2\left(1-k_{\max}-\Lambda\right)\right\} +2\mu+\left(3+\delta\right)\left(1-\gamma\right)\right]\left(\frac{1}{1-\gamma}\right)^{n}}\lesssim\\
\lesssim_{\delta,\varphi} & YC^{\left[-\min\left\{ \left(1-\alpha\right)\left(1-k_{\max}\right),1-k_{\max}-\Lambda,2\left(1-k_{\max}-\Lambda\right)\right\} +2\mu+\left(3+\delta\right)\left(1-\gamma\right)\right]\left(\frac{1}{1-\gamma}\right)^{n}},
\end{aligned}
\]
where we have applied that $\alpha\le2-\alpha$ in the last step to
bound the third term of the minimum by the second one.
\end{proof}
Lastly, we are prepared to bound the transport term of the density
equation!
\begin{prop}
\label{prop:bound density transport}Let $n\in\mathbb{N}$ with $n\ge2$,
$\alpha\in\left(0,1\right)$, $\mu>0$ and assume that $1-k_{\max}-\Lambda>0$.
As long as $C$ is taken large enough (let us say $C\ge\Upsilon\left(\delta,\mu\right)$),
the transport term
\[
\widetilde{T_{\rho}^{\left(n\right)}}^{n}\left(t,x\right)\coloneqq\left(\widetilde{U^{\left(n-1\right)}}^{n}\left(t,x\right)-\widetilde{U^{\left(n-1\right)}}^{n}\left(t,0\right)-\mathrm{J}\widetilde{U^{\left(n-1\right)}}^{n}\left(t,0\right)\cdot\left(\begin{matrix}x_{1}\\
x_{2}
\end{matrix}\right)\right)\cdot\widetilde{\nabla}^{n}\widetilde{\rho^{\left(n\right)}}^{n}\left(t,x\right)
\]
satisfies the bound
\[
\begin{aligned}\left|\left|T_{\rho}^{\left(n\right)}\left(t,\cdot\right)\right|\right|_{C^{1,\alpha}\left(\mathbb{R}^{2}\right)} & \lesssim Y^{2}C^{\left[6\mu+3\delta+3\left(1-\gamma\right)\right]\left(\frac{1}{1-\gamma}\right)^{n}}\cdot\\
 & \quad\cdot\left[C^{-\min\left\{ \left(1-\alpha\right)\left(1-k_{\max}\right),1-k_{\max}-\Lambda\right\} \left(\frac{1}{1-\gamma}\right)^{n}}+C^{\left[2\Lambda+3k_{\max}-1+\max\left\{ \alpha\left(1-k_{\max}\right)-\Lambda,0\right\} \right]\left(\frac{1}{1-\gamma}\right)^{n}}+\right.\\
 & \qquad\left.+C^{\left[\alpha\left(1+k_{\max}\right)+2\Lambda+3k_{\max}-1\right]\left(\frac{1}{1-\gamma}\right)^{n}}\right].
\end{aligned}
\]
In other words, for the term above to be decreasing in $n\in\mathbb{N}$,
we need
\begin{enumerate}
\item ~
\[
-\left(1-\alpha\right)\left(1-k_{\max}\right)+6\mu+3\delta+3\left(1-\gamma\right)<0,
\]
\item ~
\[
-\left(1-k_{\max}-\Lambda\right)+6\mu+3\delta+3\left(1-\gamma\right)<0,
\]
\item ~
\[
\alpha\left(1+k_{\max}\right)+2\Lambda+3k_{\max}-1+6\mu+3\delta+3\left(1-\gamma\right)<0,
\]
\item if $\alpha\left(1-k_{\max}\right)-\Lambda>0$, we also require
\[
\alpha\left(1-k_{\max}\right)+\Lambda+3k_{\max}-1+6\mu+3\delta+3\left(1-\gamma\right)<0.
\]
\end{enumerate}
\end{prop}
\begin{proof}
We shall make use of Propositions \ref{prop:form of gradient of rho}
and \ref{prop:bounds for Taylor development of transport}, which
we may make use of because $n\ge2$ and provided that $C$ is big
enough (let us say $C\ge\Upsilon_{1}\left(\delta,\mu\right)$). Notice
that, as $\widetilde{\nabla}^{n}\widetilde{\rho^{\left(n\right)}}^{n}$
has compact support exactly where the bounds of Proposition \ref{prop:bounds for Taylor development of transport}
are computed (see Choices \ref{choice:density} and \ref{choice:psin}),
we can use the bounds of Proposition \ref{prop:bounds for Taylor development of transport}
to bound $T_{\rho}^{\left(n\right)}$ in all $\mathbb{R}^{2}$. Bounding
$1-\gamma\le1$ in the terms of the form $\delta\left(1-\gamma\right)$,
it is immediate to obtain that
\begin{equation}
\begin{aligned}\left|\left|T_{\rho}^{\left(n\right)}\left(t,\cdot\right)\right|\right|_{L^{\infty}\left(\mathbb{R}^{2}\right)} & \lesssim_{\delta,\varphi,\mu}YC^{\left[-2\left(1-\Lambda-k_{\max}\right)+2\mu+\left(1+\delta\right)\left(1-\gamma\right)\right]\left(\frac{1}{1-\gamma}\right)^{n}}YC^{\left[2\delta+2\mu\right]\left(\frac{1}{1-\gamma}\right)^{n}}=\\
 & \lesssim_{\delta,\varphi,\mu}Y^{2}C^{\left[-2\left(1-\Lambda-k_{\max}\right)+4\mu+3\delta+\left(1-\gamma\right)\right]\left(\frac{1}{1-\gamma}\right)^{n}}.
\end{aligned}
\label{eq:bound transport density Linfty}
\end{equation}

Applying Leibniz's rule, it is not difficult to deduce a bound for
the $\left|\left|\cdot\right|\right|_{\dot{C}^{1}\left(\mathbb{R}^{2}\right)}$
seminorm. Indeed,
\[
\begin{aligned}\left|\left|T_{\rho}^{\left(n\right)}\left(t,\cdot\right)\right|\right|_{\dot{C}^{1}\left(\mathbb{R}^{2}\right)} & \lesssim_{\delta,\varphi,\mu}YC^{\left[-2\left(1-\Lambda-k_{\max}\right)+2\mu+\left(1+\delta\right)\left(1-\gamma\right)\right]\left(\frac{1}{1-\gamma}\right)^{n}}YC^{\left[\left(1+k_{\max}\right)+2\delta+3\mu\right]\left(\frac{1}{1-\gamma}\right)^{n}}+\\
 & \quad+YC^{\left[-\left(1-\Lambda-k_{\max}\right)+2\mu+\left(2+\delta\right)\left(1-\gamma\right)\right]\left(\frac{1}{1-\gamma}\right)^{n}}YC^{\left[2\delta+2\mu\right]\left(\frac{1}{1-\gamma}\right)^{n}}.
\end{aligned}
\]
Since
\[
-2\left(1-\Lambda-k_{\max}\right)+1+k_{\max}=-1+2\Lambda+3k_{\max}\ge-1+\Lambda+k_{\max}=-\left(1-\Lambda-k_{\max}\right),
\]
we conclude that
\begin{equation}
\left|\left|T_{\rho}^{\left(n\right)}\left(t,\cdot\right)\right|\right|_{\dot{C}^{1}\left(\mathbb{R}^{2}\right)}\lesssim_{\delta,\varphi,\mu}YC^{\left[-\left(1-2\Lambda-3k_{\max}\right)+5\mu+3\delta+2\left(1-\gamma\right)\right]\left(\frac{1}{1-\gamma}\right)^{n}}.\label{eq:bound transport density C1}
\end{equation}

Next in line is to compute the $\left|\left|\cdot\right|\right|_{\dot{C}^{1,\alpha}\left(\mathbb{R}^{2}\right)}$
seminorm. For this, we can apply the ``Leibniz rule'' introduced
in equation \eqref{eq:informal leibniz rule}. Thereby,
\[
\begin{aligned}\left|\left|T_{\rho}^{\left(n\right)}\left(t,\cdot\right)\right|\right|_{\dot{C}^{1,\alpha}\left(\mathbb{R}^{2}\right)} & \lesssim_{\delta,\varphi,\mu}YC^{\left[-\min\left\{ \left(1-\alpha\right)\left(1-k_{\max}\right),1-k_{\max}-\Lambda\right\} +2\mu+\left(3+\delta\right)\left(1-\gamma\right)\right]\left(\frac{1}{1-\gamma}\right)^{n}}YC^{\left[2\delta+2\mu\right]\left(\frac{1}{1-\gamma}\right)^{n}}+\\
 & \quad+YC^{\left[-\left(1-\Lambda-k_{\max}\right)+2\mu+\left(2+\delta\right)\left(1-\gamma\right)\right]\left(\frac{1}{1-\gamma}\right)^{n}}YC^{\left[\alpha\left(1+k_{\max}\right)+2\delta+3\mu\right]\left(\frac{1}{1-\gamma}\right)^{n}}+\\
 & \quad+YC^{\left[-2\left(1-\Lambda-k_{\max}\right)+\max\left\{ \alpha\left(1-k_{\max}\right)-\Lambda,0\right\} +2\mu+\left(2+\delta\right)\left(1-\gamma\right)\right]}YC^{\left[\left(1+k_{\max}\right)+2\delta+3\mu\right]\left(\frac{1}{1-\gamma}\right)^{n}}+\\
 & \quad+YC^{\left[-2\left(1-\Lambda-k_{\max}\right)+2\mu+\left(1+\delta\right)\left(1-\gamma\right)\right]\left(\frac{1}{1-\gamma}\right)^{n}}YC^{\left(\left(1+\alpha\right)\left(1+k_{\max}\right)+2\delta+4\mu\right)\left(\frac{1}{1-\gamma}\right)^{n}}\lesssim\\
 & \lesssim_{\delta,\varphi,\mu}Y^{2}C^{\left[6\mu+3\delta+3\left(1-\gamma\right)\right]\left(\frac{1}{1-\gamma}\right)^{n}}\cdot\\
 & \quad\cdot\left[C^{-\min\left\{ \left(1-\alpha\right)\left(1-k_{\max}\right),1-k_{\max}-\Lambda\right\} \left(\frac{1}{1-\gamma}\right)^{n}}+C^{\left[\alpha\left(1+k_{\max}\right)+\Lambda+k_{\max}-1\right]\left(\frac{1}{1-\gamma}\right)^{n}}+\right.\\
 & \qquad\left.C^{\left[2\Lambda+3k_{\max}-1+\max\left\{ \alpha\left(1-k_{\max}\right)-\Lambda,0\right\} \right]\left(\frac{1}{1-\gamma}\right)^{n}}+C^{\left[\alpha\left(1+k_{\max}\right)+2\Lambda+3k_{\max}-1\right]\left(\frac{1}{1-\gamma}\right)^{n}}\right].
\end{aligned}
\]
As $\Lambda>0$ by Choice \ref{choice:lambdan} and $k_{\max}>0$
by Choice \ref{choice:min requirements on C gamma and kmax}, we can
bound
\[
\alpha\left(1+k_{\max}\right)+\Lambda+k_{\max}-1\le\alpha\left(1+k_{\max}\right)+2\Lambda+3k_{\max}-1.
\]
Consequently,
\begin{equation}
\begin{aligned}\left|\left|T_{\rho}^{\left(n\right)}\left(t,\cdot\right)\right|\right|_{\dot{C}^{1,\alpha}\left(\mathbb{R}^{2}\right)} & \lesssim Y^{2}C^{\left[6\mu+3\delta+3\left(1-\gamma\right)\right]\left(\frac{1}{1-\gamma}\right)^{n}}\cdot\\
 & \quad\cdot\left[C^{-\min\left\{ \left(1-\alpha\right)\left(1-k_{\max}\right),1-k_{\max}-\Lambda\right\} \left(\frac{1}{1-\gamma}\right)^{n}}+C^{\left[2\Lambda+3k_{\max}-1+\max\left\{ \alpha\left(1-k_{\max}\right)-\Lambda,0\right\} \right]\left(\frac{1}{1-\gamma}\right)^{n}}+\right.\\
 & \qquad\left.+C^{\left[\alpha\left(1+k_{\max}\right)+2\Lambda+3k_{\max}-1\right]\left(\frac{1}{1-\gamma}\right)^{n}}\right].
\end{aligned}
\label{eq:bound transport density C1alpha}
\end{equation}

Combining equations \eqref{eq:bound transport density Linfty}, \eqref{eq:bound transport density C1}
and \eqref{eq:bound transport density C1alpha}, we see that the $\left|\left|\cdot\right|\right|_{\dot{C}^{1,\alpha}\left(\mathbb{R}^{2}\right)}$
seminorm is the most restrictive. Indeed, since, by hypothesis, $1-k_{\max}-\Lambda>0$,
we deduce that $1-k_{\max}-\Lambda<2\left(1-k_{\max}-\Lambda\right)$.
In this way, the $\left|\left|\cdot\right|\right|_{L^{\infty}\left(\mathbb{R}^{2}\right)}$
norm is bounded by the first summand of equation \eqref{eq:bound transport density C1alpha}.
Moreover, the $\left|\left|\cdot\right|\right|_{\dot{C}^{1}\left(\mathbb{R}^{2}\right)}$
seminorm is bounded by the second summand of equation \eqref{eq:bound transport density C1alpha}.
This provides the result of the statement.
\end{proof}

\section{\label{sec:bounds for vorticity force}Bounds for vorticity force}

Equation \eqref{eq:Boussinesq vorticity Taylor}, along with Choice
\ref{choice:density}, tells us how to exactly compute the force $\widetilde{f_{\omega}^{\left(n\right)}}^{n}\left(t,x\right)$.
We distinguish the following terms:
\begin{equation}
\begin{aligned} & \underbrace{\frac{\partial\widetilde{\omega^{\left(n\right)}}^{n}}{\partial t}\left(t,x\right)-\left(\begin{matrix}0 & 1\end{matrix}\right)\cdot\left[\widetilde{\nabla}^{n}\widetilde{\rho^{\left(n\right)}}^{n}\left(t,x\right)\right]}_{\text{time factor}}+\\
 & +\underbrace{\left(\widetilde{U^{\left(n-1\right)}}^{n}\left(t,x\right)-\widetilde{U^{\left(n-1\right)}}^{n}\left(t,0\right)-\mathrm{J}\widetilde{U^{\left(n-1\right)}}^{n}\left(t,0\right)\cdot\left(\begin{matrix}x_{1}\\
x_{2}
\end{matrix}\right)\right)\cdot\widetilde{\nabla}^{n}\widetilde{\omega^{\left(n\right)}}^{n}\left(t,x\right)}_{\text{transport term}}+\\
 & +\underbrace{\widetilde{u^{\left(n\right)}}^{n}\left(t,x\right)\cdot\widetilde{\nabla}^{n}\widetilde{\Omega^{\left(n-1\right)}}^{n}}_{\text{new transport of old vorticity}}+\underbrace{\widetilde{u^{\left(n\right)}}^{n}\left(t,x\right)\cdot\widetilde{\nabla}^{n}\widetilde{\omega^{\left(n\right)}}^{n}\left(t,x\right)}_{\text{pure quadratic term}}=\\
= & \widetilde{f_{\omega}^{\left(n\right)}}^{n}\left(t,x\right).
\end{aligned}
\label{eq:decomposition force vorticity}
\end{equation}

\subsection{Time factor}
\begin{lem}
\label{lem:time derivatives amplitudes vorticity}Let $n\in\mathbb{N}$
with $n\ge2$ and $\mu>0$. Provided that $C$ is taken large enough
(let us say $C\ge\Upsilon\left(\delta,\mu\right)$), we have
\[
\begin{aligned}\left|\frac{\mathrm{d}}{\mathrm{d}t}\left(B_{n}\left(t\right)a_{n}\left(t\right)^{2}\right)\right| & \lesssim_{\delta,\mu}Y^{2}C^{\left(2\delta+2\mu\right)\left(\frac{1}{1-\gamma}\right)^{n}},\\
\left|\frac{\mathrm{d}}{\mathrm{d}t}\left(B_{n}\left(t\right)b_{n}\left(t\right)^{2}\right)\right| & \lesssim_{\delta,\mu}Y^{2}C^{\left(2\delta+2\mu\right)\left(\frac{1}{1-\gamma}\right)^{n}},\\
\left|\frac{\mathrm{d}}{\mathrm{d}t}\left(B_{n}\left(t\right)\left(a_{n}\left(t\right)^{2}+b_{n}\left(t\right)^{2}\right)\right)\right| & \lesssim_{\mu}Y^{2}C^{\left(2\delta+2\mu\right)\left(\frac{1}{1-\gamma}\right)^{n}}.
\end{aligned}
\]
\end{lem}
\begin{proof}
First, we will bound $\frac{\mathrm{d}}{\mathrm{d}t}\left(B_{n}\left(t\right)b_{n}\left(t\right)^{2}\right)$.
By Leibniz's rule, we deduce that
\begin{equation}
\begin{aligned}\frac{\mathrm{d}}{\mathrm{d}t}\left(B_{n}\left(t\right)b_{n}\left(t\right)^{2}\right) & =\frac{\mathrm{d}B_{n}}{\mathrm{d}t}\left(t\right)b_{n}\left(t\right)^{2}+2B_{n}\left(t\right)b_{n}\left(t\right)\frac{\mathrm{d}b_{n}}{\mathrm{d}t}\left(t\right)=\\
 & =\left[\frac{\mathrm{d}B_{n}}{\mathrm{d}t}\left(t\right)+2B_{n}\left(t\right)\frac{\mathrm{d}}{\mathrm{d}t}\left(\ln\left(b_{n}\left(t\right)\right)\right)\right]b_{n}\left(t\right)^{2}.
\end{aligned}
\label{eq:Leibniz rule Bn(t)bn(t)2}
\end{equation}
On the one hand, by Choice \ref{choice:amplitude density}, equation
\eqref{eq:relation Mn and zn} and Choice \ref{choice:time picture},
we have
\begin{equation}
B_{n}\left(t\right)=\frac{2M_{n}}{a_{n}\left(t\right)^{2}+b_{n}\left(t\right)^{2}}\frac{\int_{t_{n}}^{t}h^{\left(n\right)}\left(s\right)b_{n}\left(s\right)\mathrm{d}s}{\int_{t_{n}}^{1}h^{\left(n\right)}\left(s\right)b_{n}\left(s\right)\mathrm{d}s}.\label{eq:expression for Bn(t)}
\end{equation}
Hence,
\begin{equation}
\begin{aligned}\frac{\mathrm{d}B_{n}}{\mathrm{d}t}\left(t\right) & =-\frac{2M_{n}}{\left(a_{n}\left(t\right)^{2}+b_{n}\left(t\right)\right)^{2}}\frac{\mathrm{d}}{\mathrm{d}t}\left(a_{n}\left(t\right)^{2}+b_{n}\left(t\right)^{2}\right)\frac{\int_{t_{n}}^{t}h^{\left(n\right)}\left(s\right)b_{n}\left(s\right)\mathrm{d}s}{\int_{t_{n}}^{1}h^{\left(n\right)}\left(s\right)b_{n}\left(s\right)\mathrm{d}s}+\\
 & \quad+\frac{2M_{n}}{a_{n}\left(t\right)^{2}+b_{n}\left(t\right)^{2}}\frac{h^{\left(n\right)}\left(t\right)b_{n}\left(t\right)}{\int_{t_{n}}^{1}h^{\left(n\right)}\left(s\right)b_{n}\left(s\right)\mathrm{d}s}=\\
 & =\frac{2M_{n}}{a_{n}\left(t\right)^{2}+b_{n}\left(t\right)^{2}}\left[\frac{h^{\left(n\right)}\left(t\right)b_{n}\left(t\right)}{\int_{t_{n}}^{1}h^{\left(n\right)}\left(s\right)b_{n}\left(s\right)\mathrm{d}s}-\frac{\frac{\mathrm{d}}{\mathrm{d}t}\left(a_{n}\left(t\right)^{2}+b_{n}\left(t\right)^{2}\right)}{a_{n}\left(t\right)^{2}+b_{n}\left(t\right)^{2}}\frac{\int_{t_{n}}^{t}h^{\left(n\right)}\left(s\right)b_{n}\left(s\right)\mathrm{d}s}{\int_{t_{n}}^{1}h^{\left(n\right)}\left(s\right)b_{n}\left(s\right)\mathrm{d}s}\right].
\end{aligned}
\label{eq:dBndt v0}
\end{equation}
To continue, we need to work on $\frac{\frac{\mathrm{d}}{\mathrm{d}t}\left(a_{n}\left(t\right)^{2}+b_{n}\left(t\right)^{2}\right)}{a_{n}\left(t\right)^{2}+b_{n}\left(t\right)^{2}}$
a little more.
\[
\begin{aligned}\frac{\frac{\mathrm{d}}{\mathrm{d}t}\left(a_{n}\left(t\right)^{2}+b_{n}\left(t\right)^{2}\right)}{a_{n}\left(t\right)^{2}+b_{n}\left(t\right)^{2}} & =2\frac{a_{n}\left(t\right)\frac{\mathrm{d}a_{n}}{\mathrm{d}t}\left(t\right)+b_{n}\left(t\right)\frac{\mathrm{d}b_{n}}{\mathrm{d}t}\left(t\right)}{a_{n}\left(t\right)^{2}+b_{n}\left(t\right)^{2}}=2\frac{a_{n}\left(t\right)^{2}\frac{\mathrm{d}}{\mathrm{d}t}\left(\ln\left(a_{n}\left(t\right)\right)\right)+b_{n}\left(t\right)^{2}\frac{\mathrm{d}}{\mathrm{d}t}\left(\ln\left(b_{n}\left(t\right)\right)\right)}{a_{n}\left(t\right)^{2}+b_{n}\left(t\right)^{2}}.\end{aligned}
\]
By Choice \ref{choice:anbncn}, we have
\[
\frac{\mathrm{d}}{\mathrm{d}t}\left(a_{n}\left(t\right)b_{n}\left(t\right)\right)=0\iff\frac{\mathrm{d}}{\mathrm{d}t}\left(\ln\left(a_{n}\left(t\right)b_{n}\left(t\right)\right)\right)=0\iff\frac{\mathrm{d}}{\mathrm{d}t}\left(\ln\left(a_{n}\left(t\right)\right)\right)=-\frac{\mathrm{d}}{\mathrm{d}t}\left(\ln\left(b_{n}\left(t\right)\right)\right)
\]
and, consequently,
\[
\frac{\frac{\mathrm{d}}{\mathrm{d}t}\left(a_{n}\left(t\right)^{2}+b_{n}\left(t\right)^{2}\right)}{a_{n}\left(t\right)^{2}+b_{n}\left(t\right)^{2}}=2\frac{\mathrm{d}}{\mathrm{d}t}\left(\ln\left(b_{n}\left(t\right)\right)\right)\underbrace{\frac{b_{n}\left(t\right)^{2}-a_{n}\left(t\right)^{2}}{b_{n}\left(t\right)^{2}+a_{n}\left(t\right)^{2}}}_{\left(*\right)}.
\]
Notice that $\left(*\right)$ is bounded in absolute value by $1$.
As long as $C$ is big enough (let us say $C\ge\Upsilon_{1}\left(\delta\right)$),
making use of Choice \ref{choice:anbncn}, Proposition \ref{prop:time convergence}
with $\beta=\frac{1}{2}$, Choices \ref{choice:Bnanbn} and \ref{choice:Mn}
and Lemma \ref{lem:estimate sum superexponential} (which we can apply
thanks to Choice \ref{choice:min requirements on C gamma and kmax}),
in that order, we obtain
\begin{equation}
\begin{aligned}\left|\frac{\frac{\mathrm{d}}{\mathrm{d}t}\left(a_{n}\left(t\right)^{2}+b_{n}\left(t\right)^{2}\right)}{a_{n}\left(t\right)^{2}+b_{n}\left(t\right)^{2}}\right| & \lesssim\left|\frac{\mathrm{d}}{\mathrm{d}t}\left(\ln\left(b_{n}\left(t\right)\right)\right)\right|=\sum_{m=1}^{n-1}B_{m}\left(t\right)a_{m}\left(t\right)b_{m}\left(t\right)\cos\left(a_{m}\left(t\right)\left(\phi_{1}^{\left(n\right)}\left(t,0\right)-\phi_{1}^{\left(m\right)}\left(t,0\right)\right)\right)\leq\\
 & \lesssim\sum_{m=1}^{n-1}B_{m}\left(t\right)a_{m}\left(1\right)b_{m}\left(1\right)\lesssim\sum_{m=1}^{n-1}B_{m}\left(1\right)a_{m}\left(1\right)b_{m}\left(1\right)=\sum_{m=1}^{n-1}YC^{\delta\left(\frac{1}{1-\gamma}\right)^{m}}\lesssim_{\delta}\\
 & \lesssim_{\delta}YC^{\delta\left(\frac{1}{1-\gamma}\right)^{n-1}}.
\end{aligned}
\label{eq:bound derivative of lnbnt}
\end{equation}
Employing \eqref{eq:bound derivative of lnbnt} in equation \eqref{eq:dBndt v0}
provides
\[
\left|\frac{\mathrm{d}B_{n}}{\mathrm{d}t}\left(t\right)\right|\lesssim_{\delta}\frac{2M_{n}}{a_{n}\left(t\right)^{2}+b_{n}\left(t\right)^{2}}\left[YC^{\delta\left(\frac{1}{1-\gamma}\right)^{n-1}}\underbrace{\frac{\int_{t_{n}}^{t}h^{\left(n\right)}\left(s\right)b_{n}\left(s\right)\mathrm{d}s}{\int_{t_{n}}^{1}h^{\left(n\right)}\left(s\right)b_{n}\left(s\right)\mathrm{d}s}}_{\leq1}+\frac{h^{\left(n\right)}\left(t\right)b_{n}\left(t\right)}{\int_{t_{n}}^{1}h^{\left(n\right)}\left(s\right)b_{n}\left(s\right)\mathrm{d}s}\right].
\]
As $n\ge2$, if we take $C$ big enough (let us say $C\ge\Upsilon_{2}\left(\delta,\mu\right)$),
Lemma \ref{lem:estimate of integral hn bn}, equation \eqref{eq:h_eps near one},
Corollary \ref{cor:bound for an(t), bn(t)} and Choice \ref{choice:ideal kn}
allow us to write
\begin{equation}
\begin{aligned}\left|\frac{\mathrm{d}B_{n}}{\mathrm{d}t}\left(t\right)\right| & \lesssim_{\delta,\mu}\frac{2M_{n}}{a_{n}\left(t\right)^{2}+b_{n}\left(t\right)^{2}}\left[YC^{\delta\left(\frac{1}{1-\gamma}\right)^{n-1}}+\frac{C^{\left(1+k_{\max}+\mu\right)\left(\frac{1}{1-\gamma}\right)^{n}}}{\frac{1}{Y}C^{\left(1+k_{\max}-\mu-\delta\left(1-\gamma\right)\right)\left(\frac{1}{1-\gamma}\right)^{n}}}\right]=\\
 & \lesssim_{\delta,\mu}\frac{2M_{n}}{a_{n}\left(t\right)^{2}+b_{n}\left(t\right)^{2}}YC^{\left(2\mu+\delta\left(1-\gamma\right)\right)\left(\frac{1}{1-\gamma}\right)^{n}}.
\end{aligned}
\label{eq:dBndt v1}
\end{equation}
Going back to equation \eqref{eq:Leibniz rule Bn(t)bn(t)2}, while
bounding the term $\frac{\mathrm{d}}{\mathrm{d}t}\left(\ln\left(b_{n}\left(t\right)\right)\right)$
as we did in equation \eqref{eq:bound derivative of lnbnt} and taking
equation \eqref{eq:expression for Bn(t)} into account, we arrive
to
\[
\begin{aligned}\left|\frac{\mathrm{d}}{\mathrm{d}t}\left(B_{n}\left(t\right)b_{n}\left(t\right)^{2}\right)\right| & \lesssim_{\delta,\mu}\left[2YM_{n}C^{\left(2\mu+\delta\left(1-\gamma\right)\right)\left(\frac{1}{1-\gamma}\right)^{n}}+YC^{\delta\left(\frac{1}{1-\gamma}\right)^{n-1}}2M_{n}\underbrace{\frac{\int_{t_{n}}^{t}h^{\left(n\right)}\left(s\right)b_{n}\left(s\right)\mathrm{d}s}{\int_{t_{n}}^{1}h^{\left(n\right)}\left(s\right)b_{n}\left(s\right)\mathrm{d}s}}_{\leq1}\right]\cdot\\
 & \quad\cdot\underbrace{\frac{b_{n}\left(t\right)^{2}}{a_{n}\left(t\right)^{2}+b_{n}\left(t\right)^{2}}}_{\leq1}\leq\\
 & \lesssim_{\delta,\mu}YM_{n}C^{\left(2\mu+\delta\left(1-\gamma\right)\right)\left(\frac{1}{1-\gamma}\right)^{n}}.
\end{aligned}
\]
By Choice \ref{choice:Mn}, we deduce that
\begin{equation}
\left|\frac{\mathrm{d}}{\mathrm{d}t}\left(B_{n}\left(t\right)b_{n}\left(t\right)^{2}\right)\right|\lesssim_{\delta,\mu}Y^{2}C^{\left(\delta+2\mu+\delta\left(1-\gamma\right)\right)\left(\frac{1}{1-\gamma}\right)^{n}}\leq Y^{2}C^{\left(2\delta+2\mu\right)\left(\frac{1}{1-\gamma}\right)^{n}},\label{eq:bound dBnbn2dt}
\end{equation}
where we have employed the fact that $1-\gamma\leq1$.

Next, we shall bound $\frac{\mathrm{d}}{\mathrm{d}t}\left(B_{n}\left(t\right)\left(a_{n}\left(t\right)^{2}+b_{n}\left(t\right)^{2}\right)\right)$.
By Choice \ref{choice:amplitude density}, equation \eqref{eq:relation Mn and zn}
and Choice \ref{choice:time picture}, we have
\[
\frac{\mathrm{d}}{\mathrm{d}t}\left(B_{n}\left(t\right)\left(a_{n}\left(t\right)^{2}+b_{n}\left(t\right)^{2}\right)\right)=2M_{n}\frac{h^{\left(n\right)}\left(t\right)b_{n}\left(t\right)}{\int_{t_{n}}^{1}h^{\left(n\right)}\left(s\right)b_{n}\left(s\right)\mathrm{d}s}.
\]
Choices \ref{choice:Mn} and \ref{choice:ideal kn}, Corollary \ref{cor:bound for an(t), bn(t)},
Lemma \ref{lem:estimate of integral hn bn} and equation \eqref{eq:h_eps near one}
provide
\begin{equation}
\frac{\mathrm{d}}{\mathrm{d}t}\left(B_{n}\left(t\right)\left(a_{n}\left(t\right)^{2}+b_{n}\left(t\right)^{2}\right)\right)\lesssim_{\mu}YC^{\delta\left(\frac{1}{1-\gamma}\right)^{n}}\frac{C^{\left(1+k_{\max}+\mu\right)\left(\frac{1}{1-\gamma}\right)^{n}}}{\frac{1}{Y}C^{\left(1+k_{\max}-\mu-\delta\left(1-\gamma\right)\right)\left(\frac{1}{1-\gamma}\right)^{n}}}\lesssim Y^{2}C^{\left(2\delta+2\mu\right)\left(\frac{1}{1-\gamma}\right)^{n}},\label{eq:bound dBnan2bn2dt}
\end{equation}
where we have also made use of the bound $1-\gamma\leq1$.

The triangular inequality, along with the equality
\[
\frac{\mathrm{d}}{\mathrm{d}t}\left[B_{n}\left(t\right)\left(a_{n}\left(t\right)^{2}+b_{n}\left(t\right)^{2}\right)\right]=\frac{\mathrm{d}}{\mathrm{d}t}\left[B_{n}\left(t\right)a_{n}\left(t\right)^{2}\right]+\frac{\mathrm{d}}{\mathrm{d}t}\left[B_{n}\left(t\right)b_{n}\left(t\right)^{2}\right]
\]
and equations \eqref{eq:bound dBnbn2dt} and \eqref{eq:bound dBnan2bn2dt}
imply the result about $\frac{\mathrm{d}}{\mathrm{d}t}\left[B_{n}\left(t\right)a_{n}\left(t\right)^{2}\right]$.
\end{proof}
\begin{prop}
\label{prop:bound vorticity time factor}Let $n\in\mathbb{N}$ with
$n\ge2$, $\mu>0$ and $\alpha\in\left(0,1\right)$. As long as $C$
is big enough (let us say $C\ge\Upsilon\left(\delta,\mu\right)$),
the time factor
\[
\widetilde{\tau^{\left(n\right)}}^{n}\left(t,x\right)\coloneqq\frac{\partial\widetilde{\omega^{\left(n\right)}}^{n}}{\partial t}\left(t,x\right)-\left(\begin{matrix}0 & 1\end{matrix}\right)\cdot\left[\widetilde{\nabla}^{n}\widetilde{\rho^{\left(n\right)}}^{n}\left(t,x\right)\right]
\]
is bounded by
\[
\left|\left|\tau^{\left(n\right)}\left(t,\cdot\right)\right|\right|_{C^{\alpha}\left(\mathbb{R}^{2}\right)}\lesssim_{\delta,\varphi,\mu}Y^{2}C^{\left(\alpha\left(1+k_{\max}\right)-\Lambda+3\mu+2\delta\right)\left(\frac{1}{1-\gamma}\right)^{n}}.
\]
This expression remains decreasing in $n\in\mathbb{N}$ as long as
\[
\alpha<\frac{\Lambda-3\mu-2\delta}{1+k_{\max}}.
\]
\end{prop}
\begin{proof}
On the one hand, to obtain an expression for $\frac{\partial\widetilde{\omega^{\left(n\right)}}^{n}}{\partial t}\left(t,x\right)$,
we resort to Proposition \ref{prop:computations vorticity}. On the
other hand, for an expression for $\left(\begin{matrix}0 & 1\end{matrix}\right)\cdot\left[\widetilde{\nabla}^{n}\widetilde{\rho^{\left(n\right)}}^{n}\left(t,x\right)\right]$,
we apply Choice \ref{choice:density}. From these two previous results,
it is easy to see that the only zeroth order term in $\lambda_{n}$
of $\frac{\partial\widetilde{\omega^{\left(n\right)}}^{n}}{\partial t}\left(t,x\right)$
exactly cancels the only zeroth order term in $\lambda_{n}$ of $\left(\begin{matrix}0 & 1\end{matrix}\right)\cdot\left[\widetilde{\nabla}^{n}\widetilde{\rho^{\left(n\right)}}^{n}\left(t,x\right)\right]$.
In this way, $\widetilde{T^{\left(n\right)}}^{n}\left(t,x\right)$
has the following form:
\[
\begin{aligned}\widetilde{\tau^{\left(n\right)}}^{n}\left(t,x\right) & =\sum\pm\lambda_{n}^{1\text{ or }2}\left(\begin{matrix}\frac{\mathrm{d}}{\mathrm{d}t}\left(B_{n}\left(t\right)\left(a_{n}\left(t\right)^{2}+b_{n}\left(t\right)^{2}\right)\right)\\
\text{or}\\
\frac{\mathrm{d}}{\mathrm{d}t}\left(B_{n}\left(t\right)a_{n}\left(t\right)^{2}\right)\\
\text{or}\\
\frac{\mathrm{d}}{\mathrm{d}t}\left(B_{n}\left(t\right)b_{n}\left(t\right)^{2}\right)
\end{matrix}\right)\cdot\left(\text{derivative of }\varphi\right)\left(\lambda_{n}x_{1}\right)\cdot\\
 & \quad\cdot\left(\text{derivative of }\varphi\right)\left(\lambda_{n}x_{2}\right)\cdot\left(\begin{matrix}\sin\\
\text{or}\\
\text{cos}
\end{matrix}\right)\left(x_{1}\right)\cdot\left(\begin{matrix}\sin\\
\text{or}\\
\text{cos}
\end{matrix}\right)\left(x_{2}\right).
\end{aligned}
\]

We can bound each term separately.
\begin{enumerate}
\item As $\lambda_{n}\le1$ by Choice \ref{choice:psin}, we have
\[
\lambda_{n}^{1\text{ or }2}\le\lambda_{n}.
\]
\item As long as $C$ is taken big enough (let us say $C\ge\Upsilon_{1}\left(\delta,\mu\right)$),
since $n\ge2$, Lemma \ref{lem:time derivatives amplitudes vorticity}
guarantees that
\[
\left|\left(\begin{matrix}\frac{\mathrm{d}}{\mathrm{d}t}\left(B_{n}\left(t\right)\left(a_{n}\left(t\right)^{2}+b_{n}\left(t\right)^{2}\right)\right)\\
\text{or}\\
\frac{\mathrm{d}}{\mathrm{d}t}\left(B_{n}\left(t\right)a_{n}\left(t\right)^{2}\right)\\
\text{or}\\
\frac{\mathrm{d}}{\mathrm{d}t}\left(B_{n}\left(t\right)b_{n}\left(t\right)^{2}\right)
\end{matrix}\right)\right|\lesssim_{\delta,\mu}Y^{2}C^{\left(2\delta+2\mu\right)\left(\frac{1}{1-\gamma}\right)^{n}}.
\]
\item Because $\left|\left|\cdot\right|\right|_{L^{\infty}\left(\mathbb{R}^ {}\right)}$
is invariant under diffeomorphisms and because of equation \eqref{eq:property Calpha composition},
we have
\[
\left|\left|\left(\text{derivative of }\varphi\right)\left(\lambda_{n}x_{1}\right)\right|\right|_{L^{\infty}\left(\mathbb{R}\right)}\lesssim_{\varphi}1,\quad\left|\left|\left(\text{derivative of }\varphi\right)\left(\lambda_{n}x_{1}\right)\right|\right|_{\dot{C}^{\alpha}\left(\mathbb{R}\right)}\lesssim_{\varphi}\lambda_{n}^{\alpha}\leq1.
\]
Analogous bounds can be obtained for $\left(\text{derivative of }\varphi\right)\left(\lambda_{n}x_{2}\right)$.
\item Furthermore,
\[
\left|\left|\left(\begin{matrix}\sin\\
\text{or}\\
\text{cos}
\end{matrix}\right)\left(\cdot\right)\right|\right|_{L^{\infty}\left(\mathbb{R}\right)}\leq1,\quad\left|\left|\left(\begin{matrix}\sin\\
\text{or}\\
\text{cos}
\end{matrix}\right)\left(\cdot\right)\right|\right|_{\dot{C}^{\alpha}\left(\mathbb{R}\right)}\leq1.
\]
\end{enumerate}
Thereby, we deduce that
\[
\begin{aligned}\left|\left|\widetilde{\tau^{\left(n\right)}}^{n}\left(t,\cdot\right)\right|\right|_{L^{\infty}\left(\mathbb{R}^{2}\right)} & \lesssim_{\delta,\mu,\varphi}\lambda_{n}Y^{2}C^{\left(2\delta+2\mu\right)\left(\frac{1}{1-\gamma}\right)^{n}},\\
\left|\left|\widetilde{\tau^{\left(n\right)}}^{n}\left(t,\cdot\right)\right|\right|_{\dot{C}^{\alpha}\left(\mathbb{R}^{2}\right)} & \lesssim_{\delta,\mu,\varphi}\lambda_{n}Y^{2}C^{\left(2\delta+2\mu\right)\left(\frac{1}{1-\gamma}\right)^{n}}.
\end{aligned}
\]
Choice \ref{choice:lambdan} lets us write
\begin{equation}
\begin{aligned}\left|\left|\widetilde{\tau^{\left(n\right)}}^{n}\left(t,\cdot\right)\right|\right|_{L^{\infty}\left(\mathbb{R}^{2}\right)} & \lesssim_{\delta,\mu,\varphi}Y^{2}C^{\left(-\Lambda+2\delta+2\mu\right)\left(\frac{1}{1-\gamma}\right)^{n}},\\
\left|\left|\widetilde{\tau^{\left(n\right)}}^{n}\left(t,\cdot\right)\right|\right|_{\dot{C}^{\alpha}\left(\mathbb{R}^{2}\right)} & \lesssim_{\delta,\mu,\varphi}Y^{2}C^{\left(-\Lambda+2\delta+2\mu\right)\left(\frac{1}{1-\gamma}\right)^{n}}.
\end{aligned}
\label{eq:T cv}
\end{equation}
Since the $\left|\left|\cdot\right|\right|_{L^{\infty}\left(\mathbb{R}^{2}\right)}$
norm is invariant under diffeomorphisms, the result above is also
valid for $\tau^{\left(n\right)}\left(t,\cdot\right)$. On the other
hand, equations \eqref{eq:property Calpha composition} and \eqref{eq:jacobian inverse}
lead us to
\[
\left|\left|\tau^{\left(n\right)}\left(t,\cdot\right)\right|\right|_{\dot{C}^{\alpha}\left(\mathbb{R}^{2}\right)}=\left|\left|\widetilde{\tau^{\left(n\right)}}^{n}\left(t,\left(\phi^{\left(n\right)}\right)^{-1}\left(t,\cdot\right)\right)\right|\right|_{\dot{C}^{\alpha}\left(\mathbb{R}^{2}\right)}\lesssim\max\left\{ a_{n}\left(t\right)^{\alpha},b_{n}\left(t\right)^{\alpha}\right\} \left|\left|\widetilde{\tau^{\left(n\right)}}^{n}\left(t,\cdot\right)\right|\right|_{\dot{C}^{\alpha}\left(\mathbb{R}^{2}\right)}.
\]
If $C$ is large enough (let us say $C\ge\Upsilon_{2}\left(\delta,\mu\right)$),
Corollary \ref{cor:bound for an(t), bn(t)}, Choice \ref{choice:ideal kn}
and equation \eqref{eq:T cv} provide
\[
\begin{aligned}\left|\left|\tau^{\left(n\right)}\left(t,\cdot\right)\right|\right|_{\dot{C}^{\alpha}\left(\mathbb{R}^{2}\right)} & \lesssim_{\delta,\varphi,\mu}C^{\alpha\left(1+k_{\max}+\mu\right)\left(\frac{1}{1-\gamma}\right)^{n}}Y^{2}C^{\left(-\Lambda+2\delta+2\mu\right)\left(\frac{1}{1-\gamma}\right)^{n}}=\\
 & \lesssim_{\delta,\varphi,\mu}Y^{2}C^{\left(\alpha\left(1+k_{\max}\right)-\Lambda+\left(2+\alpha\right)\mu+2\delta\right)\left(\frac{1}{1-\gamma}\right)^{n}}.
\end{aligned}
\]
Bounding $\alpha<1$ when it is multiplied by $\mu$, we obtain
\[
\left|\left|\tau^{\left(n\right)}\left(t,\cdot\right)\right|\right|_{\dot{C}^{\alpha}\left(\mathbb{R}^{2}\right)}\lesssim_{\delta,\varphi,\mu}Y^{2}C^{\left(\alpha\left(1+k_{\max}\right)-\Lambda+3\mu+2\delta\right)\left(\frac{1}{1-\gamma}\right)^{n}}.
\]
This result, along with equation \eqref{eq:T cv}, implies the statement.
\end{proof}

\subsection{Pure quadratic term}
\begin{prop}
\label{prop:bound vorticity pure quadratic term}Let $n\in\mathbb{N}$
with $n\ge2$, $\mu>0$ and $\alpha\in\left(0,1\right)$. Provided
that $C$ is chosen large enough (let us say $C\ge\Upsilon\left(\delta,\mu\right)$),
the pure quadratic term
\[
\widetilde{Q_{\omega}^{\left(n\right)}}^{n}\left(t,x\right)\coloneqq\widetilde{u^{\left(n\right)}}^{n}\left(t,x\right)\cdot\widetilde{\nabla}^{n}\widetilde{\omega^{\left(n\right)}}^{n}\left(t,x\right)
\]
satisfies the bounds
\[
\begin{aligned}\left|\left|Q_{\omega}^{\left(n\right)}\left(t,\cdot\right)\right|\right|_{C^{\alpha}\left(\mathbb{R}^{2}\right)} & \lesssim_{\varphi}Y^{2}C^{\left(\alpha\left(1+k_{\max}\right)-\Lambda+2\delta+5\mu\right)\left(\frac{1}{1-\gamma}\right)^{n}}.\end{aligned}
\]
This expression remains decreasing in $n\in\mathbb{N}$ as long as
\[
\alpha<\frac{\Lambda-2\delta-5\mu}{1+k_{\max}}.
\]
\end{prop}
\begin{proof}
In Proposition \ref{prop:bounds pure quadratic term}, we already
found that
\[
\begin{aligned}\left|\left|Q_{\omega}^{\left(n\right)}\left(t,\cdot\right)\right|\right|_{L^{\infty}\left(\mathbb{R}^{2}\right)} & \lesssim_{\varphi}\lambda_{n}B_{n}\left(t\right)^{2}a_{n}\left(t\right)b_{n}\left(t\right)\max\left\{ a_{n}\left(t\right)^{2},b_{n}\left(t\right)^{2}\right\} ,\\
\left|\left|Q_{\omega}^{\left(n\right)}\left(t,\cdot\right)\right|\right|_{\dot{C}^{\alpha}\left(\mathbb{R}^{2}\right)} & \lesssim_{\varphi}\lambda_{n}B_{n}\left(t\right)^{2}a_{n}\left(t\right)b_{n}\left(t\right)\max\left\{ a_{n}\left(t\right)^{2+\alpha},b_{n}\left(t\right)^{2+\alpha}\right\} .
\end{aligned}
\]
Notice that we can write this as
\[
\begin{aligned}\left|\left|Q_{\omega}^{\left(n\right)}\left(t,\cdot\right)\right|\right|_{L^{\infty}\left(\mathbb{R}^{2}\right)} & \lesssim_{\varphi}\lambda_{n}B_{n}\left(t\right)^{2}a_{n}\left(t\right)b_{n}\left(t\right)\max\left\{ a_{n}\left(t\right)^{2},b_{n}\left(t\right)^{2}\right\} ,\\
\left|\left|Q_{\omega}^{\left(n\right)}\left(t,\cdot\right)\right|\right|_{\dot{C}^{\alpha}\left(\mathbb{R}^{2}\right)} & \lesssim_{\varphi}\lambda_{n}B_{n}\left(t\right)^{2}a_{n}\left(t\right)b_{n}\left(t\right)\max\left\{ a_{n}\left(t\right)^{2},b_{n}\left(t\right)^{2}\right\} \max\left\{ a_{n}\left(t\right)^{\alpha},b_{n}\left(t\right)^{\alpha}\right\} .
\end{aligned}
\]
As long as $C$ is big enough (let us say $C\ge\Upsilon_{1}\left(\delta,\mu\right)$),
as $n\ge2$, by Corollary \ref{cor:bound for an(t), bn(t)}, we can
bound
\begin{equation}
\begin{aligned}\left|\left|Q_{\omega}^{\left(n\right)}\left(t,\cdot\right)\right|\right|_{L^{\infty}\left(\mathbb{R}^{2}\right)} & \lesssim_{\varphi}\lambda_{n}B_{n}\left(t\right)^{2}a_{n}\left(t\right)b_{n}\left(t\right)b_{n}\left(t\right)^{2}C^{4\mu\left(\frac{1}{1-\gamma}\right)^{n}},\\
\left|\left|Q_{\omega}^{\left(n\right)}\left(t,\cdot\right)\right|\right|_{\dot{C}^{\alpha}\left(\mathbb{R}^{2}\right)} & \lesssim_{\varphi}\lambda_{n}B_{n}\left(t\right)^{2}a_{n}\left(t\right)b_{n}\left(t\right)b_{n}\left(t\right)^{2}C^{4\mu\left(\frac{1}{1-\gamma}\right)^{n}}\max\left\{ a_{n}\left(t\right)^{\alpha},b_{n}\left(t\right)^{\alpha}\right\} .
\end{aligned}
\label{eq:pure quadratic v1}
\end{equation}

Now, employing Choice \ref{choice:amplitude density}, equation \eqref{eq:relation Mn and zn}
and Choice \ref{choice:time picture}, we infer that
\begin{equation}
\begin{aligned}B_{n}\left(t\right)^{2}a_{n}\left(t\right)b_{n}\left(t\right)b_{n}\left(t\right)^{2} & =\left(2M_{n}\frac{1}{a_{n}\left(t\right)^{2}+b_{n}\left(t\right)^{2}}\underbrace{\frac{\int_{t_{n}}^{t}h^{\left(n\right)}\left(s\right)b_{n}\left(s\right)\mathrm{d}s}{\int_{t_{n}}^{1}h^{\left(n\right)}\left(s\right)b_{n}\left(s\right)\mathrm{d}s}}_{\leq1}\right)^{2}a_{n}\left(t\right)b_{n}\left(t\right)b_{n}\left(t\right)^{2}\leq\\
 & \leq4M_{n}^{2}\underbrace{\frac{b_{n}\left(t\right)^{2}}{a_{n}\left(t\right)^{2}+b_{n}\left(t\right)^{2}}}_{\leq1}\frac{a_{n}\left(t\right)b_{n}\left(t\right)}{a_{n}\left(t\right)^{2}+b_{n}\left(t\right)^{2}}.
\end{aligned}
\label{eq:bound Bn2anbnbn2 v0}
\end{equation}
Applying the bound
\[
0\le\left(a-b\right)^{2}=a^{2}+b^{2}-2ab\iff ab\le\frac{1}{2}\left(a^{2}+b^{2}\right),
\]
we deduce that
\[
\frac{a_{n}\left(t\right)b_{n}\left(t\right)}{a_{n}\left(t\right)^{2}+b_{n}\left(t\right)^{2}}\le\frac{1}{2}
\]
and, consequently, by Choice \ref{choice:Mn}, \eqref{eq:bound Bn2anbnbn2 v0}
becomes
\begin{equation}
B_{n}\left(t\right)^{2}a_{n}\left(t\right)b_{n}\left(t\right)\le2M_{n}^{2}=2Y^{2}C^{2\delta\left(\frac{1}{1-\gamma}\right)^{n}}.\label{eq:bound Bn2anbnbn2 v1}
\end{equation}

Making use of \eqref{eq:bound Bn2anbnbn2 v1} in \eqref{eq:pure quadratic v1},
along with Corollary \ref{cor:bound for an(t), bn(t)} and Choice
\ref{choice:lambdan} leads to
\[
\begin{aligned}\left|\left|Q_{\omega}^{\left(n\right)}\left(t,\cdot\right)\right|\right|_{L^{\infty}\left(\mathbb{R}^{2}\right)} & \lesssim_{\varphi}C^{-\Lambda\left(\frac{1}{1-\gamma}\right)^{n}}Y^{2}C^{2\delta\left(\frac{1}{1-\gamma}\right)^{n}}C^{4\mu\left(\frac{1}{1-\gamma}\right)^{n}},\\
\left|\left|Q_{\omega}^{\left(n\right)}\left(t,\cdot\right)\right|\right|_{\dot{C}^{\alpha}\left(\mathbb{R}^{2}\right)} & \lesssim_{\varphi}C^{-\Lambda\left(\frac{1}{1-\gamma}\right)^{n}}Y^{2}C^{2\delta\left(\frac{1}{1-\gamma}\right)^{n}}C^{4\mu\left(\frac{1}{1-\gamma}\right)^{n}}C^{\alpha\left(1+\overline{k}_{n}\left(t\right)+\mu\right)\left(\frac{1}{1-\gamma}\right)^{n}}.
\end{aligned}
\]
Bounding $\overline{k}_{n}\left(t\right)\leq k_{\max}$ (which we
can do by Choice \ref{choice:ideal kn}), we obtain
\[
\begin{aligned}\left|\left|Q_{\omega}^{\left(n\right)}\left(t,\cdot\right)\right|\right|_{L^{\infty}\left(\mathbb{R}^{2}\right)} & \lesssim_{\varphi}Y^{2}C^{\left(-\Lambda+2\delta+4\mu\right)\left(\frac{1}{1-\gamma}\right)^{n}},\\
\left|\left|Q_{\omega}^{\left(n\right)}\left(t,\cdot\right)\right|\right|_{\dot{C}^{\alpha}\left(\mathbb{R}^{2}\right)} & \lesssim_{\varphi}Y^{2}C^{\left(\alpha\left(1+k_{\max}\right)-\Lambda+2\delta+\left(4+\alpha\right)\mu\right)\left(\frac{1}{1-\gamma}\right)^{n}}.
\end{aligned}
\]
Using that $\alpha<1$ in the terms where $\alpha$ is multiplied
by $\mu$, we may write
\[
\begin{aligned}\left|\left|Q_{\omega}^{\left(n\right)}\left(t,\cdot\right)\right|\right|_{L^{\infty}\left(\mathbb{R}^{2}\right)} & \lesssim_{\varphi}Y^{2}C^{\left(-\Lambda+2\delta+4\mu\right)\left(\frac{1}{1-\gamma}\right)^{n}},\\
\left|\left|Q_{\omega}^{\left(n\right)}\left(t,\cdot\right)\right|\right|_{\dot{C}^{\alpha}\left(\mathbb{R}^{2}\right)} & \lesssim_{\varphi}Y^{2}C^{\left(\alpha\left(1+k_{\max}\right)-\Lambda+2\delta+5\mu\right)\left(\frac{1}{1-\gamma}\right)^{n}}.
\end{aligned}
\]
From these two bounds, it is straightforward to derive the result
of the statement.
\end{proof}

\subsection{New transport of old vorticity}
\begin{prop}
\label{prop:bound vorticity new transport of old vorticity}Let $n\in\mathbb{N}$
with $n\ge2$, $\mu>0$, $\alpha\in\left(0,1\right)$ and 
\[
\widetilde{\mathrm{O}^{\left(n\right)}\left(t,x\right)}^{n}\coloneqq\widetilde{u^{\left(n\right)}}^{n}\left(t,x\right)\cdot\widetilde{\nabla}^{n}\widetilde{\Omega^{\left(n-1\right)}}^{n}\left(t,x\right).
\]
Then, provided that $C$ is big enough (let us say $C\ge\Upsilon\left(\delta,\mu\right)$),
\[
\left|\left|\mathrm{O}^{\left(n\right)}\left(t,\cdot\right)\right|\right|_{C^{\alpha}\left(\mathbb{R}^{2}\right)}\lesssim_{\varphi,\delta}Y^{2}C^{\left[-\left(1-\alpha\right)+4\mu+2\left(1-\gamma\right)+2\delta\right]\left(\frac{1}{1-\gamma}\right)^{n}}.
\]
\end{prop}
\begin{rem}
Proposition \ref{prop:bound vorticity new transport of old vorticity}
basically tells us that, as long as we take $\mu$ and $\delta$ small
and $\gamma$ close to one, $\left|\left|\mathrm{O}^{\left(n\right)}\left(t,\cdot\right)\right|\right|_{C^{\alpha}\left(\mathbb{R}^{2}\right)}$
remains bounded in $n\in\mathbb{N}$ $\forall\alpha\in\left(0,1\right)$.
In other words, this term is not a limiting factor for the regularity
of the force.
\end{rem}
\begin{proof}
First of all, we need some estimates for $\Pi^{\left(n\right)}\left(t,x\right)$,
where
\[
\begin{aligned}\widetilde{\Pi^{\left(n\right)}}^{n}\left(t,x\right) & \coloneqq\widetilde{\nabla}^{n}\widetilde{\Omega^{\left(n-1\right)}}^{n}\left(t,x\right)=\widetilde{\nabla}^{n}\left(\sum_{m=1}^{n-1}\widetilde{\omega^{\left(m\right)}}^{n}\left(t,x\right)\right).\end{aligned}
\]
By equation \eqref{eq:gradient tilde phi}, we have
\[
\begin{aligned}\widetilde{\Pi^{\left(n\right)}}^{n}\left(t,x\right) & =\nabla\left(\sum_{m=1}^{n-1}\widetilde{\omega^{\left(m\right)}}^{n}\left(t,x\right)\right)\cdot\left(\begin{matrix}a_{n}\left(t\right) & 0\\
0 & b_{n}\left(t\right)
\end{matrix}\right)=\\
 & =\nabla\left(\sum_{m=1}^{n-1}\omega^{\left(m\right)}\left(t,\phi^{\left(n\right)}\left(t,x\right)\right)\right)\cdot\left(\begin{matrix}a_{n}\left(t\right) & 0\\
0 & b_{n}\left(t\right)
\end{matrix}\right)=\\
 & =\nabla\left(\sum_{m=1}^{n-1}\widetilde{\omega^{\left(m\right)}}^{m}\left(t,\left(\phi^{\left(m\right)}\right)^{-1}\left(t,\phi^{\left(n\right)}\left(t,x\right)\right)\right)\right)\cdot\left(\begin{matrix}a_{n}\left(t\right) & 0\\
0 & b_{n}\left(t\right)
\end{matrix}\right).
\end{aligned}
\]
Next, equation \eqref{eq:phin inverse} and Choice \ref{choice:phin}
allow us to write
\begin{equation}
\begin{aligned}\widetilde{\Pi^{\left(n\right)}}^{n}\left(t,x\right) & =\nabla\left(\sum_{m=1}^{n-1}\widetilde{\omega^{\left(m\right)}}^{m}\left(t,\left(\begin{matrix}a_{m}\left(t\right)\left(\phi_{1}^{\left(n\right)}\left(t,x\right)-\phi_{1}^{\left(m\right)}\left(t,0\right)\right)\\
b_{m}\left(t\right)\left(\phi_{2}^{\left(n\right)}\left(t,x\right)-\phi_{2}^{\left(m\right)}\left(t,0\right)\right)
\end{matrix}\right)\right)\right)\cdot\left(\begin{matrix}a_{n}\left(t\right) & 0\\
0 & b_{n}\left(t\right)
\end{matrix}\right)=\\
 & =\sum_{m=1}^{n-1}\nabla\widetilde{\omega^{\left(m\right)}}^{m}\left(t,\left(\begin{matrix}a_{m}\left(t\right)\left(\phi_{1}^{\left(n\right)}\left(t,x\right)-\phi_{1}^{\left(m\right)}\left(t,0\right)\right)\\
b_{m}\left(t\right)\left(\phi_{2}^{\left(n\right)}\left(t,x\right)-\phi_{2}^{\left(m\right)}\left(t,0\right)\right)
\end{matrix}\right)\right)\cdot\left(\begin{matrix}\frac{a_{m}\left(t\right)}{a_{n}\left(t\right)} & 0\\
0 & \frac{b_{m}\left(t\right)}{b_{n}\left(t\right)}
\end{matrix}\right)\cdot\left(\begin{matrix}a_{n}\left(t\right) & 0\\
0 & b_{n}\left(t\right)
\end{matrix}\right)=\\
 & =\sum_{m=1}^{n-1}\nabla\widetilde{\omega^{\left(m\right)}}^{m}\left(t,\left(\begin{matrix}a_{m}\left(t\right)\left(\phi_{1}^{\left(n\right)}\left(t,x\right)-\phi_{1}^{\left(m\right)}\left(t,0\right)\right)\\
b_{m}\left(t\right)\left(\phi_{2}^{\left(n\right)}\left(t,x\right)-\phi_{2}^{\left(m\right)}\left(t,0\right)\right)
\end{matrix}\right)\right)\cdot\left(\begin{matrix}a_{m}\left(t\right) & 0\\
0 & b_{m}\left(t\right)
\end{matrix}\right)=\\
 & =\sum_{m=1}^{n-1}\widetilde{\nabla}^{m}\widetilde{\omega^{\left(m\right)}}^{m}\left(t,\left(\begin{matrix}a_{m}\left(t\right)\left(\phi_{1}^{\left(n\right)}\left(t,x\right)-\phi_{1}^{\left(m\right)}\left(t,0\right)\right)\\
b_{m}\left(t\right)\left(\phi_{2}^{\left(n\right)}\left(t,x\right)-\phi_{2}^{\left(m\right)}\left(t,0\right)\right)
\end{matrix}\right)\right),
\end{aligned}
\label{eq:form PIn}
\end{equation}
where we have used equation \eqref{eq:gradient tilde phi} in the
last step.

In this way, as the $\left|\left|\cdot\right|\right|_{L^{\infty}\left(\mathbb{R}^{2}\right)}$
norm is invariant under diffeomorphisms, it is immediate that
\[
\left|\left|\Pi_{i}^{\left(n\right)}\left(t,\cdot\right)\right|\right|_{L^{\infty}\left(\mathbb{R}^{2}\right)}\leq\sum_{m=1}^{n-1}\left|\left|\left(\widetilde{\nabla}^{m}\widetilde{\omega^{\left(m\right)}}^{m}\left(t,\cdot\right)\right)_{i}\right|\right|_{L^{\infty}\left(\mathbb{R}^{2}\right)}\quad\forall i\in\left\{ 1,2\right\} ,
\]
where the subscript $i$ denotes the $i$-th component. Then, Proposition
\ref{prop:form gradient omega} provides
\begin{equation}
\begin{aligned}\left|\left|\Pi_{1}^{\left(n\right)}\left(t,\cdot\right)\right|\right|_{L^{\infty}\left(\mathbb{R}^{2}\right)} & \lesssim_{\varphi}\sum_{m=1}^{n-1}B_{m}\left(t\right)a_{m}\left(t\right)\max\left\{ a_{m}\left(t\right)^{2},b_{m}\left(t\right)^{2}\right\} ,\\
\left|\left|\Pi_{2}^{\left(n\right)}\left(t,\cdot\right)\right|\right|_{L^{\infty}\left(\mathbb{R}^{2}\right)} & \lesssim_{\varphi}\sum_{m=1}^{n-1}B_{m}\left(t\right)b_{m}\left(t\right)\max\left\{ a_{m}\left(t\right)^{2},b_{m}\left(t\right)^{2}\right\} .
\end{aligned}
\label{eq:bounds pi Linfty v0}
\end{equation}
The $\left|\left|\cdot\right|\right|_{\dot{C}^{\alpha}\left(\mathbb{R}^{2}\right)}$
seminorm is a little more difficult to compute. As we have done in
multiple proofs by now, we will resort to marginal Hölder seminorms
(see equation \eqref{eq:marginal seminorms}). Since $\phi_{1}^{\left(n\right)}\left(t,x\right)$
only depends on $x_{1}$ and $\phi_{2}^{\left(n\right)}\left(t,x\right)$
only depends on $x_{2}$, by equations \eqref{eq:form PIn} and \eqref{eq:property Calpha composition},
we have
\[
\begin{aligned}\left|\left|\widetilde{\Pi_{1}^{\left(n\right)}}^{n}\left(t,\cdot\right)\right|\right|_{\dot{C}_{1}^{\alpha}\left(\mathbb{R}\right)} & \leq\sum_{m=1}^{n-1}\left|\left|a_{m}\left(t\right)\left(\phi_{1}^{\left(n\right)}\left(t,x\right)-\phi_{1}^{\left(m\right)}\left(t,0\right)\right)\right|\right|_{\dot{C}^{1}\left(\mathbb{R}\right)}^{\alpha}\left|\left|\left(\widetilde{\nabla}^{m}\widetilde{\omega^{\left(m\right)}}^{m}\left(t,\cdot\right)\right)_{1}\right|\right|_{\dot{C}_{1}^{\alpha}\left(\mathbb{R}^{2}\right)},\\
\left|\left|\widetilde{\Pi_{1}^{\left(n\right)}}^{n}\left(t,\cdot\right)\right|\right|_{\dot{C}_{2}^{\alpha}\left(\mathbb{R}\right)} & \leq\sum_{m=1}^{n-1}\left|\left|b_{m}\left(t\right)\left(\phi_{2}^{\left(n\right)}\left(t,x\right)-\phi_{2}^{\left(m\right)}\left(t,0\right)\right)\right|\right|_{\dot{C}^{1}\left(\mathbb{R}\right)}^{\alpha}\left|\left|\left(\widetilde{\nabla}^{m}\widetilde{\omega^{\left(m\right)}}^{m}\left(t,\cdot\right)\right)_{1}\right|\right|_{\dot{C}_{2}^{\alpha}\left(\mathbb{R}^{2}\right)},\\
\left|\left|\widetilde{\Pi_{2}^{\left(n\right)}}^{n}\left(t,\cdot\right)\right|\right|_{\dot{C}_{1}^{\alpha}\left(\mathbb{R}\right)} & \leq\sum_{m=1}^{n-1}\left|\left|a_{m}\left(t\right)\left(\phi_{1}^{\left(n\right)}\left(t,x\right)-\phi_{1}^{\left(m\right)}\left(t,0\right)\right)\right|\right|_{\dot{C}^{1}\left(\mathbb{R}\right)}^{\alpha}\left|\left|\left(\widetilde{\nabla}^{m}\widetilde{\omega^{\left(m\right)}}^{m}\left(t,\cdot\right)\right)_{2}\right|\right|_{\dot{C}_{1}^{\alpha}\left(\mathbb{R}^{2}\right)},\\
\left|\left|\widetilde{\Pi_{2}^{\left(n\right)}}^{n}\left(t,\cdot\right)\right|\right|_{\dot{C}_{2}^{\alpha}\left(\mathbb{R}\right)} & \leq\sum_{m=1}^{n-1}\left|\left|b_{m}\left(t\right)\left(\phi_{2}^{\left(n\right)}\left(t,x\right)-\phi_{2}^{\left(m\right)}\left(t,0\right)\right)\right|\right|_{\dot{C}^{1}\left(\mathbb{R}\right)}^{\alpha}\left|\left|\left(\widetilde{\nabla}^{m}\widetilde{\omega^{\left(m\right)}}^{m}\left(t,\cdot\right)\right)_{2}\right|\right|_{\dot{C}_{2}^{\alpha}\left(\mathbb{R}^{2}\right)}.
\end{aligned}
\]
By equation \eqref{eq:jacobian inverse} and Proposition \ref{prop:form gradient omega},
we deduce that
\[
\begin{aligned}\left|\left|\widetilde{\Pi_{1}^{\left(n\right)}}^{n}\left(t,\cdot\right)\right|\right|_{\dot{C}_{1}^{\alpha}\left(\mathbb{R}\right)} & \lesssim_{\varphi}\sum_{m=1}^{n-1}\left(\frac{a_{m}\left(t\right)}{a_{n}\left(t\right)}\right)^{\alpha}B_{m}\left(t\right)a_{m}\left(t\right)\max\left\{ a_{m}\left(t\right)^{2},b_{m}\left(t\right)^{2}\right\} \\
\left|\left|\widetilde{\Pi_{1}^{\left(n\right)}}^{n}\left(t,\cdot\right)\right|\right|_{\dot{C}_{2}^{\alpha}\left(\mathbb{R}\right)} & \lesssim_{\varphi}\sum_{m=1}^{n-1}\left(\frac{b_{m}\left(t\right)}{b_{n}\left(t\right)}\right)^{\alpha}B_{m}\left(t\right)a_{m}\left(t\right)\max\left\{ a_{m}\left(t\right)^{2},b_{m}\left(t\right)^{2}\right\} ,\\
\left|\left|\widetilde{\Pi_{2}^{\left(n\right)}}^{n}\left(t,\cdot\right)\right|\right|_{\dot{C}_{1}^{\alpha}\left(\mathbb{R}\right)} & \lesssim_{\varphi}\sum_{m=1}^{n-1}\left(\frac{a_{m}\left(t\right)}{a_{n}\left(t\right)}\right)^{\alpha}B_{m}\left(t\right)b_{m}\left(t\right)\max\left\{ a_{m}\left(t\right)^{2},b_{m}\left(t\right)^{2}\right\} ,\\
\left|\left|\widetilde{\Pi_{2}^{\left(n\right)}}^{n}\left(t,\cdot\right)\right|\right|_{\dot{C}_{2}^{\alpha}\left(\mathbb{R}\right)} & \lesssim_{\varphi}\sum_{m=1}^{n-1}\left(\frac{b_{m}\left(t\right)}{b_{n}\left(t\right)}\right)^{\alpha}B_{m}\left(t\right)b_{m}\left(t\right)\max\left\{ a_{m}\left(t\right)^{2},b_{m}\left(t\right)^{2}\right\} .
\end{aligned}
\]
Now, we may obtain the bounds for $\Pi^{\left(n\right)}$ via equations
\eqref{eq:property Calpha composition} and \eqref{eq:jacobian inverse},
which provide
\[
\begin{aligned}\left|\left|\Pi_{1}^{\left(n\right)}\left(t,\cdot\right)\right|\right|_{\dot{C}_{1}^{\alpha}\left(\mathbb{R}\right)} & =\left|\left|\widetilde{\Pi_{1}^{\left(n\right)}}^{n}\left(t,\left(\phi^{\left(n\right)}\right)^{-1}\left(t,\cdot\right)\right)\right|\right|_{\dot{C}_{1}^{\alpha}\left(\mathbb{R}\right)}\leq\\
 & \leq\left|\left|\left(\phi_{1}^{\left(n\right)}\right)^{-1}\left(t,\cdot\right)\right|\right|_{\dot{C}^{1}\left(\mathbb{R}\right)}^{\alpha}\left|\left|\widetilde{\Pi_{1}^{\left(n\right)}}^{n}\left(t,\cdot\right)\right|\right|_{\dot{C}_{1}^{\alpha}\left(\mathbb{R}\right)}\lesssim_{\varphi}\\
 & \lesssim_{\varphi}a_{n}\left(t\right)^{\alpha}\sum_{m=1}^{n-1}\left(\frac{a_{m}\left(t\right)}{a_{n}\left(t\right)}\right)^{\alpha}B_{m}\left(t\right)a_{m}\left(t\right)\max\left\{ a_{m}\left(t\right)^{2},b_{m}\left(t\right)^{2}\right\} =\\
 & \lesssim_{\varphi}\sum_{m=1}^{n-1}a_{m}\left(t\right)^{\alpha}B_{m}\left(t\right)a_{m}\left(t\right)\max\left\{ a_{m}\left(t\right)^{2},b_{m}\left(t\right)^{2}\right\} \leq\\
 & \lesssim_{\varphi}\sum_{m=1}^{n-1}B_{m}\left(t\right)a_{m}\left(t\right)\max\left\{ a_{m}\left(t\right)^{2+\alpha},b_{m}\left(t\right)^{2+\alpha}\right\} .
\end{aligned}
\]
Analogously, we conclude that
\[
\begin{aligned}\left|\left|\Pi_{1}^{\left(n\right)}\left(t,\cdot\right)\right|\right|_{\dot{C}_{2}^{\alpha}\left(\mathbb{R}\right)} & \lesssim_{\varphi}\sum_{m=1}^{n-1}B_{m}\left(t\right)a_{m}\left(t\right)\max\left\{ a_{m}\left(t\right)^{2+\alpha},b_{m}\left(t\right)^{2+\alpha}\right\} ,\\
\left|\left|\Pi_{2}^{\left(n\right)}\left(t,\cdot\right)\right|\right|_{\dot{C}_{1}^{\alpha}\left(\mathbb{R}\right)} & \lesssim_{\varphi}\sum_{m=1}^{n-1}B_{m}\left(t\right)b_{m}\left(t\right)\max\left\{ a_{m}\left(t\right)^{2+\alpha},b_{m}\left(t\right)^{2+\alpha}\right\} ,\\
\left|\left|\Pi_{2}^{\left(n\right)}\left(t,\cdot\right)\right|\right|_{\dot{C}_{2}^{\alpha}\left(\mathbb{R}\right)} & \lesssim_{\varphi}\sum_{m=1}^{n-1}B_{m}\left(t\right)b_{m}\left(t\right)\max\left\{ a_{m}\left(t\right)^{2+\alpha},b_{m}\left(t\right)^{2+\alpha}\right\} .
\end{aligned}
\]
With these computations done, we can recover the ``joint'' $\left|\left|\cdot\right|\right|_{\dot{C}^{\alpha}\left(\mathbb{R}^{2}\right)}$
seminorm via equation \eqref{eq:Holder seminorm by marginals}. In
this manner,
\begin{equation}
\begin{aligned}\left|\left|\Pi_{1}^{\left(n\right)}\left(t,\cdot\right)\right|\right|_{\dot{C}^{\alpha}\left(\mathbb{R}^{2}\right)} & \lesssim_{\varphi}\sum_{m=1}^{n-1}B_{m}\left(t\right)a_{m}\left(t\right)\max\left\{ a_{m}\left(t\right)^{2+\alpha},b_{m}\left(t\right)^{2+\alpha}\right\} ,\\
\left|\left|\Pi_{2}^{\left(n\right)}\left(t,\cdot\right)\right|\right|_{\dot{C}^{\alpha}\left(\mathbb{R}^{2}\right)} & \lesssim_{\varphi}\sum_{m=1}^{n-1}B_{m}\left(t\right)b_{m}\left(t\right)\max\left\{ a_{m}\left(t\right)^{2+\alpha},b_{m}\left(t\right)^{2+\alpha}\right\} .
\end{aligned}
\label{eq:bounds pi Calpha v0}
\end{equation}

As long as $C$ is chosen large enough (let us say $C\ge\Upsilon_{1}\left(\beta,\delta\right)$),
Proposition \ref{prop:time convergence} and the equality $a_{m}\left(1\right)=b_{m}\left(1\right)$
(which is due to Choices \ref{choice:anbn} and \ref{choice:time picture})
allow us to transform \eqref{eq:bounds pi Linfty v0} and \eqref{eq:bounds pi Calpha v0}
into
\[
\begin{aligned}\left|\left|\Pi_{1}^{\left(n\right)}\left(t,\cdot\right)\right|\right|_{L^{\infty}\left(\mathbb{R}^{2}\right)} & \lesssim_{\varphi}\sum_{m=1}^{n-1}\overbrace{B_{m}\left(1\right)a_{m}\left(1\right)b_{m}\left(1\right)}^{=M_{m}}b_{m}\left(1\right)=\sum_{m=1}^{n-1}YC^{\left(1+\delta\right)\left(\frac{1}{1-\gamma}\right)^{m}},\\
\left|\left|\Pi_{2}^{\left(n\right)}\left(t,\cdot\right)\right|\right|_{L^{\infty}\left(\mathbb{R}^{2}\right)} & \lesssim_{\varphi}\sum_{m=1}^{n-1}\overbrace{B_{m}\left(1\right)a_{m}\left(1\right)b_{m}\left(1\right)}^{=M_{m}}b_{m}\left(1\right)=\sum_{m=1}^{n-1}YC^{\left(1+\delta\right)\left(\frac{1}{1-\gamma}\right)^{m}},
\end{aligned}
\]
\[
\begin{aligned}\left|\left|\Pi_{1}^{\left(n\right)}\left(t,\cdot\right)\right|\right|_{\dot{C}^{\alpha}\left(\mathbb{R}^{2}\right)} & \lesssim_{\varphi}\sum_{m=1}^{n-1}\overbrace{B_{m}\left(1\right)a_{m}\left(1\right)b_{m}\left(1\right)}^{=M_{m}}b_{m}\left(1\right)^{1+\alpha}=\sum_{m=1}^{n-1}YC^{\left(1+\alpha+\delta\right)\left(\frac{1}{1-\gamma}\right)^{m}},\\
\left|\left|\Pi_{2}^{\left(n\right)}\left(t,\cdot\right)\right|\right|_{\dot{C}^{\alpha}\left(\mathbb{R}^{2}\right)} & \lesssim_{\varphi}\sum_{m=1}^{n-1}\overbrace{B_{m}\left(1\right)a_{m}\left(1\right)b_{m}\left(1\right)}^{=M_{m}}b_{m}\left(1\right)^{1+\alpha}=\sum_{m=1}^{n-1}YC^{\left(1+\alpha+\delta\right)\left(\frac{1}{1-\gamma}\right)^{m}}.
\end{aligned}
\]
Since $\alpha<1$, we can easily bound
\[
\begin{aligned}\left|\left|\Pi^{\left(n\right)}\left(t,\cdot\right)\right|\right|_{L^{\infty}\left(\mathbb{R}^{2};\mathbb{R}^{2}\right)} & \lesssim_{\varphi}\sum_{m=1}^{n-1}YC^{\left(1+\delta\right)\left(\frac{1}{1-\gamma}\right)^{m}},\\
\left|\left|\Pi^{\left(n\right)}\left(t,\cdot\right)\right|\right|_{\dot{C}^{\alpha}\left(\mathbb{R}^{2};\mathbb{R}^{2}\right)} & \lesssim_{\varphi}\sum_{m=1}^{n-1}YC^{\left(2+\delta\right)\left(\frac{1}{1-\gamma}\right)^{m}}.
\end{aligned}
\]
Now, Lemma \ref{lem:estimate sum superexponential} (which we can
make use of thanks to Choice \ref{choice:min requirements on C gamma and kmax})
provides
\begin{equation}
\begin{aligned}\left|\left|\Pi^{\left(n\right)}\left(t,\cdot\right)\right|\right|_{L^{\infty}\left(\mathbb{R}^{2};\mathbb{R}^{2}\right)} & \lesssim_{\varphi,\delta}YC^{\left(1+\delta\right)\left(\frac{1}{1-\gamma}\right)^{n-1}},\\
\left|\left|\Pi^{\left(n\right)}\left(t,\cdot\right)\right|\right|_{\dot{C}^{\alpha}\left(\mathbb{R}^{2};\mathbb{R}^{2}\right)} & \lesssim_{\varphi,\delta}YC^{\left(2+\delta\right)\left(\frac{1}{1-\gamma}\right)^{n-1}}.
\end{aligned}
\label{eq:bounds pi}
\end{equation}

Before obtaining the result of the statement, we have to rewrite a
little bit the expressions we achieved in Proposition \ref{prop:form of the velocity}.
In view of that Proposition, we have
\[
\begin{aligned}\left|\left|u^{\left(n\right)}\left(t,\cdot\right)\right|\right|_{L^{\infty}\left(\mathbb{R}^{2};\mathbb{R}^{2}\right)} & \lesssim_{\varphi}B_{n}\left(t\right)\max\left\{ a_{n}\left(t\right),b_{n}\left(t\right)\right\} ,\\
\left|\left|u^{\left(n\right)}\left(t,\cdot\right)\right|\right|_{\dot{C}^{\alpha}\left(\mathbb{R}^{2};\mathbb{R}^{2}\right)} & \lesssim_{\varphi}B_{n}\left(t\right)\max\left\{ a_{n}\left(t\right)^{1+\alpha},b_{n}\left(t\right)^{1+\alpha}\right\} .
\end{aligned}
\]
Provided that $C$ is large enough (let us say $C\ge\Upsilon_{2}\left(\delta,\mu\right)$),
Choice \ref{choice:amplitude density} and Corollary \ref{cor:bound for an(t), bn(t)}
let us write
\[
\begin{aligned}\left|\left|u^{\left(n\right)}\left(t,\cdot\right)\right|\right|_{L^{\infty}\left(\mathbb{R}^{2};\mathbb{R}^{2}\right)} & \lesssim_{\varphi}\frac{z_{n}}{a_{n}\left(t\right)^{2}+b_{n}\left(t\right)^{2}}\overbrace{\frac{\int_{t_{n}}^{t}h^{\left(n\right)}\left(s\right)b_{n}\left(s\right)\mathrm{d}s}{\int_{t_{n}}^{1}h^{\left(n\right)}\left(s\right)b_{n}\left(s\right)\mathrm{d}s}}^{\le1}C^{\left(1+\overline{k}_{n}\left(t\right)+\mu\right)\left(\frac{1}{1-\gamma}\right)^{n}},\\
\left|\left|u^{\left(n\right)}\left(t,\cdot\right)\right|\right|_{\dot{C}^{\alpha}\left(\mathbb{R}^{2};\mathbb{R}^{2}\right)} & \lesssim_{\varphi}\frac{z_{n}}{a_{n}\left(t\right)^{2}+b_{n}\left(t\right)^{2}}\underbrace{\frac{\int_{t_{n}}^{t}h^{\left(n\right)}\left(s\right)b_{n}\left(s\right)\mathrm{d}s}{\int_{t_{n}}^{1}h^{\left(n\right)}\left(s\right)b_{n}\left(s\right)\mathrm{d}s}}_{\le1}C^{\left(1+\alpha\right)\left(1+\overline{k}_{n}\left(t\right)+\mu\right)\left(\frac{1}{1-\gamma}\right)^{n}}.
\end{aligned}
\]
By Choice \ref{choice:time picture}, equation \eqref{eq:relation Mn and zn}
and Choice \ref{choice:Mn}, we have
\begin{equation}
\begin{aligned}\left|\left|u^{\left(n\right)}\left(t,\cdot\right)\right|\right|_{L^{\infty}\left(\mathbb{R}^{2};\mathbb{R}^{2}\right)} & \lesssim_{\varphi}\frac{Y}{a_{n}\left(t\right)^{2}+b_{n}\left(t\right)^{2}}C^{\left(1+\overline{k}_{n}\left(t\right)+\mu+\delta\right)\left(\frac{1}{1-\gamma}\right)^{n}},\\
\left|\left|u^{\left(n\right)}\left(t,\cdot\right)\right|\right|_{\dot{C}^{\alpha}\left(\mathbb{R}^{2};\mathbb{R}^{2}\right)} & \lesssim_{\varphi}\frac{Y}{a_{n}\left(t\right)^{2}+b_{n}\left(t\right)^{2}}C^{\left[\left(1+\alpha\right)\left(1+\overline{k}_{n}\left(t\right)+\mu\right)+\delta\right]\left(\frac{1}{1-\gamma}\right)^{n}}.
\end{aligned}
\label{eq:bounds un v0}
\end{equation}
To continue, notice that
\[
\frac{1}{a_{n}\left(t\right)^{2}+b_{n}\left(t\right)^{2}}\le\frac{1}{a_{n}\left(t\right)^{2}},\quad\frac{1}{a_{n}\left(t\right)^{2}+b_{n}\left(t\right)^{2}}\le\frac{1}{b_{n}\left(t\right)^{2}}.
\]
In other words,
\[
\frac{1}{a_{n}\left(t\right)^{2}+b_{n}\left(t\right)^{2}}\le\min\left\{ \frac{1}{a_{n}\left(t\right)^{2}},\frac{1}{b_{n}\left(t\right)^{2}}\right\} =\frac{1}{\max\left\{ a_{n}\left(t\right)^{2},b_{n}\left(t\right)^{2}\right\} }.
\]
Resorting to Corollary \ref{cor:bound for an(t), bn(t)}, we obtain
\begin{equation}
\frac{1}{a_{n}\left(t\right)^{2}+b_{n}\left(t\right)^{2}}\le\frac{1}{C^{2\left(1+\overline{k}_{n}\left(t\right)-\mu\right)\left(\frac{1}{1-\gamma}\right)^{n}}}.\label{eq:bound 1 over sum an2 bn2}
\end{equation}
Making use of \eqref{eq:bound 1 over sum an2 bn2} in \eqref{eq:bounds un v0},
we get
\[
\begin{aligned}\left|\left|u^{\left(n\right)}\left(t,\cdot\right)\right|\right|_{L^{\infty}\left(\mathbb{R}^{2};\mathbb{R}^{2}\right)} & \lesssim_{\varphi}YC^{\left(-1-\overline{k}_{n}\left(t\right)+3\mu+\delta\right)\left(\frac{1}{1-\gamma}\right)^{n}},\\
\left|\left|u^{\left(n\right)}\left(t,\cdot\right)\right|\right|_{\dot{C}^{\alpha}\left(\mathbb{R}^{2};\mathbb{R}^{2}\right)} & \lesssim_{\varphi}YC^{\left[-\left(1-\alpha\right)\left(1+\overline{k}_{n}\left(t\right)\right)+\left(3+\alpha\right)\mu+\delta\right]\left(\frac{1}{1-\gamma}\right)^{n}}.
\end{aligned}
\]
Since $\overline{k}_{n}\left(t\right)\ge0$ by Corollary \ref{cor:ideal kn non negative}
and $\alpha<1$, we have
\begin{equation}
\begin{aligned}\left|\left|u^{\left(n\right)}\left(t,\cdot\right)\right|\right|_{L^{\infty}\left(\mathbb{R}^{2};\mathbb{R}^{2}\right)} & \lesssim_{\varphi}YC^{\left(-1+3\mu+\delta\right)\left(\frac{1}{1-\gamma}\right)^{n}},\\
\left|\left|u^{\left(n\right)}\left(t,\cdot\right)\right|\right|_{\dot{C}^{\alpha}\left(\mathbb{R}^{2};\mathbb{R}^{2}\right)} & \lesssim_{\varphi}YC^{\left[-\left(1-\alpha\right)+4\mu+\delta\right]\left(\frac{1}{1-\gamma}\right)^{n}}.
\end{aligned}
\label{eq:bounds un v1}
\end{equation}

Now, we can combine equations \eqref{eq:bounds pi} and \eqref{eq:bounds un v1},
applying equation \eqref{eq:property Calpha multiplication} to obtain
\[
\begin{aligned}\left|\left|\mathrm{O}^{\left(n\right)}\left(t,\cdot\right)\right|\right|_{L^{\infty}\left(\mathbb{R}^{2}\right)} & \lesssim_{\varphi,\delta}Y^{2}C^{\left(-1+3\mu+\delta\left(2-\gamma\right)+\left(1-\gamma\right)\right)\left(\frac{1}{1-\gamma}\right)^{n}},\\
\left|\left|\mathrm{O}^{\left(n\right)}\left(t,\cdot\right)\right|\right|_{\dot{C}^{\alpha}\left(\mathbb{R}^{2}\right)} & \lesssim_{\varphi,\delta}\left|\left|u^{\left(n\right)}\left(t,\cdot\right)\right|\right|_{L^{\infty}\left(\mathbb{R}^{2};\mathbb{R}^{2}\right)}\left|\left|\Pi^{\left(n\right)}\left(t,\cdot\right)\right|\right|_{\dot{C}^{\alpha}\left(\mathbb{R}^{2};\mathbb{R}^{2}\right)}+\\
 & \quad+\left|\left|u^{\left(n\right)}\left(t,\cdot\right)\right|\right|_{\dot{C}^{\alpha}\left(\mathbb{R}^{2};\mathbb{R}^{2}\right)}\left|\left|\Pi^{\left(n\right)}\left(t,\cdot\right)\right|\right|_{L^{\infty}\left(\mathbb{R}^{2};\mathbb{R}^{2}\right)}=\\
 & \lesssim_{\varphi,\delta}YC^{\left(-1+3\mu+\delta\right)\left(\frac{1}{1-\gamma}\right)^{n}}YC^{\left(2+\delta\right)\left(\frac{1}{1-\gamma}\right)^{n-1}}+\\
 & \quad+YC^{\left[-\left(1-\alpha\right)+4\mu+\delta\right]\left(\frac{1}{1-\gamma}\right)^{n}}YC^{\left(1+\delta\right)\left(\frac{1}{1-\gamma}\right)^{n-1}}\leq\\
 & \lesssim_{\varphi,\delta}Y^{2}C^{\left[-\left(1-\alpha\right)+4\mu+2\left(1-\gamma\right)+\delta\left(2-\gamma\right)\right]\left(\frac{1}{1-\gamma}\right)^{n}}.
\end{aligned}
\]
Bounding $\gamma\ge0$ where it is multiplied by $\delta$ provides
the result of the statement.
\end{proof}

\subsection{Transport term}
\begin{prop}
\label{prop:bound vorticity transport}Let $n\in\mathbb{N}$ with
$n\ge2$, $\mu>0$ and $\alpha\in\left(0,1\right)$. As long as $C$
is big enough (let us say $C\ge\Upsilon\left(\delta,\mu\right)$),
the transport term
\[
\widetilde{T_{\omega}^{\left(n\right)}}^{n}\left(t,x\right)\coloneqq\left(\widetilde{U^{\left(n-1\right)}}^{n}\left(t,x\right)-\widetilde{U^{\left(n-1\right)}}^{n}\left(t,0\right)-\mathrm{J}\widetilde{U^{\left(n-1\right)}}^{n}\left(t,0\right)\cdot\left(\begin{matrix}x_{1}\\
x_{2}
\end{matrix}\right)\right)\cdot\widetilde{\nabla}^{n}\widetilde{\omega^{\left(n\right)}}^{n}\left(t,x\right)
\]
satisfies the bound
\[
\begin{aligned}\left|\left|T_{\omega}^{\left(n\right)}\left(t,\cdot\right)\right|\right|_{\dot{C}^{\alpha}\left(\mathbb{R}^{2}\right)} & \lesssim_{\varphi,\delta}Y^{2}C^{\left[4\mu+2\delta+2\left(1-\gamma\right)\right]\left(\frac{1}{1-\gamma}\right)^{n}}\left[C^{\left[\alpha\left(1+k_{\max}\right)-1+2\Lambda+3k_{\max}\right]\left(\frac{1}{1-\gamma}\right)^{n}}+\right.\\
 & \quad\left.+C^{\left[-1+2\Lambda+3k_{\max}+\max\left\{ \alpha\left(1-k_{\max}\right)-\Lambda,0\right\} \right]\left(\frac{1}{1-\gamma}\right)^{n}}\right].
\end{aligned}
\]
In other words, for the term above to be decreasing in $n\in\mathbb{N}$,
we need
\begin{enumerate}
\item ~
\[
\alpha\left(1+k_{\max}\right)-1+2\Lambda+3k_{\max}+4\mu+2\delta+2\left(1-\gamma\right)<0,
\]
\item ~
\[
-1+2\Lambda+3k_{\max}+4\mu+2\delta+2\left(1-\gamma\right)<0,
\]
\item if $\alpha\left(1-k_{\max}\right)-\Lambda>0$, we also require
\[
\alpha\left(1-k_{\max}\right)-1+\Lambda+3k_{\max}+4\mu+2\delta+2\left(1-\gamma\right)<0.
\]
\end{enumerate}
\end{prop}
\begin{proof}
First, we should readapt the bounds given in Proposition \ref{prop:form gradient omega}.
By the aforementioned Proposition, as $\lambda_{n}\leq1$ by Choice
\ref{choice:psin}, it is clear that
\[
\begin{aligned}\left|\left|\widetilde{\nabla}^{n}\widetilde{\omega^{\left(n\right)}}^{n}\left(t,\left(\phi^{\left(n\right)}\right)^{-1}\left(t,\cdot\right)\right)\right|\right|_{L^{\infty}\left(\mathbb{R}^{2};\mathbb{R}^{2}\right)} & \lesssim_{\varphi}B_{n}\left(t\right)\max\left\{ a_{n}\left(t\right)^{3},b_{n}\left(t\right)^{3}\right\} ,\\
\left|\left|\widetilde{\nabla}^{n}\widetilde{\omega^{\left(n\right)}}^{n}\left(t,\left(\phi^{\left(n\right)}\right)^{-1}\left(t,\cdot\right)\right)\right|\right|_{\dot{C}^{\alpha}\left(\mathbb{R}^{2};\mathbb{R}^{2}\right)} & \lesssim_{\varphi}B_{n}\left(t\right)\max\left\{ a_{n}\left(t\right)^{3+\alpha},b_{n}\left(t\right)^{3+\alpha}\right\} .
\end{aligned}
\]
By Choice \ref{choice:amplitude density}, equation \eqref{eq:relation Mn and zn}
and Choice \ref{choice:time picture}, we have
\[
\begin{aligned}\left|\left|\widetilde{\nabla}^{n}\widetilde{\omega^{\left(n\right)}}^{n}\left(t,\left(\phi^{\left(n\right)}\right)^{-1}\left(t,\cdot\right)\right)\right|\right|_{L^{\infty}\left(\mathbb{R}^{2};\mathbb{R}^{2}\right)} & \lesssim_{\varphi}2M_{n}\underbrace{\frac{\max\left\{ a_{n}\left(t\right)^{2},b_{n}\left(t\right)^{2}\right\} }{a_{n}\left(t\right)^{2}+b_{n}\left(t\right)^{2}}}_{\leq1}\underbrace{\frac{\int_{t_{n}}^{t}h^{\left(n\right)}\left(s\right)b_{n}\left(s\right)\mathrm{d}s}{\int_{t_{n}}^{1}h^{\left(n\right)}\left(s\right)b_{n}\left(s\right)\mathrm{d}s}}_{\leq1}\cdot\\
 & \quad\cdot\max\left\{ a_{n}\left(t\right),b_{n}\left(t\right)\right\} \leq\\
 & \lesssim_{\varphi}M_{n}\max\left\{ a_{n}\left(t\right),b_{n}\left(t\right)\right\} .\\
\left|\left|\widetilde{\nabla}^{n}\widetilde{\omega^{\left(n\right)}}^{n}\left(t,\left(\phi^{\left(n\right)}\right)^{-1}\left(t,\cdot\right)\right)\right|\right|_{\dot{C}^{\alpha}\left(\mathbb{R}^{2};\mathbb{R}^{2}\right)} & \lesssim_{\varphi}2M_{n}\underbrace{\frac{\max\left\{ a_{n}\left(t\right)^{2},b_{n}\left(t\right)^{2}\right\} }{a_{n}\left(t\right)^{2}+b_{n}\left(t\right)^{2}}}_{\leq1}\underbrace{\frac{\int_{t_{n}}^{t}h^{\left(n\right)}\left(s\right)b_{n}\left(s\right)\mathrm{d}s}{\int_{t_{n}}^{1}h^{\left(n\right)}\left(s\right)b_{n}\left(s\right)\mathrm{d}s}}_{\leq1}\cdot\\
 & \quad\cdot\max\left\{ a_{n}\left(t\right)^{1+\alpha},b_{n}\left(t\right)^{1+\alpha}\right\} \leq\\
 & \lesssim_{\varphi}M_{n}\max\left\{ a_{n}\left(t\right)^{1+\alpha},b_{n}\left(t\right)^{1+\alpha}\right\} .
\end{aligned}
\]
As long as $C$ is taken large enough (let us say $C\ge\Upsilon_{1}\left(\delta,\mu\right)$),
since $n\ge2$, Corollary \ref{cor:bound for an(t), bn(t)} and Choices
\ref{choice:Mn} and \ref{choice:ideal kn} provide
\begin{equation}
\begin{aligned}\left|\left|\widetilde{\nabla}^{n}\widetilde{\omega^{\left(n\right)}}^{n}\left(t,\left(\phi^{\left(n\right)}\right)^{-1}\left(t,\cdot\right)\right)\right|\right|_{L^{\infty}\left(\mathbb{R}^{2};\mathbb{R}^{2}\right)} & \lesssim_{\varphi}YC^{\left[1+k_{\max}+\mu+\delta\right]\left(\frac{1}{1-\gamma}\right)^{n}},\\
\left|\left|\widetilde{\nabla}^{n}\widetilde{\omega^{\left(n\right)}}^{n}\left(t,\left(\phi^{\left(n\right)}\right)^{-1}\left(t,\cdot\right)\right)\right|\right|_{\dot{C}^{\alpha}\left(\mathbb{R}^{2};\mathbb{R}^{2}\right)} & \lesssim_{\varphi}YC^{\left[\left(1+\alpha\right)\left(1+k_{\max}+\mu\right)+\delta\right]\left(\frac{1}{1-\gamma}\right)^{n}}.
\end{aligned}
\label{eq:updated bounds for grad omega}
\end{equation}

To obtain the result, we need to combine Proposition \ref{prop:bounds for Taylor development of transport}
with equation \eqref{eq:updated bounds for grad omega} via equation
\eqref{eq:property Calpha multiplication}. Notice that, as $\widetilde{\nabla}^{n}\widetilde{\omega^{\left(n\right)}}^{n}$
has compact support exactly where the bounds of Proposition \ref{prop:bounds for Taylor development of transport}
are computed (see Choice \ref{choice:psin}), we can use the bounds
of Proposition \ref{prop:bounds for Taylor development of transport}
to bound $T_{\omega}^{\left(n\right)}$ in all $\mathbb{R}^{2}$.
In this way, bounding $\gamma\ge0$ when it is multiplied by $\delta$,
\[
\begin{aligned}\left|\left|T_{\omega}^{\left(n\right)}\left(t,\cdot\right)\right|\right|_{L^{\infty}\left(\mathbb{R}^{2}\right)} & \leq\left|\left|W^{\left(n\right)}\left(t,\cdot\right)\right|\right|_{L^{\infty}\left(D^{\left(n\right)}\left(t\right);\mathbb{R}^{2}\right)}\left|\left|\widetilde{\nabla}^{n}\widetilde{\omega^{\left(n\right)}}^{n}\left(t,\left(\phi^{\left(n\right)}\right)^{-1}\left(t,\cdot\right)\right)\right|\right|_{L^{\infty}\left(\mathbb{R}^{2};\mathbb{R}^{2}\right)}\lesssim_{\varphi,\delta}\\
 & \lesssim_{\varphi,\delta}YC^{\left[-2\left(1-\Lambda-k_{\max}\right)+2\mu+\left(1+\delta\right)\left(1-\gamma\right)\right]\left(\frac{1}{1-\gamma}\right)^{n}}YC^{\left[1+k_{\max}+\mu+\delta\right]\left(\frac{1}{1-\gamma}\right)^{n}}=\\
 & \lesssim_{\varphi,\delta}Y^{2}C^{\left[-1+2\Lambda+3k_{\max}+3\mu+2\delta+\left(1-\gamma\right)\right]\left(\frac{1}{1-\gamma}\right)^{n}}.
\end{aligned}
\]
\[
\begin{aligned}\left|\left|T_{\omega}^{\left(n\right)}\left(t,\cdot\right)\right|\right|_{\dot{C}^{\alpha}\left(\mathbb{R}^{2}\right)} & \leq\left|\left|W^{\left(n\right)}\left(t,\cdot\right)\right|\right|_{L^{\infty}\left(D^{\left(m\right)}\left(t\right);\mathbb{R}^{2}\right)}\left|\left|\widetilde{\nabla}^{n}\widetilde{\omega^{\left(n\right)}}^{n}\left(t,\left(\phi^{\left(n\right)}\right)^{-1}\left(t,\cdot\right)\right)\right|\right|_{\dot{C}^{\alpha}\left(\mathbb{R}^{2};\mathbb{R}^{2}\right)}+\\
 & \quad+\left|\left|W^{\left(n\right)}\left(t,\cdot\right)\right|\right|_{\dot{C}^{\alpha}\left(D^{\left(m\right)}\left(t\right);\mathbb{R}^{2}\right)}\left|\left|\widetilde{\nabla}^{n}\widetilde{\omega^{\left(n\right)}}^{n}\left(t,\left(\phi^{\left(n\right)}\right)^{-1}\left(t,\cdot\right)\right)\right|\right|_{L^{\infty}\left(\mathbb{R}^{2};\mathbb{R}^{2}\right)}\lesssim_{\varphi,\delta}\\
 & \lesssim_{\varphi,\delta}YC^{\left[-2\left(1-\Lambda-k_{\max}\right)+2\mu+\left(1+\delta\right)\left(1-\gamma\right)\right]\left(\frac{1}{1-\gamma}\right)^{n}}YC^{\left[\left(1+\alpha\right)\left(1+k_{\max}+\mu\right)+\delta\right]\left(\frac{1}{1-\gamma}\right)^{n}}+\\
 & \quad+YC^{\left[-2\left(1-\Lambda-k_{\max}\right)+\max\left\{ \alpha\left(1-k_{\max}\right)-\Lambda,0\right\} +2\mu+\left(2+\delta\right)\left(1-\gamma\right)\right]\left(\frac{1}{1-\gamma}\right)^{n}}YC^{\left[1+k_{\max}+\mu+\delta\right]\left(\frac{1}{1-\gamma}\right)^{n}}\lesssim_{\varphi,\delta}\\
 & \lesssim_{\varphi,\delta}Y^{2}C^{\left[\alpha\left(1+k_{\max}\right)-1+2\Lambda+3k_{\max}+4\mu+2\delta+\left(1-\gamma\right)\right]\left(\frac{1}{1-\gamma}\right)^{n}}+\\
 & \quad+Y^{2}C^{\left[-1+2\Lambda+3k_{\max}+\max\left\{ \alpha\left(1-k_{\max}\right)-\Lambda,0\right\} +3\mu+2\delta+2\left(1-\gamma\right)\right]\left(\frac{1}{1-\gamma}\right)^{n}}\lesssim\\
 & \lesssim_{\varphi,\delta}Y^{2}C^{\left[4\mu+2\delta+2\left(1-\gamma\right)\right]\left(\frac{1}{1-\gamma}\right)^{n}}\left[C^{\left[\alpha\left(1+k_{\max}\right)-1+2\Lambda+3k_{\max}\right]\left(\frac{1}{1-\gamma}\right)^{n}}+\right.\\
 & \quad\left.+C^{\left[-1+2\Lambda+3k_{\max}+\max\left\{ \alpha\left(1-k_{\max}\right)-\Lambda,0\right\} \right]\left(\frac{1}{1-\gamma}\right)^{n}}\right].
\end{aligned}
\]
As $\left|\left|T_{\omega}^{\left(n\right)}\left(t,\cdot\right)\right|\right|_{L^{\infty}\left(\mathbb{R}^{2}\right)}$
is bounded by the last summand of $\left|\left|T_{\omega}^{\left(n\right)}\left(t,\cdot\right)\right|\right|_{\dot{C}^{\alpha}\left(\mathbb{R}^{2}\right)}$,
we arrive to the result of the statement.
\end{proof}

\section{\label{sec:final parameter optimization and closing arguments}Final
parameter optimization and closing arguments}

Having accomplished the task of finding bounds for the different summands
of the forces $f_{\rho}^{\left(n\right)}$ and $f_{\omega}^{\left(n\right)}$
in sections \ref{sec:bounds for density force} and \ref{sec:bounds for vorticity force},
in this section, we will finally compute the regularity of the force
and prove that it has compact support. Finally, we will prove the
main Theorem (see \ref{thm:MAIN THM}).

\subsection{Regularity}
\begin{prop}
\label{prop:maximal regularity}$\forall\alpha\in\left(0,\alpha_{*}\right)$,
where $\alpha_{*}=\sqrt{\frac{4}{3}}-1$, taking
\[
k_{\max}=k_{\max*}\coloneqq\alpha_{*}=\sqrt{\frac{4}{3}}-1,\quad\Lambda=\Lambda_{*}\coloneqq\left(\alpha_{*}+1\right)\alpha_{*}=\sqrt{\frac{4}{3}}\left(\sqrt{\frac{4}{3}}-1\right),
\]
we may choose $\delta,\mu,\zeta$ sufficiently small (i.e., $\delta,\mu,\zeta\le\Upsilon_{1}\left(\alpha\right)$)
and $\gamma$ sufficiently close to one (i.e., $1-\gamma\le\Upsilon_{2}\left(\alpha\right)$),
so that, once the value of $\mu,\delta,\gamma,\zeta$ are fixed (whatever
they are but satisfying the restrictions above), the forces $F_{\rho}^{\left(n\right)}$
and $F_{\omega}^{\left(n\right)}$satisfy $\lim_{n\to\infty}F_{\rho}^{\left(n\right)}\in C_{t}^{0}C_{x}^{1,\alpha}\left(\left[0,1\right]\times\mathbb{R}^{2}\right)$
and $\lim_{n\to\infty}F_{\omega}^{\left(n\right)}\in C_{t}^{0}C_{x}^{\alpha}\left(\left[0,1\right]\times\mathbb{R}^{2}\right)$.
\end{prop}
\begin{proof}
First of all, let us compute the maximal value of $\alpha$. To do
this, we need to combine the restrictions imposed by Propositions
\ref{prop:bound density time derivative}, \ref{prop:bound density pure quadratic term},
\ref{prop:bound density new transport of old density}, \ref{prop:bound density transport},
\ref{prop:bound vorticity time factor}, \ref{prop:bound vorticity pure quadratic term},
\ref{prop:bound vorticity new transport of old vorticity} and \ref{prop:bound vorticity transport}.
These are
\begin{enumerate}
\item (Proposition \ref{prop:bound density time derivative})
\[
\alpha<k_{\max}-\epsilon_{1}\left(\zeta,\mu,\delta\right),
\]
\item (Proposition \ref{prop:bound density pure quadratic term})
\[
\alpha<\min\left\{ k_{\max},\left(1-\alpha\right)k_{\max}+\alpha\Lambda,\Lambda-\alpha k_{\max}\right\} -\epsilon_{2}\left(\mu,\delta\right),
\]
\item (Proposition \ref{prop:bound density transport})
\[
-\left(1-\alpha\right)\left(1-k_{\max}\right)+\epsilon_{3}\left(\mu,\delta,\gamma\right)<0\iff\alpha<1-\frac{\epsilon_{3}\left(\mu,\delta,\gamma\right)}{1-k_{\max}},
\]
\item (Proposition \ref{prop:bound density transport})
\[
-\left(1-k_{\max}-\Lambda\right)+\epsilon_{4}\left(\mu,\delta,\gamma\right)<0,
\]
\item (Proposition \ref{prop:bound density transport})
\[
\alpha\left(1+k_{\max}\right)+2\Lambda+3k_{\max}-1+\epsilon_{5}\left(\mu,\delta,\gamma\right)<0\iff\alpha<\frac{1-2\Lambda-3k_{\max}-\epsilon_{5}\left(\mu,\delta,\gamma\right)}{1+k_{\max}},
\]
\item (Proposition \ref{prop:bound density transport}) if $\alpha\left(1-k_{\max}\right)-\Lambda>0$,
we also require
\[
\alpha\left(1-k_{\max}\right)+\Lambda+3k_{\max}-1+\epsilon_{6}\left(\mu,\delta,\gamma\right)<0\iff\alpha<\frac{1-\Lambda-3k_{\max}-\epsilon_{6}\left(\mu,\delta,\gamma\right)}{1-k_{\max}},
\]
\item (Proposition \ref{prop:bound vorticity time factor})
\[
\alpha<\frac{\Lambda-\epsilon_{7}\left(\mu,\delta\right)}{1+k_{\max}},
\]
\item (Proposition \ref{prop:bound vorticity pure quadratic term})
\[
\alpha<\frac{\Lambda-\epsilon_{8}\left(\mu,\delta\right)}{1+k_{\max}},
\]
\item (Proposition \ref{prop:bound vorticity new transport of old vorticity})
\[
-\left(1-\alpha\right)+\epsilon_{9}\left(\mu,\delta,\gamma\right)<0\iff\alpha<1-\epsilon_{9}\left(\mu,\delta,\gamma\right),
\]
\item (Proposition \ref{prop:bound vorticity transport})
\[
\alpha\left(1+k_{\max}\right)-1+2\Lambda+3k_{\max}+\epsilon_{10}\left(\mu,\delta,\gamma\right)<0\iff\alpha<\frac{1-2\Lambda-3k_{\max}-\epsilon_{10}\left(\mu,\delta,\gamma\right)}{1+k_{\max}},
\]
\item (Proposition \ref{prop:bound vorticity transport})
\[
-1+2\Lambda+3k_{\max}+\epsilon_{11}\left(\mu,\delta,\gamma\right)<0,
\]
\item (Proposition \ref{prop:bound vorticity transport}) if $\alpha\left(1-k_{\max}\right)-\Lambda>0$,
we also require
\[
\alpha\left(1-k_{\max}\right)-1+\Lambda+3k_{\max}+\epsilon_{12}\left(\mu,\delta,\gamma\right)<0\iff\alpha<\frac{1-\Lambda-3k_{\max}-\epsilon_{12}\left(\mu,\delta,\gamma\right)}{1-k_{\max}},
\]
\end{enumerate}
where the $\epsilon_{i}$ are continuous functions that vanish when
$\mu=0$, $\delta=0$, $\gamma=1$ and $\zeta=0$. We will call $\epsilon\coloneqq\max_{i\in\left\{ 1,\dots,12\right\} }\epsilon_{i}$,
which is also continuous and vanishes when $\mu=0$, $\delta=0$,
$\gamma=1$ and $\zeta=0$. Thereby, replacing the $\epsilon_{i}$
by $\epsilon$, notice that we obtain equations that are more restrictive,
i.e., if we find a solution for the inequalities with $\epsilon$,
that solution will also be a solution for the real inequalities (with
the $\epsilon_{i}$). Then, in terms of $\epsilon$, some equations
are superfluous. For example, 5 and 10 become the same. The same happens
with equations 6 and 12 and equations 7 and 8. Besides, as $k_{\max}<1$
by Choice \ref{choice:ideal kn}, if equation 1 is satisfied, so will
be equation 9.

To solve the system, we will focus on just three equations and prove
that any solution that satisfies those three will also satisfy the
other nine. In this manner, we will restrict our attention to equations
1, 5 and 7.
\begin{equation}
\alpha<k_{\max}-\epsilon\left(\zeta,\mu,\delta,\gamma\right),\quad\alpha<\frac{1-2\Lambda-3k_{\max}-\epsilon\left(\zeta,\mu,\delta,\gamma\right)}{1+k_{\max}},\quad\alpha<\frac{\Lambda-\epsilon\left(\zeta,\mu,\delta,\gamma\right)}{1+k_{\max}}.\label{eq:the three subequations}
\end{equation}
Furthermore, we shall first solve the case of equality and as if it
were $\epsilon\equiv0$, i.e., 
\begin{equation}
\alpha=k_{\max},\quad\alpha=\frac{1-2\Lambda-3k_{\max}}{1+k_{\max}},\quad\alpha=\frac{\Lambda}{1+k_{\max}}.\label{eq:reduced equations}
\end{equation}
From the second and third equations of \eqref{eq:reduced equations},
we get
\[
1-2\Lambda-3k_{\max}=\alpha\left(1+k_{\max}\right)=\Lambda\iff1-3k_{\max}=3\Lambda\iff\Lambda=\frac{1-3k_{\max}}{3}=\frac{1}{3}-k_{\max}.
\]
By the first equation of \eqref{eq:reduced equations}, we obtain
\[
\Lambda=\frac{1}{3}-\alpha.
\]
In this way, the third equation of \eqref{eq:reduced equations} becomes
\[
\alpha=\frac{\frac{1}{3}-\alpha}{1+\alpha}\iff\alpha\left(1+\alpha\right)=\frac{1}{3}-\alpha\iff\alpha^{2}+2\alpha-\frac{1}{3}=0\iff\alpha=-1\pm\sqrt{\frac{4}{3}}.
\]
The negative solution is of no interest to us, so we deduce $\alpha=\sqrt{\frac{4}{3}}-1$.
We shall call this $\alpha_{*}\coloneqq\sqrt{\frac{4}{3}}-1$. The
corresponding $\Lambda$ and $k_{\max}$ values are
\[
\Lambda_{*}=\alpha_{*}\left(1+\alpha_{*}\right)=\sqrt{\frac{4}{3}}\left(\sqrt{\frac{4}{3}}-1\right),\quad k_{\max*}=\alpha_{*}=\sqrt{\frac{4}{3}}-1.
\]
Notice that $k_{\max*}$ and $\Lambda_{*}$ coincide with the ones
given in the statement.

Now, since the inequalities given in \eqref{eq:the three subequations}
are all of the form $\alpha<\text{something}$, if we choose any $\alpha<\alpha_{*}$,
provided that $\epsilon$ is sufficiently small, for example, satisfying
\[
\epsilon\left(\zeta,\mu,\delta,\gamma\right)<\alpha_{*}-\alpha,
\]
the triple $\left(\alpha,\Lambda_{*},k_{\max*}\right)$ will be a
solution of inequalities \eqref{eq:the three subequations}. Indeed,
since \eqref{eq:reduced equations} are satisfied for the triple $\left(\alpha_{*},\Lambda_{*},k_{\max*}\right)$
and $k_{\max*}>0$, we have
\[
\begin{aligned}\alpha & =\alpha_{*}-\left(\alpha_{*}-\alpha\right)<k_{\max*}-\epsilon,\\
\alpha & =\alpha_{*}-\left(\alpha_{*}-\alpha\right)<\frac{1-2\Lambda_{*}-3k_{\max*}}{1+k_{\max*}}-\epsilon<\frac{1-2\Lambda_{*}-3k_{\max*}}{1+k_{\max*}}-\frac{\epsilon}{1+k_{\max*}}=\frac{1-2\Lambda_{*}-3k_{\max*}-\epsilon}{1+k_{\max*}},\\
\alpha & =\alpha_{*}-\left(\alpha_{*}-\alpha\right)<\frac{\Lambda_{*}}{1+k_{\max*}}-\epsilon<\frac{\Lambda_{*}}{1+k_{\max*}}-\frac{\epsilon}{1+k_{\max*}}=\frac{\Lambda_{*}-\epsilon}{1+k_{\max*}},
\end{aligned}
\]
where we have omitted the dependencies of $\epsilon$ for simplicity.
By the continuity of $\epsilon$, smallness of $\epsilon$ translates
into smallness of $\zeta$, $\mu$ and $\delta$ and to closeness
of $\gamma$ to $1$. To continue, we shall check that, provided that
$\epsilon$ is sufficiently small, the triple $\left(\alpha,\Lambda_{*},k_{\max*}\right)$
is a solution of equations 1 to 12. By construction, equations 1 (which
implies 9), 5 (which is the same as 10) and 7 (which is the same as
8) are satisfied. Looking at equation 3, since we are taking $k_{\max}=k_{\max*}$
and $k_{\max*}<1$, taking $\epsilon$ small enough guarantees that
equation 3 holds. As for equation 4, we have
\[
1-k_{\max*}-\Lambda_{*}=1-\left(\sqrt{\frac{4}{3}}-1\right)-\sqrt{\frac{4}{3}}\left(\sqrt{\frac{4}{3}}-1\right)=2-\frac{4}{3}=\frac{2}{3}.
\]
Thus, as long as $\epsilon$ is small enough, equation 4 is satisfied.
Concerning equation 11, we compute
\[
-1+2\Lambda_{*}+3k_{\max*}=-1+\sqrt{\frac{4}{3}}\left(\sqrt{\frac{4}{3}}-1\right)+3\left(\sqrt{\frac{4}{3}}-1\right)=-4+\frac{4}{3}+2\,\sqrt{\frac{4}{3}}=-5+\left(1+\sqrt{\frac{4}{3}}\right)^{2}.
\]
As $\sqrt{x}\le\frac{x}{2}+\frac{1}{2}$ $\forall x\ge0$, we deduce
\[
-1+2\Lambda+3k_{\max}\le-5+\left(1+\frac{2}{3}+\frac{1}{2}\right)^{2}=-5+\left(\frac{6+4+3}{6}\right)^{2}=-5+\frac{169}{36}=\frac{-180+169}{36}=-\frac{11}{36}<0.
\]
Consequently, taking $\epsilon$ sufficiently small, equation 11 is
also satisfied. Let us now turn to equation 6. We have
\[
\alpha_{*}\left(1-k_{\max*}\right)-\Lambda_{*}=\alpha_{*}\left(1-\alpha_{*}\right)-\alpha_{*}\left(1+\alpha_{*}\right)=-2\alpha_{*}^{2}<0
\]
As $\alpha<\alpha_{*}$ and $k_{\max*}<1$, we infer that
\[
\alpha\left(1-k_{\max*}\right)-\Lambda_{*}<\alpha_{*}\left(1-k_{\max*}\right)-\Lambda_{*}=-2\alpha_{*}^{2}<0
\]
and, consequently, the condition to apply equation 6 will not be satisfied.
In this way, equation 6 (which was the same as equation 12) imposes
no restriction. Only equation 2 remains. Notice that, as $\alpha<\alpha_{*}$
and $k_{\max*}>0$,
\[
\Lambda_{*}-\alpha k_{\max*}\ge\Lambda_{*}-\alpha_{*}k_{\max*}=\alpha_{*}\left(1+\alpha_{*}\right)-\alpha_{*}\alpha_{*}=\alpha_{*}=k_{\max*}.
\]
This means that the third argument of the $\min$ that appears in
equation 2 is never strictly smaller than the first one and, in this
manner, can be ignored. Let us study
\[
\left(1-\alpha\right)k_{\max*}+\alpha\Lambda_{*}=k_{\max*}+\left(\Lambda_{*}-k_{\max*}\right)\alpha.
\]
As
\[
k_{\max*}=\alpha_{*}<\alpha_{*}\left(1+\alpha_{*}\right)=\Lambda_{*},
\]
we deduce that $\left(1-\alpha\right)k_{\max*}+\alpha\Lambda_{*}$
is increasing in $\alpha$. Thereby,
\[
\left(1-\alpha\right)k_{\max*}+\alpha\Lambda_{*}\ge k_{\max*}
\]
and, as a consequence, the second argument of the $\min$ that appears
in equation 2 is never strictly smaller than the first one and, thus,
can also be ignored. In other words, equation 2 becomes
\[
\alpha<k_{\max*}-\epsilon\left(\zeta,\mu,\delta,\gamma\right).
\]
As $\alpha<\alpha_{*}=k_{\max*}$, this equation is satisfied as long
as $\epsilon$ is small enough. In conclusion, equations 1 to 12 are
all satisfied if $\epsilon$ is sufficiently small. As we have argued
before, the smallness of $\epsilon$ translates into closeness of
$\zeta$, $\mu$ and $\delta$ to zero and closeness of $\gamma$
to $1$.

To continue the argument, recall that $F_{\omega}^{\left(n\right)}=\sum_{m=1}^{n}f_{\omega}^{\left(m\right)}$,
$F_{\rho}^{\left(n\right)}=\sum_{m=1}^{n}f_{\rho}^{\left(m\right)}$
and bear in mind the decomposition of the forces given in equations
\eqref{eq:decomposition force density} and \eqref{eq:decomposition force vorticity}.
We want to show that $\lim_{n\to\infty}F_{\rho}^{\left(n\right)}\in C_{t}^{0}C_{x}^{1,\alpha}\left(\left[0,1\right]\times\mathbb{R}^{2}\right)$
and that $\lim_{n\to\infty}F_{\omega}^{\left(n\right)}\in C_{t}^{0}C_{x}^{\alpha}\left(\left[0,1\right]\times\mathbb{R}^{2}\right)$.
To do this, we will prove that $\left(F_{\rho}^{\left(n\right)}\right)_{n\in\mathbb{N}}$
and $\left(F_{\omega}^{\left(n\right)}\right)_{n\in\mathbb{N}}$ are
Cauchy sequences in $C_{t}^{0}C_{x}^{1,\alpha}\left(\left[0,1\right]\times\mathbb{R}^{2}\right)$
and $C_{t}^{0}C_{x}^{\alpha}\left(\left[0,1\right]\times\mathbb{R}^{2}\right)$,
respectively. The completeness of the aforementioned spaces will then
provide the desired result. First of all, notice that $\forall n\in\mathbb{N}$,
$F_{\rho}^{\left(n\right)}$ and $F_{\omega}^{\left(n\right)}$ are
$C^{\infty}$ in both space and time, since they only contain a finite
number of summands which are themselves compositions of smooth functions
in space and time. Thus, $\left(F_{\rho}^{\left(n\right)}\right)_{n\in\mathbb{N}}$
and $\left(F_{\omega}^{\left(n\right)}\right)_{n\in\mathbb{N}}$ are
contained in the desired spaces. By the telescopic nature of $F_{\rho}^{\left(n\right)}$
and $F_{\omega}^{\left(n\right)}$, $\forall l,m\in\mathbb{N}$ with
$l\ge m$ we have
\begin{equation}
\begin{aligned}\left|\left|F_{\rho}^{\left(l\right)}-F_{\rho}^{\left(m\right)}\right|\right|_{C_{t}^{0}C_{x}^{1,\alpha}\left(\left[0,1\right]\times\mathbb{R}^{2}\right)} & =\left|\left|\sum_{n=m+1}^{l}f_{\rho}^{\left(n\right)}\right|\right|_{C_{t}^{0}C_{x}^{1,\alpha}\left(\left[0,1\right]\times\mathbb{R}^{2}\right)}\le\sum_{n=m+1}^{\infty}\sup_{t\in\left[0,1\right]}\left|\left|f_{\rho}^{\left(n\right)}\left(t,\cdot\right)\right|\right|_{C^{1,\alpha}\left(\mathbb{R}^{2}\right)},\\
\left|\left|F_{\omega}^{\left(l\right)}-F_{\omega}^{\left(m\right)}\right|\right|_{C_{t}^{0}C_{x}^{\alpha}\left(\left[0,1\right]\times\mathbb{R}^{2}\right)} & =\left|\left|\sum_{n=m+1}^{l}f_{\omega}^{\left(n\right)}\right|\right|_{C_{t}^{0}C_{x}^{\alpha}\left(\left[0,1\right]\times\mathbb{R}^{2}\right)}\le\sum_{n=m+1}^{\infty}\sup_{t\in\left[0,1\right]}\left|\left|f_{\omega}^{\left(n\right)}\left(t,\cdot\right)\right|\right|_{C^{\alpha}\left(\mathbb{R}^{2}\right)}.
\end{aligned}
\label{eq:bounds Cauchy sequence}
\end{equation}
Now, notice that, if we sum the bounds obtained in Propositions \ref{prop:bound density time derivative},
\ref{prop:bound density pure quadratic term}, \ref{prop:bound density new transport of old density},
\ref{prop:bound density transport}, \ref{prop:bound vorticity time factor},
\ref{prop:bound vorticity pure quadratic term}, \ref{prop:bound vorticity new transport of old vorticity}
and \ref{prop:bound vorticity transport} (which are all uniform in
time), we can find a bound for $\sup_{t\in\left[0,1\right]}\left|\left|f_{\rho}^{\left(n\right)}\left(t,\cdot\right)\right|\right|_{C^{1,\alpha}\left(\mathbb{R}^{2}\right)}$
and $\sup_{t\in\left[0,1\right]}\left|\left|f_{\omega}^{\left(n\right)}\left(t,\cdot\right)\right|\right|_{C^{\alpha}\left(\mathbb{R}^{2}\right)}$.
For $\left(F_{\rho}^{\left(n\right)}\right)_{n\in\mathbb{N}}$ and
$\left(F_{\omega}^{\left(n\right)}\right)_{n\in\mathbb{N}}$ to be
Cauchy sequences, we need these bounds to be summable in $n\in\mathbb{N}$.
Moreover, observe that all the mentioned bounds are superexponential
in $n\in\mathbb{N}$. Consequently, being decreasing in $n\in\mathbb{N}$
is equivalent to being summable in $n\in\mathbb{N}$. Nevertheless,
each one of these bounds requires $n\in\mathbb{N}$ to be large enough
to be applicable. Concretely, the most restrictive bound is the one
given by Proposition \ref{prop:bound density time derivative}, where
the minimum value of $n\in\mathbb{N}$ depends on $\mu$. In this
way, to prove the result, first, we need to fix the values of $\mu,\delta,\gamma,\zeta$
so that $\epsilon$ is small enough in order for equations 1 to 12
to be satisfied. Once these values are fixed, there is an $n_{0}\left(\mu\right)\in\mathbb{N}$
starting from which all bounds given in Propositions \ref{prop:bound density time derivative},
\ref{prop:bound density pure quadratic term}, \ref{prop:bound density new transport of old density},
\ref{prop:bound density transport}, \ref{prop:bound vorticity time factor},
\ref{prop:bound vorticity pure quadratic term}, \ref{prop:bound vorticity new transport of old vorticity}
and \ref{prop:bound vorticity transport} become valid. As a finite
number of terms does not affect summability and all these bounds are
superexponentially decreasing in $n\in\mathbb{N}$ thanks to the fulfillment
of equations 1 to 12, we conclude that the series $\sum_{n\in\mathbb{N}}\sup_{t\in\left[0,1\right]}\left|\left|f_{\rho}^{\left(n\right)}\left(t,\cdot\right)\right|\right|_{C^{1,\alpha}\left(\mathbb{R}^{2}\right)}$
and $\sum_{n\in\mathbb{N}}\sup_{t\in\left[0,1\right]}\left|\left|f_{\omega}^{\left(n\right)}\left(t,\cdot\right)\right|\right|_{C^{\alpha}\left(\mathbb{R}^{2}\right)}$
are convergent. In this way, the bounds given in \eqref{eq:bounds Cauchy sequence}
are tails of convergent series and, consequently, tend to zero as
$m\to\infty$. Hence, $\left(F_{\rho}^{\left(n\right)}\right)_{n\in\mathbb{N}}$
and $\left(F_{\omega}^{\left(n\right)}\right)_{n\in\mathbb{N}}$ are
Cauchy sequences and the completeness of the spaces $C_{t}^{0}C_{x}^{1,\alpha}\left(\left[0,1\right]\times\mathbb{R}^{2}\right)$
and $C_{t}^{0}C_{x}^{\alpha}\left(\left[0,1\right]\times\mathbb{R}^{2}\right)$
provides the result.
\end{proof}

\subsection{Compact support}
\begin{prop}
\label{prop:compact support}Taking $\Lambda=\Lambda_{*}$ and $k_{\max}=k_{\max*}$,
where $\Lambda_{*}$ and $k_{\max*}$ are as given in Proposition
\ref{prop:maximal regularity}, as long as $C$ is big enough and
$\mu,\delta$ are small enough (let us say $C\ge\Upsilon_{1}\left(\delta,\mu\right)$
and $\mu,\delta\le\Upsilon_{2}$), there exists $K\subseteq\mathbb{R}^{2}$
compact such that
\[
\mathrm{supp}\psi^{\left(n\right)}\left(t,\cdot\right)\subseteq K\quad\forall t\in\left[0,1\right]\;\land\;\forall n\in\mathbb{N}.
\]
In other words, our solution is compactly supported uniformly in time.
The same is true for $\rho^{\left(n\right)}\left(t,\cdot\right)$.
\end{prop}
\begin{proof}
By Choice \ref{choice:psin}, we have
\begin{equation}
\begin{aligned}\psi^{\left(n\right)}\left(t,x\right) & =B_{n}\left(t\right)\varphi\left(\lambda_{n}\left(\phi^{\left(n\right)}\right)_{1}^{-1}\left(t,x\right)\right)\varphi\left(\lambda_{n}\left(\phi^{\left(n\right)}\right)_{2}^{-1}\left(t,x\right)\right)\cdot\\
 & \quad\cdot\sin\left(\left(\phi^{\left(n\right)}\right)_{1}^{-1}\left(t,x\right)\right)\sin\left(\left(\phi^{\left(n\right)}\right)_{2}^{-1}\left(t,x\right)\right).
\end{aligned}
\label{eq:compact support explicit psin}
\end{equation}
By Choice \ref{choice:amplitude density}, we know that $B_{n}\left(t\right)=0$
$\forall t\in\left[0,t_{n}\right]$. Thus,
\begin{equation}
\mathrm{supp}\psi^{\left(n\right)}\left(t,\cdot\right)=\emptyset\quad\forall t\in\left[0,t_{n}\right]\;\land\;\forall n\in\mathbb{N}.\label{eq:compact support empty support in certain time intervals}
\end{equation}
This means that we may restrict our study of the support of $\psi^{\left(n\right)}$
to the time interval $\left[t_{n},1\right]$. 

Bearing equation \eqref{eq:phin inverse} in mind, we see that the
arguments of $\varphi$ in equation \eqref{eq:compact support explicit psin}
are
\begin{equation}
\lambda_{n}a_{n}\left(t\right)\left(x_{1}-\phi_{1}^{\left(n\right)}\left(t,0\right)\right)\quad\text{and}\quad\lambda_{n}b_{n}\left(t\right)\left(x_{2}-\phi_{2}^{\left(n\right)}\left(t,0\right)\right).\label{eq:compact support arguments of phi}
\end{equation}
By Choices \ref{choice:lambdan} and \ref{choice:anbn},
\[
\lambda_{n}a_{n}\left(t\right)=C^{\left[-\Lambda+1-k_{n}\left(t\right)\right]\left(\frac{1}{1-\gamma}\right)^{n}},\quad\lambda_{n}b_{n}\left(t\right)=C^{\left[-\Lambda+1+k_{n}\left(t\right)\right]\left(\frac{1}{1-\gamma}\right)^{n}}.
\]
Next, by Corollary \ref{cor:bound for an(t), bn(t)}, which we may
apply provided that $C$ is large enough (let us say $C\ge\Upsilon_{1}\left(\delta,\mu\right)$),
we deduce that
\[
\lambda_{n}a_{n}\left(t\right)\ge C^{\left[-\Lambda+1-\overline{k}_{n}\left(t\right)-\mu\right]\left(\frac{1}{1-\gamma}\right)^{n}},\quad\lambda_{n}b_{n}\left(t\right)\ge C^{\left[-\Lambda+1+\overline{k}_{n}\left(t\right)-\mu\right]\left(\frac{1}{1-\gamma}\right)^{n}}.
\]
As $\overline{k}_{n}\left(t\right)\ge0$ $\forall t\in\left[t_{n},1\right]$
by Corollary \ref{cor:ideal kn non negative} and $\overline{k}_{n}\left(t\right)\le k_{\max}$
$\forall t\in\left[t_{n},1\right]$ by Choice \ref{choice:ideal kn},
we arrive to
\begin{equation}
\lambda_{n}a_{n}\left(t\right)\ge C^{\left[-\Lambda+1-k_{\max}-\mu\right]\left(\frac{1}{1-\gamma}\right)^{n}},\quad\lambda_{n}b_{n}\left(t\right)\ge C^{\left[-\Lambda+1-\mu\right]\left(\frac{1}{1-\gamma}\right)^{n}}.\label{eq:compact support lambdaan lambabn}
\end{equation}
Since we have taken $\Lambda=\Lambda_{*}$ and $k_{\max}=k_{\max*}$,
by Proposition \ref{prop:maximal regularity}, we have
\[
-\Lambda+1-k_{\max}=-\,\sqrt{\frac{4}{3}}\left(\sqrt{\frac{4}{3}}-1\right)+1-\left(\sqrt{\frac{4}{3}}-1\right)=2-\frac{4}{3}=\frac{2}{3}.
\]
As a consequence, we infer that
\[
-\Lambda+1-\mu\ge-\Lambda+1-k_{\max}-\mu=\frac{2}{3}-\mu.
\]
Provided that $\mu<\frac{2}{3}\coloneqq\Upsilon_{2}$, we can guarantee
that the exponents of \eqref{eq:compact support lambdaan lambabn}
are both positive and, since $C>2$ by Choice \ref{choice:min requirements on C gamma and kmax},
also that $\lambda_{n}a_{n}\left(t\right)\ge1$ and $\lambda_{n}b_{n}\left(t\right)\ge1$
$\forall t\in\left[t_{n},1\right]$ and $\forall n\in\mathbb{N}$.
Coming back to equation \eqref{eq:compact support arguments of phi},
in view of Choice \ref{choice:psin}, this means that
\[
\mathrm{supp}\psi^{\left(n\right)}\left(t,\cdot\right)\subseteq\left[\phi_{1}^{\left(n\right)}\left(t,0\right)-16\pi,\phi_{1}^{\left(n\right)}\left(t,0\right)+16\pi\right]\times\left[\phi_{2}^{\left(n\right)}\left(t,0\right)-16\pi,\phi_{2}^{\left(n\right)}\left(t,0\right)+16\pi\right]\quad\forall t\in\left[t_{n},1\right].
\]
Since $\phi_{2}^{\left(n\right)}\left(t,0\right)\equiv0$ by Choice
\ref{choice:anbncn}, we easily obtain that
\begin{equation}
\mathrm{supp}\psi^{\left(n\right)}\left(t,\cdot\right)\subseteq\left[\phi_{1}^{\left(n\right)}\left(t,0\right)-16\pi,\phi_{1}^{\left(n\right)}\left(t,0\right)+16\pi\right]\times\left[-16\pi,16\pi\right]\quad\forall t\in\left[t_{n},1\right].\label{eq:compact support first bound support}
\end{equation}
What can we say about $\phi_{1}^{\left(n\right)}\left(t,0\right)$?
Well, by Proposition \ref{prop:layer center never leafs cutoff area},
which we can make use of provided that $C$ is large enough (let us
say $C\ge\Upsilon_{2}\left(\delta\right)$), we obtain that 
\[
\left|\phi_{1}^{\left(n\right)}\left(t,0\right)-\phi_{1}^{\left(n-1\right)}\left(t,0\right)\right|\le\frac{8\pi}{a_{n-1}\left(t\right)}\quad\forall t\in\left[t_{n},1\right].
\]
Using Proposition \ref{prop:time convergence} with $\beta=\frac{1}{2}$,
which we may do as long as $C$ is big enough (let us say $C\ge\Upsilon_{3}\left(\delta\right)$),
we can write
\[
\left|\phi_{1}^{\left(n\right)}\left(t,0\right)-\phi_{1}^{\left(n-1\right)}\left(t,0\right)\right|\lesssim\frac{1}{a_{n-1}\left(1\right)}\quad\forall t\in\left[t_{n},1\right].
\]
By Choices \ref{choice:anbn} and \ref{choice:time picture},
\begin{equation}
\left|\phi_{1}^{\left(n\right)}\left(t,0\right)-\phi_{1}^{\left(n-1\right)}\left(t,0\right)\right|\lesssim C^{-\left(\frac{1}{1-\gamma}\right)^{n-1}}\quad\forall t\in\left[t_{n},1\right].\label{eq:compact support constant closeness phi1}
\end{equation}
Therefore, we infer that $\forall n\in\mathbb{N}$,
\begin{equation}
\begin{aligned}\left|\phi_{1}^{\left(n\right)}\left(t,0\right)-\phi_{1}^{\left(1\right)}\left(t,0\right)\right| & \leq\sum_{m=2}^{n}\left|\phi_{1}^{\left(m\right)}\left(t,0\right)-\phi_{1}^{\left(m-1\right)}\left(t,0\right)\right|\le\sum_{m\in\mathbb{N}\setminus\left\{ 1\right\} }\left|\phi_{1}^{\left(m\right)}\left(t,0\right)-\phi_{1}^{\left(m-1\right)}\left(t,0\right)\right|\lesssim\\
 & \lesssim\sum_{m\in\mathbb{N}}C^{-\left(\frac{1}{1-\gamma}\right)^{m}}\quad\forall t\in\left[t_{n},1\right].
\end{aligned}
\label{eq:compact support closness phi1}
\end{equation}
As $\mathrm{e}^{x}\ge x^{2}$ $\forall x\in\left[0,\infty\right)$
by Lemma \ref{lem:exponential x2} and, equivalently, $\mathrm{e}^{-x}\le\frac{1}{x^{2}}$
$\forall x\in\left[0,\infty\right)$ , we have
\[
\begin{aligned}\sum_{m\in\mathbb{N}}C^{-\left(\frac{1}{1-\gamma}\right)^{m}} & \leq\sum_{m=0}^{\infty}\exp\left(-\left(\frac{1}{1-\gamma}\right)^{m}\ln\left(C\right)\right)\leq\sum_{m=0}^{\infty}\frac{\left(1-\gamma\right)^{2m}}{\ln\left(C\right)^{2}}=\\
 & =\frac{1}{\ln\left(C\right)^{2}}\frac{1}{1-\left(1-\gamma\right)^{2}}=\frac{1}{\ln\left(C\right)^{2}}\frac{1}{\gamma\left(2-\gamma\right)}.
\end{aligned}
\]
As $\gamma\ge\frac{1}{2}$, $\gamma<1$ and $C>2$ by Choice \ref{choice:min requirements on C gamma and kmax},
we deduce that
\[
\sum_{m\in\mathbb{N}}C^{-\left(\frac{1}{1-\gamma}\right)^{m}}\leq\frac{2}{\ln\left(C\right)^{2}}\le\frac{2}{\ln\left(2\right)^{2}}\lesssim1.
\]
Hence, bearing equations \eqref{eq:compact support closness phi1}
and \eqref{eq:compact support first bound support} in mind, there
must be $L>0$ (independent of $n\in\mathbb{N}$) such that
\[
\mathrm{supp}\psi^{\left(n\right)}\left(t,\cdot\right)\subseteq\left[\phi_{1}^{\left(1\right)}\left(t,0\right)-L-16\pi,\phi_{1}^{\left(1\right)}\left(t,0\right)+L+16\pi\right]\times\left[-16\pi,16\pi\right]\quad\forall t\in\left[t_{n},1\right].
\]
Now, by Choice \ref{choice:anbncn}, we have $\frac{\partial\phi_{1}^{\left(1\right)}}{\partial t}\left(t,0\right)=0$
and, in this way, we conclude that
\[
\mathrm{supp}\psi^{\left(n\right)}\left(t,\cdot\right)\subseteq\left[\phi_{1}^{\left(1\right)}\left(1,0\right)-L-16\pi,\phi_{1}^{\left(1\right)}\left(1,0\right)+L+16\pi\right]\times\left[-16\pi,16\pi\right],
\]
being this true $\forall t\in\left[t_{n},1\right]$ and $\forall n\in\mathbb{N}$.
Combining this with equation \eqref{eq:compact support empty support in certain time intervals}
provides the result.

Observe that the argument we have constructed has not used any property
of $\psi^{\left(n\right)}$ apart from where its cutoff factors vanish.
Since, by Choice \ref{choice:density}, the density $\rho^{\left(n\right)}\left(t,\cdot\right)$
has exactly the same cutoff factors, the argument is also applicable
to the density.
\end{proof}

\subsection{Proof of the main Theorem}
\begin{thm}
Let $\alpha\in\left(0,\alpha_{*}\right)$, where $\alpha_{*}=\sqrt{\frac{4}{3}}-1$.
There are classical solutions $\left(u,\rho\right)$ in $\left[0,1\right)\times\mathbb{R}^{2}$
of the forced Boussinesq system \eqref{eq:Boussinesq system velocity formulation}
that satisfy:
\begin{enumerate}
\item $\forall\varepsilon>0$
\[
u\in C_{t}^{\infty}C_{x,c}^{\infty}\left(\left[0,1-\varepsilon\right]\times\mathbb{R}^{2}\right),\quad\rho\in C_{t}^{\infty}C_{x,c}^{\infty}\left(\left[0,1-\varepsilon\right]\times\mathbb{R}^{2}\right).
\]
\item $\forall\varepsilon>0$
\[
f_{\omega}\in C_{t}^{\infty}C_{x,c}^{\infty}\left(\left[0,1-\varepsilon\right]\times\mathbb{R}^{2}\right),\quad f_{\rho}\in C_{t}^{\infty}C_{x,c}^{\infty}\left(\left[0,1-\varepsilon\right]\times\mathbb{R}^{2}\right).
\]
\item ~
\[
f_{\omega}\in C_{t}^{0}C_{x,c}^{\alpha}\left(\left[0,1\right]\times\mathbb{R}^{2}\right),\quad f_{\rho}\in C_{t}^{0}C_{x,c}^{1,\alpha}\left(\left[0,1\right]\times\mathbb{R}^{2}\right).
\]
In particular, $\left|\left|f_{\omega}\left(t,\cdot\right)\right|\right|_{C^{\alpha}\left(\mathbb{R}^{2}\right)}$
and $\left|\left|f_{\rho}\left(t,\cdot\right)\right|\right|_{C^{1,\alpha}\left(\mathbb{R}^{2}\right)}$
are uniformly bounded in $t\in\left[0,1\right]$.
\item $f_{\omega}\left(t,x\right)$ is odd in $x_{2}$ $\forall t\in\left[0,1\right]$.
\item There is a finite-time singularity at $t=1$, i.e.,
\[
\lim_{T\to1^{-}}\int_{0}^{T}\left|\left|\nabla\rho\left(t,\cdot\right)\right|\right|_{L^{\infty}\left(\mathbb{R}^{2};\mathbb{R}^{2}\right)}\mathrm{d}t=\infty.
\]
(See \cite{Chae Kim Nam} for the blow-up criterion).
\end{enumerate}
\end{thm}
\begin{proof}
We shall take the stream function as indicated in Choice \ref{choice:psin},
i.e.,
\[
\psi\left(t,x\right)=\sum_{n=1}^{\infty}\psi^{\left(n\right)}\left(t,x\right),\quad\widetilde{\psi^{\left(n\right)}}^{n}\left(t,x\right)=B_{n}\left(t\right)\varphi\left(\lambda_{n}x_{1}\right)\varphi\left(\lambda_{n}x_{2}\right)\sin\left(x_{1}\right)\sin\left(x_{2}\right),
\]
where
\begin{enumerate}
\item $\varphi\in C_{c}^{\infty}\left(\mathbb{R}\right)$ is any compactly
supported even smooth function that satisfies $\varphi\equiv1$ in
the interval $\left[-8\pi,8\pi\right]$ and $\varphi\equiv0$ in $\mathbb{R}\setminus\left[-16\pi,16\pi\right]$.
\item $\lambda_{n}=C^{-\Lambda\left(\frac{1}{1-\gamma}\right)^{n}}$ as
given by Choice \ref{choice:lambdan}, where $\Lambda=\Lambda_{*}=\sqrt{\frac{4}{3}}\left(\sqrt{\frac{4}{3}}-1\right)$,
as required in Proposition \ref{prop:maximal regularity} and $\gamma\in\left(0,1\right)$
is to be chosen later.
\item $\widetilde{\psi^{\left(n\right)}}^{n}\left(t,x\right)=\psi^{\left(n\right)}\left(t,\phi^{\left(n\right)}\left(t,x\right)\right)$,
where $\phi^{\left(n\right)}\left(t,x\right)$ is the change of variables
given in Choice \ref{choice:phin}:
\[
\phi^{\left(n\right)}\left(t,x\right)=\phi^{\left(n\right)}\left(t,0\right)+\left(\frac{1}{a_{n}\left(t\right)}x_{1},\frac{1}{b_{n}\left(t\right)}x_{2}\right).
\]
\item $a_{n}\left(t\right)$, $b_{n}\left(t\right)$, $B_{n}\left(t\right)$
and $\phi^{\left(n\right)}\left(t,0\right)$ evolve in time according
to Choice \ref{choice:anbncn} and Proposition \ref{prop:relation between anbn and jacobian},
i.e., $\forall t\in\left[t_{n},1\right]$,
\[
\begin{alignedat}{1}\phi_{2}^{\left(n\right)}\left(t,0\right) & =0,\\
\frac{\partial\phi_{1}^{\left(n\right)}}{\partial t}\left(t,0\right) & =\sum_{m=1}^{n-1}B_{m}\left(t\right)b_{m}\left(t\right)\sin\left(a_{m}\left(t\right)\left(\phi_{1}^{\left(n\right)}\left(t,0\right)-\phi_{1}^{\left(m\right)}\left(t,0\right)\right)\right),\\
\frac{\mathrm{d}}{\mathrm{d}t}\left(\ln\left(b_{n}\left(t\right)\right)\right) & =\sum_{m=1}^{n-1}B_{m}\left(t\right)a_{m}\left(t\right)b_{m}\left(t\right)\cos\left(a_{m}\left(t\right)\left(\phi_{1}^{\left(n\right)}\left(t,0\right)-\phi_{1}^{\left(m\right)}\left(t,0\right)\right)\right),\\
\frac{\mathrm{d}}{\mathrm{d}t}\left(a_{n}\left(t\right)b_{n}\left(t\right)\right) & =0.
\end{alignedat}
\]
Thanks to Proposition \ref{prop:layer center never leafs cutoff area},
we know that the assumption of Proposition \ref{prop:relation between anbn and jacobian}
will actually be superfluous under some conditions.
\item The time scales $\left(t_{n}\right)_{n\in\mathbb{N}}$ are given by
Choices \ref{choice:time (initial)} and \ref{choice:Mn}:
\[
1-t_{n}=\frac{1}{Y}C^{-\delta\left(\frac{1}{1-\gamma}\right)^{n-1}}\mathrm{arccosh}\left(C^{k_{\max}\left(\frac{1}{1-\gamma}\right)^{n}}\right),
\]
where $Y>0$ and $\delta>0$ are parameters whose value will be chosen
later. We set $k_{\max}=k_{\max*}=\sqrt{\frac{4}{3}}-1$ as required
by Proposition \ref{prop:maximal regularity}. 
\item The initial conditions for $a_{n}\left(t\right)$ and $b_{n}\left(t\right)$
are specified through Choices \ref{choice:anbn} and \ref{choice:time picture}:
\[
b_{n}\left(t\right)=C^{\left(1+k_{n}\left(t\right)\right)\left(\frac{1}{1-\gamma}\right)^{n}},\quad a_{n}\left(t\right)=C^{\left(1-k_{n}\left(t\right)\right)\left(\frac{1}{1-\gamma}\right)^{n}},\quad k_{n}\left(1\right)=0,
\]
were $C$ is a parameter whose value will be chosen later. The initial
conditions $\phi^{\left(n\right)}\left(t_{n},0\right)=c^{\left(n\right)}$
are given indirectly by Proposition \ref{prop:relation between anbn and jacobian}
and Remark \ref{rem:initial condition}:
\[
\begin{aligned}a_{n-1}\left(1\right)\Xi^{\left(n\right)}\left(t_{n}\right) & =\arcsin\left(C^{-k_{\max}\left(\frac{1}{1-\gamma}\right)^{n}}\right),\\
\Xi^{\left(n\right)}\left(t\right) & =\phi_{1}^{\left(n\right)}\left(t,0\right)-\phi_{1}^{\left(n-1\right)}\left(t,0\right),\\
\phi_{2}^{\left(n\right)}\left(t,0\right) & \equiv0,
\end{aligned}
\]
where we take $\phi^{\left(0\right)}\left(t,0\right)\equiv0$.
\item $B_{n}\left(t\right)$ is given by Choice \ref{choice:amplitude density},
equation \eqref{eq:relation Mn and zn} and Choice \ref{choice:time picture}:
\[
B_{n}\left(t\right)=\frac{2M_{n}}{a_{n}\left(t\right)^{2}+b_{n}\left(t\right)^{2}}\frac{\int_{t_{n}}^{t}h^{\left(n\right)}\left(s\right)b_{n}\left(s\right)\mathrm{d}s}{\int_{t_{n}}^{1}h^{\left(n\right)}\left(s\right)b_{n}\left(s\right)\mathrm{d}s}.
\]
Here, $h^{\left(n\right)}\left(s\right)$ is a function given by Choice
\ref{choice:time picture} and equation \eqref{eq:def dhndt second way},
where we choose $\varepsilon=\frac{1}{10}$ and $\zeta$ will be associated
a value later.
\item $M_{n}$ is given by Choice \ref{choice:Mn}:
\[
M_{n}\coloneqq YC^{\delta\left(\frac{1}{1-\gamma}\right)^{n}},
\]
where $Y$ and $\delta$ will be given values later.
\end{enumerate}
Notice that, actually, because of how $B_{n}\left(t\right)$ and $t_{n}$
have been defined, for every $t\in\left[0,1\right)$, only a finite
number of summands in $\Psi^{\left(\infty\right)}\left(t,x\right)=\sum_{n=1}^{\infty}\psi^{\left(n\right)}\left(t,x\right)$
are non zero. Thereby, for $t<1$, this is always a finite sum, i.e.,
the intrinsic limit associated to the series is fictitious.

Under the choice we have made for the stream function, the velocity
$u$ is given by
\begin{equation}
\begin{aligned}u\left(t,x\right) & =\sum_{n\in\mathbb{N}}u^{\left(n\right)}\left(t,x\right),\\
\widetilde{u^{\left(n\right)}}^{n}\left(t,x\right) & =B_{n}\left(t\right)\varphi\left(\lambda_{n}x_{1}\right)\varphi\left(\lambda_{n}x_{2}\right)\left(\begin{matrix}b_{n}\left(t\right)\sin\left(x_{1}\right)\cos\left(x_{2}\right)\\
-a_{n}\left(t\right)\cos\left(x_{1}\right)\sin\left(x_{2}\right)
\end{matrix}\right)+\\
 & \quad+\lambda_{n}B_{n}\left(t\right)\left(\begin{matrix}b_{n}\left(t\right)\varphi\left(\lambda_{n}x_{1}\right)\varphi'\left(\lambda_{n}x_{2}\right)\\
-a_{n}\left(t\right)\varphi'\left(\lambda_{n}x_{1}\right)\varphi\left(\lambda_{n}x_{2}\right)
\end{matrix}\right)\sin\left(x_{1}\right)\sin\left(x_{2}\right),
\end{aligned}
\label{eq:u final theorem}
\end{equation}
where $u^{\left(n\right)}\left(t,x\right)$ is as stated in Proposition
\ref{prop:computations vorticity}. Moreover, as happened with the
stream function, $\forall t\in\left[0,1\right)$, the sum is actually
finite. On the other hand, we choose the density as specified by Choice
\ref{choice:density}:
\begin{equation}
\begin{aligned}\rho\left(t,x\right) & =\sum_{n\in\mathbb{N}}\rho^{\left(n\right)}\left(t,x\right),\\
\widetilde{\rho^{\left(n\right)}}^{n}\left(t,x\right) & =-\frac{1}{b_{n}\left(t\right)}\frac{\mathrm{d}}{\mathrm{d}t}\left[B_{n}\left(t\right)\left(a_{n}\left(t\right)^{2}+b_{n}\left(t\right)^{2}\right)\right]\varphi\left(\lambda_{n}x_{1}\right)\varphi\left(\lambda_{n}x_{2}\right)\sin\left(x_{1}\right)\cos\left(x_{2}\right).
\end{aligned}
\label{eq:rho final theorem}
\end{equation}
Here, $\forall t\in\left[0,1\right)$, because of Choice \ref{choice:no two simultanous densities}
(which is satisfied thanks to equation \eqref{eq:def dhndt second way}),
only one summand is non-zero, i.e., not only is the sum finite, it
actually encompasses one single element.

Next in line is to choose the value of the remaining parameters: $\gamma$,
$Y$, $\mu$, $\delta$, $C$ and $\zeta$. First of all, we choose
$\delta,\mu,\zeta$ and $\gamma$ so that the hypothesis of Proposition
\ref{prop:maximal regularity}, Proposition \ref{prop:compact support}
and Choice \ref{choice:min requirements on C gamma and kmax} are
satisfied. Once we have fixed $\delta,\mu,\zeta$ and $\gamma$, we
choose $C$ so that Choice \ref{choice:min requirements on C gamma and kmax}
and the hypothesis of Corollary \ref{cor:bound for an(t), bn(t)},
Propositions \ref{prop:time convergence}, \ref{prop:JI convergence to ideal},
\ref{prop:convergence kn to ideal model}, \ref{prop:layer center never leafs cutoff area},
\ref{prop:bound density time derivative}, \ref{prop:form of gradient of rho},
\ref{prop:bound density pure quadratic term}, \ref{prop:bounds for Taylor development of transport},
\ref{prop:bound density transport}, \ref{prop:bound vorticity time factor},
\ref{prop:bound vorticity pure quadratic term}, \ref{prop:bound vorticity new transport of old vorticity},
\ref{prop:bound vorticity transport}, \ref{prop:maximal regularity},
\ref{prop:compact support} and Lemmas \ref{lem:estimate sum superexponential},
\ref{lem:estimate of integral hn bn}, \ref{lem:time derivatives amplitudes vorticity}
are satisfied. This makes Proposition \ref{prop:maximal regularity}
true and ensures that the assumption presented in Proposition \ref{prop:relation between anbn and jacobian}
is superfluous. Lastly, we choose the value of $Y$ so that $t_{1}=0$
as required by Choice \ref{choice:time (initial)}. With all the values
of the parameters settled, it is not difficult to see that all Choices
\ref{choice:time (initial)} to \ref{choice:min requirements on C gamma and kmax}
are fulfilled and, consequently, all the results that have appeared
in the paper until this moment are applicable to our solution.
\begin{enumerate}
\item By Proposition \ref{prop:compact support}, both the velocity and
the vorticity have compact support $\forall t\in\left[0,1\right]$.
And, concerning their regularity, for any fixed $t\in\left[0,1\right)$,
as only a finite number of terms appear in the sums that define $u$
and $\rho$ and each of those summands is but a composition of $C^{\infty}$
functions (see equations \eqref{eq:u final theorem} and \eqref{eq:rho final theorem})
in space and time, we conclude that $u$ and $\rho$ are also smooth
in space and time in the interval $t\in\left[0,1\right)$. This implies
the statement.
\item Thanks to Propositions \ref{prop:Bousinessq in layers} and \ref{prop:change of variables}
and Choice \ref{choice:phin}, the Boussinesq system \eqref{eq:Boussinesq system vorticity formulation}
is equivalent to \eqref{eq:Boussinesq 1 con tildes} and \eqref{eq:Boussinesq 2 con tildes},
where $\widetilde{\nabla}^{n}$ is defined by equation \eqref{eq:def tilde nabla}.
These, in turn, are equivalent to equations \eqref{eq:Boussinesq vorticity Taylor}
and \eqref{eq:Boussinesq density Taylor} thanks to Choice \ref{choice:anbncn}
and Proposition \ref{prop:relation between anbn and jacobian}. Thus,
the forces $f_{\rho}$ and $f_{\omega}$ that appear in \eqref{eq:Boussinesq system vorticity formulation}
are given through equations \eqref{eq:decomposition force density}
and \eqref{eq:decomposition force vorticity}, which are equivalent
to \eqref{eq:Boussinesq vorticity Taylor} and \eqref{eq:Boussinesq density Taylor}.
Looking at equations \eqref{eq:decomposition force density} and \eqref{eq:decomposition force vorticity},
we see that, for any fixed $t\in\left[0,1\right)$, $f_{\rho}$ and
$f_{\omega}$ are but a finite sum of functions which are $C^{\infty}$
compositions of $\rho^{\left(n\right)}$, $\omega^{\left(n\right)}$
and $u^{\left(n\right)}$ and their (spatial and temporal) derivatives.
We have already argued in point 1 that $\rho^{\left(n\right)}$ and
$u^{\left(n\right)}$ are $C^{\infty}$ in both space and time and,
since $\omega^{\left(n\right)}=\nabla\times u^{\left(n\right)}$,
we deduce that $\omega^{\left(n\right)}$ is smooth as well. Thereby,
we must conclude that $f_{\rho}$ and $f_{\omega}$ are also of class
$C^{\infty}$ in both space and time. Furthermore, as $\rho$, $\omega$
and $u$ are compactly supported in the same set uniformly in time
by Proposition \ref{prop:compact support}, the same argument we have
used to prove the regularity shows that $f_{\rho}$ and $f_{\omega}$
must also be compactly supported in the same set as $\rho$, $\omega$
and $u$.
\item This is true by Propositions \ref{prop:maximal regularity} and \ref{prop:compact support}.
\item Recall that
\[
f_{\omega}=\sum_{n\in\mathbb{N}}f_{\omega}^{\left(n\right)}.
\]
As the sum of odd functions is odd and the point-wise limit of odd
functions is odd, it is enough to show that $f_{\omega}^{\left(n\right)}$
is odd in $x_{2}$ $\forall t\in\left[0,1\right]$ $\forall n\in\mathbb{N}$.
Since $\phi_{2}^{\left(n\right)}\left(t,x\right)$ is a linear function
of $x_{2}$ (with no dependence on $x_{1}$) according to Choices
\ref{choice:phin} and \ref{choice:anbncn}, $f_{\omega}^{\left(n\right)}\left(t,x\right)$
is odd in $x_{2}$ $\forall t\in\left[0,1\right]$ if and only if
$\widetilde{f_{\omega}^{\left(n\right)}}^{n}\left(t,x\right)=f_{\omega}^{\left(n\right)}\left(t,\phi^{\left(n\right)}\left(t,x\right)\right)$
is odd in $x_{2}$ $\forall t\in\left[0,1\right]$. Now, according
to equation \eqref{eq:decomposition force vorticity},
\[
\begin{aligned} & \underbrace{\frac{\partial\widetilde{\omega^{\left(n\right)}}^{n}}{\partial t}\left(t,x\right)-\left(\begin{matrix}0 & 1\end{matrix}\right)\cdot\left[\widetilde{\nabla}^{n}\widetilde{\rho^{\left(n\right)}}^{n}\left(t,x\right)\right]}_{\text{time factor}}+\\
 & +\underbrace{\left(\widetilde{U^{\left(n-1\right)}}^{n}\left(t,x\right)-\widetilde{U^{\left(n-1\right)}}^{n}\left(t,0\right)-\mathrm{J}\widetilde{U^{\left(n-1\right)}}^{n}\left(t,0\right)\cdot\left(\begin{matrix}x_{1}\\
x_{2}
\end{matrix}\right)\right)\cdot\widetilde{\nabla}^{n}\widetilde{\omega^{\left(n\right)}}^{n}\left(t,x\right)}_{\text{transport term}}+\\
 & +\underbrace{\widetilde{u^{\left(n\right)}}^{n}\left(t,x\right)\cdot\widetilde{\nabla}^{n}\widetilde{\Omega^{\left(n-1\right)}}^{n}}_{\text{new transport of old vorticity}}+\underbrace{\widetilde{u^{\left(n\right)}}^{n}\left(t,x\right)\cdot\widetilde{\nabla}^{n}\widetilde{\omega^{\left(n\right)}}^{n}\left(t,x\right)}_{\text{pure quadratic term}}=\\
= & \widetilde{f_{\omega}^{\left(n\right)}}^{n}\left(t,x\right).
\end{aligned}
\]
Let us see that every summand is odd in $x_{2}$ $\forall t\in\left[0,1\right]$.
Before proceeding, notice that $\widetilde{f_{\omega}^{\left(n\right)}}^{n}\left(t,x\right)=0$
$\forall t\in\left[0,t_{n}\right]$. Hence, we may restrict the time
$t$ to the interval $t\in\left[t_{n},1\right]$.
\begin{enumerate}
\item Time factor. Looking at the expression for $\widetilde{\omega^{\left(n\right)}}^{n}$
given in Proposition \ref{prop:computations vorticity}, since $\varphi$
is an even function by Choice \ref{choice:psin}, it is easy to see
that $\widetilde{\omega^{\left(n\right)}}^{n}\left(t,x\right)$ is
odd in $x_{2}$. Consequently, $\frac{\partial\widetilde{\omega^{\left(n\right)}}^{n}}{\partial t}\left(t,x\right)$
is also odd in $x_{2}$. On the other hand, looking at Choice \ref{choice:density},
it is also immediate that $\left(0,1\right)\cdot\widetilde{\nabla}^{n}\widetilde{\rho^{\left(n\right)}}^{n}\left(t,x\right)$
is odd in $x_{2}$. Hence, the time factor is odd in $x_{2}$.
\item Pure quadratic term. First, recall the definition of the tilde gradient
(see equation \eqref{eq:def tilde nabla}). As $\widetilde{\omega^{\left(n\right)}}^{n}\left(t,x\right)$
is odd in $x_{2}$, $a_{n}\left(t\right)\frac{\partial\widetilde{\omega^{\left(n\right)}}^{n}}{\partial x_{1}}\left(t,x\right)$
will be also odd in $x_{2}$, whereas $b_{n}\left(t\right)\frac{\partial\widetilde{\omega^{\left(n\right)}}^{n}}{\partial x_{2}}$
will be even in $x_{2}$. On the other hand, looking at Proposition
\ref{prop:computations vorticity}, we deduce that $\widetilde{u_{1}^{\left(n\right)}}^{n}\left(t,x\right)$
is even in $x_{2}$ and $\widetilde{u_{2}^{\left(n\right)}}^{n}$
is odd in $x_{2}$. Thereby, the pure quadratic term is odd in $x_{2}$. 
\item New transport of old vorticity. We have already argued that $\widetilde{\omega^{\left(n\right)}}^{n}\left(t,x\right)$
is odd in $x_{2}$. As $\phi_{2}^{\left(n\right)}\left(t,x\right)$
was a linear function of $x_{2}$, $\omega^{\left(n\right)}\left(t,x\right)$
is also odd in $x_{2}$. As a consequence, so is $\Omega^{\left(n-1\right)}\left(t,x\right)=\sum_{m=1}^{n-1}\omega^{\left(m\right)}\left(t,x\right)$.
Again, since $\phi_{2}^{\left(n\right)}\left(t,x\right)$ is a linear
function of $x_{2}$. We infer that $\widetilde{\Omega^{\left(n-1\right)}}^{n}$
is also odd in $x_{2}$. In this manner, $a_{n}\left(t\right)\frac{\partial\widetilde{\Omega^{\left(n-1\right)}}^{n}}{\partial x_{1}}$
will be odd in $x_{2}$, whereas $b_{n}\left(t\right)\frac{\partial\widetilde{\Omega^{\left(n-1\right)}}^{n}}{\partial x_{2}}$
shall be even in $x_{2}$. Taking into account that $\widetilde{u_{1}^{\left(n\right)}}^{n}\left(t,x\right)$
is even in $x_{2}$ and $\widetilde{u_{2}^{\left(n\right)}}^{n}\left(t,x\right)$
is odd in $x_{2}$ (see Proposition \ref{prop:computations vorticity}),
we conclude that the new transport of old vorticity is odd in $x_{2}$. 
\item Transport term. As already explained, $a_{n}\left(t\right)\frac{\partial\widetilde{\omega^{\left(n\right)}}^{n}}{\partial x_{1}}\left(t,x\right)$
is odd in $x_{2}$, while $b_{n}\left(t\right)\frac{\partial\widetilde{\omega^{\left(n\right)}}^{n}}{\partial x_{2}}$
is even in $x_{2}$. Therefore, for the transport term to be odd in
$x_{2}$, we need 
\[
\left[\widetilde{U^{\left(n-1\right)}}^{n}\left(t,x\right)-\widetilde{U^{\left(n-1\right)}}^{n}\left(t,0\right)-\mathrm{J}\widetilde{U^{\left(n-1\right)}}^{n}\left(t,0\right)\cdot\left(\begin{matrix}x_{1}\\
x_{2}
\end{matrix}\right)\right]_{1}
\]
 to be even in $x_{2}$ and 
\[
\left[\widetilde{U^{\left(n-1\right)}}^{n}\left(t,x\right)-\widetilde{U^{\left(n-1\right)}}^{n}\left(t,0\right)-\mathrm{J}\widetilde{U^{\left(n-1\right)}}^{n}\left(t,0\right)\cdot\left(\begin{matrix}x_{1}\\
x_{2}
\end{matrix}\right)\right]_{2}
\]
 to be odd in $x_{2}$.
\begin{enumerate}
\item $\widetilde{U_{1}^{\left(n-1\right)}}^{n}\left(t,x\right)$ has to
be even in $x_{2}$ and $\widetilde{U_{2}^{\left(n-1\right)}}^{n}\left(t,x\right)$
needs to be odd in $x_{2}$. Let us see that this is the case. In
view of Proposition \ref{prop:computations vorticity}, we know that
$\widetilde{u_{1}^{\left(n\right)}}^{n}\left(t,x\right)$ is even
in $x_{2}$ and $\widetilde{u_{2}^{\left(n\right)}}^{n}\left(t,x\right)$
is odd in $x_{2}$. As $\phi_{2}^{\left(n\right)}\left(t,x\right)$
is a linear function of $x_{2}$, we deduce that $u_{1}^{\left(n\right)}\left(t,x\right)$
is even in $x_{2}$, whereas $u_{2}^{\left(n\right)}\left(t,x\right)$
is odd in $x_{2}$. Consequently, $U_{1}^{\left(n-1\right)}=\sum_{m=1}^{n-1}u_{1}^{\left(m\right)}$
is even in $x_{2}$ and $U_{2}^{\left(n-1\right)}=\sum_{m=1}^{n-1}u_{2}^{\left(m\right)}$
is odd in $x_{2}$. Again, since $\phi_{2}^{\left(n\right)}\left(t,x\right)$
is a linear function of $x_{2}$, we infer that $\widetilde{U_{1}^{\left(n-1\right)}}^{n}$
is even in $x_{2}$ while $\widetilde{U_{2}^{\left(n-1\right)}}^{n}$
is odd in $x_{2}$, as needed.
\item $\widetilde{U^{\left(n-1\right)}}^{n}\left(t,0\right)$ is constant
in space, so it is even in $x_{2}$. This is what we need for its
first component $\widetilde{U_{1}^{\left(n-1\right)}}^{n}\left(t,0\right)$,
but poses a problem for the second component $\widetilde{U_{2}^{\left(n-1\right)}}^{n}\left(t,0\right)$,
which has to be odd in $x_{2}$. This controversy might seem impossible
to overcome at first, but we have argued in point i that $\widetilde{U_{2}^{\left(n-1\right)}}^{n}\left(t,x\right)$
is odd in $x_{2}$, which means that $\widetilde{U_{2}^{\left(n-1\right)}}^{n}\left(t,0\right)=0$,
eliminating the problem.
\item Bear in mind that
\[
\begin{aligned}\left[\mathrm{J}\widetilde{U^{\left(n-1\right)}}^{n}\left(t,0\right)\cdot\left(\begin{matrix}x_{1}\\
x_{2}
\end{matrix}\right)\right]_{1} & =\frac{\partial\widetilde{U_{1}^{\left(n-1\right)}}^{n}}{\partial x_{1}}\left(t,0\right)x_{1}+\frac{\partial\widetilde{U_{1}^{\left(n-1\right)}}^{n}}{\partial x_{2}}\left(t,0\right)x_{2},\\
\left[\mathrm{J}\widetilde{U^{\left(n-1\right)}}^{n}\left(t,0\right)\cdot\left(\begin{matrix}x_{1}\\
x_{2}
\end{matrix}\right)\right]_{2} & =\frac{\partial\widetilde{U_{2}^{\left(n-1\right)}}^{n}}{\partial x_{1}}\left(t,0\right)x_{1}+\frac{\partial\widetilde{U_{2}^{\left(n-1\right)}}^{n}}{\partial x_{2}}\left(t,0\right)x_{2}.
\end{aligned}
\]
By equation \eqref{eq:jacobian at x=00003D0}, we have
\[
\frac{\partial\widetilde{U_{1}^{\left(n-1\right)}}^{n}}{\partial x_{2}}\left(t,0\right)=0=\frac{\partial\widetilde{U_{2}^{\left(n-1\right)}}^{n}}{\partial x_{1}}\left(t,0\right),
\]
which guarantees that
\[
\begin{aligned}\left[\mathrm{J}\widetilde{U^{\left(n-1\right)}}^{n}\left(t,0\right)\cdot\left(\begin{matrix}x_{1}\\
x_{2}
\end{matrix}\right)\right]_{1} & =\frac{\partial\widetilde{U_{1}^{\left(n-1\right)}}^{n}}{\partial x_{1}}\left(t,0\right)x_{1},\\
\left[\mathrm{J}\widetilde{U^{\left(n-1\right)}}^{n}\left(t,0\right)\cdot\left(\begin{matrix}x_{1}\\
x_{2}
\end{matrix}\right)\right]_{2} & =\frac{\partial\widetilde{U_{2}^{\left(n-1\right)}}^{n}}{\partial x_{2}}\left(t,0\right)x_{2}.
\end{aligned}
\]
In this way, $\left[\mathrm{J}\widetilde{U^{\left(n-1\right)}}^{n}\left(t,0\right)\cdot\left(\begin{matrix}x_{1}\\
x_{2}
\end{matrix}\right)\right]_{1}$ is independent of $x_{2}$, so it is even in $x_{2}$, as required.
Moreover, $\left[\mathrm{J}\widetilde{U^{\left(n-1\right)}}^{n}\left(t,0\right)\cdot\left(\begin{matrix}x_{1}\\
x_{2}
\end{matrix}\right)\right]_{2}$ is clearly odd in $x_{2}$, as needed.
\end{enumerate}
Thereby, we may ensure that the transport term is also odd in $x_{2}$,
which completes the proof that $f_{\omega}$ is odd in $x_{2}$.
\end{enumerate}
\item By Choice \ref{choice:density}, we have
\begin{equation}
\begin{aligned}\left(0,1\right)\cdot\widetilde{\nabla}^{n}\widetilde{\rho^{\left(n\right)}}^{n}\left(t,x\right) & =\frac{\mathrm{d}}{\mathrm{d}t}\left[B_{n}\left(t\right)\left(a_{n}\left(t\right)^{2}+b_{n}\left(t\right)^{2}\right)\right]\varphi\left(\lambda_{n}x_{1}\right)\varphi\left(\lambda_{n}x_{2}\right)\sin\left(x_{1}\right)\sin\left(x_{2}\right)+\\
 & \quad-\lambda_{n}\frac{\mathrm{d}}{\mathrm{d}t}\left[B_{n}\left(t\right)\left(a_{n}\left(t\right)^{2}+b_{n}\left(t\right)^{2}\right)\right]\varphi\left(\lambda_{n}x_{1}\right)\varphi'\left(\lambda_{n}x_{2}\right)\sin\left(x_{1}\right)\cos\left(x_{2}\right).
\end{aligned}
\label{eq:partial rho x2}
\end{equation}
By the definition of $\widetilde{\nabla}^{n}$ (see equation \eqref{eq:def tilde nabla}),
we know that
\[
\begin{aligned}\left(0,1\right)\cdot\widetilde{\nabla}^{n}\widetilde{\rho^{\left(n\right)}}^{n}\left(t,x\right) & =b_{n}\left(t\right)\frac{\partial\widetilde{\rho^{\left(n\right)}}^{n}}{\partial x_{2}}\left(t,x\right)=b_{n}\left(t\right)\frac{\partial}{\partial x_{2}}\left(\rho^{\left(n\right)}\left(t,\phi^{\left(n\right)}\left(t,x\right)\right)\right)=\\
 & =b_{n}\left(t\right)\nabla\rho^{\left(n\right)}\left(t,\phi^{\left(n\right)}\left(t,x\right)\right)\cdot\frac{\partial\phi^{\left(n\right)}}{\partial x_{2}}.
\end{aligned}
\]
Choice \ref{choice:phin} leads to
\[
\left(0,1\right)\cdot\widetilde{\nabla}^{n}\widetilde{\rho^{\left(n\right)}}^{n}\left(t,x\right)=\frac{\partial\rho^{\left(n\right)}}{\partial x_{2}}\left(t,\phi^{\left(n\right)}\left(t,x\right)\right).
\]
We wish to compute a lower bound of $\left|\left|\frac{\partial\rho^{\left(n\right)}}{\partial x_{2}}\left(t,\cdot\right)\right|\right|_{L^{\infty}\left(\mathbb{R}^{2}\right)}$.
As the $\left|\left|\cdot\right|\right|_{L^{\infty}\left(\mathbb{R}^{2}\right)}$
norm is invariant under diffemorphisms of the domain, $\lambda_{n}<1$
and $\left.\varphi\right|_{\left[-8\pi,8\pi\right]}\equiv1$ by Choice
\ref{choice:psin}, in view of equation \eqref{eq:partial rho x2},
we can guarantee that, when trying to compute $\left|\left|\frac{\partial\rho^{\left(n\right)}}{\partial x_{2}}\left(t,\cdot\right)\right|\right|_{L^{\infty}\left(\mathbb{R}^{2}\right)}$,
we would evaluate \eqref{eq:partial rho x2} at a certain point where
the sines are both $1$ and $\varphi$ is identically $1$ as well.
In that point, necessarily, the derivatives of $\varphi$ must vanish
and, in this way, only the first summand of \eqref{eq:partial rho x2}
will remain. In conclusion,
\begin{equation}
\left|\left|\frac{\partial\rho^{\left(n\right)}}{\partial x_{2}}\left(t,\cdot\right)\right|\right|_{L^{\infty}\left(\mathbb{R}^{2}\right)}\ge\frac{\mathrm{d}}{\mathrm{d}t}\left[B_{n}\left(t\right)\left(a_{n}\left(t\right)^{2}+b_{n}\left(t\right)^{2}\right)\right].\label{eq:lower bound partialrhonpartialx2}
\end{equation}
Observe that the right hand side is positive thanks of Choice \ref{choice:amplitude density}.

To continue, notice that
\[
\left|\left|\nabla\rho\left(t,\cdot\right)\right|\right|_{L^{\infty}\left(\mathbb{R}^{2};\mathbb{R}^{2}\right)}\ge\left|\left|\frac{\partial\rho}{\partial x_{2}}\left(t,\cdot\right)\right|\right|_{L^{\infty}\left(\mathbb{R}^{2}\right)}.
\]
Moreover, as we have argued before, the sum $\rho\left(t,x\right)=\sum_{n\in\mathbb{N}}\rho^{\left(n\right)}\left(t,x\right)$
is fictitious in the sense that, at any fixed time instant, there
is only one summand. Indeed, $\forall t\in\left[t_{n},t_{n+1}\right]$,
we have $\rho\left(t,x\right)=\rho^{\left(n\right)}\left(t,x\right)$.
Thereby, by equation \eqref{eq:lower bound partialrhonpartialx2},
\[
\begin{aligned}\left|\left|\nabla\rho\left(t,\cdot\right)\right|\right|_{L^{\infty}\left(\mathbb{R}^{2};\mathbb{R}^{2}\right)} & \ge\left|\left|\frac{\partial\rho}{\partial x_{2}}\left(t,\cdot\right)\right|\right|_{L^{\infty}\left(\mathbb{R}^{2}\right)}\ge\left|\left|\frac{\partial\rho^{\left(n\right)}}{\partial x_{2}}\left(t,\cdot\right)\right|\right|_{L^{\infty}\left(\mathbb{R}^{2}\right)}\geq\\
 & \geq\frac{\mathrm{d}}{\mathrm{d}t}\left[B_{n}\left(t\right)\left(a_{n}\left(t\right)^{2}+b_{n}\left(t\right)^{2}\right)\right]\quad\forall t\in\left[t_{n},t_{n+1}\right].
\end{aligned}
\]
Then, $\forall n\in\mathbb{N}$, since $B_{n}\left(t\right)=0$ $\forall t\in\left[0,t_{n}\right]$
by Choice \ref{choice:amplitude density}, we have
\[
\begin{aligned}\int_{0}^{1}\left|\left|\nabla\rho\left(t,\cdot\right)\right|\right|_{L^{\infty}\left(\mathbb{R}^{2};\mathbb{R}^{2}\right)}\mathrm{d}t & \geq\int_{0}^{1}\frac{\mathrm{d}}{\mathrm{d}t}\left[B_{n}\left(t\right)\left(a_{n}\left(t\right)^{2}+b_{n}\left(t\right)^{2}\right)\right]\mathrm{d}t=\\
 & \ge\int_{t_{n}}^{1}\frac{\mathrm{d}}{\mathrm{d}t}\left[B_{n}\left(t\right)\left(a_{n}\left(t\right)^{2}+b_{n}\left(t\right)^{2}\right)\right]\mathrm{d}t=B_{n}\left(1\right)\left(a_{n}\left(1\right)^{2}+b_{n}\left(1\right)^{2}\right).
\end{aligned}
\]
By Choice \ref{choice:amplitude density}, equation \eqref{eq:relation Mn and zn}
and Choice \ref{choice:time picture}, we obtain
\[
\int_{0}^{1}\left|\left|\nabla\rho\left(t,\cdot\right)\right|\right|_{L^{\infty}\left(\mathbb{R}^{2};\mathbb{R}^{2}\right)}\mathrm{d}t\ge2M_{n}\underbrace{\frac{\int_{t_{n}}^{1}h^{\left(n\right)}\left(s\right)b_{n}\left(s\right)\mathrm{d}s}{\int_{t_{n}}^{1}h^{\left(n\right)}\left(s\right)b_{n}\left(s\right)\mathrm{d}s}}_{=1}=2M_{n}.
\]
Lastly, Choice \eqref{choice:Mn} provides
\[
\int_{0}^{1}\left|\left|\nabla\rho\left(t,\cdot\right)\right|\right|_{L^{\infty}\left(\mathbb{R}^{2};\mathbb{R}^{2}\right)}\mathrm{d}t\ge2YC^{\delta\left(\frac{1}{1-\gamma}\right)^{n}}\xrightarrow[n\to\infty]{}\infty,
\]
which proves the result.
\end{enumerate}
\end{proof}
\newpage{}

\section{\label{sec:useful formulae}Useful formulae (we advise to print this
section)}
\begin{itemize}
\item Choice \ref{choice:phin} and equations \eqref{eq:phin inverse} and
\eqref{eq:jacobian inverse}:
\[
\begin{aligned}\phi^{\left(n\right)}\left(t,x\right) & =\phi^{\left(n\right)}\left(t,0\right)+\left(\frac{1}{a_{n}\left(t\right)}x_{1},\frac{1}{b_{n}\left(t\right)}x_{2}\right),\\
\left(\phi^{\left(n\right)}\right)^{-1}\left(t,x\right) & =\left(a_{n}\left(t\right)\left(x_{1}-\phi_{1}^{\left(n\right)}\left(t,0\right)\right),b_{n}\left(t\right)\left(x_{2}-\phi_{2}^{\left(n\right)}\left(t,0\right)\right)\right),\\
\mathrm{J}^{-1}\phi^{\left(n\right)}\left(t,x\right) & =\mathrm{J}\left(\phi^{\left(n\right)}\right)^{-1}\left(t,\phi^{\left(n\right)}\left(t,x\right)\right)=\left(\begin{matrix}a_{n}\left(t\right) & 0\\
0 & b_{n}\left(t\right)
\end{matrix}\right).
\end{aligned}
\]
\item Equation \eqref{eq:def tilde nabla}:
\[
\widetilde{\nabla}^{n}\coloneqq\left(a_{n}\left(t\right)\frac{\partial}{\partial x_{1}},b_{n}\left(t\right)\frac{\partial}{\partial x_{2}}\right).
\]
\item Proposition \ref{prop:computations vorticity}:
\[
\begin{aligned}\widetilde{u^{\left(n\right)}}^{n}\left(t,x\right) & =B_{n}\left(t\right)\varphi\left(\lambda_{n}x_{1}\right)\varphi\left(\lambda_{n}x_{2}\right)\left(\begin{matrix}b_{n}\left(t\right)\sin\left(x_{1}\right)\cos\left(x_{2}\right)\\
-a_{n}\left(t\right)\cos\left(x_{1}\right)\sin\left(x_{2}\right)
\end{matrix}\right)+\\
 & \quad+\lambda_{n}B_{n}\left(t\right)\left(\begin{matrix}b_{n}\left(t\right)\varphi\left(\lambda_{n}x_{1}\right)\varphi'\left(\lambda_{n}x_{2}\right)\\
-a_{n}\left(t\right)\varphi'\left(\lambda_{n}x_{1}\right)\varphi\left(\lambda_{n}x_{2}\right)
\end{matrix}\right)\sin\left(x_{1}\right)\sin\left(x_{2}\right).
\end{aligned}
\]
\[
\begin{aligned}\widetilde{\omega^{\left(n\right)}}^{n}\left(t,x\right) & =B_{n}\left(t\right)\left(a_{n}\left(t\right)^{2}+b_{n}\left(t\right)^{2}\right)\varphi\left(\lambda_{n}x_{1}\right)\varphi\left(\lambda_{n}x_{2}\right)\sin\left(x_{1}\right)\sin\left(x_{2}\right)+\\
 & \quad-\lambda_{n}^{2}B_{n}\left(t\right)a_{n}\left(t\right)^{2}\varphi''\left(\lambda_{n}x_{1}\right)\varphi\left(\lambda_{n}x_{2}\right)\sin\left(x_{1}\right)\sin\left(x_{2}\right)+\\
 & \quad-\lambda_{n}^{2}B_{n}\left(t\right)b_{n}\left(t\right)^{2}\varphi\left(\lambda_{n}x_{1}\right)\varphi''\left(\lambda_{n}x_{2}\right)\sin\left(x_{1}\right)\sin\left(x_{2}\right)+\\
 & \quad-2\lambda_{n}B_{n}\left(t\right)a_{n}\left(t\right)^{2}\varphi'\left(\lambda_{n}x_{1}\right)\varphi\left(\lambda_{n}x_{2}\right)\cos\left(x_{1}\right)\sin\left(x_{2}\right)+\\
 & \quad-2\lambda_{n}B_{n}\left(t\right)b_{n}\left(t\right)^{2}\varphi\left(\lambda_{n}x_{1}\right)\varphi'\left(\lambda_{n}x_{2}\right)\sin\left(x_{1}\right)\cos\left(x_{2}\right).
\end{aligned}
\]
\item Proposition \ref{prop:form of the velocity}:
\[
\begin{aligned}\widetilde{u^{\left(n\right)}}^{n}\left(t,x\right) & =\widetilde{V^{\left(n\right)}}\left(t,x\right)+\lambda_{n}\widetilde{W^{\left(n\right)}}\left(t,x\right),\\
\widetilde{V^{\left(n\right)}}\left(t,x\right) & =B_{n}\left(t\right)\varphi\left(\lambda_{n}x_{1}\right)\varphi\left(\lambda_{n}x_{2}\right)\left(\begin{matrix}b_{n}\left(t\right)\sin\left(x_{1}\right)\cos\left(x_{2}\right)\\
-a_{n}\left(t\right)\cos\left(x_{1}\right)\sin\left(x_{2}\right)
\end{matrix}\right).
\end{aligned}
\]
\[
{\small \begin{matrix}\begin{aligned}\left|\left|V_{1}^{\left(n\right)}\left(t,\cdot\right)\right|\right|_{L^{\infty}\left(\mathbb{R}^{2}\right)} & \le B_{n}\left(t\right)b_{n}\left(t\right),\\
\left|\left|W_{1}^{\left(n\right)}\left(t,\cdot\right)\right|\right|_{L^{\infty}\left(\mathbb{R}^{2}\right)} & \lesssim_{\varphi}B_{n}\left(t\right)b_{n}\left(t\right),\\
\left|\left|V_{1}^{\left(n\right)}\left(t,\cdot\right)\right|\right|_{\dot{C}^{\alpha}\left(\mathbb{R}^{2}\right)} & \lesssim_{\varphi}B_{n}\left(t\right)b_{n}\left(t\right)\max\left\{ a_{n}\left(t\right)^{\alpha},b_{n}\left(t\right)^{\alpha}\right\} ,\\
\left|\left|W_{1}^{\left(n\right)}\left(t,\cdot\right)\right|\right|_{\dot{C}^{\alpha}\left(\mathbb{R}^{2}\right)} & \lesssim_{\varphi}B_{n}\left(t\right)b_{n}\left(t\right)\max\left\{ a_{n}\left(t\right)^{\alpha},b_{n}\left(t\right)^{\alpha}\right\} ,
\end{aligned}
 &  & \begin{aligned}\left|\left|V_{2}^{\left(n\right)}\left(t,\cdot\right)\right|\right|_{L^{\infty}\left(\mathbb{R}^{2}\right)} & \le B_{n}\left(t\right)a_{n}\left(t\right),\\
\left|\left|W_{2}^{\left(n\right)}\left(t,\cdot\right)\right|\right|_{L^{\infty}\left(\mathbb{R}^{2}\right)} & \lesssim_{\varphi}B_{n}\left(t\right)a_{n}\left(t\right),\\
\left|\left|V_{2}^{\left(n\right)}\left(t,\cdot\right)\right|\right|_{\dot{C}^{\alpha}\left(\mathbb{R}^{2}\right)} & \lesssim_{\varphi}B_{n}\left(t\right)a_{n}\left(t\right)\max\left\{ a_{n}\left(t\right)^{\alpha},b_{n}\left(t\right)^{\alpha}\right\} ,\\
\left|\left|W_{2}^{\left(n\right)}\left(t,\cdot\right)\right|\right|_{\dot{C}^{\alpha}\left(\mathbb{R}^{2}\right)} & \lesssim_{\varphi}B_{n}\left(t\right)a_{n}\left(t\right)\max\left\{ a_{n}\left(t\right)^{\alpha},b_{n}\left(t\right)^{\alpha}\right\} .
\end{aligned}
\end{matrix}}
\]
\item Proposition \ref{prop:form gradient omega}:
\[
\begin{aligned}\widetilde{\nabla}^{n}\widetilde{\omega^{\left(n\right)}}^{n}\left(t,x\right) & =\widetilde{\Gamma^{\left(n\right)}}^{n}\left(t,x\right)+\lambda_{n}\widetilde{G^{\left(n\right)}}^{n}\left(t,x\right),\\
\widetilde{\Gamma^{\left(n\right)}}^{n}\left(t,x\right) & =B_{n}\left(t\right)\left(a_{n}\left(t\right)^{2}+b_{n}\left(t\right)^{2}\right)\varphi\left(\lambda_{n}x_{1}\right)\varphi\left(\lambda_{n}x_{2}\right)\left(\begin{matrix}a_{n}\left(t\right)\cos\left(x_{1}\right)\sin\left(x_{2}\right)\\
b_{n}\left(t\right)\sin\left(x_{1}\right)\cos\left(x_{2}\right)
\end{matrix}\right).
\end{aligned}
\]
\[
{\small \begin{matrix}\begin{aligned}\left|\left|\Gamma_{1}^{\left(n\right)}\left(t,\cdot\right)\right|\right|_{L^{\infty}\left(\mathbb{R}^{2}\right)} & \le B_{n}\left(t\right)a_{n}\left(t\right)\left(a_{n}\left(t\right)^{2}+b_{n}\left(t\right)^{2}\right),\\
\left|\left|G_{1}^{\left(n\right)}\left(t,\cdot\right)\right|\right|_{L^{\infty}\left(\mathbb{R}^{2}\right)} & \lesssim_{\varphi}B_{n}\left(t\right)a_{n}\left(t\right)\max\left\{ a_{n}\left(t\right)^{2},b_{n}\left(t\right)^{2}\right\} ,\\
\left|\left|\widetilde{\Gamma_{1}^{\left(n\right)}}^{n}\left(t,\cdot\right)\right|\right|_{\dot{C}^{\alpha}\left(\mathbb{R}^{2}\right)} & \lesssim_{\varphi}B_{n}\left(t\right)a_{n}\left(t\right)\left(a_{n}\left(t\right)^{2}+b_{n}\left(t\right)^{2}\right),\\
\left|\left|\widetilde{G_{1}^{\left(n\right)}}^{n}\left(t,\cdot\right)\right|\right|_{\dot{C}^{\alpha}\left(\mathbb{R}^{2}\right)} & \lesssim_{\varphi}B_{n}\left(t\right)a_{n}\left(t\right)\max\left\{ a_{n}\left(t\right)^{2},b_{n}\left(t\right)^{2}\right\} ,\\
\left|\left|\Gamma_{1}^{\left(n\right)}\left(t,\cdot\right)\right|\right|_{\dot{C}^{\alpha}\left(\mathbb{R}^{2}\right)} & \lesssim_{\varphi}B_{n}\left(t\right)a_{n}\left(t\right)\left(a_{n}\left(t\right)^{2}+b_{n}\left(t\right)^{2}\right)\cdot\\
 & \quad\cdot\max\left\{ a_{n}\left(t\right)^{\alpha},b_{n}\left(t\right)^{\alpha}\right\} ,\\
\left|\left|G_{1}^{\left(n\right)}\left(t,\cdot\right)\right|\right|_{\dot{C}^{\alpha}\left(\mathbb{R}^{2}\right)} & \lesssim_{\varphi}B_{n}\left(t\right)a_{n}\left(t\right)\cdot\\
 & \quad\cdot\max\left\{ a_{n}\left(t\right)^{2+\alpha},b_{n}\left(t\right)^{2+\alpha}\right\} ,
\end{aligned}
 &  & \begin{aligned}\left|\left|\Gamma_{2}^{\left(n\right)}\left(t,\cdot\right)\right|\right|_{L^{\infty}\left(\mathbb{R}^{2}\right)} & \le B_{n}\left(t\right)b_{n}\left(t\right)\left(a_{n}\left(t\right)^{2}+b_{n}\left(t\right)^{2}\right),\\
\left|\left|G_{2}^{\left(n\right)}\left(t,\cdot\right)\right|\right|_{L^{\infty}\left(\mathbb{R}^{2}\right)} & \lesssim_{\varphi}B_{n}\left(t\right)b_{n}\left(t\right)\max\left\{ a_{n}\left(t\right)^{2},b_{n}\left(t\right)^{2}\right\} ,\\
\left|\left|\widetilde{\Gamma_{2}^{\left(n\right)}}^{n}\left(t,\cdot\right)\right|\right|_{\dot{C}^{\alpha}\left(\mathbb{R}^{2}\right)} & \lesssim_{\varphi}B_{n}\left(t\right)b_{n}\left(t\right)\left(a_{n}\left(t\right)^{2}+b_{n}\left(t\right)^{2}\right),\\
\left|\left|\widetilde{G_{2}^{\left(n\right)}}^{n}\left(t,\cdot\right)\right|\right|_{\dot{C}^{\alpha}\left(\mathbb{R}^{2}\right)} & \lesssim_{\varphi}B_{n}\left(t\right)b_{n}\left(t\right)\max\left\{ a_{n}\left(t\right)^{2},b_{n}\left(t\right)^{2}\right\} ,\\
\left|\left|\Gamma_{2}^{\left(n\right)}\left(t,\cdot\right)\right|\right|_{\dot{C}^{\alpha}\left(\mathbb{R}^{2}\right)} & \lesssim_{\varphi}B_{n}\left(t\right)b_{n}\left(t\right)\left(a_{n}\left(t\right)^{2}+b_{n}\left(t\right)^{2}\right)\cdot\\
 & \quad\cdot\max\left\{ a_{n}\left(t\right)^{\alpha},b_{n}\left(t\right)^{\alpha}\right\} ,\\
\left|\left|G_{2}^{\left(n\right)}\left(t,\cdot\right)\right|\right|_{\dot{C}^{\alpha}\left(\mathbb{R}^{2}\right)} & \lesssim_{\varphi}B_{n}\left(t\right)b_{n}\left(t\right)\cdot\\
 & \quad\cdot\max\left\{ a_{n}\left(t\right)^{2+\alpha},b_{n}\left(t\right)^{2+\alpha}\right\} .
\end{aligned}
\end{matrix}}
\]
\item Choice \ref{choice:anbncn} and Proposition \ref{prop:relation between anbn and jacobian}:
\[
\begin{alignedat}{1}\phi_{2}^{\left(n\right)}\left(t,0\right) & =0,\\
\frac{\partial\phi_{1}^{\left(n\right)}}{\partial t}\left(t,0\right) & =\sum_{m=1}^{n-1}B_{m}\left(t\right)b_{m}\left(t\right)\sin\left(a_{m}\left(t\right)\left(\phi_{1}^{\left(n\right)}\left(t,0\right)-\phi_{1}^{\left(m\right)}\left(t,0\right)\right)\right),\\
\frac{\mathrm{d}}{\mathrm{d}t}\left(\ln\left(b_{n}\left(t\right)\right)\right) & =\sum_{m=1}^{n-1}B_{m}\left(t\right)a_{m}\left(t\right)b_{m}\left(t\right)\cos\left(a_{m}\left(t\right)\left(\phi_{1}^{\left(n\right)}\left(t,0\right)-\phi_{1}^{\left(m\right)}\left(t,0\right)\right)\right),\\
\frac{\mathrm{d}}{\mathrm{d}t}\left(a_{n}\left(t\right)b_{n}\left(t\right)\right) & =0.
\end{alignedat}
\]
\item Equation \eqref{eq:ODE JI}:
\[
\begin{aligned}\frac{\mathrm{d}\Xi^{\left(n\right)}}{\mathrm{d}t}\left(t\right) & =B_{n-1}\left(t\right)b_{n-1}\left(t\right)\sin\left(a_{n-1}\left(t\right)\Xi^{\left(n\right)}\left(t\right)\right)+\\
 & \quad+\sum_{m=1}^{n-2}B_{m}\left(t\right)b_{m}\left(t\right)\left[\sin\left(a_{m}\left(t\right)\left(\phi_{1}^{\left(n\right)}\left(t,0\right)-\phi_{1}^{\left(m\right)}\left(t,0\right)\right)\right)+\right.\\
 & \qquad\left.-\sin\left(a_{m}\left(t\right)\left(\phi_{1}^{\left(n-1\right)}\left(t,0\right)-\phi_{1}^{\left(m\right)}\left(t,0\right)\right)\right)\right].
\end{aligned}
\]
\item Equation \eqref{eq:ODE Jin0}:
\[
\frac{\mathrm{d}\Xi_{0}^{\left(n\right)}}{\mathrm{d}t}\left(t\right)=B_{n-1}\left(1\right)b_{n-1}\left(1\right)\sin\left(a_{n-1}\left(1\right)\Xi_{0}^{\left(n\right)}\left(t\right)\right).
\]
\item Choices \ref{choice:anbn} and \ref{choice:lambdan}:
\[
b_{n}\left(t\right)=C^{\left(1+k_{n}\left(t\right)\right)\left(\frac{1}{1-\gamma}\right)^{n}},\quad a_{n}\left(t\right)=C^{\left(1-k_{n}\left(t\right)\right)\left(\frac{1}{1-\gamma}\right)^{n}},\quad\lambda_{n}=C^{-\Lambda\left(\frac{1}{1-\gamma}\right)^{n}}.
\]
\item Choice \ref{choice:density}:
\[
\widetilde{\rho^{\left(n\right)}}^{n}\left(t,x\right)=-\frac{1}{b_{n}\left(t\right)}\frac{\mathrm{d}}{\mathrm{d}t}\left[B_{n}\left(t\right)\left(a_{n}\left(t\right)^{2}+b_{n}\left(t\right)^{2}\right)\right]\varphi\left(\lambda_{n}x_{1}\right)\varphi\left(\lambda_{n}x_{2}\right)\sin\left(x_{1}\right)\cos\left(x_{2}\right).
\]
\item Choice \ref{choice:amplitude density}:
\[
\begin{aligned}\frac{1}{b_{n}\left(t\right)}\frac{\mathrm{d}}{\mathrm{d}t}\left[B_{n}\left(t\right)\left(a_{n}\left(t\right)^{2}+b_{n}\left(t\right)^{2}\right)\right] & =z_{n}\frac{h^{\left(n\right)}\left(t\right)}{\int_{t_{n}}^{1}h^{\left(n\right)}\left(s\right)b_{n}\left(s\right)\mathrm{d}s},\\
B_{n}\left(t\right) & =\frac{z_{n}}{a_{n}\left(t\right)^{2}+b_{n}\left(t\right)^{2}}\frac{\int_{t_{n}}^{t}h^{\left(n\right)}\left(s\right)b_{n}\left(s\right)\mathrm{d}s}{\int_{t_{n}}^{1}h^{\left(n\right)}\left(s\right)b_{n}\left(s\right)\mathrm{d}s}.
\end{aligned}
\]
\item Proposition \ref{prop:form of gradient of rho}:
\[
\begin{aligned}\widetilde{\nabla}^{n}\widetilde{\rho^{\left(n\right)}}^{n}\left(t,x\right)= & -2M_{n}\frac{h^{\left(n\right)}\left(t\right)}{\int_{t_{n}}^{1}h^{\left(n\right)}\left(s\right)b_{n}\left(s\right)\mathrm{d}s}\left[\varphi\left(\lambda_{n}x_{1}\right)\varphi\left(\lambda_{n}x_{2}\right)\left(\begin{matrix}a_{n}\left(t\right)\cos\left(x_{1}\right)\cos\left(x_{2}\right)\\
-b_{n}\left(t\right)\sin\left(x_{1}\right)\sin\left(x_{2}\right)
\end{matrix}\right)+\right.\\
 & \quad\left.+\lambda_{n}\left(\begin{matrix}a_{n}\left(t\right)\varphi'\left(\lambda_{n}x_{1}\right)\varphi\left(\lambda_{n}x_{2}\right)\\
b_{n}\left(t\right)\varphi\left(\lambda_{n}x_{1}\right)\varphi'\left(\lambda_{n}x_{2}\right)
\end{matrix}\right)\sin\left(x_{1}\right)\cos\left(x_{2}\right)\right].
\end{aligned}
\]
\item Equation \eqref{eq:Taylor remainder real variables}:
\[
\begin{aligned}W^{\left(n\right)}\left(t,x\right) & =\frac{1}{2}\left[\int_{0}^{1}\left(1-s\right)\frac{\partial^{2}\widetilde{U^{\left(n-1\right)}}^{n}}{\partial x_{1}^{2}}\left(s\left(\phi^{\left(n\right)}\right)^{-1}\left(t,x\right)\right)\mathrm{d}s\right]\left[\left(\phi^{\left(n\right)}\right)_{1}^{-1}\left(t,x\right)\right]^{2}+\\
 & \quad+\left[\int_{0}^{1}\left(1-s\right)\frac{\partial^{2}\widetilde{U^{\left(n-1\right)}}^{n}}{\partial x_{1}\partial x_{2}}\left(s\left(\phi^{\left(n\right)}\right)^{-1}\left(t,x\right)\right)\mathrm{d}s\right]\left(\phi^{\left(n\right)}\right)_{1}^{-1}\left(t,x\right)\left(\phi^{\left(n\right)}\right)_{2}^{-1}\left(t,x\right)+\\
 & \quad+\frac{1}{2}\left[\int_{0}^{1}\left(1-s\right)\frac{\partial^{2}\widetilde{U^{\left(n-1\right)}}^{n}}{\partial x_{2}^{2}}\left(s\left(\phi^{\left(n\right)}\right)^{-1}\left(t,x\right)\right)\mathrm{d}s\right]\left[\left(\phi^{\left(n\right)}\right)_{2}^{-1}\left(t,x\right)\right]^{2}.
\end{aligned}
\]
\item Proposition \ref{prop:form of gradient of rho}:
\[
\begin{aligned}\left|\left|\widetilde{\nabla}^{n}\widetilde{\rho^{\left(n\right)}}^{n}\left(t,\left(\phi^{\left(n\right)}\right)^{-1}\left(t,\cdot\right)\right)\right|\right|_{L^{\infty}\left(\mathbb{R}^{2};\mathbb{R}^{2}\right)} & \lesssim_{\varphi,\mu}YC^{\left[2\delta+2\mu\right]\left(\frac{1}{1-\gamma}\right)^{n}},\\
\left|\left|\widetilde{\nabla}^{n}\widetilde{\rho^{\left(n\right)}}^{n}\left(t,\left(\phi^{\left(n\right)}\right)^{-1}\left(t,\cdot\right)\right)\right|\right|_{\dot{C}^{\alpha}\left(\mathbb{R}^{2};\mathbb{R}^{2}\right)} & \lesssim_{\varphi,\mu}YC^{\left[\alpha\left(1+k_{\max}\right)+2\delta+3\mu\right]\left(\frac{1}{1-\gamma}\right)^{n}},\\
\left|\left|\widetilde{\nabla}^{n}\widetilde{\rho^{\left(n\right)}}^{n}\left(t,\left(\phi^{\left(n\right)}\right)^{-1}\left(t,\cdot\right)\right)\right|\right|_{\dot{C}^{1}\left(\mathbb{R}^{2};\mathbb{R}^{2}\right)} & \lesssim_{\varphi,\mu}YC^{\left[\left(1+k_{\max}\right)+2\delta+3\mu\right]\left(\frac{1}{1-\gamma}\right)^{n}},\\
\left|\left|\widetilde{\nabla}^{n}\widetilde{\rho^{\left(n\right)}}^{n}\left(t,\left(\phi^{\left(n\right)}\right)^{-1}\left(t,\cdot\right)\right)\right|\right|_{\dot{C}^{1,\alpha}\left(\mathbb{R}^{2};\mathbb{R}^{2}\right)} & \lesssim_{\varphi,\mu}YC^{\left[\left(1+\alpha\right)\left(1+k_{\max}\right)+2\delta+4\mu\right]\left(\frac{1}{1-\gamma}\right)^{n}}.
\end{aligned}
\]
\item Proposition \ref{prop:bounds for Taylor development of transport}
under the hypothesis $1-k_{\max}-\Lambda>0$:
\[
\begin{aligned}\left|\left|W^{\left(n\right)}\left(t,\cdot\right)\right|\right|_{L^{\infty}\left(D^{\left(n\right)}\left(t\right);\mathbb{R}^{2}\right)} & \lesssim_{\delta,\varphi}YC^{\left[-2\left(1-\Lambda-k_{\max}\right)+2\mu+\left(1+\delta\right)\left(1-\gamma\right)\right]\left(\frac{1}{1-\gamma}\right)^{n}},\\
\left|\left|W^{\left(n\right)}\left(t,\cdot\right)\right|\right|_{\dot{C}^{\alpha}\left(D^{\left(n\right)}\left(t\right);\mathbb{R}^{2}\right)} & \lesssim_{\delta,\varphi}YC^{\left[-2\left(1-\Lambda-k_{\max}\right)+\max\left\{ \alpha\left(1-k_{\max}\right)-\Lambda,0\right\} +2\mu+\left(2+\delta\right)\left(1-\gamma\right)\right]\left(\frac{1}{1-\gamma}\right)^{n}},\\
\left|\left|W^{\left(n\right)}\left(t,\cdot\right)\right|\right|_{\dot{C}^{1}\left(D^{\left(n\right)}\left(t\right);\mathbb{R}^{2}\right)} & \lesssim_{\delta,\varphi}YC^{\left[-\left(1-\Lambda-k_{\max}\right)+2\mu+\left(2+\delta\right)\left(1-\gamma\right)\right]\left(\frac{1}{1-\gamma}\right)^{n}},\\
\left|\left|W^{\left(n\right)}\left(t,\cdot\right)\right|\right|_{\dot{C}^{1,\alpha}\left(D^{\left(n\right)}\left(t\right);\mathbb{R}^{2}\right)} & \lesssim_{\delta,\varphi}YC^{\left[-\min\left\{ \left(1-\alpha\right)\left(1-k_{\max}\right),1-k_{\max}-\Lambda\right\} +2\mu+\left(3+\delta\right)\left(1-\gamma\right)\right]\left(\frac{1}{1-\gamma}\right)^{n}}.
\end{aligned}
\]
\end{itemize}
\newpage{}

\section*{Acknowledgements}

This work is supported in part by the Spanish Ministry of Science
and Innovation, through the “Severo Ochoa Programme for Centres of
Excellence in R\&D (CEX2019-000904-S \& CEX2023-001347-S)” and 152878NB-I00.
We were also partially supported by the ERC Advanced Grant 788250,
and by the SNF grant FLUTURA: Fluids, Turbulence, Advection No. 212573.

\bibliographystyle{alpha}

\end{document}